\documentclass[a4paper,12pt]{article}
\usepackage{tikz}
\usetikzlibrary{matrix,arrows,decorations.markings}
\usepackage{amsmath,amsfonts,amssymb,amsthm,xurl,textcomp,mathrsfs}
\usepackage{mathtools}
\usepackage{csquotes}
\usepackage{a4wide}
\usepackage[numbers]{natbib}
\usepackage{fancyhdr}
\usepackage{anyfontsize}
\usepackage{hyperref}
\usepackage{hhline}
\usepackage{leftidx}
\usepackage{enumitem}
\usepackage{mdframed}
\usepackage[ruled,vlined]{algorithm2e}
\usepackage{makecell}
\usepackage{longtable}
\usepackage{soul}

\allowdisplaybreaks

\makeatletter
\let\@fnsymbol\@arabic
\makeatother

\newtheorem{theorem}{Theorem}\numberwithin{theorem}{section}
\newtheorem*{theorem*}{Theorem}
\newtheorem*{lemma*}{Lemma}
\newtheorem{definition}[theorem]{Definition}
\newtheorem*{definition*}{Definition}
\newtheorem*{proposition*}{Proposition}

\newtheorem{theoremm}{Theorem}\numberwithin{theoremm}{subsection}
\newtheorem{deffinition}[theoremm]{Definition}
\newtheorem{lemmma}[theoremm]{Lemma}
\newtheorem{corrollary}[theoremm]{Corollary}

\newtheorem{propposition}[theoremm]{Proposition}
\newtheorem{connjecture}[theoremm]{Conjecture}
\newtheorem{quesstion}[theoremm]{Question}
\newtheorem{problemm}[theoremm]{Problem}
\newtheorem{theoremmm}{Theorem}\numberwithin{theoremmm}{subsubsection}
\newtheorem{lemmmma}[theoremmm]{Lemma}
\newtheorem{defffinition}[theoremmm]{Definition}
\newtheorem{corrrollary}[theoremmm]{Corollary}
\newtheorem{proppposition}[theoremmm]{Proposition}
\theoremstyle{remark}

\newtheorem{remark}[theorem]{Remark}

\newtheorem*{example*}{Example}
\newtheorem{remmark}[theoremm]{Remark}
\newtheorem{exxample}[theoremm]{Example}
\newtheorem{remmmark}[theoremmm]{Remark}

\setlistdepth{9}

\newlist{denumerate}{enumerate}{9}
\setlist[denumerate,1]{label=(\arabic*)}
\setlist[denumerate,2]{label=(\Roman*)}
\setlist[denumerate,3]{label=(\Alph*)}
\setlist[denumerate,4]{label=(\roman*)}
\setlist[denumerate,5]{label=(\alph*)}
\setlist[denumerate,6]{label=(\arabic*)}
\setlist[denumerate,7]{label=(\Roman*)}
\setlist[denumerate,8]{label=(\Alph*)}
\setlist[denumerate,9]{label=(\roman*)}

\newcommand{\Aut}{\operatorname{Aut}}

\newcommand{\lcm}{\operatorname{lcm}}

\newcommand{\ord}{\operatorname{ord}}
\newcommand{\Sym}{\operatorname{Sym}}

\renewcommand{\l}{\operatorname{l}}

\newcommand{\id}{\operatorname{id}}

\newcommand{\e}{\mathrm{e}}

\newcommand{\im}{\operatorname{im}}
\newcommand{\V}{\operatorname{V}}

\newcommand{\Mod}[1]{\ \left(\textup{mod}\ #1\right)}

\newcommand{\conj}{\operatorname{conj}}
\newcommand{\IN}{\mathbb{N}}
\newcommand{\Hcal}{\mathcal{H}}

\newcommand{\IF}{\mathbb{F}}
\newcommand{\Stab}{\operatorname{Stab}}

\newcommand{\IZ}{\mathbb{Z}}

\newcommand{\rem}{\operatorname{rem}}

\newcommand{\dom}{\operatorname{dom}}

\newcommand{\Acal}{\mathcal{A}}
\newcommand{\Pcal}{\mathcal{P}}

\newcommand{\CT}{\operatorname{CT}}

\newcommand{\Q}{\operatorname{Q}}

\newcommand{\Rcal}{\mathcal{R}}

\newcommand{\inv}{\operatorname{inv}}

\newcommand{\Tcal}{\mathcal{T}}

\newcommand{\Ocal}{\mathcal{O}}

\newcommand{\hfrak}{\mathfrak{h}}

\newcommand{\num}{\mathrm{num}}
\newcommand{\den}{\mathrm{den}}

\newcommand{\rfrak}{\mathfrak{r}}
\newcommand{\dfrak}{\mathfrak{d}}

\newcommand{\Gcal}{\mathcal{G}}
\newcommand{\Ical}{\mathcal{I}}

\newcommand{\Ccal}{\mathcal{C}}

\newcommand{\IQ}{\mathbb{Q}}

\newcommand{\BU}{\operatorname{BU}}
\newcommand{\ufrak}{\mathfrak{u}}
\newcommand{\Vcal}{\mathcal{V}}
\newcommand{\Bcal}{\mathcal{B}}
\newcommand{\Lcal}{\mathcal{L}}

\newcommand{\wfrak}{\mathfrak{w}}
\newcommand{\Ucal}{\mathcal{U}}

\newcommand{\Wcal}{\mathcal{W}}

\newcommand{\AGL}{\operatorname{AGL}}

\newcommand{\Pfrak}{\mathfrak{P}}

\newcommand{\AC}{\operatorname{AC}}

\newcommand{\Cfrak}{\mathfrak{C}}

\newcommand{\per}{\operatorname{per}}
\newcommand{\nil}{\operatorname{nil}}
\newcommand{\Tree}{\operatorname{Tree}}

\newcommand{\mpe}{\operatorname{mpe}}
\newcommand{\Xcal}{\mathcal{X}}
\newcommand{\Qcal}{\mathcal{Q}}
\newcommand{\pfrak}{\mathfrak{p}}
\newcommand{\aord}{\operatorname{aord}}
\newcommand{\ffrak}{\mathfrak{f}}
\newcommand{\Good}{\operatorname{Good}}

\newcommand{\Expand}{\operatorname{Expand}}
\newcommand{\Scal}{\mathcal{S}}
\newcommand{\height}{\operatorname{ht}}

\newcommand{\proj}{\operatorname{proj}}
\newcommand{\Ncal}{\mathcal{N}}
\newcommand{\Tfrak}{\mathfrak{T}}
\newcommand{\trans}{\operatorname{trans}}
\newcommand{\bfrak}{\mathfrak{b}}
\newcommand{\afrak}{\mathfrak{a}}
\newcommand{\minper}{\operatorname{minper}}
\newcommand{\nfrak}{\mathfrak{n}}
\newcommand{\Rfrak}{\mathfrak{R}}
\newcommand{\Nfrak}{\mathfrak{N}}

\newcommand{\Mfrak}{\mathfrak{M}}
\newcommand{\Vfrak}{\mathfrak{V}}
\newcommand{\Lfrak}{\mathfrak{L}}
\newcommand{\lab}{\operatorname{lab}}
\newcommand{\Lab}{\operatorname{Lab}}
\newcommand{\mfrak}{\mathfrak{m}}
\newcommand{\tfrak}{\mathfrak{t}}
\newcommand{\pperl}{\operatorname{pperl}}
\newcommand{\perl}{\operatorname{perl}}
\newcommand{\lfrak}{\mathfrak{l}}
\newcommand{\xfrak}{\mathfrak{x}}
\newcommand{\Xfrak}{\mathfrak{X}}

\newcommand{\proc}{\operatorname{proc}}
\newcommand{\rt}{\operatorname{rt}}
\newcommand{\Ifrak}{\mathfrak{I}}
\newcommand{\rrm}{\mathrm{r}}
\newcommand{\sfrak}{\mathfrak{s}}
\newcommand{\vfrak}{\mathfrak{v}}
\newcommand{\SF}{\operatorname{SF}}
\newcommand{\Dfrak}{\mathfrak{D}}
\newcommand{\Zcal}{\mathcal{Z}}
\newcommand{\class}{\mathrm{class}}
\newcommand{\quant}{\mathrm{quant}}
\newcommand{\conv}{\mathrm{conv}}

\newcommand{\inn}{\mathrm{in}}
\newcommand{\fact}{\mathrm{fact}}
\newcommand{\mdl}{\mathrm{mdl}}
\newcommand{\fdl}{\mathrm{fdl}}
\newcommand{\prt}{\mathrm{prt}}
\newcommand{\qry}{\mathrm{qry}}
\newcommand{\mord}{\mathrm{mord}}

\newcommand{\yfrak}{\mathfrak{y}}

\newcommand{\LV}{\mathrm{LV}}
\newcommand{\para}{\operatorname{par}}
\newcommand{\Jcal}{\mathcal{J}}
\newcommand{\Hfrak}{\mathfrak{H}}
\newcommand{\Layer}{\operatorname{Layer}}
\newcommand{\Type}{\operatorname{Type}}
\newcommand{\Yfrak}{\mathfrak{Y}}
\newcommand{\test}{\operatorname{test}}
\newcommand{\Test}{\operatorname{Test}}
\newcommand{\minperl}{\operatorname{minperl}}
\newcommand{\ifrak}{\mathfrak{i}}
\newcommand{\product}{\operatorname{prod}}
\newcommand{\up}{\mathrm{upper}}
\newcommand{\low}{\mathrm{lower}}
\newcommand{\WX}{\operatorname{WX}}
\newcommand{\Ycal}{\mathcal{Y}}
\newcommand{\TRAG}{\operatorname{TRAG}}
\newcommand{\STRAG}{\operatorname{STRAG}}
\newcommand{\Fcal}{\mathcal{F}}
\newcommand{\total}{\mathrm{total}}
\newcommand{\bit}{\mathrm{bit}}

\begin{document}

\title{Functional graphs of generalized cyclotomic mappings of finite fields}

\author{Alexander Bors\textsuperscript{1} \and Daniel Panario\textsuperscript{1} \and Qiang Wang\thanks{School of Mathematics and Statistics, Carleton University, 1125 Colonel By Drive, Ottawa ON K1S 5B6, Canada. \newline First author's e-mail: \href{mailto:alexanderbors@cunet.carleton.ca}{alexanderbors@cunet.carleton.ca} \newline Second author's e-mail: \href{mailto:daniel@math.carleton.ca}{daniel@math.carleton.ca} \newline Third author's e-mail: \href{mailto:wang@math.carleton.ca}{wang@math.carleton.ca} \newline The authors were supported by the Natural Sciences and Engineering Research Council of Canada, projects RGPIN-2018-05328 (A.~Bors and D.~Panario) and RGPIN-2017-06410 (Q.~Wang). \newline 2020 \emph{Mathematics Subject Classification}: Primary: 11T22. Secondary: 05C05, 05C25, 05C60, 11A07, 37P25. \newline \emph{Keywords and phrases}: cyclotomic mapping; finite dynamical system; finite field; functional graph; generalized cyclotomic mapping.}}

\date{\today}

\maketitle

\abstract{The functional graph of a function $g:X\rightarrow X$ is the directed graph with vertex set $X$ the edges of which are of the form $x\rightarrow g(x)$ for $x\in X$. Functional graphs are heavily studied because they allow one to understand the behavior of $g$ under iteration (i.e., to understand the discrete dynamical system $(X,g)$), which has various applications, especially when $X$ is a finite field $\IF_q$. This paper is an extensive study of the functional graphs of so-called index $d$ generalized cyclotomic mappings of $\IF_q$, which are a natural and manageable generalization of monomial functions. We provide both theoretical results on the structure of their functional graphs and Las Vegas algorithms for solving fundamental problems, such as parametrizing the connected components of the functional graph by representative vertices, or describing the structure of a connected component given by a representative vertex. The complexity of these algorithms is analyzed in detail, and we make the point that for fixed index $d$ and most prime powers $q$ (in the sense of asymptotic density), suitable implementations of these algorithms have an expected runtime that is polynomial in $\log{q}$ on quantum computers, whereas their expected runtime is subexponential in $\log{q}$ on a classical computer. We also discuss four special cases in which one can devise Las Vegas algorithms with this kind of complexity behavior over most finite fields that solve the graph isomorphism problem for functional graphs of generalized cyclotomic mappings.}

\section{Introduction}\label{sec1}

A \emph{discrete dynamical system}\phantomsection\label{term1} is a pair $(X,g)$ where $X$\phantomsection\label{not1} is a set and $g$\phantomsection\label{not2} is a function $X\rightarrow X$. The motivation behind this definition is to think of a complicated system that evolves in discrete time steps (such as a neural network), with $X$ being the set of all states which the system can assume, and $g(x)$ being the successor state of $x\in X$. For this reason, one calls $X$ the \emph{state space}\phantomsection\label{term2} and $g$ the \emph{(state) transition function}\phantomsection\label{term3} of $(X,g)$. When studying a discrete dynamical system $(X,g)$, one is naturally interested in the behavior of $g$ under iteration (i.e., in the function iterates $g^n$ for $n\in\IN_0=\{n\in\IZ: n\geq0\}$). See the monograph \cite{Mar99a} for a general introduction to discrete dynamical systems, and \cite[Chapter 7]{Mar99a} in particular for some examples of practical applications of them.

When $X$ is finite, one also calls $(X,g)$ a \emph{finite dynamical system}\phantomsection\label{term4}. Some important special cases with regard to applications are when $X=\IZ/m\IZ$\phantomsection\label{not3} and $g$ is a polynomial modulo $m$\phantomsection\label{not4} (which is used in Pollard's rho algorithm \cite{Pol75a}), or when $X=\IF_q^n$ (Cartesian power of the finite field $\IF_q$\phantomsection\label{not5}\phantomsection\label{not6}), with a particular focus on $q=2$ in the literature (see \cite{JLSS07a,LP01a,LBL91a,MW93a,VL12a}). It should be noted that one may identify $\IF_q^n$ with $\IF_{q^n}$ by fixing an $\IF_q$-basis in the latter, so there is in fact no loss of generality when assuming $n=1$ (i.e., when only considering finite fields themselves as state spaces).

A simple yet remarkable fact when $X$ is finite is that all points $x\in X$\phantomsection\label{not7} are \emph{pre-periodic under $g$}\phantomsection\label{term5}, i.e., there exist unique smallest integers $\pperl_g(x)\geq0$\phantomsection\label{not8} and $\perl_g(x)\geq1$\phantomsection\label{not9}, called the \emph{pre-period (length)}\phantomsection\label{term6} and \emph{period (length)}\phantomsection\label{term7} of $x$ under $g$ respectively, such that $g^{\pperl_g(x)+\perl_g(x)}(x)=g^{\pperl_g(x)}(x)$; in case $\pperl_g(x)=0$, one says that $x$ is \emph{periodic under $g$} (or \emph{$g$-periodic})\phantomsection\label{term8}. The subset of $X$ consisting of all $g$-periodic points is denoted by $\per(g)$\phantomsection\label{not10}. A point in $X$ that is not $g$-periodic is called \emph{transient under $g$} (or $g$-transient)\phantomsection\label{term9}. Various stochastic parameters of random functions $X\rightarrow X$ that are of interest for the study of finite dynamical systems, such as the expected pre-period and period length of a point, were determined in \cite{FO90a}.

An important means of visualizing a discrete dynamical system $(X,g)$, especially when $X$ is finite, is the so-called \emph{functional graph of $g$}\phantomsection\label{term10}, denoted by $\Gamma_g$\phantomsection\label{not11}. This is the directed graph with vertex set $X$ that has an arc (directed edge)\phantomsection\label{term11} $x\rightarrow g(x)$ for each $x\in X$, and no other arcs. It is straightforward to show that a directed graph $\Gamma$\phantomsection\label{not11P5} with vertex set $X$ is a functional graph (i.e., is of the form $\Gamma_g$ for some $g:X\rightarrow X$) if and only if each $x\in X$ has out-degree $1$ in $\Gamma$.

Particularly for \emph{finite} functional graphs $\Gamma_g$, one can give the following precise characterization of their shape: A \emph{connected component of $\Gamma_g$}\phantomsection\label{term12} is the induced subgraph of $\Gamma_g$ on a subset of $X$ that is the vertex set of a connected component of the underlying undirected graph of $\Gamma_g$. Each such connected component contains a single cycle of periodic points of $g$. Apart from those periodic points, the connected component consists precisely of those points which eventually map to the cycle after sufficiently many iterations of $g$ -- the \emph{iterated pre-images (under $g$)}\phantomsection\label{term13} of points on the cycle. For each $x$ on the cycle, the iterated pre-images of $x$ form a directed rooted tree, with root $x$, that has all of its arcs oriented toward the root. Henceforth, for simplicity, whenever we say \enquote{(directed) rooted tree}\phantomsection\label{term14}, it means \enquote{directed rooted tree in which all arcs are oriented toward the root}. Here is a picture to illustrate the situation:

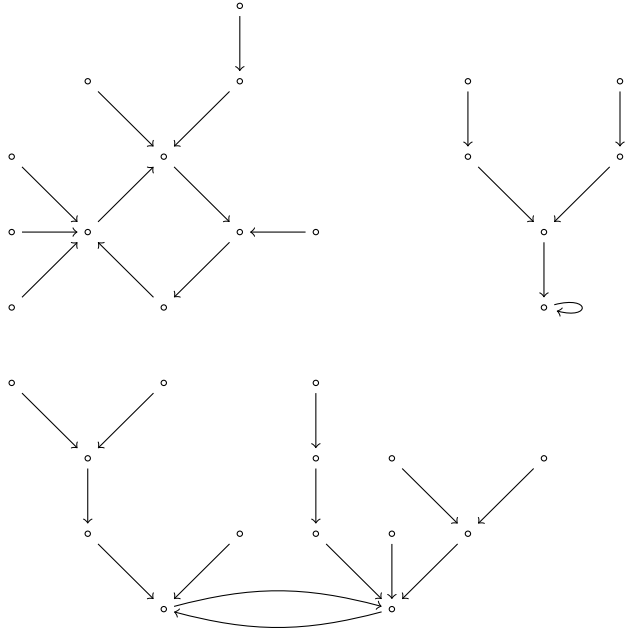
\begin{center}
\begin{tikzpicture}
\node (v03) at (0,3) {};
\draw (v03) circle [radius=1pt];
\node (v04) at (0,4) {};
\draw (v04) circle [radius=1pt];
\node (v05) at (0,5) {};
\draw (v05) circle [radius=1pt];
\node (v06) at (0,6) {};
\draw (v06) circle [radius=1pt];
\node (v11) at (1,1) {};
\draw (v11) circle [radius=1pt];
\node (v12) at (1,2) {};
\draw (v12) circle [radius=1pt];
\node (v15) at (1,5) {};
\draw (v15) circle [radius=1pt];
\node (v17) at (1,7) {};
\draw (v17) circle [radius=1pt];
\node (v20) at (2,0) {};
\draw (v20) circle [radius=1pt];
\node (v23) at (2,3) {};
\draw (v23) circle [radius=1pt];
\node (v24) at (2,4) {};
\draw (v24) circle [radius=1pt];
\node (v26) at (2,6) {};
\draw (v26) circle [radius=1pt];
\node (v31) at (3,1) {};
\draw (v31) circle [radius=1pt];
\node (v35) at (3,5) {};
\draw (v35) circle [radius=1pt];
\node (v37) at (3,7) {};
\draw (v37) circle [radius=1pt];
\node (v38) at (3,8) {};
\draw (v38) circle [radius=1pt];
\node (v41) at (4,1) {};
\draw (v41) circle [radius=1pt];
\node (v42) at (4,2) {};
\draw (v42) circle [radius=1pt];
\node (v43) at (4,3) {};
\draw (v43) circle [radius=1pt];
\node (v45) at (4,5) {};
\draw (v45) circle [radius=1pt];
\node (v50) at (5,0) {};
\draw (v50) circle [radius=1pt];
\node (v51) at (5,1) {};
\draw (v51) circle [radius=1pt];
\node (v52) at (5,2) {};
\draw (v52) circle [radius=1pt];
\node (v61) at (6,1) {};
\draw (v61) circle [radius=1pt];
\node (v66) at (6,6) {};
\draw (v66) circle [radius=1pt];
\node (v67) at (6,7) {};
\draw (v67) circle [radius=1pt];
\node (v72) at (7,2) {};
\draw (v72) circle [radius=1pt];
\node (v74) at (7,4) {};
\draw (v74) circle [radius=1pt];
\node (v75) at (7,5) {};
\draw (v75) circle [radius=1pt];
\node (v86) at (8,6) {};
\draw (v86) circle [radius=1pt];
\node (v87) at (8,7) {};
\draw (v87) circle [radius=1pt];
\draw[->]
(v03) edge (v12)
(v04) edge (v15)
(v05) edge (v15)
(v06) edge (v15)
(v11) edge (v20)
(v12) edge (v11)
(v15) edge (v26)
(v17) edge (v26)
(v20) edge[bend left=15] (v50)
(v23) edge (v12)
(v24) edge (v15)
(v26) edge (v35)
(v31) edge (v20)
(v35) edge (v24)
(v37) edge (v26)
(v38) edge (v37)
(v41) edge (v50)
(v42) edge (v41)
(v43) edge (v42)
(v45) edge (v35)
(v50) edge[bend left=15] (v20)
(v51) edge (v50)
(v52) edge (v61)
(v61) edge (v50)
(v66) edge (v75)
(v67) edge (v66)
(v72) edge (v61)
(v74) edge[loop right] (v74)
(v75) edge (v74)
(v86) edge (v75)
(v87) edge (v86);
\end{tikzpicture}
\end{center}

Conversely, each finite digraph of the shape described above is a functional graph, as it is readily verified that all vertices in it have out-degree $1$. The study of finite dynamical systems may be understood as the study of finite functional graphs. In this context, it is also noteworthy that in case $\psi_1$ and $\psi_2$ are permutations of a finite set, we have $\Gamma_{\psi_1}\cong\Gamma_{\psi_2}$ if and only if $\psi_1$ and $\psi_2$ are of the same \emph{cycle type}\phantomsection\label{term15}, i.e., they have the same number of cycles of each given length. Formally, the cycle type of a permutation $\psi$\phantomsection\label{not12} of $X$, denoted by $\CT(\psi)$\phantomsection\label{not13}, is defined as the unique monomial in $\IQ[x_n: n\in\IN^+]$\phantomsection\label{not14}, where $\IN^+=\{n\in\IZ: n\geq1\}$, in which the degree of each variable $x_n$\phantomsection\label{not15} is the number of $\psi$-cycles of length equal to $n$. For example, if $X=\{1,2,\ldots,9\}$ and $\psi=(1,2,3)(4,5)(6,7)(8)(9)$, then $\CT(\psi)=x_1^2x_2^2x_3$. Cycle types (and the related notion of cycle indices) are well-studied in combinatorics, and studying isomorphism types of functional graphs may be seen as a natural generalization of this to arbitrary functions on finite sets.

Functional graphs of certain classes of functions on finite fields received considerable attention recently, see the papers \cite{CS04a,MP16a,PR19a,QP15a,QP19a,Ugo13a,VS04a} and references therein. Additionally, the papers \cite{CS04a,PMMY01a} do not deal explicitly with functional graphs, but with the iteration of functions on finite fields, and their results could be reformulated in terms of functional graphs. In this paper, we contribute to this line of research by investigating functional graphs of so-called \emph{generalized cyclotomic mappings} in the following sense.

\begin{definition}\label{genCycMapDef}
Let $q$ be a prime power, and let $d\mid q-1$. A \emph{generalized cyclotomic mapping of $\IF_q$ of index $d$}\phantomsection\label{term16}\phantomsection\label{not16} is a function $f:\IF_q\rightarrow\IF_q$\phantomsection\label{not17} with $f(0)=0$ such that the restriction of $f$ to each coset $C_i$ of the unique index $d$ subgroup of $\IF_q^{\ast}$ agrees with a monomial function $x\mapsto a_ix^{r_i}$.
\end{definition}

More specifically, let $\omega$\phantomsection\label{not18} be a primitive element of $\IF_q$ (i.e., a generator of the cyclic multiplicative group $\IF_q^{\ast}$), and let $C$\phantomsection\label{not19} be the index $d$ subgroup of $\IF_q^{\ast}$. The $d$ cosets of $C$ in $\IF_q^{\ast}$ are of the form $C_i=\omega^iC$\phantomsection\label{not20} for $i=0,1,\ldots,d-1$\phantomsection\label{not21}. The general form of an index $d$ generalized cyclotomic mapping $f$ of $\IF_q$ is
\begin{equation}\label{cyclotomicFormEq}
f(x)=
\begin{cases}
0, & \text{if }x=0, \\
a_0x^{r_0}, & \text{if }x\in C=C_0, \\
a_1x^{r_1}, & \text{if }x\in C_1, \\
\vdots & \vdots \\
a_{d-1}x^{r_{d-1}}, & \text{if }x\in C_{d-1},
\end{cases}
\end{equation}
where\phantomsection\label{not22}\phantomsection\label{not23} $a_i\in\IF_q$ and $r_i\in\{0,1,\ldots,q-2\}$\phantomsection\label{not24} for $i=0,1,\ldots,d-1$. These functions are interesting because they generalize monomial mappings (which constitute the special case $d=1$) while still being relatively well-controlled. From an abstract algebraic point of view, it is noteworthy that monomial functions $\IF_q^{\ast}\rightarrow\IF_q^{\ast}$, $x\mapsto ax^r$ where $a\not=0_{\IF_q}$ necessarily, are \emph{affine maps} of the multiplicative group $\IF_q^{\ast}$, in the sense that they are compositions of a group endomorphism of $\IF_q^{\ast}$ (viz., the power function $x\mapsto x^r$) with a (multiplicative) translation $x\mapsto ax$ by a fixed group element $a$ (see also Definition \ref{affineMapDef}). Hence, at least generalized cyclotomic mappings in which all coefficients $a_i$ from (\ref{cyclotomicFormEq}) are nonzero may be viewed as \enquote{coset-wise affine} functions, and we explore the idea of generalizing the methods and results from this paper to other (possibly non-abelian) groups in Subsection \ref{subsec6P4}. In this context, we also note that the celebrated Collatz function $g:\IZ\rightarrow\IZ$, given by the formula
\[
g(x)=
\begin{cases}
x/2, & \text{if }x\in 2\IZ, \\
3x+1, & \text{if }x\in 2\IZ+1,
\end{cases}
\]
is also a coset-wise affine function, of its respective domain of definition group $\IZ$. The reason why we are able to develop a theory for understanding the behavior of generalized cyclotomic mappings under iteration in this paper (while the analogous task for the Collatz function is wide open) is because generalized cyclotomic mappings preserve the associated partition of $\IF_q$ into the cosets $C_i$ and the singleton set $\{0_{\IF_q}\}$ (and, relatedly, they form a semigroup under function composition) -- see also the distinction between the two concepts introduced in Definition \ref{cosetwiseAffineDef}(2,3).

We observe that a given generalized cyclotomic mapping of $\IF_q$ may have several possible indices, and that \emph{every} function $f:\IF_q\rightarrow\IF_q$ with $f(0)=0$ is a generalized cyclotomic mapping of $\IF_q$ of index $q-1$, though the study of generalized cyclotomic mappings is mostly focused on small values of $d$. For $d=q-1$, known methods of handling generalized cyclotomic mappings, such as \cite[Algorithm 1]{BW22b}, are essentially the trivial brute-force approaches. Apart from \cite{BW22b}, generalized cyclotomic mappings were also studied in \cite{Wan13a,Wan17a,ZYZP16a}. An important special case is when $r_i=r$ for all $i$; then one speaks of an \emph{$r$-th order cyclotomic mapping of $\IF_q$ of index $d$}\phantomsection\label{term17}, and those functions were studied e.g.~in \cite{NW05a,WL91a,Wan07a,Wan19a}.

Our goal in this paper is to develop algorithms that answer fundamental questions concerning the structure of the functional graph $\Gamma_f$ of a given index $d$ generalized cyclotomic mapping $f$ of $\IF_q$, specified in the form (\ref{cyclotomicFormEq}). For example, let $\omega$ be any fixed primitive element of $\IF_{256}$, and consider the following index $5$ generalized cyclotomic mapping $f$ of $\IF_{256}$.
\begin{equation}\label{concreteFIntroEq}
f(x)=
\begin{cases}
0, & \text{if }x=0, \\
\omega^5x^9, & \text{if }x\in C_0, \\
x^3, & \text{if }x\in C_1, \\
x^{17}, & \text{if }x\in C_2, \\
\omega^3x^{34}, & \text{if }x\in C_3, \\
\omega^4x^9, & \text{if }x\in C_4.
\end{cases}
\end{equation}
The functional graph $\Gamma_f$ has $256$ vertices, and one can understand its structure by drawing it, which we do below. In this drawing, a vertex labeled $n\in\{0,1,\ldots,254\}$ corresponds to the field element $\omega^n\in\IF_q^{\ast}$ (in particular, the label $0$ corresponds to the field element $\omega^0=1_{\IF_{256}}$), whereas the vertex representing the field element $0$ is labeled by $0_{\IF_{256}}$. It turns out that $\Gamma_f$ has four connected components, and in one of them (the fourth one in our order of drawing), the rooted trees formed by the $f$-transient iterated pre-images of three particular $f$-periodic points in that connected component are relatively large and thus drawn separately; we mark those rooted trees with $\Delta_k$ for $k\in\{1,2,3\}$ in the schematic drawing of the corresponding connected component.

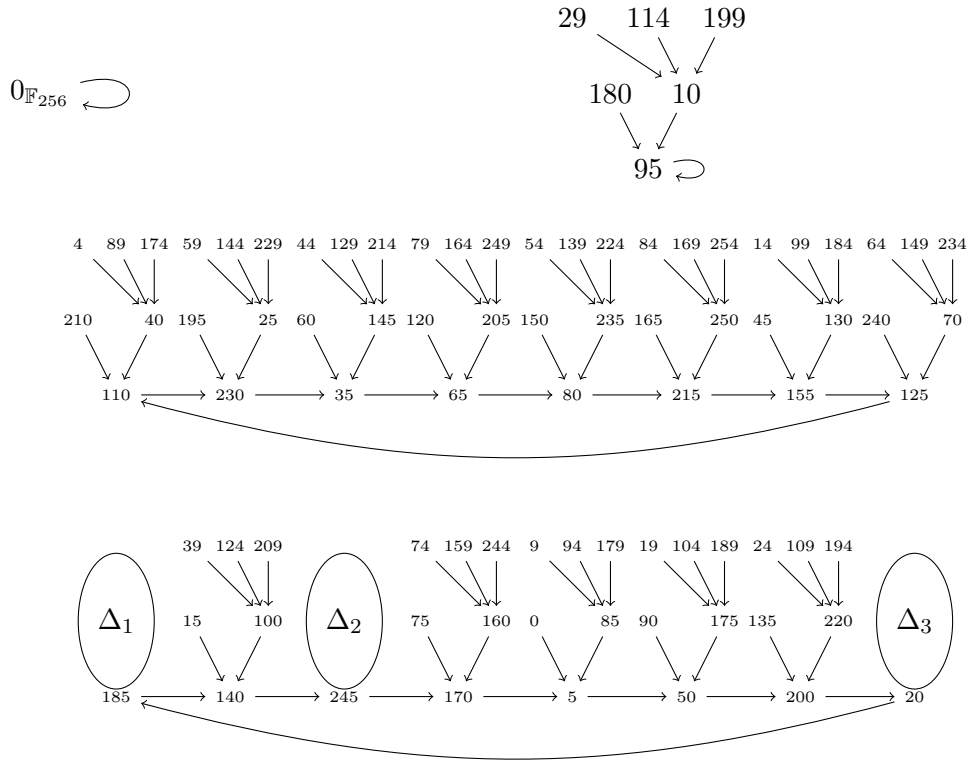
\begin{center}
\begin{tikzpicture}
\node (0Fq) at (0,0) {$0_{\IF_{256}}$};
\path[->]
(0Fq) edge [loop right] (0Fq);
\node (95) at (8,-1) {$95$};
\node (180) at (7.5,0) {$180$};
\node (10) at (8.5,0) {$10$};
\node (29) at (7,1) {$29$};
\node (114) at (8,1) {$114$};
\node (199) at (9,1) {$199$};
\path[->]
(95) edge [loop right] (95)
(180) edge (95)
(10) edge (95)
(29) edge (10)
(114) edge (10)
(199) edge (10);
\node (110) at (1,-4) {\tiny $110$};
\node (210) at (0.5,-3) {\tiny $210$};
\node (40) at (1.5,-3) {\tiny $40$};
\node (4) at (0.5,-2) {\tiny $4$};
\node (89) at (1,-2) {\tiny $89$};
\node (174) at (1.5,-2) {\tiny $174$};
\path[->]
(210) edge (110)
(40) edge (110)
(4) edge (40)
(89) edge (40)
(174) edge (40);
\node (230) at (2.5,-4) {\tiny $230$};
\node (195) at (2,-3) {\tiny $195$};
\node (25) at (3,-3) {\tiny $25$};
\node (59) at (2,-2) {\tiny $59$};
\node (144) at (2.5,-2) {\tiny $144$};
\node (229) at (3,-2) {\tiny $229$};
\path[->]
(195) edge (230)
(25) edge (230)
(59) edge (25)
(144) edge (25)
(229) edge (25);
\node (35) at (4,-4) {\tiny $35$};
\node (60) at (3.5,-3) {\tiny $60$};
\node (145) at (4.5,-3) {\tiny $145$};
\node (44) at (3.5,-2) {\tiny $44$};
\node (129) at (4,-2) {\tiny $129$};
\node (214) at (4.5,-2) {\tiny $214$};
\path[->]
(60) edge (35)
(145) edge (35)
(44) edge (145)
(129) edge (145)
(214) edge (145);
\node (65) at (5.5,-4) {\tiny $65$};
\node (120) at (5,-3) {\tiny $120$};
\node (205) at (6,-3) {\tiny $205$};
\node (79) at (5,-2) {\tiny $79$};
\node (164) at (5.5,-2) {\tiny $164$};
\node (249) at (6,-2) {\tiny $249$};
\path[->]
(120) edge (65)
(205) edge (65)
(79) edge (205)
(164) edge (205)
(249) edge (205);
\node (80) at (7,-4) {\tiny $80$};
\node (150) at (6.5,-3) {\tiny $150$};
\node (235) at (7.5,-3) {\tiny $235$};
\node (54) at (6.5,-2) {\tiny $54$};
\node (139) at (7,-2) {\tiny $139$};
\node (224) at (7.5,-2) {\tiny $224$};
\path[->]
(150) edge (80)
(235) edge (80)
(54) edge (235)
(139) edge (235)
(224) edge (235);
\node (215) at (8.5,-4) {\tiny $215$};
\node (165) at (8,-3) {\tiny $165$};
\node (250) at (9,-3) {\tiny $250$};
\node (84) at (8,-2) {\tiny $84$};
\node (169) at (8.5,-2) {\tiny $169$};
\node (254) at (9,-2) {\tiny $254$};
\path[->]
(165) edge (215)
(250) edge (215)
(84) edge (250)
(169) edge (250)
(254) edge (250);
\node (155) at (10,-4) {\tiny $155$};
\node (45) at (9.5,-3) {\tiny $45$};
\node (130) at (10.5,-3) {\tiny $130$};
\node (14) at (9.5,-2) {\tiny $14$};
\node (99) at (10,-2) {\tiny $99$};
\node (184) at (10.5,-2) {\tiny $184$};
\path[->]
(45) edge (155)
(130) edge (155)
(14) edge (130)
(99) edge (130)
(184) edge (130);
\node (125) at (11.5,-4) {\tiny $125$};
\node (240) at (11,-3) {\tiny $240$};
\node (70) at (12,-3) {\tiny $70$};
\node (64) at (11,-2) {\tiny $64$};
\node (149) at (11.5,-2) {\tiny $149$};
\node (234) at (12,-2) {\tiny $234$};
\path[->]
(240) edge (125)
(70) edge (125)
(64) edge (70)
(149) edge (70)
(234) edge (70);
\path[->]
(110) edge (230)
(230) edge (35)
(35) edge (65)
(65) edge (80)
(80) edge (215)
(215) edge (155)
(155) edge (125)
(125) edge[bend left=15] (110);
\node (185) at (1,-8) {\tiny $185$};
\node (Delta1) at (1,-7) {$\Delta_1$};
\draw (Delta1) circle[x radius=0.5, y radius=0.9];
\node (140) at (2.5,-8) {\tiny $140$};
\node (15) at (2,-7) {\tiny $15$};
\node (100) at (3,-7) {\tiny $100$};
\node (39) at (2,-6) {\tiny $39$};
\node (124) at (2.5,-6) {\tiny $124$};
\node (209) at (3,-6) {\tiny $209$};
\path[->]
(15) edge (140)
(100) edge (140)
(39) edge (100)
(124) edge (100)
(209) edge (100);
\node (245) at (4,-8) {\tiny $245$};
\node (Delta2) at (4,-7) {$\Delta_2$};
\draw (Delta2) circle[x radius=0.5, y radius=0.9];
\node (170) at (5.5,-8) {\tiny $170$};
\node (75) at (5,-7) {\tiny $75$};
\node (160) at (6,-7) {\tiny $160$};
\node (74) at (5,-6) {\tiny $74$};
\node (159) at (5.5,-6) {\tiny $159$};
\node (244) at (6,-6) {\tiny $244$};
\path[->]
(75) edge (170)
(160) edge (170)
(74) edge (160)
(159) edge (160)
(244) edge (160);
\node (5) at (7,-8) {\tiny $5$};
\node (0) at (6.5,-7) {\tiny $0$};
\node (85) at (7.5,-7) {\tiny $85$};
\node (9) at (6.5,-6) {\tiny $9$};
\node (94) at (7,-6) {\tiny $94$};
\node (179) at (7.5,-6) {\tiny $179$};
\path[->]
(0) edge (5)
(85) edge (5)
(9) edge (85)
(94) edge (85)
(179) edge (85);
\node (50) at (8.5,-8) {\tiny $50$};
\node (90) at (8,-7) {\tiny $90$};
\node (175) at (9,-7) {\tiny $175$};
\node (19) at (8,-6) {\tiny $19$};
\node (104) at (8.5,-6) {\tiny $104$};
\node (189) at (9,-6) {\tiny $189$};
\path[->]
(90) edge (50)
(175) edge (50)
(19) edge (175)
(104) edge (175)
(189) edge (175);
\node (200) at (10,-8) {\tiny $200$};
\node (135) at (9.5,-7) {\tiny $135$};
\node (220) at (10.5,-7) {\tiny $220$};
\node (24) at (9.5,-6) {\tiny $24$};
\node (109) at (10,-6) {\tiny $109$};
\node (194) at (10.5,-6) {\tiny $194$};
\path[->]
(135) edge (200)
(220) edge (200)
(24) edge (220)
(109) edge (220)
(194) edge (220);
\node (20) at (11.5,-8) {\tiny $20$};
\node (Delta3) at (11.5,-7) {$\Delta_3$};
\draw (Delta3) circle[x radius=0.5, y radius=0.9];
\path[->]
(185) edge (140)
(140) edge (245)
(245) edge (170)
(170) edge (5)
(5) edge (50)
(50) edge (200)
(200) edge (20)
(20) edge[bend left=15] (185);
\end{tikzpicture}
\end{center}

The rooted tree $\Delta_1$ looks as follows.

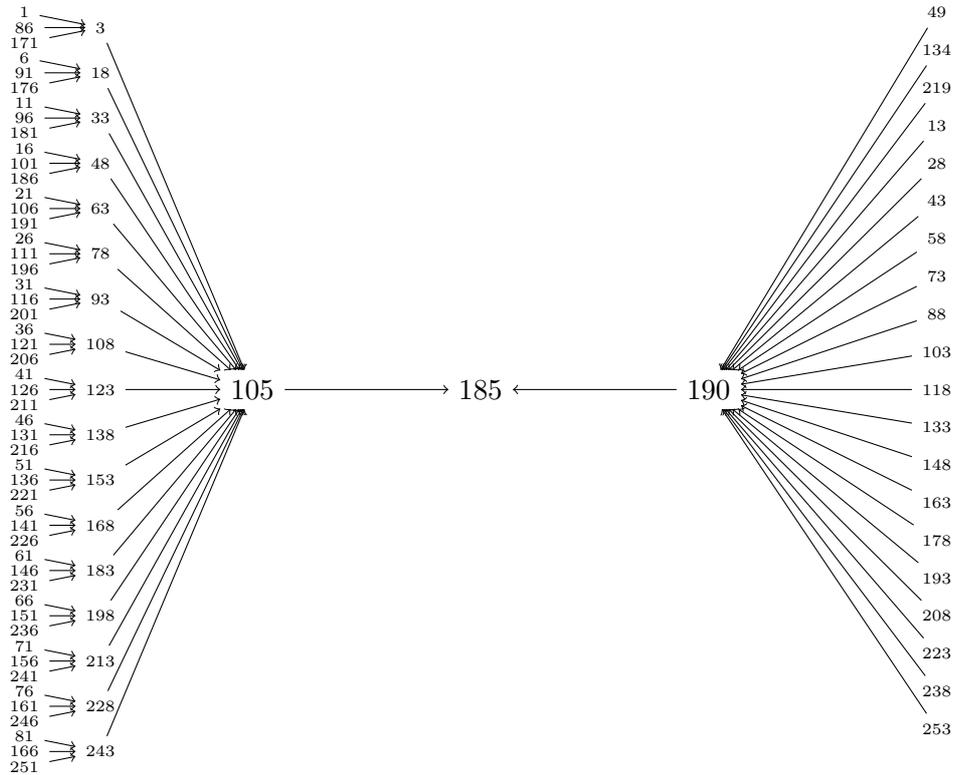
\begin{center}
\begin{tikzpicture}
\node (1) at (0,0) {\tiny $1$};
\node (86) at (0,-0.2) {\tiny $86$};
\node (171) at (0,-0.4) {\tiny $171$};
\node (6) at (0,-0.6) {\tiny $6$};
\node (91) at (0,-0.8) {\tiny $91$};
\node (176) at (0,-1) {\tiny $176$};
\node (11) at (0,-1.2) {\tiny $11$};
\node (96) at (0,-1.4) {\tiny $96$};
\node (181) at (0,-1.6) {\tiny $181$};
\node (16) at (0,-1.8) {\tiny $16$};
\node (101) at (0,-2) {\tiny $101$};
\node (186) at (0,-2.2) {\tiny $186$};
\node (21) at (0,-2.4) {\tiny $21$};
\node (106) at (0,-2.6) {\tiny $106$};
\node (191) at (0,-2.8) {\tiny $191$};
\node (26) at (0,-3) {\tiny $26$};
\node (111) at (0,-3.2) {\tiny $111$};
\node (196) at (0,-3.4) {\tiny $196$};
\node (31) at (0,-3.6) {\tiny $31$};
\node (116) at (0,-3.8) {\tiny $116$};
\node (201) at (0,-4) {\tiny $201$};
\node (36) at (0,-4.2) {\tiny $36$};
\node (121) at (0,-4.4) {\tiny $121$};
\node (206) at (0,-4.6) {\tiny $206$};
\node (41) at (0,-4.8) {\tiny $41$};
\node (126) at (0,-5) {\tiny $126$};
\node (211) at (0,-5.2) {\tiny $211$};
\node (46) at (0,-5.4) {\tiny $46$};
\node (131) at (0,-5.6) {\tiny $131$};
\node (216) at (0,-5.8) {\tiny $216$};
\node (51) at (0,-6) {\tiny $51$};
\node (136) at (0,-6.2) {\tiny $136$};
\node (221) at (0,-6.4) {\tiny $221$};
\node (56) at (0,-6.6) {\tiny $56$};
\node (141) at (0,-6.8) {\tiny $141$};
\node (226) at (0,-7) {\tiny $226$};
\node (61) at (0,-7.2) {\tiny $61$};
\node (146) at (0,-7.4) {\tiny $146$};
\node (231) at (0,-7.6) {\tiny $231$};
\node (66) at (0,-7.8) {\tiny $66$};
\node (151) at (0,-8) {\tiny $151$};
\node (236) at (0,-8.2) {\tiny $236$};
\node (71) at (0,-8.4) {\tiny $71$};
\node (156) at (0,-8.6) {\tiny $156$};
\node (241) at (0,-8.8) {\tiny $241$};
\node (76) at (0,-9) {\tiny $76$};
\node (161) at (0,-9.2) {\tiny $161$};
\node (246) at (0,-9.4) {\tiny $246$};
\node (81) at (0,-9.6) {\tiny $81$};
\node (166) at (0,-9.8) {\tiny $166$};
\node (251) at (0,-10) {\tiny $251$};
\node (3) at (1,-0.2) {\tiny $3$};
\node (18) at (1,-0.8) {\tiny $18$};
\node (33) at (1,-1.4) {\tiny $33$};
\node (48) at (1,-2) {\tiny $48$};
\node (63) at (1,-2.6) {\tiny $63$};
\node (78) at (1,-3.2) {\tiny $78$};
\node (93) at (1,-3.8) {\tiny $93$};
\node (108) at (1,-4.4) {\tiny $108$};
\node (123) at (1,-5) {\tiny $123$};
\node (138) at (1,-5.6) {\tiny $138$};
\node (153) at (1,-6.2) {\tiny $153$};
\node (168) at (1,-6.8) {\tiny $168$};
\node (183) at (1,-7.4) {\tiny $183$};
\node (198) at (1,-8) {\tiny $198$};
\node (213) at (1,-8.6) {\tiny $213$};
\node (228) at (1,-9.2) {\tiny $228$};
\node (243) at (1,-9.8) {\tiny $243$};
\node (105) at (3,-5) {$105$};
\node (185) at (6,-5) {$185$};
\path[->]
(1) edge (3)
(86) edge (3)
(171) edge (3)
(6) edge (18)
(91) edge (18)
(176) edge (18)
(11) edge (33)
(96) edge (33)
(181) edge (33)
(16) edge (48)
(101) edge (48)
(186) edge (48)
(21) edge (63)
(106) edge (63)
(191) edge (63)
(26) edge (78)
(111) edge (78)
(196) edge (78)
(31) edge (93)
(116) edge (93)
(201) edge (93)
(36) edge (108)
(121) edge (108)
(206) edge (108)
(41) edge (123)
(126) edge (123)
(211) edge (123)
(46) edge (138)
(131) edge (138)
(216) edge (138)
(51) edge (153)
(136) edge (153)
(221) edge (153)
(56) edge (168)
(141) edge (168)
(226) edge (168)
(61) edge (183)
(146) edge (183)
(231) edge (183)
(66) edge (198)
(151) edge (198)
(236) edge (198)
(71) edge (213)
(156) edge (213)
(241) edge (213)
(76) edge (228)
(161) edge (228)
(246) edge (228)
(81) edge (243)
(166) edge (243)
(251) edge (243)
(3) edge (105)
(18) edge (105)
(33) edge (105)
(48) edge (105)
(63) edge (105)
(78) edge (105)
(93) edge (105)
(108) edge (105)
(123) edge (105)
(138) edge (105)
(153) edge (105)
(168) edge (105)
(183) edge (105)
(198) edge (105)
(213) edge (105)
(228) edge (105)
(243) edge (105)
(105) edge (185);
\node (190) at (9,-5) {$190$};
\node (49) at (12,0) {\tiny $49$};
\node (134) at (12,-0.5) {\tiny $134$};
\node (219) at (12,-1) {\tiny $219$};
\node (13) at (12,-1.5) {\tiny $13$};
\node (28) at (12,-2) {\tiny $28$};
\node (43) at (12,-2.5) {\tiny $43$};
\node (58) at (12,-3) {\tiny $58$};
\node (73) at (12,-3.5) {\tiny $73$};
\node (88) at (12,-4) {\tiny $88$};
\node (103) at (12,-4.5) {\tiny $103$};
\node (118) at (12,-5) {\tiny $118$};
\node (133) at (12,-5.5) {\tiny $133$};
\node (148) at (12,-6) {\tiny $148$};
\node (163) at (12,-6.5) {\tiny $163$};
\node (178) at (12,-7) {\tiny $178$};
\node (193) at (12,-7.5) {\tiny $193$};
\node (208) at (12,-8) {\tiny $208$};
\node (223) at (12,-8.5) {\tiny $223$};
\node (238) at (12,-9) {\tiny $238$};
\node (253) at (12,-9.5) {\tiny $253$};
\path[->]
(49) edge (190)
(134) edge (190)
(219) edge (190)
(13) edge (190)
(28) edge (190)
(43) edge (190)
(58) edge (190)
(73) edge (190)
(88) edge (190)
(103) edge (190)
(118) edge (190)
(133) edge (190)
(148) edge (190)
(163) edge (190)
(178) edge (190)
(193) edge (190)
(208) edge (190)
(223) edge (190)
(238) edge (190)
(253) edge (190)
(190) edge (185);
\end{tikzpicture}
\end{center}

The rooted tree $\Delta_2$ looks as follows.

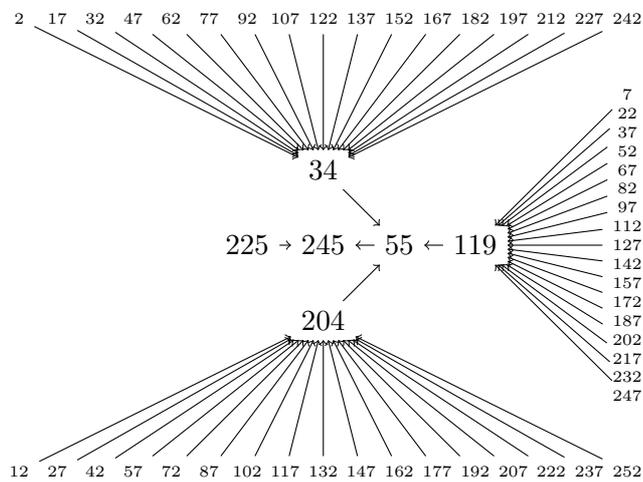
\begin{center}
\begin{tikzpicture}
\node (245) at (0,0) {$245$};
\node (225) at (-1,0) {$225$};
\node (55) at (1,0) {$55$};
\node (119) at (2,0) {$119$};
\node (34) at (0,1) {$34$};
\node (204) at (0,-1) {$204$};
\node (2) at (-4,3) {\tiny $2$};
\node (17) at (-3.5,3) {\tiny $17$};
\node (32) at (-3,3) {\tiny $32$};
\node (47) at (-2.5,3) {\tiny $47$};
\node (62) at (-2,3) {\tiny $62$};
\node (77) at (-1.5,3) {\tiny $77$};
\node (92) at (-1,3) {\tiny $92$};
\node (107) at (-0.5,3) {\tiny $107$};
\node (122) at (0,3) {\tiny $122$};
\node (137) at (0.5,3) {\tiny $137$};
\node (152) at (1,3) {\tiny $152$};
\node (167) at (1.5,3) {\tiny $167$};
\node (182) at (2,3) {\tiny $182$};
\node (197) at (2.5,3) {\tiny $197$};
\node (212) at (3,3) {\tiny $212$};
\node (227) at (3.5,3) {\tiny $227$};
\node (242) at (4,3) {\tiny $242$};
\node (12) at (-4,-3) {\tiny $12$};
\node (27) at (-3.5,-3) {\tiny $27$};
\node (42) at (-3,-3) {\tiny $42$};
\node (57) at (-2.5,-3) {\tiny $57$};
\node (72) at (-2,-3) {\tiny $72$};
\node (87) at (-1.5,-3) {\tiny $87$};
\node (102) at (-1,-3) {\tiny $102$};
\node (117) at (-0.5,-3) {\tiny $117$};
\node (132) at (0,-3) {\tiny $132$};
\node (147) at (0.5,-3) {\tiny $147$};
\node (162) at (1,-3) {\tiny $162$};
\node (177) at (1.5,-3) {\tiny $177$};
\node (192) at (2,-3) {\tiny $192$};
\node (207) at (2.5,-3) {\tiny $207$};
\node (222) at (3,-3) {\tiny $222$};
\node (237) at (3.5,-3) {\tiny $237$};
\node (252) at (4,-3) {\tiny $252$};
\node (7) at (4,2) {\tiny $7$};
\node (22) at (4,1.75) {\tiny $22$};
\node (37) at (4,1.5) {\tiny $37$};
\node (52) at (4,1.25) {\tiny $52$};
\node (67) at (4,1) {\tiny $67$};
\node (82) at (4,0.75) {\tiny $82$};
\node (97) at (4,0.5) {\tiny $97$};
\node (112) at (4,0.25) {\tiny $112$};
\node (127) at (4,0) {\tiny $127$};
\node (142) at (4,-0.25) {\tiny $142$};
\node (157) at (4,-0.5) {\tiny $157$};
\node (172) at (4,-0.75) {\tiny $172$};
\node (187) at (4,-1) {\tiny $187$};
\node (202) at (4,-1.25) {\tiny $202$};
\node (217) at (4,-1.5) {\tiny $217$};
\node (232) at (4,-1.75) {\tiny $232$};
\node (247) at (4,-2) {\tiny $247$};
\path[->]
(225) edge (245)
(55) edge (245)
(119) edge (55)
(34) edge (55)
(204) edge (55)
(2) edge (34)
(17) edge (34)
(32) edge (34)
(47) edge (34)
(62) edge (34)
(77) edge (34)
(92) edge (34)
(107) edge (34)
(122) edge (34)
(137) edge (34)
(152) edge (34)
(167) edge (34)
(182) edge (34)
(197) edge (34)
(212) edge (34)
(227) edge (34)
(242) edge (34)
(7) edge (119)
(22) edge (119)
(37) edge (119)
(52) edge (119)
(67) edge (119)
(82) edge (119)
(97) edge (119)
(112) edge (119)
(127) edge (119)
(142) edge (119)
(157) edge (119)
(172) edge (119)
(187) edge (119)
(202) edge (119)
(217) edge (119)
(232) edge (119)
(247) edge (119)
(12) edge (204)
(27) edge (204)
(42) edge (204)
(57) edge (204)
(72) edge (204)
(87) edge (204)
(102) edge (204)
(117) edge (204)
(132) edge (204)
(147) edge (204)
(162) edge (204)
(177) edge (204)
(192) edge (204)
(207) edge (204)
(222) edge (204)
(237) edge (204)
(252) edge (204);
\end{tikzpicture}
\end{center}

The rooted tree $\Delta_3$ looks as follows.

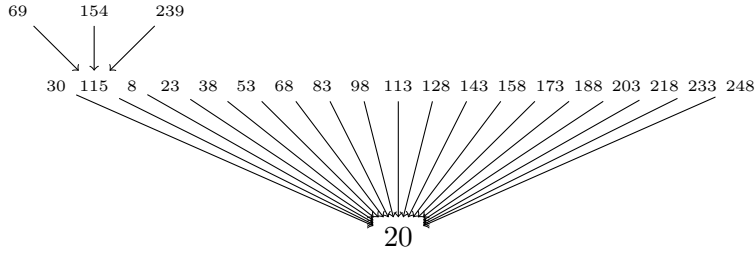
\begin{center}
\begin{tikzpicture}
\node (20) at (0,0) {$20$};
\node (30) at (-4.5,2) {\tiny $30$};
\node (115) at (-4,2) {\tiny $115$};
\node (8) at (-3.5,2) {\tiny $8$};
\node (23) at (-3,2) {\tiny $23$};
\node (38) at (-2.5,2) {\tiny $38$};
\node (53) at (-2,2) {\tiny $53$};
\node (68) at (-1.5,2) {\tiny $68$};
\node (83) at (-1,2) {\tiny $83$};
\node (98) at (-0.5,2) {\tiny $98$};
\node (113) at (0,2) {\tiny $113$};
\node (128) at (0.5,2) {\tiny $128$};
\node (143) at (1,2) {\tiny $143$};
\node (158) at (1.5,2) {\tiny $158$};
\node (173) at (2,2) {\tiny $173$};
\node (188) at (2.5,2) {\tiny $188$};
\node (203) at (3,2) {\tiny $203$};
\node (218) at (3.5,2) {\tiny $218$};
\node (233) at (4,2) {\tiny $233$};
\node (248) at (4.5,2) {\tiny $248$};
\node (69) at (-5,3) {\tiny $69$};
\node (154) at (-4,3) {\tiny $154$};
\node (239) at (-3,3) {\tiny $239$};
\path[->]
(30) edge (20)
(115) edge (20)
(8) edge (20)
(23) edge (20)
(38) edge (20)
(53) edge (20)
(68) edge (20)
(83) edge (20)
(98) edge (20)
(113) edge (20)
(128) edge (20)
(143) edge (20)
(158) edge (20)
(173) edge (20)
(188) edge (20)
(203) edge (20)
(218) edge (20)
(233) edge (20)
(248) edge (20)
(69) edge (115)
(154) edge (115)
(239) edge (115);
\end{tikzpicture}
\end{center}

Of course, this approach of understanding $\Gamma_f$ by drawing it becomes intractable for large values of $q$, as its complexity is at least linear in $q$ (i.e., exponential in $\log{q}$). The aim of our algorithms is to obtain an understanding of the structure of $\Gamma_f$ without needing to draw it vertex by vertex. A detailed complexity analysis of those algorithms, which we carry out in Section \ref{sec5}, shows that for asymptotically almost every finite field and fixed index $d$, our algorithms have implementations with polynomial complexity (in $\log{q}$) on quantum computers, and implementations with subexponential complexity on classical computers. In the remainder of this introduction, we discuss the main ideas underlying our algorithms. We also note that we revisit the example (\ref{concreteFIntroEq}) in Subsection \ref{subsec4P2}, where we derive the structure of its functional graph with our methods.

The first step in understanding the functional graph $\Gamma_g$ of any function $g:X\rightarrow X$ is to obtain a suitable parametrization of the connected components of $\Gamma_g$. The following notion is helpful in that regard.

\begin{definition}\label{crlListDef}
Let $X$ be a finite set, and $g:X\rightarrow X$. A \emph{cycle representatives and lengths list} (or \emph{CRL-list}\phantomsection\label{term18} for short) \emph{of $g$} is a (finite) set $\Lcal\subseteq X\times\IN^+$\phantomsection\label{not25} with the following properties:
\begin{enumerate}
\item The first entries of the ordered pairs in $\Lcal$ form a system of representatives for the cycles of $g$ on its periodic points.
\item If $(r,l)\in\Lcal$\phantomsection\label{not26}\phantomsection\label{not27}, then $l$ is the cycle length of $r$ under $g$.
\end{enumerate}
\end{definition}

\begin{remark}\label{crlListRem}
When $g$ is a function on a finite set, it is easy to determine the cycle type of the restriction $g_{\mid\per(g)}$\phantomsection\label{not28} from any CRL-list $\Lcal$ of $g$. Namely,
\[
\CT(g_{\mid\per(g)})=\prod_{(r,l)\in\Lcal}{x_l}.
\]
A CRL-list of $g$ can thus be seen as a more refined piece of information than $\CT(g_{\mid\per(g)})$.
\end{remark}

We recall from above that each connected component of $\Gamma_g$ contains precisely one cycle of $g$ on its periodic points. This means that a CRL-list of $g$ also gives a parametrization of the connected components of $g$ via representative vertices, along with the basic information how long the cycle of each representative is. Let us give some more details on how to obtain $\Lcal$ when $X=\IF_q$ and $g$ is an index $d$ generalized cyclotomic mapping $f$ of $\IF_q$.

We already introduced the notation $C_i=\omega^iC$ for $i\in\{0,1,\ldots,d-1\}$ to denote the cosets of $C$ in $\IF_q^{\ast}$. Let us additionally set $C_d:=\{0_{\IF_q}\}$. Then the sets $C_i$ for $i=0,1,\ldots,d$ form a partition of $\IF_q$ that is preserved by $f$ in the sense that $f$ maps blocks of this partition to other such blocks (not necessarily surjectively). In other words, there is a unique function $\overline{f}:\{0,1,\ldots,d\}\rightarrow\{0,1,\ldots,d\}$\phantomsection\label{not29}, which we call \emph{induced by $f$}\phantomsection\label{term19}, such that $f(C_i)\subseteq C_{\overline{f}(i)}$ for each $i=0,1,\ldots,d$. We note in particular that $\overline{f}(d)=d$, and that $\overline{f}^{-1}(\{d\})=\{d\}$ unless at least one of the coefficients $a_i$ in (\ref{cyclotomicFormEq}) is $0$.

Setting $s:=(q-1)/d=|C|$\phantomsection\label{not30}, we may view each coset $C_i=\omega^iC$ for $i=0,1,\ldots,d-1$ as a copy of the cyclic group $\IZ/s\IZ$ (with underlying set $\{0,1,\ldots,s-1\}$ and modular addition as its group operation) via the bijection $\iota_i:\IZ/s\IZ\rightarrow C_i, x\mapsto\omega^{i+dx}$\phantomsection\label{not31}. As such, $f$ may be viewed as a function that maps between copies of $\IZ/s\IZ$ (as well as a unique singleton block). More specifically, if $i\in\{0,1,\ldots,d-1\}$ and $a_i\not=0$, and if we write $a_i=\omega^{e_i}$\phantomsection\label{not32}, then we have $d\mid e_i+r_ii-\overline{f}(i)$ necessarily, and $f$ maps $\omega^{i+dx}\in C_i$ to
\[
a_i(\omega^{i+dx})^{r_i}=\omega^{e_i+r_ii+r_idx}=\omega^{\overline{f}(i)+d\cdot\left(\frac{e_i+r_ii-\overline{f}(i)}{d}+r_ix\right)}\in C_{\overline{f}(i)}.
\]
This means that under the identifications of $C_i$ and $C_{\overline{f}(i)}$ with $\IZ/s\IZ$, the restriction of $f$ to $C_i$ corresponds to the affine function $A_i:x\mapsto r_ix+\frac{e_i+r_ii-\overline{f}(i)}{d}$\phantomsection\label{not33} of $\IZ/s\IZ$.

In summary, $f$ consists essentially of affine functions mapping between copies of $\IZ/s\IZ$, though some copies of $\IZ/s\IZ$ may also be constantly mapped into $C_d=\{0\}$ by $f$, in case the corresponding coefficient $a_i$ is $0$. We note that if all $a_i$ are nonzero, then the way $f_{\mid\IF_q^{\ast}}$ preserves the partition of $\IF_q^{\ast}$ into the cosets $C_i$ for $i\in\{0,1,\ldots,d-1\}$ is analogous to the way the elements of the imprimitive permutational wreath product\phantomsection\label{term20} $\Sym(C)\wr\Sym(d)$\phantomsection\label{not34} (where $\Sym(X)$\phantomsection\label{not35} and $\Sym(n)$\phantomsection\label{not36} denote the symmetric group on the set $X$ and on $\{0,1,\ldots,n-1\}$, respectively) preserve this partition. In fact, the definition of an imprimitive permutational wreath product naturally extends to one of an \emph{imprimitive wreath product of transformation semigroups} such that $f_{\mid\IF_q^{\ast}}$ is an element of the imprimitive wreath product of $C^C$\phantomsection\label{not37} (the transformation semigroup of all functions $C\rightarrow C$) with $\{0,1,\ldots,d-1\}^{\{0,1,\ldots,d-1\}}$, and $\overline{f}_{\mid\{0,1,\ldots,d-1\}}$ is the projection of $f$ to $\{0,1,\ldots,d-1\}^{\{0,1,\ldots,d-1\}}$. Wreath products of transformation semigroups have been studied before and play a central role in algebraic automata theory, though the notion used in that theory is the natural generalization of \emph{primitive} permutational wreath products \cite[pp.~55f.]{Hol82a}.

In any case, these ideas allow us to easily reduce the determination of a CRL-list $\Lcal$ of $f$ to the determination of CRL-lists of affine functions on $\IZ/s\IZ$ -- see Subsection \ref{subsec3P1} for the details of this. CRL-lists of affine functions of finite cyclic groups are determined in Subsection \ref{subsec2P3}.

The remainder of our algorithmic approach is concerned with understanding, for each given $(r,l)\in\Lcal$, the isomorphism type of the connected component of $\Gamma_f$ containing $r$. We recall from above that this connected component is essentially obtained by glueing certain directed rooted trees to the vertices on the cycle. Let us introduce the following precise notation.

\begin{definition}\label{treeAboveDef}
Let $\Gamma$ be a finite functional graph with vertex set $X$, and let $g$ be the unique function $X\rightarrow X$ such that $\Gamma=\Gamma_g$. For each $x\in X$, we define $\Tree_{\Gamma}(x)$\phantomsection\label{not38}, the so-called \emph{tree above $x$ in $\Gamma$}\phantomsection\label{term21}, as follows.
\begin{enumerate}
\item If $x$ is $g$-transient, we define $\Tree_{\Gamma}(x)$ as the induced subgraph of $\Gamma$ on the set
\[
\{x\}\cup\{y\in \V(\Gamma): y\not=x\text{, and }g^k(y)=x\text{ for some }k=k(y)\geq1\}.
\]
\phantomsection\label{not38P5}
\phantomsection\label{not39}
\item If $x$ is $g$-periodic, we define $\Tree_{\Gamma}(x)$ as the induced subgraph of $\Gamma$ on the set
\[
\{x\}\cup\bigcup\{\V(\Tree_{\Gamma}(y)): y\text{ is }g\text{-transient and }g(y)=x\},
\]
with the convention that in case $g(x)=x$, the loop at $x$ is deleted from $\Tree_{\Gamma}(x)$.
\end{enumerate}
\end{definition}

With this definition, $\Tree_{\Gamma_g}(x)$ is defined for all $x\in X=\V(\Gamma_g)$, and for periodic vertices $x$, those are the trees that need to be glued to the cycles of $g$ in order to obtain the full connected components of $\Gamma_g$.

Necklaces are a well-studied concept in combinatorics. Let us consider vertex-labeled, directed graphs that consist of a single, directed cycle (let us call such a graph a \emph{necklace graph}\phantomsection\label{term22}). Intuitively, one may think of the vertices as beads on a necklace (in the common sense of the word), and the vertex labels represent colors of those beads. An \emph{isomorphism of vertex-labeled digraphs}\phantomsection\label{term23} is a digraph isomorphism preserving vertex labels, and a \emph{necklace}\phantomsection\label{term24} is an isomorphism class of necklace graphs under isomorphism of vertex-labeled digraphs. If $\vec{\xfrak}=(\xfrak_0,\xfrak_1,\ldots,\xfrak_{L-1})$\phantomsection\label{not40} is a length $L$ sequence with entries from a set $\Xfrak$\phantomsection\label{not41}, then we denote by $[\vec{\xfrak}]=[\xfrak_0,\xfrak_1,\ldots,\xfrak_{L-1}]$\phantomsection\label{not42} the orbit of $\vec{\xfrak}$ under the natural action of the cyclic group $\IZ/L\IZ$ on $\Xfrak^L$. Hence, $[\vec{\xfrak}]$ consists of those length $L$ sequences over $\Xfrak$ that can be obtained from $\vec{\xfrak}$ through cyclic shifts. We also call $[\vec{\xfrak}]$ the \emph{cyclic sequence associated with $\vec{\xfrak}$}\phantomsection\label{term25}. We observe that two necklace graphs are isomorphic if and only if their sequences of vertex labels are cyclically equivalent, whence in combinatorics, a necklace is often simply defined as a cyclic sequence\phantomsection\label{term26} (cyclic equivalence class of strings).

The connected components of a functional graph $\Gamma_g$ of a function $g:X\rightarrow X$, where $X$ is a finite set, may be viewed as necklace graphs. Indeed, we take the unique directed cycle contained in a given connected component as the underlying digraph of the associated necklace graph. The label of a (representative) vertex $r$ on that cycle is defined as the rooted tree isomorphism type of $\Tree_{\Gamma_g}(r)$. For example, if we denote by
\begin{itemize}
\item $\Ifrak_0$ the digraph isomorphism type of the trivial rooted tree (consisting of a single vertex without arcs);
\item $\Ifrak_1$ most common rooted tree isomorphism type above a periodic vertex in the functional graph of the exemplary generalized cyclotomic mapping $f$ of $\IF_{256}$ defined in (\ref{concreteFIntroEq}) above (i.e., a rooted tree of height $2$ where the root has in-degree $2$ and one of the two neighbors of the root has in-degree $0$, the other has in-degree $3$);
\item $\Ifrak_2$ the digraph isomorphism type of $\Delta_3$ in the above example (the seemingly chaotic numbering for the $\Ifrak_j$ is chosen such that it matches with Table \ref{fullTreesConcreteTable} in Subsection \ref{subsec4P2});
\item $\Ifrak_3$ the digraph isomorphism type of $\Delta_1$;
\item $\Ifrak_4$ the digraph isomorphism type of $\Delta_2$;
\end{itemize}
then the four connected components of the example above may be identified with necklace graphs corresponding to the following cyclic sequences of rooted tree isomorphism types (in order of drawing).
\begin{itemize}
\item $[\Ifrak_0]$;
\item $[\Ifrak_1]$;
\item $[\Ifrak_1,\Ifrak_1,\Ifrak_1,\Ifrak_1,\Ifrak_1,\Ifrak_1,\Ifrak_1,\Ifrak_1]$;
\item $[\Ifrak_3,\Ifrak_1,\Ifrak_4,\Ifrak_1,\Ifrak_1,\Ifrak_1,\Ifrak_1,\Ifrak_2]$.
\end{itemize}

With the above convention of identifying connected components of functional graphs with certain necklace graphs, two digraphs that are connected components of finite functional graphs are isomorphic as digraphs if and only if they are isomorphic as necklace graphs (i.e., they represent the same necklace of rooted tree isomorphism types). This means that in order to understand the connected components of $\Gamma_g$, we need to understand the associated cyclic sequences of rooted tree isomorphism types.

We note that if the goal is to understand the (undirected graph) isomorphism type of the underlying undirected graph of a connected component of a functional graph, an analogous approach can be used. One needs to replace necklace graphs by \emph{bracelet graphs}\phantomsection\label{term27} (undirected, vertex-labeled cycle graphs), cyclic sequences by \emph{dihedral sequences}\phantomsection\label{term28} (orbits of the natural action of the dihedral group of order $2L$ on $\Xfrak^L$, where the generating reflection acts by writing the sequence in reverse order), and necklaces by \emph{bracelets}\phantomsection\label{term29} (isomorphism classes of bracelet graphs).

Let us next explain our approach for understanding the digraph isomorphism types of the connected components of $\Gamma_g$ via necklaces of rooted tree isomorphism types in case $X=\IF_q$ and $g$ is an index $d$ generalized cyclotomic mapping $f$ of $\IF_q$. For this, we first need to understand the rooted tree above a given (periodic) point. The basic idea is to construct a certain partition $\Pcal_i$\phantomsection\label{not43} of each coset $C_i$ that \enquote{controls} the isomorphism types of rooted trees above the vertices in each of its blocks. Dealing with entire blocks of vertices at once is crucial to ensure that the complexities of our algorithms are not at least linear in the number of vertices $q$ like many general-purpose algorithms for handling graph isomorphism, including Babai's breakthrough quasi-polynomial algorithm from \cite{Bab18a}.

In order to sketch how the said partition $\Pcal_i$ of $C_i$ is constructed, we need to introduce some more concepts. For a given positive integer $m$, we define the notion of an \emph{$m$-(in)congruence}\phantomsection\label{term30} to be an (in)congruence of the form $\nu(x\equiv \bfrak\Mod{\afrak})$ where $\nu\in\{\emptyset,\neg\}$\phantomsection\label{not44} is a \enquote{logical sign}, and $\afrak,\bfrak$\phantomsection\label{not45}\phantomsection\label{not46} are integers with $\afrak\geq1$ and $\afrak\mid m$. We subsume these two notions under the name \emph{$m$-congruential condition}, or \emph{$m$-CC}\phantomsection\label{term31} for short. Next, we consider the concept of an \emph{arithmetic partition of $\IZ/m\IZ$}, see point (2) of the following definition.

\begin{definition}\label{sArithDef}
Let $m$ be a positive integer. We identify the elements of $\IZ/m\IZ$ with their standard representatives in $\{0,1,\ldots,m-1\}$.
\begin{enumerate}
\item Let $x\equiv \bfrak_j\Mod{\afrak_j}$ for $j=1,2,\ldots,K$ be $m$-congruences. There is a unique partition of $\IZ/m\IZ$, which we denote by $\Pfrak(x\equiv \bfrak_j\Mod{\afrak_j}: j=1,2,\ldots,K)$\phantomsection\label{not47}, such that each block of this partition is the solution set modulo $m$ of a system of $m$-CCs of the form
\begin{align}\label{addedSignsEq}
\notag \nu_1(x &\equiv \bfrak_1\Mod{\afrak_1}) \\
\notag \nu_2(x &\equiv \bfrak_2\Mod{\afrak_2}) \\
\notag &\vdots \\
\nu_K(x &\equiv \bfrak_K\Mod{\afrak_K})
\end{align}
where $\nu_1,\nu_2,\ldots,\nu_K\in\{\emptyset,\neg\}$ are logical signs.
\item A partition $\Pcal$\phantomsection\label{not48} of $\IZ/m\IZ$ is called \emph{arithmetic}\phantomsection\label{term32} if it is of the form
\[
\Pfrak(x\equiv \bfrak_j\Mod{\afrak_j}: j=1,2,\ldots,K)
\]
for a suitable nonnegative integer $K$ and suitable $m$-congruences $x\equiv \bfrak_j\Mod{\afrak_j}$ for $j=1,2,\ldots,K$. If $\Pcal=\Pfrak(x\equiv \bfrak_j\Mod{\afrak_j}: j=1,2,\ldots,K)$, then we also say that $\Pcal$ is the arithmetic partition of $\IZ/m\IZ$ \emph{spanned by}\phantomsection\label{term33} the congruences $x\equiv \bfrak_j\Mod{\afrak_j}$ for $j=1,2,\ldots,K$.
\item When $\Pcal$ is an arithmetic partition of $\IZ/m\IZ$, then the smallest value of $K\in\IN_0$ such that $\Pcal$ is spanned by $K$ suitably chosen $m$-congruences is called the \emph{(arithmetic) complexity of $\Pcal$}\phantomsection\label{term34}, written $\AC(\Pcal)$\phantomsection\label{not49}.
\item When $\Pcal$ is an arithmetic partition of $\IZ/m\IZ$ and a sequence of spanning congruences $(x\equiv \bfrak_j\Mod{\afrak_j})_{j=1,2,\ldots,K}$ has been fixed for $\Pcal$, we also denote for each $\vec{\nu}=(\nu_1,\nu_2,\ldots,\nu_K)\in\{\emptyset,\neg\}^K$ by $\Bcal(\Pcal,\vec{\nu})$\phantomsection\label{not50} the unique subset of $\IZ/m\IZ$ that is the solution set of the system (\ref{addedSignsEq}) (this solution set is a block of $\Pcal$ as long as it is non-empty).
\end{enumerate}
\end{definition}

\begin{remark}\label{sArithRem}
There are significantly fewer arithmetic partitions of $\IZ/m\IZ$ than there are partitions in total. Indeed, the total number of (set) partitions of $\IZ/m\IZ$ is the Bell number $B_m$\phantomsection\label{not51}, which satisfies
\[
B_m\sim\frac{1}{\sqrt{m}}\left(\frac{m}{W(m)}\right)^{m+\frac{1}{2}}\exp\left(\frac{m}{W(m)}-m-1\right)
\]
as $m\to\infty$, where $W(m)\sim\log{m}$\phantomsection\label{not52} is the Lambert W function (see \cite[Section 1.14, Problem 9]{Lov93a}). In particular, as $m\to\infty$,
\[
\log{B_m} \sim -\frac{1}{2}\log{m}+\left(m+\frac{1}{2}\right)(\log{m}-\log{W(m)})+\frac{m}{W(m)}-m-1 \sim m\log{m}.
\]
On the other hand, every arithmetic partition of $\IZ/m\IZ$ is spanned by a selection of congruences of the form $x\equiv \bfrak\Mod{\afrak}$ where $\afrak$ ranges over the positive divisors of $m$, and $\bfrak\in\{0,1,\ldots,\afrak-1\}$. Because the total number of such congruences is $\sigma(m)$\phantomsection\label{not53} (the sum of all positive divisors of $m$), it follows that the number of arithmetic partitions of $\IZ/m\IZ$ is at most $2^{\sigma(m)}$, and so its natural logarithm is at most
\[
\log{2}\cdot\sigma(m)\leq\log{2}(\e^{\gamma}+\epsilon)m\log\log{m},
\]
where $\gamma$\phantomsection\label{not54} denotes the Euler-Mascheroni constant and the second bound follows from a result of Robin \cite{Rob84a}.
\end{remark}

Let us return to our index $d$ generalized cyclotomic mapping $f$ of $\IF_q$. We recall that $s=(q-1)/d$ denotes the order (size) of the index $d$ subgroup $C$ of $\IF_q^{\ast}$. The aforementioned partitions $\Pcal_i$ of the cosets $C_i$ are constructed as arithmetic partitions of $\IZ/s\IZ$, with which $C_i$ is to be identified via the bijection $\iota_i$ introduced above. They have the property that for vertices $x,y\in C_i$ chosen from a common block $\Bcal(\Pcal_i,\vec{\nu})$ of $\Pcal_i$, one has $\Tree_{\Gamma_f}(x)\cong\Tree_{\Gamma_f}(y)$, and this common isomorphism type is denoted by $\Tree_i(\Pcal_i,\vec{\nu})$\phantomsection\label{not55}. The constructions of the partitions $\Pcal_i$ and of the associated rooted tree isomorphism types $\Tree_i(\Pcal_i,\vec{\nu})$, which are carried out in detail in Subsection \ref{subsec3P3}, are based on two crucial tools:
\begin{itemize}
\item the elementary result Lemma \ref{masterLem} from Subsection \ref{subsec2P2}; and
\item an explicit understanding, developed in Subsection \ref{subsec3P2} but also based on earlier theory developed in Subsection \ref{subsec2P1}, of the structures of rooted trees in the induced subgraph $\Gamma_{\per}$\phantomsection\label{not56} of $\Gamma_f$ on the union of all cosets $C_i$ where $i$ is $\overline{f}$-periodic.
\end{itemize}

Once the $\Pcal_i$ and $\Tree_i(\Pcal_i,\vec{\nu})$ have been constructed explicitly, in order to understand the isomorphism type of the connected component of $\Gamma_f$ with representative periodic vertex $r$, one needs to understand how the cycle moves through the various blocks of the respective coset partitions. Of course, if the cycle length $l$ of $r$ under $f$ is small, one can just enumerate the points on the cycle by brute force, check in which blocks they lie and spell out the corresponding cyclic sequence of rooted trees; this is what we do at the end of the example in Subsection \ref{subsec4P2}. However, if $l$ is large, then one can obtain a more concise description of the cyclic rooted tree sequence via a certain tuple of arithmetic partitions, the blocks of which represent intersections of the cycle of $r$ with blocks of the involved arithmetic partitions $\Pcal_i$. For details on this, see Subsection \ref{subsec3P4}, which builds on Subsection \ref{subsec2P4}.

Here is an overview of our approach for understanding $\Gamma_f$.

\begin{enumerate}
\item Determine the induced function $\overline{f}:\{0,1,\ldots,d\}\rightarrow\{0,1,\ldots,d\}$, and rewrite $f$ into a collection of affine functions that map between $d+1$ sets $C_i$, each of the form $\IZ/s\IZ$ or $\{0\}$.
\item Compute a CRL-list $\Lcal$ for $f$ as specified in Subsection \ref{subsec3P1}, which is based on the results for affine maps of finite cyclic groups from Subsection \ref{subsec2P3}.
\item For each $i\in\{0,1,\ldots,d-1\}$, compute the arithmetic partition $\Pcal_i$ and associated rooted tree isomorphism types $\Tree_i(\Pcal_i,\vec{\nu})$, as well as the isomorphism type of $\Tree_{\Gamma_f}(0_{\IF_q})$, as specified in Subsection \ref{subsec3P3}. This requires the theory developed in Subsections \ref{subsec2P1} and \ref{subsec3P2}.
\item For each $(r,l)\in\Lcal$, understand the associated cyclic sequence of rooted tree isomorphism types along the cycle of $r$ under $f$, either
\begin{itemize}
\item by listing elements on the cycle of $r$ by brute force, then looking up in which blocks of the relevant arithmetic partitions they lie, or
\item by following the approach from Subsection \ref{subsec3P4}, which relies on Subsection \ref{subsec2P4}.
\end{itemize}
\end{enumerate}

In Section \ref{sec5}, where we give a detailed algorithmic complexity analysis, we treat the procedures described in Steps (2)--(4) each as a separate algorithm to be analyzed. We note that in general, $f$ has too many cycles in order for it to be possible to spell out a CRL-list of $f$ element-wise if the procedure is to be efficient (i.e., subexponential in $\log{q}$); one can, however, obtain a concise parametrization of a CRL-list of $f$ efficiently. Likewise, the approach described in point (4) can be carried out for each given pair $(r,l)$ individually in an efficient manner for asymptotically almost all finite fields $\IF_q$, but it is not clear in general how to obtain a \enquote{global} understanding of $\Gamma_f$ efficiently. In fact, the number of distinct isomorphism types of connected components of $\Gamma_f$ might be super-polynomial in $\log{q}$ even for fixed $d$ (cf.~Problem \ref{isomorphismTypesProb}), so one would first need to come up with a compact way of parametrizing those isomorphism types. Still, as we will see in Subsection \ref{subsec5P3}, for some special cases of generalized cyclotomic mappings $f$ of $\IF_q$, there are algorithms for describing $\Gamma_f$ as a whole which are efficient for all or at least for \enquote{most} $q$ (in an asymptotic density sense). In particular, in those cases, it can be efficiently decided whether the functional graphs of two given generalized cyclotomic mappings of $\IF_q$ are isomorphic.

Section \ref{sec6} concludes the paper with a list of open problems for further research. For the reader's convenience, an extensive index of the notation and terminology appearing in this paper is given in Tables \ref{termNotTableShort} and \ref{termNotTable} in the Appendix (at the very end of the paper).

\section{Preparations}\label{sec2}

In this section, we prove some auxiliary results that are used when discussing the details of our algorithm in Section \ref{sec3}.

\subsection{Functional graphs of affine maps of finite groups}\label{subsec2P1}

In this subsection, we derive some results on functional graphs of affine maps $A:x\mapsto ax+b$\phantomsection\label{not57}\phantomsection\label{not58}\phantomsection\label{not59} of finite cyclic groups $\IZ/m\IZ$. We note that these graphs were studied earlier by Deng \cite{Den13a}, and we use several of Deng's results and ideas here, as pointed out where appropriate. However, for the reader's convenience, we aim to keep our exposition self-contained. We also observe (at the end of the subsection) that these results can in fact be generalized to arbitrary finite groups. First, we consider the following concept.

\begin{deffinition}\label{tensorProductDef}
Let $(\Gamma_j)_{j\in I}$\phantomsection\label{not60} be a family of digraphs. Their \emph{tensor product}\phantomsection\label{term35}, written $\bigotimes_{j\in I}{\Gamma_j}$\phantomsection\label{not61}, is the digraph with vertex set $\prod_{j\in I}{\V(\Gamma_j)}$ having an arc $(y_j)_{j\in I}\rightarrow(z_j)_{j\in I}$\phantomsection\label{not62} if and only if for each $j\in I$, there is an arc $y_j\rightarrow z_j$ in $\Gamma_j$.
\end{deffinition}

This concept corresponds to Deng's product graph from \cite[formula (1) in Section 2]{Den13a}.

\begin{remmark}\label{tensorProductRem}
If $(g_j)_{j\in I}$ is a family of functions $g_j:X_j\rightarrow X_j$, and if $\bigotimes_{j\in I}{g_j}$\phantomsection\label{not63} denotes the function
\[
\prod_{j\in I}{X_j}\rightarrow\prod_{j\in I}{X_j},\quad (y_j)_{j\in I}\mapsto (g_j(y_j))_{j\in I},
\]
then
\[
\Gamma_{\bigotimes_{j\in I}{g_j}}=\bigotimes_{j\in I}{\Gamma_{g_j}}.
\]
\end{remmark}

Due to Remark \ref{tensorProductRem}, the tensor product of digraphs is a useful tool when studying functional graphs of affine maps (in particular of endomorphisms) of finite cyclic groups $\IZ/m\IZ$. Indeed, if we factor $m=p_1^{v_1}\cdots p_K^{v_K}$\phantomsection\label{not56P5}\phantomsection\label{not64}, then for each given affine map $A:x\mapsto ax+b$ of $\IZ/m\IZ$ and each $j\in\{1,2,\ldots,K\}$, we may consider the reduction $A_j$ of $A$ modulo $p_j^{v_j}$, which is the affine map $A_j:x\mapsto ax+b$ of $\IZ/p_j^{v_j}\IZ$. By the Chinese Remainder Theorem, $\IZ/m\IZ$ is in a natural isomorphism with $\prod_{j=1}^K{\IZ/p_j^{v_j}\IZ}$, and under this isomorphism, $A$ corresponds to $\bigotimes_{j=1}^K{A_j}$. Hence we obtain the following, which is \cite[Theorem 2]{Den13a}.

\begin{lemmma}\label{affineTensorProdLem}
Let $m=p_1^{v_1}\cdots p_K^{v_K}$ be a positive integer with displayed factorization into pairwise coprime prime powers. Let $A:x\mapsto ax+b$ be an affine map of $\IZ/m\IZ$, and denote by $A_j$ the reduction of $A$ modulo $p_j^{v_j}$ for $j=1,2,\ldots,K$. Then $\Gamma_A\cong\bigotimes_{j=1}^K{\Gamma_{A_j}}$.
\end{lemmma}

Lemma \ref{affineTensorProdLem} allows us to reduce many arguments concerning functional graphs of affine maps of $\IZ/m\IZ$ to the case where $m=p^v$ is a prime power. We note the following interesting dichotomy (see also \cite[Lemma 3]{Den13a}).

\begin{propposition}\label{primaryDichotomyProp}
Let $m=p^v$ be a prime power, and let $A:x\mapsto ax+b$ be an affine map of $\IZ/m\IZ$.
\begin{enumerate}
\item If $p\mid a$, then $A$ has exactly one periodic point $x$ (a fixed point necessarily), and $\Gamma_A$ is obtained from $\Tree_{\Gamma_A}(x)$ by adding a loop to the root $x$.
\item If $p\nmid a$, then $A$ is a permutation of $\IZ/m\IZ$, whence $\Gamma_A$ is a disjoint union of directed cycles.
\end{enumerate}
\end{propposition}

We would like to use these reduction ideas to prove the following theorem.

\begin{theoremm}\label{allTreesIsomorphicTheo}
Let $A:x\mapsto ax+b$ be an affine map of the finite cyclic group $\IZ/m\IZ$. Then all trees above periodic vertices in $\Gamma_A$ are isomorphic to each other. In fact, they are all isomorphic to any tree above a periodic vertex in $\Gamma_{\mu_a}$, where $\mu_a$\phantomsection\label{not65} is the endomorphism $x\mapsto ax$ of $\IZ/m\IZ$.
\end{theoremm}

We note that it was proved by Sha \cite[Corollary 3.4]{Sha12a} that all trees above periodic vertices in $\Gamma_{\mu_a}$ are isomorphic to each other. The statement of Theorem \ref{allTreesIsomorphicTheo} itself is implicit in Deng's proof of \cite[Theorem 11]{Den13a}.

Using Proposition \ref{primaryDichotomyProp}, we can derive the following partial result swiftly, in the proof of which we use the notation
\[
\nu_p^{(v)}(n):=\min\{v,\nu_p(n)\}\phantomsection\label{not66}
\]
where $\nu_p(n)$\phantomsection\label{not67} is the \emph{$p$-adic valuation of $n$}, i.e., the exponent of $p$ in the prime power factorization of the integer $n$, defined to be $\infty$ if $n=0$.

\begin{lemmma}\label{allTreesIsomorphicLem}
Theorem \ref{allTreesIsomorphicTheo} holds when $m=p^v$ is a prime power.
\end{lemmma}

\begin{proof}
By \cite[Lemma 4]{Den13a}, if $\nu_p^{(v)}(a-1)\leq\nu_p^{(v)}(b)$, then $A$ has a fixed point, which leads to a digraph isomorphism between $\Gamma_A$ and $\Gamma_{\mu_a}$ (cf.~also our Lemma \ref{affineFixedPointLem}). The result is thus clear by \cite[Corollary 3.4]{Sha12a}.

On the other hand, if $\nu_p^{(v)}(b)<\nu_p^{(v)}(a-1)$, then $a\equiv1\Mod{p}$, which implies that $p\nmid a$. Therefore, by Proposition \ref{primaryDichotomyProp}(2), all rooted trees above periodic vertices in $\Gamma_A$ are trivial (i.e., they are isomorphic to a single vertex without arcs), and so are those trees in $\Gamma_{\mu_a}$, as required.
\end{proof}

Of course, we could now derive Theorem \ref{allTreesIsomorphicTheo} in its full strength by observing that the property \enquote{all rooted trees above periodic vertices are isomorphic} is preserved under taking tensor products of functional graphs. We can, however, obtain an even more detailed result with some extra work, which we carry out. This requires some concepts and results from the first author's paper \cite{Bor17a}, in which the structure of the trees above periodic vertices in functional graphs of finite group endomorphisms was characterized (thus extending Sha's result \cite[Corollary 3.4]{Sha12a}).

\begin{deffinition}\label{childProcDef}
Let $\Gamma=(V,E)$\phantomsection\label{not68}\phantomsection\label{not69} be a finite digraph, and let $x\in V$.
\begin{enumerate}
\item The \emph{dual digraph}\phantomsection\label{term36} $\Gamma^{\ast}$\phantomsection\label{not70} is obtained from $\Gamma$ by inverting each arc; formally, $\Gamma^{\ast}=(V,E^{-1})$ where $E^{-1}=\{(y,z)\in V^2: (z,y)\in E\}$\phantomsection\label{not71} is the inverse relation of $E$.
\item A vertex $y\in V$ such that $\Gamma$ has an edge $x\rightarrow y$ is a \emph{successor} or \emph{child of $x$}\phantomsection\label{term37}.
\item For each positive integer $k$, the \emph{$k$-th procreation number of $x$}\phantomsection\label{term38}, written $\proc_k(x)=\proc^{(\Gamma)}_k(x)$\phantomsection\label{not72}, is the number of childen $y$ of $x$ such that there is a length $k-1$ directed path $(w_1,w_2,\ldots,w_k)$ in $\Gamma$ with $w_1=y$ (in this situation, we also say that \emph{$y$ has (at least) $k-1$ successor generations}\phantomsection\label{term39}).
\item We say that $\Gamma$ has \emph{rigid procreation}\phantomsection\label{term40} if for all $y,z\in V$ and all positive integers $k$ with $\proc_k(y),\proc_k(z)>0$, one has $\proc_k(y)=\proc_k(z)$.
\end{enumerate}
\end{deffinition}

Using the notation of Definition \ref{childProcDef}, we note that $\proc_1(x)$ is simply the number of all childen of $x$ (i.e., the out-degree of $x$), that $\proc_2(x)$ is the number of childen of $x$ that have childen themselves, etc. Rigid procreation means that all vertices with children must have the same number of children (though it is fine for vertices without children to exist), that all vertices with at least one \enquote{grandchild} must have the same number of children that have a child (in particular the same number of grandchildren), etc. The following fact was noted but not proved in \cite[Remark before Theorem 3]{Bor17a}, and we prove it here (after the proof of Lemma \ref{rigidLem}) for the reader's convenience.

\begin{propposition}\label{rigidProp}
Let $\Gamma$ be a finite digraph that is a functional graph (i.e., all vertices of $\Gamma$ have out-degree $1$), say $\Gamma=\Gamma_g$. If the dual digraph $\Gamma^{\ast}$ has rigid procreation, then for any two $g$-periodic vertices $x$ and $y$, we have $\Tree_{\Gamma}(x)\cong\Tree_{\Gamma}(y)$. Moreover, the common rooted tree isomorphism type above periodic vertices $x$ in $\Gamma$ is determined by the procreation number sequence $(\proc_k(x))_{k\geq1}$ alone (i.e., it is the same in any finite functional graph with rigid procreation and the same procreation number sequence of periodic vertices).
\end{propposition}

Readers interested in how the isomorphism type of $\Tree_{\Gamma}(x)$ can be derived from $(\proc_k(x))_{k\geq1}$ for periodic vertices $x$ if $\Gamma^{\ast}$ has rigid procreation can find the details of this in Subsection \ref{subsec4P1}. Before proving Proposition \ref{rigidProp}, we extend the notation $\Tree_{\Gamma}(x)$, which was already defined for finite functional graphs $\Gamma$ in Section \ref{sec1}, to the case where $\Gamma$ is a finite directed rooted tree $\Delta$ (with all arcs oriented toward the root) and prove a lemma.

\begin{deffinition}\label{treeAboveDef2}
Let $\Delta$\phantomsection\label{not73} be a finite directed rooted tree, with root $\rt(\Delta)$\phantomsection\label{not74}. We observe that all vertices except $\rt(\Delta)$ have out-degree $1$, and we let $g$ be the unique function $\V(\Delta)\setminus\{\rt(\Delta)\}\rightarrow \V(\Delta)$ such that $\Delta$ has an arc $x\rightarrow g(x)$ for each vertex $x\not=\rt(\Delta)$. For each $x\in\V(\Delta)$, we define $\Tree_{\Delta}(x)$, the so-called \emph{tree above $x$ in $\Delta$}, as the induced subgraph of $\Delta$ on the set
\[
\{x\}\cup\{y\in\V(\Delta): y\not=x\text{, and }g^k(y)=x\text{ for some }k=k(y)\geq1\}.
\]
\end{deffinition}

In the statement of the following lemma and beyond, we denote the height of a finite directed rooted tree $\Delta$ by $\height(\Delta)$\phantomsection\label{not75}.

\begin{lemmma}\label{rigidLem}
Let $\Delta_1$ and $\Delta_2$ be finite directed rooted trees. Moreover, we assume that $\Delta_1$ and $\Delta_2$ have the same height\phantomsection\label{term41}, that the dual digraphs $\Delta_1^{\ast}$ and $\Delta_2^{\ast}$ both have rigid procreation, and that $\proc_h^{(\Delta_1^{\ast})}(\rt(\Delta_1))=\proc_h^{(\Delta_2^{\ast})}(\rt(\Delta_2))$ for $1\leq h\leq\height(\Delta_1)=\height(\Delta_2)$. Then $\Delta_1$ and $\Delta_2$ are isomorphic.
\end{lemmma}

\begin{proof}
We proceed by induction on the common height of $\Delta_1$ and $\Delta_2$. If $\height(\Delta_1)=0$, then both $\Delta_1$ and $\Delta_2$ consist of a single vertex without any arcs and thus are isomorphic. Now we assume that $\height(\Delta_1)\geq1$ and that the statement holds for all smaller heights. We note that $\Delta_1$ and $\Delta_2$ are isomorphic if and only if the following equality of multisets holds, where $[\Gamma]_{\cong}$\phantomsection\label{not76} denotes the isomorphism type of the finite digraph $\Gamma$:
\begin{align}\label{multisetEq}
\notag &\{[\Tree_{\Delta_1}(y_1)]_{\cong}: y_1\text{ is a child of }\rt(\Delta_1)\text{ in }\Delta_1^{\ast}\}= \\
&\{[\Tree_{\Delta_2}(y_2)]_{\cong}: y_2\text{ is a child of }\rt(\Delta_2)\text{ in }\Delta_2^{\ast}\}.
\end{align}
It is thus our goal to prove equality (\ref{multisetEq}). Let us fix $h\in\{0,1,2,\ldots,\height(\Delta_1)-1\}$. The number of children $y_1$ of $\rt(\Delta_1)$ in $\Delta_1^{\ast}$ such that $\Tree_{\Delta_1}(y_1)$ has height exactly $h$ is
\[
\proc^{(\Delta_1^{\ast})}_{h+1}(\rt(\Delta_1))-\proc^{(\Delta_1^{\ast})}_{h+2}(\rt(\Delta_1)).
\]
As for children $y_2$ of $\rt(\Delta_2)$ in $\Delta_2^{\ast}$ such that $\Tree_{\Delta_2}(y_2)$ has height exactly $h$, one finds analogously that their number is
\[
\proc^{(\Delta_2^{\ast})}_{h+1}(\rt(\Delta_2))-\proc^{(\Delta_2^{\ast})}_{h+2}(\rt(\Delta_2))=\proc^{(\Delta_1^{\ast})}_{h+1}(\rt(\Delta_1))-\proc^{(\Delta_1^{\ast})}_{h+2}(\rt(\Delta_1)).
\]
Moreover, for each child $y_1$ of $\rt(\Delta_1)$ in $\Delta_1^{\ast}$ such that $\Tree_{\Delta_1}(y_1)$ has height exactly $h$, the first $h$ procreation numbers of $y_1$ in $(\Tree_{\Delta_1}(y_1))^{\ast}$ are the same as those of $\rt(\Delta_1)$ in $\Delta_1^{\ast}$, since $\Delta_1^{\ast}$ has rigid procreation. In particular, by the induction hypothesis and for fixed $h$, all digraphs $\Tree_{\Delta_1}(y_1)$ where $y_1$ is a child of $\rt(\Delta_1)$ in $\Delta_1^{\ast}$ with exactly $h$ successor generations in $\Delta_1^{\ast}$ are isomorphic, their isomorphism type $\Ifrak^{(1)}_h$\phantomsection\label{not77} being determined by the first $h$ procreation numbers of $\rt(\Delta_1)$ in $\Delta_1^{\ast}$. An analogous statement holds with $\Delta_2$ and $\rt(\Delta_2)$ in place of $\Delta_1$ and $\rt(\Delta_1)$, say with isomorphism type $\Ifrak^{(2)}_h$. But by assumption, the procreation numbers of $\rt(\Delta_1)$ in $\Delta_1^{\ast}$ and of $\rt(\Delta_2)$ in $\Delta_2^{\ast}$ are the same, whence $\Ifrak^{(1)}_h=\Ifrak^{(2)}_h$ for each $h$. This shows that the two multisets in formula (\ref{multisetEq}) are the same, each consisting of exactly $\proc^{(\Delta_1^{\ast})}_{h+1}(\rt(\Delta_1))-\proc^{(\Delta_1^{\ast})}_{h+2}(\rt(\Delta_1))$ copies of $\Ifrak^{(1)}_h$ for each $h=0,1,\ldots,\height(\Delta_1)-1$.
\end{proof}

\begin{proof}[Proof of Proposition \ref{rigidProp}]
First of all, in order for the assertion to make sense, we observe that the sequence $(\proc^{(\Gamma^{\ast})}_k(x))_{k\geq1}$ does not depend on the choice of periodic vertex $x$. Indeed, if $y$ is another periodic vertex, then it is possible to form arbitrarily long directed paths in $\Gamma^{\ast}$ starting at either of $x$ or $y$ by going along the respective cycle. This implies that $\proc^{(\Gamma^{\ast})}_k(x),\proc^{(\Gamma^{\ast})}_k(y)>0$ for each $k$, and thus $\proc^{(\Gamma^{\ast})}_k(x)=\proc^{(\Gamma^{\ast})}_k(y)$ since $\Gamma^{\ast}$ has rigid procreation.

Let us write $\Gamma=\Gamma_g$ for a suitably chosen function $g:\V(\Gamma)\rightarrow\V(\Gamma)$, and fix a $g$-periodic vertex $x$. We note that $\Tree_{\Gamma}(y)=\Tree_{\Tree_{\Gamma}(x)}(y)$ for each $y\in\V(\Tree_{\Gamma}(x))$. We observe that
\[
\proc^{(\Tree_{\Gamma}(x)^{\ast})}_k(x)=\proc^{(\Gamma^{\ast})}_k(x)-1
\]
for each $k\geq1$; this is because exactly one of the children of $x$ counted by $\proc^{(\Gamma^{\ast})}_k(x)$ is periodic and hence must be ignored in the procreation number in $\Tree_{\Gamma}(x)^{\ast}$. In particular, for each $h\geq0$, the number of children $y$ of $x$ in $\Tree_{\Gamma}(x)^{\ast}$ such that $\Tree_{\Gamma}(y)$ has height exactly $h$ is
\[
\proc^{(\Tree_{\Gamma}(x)^{\ast})}_{h+1}(x)-\proc^{(\Tree_{\Gamma}(x)^{\ast})}_{h+2}(x)=\proc^{(\Gamma^{\ast})}_{h+1}(x)-\proc^{(\Gamma^{\ast})}_{h+2}(x).
\]
Moreover, for each child $y$ of $x$ in $\Tree_{\Gamma}(x)^{\ast}$, the fact that $\Gamma^{\ast}$ has rigid procreation implies that $\Tree_{\Gamma}(y)^{\ast}$ has rigid procreation, and, more specifically, whenever $\proc^{(\Tree_{\Gamma}(y)^{\ast})}_k(z)>0$ for some $z\in \V(\Tree_{\Gamma}(y))$, one has $\proc^{(\Tree_{\Gamma}(y)^{\ast})}_k(z)=\proc^{\Gamma^{\ast}}_k(z)=\proc^{\Gamma^{\ast}}_k(x)$. Lemma \ref{rigidLem} thus implies that the multiset of isomorphism types
\[
\{[\Tree_{\Gamma}(y)]_{\cong}: g(y)=x, y\text{ is }g\text{-transient}\},
\]
which determines the isomorphism type of $\Tree_{\Gamma}(x)$, is in turn entirely determined by the procreation number sequence ($\proc^{(\Gamma^{\ast})}_k(x))_{k\geq1}$, which is what we needed to prove.
\end{proof}

In view of Proposition \ref{rigidProp}, the following result, which is \cite[Theorem 2]{Bor17a}, both implies that the rooted trees above periodic vertices in $\Gamma_{\mu_a}$, the functional graph of the endomorphism $x\mapsto ax$ of $\IZ/m\IZ$, are pairwise isomorphic, and characterizes the corresponding rooted tree isomorphism type.

\begin{theoremm}\label{allTreesIsomorphicKnownTheo}
Let $m$ be a positive integer, and let $\mu_a:x\mapsto ax$, be an endomorphism of the cyclic group $\IZ/m\IZ$. The dual functional graph $(\Gamma_{\mu_a})^{\ast}$ has rigid procreation, and for each $k\in\IN^+$ and each periodic vertex $x$ of $\Gamma_{\mu_a}$, one has $\proc_k(x)=|\ker^{(k)}(\mu_a):\ker^{(k-1)}(\mu_a)|$\phantomsection\label{not77P25}, where $\ker^{(j)}(\mu_a):=\{y\in\IZ/m\IZ:(\mu_a)^j(y)=a^jy=0\}$\phantomsection\label{not77P5}.
\end{theoremm}

In view of this, Theorem \ref{allTreesIsomorphicTheo} is clear once we have proved the following result, which allows us more generally to derive the isomorphism type of $\Gamma_A$ from the isomorphism types of the functional graphs $\Gamma_{A_j}$ of the reductions of $A$ modulo the prime power factors $p_j^{v_j}$ of $m$.

\begin{propposition}\label{cyclicTensorProductProp}
Let $g_j:X_j\rightarrow X_j$ for $j=1,2$ be functions on finite sets.
\begin{enumerate}
\item We have $\per(g_1\otimes g_2)=\per(g_1)\times\per(g_2)$. In particular, if $g_1$ and $g_2$ each have precisely one periodic point, then so does $g_1\otimes g_2$.
\item We assume that the dual functional graphs $\Gamma_{g_1}^{\ast}$ and $\Gamma_{g_2}^{\ast}$ have rigid procreation. Then the dual functional graph
\[
\Gamma_{g_1\otimes g_2}^{\ast}\cong\left(\Gamma_{g_1}\otimes\Gamma_{g_2}\right)^{\ast}=\left(\Gamma_{g_1}\right)^{\ast}\otimes\left(\Gamma_{g_2}\right)^{\ast}
\]
has rigid procreation, and if $y_j$ is a periodic point of $g_j$ for $j=1,2$, then $\vec{y}=(y_1,y_2)$ is a periodic point of $g_1\otimes g_2$ and $\proc_k(\vec{y})=\proc_k(y_1)\cdot \proc_k(y_2)$.
\item Let us denote by $\divideontimes$\phantomsection\label{not78} the unique $\IQ$-bilinear product of polynomials in $\IQ[x_n: n\in\IN^+]$ such that
\[
(x_1^{e_1}\cdots x_N^{e_N})\divideontimes(x_1^{e'_1}\cdots x_{N'}^{e'_{N'}})=\prod_{1\leq n\leq N,1\leq n'\leq N'}{x_n^{e_n}\divideontimes x_{n'}^{e_{n'}}}
\]
\phantomsection\label{not79}
and
\[
x_n^e\divideontimes x_{n'}^{e'}=x_{\lcm(n,n')}^{ee'\gcd(n,n')}.
\]
We assume that each $g_j$ is a permutation of the respective set $X_j$. Then $g_1\otimes g_2$ is a permutation of $X_1\times X_2$, and its cycle type can be computed as the $\divideontimes$-product of the cycle types of $g_1$ and $g_2$.
\item We assume that $g_1$ has precisely one periodic point $y_1$ (a fixed point, necessarily) and that $g_2$ is a permutation of $X_2$. Then the induced subgraph of $g_1\otimes g_2$ on $\per(g_1\otimes g_2)$ is isomorphic to $\Gamma_{g_2}$, and for each $\vec{y}\in\per(g_1\otimes g_2)$, one has $\Tree_{\Gamma_{g_1\otimes g_2}}(\vec{y})\cong\Tree_{\Gamma_{g_1}}(y_1)$.
\end{enumerate}
\end{propposition}

\begin{proof}
Since $g_1\otimes g_2$ is the component-wise application of $g_1$ and $g_2$ on $X_1\times X_2$, it is clear that a point $(y_1,y_2)\in X_1\times X_2$ is periodic under $g_1\otimes g_2$ if and only if $y_j$ is periodic under $g_j$ for $j=1,2$, which settles statement (1) as well as the first assertion on $\vec{y}$ in statement (2).

For the rest of statement (2), we proceed as follows. In order to see that $\left(\Gamma_{g_1}\right)^{\ast}\otimes\left(\Gamma_{g_2}\right)^{\ast}$ has rigid procreation, let $k$ be a positive integer, and let $\vec{y}=(y_1,y_2)$ and $\vec{z}=(z_1,z_2)$ be points in $X_1\times X_2$ which have at least $k$ successor generations each in that graph. This is equivalent to each of $y_1,y_2,z_1,z_2$ having at least $k$ successor generations in the respective graph $\Gamma_{g_1}^{\ast}$ or $\Gamma_{g_2}^{\ast}$. It follows that $\proc_k(y_1)=\proc_k(z_1)$ and $\proc_k(y_2)=\proc_k(z_2)$. Now, for $\vec{w}=(w_1,w_2)\in X_1\times X_2$\phantomsection\label{not80}, the procreation number $\proc_k(\vec{w})$ counts the number of children $\vec{w'}=(w'_1,w'_2)$ of $\vec{w}$ in $\left(\Gamma_{g_1}\right)^{\ast}\otimes\left(\Gamma_{g_2}\right)^{\ast}$ that have at least $k-1$ successor generations. But $\vec{w'}$ has at least $k-1$ successor generations if and only if each $w'_j$ is a child of $w_j$ in $\Gamma_{g_j}^{\ast}$ that has at least $k-1$ successor generations. Therefore, $\proc_k(\vec{w})=\proc_k(w_1)\cdot\proc_k(w_2)$. In particular,
\[
\proc_k(\vec{y})=\proc_k(y_1)\proc_k(y_2)=\proc_k(z_1)\proc_k(z_2)=\proc_k(\vec{z}),
\]
as required.

For statement (3), see \cite[Theorem 2.4 and its proof]{WX93a}.

For statement (4), we observe that this is implicit in \cite[Theorem 1(4)]{Bor17a}, but since it is not proved in detail there, let us do so here. By statement (1), we know that $\vec{z}=(z_1,z_2)\in X_1\times X_2$ is a periodic point of $g_1\otimes g_2$ if and only if $z_j\in\per(g_j)$ for $j=1,2$. Hence, the periodic points of $g_1\otimes g_2$ are in bijection with those of $g_2$ via $y_2\mapsto(y_1,y_2)$, and this bijection preserves cycle lengths. Therefore, the asserted isomorphism $\Gamma_{g_2}\cong\Gamma_{(g_1\otimes g_2)_{\mid\per(g_1\otimes g_2)}}$ is clear. Finally, it is not hard to check that for each $y_2\in\per(g_2)$, the function $X_1\rightarrow X_1\times X_2$, $z_1\mapsto (z_1,((g_2)_{\mid\per(g_2)})^{-\pperl_{g_2}(z_1)}(y_2))$, is a digraph isomorphism between $\Tree_{\Gamma_{g_1}}(y_1)$ and $\Tree_{\Gamma_{g_1\otimes g_2}}((y_1,y_2))$.
\end{proof}

We are now ready to prove Theorem \ref{allTreesIsomorphicTheo}. In fact, we prove the following stronger version of it.

\begin{theoremm}\label{allTreesIsomorphicStrongTheo}
Let $m$ be a positive integer, and let $A:x\mapsto ax+b$ be an affine map of the cyclic group $\IZ/m\IZ$. The dual functional graph $(\Gamma_A)^{\ast}$ has rigid procreation, and for each $k\in\IN^+$ and each periodic vertex $x$ of $\Gamma_A$, one has $\proc_k(x)=|\ker^{(k)}(\mu_a):\ker^{(k-1)}(\mu_a)|$. In particular, each rooted tree above a periodic vertex in $\Gamma_A$ is isomorphic to each rooted tree above a periodic vertex in $\Gamma_{\mu_a}$.
\end{theoremm}

We note that in the situation of Theorem \ref{allTreesIsomorphicStrongTheo}, one has $|\ker^{(j)}(\mu_a)|=\gcd(a^j,m)$ for all $j\in\IN_0$, making the computation of the procreation numbers (and thus the understanding of the isomorphism types of rooted trees above periodic vertices) easy.

\begin{proof}[Proof of Theorem \ref{allTreesIsomorphicStrongTheo}]
The \enquote{In particular} statement is clear by Proposition \ref{rigidProp} and Theorem \ref{allTreesIsomorphicKnownTheo}, so we focus on the proof of the main statement.

As above, we factor $m=p_1^{v_1}\cdots p_K^{v_K}$ and denote by $A_j$ the reduction of $A$ modulo $p_j^{v_j}$. We claim that for each $j=1,2,\ldots,K$, the dual functional graph $(\Gamma_{A_j})^{\ast}$ has rigid procreation, and that for each $A_j$-periodic vertex $x\in\IZ/p_j^{v_j}\IZ$, the procreation number sequence of $x$ in $(\Gamma_{A_j})^{\ast}$ agrees with that of a periodic vertex in the dual functional graph of the reduction $(\mu_a)_j$ of $\mu_a$ modulo $p_j^{v_j}$. Indeed, following the proof of Lemma \ref{allTreesIsomorphicLem}, this is clear both if $\nu_{p_j}^{(v_j)}(a-1)\leq\nu_{p_j}^{(v_j)}(b)$ (where the digraphs $(\Gamma_{A_j})^{\ast}$ and $(\Gamma_{(\mu_a)_j})^{\ast}$ are actually isomorphic as a whole) and if $\nu_{p_j}^{(v_j)}(a-1)>\nu_{p_j}^{(v_j)}(b)$ (where both $(\Gamma_{A_j})^{\ast}$ and $(\Gamma_{(\mu_a)_j})^{\ast}$ are disjoint unions of directed cycles).

Now, we recall that $\Gamma_A=\bigotimes_{j=1}^K{\Gamma_{A_j}}$ and $\Gamma_{\mu_a}=\bigotimes_{j=1}^K{\Gamma_{(\mu_a)_j}}$. In both tensor products, the duals of the factors indexed by $j$ have rigid procreation with the same procreation number sequences of periodic vertices. Therefore, by Proposition \ref{cyclicTensorProductProp}(2), the same applies to $(\Gamma_A)^{\ast}$ versus $(\Gamma_{\mu_a})^{\ast}$, and this settles the main statement by virtue of the formula for procreation numbers in Theorem \ref{allTreesIsomorphicKnownTheo}.
\end{proof}

We also note the following consequence of Proposition \ref{cyclicTensorProductProp}(1), which will become important later.

\begin{lemmma}\label{periodicCharLem}
Let $m$ be a positive integer, let $a,b\in\IZ/m\IZ$, and let $L$ be a positive integer that is so large that $\gcd(a^L,m)=\prod_{p\mid\gcd(a,m)}{p^{\nu_p(m)}}$ (for example, $\mpe(m):=\max_p{\nu_p(m)}\leq\lfloor\log_2{m}\rfloor$\phantomsection\label{not81} is a valid choice for $L$). Then $y\in\IZ/m\IZ$ is a periodic point of the affine map $A:z\mapsto az+b$ of $\IZ/m\IZ$ if and only if
\[
y\equiv\sum_{t=0}^{L-1}{a^t}\cdot b\Mod{\gcd(a^L,m)}.
\]
\end{lemmma}

\begin{proof}
By Proposition \ref{cyclicTensorProductProp}(1) and the Chinese Remainder Theorem, a point $y\in\IZ/m\IZ$ is periodic under $A$ if and only if the reduction $y_p$ of $y$ modulo $p^{\nu_p(m)}$ is periodic under the reduction $A_p$ of $A$ modulo $p^{\nu_p(m)}$ for each prime $p\mid m$. But if $p\nmid a$, then by Proposition \ref{primaryDichotomyProp}(2), $A_p$ is a permutation of $\IZ/p^{\nu_p(m)}\IZ$, so $y_p$ is periodic under $A_p$. Therefore, only the primes $p\mid\gcd(a,m)$ actually give a restriction on $y$. For those primes $p$, we know by Proposition \ref{primaryDichotomyProp}(1) that $A_p$ has precisely one periodic point. It follows that if $y_0\in\IZ/m\IZ$ is a periodic point of $A$, then the periodic points $y$ of $A$ are characterized by the congruence $y\equiv y_0\Mod{\gcd(a^L,m)}$. Therefore, it only remains to prove that $\sum_{t=0}^{L-1}{a^t}\cdot b$ is a periodic point of $A$.

Now, using the same argument with $b=0$ so that it applies to $\mu_a:z\mapsto az$, we see that periodic points $y$ of $\mu_a$ are characterized by the congruence $y\equiv 0\Mod{\gcd(a^L,m)}$. In particular, $(\mu_a)^L(z)=a^Lz$ is periodic under $\mu_a$ for each $z\in\IZ/m\IZ$. Since the trees above periodic vertices in the functional graphs of $\mu_a$ and $A$ are isomorphic (by Theorem \ref{allTreesIsomorphicTheo}) and thus have the same height, it follows that $A^L(z)$ is periodic under $A$ for each $z\in\IZ/m\IZ$. In particular, $A^L(0)=\sum_{t=0}^{L-1}{a^t}\cdot b$ is periodic under $A$, as we needed to show.
\end{proof}

This concludes our results for cyclic groups. To round this subsection off, we put in some extra work to generalize Theorem \ref{allTreesIsomorphicStrongTheo} to arbitrary finite groups. First, we need to clarify what we mean by \enquote{affine map} in general. In what follows, for a fixed group $G$ we denote by $\rho_{\rrm}:G\rightarrow\Sym(G)$, $x\mapsto(y\mapsto yx)$\phantomsection\label{not84}, the so-called \emph{right-regular representation of $G$ on itself}\phantomsection\label{termRightReg} (for each $x\in G$, the function value $\rho_{\rrm}(x)\in\Sym(G)$ is the right-multiplication by $x$ on $G$). Analogously, $\rho_{\l}$\phantomsection\label{not85} denotes the \emph{left-regular representation of $G$ on itself}\phantomsection\label{termLeftReg}, whose function values are the left-multiplications on $G$ by fixed elements of $G$. As is common in group theory, we write the composition of a function $g:X\rightarrow Y$ with a function $g':Y\rightarrow Z$ as the product $gg':X\rightarrow Z$\phantomsection\label{notComp} (a synonymous notation is $g'\circ g$). When using this notation for composition, applications of functions to arguments are commonly written using exponents: $x^g$\phantomsection\label{notExp} instead of $g(x)$, so that $x^{gg'}=(x^g)^{g'}$.

\begin{deffinition}\label{affineMapDef}
Let $G$\phantomsection\label{not82} be a group. An \emph{affine map of $G$}\phantomsection\label{term42} is a function $G\rightarrow G$ of the form $\varphi\rho_{\rrm}(b): x\mapsto x^{\varphi}b$\phantomsection\label{not83} for some fixed element $b\in G$ and group endomorphism $\varphi$ of $G$.
\end{deffinition}

\begin{remmark}\label{affineMapRem}
We note the following concerning the concept of an affine map.
\begin{enumerate}
\item Since $\varphi\rho_{\rrm}(1_G)=\varphi$, affine maps are generalizations of group endomorphisms.
\item The affine maps of a given group $G$ form a monoid of functions on $G$, as they are composed via the formula
\[
\varphi\rho_{\rrm}(b)\cdot\varphi'\rho_{\rrm}(b)=\varphi\varphi'\rho_{\rrm}(b^{\varphi'}b').
\]
\item Alternatively, one could define affine maps through left-multiplication by a constant after application of a group endomorphism $\varphi$, i.e., as compositions $\varphi\rho_{\l}(b)=\rho_{\l}(b)\circ\varphi$. This leads to the same class of functions, as
\[
(\rho_{\l}(b)\circ\varphi)(x)=b\varphi(x)=bx^{\varphi}=bx^{\varphi}b^{-1}b=x^{\varphi\conj(b^{-1})}b=x^{\varphi\conj(b^{-1})\cdot\rho_{\rrm}(b)},
\]
where $\conj(g)$\phantomsection\label{not86} is the inner automorphism (conjugation) $x\mapsto g^{-1}xg$ of $G$.
\end{enumerate}
\end{remmark}

We briefly review some more concepts and results from \cite{Bor17a}.

\begin{deffinition}\label{perNilDef}
Let $G$ be a finite group and $\varphi$ an endomorphism of $G$. The \emph{hyperkernel of $\varphi$}\phantomsection\label{term43}, written $\nil(\varphi)$\phantomsection\label{not87}, consists of those $x\in G$ such that $\varphi^n(x)=1_G$ for some $n=n(x)\in\IN_0$.
\end{deffinition}

\begin{remmark}\label{perNilRem}
The notation $\nil(\varphi)$ stems from the fact that this is the largest $\varphi$-invariant subgroup of $G$ of which the corresponding restriction of $\varphi$ is a nilpotent endomorphism (i.e., an endomorphism that becomes trivial if composed with itself sufficiently often).
\end{remmark}

In the following theorem, we use the notation $G=H\ltimes N$\phantomsection\label{notSDP} to express that the group $G$ is the \emph{(internal) semidirect product of $H$ and $N$}\phantomsection\label{termSDP}, which means that $H$ is a subgroup of $G$, that $N$ is a normal subgroup of $G$, and that one has $H\cap N=\{1_G\}$ and $HN=\{hn: h\in H, N\in N\}=G$. In this situation, each element of $G$ can be written as a product $hn$ for $h\in H$ and $n\in N$ \emph{in a unique way}, and one may multiply elements of $G$ via the formula $(hn)\cdot(h'n')=hh'n^{h'}n'$ where $n^{h'}=n^{\conj(h')}=(h')^{-1}nh'$.

\begin{theoremm}\label{perNilTheo}
Let $G$ be a finite group, and $\varphi$ an endomorphism of $G$. Then $G=\per(\varphi)\ltimes\nil(\varphi)$.
\end{theoremm}

\begin{proof}
See \cite[proof of Theorem 1(1--3)]{Bor17a}. We note that after observing that $\per(\varphi)$ is a subgroup and $\nil(\varphi)$ is a normal subgroup of $G$, the rest of the statement follows easily from a group-theoretic version of Fitting's lemma that was proved by Caranti, see \cite[Theorem 4.2]{Car13a}.
\end{proof}

Theorem \ref{perNilTheo} has an interesting consequence concerning the functional graph $\Gamma_{\varphi}$, which was originally stated as \cite[Theorem 1(4)]{Bor17a} and is easy to prove using Proposition \ref{cyclicTensorProductProp} with $X_1=\per(\varphi)$ and $X_2=\nil(\varphi)$.

\begin{corrollary}\label{perNilCor}
Let $G$ be a finite group, and $\varphi$ an endomorphism of $G$. Then
\[
\Gamma_{\varphi}=\Gamma_{\varphi_{\mid\per(\varphi)}}\otimes\Gamma_{\varphi_{\mid\nil(\varphi)}},
\]
and, consequently, for each $x\in\per(\varphi)$, one has $\Tree_{\Gamma_{\varphi}}(x)\cong\Tree_{\Gamma_{\varphi_{\mid\nil(\varphi)}}}(1_G)$, a rooted tree that can be obtained from $\Gamma_{\varphi_{\mid\nil(\varphi)}}$ by deleting the unique loop of the latter at $1_G$.
\end{corrollary}

The following result extends Theorem \ref{allTreesIsomorphicStrongTheo} to arbitrary finite groups.

\begin{theoremm}\label{rigidTheo}
Let $G$ be a finite group, $b\in G$, and $\varphi$ an endomorphism of $G$. Then $\Gamma_{\varphi\rho_{\rrm}(b)}^{\ast}$ has rigid procreation. Moreover, the sequence of procreation numbers $(\proc_k(x))_{k\geq1}$ of any periodic vertex $x$ in $\Gamma_{\varphi\rho_{\rrm}(b)}^{\ast}$ is the same as that of a periodic vertex in $\Gamma_{\varphi}^{\ast}$.
\end{theoremm}

\begin{proof}
We prove by induction on $k\geq1$ that if $x,y\in G$ each have at least $k$ successor generations in $\Gamma_{\varphi\rho_{\rrm}(b)}^{\ast}$, then $\proc_k(x)=\proc_k(y)$ in $\Gamma_{\varphi\rho_{\rrm}(b)}^{\ast}$. For the induction base, $k=1$, we observe that for $z\in\{x,y\}$, one has
\begin{align*}
\proc_1(z)&=\#\text{ children of }z\text{ in }\Gamma_{\varphi\rho_{\rrm}(b)}^{\ast} \\
&=|\{w\in G: w^{\varphi\rho_{\rrm}(b)}=z\}|=|\{w\in G: w^{\varphi}=zb^{-1}\}|.
\end{align*}
Since $\{w\in G: w^{\varphi}=zb^{-1}\}$ is either empty or a coset of $\ker{\varphi}$, it follows that $\proc_1(x)=\proc_1(y)=|\ker{\varphi}|$ whenever $\proc_1(x),\proc_1(y)>0$, as required.

Now we assume that $k\geq2$ and that the statement holds up to $k-1$. Each of the two vertices $x$ and $y$ has at least $n$ successor generations in $\Gamma_{\varphi\rho_{\rrm}(b)}^{\ast}$ for each $n\in\{1,2,\ldots,k-1\}$, whence by the induction hypothesis, one has $\proc_n(x)=\proc_n(y)$ for $n=1,2,\ldots,k-1$. Now, for each $z\in G$, the number of endpoints of directed paths of length $k$ in $\Gamma_{\varphi\rho_{\rrm}(b)}^{\ast}$ starting at $z$ is
\begin{align*}
|\{w\in G: w^{\left((\varphi\rho_{\rrm}(b))^k\right)}=z\}|&=|\{w\in G: \varphi^k(w)\cdot \varphi^{k-1}(b)\varphi^{k-2}(b)\cdots \varphi(b)b=z\}| \\
&=|\{w\in G: \varphi^k(w)=zb^{-1}\varphi(b)^{-1}\cdots\varphi^{k-1}(b)^{-1}\}|,
\end{align*}
and the set $\{w\in G: \varphi^k(w)=zb^{-1}\varphi(b)^{-1}\cdots\varphi^{k-1}(b)^{-1}\}$ is either empty or a coset of
\[
\ker^{(k)}(\varphi):=\{w\in G: \varphi^{k}(w)=1_G\}.
\]
Hence, $x$ and $y$ have the same number of endpoints of length $k$ directed paths in $\Gamma_{\varphi\rho_{\rrm}(b)}^{\ast}$ starting at them, namely $|\ker^{(k)}(\varphi)|$. Using the induction hypothesis in its general form (which basically states that $\Gamma_{\varphi\rho_{\rrm}(b)}^{\ast}$ has rigid procreation \enquote{for $k-1$ generations}), an easy induction on $n=1,2,\ldots,k-1$ shows that for each vertex $w\in G$ with at least $n$ successor generations in $\Gamma_{\varphi\rho_{\rrm}(b)}^{\ast}$, the number of endpoints of directed paths in $\Gamma_{\varphi\rho_{\rrm}(b)}^{\ast}$ of length $n$ starting at $w$ is $\prod_{j=1}^n{\proc_j(w)}$. In particular, if $w$ is a child of $z\in\{x,y\}$ that has at least $k-1$ successor generations, then the number of endpoints of length $k-1$ paths starting at $w$ is $\prod_{j=1}^{k-1}{\proc_j(w)}$, which is equal to $\prod_{j=1}^{k-1}{\proc_j(z)}$ by the induction hypothesis. Since the number of such children $w$ of $z$ is $\proc_k(z)$, we conclude that
\begin{align*}
\prod_{j=1}^k{\proc_j(z)}&=\proc_k(z)\cdot\prod_{j=1}^{k-1}{\proc_j(z)} \\
&=\#\text{ endpoints of length }k\text{ paths starting at }z=|\ker^{(k)}(\varphi)|.
\end{align*}
Because $z\in\{x,y\}$ is arbitrary, it follows that
\[
\prod_{j=1}^k{\proc_j(x)}=\prod_{j=1}^k{\proc_j(y)},
\]
and since $\proc_j(x)=\proc_j(y)>0$ for $j=1,2,\ldots,k-1$, this allows us to conclude that $\proc_k(x)=\proc_k(y)$, as required.

Concerning the claim that the procreation numbers of $\Gamma_{\varphi\rho_{\rrm}(b)}^{\ast}$ and $\Gamma_{\varphi}^{\ast}$ are the same, the above argument shows that for any positive integer $k$ and any periodic vertex $x$ of $\Gamma_{\varphi\rho_{\rrm}(b)}^{\ast}$, we have
\[
\prod_{j=1}^k{\proc_j(x)}=|\ker^{(k)}(\varphi)|.
\]
It follows that
\[
\proc_k(x)=|\ker^{(k)}(\varphi):\ker^{(k-1)}(\varphi)|
\]
for each $k\in\IN^+$ (we note that $\ker^{(0)}(\varphi)=\{1_G\}$). Therefore, the procreation number sequence of a periodic vertex of $\Gamma_{\varphi\rho_{\rrm}(b)}^{\ast}$ only depends on $\varphi$, not on $b$, and for $b:=1_G$, one has $\Gamma_{\varphi\rho_{\rrm}(b)}^{\ast}=\Gamma_{\varphi}^{\ast}$.
\end{proof}

\subsection{The Master Lemma}\label{subsec2P2}

From the introduction, we recall the notion of an $m$-CC (short for \enquote{$m$-congruential condition}), which we defined as a condition of the form $\nu(x\equiv\bfrak\Mod{\afrak})$ with $\afrak\mid m$ and $\nu\in\{\emptyset,\neg\}$. In this subsection, we consider systems formed from $m$-CCs in one common variable. Such a system is \emph{consistent}\phantomsection\label{term44} if it has an integer solution, and two such systems are \emph{equivalent}\phantomsection\label{term45} if they have the same solution set in $\IZ$ (or, equivalently, in $\IZ/m\IZ$). The solution set in $\IZ/m\IZ$ of a consistent system of $m$-CCs is a block of the associated arithmetic partition of $\IZ/m\IZ$ (see Definition \ref{sArithDef}). We note the following fundamental result on systems of $m$-congruences, which is a well-known generalization of the Chinese Remainder Theorem.

\begin{propposition}\label{sCongProp}
Let $m$ be a positive integer. We consider a system of $m$-congruences, of the form
\begin{align}\label{sCongEq}
\notag x&\equiv \bfrak_1\Mod{\afrak_1} \\
\notag x&\equiv \bfrak_2\Mod{\afrak_2} \\
\notag &\vdots \\
x&\equiv \bfrak_K\Mod{\afrak_K}
\end{align}
The following statements are equivalent.
\begin{enumerate}
\item System (\ref{sCongEq}) is consistent.
\item For all $1\leq j<k\leq K$, one has $\gcd(\afrak_j,\afrak_k)\mid \bfrak_j-\bfrak_k$.
\item Any pair of $m$-congruences in system (\ref{sCongEq}) form a consistent system.
\item System (\ref{sCongEq}) is equivalent to a single $m$-congruence, of the form
\[
x\equiv \bfrak\Mod{\lcm(\afrak_1,\afrak_2,\ldots,\afrak_K)}.
\]
\end{enumerate}
In particular, if system (\ref{sCongEq}) is consistent, then its number of solutions modulo $m$ is equal to $m/\lcm(\afrak_1,\afrak_2,\ldots,\afrak_K)$ and is, therefore, independent of the $\bfrak_j$.
\end{propposition}

\begin{proof}
For the equivalence \enquote{(1)$\Leftrightarrow$(2)} and the implication \enquote{(1)$\Rightarrow$(4)}, see \cite[Theorem 3.3.4 on p.~78]{KR98a}, for example. Moreover, the implication \enquote{(4)$\Rightarrow$(1)} is trivial. As for the equivalence \enquote{(1)$\Leftrightarrow$(3)}, we note that by the already established equivalence \enquote{(1)$\Leftrightarrow$(2)}, applied to the system
\begin{align}\label{sCongEq2}
\notag x &\equiv \bfrak_j\Mod{\afrak_j} \\
x &\equiv \bfrak_k\Mod{\afrak_k},
\end{align}
that system is consistent if and only if the single divisibility $\gcd(\afrak_j,\afrak_k)\mid \bfrak_j-\bfrak_k$ holds. But statement (3) just demands that system (\ref{sCongEq2}) be consistent for all $1\leq j<k\leq K$, which is therefore equivalent to statement (2), and thus to statement (1).
\end{proof}

Proposition \ref{sCongProp} is the basis for proving the following lemma, which is crucial for our recursive approach for understanding the rooted trees in Subsection \ref{subsec3P3}.

\begin{lemmma}\label{masterLem}(Master Lemma)
Let $m$ be a positive integer, and let $\Pcal$ be an arithmetic partition of $\IZ/m\IZ$. Moreover, let $a,b\in\IZ/m\IZ$, and consider the affine map $A:x\mapsto ax+b$ of $\IZ/m\IZ$. There is an arithmetic partition $\Pcal'$ of $\IZ/m\IZ$ with $\AC(\Pcal')\leq\AC(\Pcal)+1$ such that if $x,y\in\IZ/m\IZ$ are from a common block of $\Pcal'$, and if $B$\phantomsection\label{not88} is a block of $\Pcal$, then $|A^{-1}(\{x\})\cap B|=|A^{-1}(\{y\})\cap B|$.

More specifically, if $\Pcal=\Pfrak(x\equiv \bfrak_j\Mod{\afrak_j}: j=1,2,\ldots,K)$, then
\[
\Pfrak'(\Pcal,A):=
\Pfrak
\begingroup
\setlength\arraycolsep{2pt}
\left(
\begin{array}{lll}
x & \equiv & a\bfrak_1+b\Mod{\gcd(a\afrak_1,m)} \\
x & \equiv & a\bfrak_2+b\Mod{\gcd(a\afrak_2,m)} \\
& \vdots & \\
x & \equiv & a\bfrak_K+b\Mod{\gcd(a\afrak_K,m)} \\
x &\equiv & b\Mod{\gcd(a,m)}
\end{array}
\right)
\endgroup
\]
is\phantomsection\label{not89} a valid choice for $\Pcal'$, and a formula for $|A^{-1}(\{x\})\cap B|$ in terms of the $\Pcal$-block $B$ and the unique $\Pfrak'(\Pcal,A)$-block $B'$ containing $x$ can be obtained as follows: write $B=\Bcal(\Pcal,\vec{\nu})$ and $B'=\Bcal(\Pfrak'(\Pcal,A),\vec{\nu'})$ with $\vec{\nu}=(\nu_1,\ldots,\nu_K)\in\{\emptyset,\neg\}^K$ and $\vec{\nu'}=(\nu'_1,\ldots,\nu'_{K+1})\in\{\emptyset,\neg\}^{K+1}$. We introduce the following notations.
\begin{itemize}
\item $J_-(\vec{\nu}):=\{j\in\{1,2,\ldots,K\}: \nu_j=\neg\}$\phantomsection\label{not90} and (analogously) $J_-(\vec{\nu'}):=\{j'\in\{1,2,\ldots,K+1\}: \nu'_{j'}=\neg\}$;
\item $J_+(\vec{\nu}):=\{1,2,\ldots,K\}\setminus J_-(\vec{\nu})$\phantomsection\label{not91};
\item For $J\subseteq J_-(\vec{\nu})$\phantomsection\label{not92}, we denote by $E(\vec{\nu},J)$\phantomsection\label{not93} the condition \enquote{For all $j_1,j_2\in J_+(\vec{\nu})\cup J$: $\gcd(\afrak_{j_1},\afrak_{j_2})\mid \bfrak_{j_1}-\bfrak_{j_2}$.}.
\end{itemize}
Then for each $x\in\Bcal(\Pfrak'(\Pcal,A),\vec{\nu'})$, the intersection size $|A^{-1}(\{x\})\cap\Bcal(\Pcal,\vec{\nu})|$ is equal to
\[
\sum_{J\subseteq J_-(\vec{\nu})}{(-1)^{|J|}\kappa(\vec{\nu},\vec{\nu'},J)}=:\sigma_{\Pcal,A}(\vec{\nu},\vec{\nu'})
\]
\phantomsection\label{not94}
where
\begin{align*}
&\kappa_{\Pcal,A}(\vec{\nu},\vec{\nu'},J):= \\
&\delta_{\nu'_{K+1}=\emptyset}\cdot\delta_{E(\vec{\nu},J)}\cdot\delta_{(J_+(\vec{\nu})\cup J)\cap J_-(\vec{\nu'})=\emptyset}\cdot\frac{m}{\lcm(\frac{m}{\gcd(a,m)},\afrak_j: j\in J_+(\vec{\nu})\cup J)},
\end{align*}
the\phantomsection\label{not95} three deltas being Kronecker deltas.
\end{lemmma}

We call a number of the form $\sigma_{\Pcal,A}(\vec{\nu},\vec{\nu'})$ a \emph{distribution number of $\Pcal$ (under $A$)}\phantomsection\label{term46}.

\begin{proof}[Proof of Lemma \ref{masterLem}]
To verify that $\Pcal':=\Pfrak'(\Pcal,A)$ has the desired property, we set
\begin{itemize}
\item $M_j:=\{x\in\IZ/m\IZ: x\equiv \bfrak_j\Mod{\afrak_j}\}$ for $j=1,2,\ldots,K$;
\item $M_+(\vec{\nu}):=\bigcap_{j\in J_+(\vec{\nu})}{M_j}$;
\end{itemize}
and note that $B=M_+(\vec{\nu})\cap\bigcap_{j\in J_-(\vec{\nu})}{M_j^c}$ (the superscript $c$ denoting set complementation in $\IZ/m\IZ$\phantomsection\label{not96}). Therefore, by the inclusion-exclusion principle, the intersection size $|A^{-1}(\{x\})\cap B|$ we are looking for is equal to
\[
\sum_{J\subseteq J_-(\vec{\nu})}{(-1)^{|J|}|(M_+(\vec{\nu})\cap A^{-1}(\{x\}))\cap\bigcap_{j\in J}{M_j}|}.
\]
It is thus our goal to argue that the value of this sum is equal to $\sigma_{\Pcal,A}(\vec{\nu},\vec{\nu'})$ (in particular, it is independent of the choice of $x\in B'=\Bcal(\Pcal',\vec{\nu'})$). We do so by arguing that, in fact, the intersection size
\[
|(M_+(\vec{\nu})\cap A^{-1}(\{x\}))\cap\bigcap_{j\in J}{M_j}|
\]
is equal to $\kappa_{\Pcal,A}(\vec{\nu},\vec{\nu'},J)$, for each $J\subseteq J_-(\vec{\nu})$. Now, writing $J_+(\vec{\nu})\cup J=\{j_1,j_2,\ldots,j_N\}$, the intersection $(M_+(\vec{\nu})\cap A^{-1}(\{x\}))\cap\bigcap_{j\in J}{M_j}$ is the solution set in $\IZ/m\IZ$ of the following system in the variable $y$:
\begin{align}\label{summandSystemEq}
\notag y&\equiv \bfrak_{j_1}\Mod{\afrak_{j_1}} \\
\notag y&\equiv \bfrak_{j_2}\Mod{\afrak_{j_2}} \\
\notag &\vdots \\
\notag y&\equiv \bfrak_{j_N}\Mod{\afrak_{j_N}} \\
ay+b&\equiv x\Mod{m}.
\end{align}
Now, the single congruence $ay+b\equiv x\Mod{m}$ is solvable in $y$ if and only if $\gcd(a,m)\mid x-b$, i.e., if and only if $x\equiv b\Mod{\gcd(a,m)}$. This is one of the spanning congruences for $\Pcal'$, whence its truth value is constant for all $x\in B'$. If that congruence is false for all $x\in B'$ (equivalently, if $\nu'_{K+1}=\neg$), then system (\ref{summandSystemEq}) is always false, whence $|(M_+(\vec{\nu})\cap A^{-1}(\{x\}))\cap\bigcap_{j\in J}{M_j}|=0$ for all $J\subseteq J_-(\vec{\nu})$, independently of $x$. This explains the first Kronecker delta in the definition of $\kappa_{\Pcal,A}(\vec{\nu},\vec{\nu'},J)$.

We may thus henceforth assume that $\nu'_{K+1}=\emptyset$, i.e., that $x\equiv b\Mod{\gcd(a,m)}$ for all $x\in B'$. Then the congruence $ay+b\equiv x\Mod{m}$ can be equivalently rewritten into
\begin{equation}\label{replacementEq}
y\equiv \inv_{m/\gcd(a,m)}\left(\frac{a}{\gcd(a,m)}\right)\cdot\frac{x-b}{\gcd(a,m)}\Mod{\frac{m}{\gcd(a,m)}},
\end{equation}
where $\inv_n(x)$\phantomsection\label{not97} denotes the multiplicative inverse of the unit $x$ modulo $n$. If we replace the congruence $ay+b\equiv x\Mod{m}$ in the system (\ref{summandSystemEq}) by the equivalent congruence (\ref{replacementEq}), then the resulting system consists entirely of $m$-congruences. By Proposition \ref{sCongProp}, there are only two possibilities for the number of solutions modulo $m$ of this system: either the system is inconsistent and, thus, has $0$ solutions, or it has $m/L$ solutions, where $L$ is the least common multiple of the moduli that occur. Neither of these two expressions for the number of solutions depends on $x$, so we aim to show that it does not depend on the choice of $x\in B'$ which of the two cases occurs.

Now, Proposition \ref{sCongProp} also implies that system (\ref{summandSystemEq}) is consistent if and only if any pair of conditions in it is consistent. It thus suffices to argue that for no pair of conditions in system (\ref{summandSystemEq}) does the consistency of the system formed from those two conditions depend on the choice of $x\in B'$. If both of those conditions are distinct from the congruence $ay+b\equiv x\Mod{m}$, then they are of the forms $y\equiv \bfrak_j\Mod{\afrak_j}$ and $y\equiv \bfrak_k\Mod{\afrak_k}$, for suitable $j,k\in J_+(\vec{\nu})\cup J$, and by Proposition \ref{sCongProp}, those two conditions form a consistent system if and only if $\gcd(\afrak_j,\afrak_k)\mid \bfrak_j-\bfrak_k$. Equivalently, the system obtained from (\ref{summandSystemEq}) by deleting the single congruence $ay+b\equiv x\Mod{m}$ is consistent if and only if the condition $E(\vec{\nu},J)$ holds, which explains the second Kronecker delta in the definition of $\kappa_{\Pcal,A}(\vec{\nu},\vec{\nu'},J)$.

It remains to consider two-condition subsystems of (\ref{summandSystemEq}) of the form
\begin{align}\label{twoConditionEq}
\notag ay+b &\equiv x\Mod{m} \\
y &\equiv\bfrak_j\Mod{\afrak_j}
\end{align}
for some $j\in J_+(\vec{\nu})\cup J$. We claim that system (\ref{twoConditionEq}) is consistent if and only if $x\equiv a\bfrak_j+b\Mod{\gcd(a\afrak_j,m)}$. Indeed, if system (\ref{twoConditionEq}) is consistent, then there is a $k\in\IZ$ such that some $y=\bfrak_j+k\afrak_j$ satisfies the first condition of the system. That is, one then has
\[
x\equiv ay+b=a\bfrak_j+ka\afrak_j+b\Mod{m}.
\]
In particular,
\[
x\equiv a\bfrak_j+ka\afrak_j+b\equiv a\bfrak_j+b\Mod{\gcd(a\afrak_j,m)},
\]
as required. On the other hand, let us assume that $x\equiv a\bfrak_j+b\Mod{\gcd(a\afrak_j,m)}$. Then we can write $x=a\bfrak_j+b+k'\gcd(a\afrak_j,m)$ for some $k'\in\IZ$. We need to verify that there is an integer $k$ such that for $y=\bfrak_j+k\afrak_j$, one has
\[
ay+b=a\bfrak_j+ka\afrak_j+b\equiv x=a\bfrak_j+b+k'\gcd(a\afrak_j,m)\Mod{m},
\]
which is equivalent to $ka\afrak_j\equiv k'\gcd(a\afrak_j,m)\Mod{m}$. And indeed, this is solvable in $k$, because $\gcd(a\afrak_j,m)\mid k'\gcd(a\afrak_j,m)$.

Now, because $x\equiv a\bfrak_j+b\Mod{\gcd(a\afrak_j,m)}$ is the $j$-th spanning congruence of $\Pcal'$, it follows that if $\nu'_j=\neg$ (equivalently, if $j\in J_-(\vec{\nu'})$), then the intersection $(M_+(\vec{\nu})\cap A^{-1}(\{x\}))\cap\bigcap_{j\in J}{M_j}$ is empty whenever $j\in J_+(\vec{\nu})\cup K$ as well, which explains the third Kronecker delta in the definition of $\kappa_{\Pcal,A}(\vec{\nu},\vec{\nu'},J)$.

If all three Kronecker deltas in the definition of $\kappa_{\Pcal,A}(\vec{\nu},\vec{\nu'},J)$ are $1$, then our argumentation shows that system (\ref{summandSystemEq}) is consistent and is thus equivalent to a single $m$-congruence by an application of Proposition \ref{sCongProp} (we recall that the last congruence in system (\ref{summandSystemEq}) may be replaced by the equivalent $m$-congruence (\ref{replacementEq})), the modulus of which is the least common multiple $L$ of the moduli involved in the $m$-congruence forms of the conditions in system (\ref{summandSystemEq}). It follows that the size of the solution set of system (\ref{summandSystemEq}) then is
\[
\frac{m}{L}=\frac{m}{\lcm\left(\frac{m}{\gcd(a,m)},\afrak_j: j\in J_+(\vec{\nu})\cup J\right)}.
\]
Therefore, our technical parameter $\kappa_{\Pcal,A}(\vec{\nu},\vec{\nu'},J)$ indeed always agrees with the intersection size $|(M_+(\vec{\nu})\cap A^{-1}(\{x\}))\cap\bigcap_{j\in J}{M_j}|$, and this concludes the proof.
\end{proof}

\subsection{CRL-lists of affine maps of finite cyclic groups}\label{subsec2P3}

We view $\IZ/m\IZ$, where $m\in\IN^+$, as a ring with underlying set $\{0,1,\ldots,m-1\}$ and modular addition and modular multiplication as the ring operations. In particular, if $m_1\leq m_2$ are positive integers, then we have an inclusion of sets $\IZ/m_1\IZ\subseteq\IZ/m_2\IZ$. We may also view integers outside of the range $\{0,1,\ldots,m-1\}$ as elements of $\IZ/m\IZ$, via reduction modulo $m$ (identifying $x\in\IZ$ with $x\bmod{m}$). We remind the reader of the notation $\nu_p^{(v)}(x):=\min\{v,\nu_p(x)\}$ for $p$ prime, $v\in\IN_0$ and $x\in\IZ$, originally introduced after Theorem \ref{allTreesIsomorphicTheo}.

As was mentioned in the introduction, the construction of a CRL-list (in the sense of Definition \ref{crlListDef}) for a generalized cyclotomic mapping of $\IF_q$ can be reduced to the corresponding problem for affine maps of finite cyclic groups, which we solve in this subsection. We also observed in the introduction that determining a CRL-list for a function $g:X\rightarrow X$ with $X$ finite is generally a harder problem than the determination of the cycle type of $g_{\mid\per(g)}$, and we would like to give an overview of the history of the latter problem for the case where $g$ is an affine permutation of a finite cyclic group.

Ahmad \cite{Ahm69a} determined the cycle structure of \emph{automorphisms} of finite cyclic groups. The cycle index of the group of affine permutations of a finite cyclic group $\IZ/m\IZ$ (which is a polynomial that encodes how many affine permutations of each given cycle type there are) was described by Wei and Xu \cite{WX93a}. In their paper, they gave the formulas for the case where $m$ is a prime power without proof, referring to the two-page research announcement \cite{WGY88a} by Wei, Gao and Yang. Unfortunately, while Bors and Wang were working on \cite{BW22b}, they were unable to find \cite{WGY88a} through an online search, which led them to derive those formulas independently as \cite[Theorem 4.8]{BW22b}, based on a precise description of the cycle type of a given affine permutation $A:x\mapsto ax+b$ of $\IZ/p^v\IZ$ in terms of $a$ and $b$, stated as \cite[Proposition 4.7]{BW22b}. This latter result is as of now, to the authors' knowledge, the only accessible reference that lists those cycle types explicitly, in a tabular form. While working on the current paper, the authors realized that \cite[Proposition 4.7]{BW22b} could also have been easily derived from Deng's results \cite[Lemmas 4 and 7]{Den13a} and Ahmad's result \cite[Theorem 1]{Ahm69a}.

Let us now turn to the determination of CRL-lists. Let $m$ be a positive integer, and let $a,b\in\IZ$. We consider the affine map $A:x\mapsto ax+b$ of $\IZ/m\IZ$. From Proposition \ref{primaryDichotomyProp}, we know the following.
\begin{itemize}
\item The reduction of $A$ modulo $m':=\prod_{p\mid m,p\nmid a}{p^{\nu_p(m)}}$\phantomsection\label{not98} is an affine permutation of $\IZ/m'\IZ$.
\item The reduction of $A$ modulo $m'':=\prod_{p\mid\gcd(a,m)}{p^{\nu_p(m)}}$\phantomsection\label{not99} has exactly one periodic point $\ffrak''$\phantomsection\label{not100} in $\IZ/m''\IZ$, which we know explicitly thanks to Lemma \ref{periodicCharLem}.
\end{itemize}
Now, the restriction of the projection
\[
\IZ/m\IZ\cong \IZ/m'\IZ\times\IZ/m''\IZ\rightarrow\IZ/m'\IZ
\]
to $\IZ/m'\IZ\times\{\ffrak''\}$ is bijective; we denote by $\Lambda$\phantomsection\label{not101} its inverse function $\IZ/m'\IZ\rightarrow\IZ/m\IZ$. Then, if $\Lcal'$ is a CRL-list of the reduction of $A$ modulo $m'$, the set $\{(\Lambda(r),l): (r,l)\in\Lcal'\}$ is a CRL-list of $A$. This reduces the problem to the special case where $A$ is an affine \emph{permutation} of a finite cyclic group, which we henceforth assume.

In order to understand CRL-lists of affine permutations of finite cyclic groups, it is helpful to proceed in several steps:
\begin{enumerate}
\item First, we determine a CRL-list for each \emph{group automorphism} of each finite \emph{primary} cyclic group (i.e., $\IZ/p^v\IZ$).
\item Next, we extend this to arbitrary affine permutations of finite \emph{primary} cyclic groups.
\item Finally, we use the Chinese Remainder Theorem and some extra ideas to construct a CRL-list of any affine permutation $A$ of each finite cyclic group from CRL-lists of the reductions of $A$ modulo the various prime powers $p^{\nu_p(m)}$.
\end{enumerate}
Before tackling Step (1) properly, we prove the following useful lemma.

\begin{lemmma}\label{crlListPowerLem}
Let $X$ be a finite set, $\psi\in\Sym(X)$, $\Lcal$ a CRL-list of $\psi$, and $n\in\IZ$. Then the following is a CRL-list of $\psi^n$:
\[
\left\{\left(\psi^j(r),\frac{l}{\gcd(n,l)}\right): (r,l)\in\Lcal, j=0,1,\ldots,\gcd(n,l)-1\right\}.
\]
In particular, if $\gcd(n,\ord(\psi))=1$\phantomsection\label{not101P5}, then $\Lcal$ is also a CRL-list of $\psi^n$.
\end{lemmma}

\begin{proof}
Each cycle of $\psi^n$ is contained in a cycle of $\psi$, and for any given cycle $\zeta$\phantomsection\label{not102} of $\psi$ of length $l$, the $\psi^n$-cycles into which $\zeta$ decomposes correspond to the cosets of the subgroup of $\IZ/l\IZ$ generated by $n+l\IZ$. Since the additive order of $n$ modulo $l$ is $l/\gcd(n,l)$, it follows that $\zeta$ decomposes into $\gcd(n,l)$ cycles of $\psi^n$, each of length $l/\gcd(n,l)$. If $r$ is the representative of $\zeta$ from $\Lcal$, then for each $\psi^n$-cycle $\zeta'$ contained in $\zeta$, the elements on $\zeta'$ are just those of the form $\psi^t(r)$ for $t\in k+\gcd(n,l)\IZ$, for some $k=k(\zeta')\in\IZ$. It follows that the $\gcd(n,l)$ elements $\psi^j(r)$ of (the support set of) $\zeta$ for $j=0,1,\ldots,\gcd(n,l)-1$ lie on pairwise distinct $\psi^n$-cycles and thus form a system of representatives for the $\psi^n$-cycles contained in $\zeta$. This proves the main statement of the lemma. The \enquote{In particular} statement follows because the equality $\gcd(n,\ord(\psi))=1$ is equivalent to \enquote{$\gcd(n,l)=1$ for each cycle length $l$ of $\psi$}.
\end{proof}

We are now ready to specify a CRL-list for each automorphism of each finite primary cyclic group. In the proof of the following lemma and beyond, we use the notation $H\leq G$\phantomsection\label{notSubgroup} for \enquote{$H$ is a subgroup of $G$}, and $\langle g_1,g_2,\ldots,g_n\rangle$\phantomsection\label{notGenerate} to denote the subgroup of the group $G$ generated by the elements $g_1,g_2,\ldots,g_n\in G$.

\begin{lemmma}\label{crlListAutomorphismLem}
Let $p$ be a prime, $v\in\IN^+$, and $a\in\IZ$ with $p\nmid a$. If $p$ is odd, let $\rfrak$\phantomsection\label{not103} be a fixed primitive root modulo $p^v$, and let $\phi$\phantomsection\label{not104} denote Euler's totient function. Table \ref{crlListAutomorphismTable} provides a CRL-list $\Lcal(p^v,a)$\phantomsection\label{not105} of the automorphism $\mu_a:x\mapsto ax$ of $\IZ/p^v\IZ$.
\end{lemmma}

\begin{longtable}[h]{|c|c|}\hline
Case for $p$ and $a$ & elements of $\Lcal(p^v,a)$ \\ \hline
$p>2$ & \thead{$\left(\rfrak^jp^t,\frac{\phi(p^{v-t})}{\gcd\left(\frac{\phi(p^v)}{\ord(a)},\phi(p^{v-t})\right)}\right)$ for \\ $t=0,1,\ldots,v$ and $j=0,1,\ldots,\gcd\left(\frac{\phi(p^v)}{\ord(a)},\phi(p^{v-t})\right)-1$.} \\ \hline
$p=2$, $a\equiv1\Mod{4}$ & \thead{$(0,1)$, $(2^{v-1},1)$; \\ $\left(5^j2^t,\frac{2^{v-t-2}}{\gcd\left(\frac{2^{v-2}}{\ord(a)},2^{v-t-2}\right)}\right)$, $\left(-5^j2^t,\frac{2^{v-t-2}}{\gcd\left(\frac{2^{v-2}}{\ord(a)},2^{v-t-2}\right)}\right)$ for \\ $t=0,1,\ldots,v-2$ and $j=0,1,\ldots,\gcd\left(\frac{2^{v-2}}{\ord(a)},2^{v-t-2}\right)-1$.} \\ \hline
$p=2$, $a\equiv3\Mod{4}$ & \thead{$(0,1)$, $(2^{v-1},1)$; \\ $(j\cdot\ord(-a),2)$ for $j=1,2,\ldots,\frac{2^{v-1}}{\ord(-a)}-1$; \\ $\left(5^j2^t,\frac{\ord(-a)}{2^t}\right)$, $\left(-5^j2^t,\frac{\ord(-a)}{2^t}\right)$ for \\ $t=0,1,\ldots,\log_2(\ord(-a))-1$ and \\ $j=0,1,\ldots,\frac{2^{v-2}}{\ord(-a)}-1$.}  \\ \hline
\caption{CRL-lists of automorphisms of finite primary cyclic groups.}
\label{crlListAutomorphismTable}
\end{longtable}

\begin{proof}[Proof of Lemma \ref{crlListAutomorphismLem}]
First, we assume that $p>2$. If $a$ is a primitive root modulo $p^v$, then the cyclic group $\langle a\rangle=(\IZ/p^v\IZ)^{\ast}\cong\Aut(\IZ/p^v\IZ)$\phantomsection\label{not106}\phantomsection\label{not107} acts transitively on each subset of $\IZ/p^v\IZ$ consisting of all elements of a given additive order. Indeed, on the one hand, automorphisms of $\IZ/p^v\IZ$ must preserve the additive order of elements, and conversely, if $x,y\in\IZ/p^v\IZ$ are of the same order, then they are multiples of each other. Hence $y=z\cdot x$ for some $z\in(\IZ/p^v\IZ)^{\ast}$. Since $z$ is a power of $a$, the transitivity assertion follows. We conclude that if $a$ is a primitive root modulo $p^v$, then $\Lcal(p^v,a)$ may be chosen as $\{(p^t,\phi(p^{v-t})): t=0,1,\ldots,v\}$, which matches with Table \ref{crlListAutomorphismTable}.

For general $a$, we note that $a$ and $\rfrak^{\phi(p^v)/\ord(a)}$ are powers of each other, whence by the \enquote{In particular} of Lemma \ref{crlListPowerLem}, we may assume without loss of generality that $a=\rfrak^{\phi(p^v)/\ord(a)}$. The claim now follows by applying the main statement of Lemma \ref{crlListPowerLem} with $n:=\phi(p^v)/\ord(a)$ and $\Lcal:=\{(p^t,\phi(p^{v-t})): t=0,1,\ldots,v\}$.

Now we assume that $p=2$. First, let us discuss the case $a=5$. The automorphism $\mu_5:x\mapsto 5x$, like any automorphism of $\IZ/2^v\IZ$, fixes the unique elements $0$ and $2^{v-1}$ of additive orders $1$ and $2$, respectively. It also fixes the order $4$ elements $2^{v-2}$ and $3\cdot 2^{v-2}=-2^{v-2}$. Moreover, we claim that for each $t'\in\{3,4,\ldots,v\}$, the automorphism $\mu_5$ has exactly two cycles on the elements of $\IZ/2^v\IZ$ of additive order $2^{t'}$, both of length $2^{t'-2}$ and spanned by $2^{v-t'}$ and $-2^{v-t'}$, respectively. Indeed, this is clear for $t'=3$ and $t'=v$; for the latter, we use that the multiplicative order of $5=1+2^2$ modulo $2^v$ is $2^{v-2}$, that $\langle 5\rangle\leq(\IZ/2^v\IZ)^{\ast}$ acts semiregularly (i.e., such that no element of that group except the neutral element $1$ admits fixed points in that action) on the set of generators (i.e., elements of additive order $2^v$) of $\IZ/2^v\IZ$, and that $1\not\equiv-1\Mod{4}$. For each other value of $t'$, denoting by $\aord(x)$ the additive order of $x$ modulo $2^v$, it follows from the commutativity of the diagram
\begin{center}
\begin{tikzpicture}
\matrix (m) [matrix of math nodes,row sep=3em,column sep=4em,minimum width=2em]
  {
     \{x\in\IZ/2^v\IZ: \aord(x)=2^v\} & \{x\in\IZ/2^v\IZ: \aord(x)=2^v\} \\
     \{x\in\IZ/2^v\IZ: \aord(x)=2^{t'}\} & \{x\in\IZ/2^v\IZ: \aord(x)=2^{t'}\} \\};
  \path[-stealth]
    (m-1-1) edge node[above] {$x\mapsto 5x$} (m-1-2)
    (m-1-1) edge node [left] {$x\mapsto 2^{v-t'}x$} (m-2-1)
		(m-1-2) edge node [right] {$x\mapsto 2^{v-t'}x$} (m-2-2)
		(m-2-1) edge node [below] {$x\mapsto 5x$} (m-2-2);
\end{tikzpicture}
\end{center}
\phantomsection\label{not108}
that $x\mapsto 5x$ has \emph{at most} two cycles on the set of order $2^{t'}$ elements, namely the ones spanned by $2^{v-t'}\cdot1=2^{v-t'}$ and $2^{v-t'}\cdot(-1)=-2^{v-t'}$. Likewise, the commutativity of the diagram
\begin{center}
\begin{tikzpicture}
\matrix (m) [matrix of math nodes,row sep=3em,column sep=4em,minimum width=2em]
  {
     \{x\in\IZ/2^v\IZ: \aord(x)=2^{t'}\} & \{x\in\IZ/2^v\IZ: \aord(x)=2^{t'}\} \\
     \{x\in\IZ/2^v\IZ: \aord(x)=8\} & \{x\in\IZ/2^v\IZ: \aord(x)=8\} \\};
  \path[-stealth]
    (m-1-1) edge node[above] {$x\mapsto 5x$} (m-1-2)
    (m-1-1) edge node [left] {$x\mapsto 2^{t'-3}x$} (m-2-1)
		(m-1-2) edge node [right] {$x\mapsto 2^{t'-3}x$} (m-2-2)
		(m-2-1) edge node [below] {$x\mapsto 5x$} (m-2-2);
\end{tikzpicture}
\end{center}
implies that $2^{v-t'}$ and $-2^{v-t'}$ lie on \emph{distinct} cycles of $x\mapsto 5x$. This shows that $\Lcal(2^v,5)$ can be chosen as indicated in Table \ref{crlListAutomorphismTable}.

As for other values of $a$, if $a\equiv1\Mod{4}$, then $a$ is congruent to a power of $5$ modulo $2^v$, and the choice for $\Lcal(2^v,a)$ specified in Table \ref{crlListAutomorphismTable} can be derived from the one for $\Lcal(2^v,5)$ using Lemma \ref{crlListPowerLem} (analogously to the end of the argument for $p>2$ above).

It remains to deal with the case $a\equiv3\Mod{4}$. Then $-a\equiv1\Mod{4}$. We view the automorphism $\mu_a$ of $\IZ/2^v\IZ$ as the composition of the automorphisms $\mu_{-a}$ and $\mu_{-1}$. For each $t'\in\{0,1,\ldots,v\}$, we want to understand the cycles of $\mu_a$ on the set of elements $x$ of additive order $2^{t'}$, and we do so by distinguishing some cases for $t'$.
\begin{itemize}
\item If $t'\in\{0,1\}$ (i.e., $x\in\{0,2^{v-1}\}$), then $x$, being the only element in $\IZ/2^v\IZ$ of its additive order, is fixed by $\mu_a$.
\item If $t'\in\{2,3,\ldots,v-\log_2(\ord(-a))\}$, then $\mu_{-a}$ fixes each element of order $2^{t'}$; this can be seen by using the formula for \enquote{$a\equiv1\Mod{4}$} in Table \ref{crlListAutomorphismTable} and noting that if $x\in\{a^j2^t,-a^j2^t\}$, then $2^{t'}=\aord(x)=2^{v-t}$. It follows that the restriction of $\mu_a$ to set of order $2^{t'}$ elements is the same as that of $\mu_{-1}$. Therefore,
\[
\left\{(j\cdot\ord(-a),2): j=1,2,\ldots,\frac{2^{v-1}}{\ord(-a)}-1\right\},
\]
which is a CRL-list of the restriction of $\mu_{-1}$ to the set of elements of $\IZ/2^v\IZ$ with order in $\{2^2,2^3,\ldots,2^{v-\log_2(\ord(-a))}\}$, is also a CRL-list of the corresponding restriction of $\mu_a$.
\item Finally, if $t'>v-\log_2(\ord(-a))$, then all cycles of $\mu_{-a}$ on the set of order $2^{t'}$ elements in $\IZ/2^v\IZ$ are of even length (in fact, their length is a nontrivial power of $2$). For a fixed $j\in\{0,1,\ldots,\gcd(\frac{2^{v-2}}{\ord(-a)},2^{t'-2})-1\}=\{0,1,\ldots,\frac{2^{v-2}}{\ord(-a)}-1\}$, we consider the two cycles of $\mu_{-a}$ spanned by $5^j2^{v-t'}$ and $-5^j2^{v-t'}$, respectively. These cycles are distinct (according to the case \enquote{$a\equiv1\Mod{4}$} in Table \ref{crlListAutomorphismTable}, applied to $-a$), both have length
\[
\frac{2^{t'-2}}{\gcd(\frac{2^{v-2}}{\ord(-a)},2^{t'-2})}=\frac{2^{t'-2}}{2^{v-2-\log_2(\ord(-a))}}=2^{t'-v+\log_2(\ord(-a))}
\]
and are images of each other under $\mu_{-1}$. In fact, since $\mu_{-a}$ and $\mu_{-1}$ commute, we have $\mu_{-1}(\mu_{-a}^n(\pm5^j2^{v-t'}))=\mu_{-a}^n(\mu_{-1}(\pm5^j2^{v-t'}))$ for each $n\in\IZ$, which leads to the following diagrammatic picture of how the cycles are matched under $\mu_{-1}$, setting $k:=2^{t'-v+\log_2(\ord(-a))-1}$, so that $k$ is half of the common cycle length of $\pm5^j2^{v-t'}$ under $\mu_{-a}$:
\begin{center}
\begin{tikzpicture}
\matrix (m) [matrix of math nodes,row sep=3em,column sep=1em]
  {
     5^j2^{v-t'} & (-a)5^j2^{v-t'} & \cdots & (-a)^{2k-1}5^j2^{v-t'} & 5^j2^{v-t'} \\
     -5^j2^{v-t'} & -(-a)5^j2^{v-t'} & \cdots & -(-a)^{2k-1}5^j2^{v-t'} & -5^j2^{v-t'} \\};
  \path[-stealth]
    (m-1-1) edge node[above] {$\mu_{-a}$} (m-1-2)
		(m-1-2) edge node[above] {$\mu_{-a}$} (m-1-3)
		(m-1-3) edge node[above] {$\mu_{-a}$} (m-1-4)
		(m-1-4) edge node[above] {$\mu_{-a}$} (m-1-5)
		(m-2-1) edge node[above] {$\mu_{-a}$} (m-2-2)
		(m-2-2) edge node[above] {$\mu_{-a}$} (m-2-3)
		(m-2-3) edge node[above] {$\mu_{-a}$} (m-2-4)
		(m-2-4) edge node[above] {$\mu_{-a}$} (m-2-5)
    (m-1-1) edge node [left] {$\mu_{-1}$} (m-2-1)
		(m-1-2) edge node [left] {$\mu_{-1}$} (m-2-2)
		(m-1-3) edge node [left] {$\mu_{-1}$} (m-2-3)
		(m-1-5) edge node [left] {$\mu_{-1}$} (m-2-5);
\end{tikzpicture}
\end{center}
It follows that $\mu_a=\mu_{-1}\circ\mu_{-a}$ decomposes into two cycles on the union of the support sets of the above two cycles of $\mu_{-a}$: one is spanned by $5^j2^{v-t'}$ and consists of the \enquote{even elements} $5^j2^{v-t'},(-a)^25^j2^{v-t'},(-a)^45^j2^{v-t'},\ldots$ of the upper cycle as well as the \enquote{odd elements} of the lower cycle, whereas the other is spanned by $-5^j2^{v-t'}$ and consists of the \enquote{odd elements} of the upper and \enquote{even elements} of the lower cycle. If we let $t'$ and $j$ run through their respective range, the corresponding cycle pairs partition the set of all elements of additive order larger than $\frac{2^v}{\ord(-a)}$. Together with the observations for smaller values of $t'$ from above, we obtain a CRL-list for $\mu_a$, which can be easily checked to coincide with the one specified in Table \ref{crlListAutomorphismTable} (we note that $t'=v-t$ in the notation of that table).
\end{itemize}
\end{proof}

Now we tackle Step (2) in our plan for this subsection, i.e., working out a CRL-list for every affine permutation $A$ of every finite primary cyclic group. A useful observation, which we explain in more detail, is that the case where $A$ has a fixed point can be reduced to the automorphism case (i.e., to Lemma \ref{crlListAutomorphismLem}). This idea appears in \cite[proof of Lemma 4]{Den13a}, which was concerned with cyclic groups, but it can easily be extended to general groups. We remind the reader that we write $\rho_{\rrm}$ for the right-regular representation of a group (that must be clear from context) on itself.

\begin{lemmma}\label{affineFixedPointLem}
Let $G$ be a group, $b\in G$, $\alpha$ an automorphism of $G$, $A$ the affine permutation $x\mapsto x^{\alpha}b$ of $G$, and $\ffrak\in G$ a fixed point of $A$. Then $A=\rho_{\rrm}(\ffrak)^{-1}\alpha\rho_{\rrm}(\ffrak)$. In particular, $\rho_{\rrm}(\ffrak)$ is a digraph isomorphism from $\Gamma_{\alpha}$ to $\Gamma_A$.
\end{lemmma}

\begin{proof}
The equality $x^A=x$ is equivalent to $b=(\ffrak^{-1})^{\alpha}\ffrak$. For each $x\in G$, we have
\[
x^{\rho_{\rrm}(\ffrak)^{-1}\alpha\rho_{\rrm}(\ffrak)}=(x\ffrak^{-1})^{\alpha\rho_{\rrm}(\ffrak)}=x^{\alpha}(\ffrak^{-1})^{\alpha}\ffrak=x^{\alpha}b=x^A,
\]
as required. The \enquote{In particular} statement follows because of the well-known (and easy to verify) fact that for each set $X$ and all $\psi,\psi'\in\Sym(X)$, the permutation $\psi$ maps $x\in X$ to $y\in X$ if and only if its $\psi'$-conjugate $(\psi')^{-1}\psi\psi'$ maps $x^{\psi'}$ to $y^{\psi'}$.
\end{proof}

Lemma \ref{affineFixedPointLem} is interesting for us because of the following elementary observation.

\begin{lemmma}\label{crlListConjugateLem}
Let $X$ be a finite set, and let $\psi_1,\psi_2\in\Sym(X)$ be conjugate permutations, say $\psi_2=(\psi')^{-1}\psi_1\psi'$. If $\Lcal$ is a CRL-list for $\psi_1$, then $\{(r^{\psi'},l): (r,l)\in\Lcal\}$ is a CRL-list for $\psi_2$.
\end{lemmma}

\begin{proof}
This is clear because $(y_0,y_1,\ldots,y_{l-1})$ is a cycle of $\psi_1$ if and only if
\[
(y_0^{\psi'},y_1^{\psi'},\ldots,y_{l-1}^{\psi'})
\]
is a cycle of $\psi_2$ (see the last sentence in the proof of Lemma \ref{affineFixedPointLem}).
\end{proof}

Through combining Lemmas \ref{affineFixedPointLem} and \ref{crlListConjugateLem}, we get the next result.

\begin{lemmma}\label{crlListFixedPointLem}
Let $G$ be a finite group, $b\in G$, $\alpha$ an automorphism of $G$, $A$ the affine permutation $x\mapsto x^{\alpha}b$ of $G$, and $\ffrak\in G$ a fixed point of $A$. If $\Lcal$ is a CRL-list of $\alpha$, then $\{(r\ffrak,l): (r,l)\in\Lcal\}$ is a CRL-list of $A$.
\end{lemmma}

We are now ready to construct a CRL-list for each affine permutation of each finite primary cyclic group.

\begin{propposition}\label{crlListPrimaryProp}
Let $p$ be a prime, $v\in\IN^+$, and $a,b\in\IZ$ with $p\nmid a$. Table \ref{crlListPrimaryTable} provides a CRL-list $\Lcal(p^v,a,b)$\phantomsection\label{not109} for the affine permutation $A:x\mapsto ax+b$ of $\IZ/p^v\IZ$, using the following conditional notations.
\begin{itemize}
\item If $p$ is odd, we denote by $\rfrak$ a fixed primitive root modulo $p^v$.
\item If $\nu_p^{(v)}(b)\geq\nu_p^{(v)}(a-1)$, we set
\[
\ffrak:=-\frac{b}{p^{\nu_p^{(v)}(a-1)}}\cdot\inv_{p^{v-\nu_p^{(v)}(a-1)}}\left(\frac{a-1}{p^{\nu_p^{(v)}(a-1)}}\right).
\]
\end{itemize}
\end{propposition}

\begin{longtable}[h]{|c|c|c|}\hline
No. & Case for $p^v,a,b$ & elements of $\Lcal(p^v,a,b)$ \\ \hline
1 & \thead{$p>2$, \\ $\nu_p^{(v)}(b)\geq\nu_p^{(v)}(a-1)$} & \thead{$\left(\rfrak^jp^t+\ffrak,\frac{\phi(p^{v-t})}{\gcd\left(\frac{\phi(p^v)}{\ord(a)},\phi(p^{v-t})\right)}\right)$ \\ for $t=0,1,\ldots v$ and \\ $j=0,1,\ldots,\gcd\left(\frac{\phi(p^v)}{\ord(a)},\phi(p^{v-t})\right)-1$.} \\ \hline
2 & \thead{$p>2$, \\ $\nu_p^{(v)}(b)<\nu_p^{(v)}(a-1)$} & $\left(j,p^{v-\nu_p^{(v)}(b)}\right)$ for $j=0,1,\ldots,p^{\nu_p^{(v)}(b)}-1$. \\ \hline
3 & $p=2$, $v\leq2$, $a=1$ & $(j,\aord(b))$ for $j=0,1,\ldots,\frac{2^v}{\aord(b)}-1$. \\ \hline
4 & $p^v=4$, $a=3$, $b=0$ & $(0,1)$, $(1,2)$, $(2,1)$. \\ \hline
5 & $p^v=4$, $a=3$, $b=2$ & $(0,2)$, $(1,1)$, $(3,1)$. \\ \hline
6 & $p^v=4$, $a=3$, $2\nmid b$ & $(0,2)$, $(2,2)$. \\ \hline
7 & \thead{$p=2$, $v\geq3$, \\ $\nu_2^{(v)}(b)\geq\nu_2^{(v)}(a-1)$, \\ $a\equiv1\Mod{4}$} & \thead{$(\ffrak,1)$, $(2^{v-1}+\ffrak,1)$; \\ $\left(5^j2^t+\ffrak,\frac{2^{v-t-2}}{\gcd\left(\frac{2^{v-2}}{\ord(a)},2^{v-t-2}\right)}\right)$, $\left(-5^j2^t+\ffrak,\frac{2^{v-t-2}}{\gcd\left(\frac{2^{v-2}}{\ord(a)},2^{v-t-2}\right)}\right)$ \\ for $t=0,1,\ldots,v-2$ and \\ $j=0,1,\ldots,\gcd\left(\frac{2^{v-2}}{\ord(a)},2^{v-t-2}\right)-1$.}  \\ \hline
8 & \thead{$p=2$, $v\geq3$, \\ $\nu_2^{(v)}(b)\geq\nu_2^{(v)}(a-1)$, \\ $a\equiv3\Mod{4}$} & \thead{$(\ffrak,1)$, $(2^{v-1}+\ffrak,1)$; \\ $(j\cdot\ord(-a)+\ffrak,2)$ for $j=1,2,\ldots,\frac{2^{v-1}}{\ord(-a)}-1$; \\ $\left(5^j2^t+\ffrak,\frac{\ord(-a)}{2^t}\right)$, $\left(-5^j2^t+\ffrak,\frac{\ord(-a)}{2^t}\right)$ \\ for $t=0,1,\ldots,\log_2(\ord(-a))-1$ \\ and $j=0,1,\ldots,\frac{2^{v-2}}{\ord(-a)}-1$.} \\ \hline
9 & \thead{$p=2$, $v\geq3$, \\ $\nu_2^{(v)}(b)<\nu_2^{(v)}(a-1)$, \\ $a\equiv1\Mod{4}$} & $\left(j,2^{v-\nu_2^{(v)}(b)}\right)$ for $j=0,1,\ldots,2^{\nu_2^{(v)}(b)}-1$. \\ \hline
10 & \thead{$p=2$, $v\geq3$, \\ $\nu_2^{(v)}(b)<\nu_2^{(v)}(a-1)$, \\ $a\equiv3\Mod{4}$} & $(b\cdot j,2\ord(-a))$ for $j=0,2,3,4,\ldots,\frac{2^{v-1}}{\ord(-a)}$. \\ \hline
\caption{CRL-lists of affine permutations of finite primary cyclic groups.}
\label{crlListPrimaryTable}
\end{longtable}

\begin{proof}
First, we observe that if $\nu_p^{(v)}(b)\geq\nu_p^{(v)}(a-1)$, then $\ffrak$ is a fixed point of $A$. Indeed, $x\in\IZ/p^v\IZ$ is a fixed point of $A$ if and only if
\begin{align*}
&ax+b\equiv x\Mod{p^v} \Leftrightarrow (a-1)x\equiv-b\Mod{p^v}  \\
&\Leftrightarrow \frac{a-1}{p^{\nu_p^{(v)}(a-1)}}x\equiv-\frac{b}{p^{\nu_p^{(v)}(a-1)}}\Mod{p^{v-\nu_p^{(v)}(a-1)}} \\
&\Leftrightarrow x\equiv-\frac{b}{p^{\nu_p^{(v)}(a-1)}}\cdot\inv_{p^{v-\nu_p^{(v)}(a-1)}}\left(\frac{a-1}{p^{\nu_p^{(v)}(a-1)}}\right)=\ffrak\Mod{p^{v-\nu_p^{(v)}(a-1)}}.
\end{align*}
The form of the CRL-list for $A$ specified in cases 1, 7 and 8 in Table \ref{crlListPrimaryTable} thus follows from Lemma \ref{crlListFixedPointLem} and the corresponding CRL-list for $\mu_a$, read off from Table \ref{crlListAutomorphismTable}. Moreover, cases 3--6 in Table \ref{crlListPrimaryTable} are easy to check separately. It remains to justify the specified CRL-list in cases 2, 9 and 10 in Table \ref{crlListPrimaryTable}, which we do now.
\begin{itemize}
\item Case 2: We note that in this case, $a\equiv1\Mod{p}$ necessarily. The units modulo $p^v$ that are congruent to $1$ modulo $p$ form the unique, cyclic Sylow $p$-subgroup of $(\IZ/p^v\IZ)^{\ast}$, of order $p^{v-1}$. For each $t\in\{0,1,\ldots,v-1\}$, the unit $1+p^{v-t}$ has order $p^t$, and thus all order $p^t$ units modulo $p^v$ are powers of $1+p^{v-t}$ with exponent coprime to $p$ and vice versa. Therefore, using the \enquote{In particular} statement of Lemma \ref{crlListPowerLem} and that $\ord(A)$ is a power of $p$, we may assume without loss of generality that $a=1+p^{v-t}$ for some $t\in\{0,1\ldots,v-1\}$. We observe that $v-t=\nu_p^{(v)}(a-1)$, and thus $v-t>\nu_p^{(v)}(b)$ by the case assumptions. For each $x\in\IZ/p^v\IZ$, we have $A(x)=ax+b=(1+p^{v-t})x+b\equiv x\Mod{p^{\nu_p^{(v)}(b)}}$. Hence, the elements $0,1,\ldots,p^{\nu_p^{(v)}(b)}-1$ lie on pairwise distinct cycles of $A$. On the other hand, by \cite[Table 3]{BW22b}, $A$ has exactly $\frac{p^v}{\aord(b)}=p^{\nu_p^{(v)}(b)}$ cycles, each of length $\aord(b)=p^{v-\nu_p^{(v)}(b)}$, so the said elements are representatives for all cycles of $A$ and $\{(j,p^{v-\nu_p^{(v)}(b)}): j=0,1,\ldots,p^{\nu_p^{(v)}(b)}-1\}$ is a CRL-list of $A$, as required.
\item Case 9: This can be dealt with similarly to Case 2. We observe that the unit $5=1+2^2$ has multiplicative order $2^{v-2}$ and generates an index $2$ cyclic subgroup of $(\IZ/2^v\IZ)^{\ast}$, which consists precisely of those units that are congruent to $1$ modulo $4$. For each $t\in\{0,1,\ldots,v-2\}$, the unit $1+2^{v-t}$ lies in this subgroup and has order $2^t$, so any unit of order $2^t$ that is congruent to $1$ modulo $4$ is a power of $1+2^{v-t}$ with odd exponent and vice versa. Using the \enquote{In particular} statement of Lemma \ref{crlListPowerLem} and that $\ord(A)$ is a power of $2$, we may assume without loss of generality that $a=1+2^{v-t}$ for some $t\in\{0,1,\ldots,v-2\}$, and the remainder of this argument is analogous to the one for case 2, resulting in $\{(j,2^{v-\nu_2^{(v)}(b)}): j=0,1,\ldots,2^{\nu_2^{(v)}(b)}-1\}$ being a CRL-list of $A$.
\item Case 10: Due to $-a\equiv1\Mod{4}$, we may assume without loss of generality that $-a=1+2^{v-t}$ for some $t\in\{0,1,\ldots,v-2\}$ (see the argument for case 9). Let $A'$ be the affine function $x\mapsto -x+b$ of $\IZ/2^{v-t}\IZ$. For each $x\in\IZ/2^v\IZ$, we have $A(x)=-(1+2^{v-t})x+b\equiv -x+b\Mod{2^{v-t}}$. This means that elements of $\IZ/2^{v-t}\IZ$ that lie on different cycles of $A'$ also lie on different cycles of $A$ (we remind the reader that $\IZ/2^{v-t}\IZ\subseteq\IZ/2^v\IZ$ by our convention on the underlying set of $\IZ/m\IZ$ stated at the beginning of this subsection). Now, $A'$ is an involution without fixed points (because $2\nmid b$) and thus consists of $2^{v-t-1}$ transpositions. But by \cite[Table 4]{BW22b}, $A$ has exactly $2^{v-t-1}$ cycles. Indeed, if $a=-5^n$, then $2^t=\ord(-a)=2^{v-2-\nu_2^{(v-2)}(n)}$, and therefore $t=v-2-\nu_2^{(v-2)}(n)$, whence the cycle number $2^{1+\nu_2^{(v-2)}(n)}$ specified in \cite[Table 4]{BW22b} equals $2^{v-t-1}$. Therefore, any set of of representatives for the cycles of $A'$ on $\IZ/2^{v-t}\IZ$ is also a set of representatives for the cycles of $A$ on $\IZ/2^v\IZ$, all of which are of length $2^{t+1}$ by \cite[Table 4]{BW22b}. Thus, in order to find a CRL-list for $A$, it suffices to find cycle representatives for $A'$. To that end, we first assume that $b=1$. Then every cycle (i.e., transposition) of $A'$ on $\IZ/2^{v-t}\IZ$ contains exactly one element from the \enquote{left half} $\{1,2,\ldots,2^{v-t-1}\}$ and one from the \enquote{right half} $\{2^{v-t-1}+1,2^{v-t-1}+2,\ldots,2^{v-t}-1,2^{v-t}\equiv 0\}$. It follows that $\{0,2,3,4,\ldots,2^{v-t-1}\}$ is a set of representatives for the cycles of $A'$, and this matches with the CRL-list for $A$ specified in Table \ref{crlListPrimaryTable}. For general $b$, we observe that
\[
A'=(x\mapsto -x+b)=(x\mapsto b^{-1}x)\cdot(x\mapsto -x+1)\cdot(x\mapsto bx),
\]
whence Lemma \ref{crlListConjugateLem} allows us to conclude that
\[
b\cdot\{0,2,3,4,\ldots,2^{v-t-1}\}=\{0,2b,3b,4b,\ldots,2^{v-t-1}b\}
\]
is a set of representatives for the cycles of $A'$, as required.
\end{itemize}
\end{proof}

Now that we know a CRL-list for each affine permutation of each finite \emph{primary} cyclic group, let us discuss how to deal with general finite cyclic groups. Through identifying the group $\IZ/m\IZ$ with the direct product $\prod_{p\mid m}{\IZ/p^{\nu_p(m)}\IZ}$ via the Chinese Remainder Theorem, we can view any affine permutation $A:x\mapsto ax+b$ of $\IZ/m\IZ$ as the \enquote{function tensor product} $\bigotimes_{p\mid m}A_p$, where $A_p$ is the \emph{reduction of $A$ modulo $p^{\nu_p(m)}$}, i.e., the affine permutation $x\mapsto ax+b$ of $\IZ/p^{\nu_p(m)}\IZ$, as introduced in Remark \ref{tensorProductRem}. That is, $A$ becomes the component-wise application of its reductions $A_p$ under this identification. This leads to the following, more general problem, which we solve next.

\begin{problemm}\label{crlListTensorProb}
Given finite sets $X_1,X_2,\ldots,X_n$, permutations $\psi_j\in\Sym(X_j)$ for $j=1,2,\ldots,n$, and a CRL-list $\Lcal_j$ of $\psi_j$ for $j=1,2,\ldots,n$, construct a CRL-list $\Lcal$ of $\psi:=\bigotimes_{j=1}^n{\psi_j}\in\Sym(\prod_{j=1}^n{X_j})$.
\end{problemm}

For the rest of this subsection, we use the notation fixed in Problem \ref{crlListTensorProb}. We denote by $\Lcal_j^{(1)}\subseteq X_j$\phantomsection\label{not110} the set of first entries of the pairs in $\Lcal_j$ (i.e., the set of cycle representatives of $\psi_j$ exhibited by $\Lcal_j$), and for $r\in\Lcal_j^{(1)}$, we denote by $r^{\langle\psi_j\rangle}$\phantomsection\label{not111} the orbit of $r$ under the action of the permutation group $\langle\psi_j\rangle$ (i.e., the set of points on the $\psi_j$-cycle of $r$).

For each $\vec{r}=(r_1,r_2,\ldots,r_n)\in\prod_{j=1}^n{\Lcal_j^{(1)}}$, we set $B_{\vec{r}}:=\prod_{j=1}^n{r_j^{\langle\psi_j\rangle}}$. These sets $B_{\vec{r}}$ form a partition of $\prod_{j=1}^n{X_j}$, and each set $B_{\vec{r}}$ is a union of cycles of $\psi$. Therefore, it suffices to find a CRL-list $\Lcal_{\vec{r}}$ of the restriction $\psi_{\mid B_{\vec{r}}}$ for each $\vec{r}$, then set $\Lcal:=\bigcup_{\vec{r}}{\Lcal_{\vec{r}}}$.

Let us thus assume that $\vec{r}$ is fixed. For $j=1,2,\ldots,n$, we denote by $l_j=l_j(\vec{r})$ the $\psi_j$-cycle length of $r_j$. Every cycle of $\psi$ on $B_{\vec{r}}$ has length $l_{\vec{r}}:=\lcm(l_1,l_2,\ldots,l_n)$, and there are exactly $\left(\prod_{j=1}^n{l_j}\right)/l_{\vec{r}}$ such cycles (see also \cite[Lemma 2.1]{WX93a}). It remains to find representatives for them.

We consider the bijection
\[
\iota_{\vec{r}}:\prod_{j=1}^n{\IZ/l_j\IZ}\rightarrow B_{\vec{r}}, (k_1,k_2,\ldots,k_n)\mapsto (\psi_1^{k_1}(r_1),\psi_2^{k_2}(r_2),\ldots,\psi_n^{k_n}(r_n)).
\]
If\phantomsection\label{not112} we identify $B_{\vec{r}}$ with $\prod_{j=1}^n{\IZ/l_j\IZ}$ via this bijection, then the action of $\psi$ on $B_{\vec{r}}$ turns into that of the function
\[
\sfrak_{\vec{r}}:\prod_{j=1}^n{\IZ/l_j\IZ}\rightarrow\prod_{j=1}^n{\IZ/l_j\IZ}, (k_1,k_2,\ldots,k_n)\mapsto (k_1+1,k_2+1,\ldots,k_n+1),
\]
each\phantomsection\label{not113} displayed addition being modulo the corresponding $l_j$, of course. So it suffices to find a set of representatives for the cycles of $\sfrak_{\vec{r}}$ on $\prod_{j=1}^n{\IZ/l_j\IZ}$, then map that set under $\iota_{\vec{r}}$. In order to describe a particular set of cycle representatives for $\sfrak_{\vec{r}}$ neatly, we introduce the following auxiliary concepts.

\begin{deffinition}\label{admissibleGoodDef}
We denote by $\pi(l_{\vec{r}})$\phantomsection\label{not114} the set of all prime divisors of $l_{\vec{r}}$.
\begin{enumerate}
\item A function $\Ical:\pi(l_{\vec{r}})\rightarrow\{1,2,\ldots,n\}$\phantomsection\label{not115} is an \emph{$\vec{r}$-admissible indexing function}\phantomsection\label{term47} if for each $p\in\pi(l_{\vec{r}})$, we have $\nu_p(l_{\Ical(p)})=\max\{\nu_p(l_j): j=1,2,\ldots,n\}$.
\item If $\Ical$ is an $\vec{r}$-admissible indexing function, then a tuple
\[
(k_1,k_2,\ldots,k_n)\in\prod_{j=1}^n{\IZ/l_j\IZ}
\]
is \emph{$\Ical$-good}\phantomsection\label{term48} if $k_{\Ical(p)}\equiv 0\Mod{p^{\nu_p(l_{\Ical(p)})}}$ for each $p\in\pi(l_{\vec{r}})$.
\item For each $\vec{r}$-admissible indexing function $\Ical$, we denote by $\Good_{\vec{r}}(\Ical)$\phantomsection\label{not116} the set of all $\Ical$-good tuples in $\prod_{j=1}^n{\IZ/l_j\IZ}$.
\end{enumerate}
\end{deffinition}

The following result solves Problem \ref{crlListTensorProb}.

\begin{propposition}\label{crListTensorProp}
Let $\Ical$ be an $\vec{r}$-admissible indexing function. Then $\Good_{\vec{r}}(\Ical)$ is a set of representatives for the cycles of $\sfrak_{\vec{r}}$ on $\prod_{j=1}^n{\IZ/l_j\IZ}$. Equivalently,
\[
\iota_{\vec{r}}(\Good_{\vec{r}}(\Ical))\times\{l_{\vec{r}}\}
\]
is a CRL-list for $\psi_{\mid B_{\vec{r}}}$, and so
\[
\bigcup_{\vec{r}\in\prod_{j=1}^n{\Lcal^{(1)}_j}}{\left(\iota_{\vec{r}}(\Good_{\vec{r}}(\Ical))\times\{l_{\vec{r}}\}\right)}
\]
is a CRL-list for $\psi$.
\end{propposition}

\begin{proof}
By definition and the Chinese Remainder Theorem, the number of $\Ical$-good tuples in $\prod_{j=1}^n{\IZ/l_j\IZ}$ is
\[
\frac{\prod_{j=1}^n{l_j}}{\prod_{p\in\pi(l_{\vec{r}})}{p^{\nu_p(l_{\Ical(p)})}}}=\frac{|\prod_{j=1}^n{\IZ/l_j\IZ}|}{l_{\vec{r}}},
\]
which is also the number of cycles of $\sfrak_{\vec{r}}$ on $\prod_{j=1}^n{\IZ/l_j\IZ}$. Hence, it suffices to show that different $\Ical$-good tuples lie on distinct cycles of $\sfrak_{\vec{r}}$. Let $\vec{k}=(k_1,k_2,\ldots,k_n)$ and $\vec{k'}=(k'_1,k'_2,\ldots,k'_n)$ be $\Ical$-good tuples that lie on the same cycle of $\sfrak_{\vec{r}}$. This means that there is a $t\in\IZ$ such that $k_j+t\equiv k'_j\Mod{l_j}$ for each $j=1,2,\ldots,n$. Now, let $p\in\pi(l_{\vec{r}})$. Since $t\equiv k'_{\Ical(p)}-k_{\Ical(p)}\Mod{l_{\Ical(p)}}$ and $\vec{k},\vec{k'}$ are $\Ical$-good, it follows that $t\equiv0\Mod{p^{\nu_p(l_{\Ical(p)})}}$. Because this holds for every $p\in\pi(l_{\vec{r}})$, we conclude that
\[
\lcm(l_1,l_2,\ldots,l_n)=l_{\vec{r}}=\prod_{p\in\pi(l_{\vec{r}})}{p^{\nu_p(l_{\Ical(p)})}}
\]
divides $t$, whence $k_j\equiv k'_j\Mod{l_j}$ for each $j=1,2,\ldots,n$. This means that $\vec{k}=\vec{k'}$, as required.
\end{proof}

\subsection{Affine discrete logarithms and cycle lengths}\label{subsec2P4}

Let $m\geq1$ be an integer, and let $a,b\in\IZ/m\IZ$ with $\gcd(a,m)=1$. We consider the affine permutation $A:x\mapsto ax+b$ of $\IZ/m\IZ$. Given $x,y\in\IZ/m\IZ$, we set
\[
\log_A^{(m)}(x,y):=
\begin{cases}
\infty, & \text{if there is no }k\in\IZ\text{ with }A^k(x)=y, \\
\min\{k\in\IN_0: A^k(x)=y\}, & \text{otherwise}.
\end{cases}
\]
\phantomsection\label{not117}
In this short subsection, we discuss how to compute $\log_A^{(m)}(x,y)$ and the cycle length of $x$ under $A$, which is closely related, as it is the minimal \emph{positive} integer $k$ such that $A^k(x)=x$ (while $\log_A^{(m)}(x,x)=0$). It is not surprising that modular discrete logarithms play an important role in this, because they are a special case of the notion $\log_A^{(m)}(x,y)$. Namely, the discrete logarithm of $x\in(\IZ/m\IZ)^{\ast}$ modulo $m$ with base $a\in(\IZ/m\IZ)^{\ast}$, written $\log_a^{(m)}(x)$\phantomsection\label{not118}, is equal to $\log_{\mu_a}^{(m)}(1,x)$.

In order to discuss the computational details, we make a case distinction.
\begin{itemize}
\item First, we assume that $a\equiv1\Mod{m}$, a simple case for which no discrete logarithms need to be computed. Indeed, one then has $A^k(x)=x+kb\equiv y\Mod{m}$ if and only if $kb\equiv y-x\Mod{m}$. That last congruence is solvable in $k$ if and only if $\gcd(b,m)\mid y-x$, in which case the congruence is equivalent to
\begin{equation}\label{canceledCong}
\frac{b}{\gcd(b,m)}k\equiv\frac{y-x}{\gcd(b,m)}\Mod{\frac{m}{\gcd(b,m)}}
\end{equation}
and has the minimal solution
\[
\left(\frac{y-x}{\gcd(b,m)}\cdot\inv_{m/\gcd(b,m)}\left(\frac{b}{\gcd(b,m)}\right)\right)\bmod{\frac{m}{\gcd(b,m)}} = \log_A^{(m)}(x,y).
\]
We note that in case $x=y$, the minimal \emph{positive} solution of congruence (\ref{canceledCong}), and thus the cycle length of $x$ under $A$ modulo $m$, is $m/\gcd(b,m)$.
\item Now we assume that $a\not\equiv1\Mod{m}$. Then $m>1$, and $a\not\equiv0\Mod{m}$ due to $\gcd(a,m)=1$, so we may assume that as an integer, $a>1$. We have $A^k(x)=y$ if and only if
\begin{align}\label{discreteLogEq1}
\notag &a^kx+\frac{a^k-1}{a-1}b\equiv y\Mod{m} \\
\notag \Leftrightarrow &a^k(a-1)x+(a^k-1)b\equiv (a-1)y\Mod{(a-1)m} \\
\Leftrightarrow &a^k((a-1)x+b)\equiv (a-1)y+b\Mod{(a-1)m}.
\end{align}
In order for congruence (\ref{discreteLogEq1}) to be solvable in $k$, it is necessary that $(a-1)x+b$ and $(a-1)y+b$ have the same additive order modulo $(a-1)m$, i.e., that
\begin{equation}\label{discreteLogEq2}
\gcd((a-1)x+b,(a-1)m)=\gcd((a-1)y+b,(a-1)m)=:\dfrak.
\end{equation}
If\phantomsection\label{not119} condition (\ref{discreteLogEq2}) is satisfied, then congruence (\ref{discreteLogEq1}) is equivalent to
\[
a^k\cdot\frac{(a-1)x+b}{\dfrak}\equiv\frac{(a-1)y+b}{\dfrak}\Mod{\frac{(a-1)m}{\dfrak}},
\]
i.e., to
\begin{equation}\label{canceledCong2}
a^k\equiv\frac{(a-1)y+b}{\dfrak}\cdot\inv_{(a-1)m/\dfrak}\left(\frac{(a-1)x+b}{\dfrak}\right)\Mod{\frac{(a-1)m}{\dfrak}},
\end{equation}
which shows that
\[
\log_A^{(m)}(x,y)=\log_a^{((a-1)m/\dfrak)}\left(\frac{(a-1)y+b}{\dfrak}\cdot\inv_{(a-1)m/\dfrak}\left(\frac{(a-1)x+b}{\dfrak}\right)\right),
\]
with the convention that $\log_a^{(m)}(x)=\infty$ if $x$ is not a power of $a$ modulo $m$. If $x=y$, then the right-hand side in congruence (\ref{canceledCong2}) simplifies to $1$, whence the cycle length of $x$ under $A$ equals the multiplicative order of $a$ modulo $(a-1)m/\dfrak$.
\end{itemize}
The upshot of this discussion is that $\log_A^{(m)}(x,y)$ and the cycle length of $x$ under $A$ modulo $m$ can be computed efficiently if one has efficient algorithms for computing discrete logarithms and multiplicative orders of units in $(\IZ/m\IZ)^{\ast}$. Hence, $\log_A^{(m)}(x,y)$ can be computed efficiently on a quantum computer. Indeed, Shor showed that such computers admit efficient algorithms both for computing discrete logarithms and for integer factorization \cite{Sho94a}, the latter of which is sufficient to compute element orders in $(\IZ/m\IZ)^{\ast}$ efficiently; in fact, all one needs for that is an explicit factorization of the Euler totient function value $\phi(m)$, see also the proof of Lemma \ref{complexitiesLem2}(2).

\section{Functional graphs of generalized cyclotomic mappings}\label{sec3}

Let $f$ be a generalized cyclotomic mapping of $\IF_q$ of index $d$. From the introduction, we recall our notation $C_i$ for $i\in\{0,1,\ldots,d\}$, where $C_d=\{0_{\IF_q}\}$, and $C_i=\omega^iC$ for $i<d$ is a coset of the index $d$ subgroup $C$ of $\IF_q^{\ast}=\langle\omega\rangle$. Moreover, we recall that for each $i\in\{0,1,\ldots,d-1\}$, we have a natural bijection $\iota_i:\IZ/s\IZ\rightarrow C_i=\omega^iC$, $x\mapsto\omega^{i+dx}$, by virtue of which we view $C_i$ as a copy of $\IZ/s\IZ$. As long as $f$ does not map $C_i$ constantly to $C_d=\{0\}$, this allows us to view the restriction $f_{\mid C_i}$ as an affine function $A_i:\IZ/s\IZ\rightarrow\IZ/s\IZ$. Finally, we recall the induced function $\overline{f}:\{0,1,\ldots,d\}\rightarrow\{0,1,\ldots,d\}$ (the unique function such that $f(C_i)\subseteq C_{\overline{f}(i)}$ for each $i$).

Our goal in this section is to describe methods through which the isomorphism type of the functional graph $\Gamma_f$ can be understood, following the approach outlined in the introduction.

\subsection{Periodic points and CRL-lists}\label{subsec3P1}

Understanding the periodic points and finding a CRL-list of $f$ can be reduced to the corresponding tasks for affine maps of finite cyclic groups. We observe that periodic points of $f$ are necessarily contained in \enquote{periodic blocks} (i.e., blocks $C_i$ such that $i$ is periodic under $\overline{f}$). We assume that it is an easy task (due to $d$ being sufficiently small) to find a CRL-list $\overline{\Lcal}$\phantomsection\label{not120} for $\overline{f}$. We determine the periodic points of $f$ according to the \enquote{block cycle} of $\overline{f}$ they lie on, so let $(i,\ell)\in\overline{\Lcal}$.

If $i=d$, then $\ell=1$, and the only point to consider is the field element $0$, which is by definition periodic of cycle length $1$ under $f$. We note the contribution $\Lcal_d:=\{(0,1)\}$ to the CRL-list $\Lcal$ of $f$ we are building.

Now we assume that $i<d$, and that the cycle of $i$ under $\overline{f}$ is $(i_0,i_1,\ldots,i_{\ell-1})$\phantomsection\label{not121}\phantomsection\label{not122} with $i_0=i$. A point $x\in C_i$ is periodic under $f$ if and only if it is periodic under the iterate $f^{\ell}$, which stabilizes $C_i$ and acts on the corresponding copy of $\IZ/s\IZ$ via the affine map $\Acal_i:=A_{i_0}A_{i_1}\cdots A_{i_{\ell-1}}$\phantomsection\label{not123}. In other words, the periodic points of $f$ in $C_i$ are in bijection (via $\iota_i$) to the periodic points of $\Acal_i$ in $\IZ/s\IZ$, and are thus characterized by Lemma \ref{periodicCharLem}. We note that the set of periodic points of $f$ in a different coset $C_{i_t}$ of the same $\overline{f}$-cycle is simply the iterated set image $f^t(\per(f_{\mid C_i}))$. Moreover, the cycle length of a periodic point $x\in C_i$ under $f$ is the $\ell$-fold of its cycle length under $\Acal_i$. It follows that if $\Lcal'_i\subseteq\IZ/s\IZ\times\IN^+$\phantomsection\label{not124} is a CRL-list of $\Acal_i$ (which we can determine as described in Subsection \ref{subsec2P3}), then $\Lcal_i:=\{(\iota_i(r),\ell\cdot l): (r,l)\in\Lcal'_i\}$\phantomsection\label{not125} is CRL-list of the restriction of $f$ to the entire \enquote{coset cycle spanned by $C_i$} (i.e., to the set $\bigcup_{t=0}^{\ell-1}{C_{i_t}}$). In summary, we obtain the following proposition.

\begin{propposition}\label{crlListConstructProp}
Let $\overline{\Lcal}$ be a CRL-list for $\overline{f}$. We set $\Lcal_d:=\{(0_{\IF_q},1)\}$. Moreover, for $i<d$ with $(i,\ell)\in\overline{\Lcal}$,  we define $\Lcal_i$ as follows. Let $(i_0,i_1,\ldots,i_{\ell-1})$ with $i_0=i$ be the $\overline{f}$-cycle of $i$, and let $\Lcal'_i$ be a CRL-list of the affine map $\Acal_i=A_{i_0}A_{i_1}\cdots A_{i_{\ell-1}}$ of $\IZ/s\IZ$. Then we set $\Lcal_i:=\{(\iota_i(r),\ell\cdot l): (r,l)\in\Lcal'_i\}$. With this definition of $\Lcal_i$ for each $(i,\ell)\in\overline{\Lcal}$, we have that
\[
\Lcal:=\bigcup_{(i,\ell)\in\overline{\Lcal}}{\Lcal_i}.
\]
is a CRL-list of $f$.
\end{propposition}

\subsection{The induced subgraph on the periodic cosets}\label{subsec3P2}

Our next goal is to understand the trees $\Tree_{\Gamma_f}(x)$ in $\Gamma_f$ above periodic points $x$ of $f$, in the sense of Definition \ref{treeAboveDef}. We remind the reader that $\Tree_{\Gamma_f}(x)$ is defined for arbitrary vertices $x$ of $\Gamma_f$, not just periodic ones. In general, it is advantageous to take a recursive approach, understanding $\Tree_{\Gamma_f}(x)$ for vertices $x$ according to their depth in $\Gamma_f$, starting with leaves and working toward periodic vertices, which are at the end of the recursion. Before we carry this out, however, we must understand the induced subgraph $\Gamma_{\per}$ of $\Gamma_f$ on the union of all periodic blocks $C_i$ (i.e., blocks where $i$ is periodic under $\overline{f}$) as a stepping stone.

We observe that $\Gamma_{\per}$ is the functional graph of the restriction $f_{\per}$\phantomsection\label{not126} of $f$ to the union of all periodic blocks. Just like $f$ has the induced function $\overline{f}$ on the index set $\{0,1,\ldots,d\}$, the restriction $f_{\per}$ has the induced function $\overline{f_{\per}}$\phantomsection\label{not127}, which is the restriction of $\overline{f}$ to its set of periodic points. Hence, $\overline{f_{\per}}$ is a permutation of its domain of definition, a fact that is important for our argument.

Similarly to the situation described in Proposition \ref{rigidProp}, if we know, for a given periodic vertex $x$ of $\Gamma_{\per}=\Gamma_{f_{\per}}$, that each $\Tree_{\Gamma_{\per}}(y)$, where $y$ is a child of $x$ in $\Gamma_{\per}^{\ast}$, has rigid procreation, and we know the first $h=h(y)$ procreation numbers of each child $y$ where $h$ is the height of $\Tree_{\Gamma_{\per}}(y)$, then this characterizes the isomorphism type of each $\Tree_{\Gamma_{\per}}(y)$, and thus of $\Tree_{\Gamma_{\per}}(x)$, uniquely. And indeed, while $\Gamma_{\per}^{\ast}$ itself need not have rigid procreation in the more general setting we are considering here, the trees we just referred to do have it. More specifically, we have the following result (in which the exclusion of $i=d$ is without loss of generality, because $\Tree_{\Gamma_{\per}}(0_{\IF_q})$ is trivial anyway).

\begin{theoremm}\label{cosetPermTheo}
Let $f_{\per}$ and $\overline{f_{\per}}$ be as above. Moreover, let $i\in\dom(\overline{f_{\per}})=\per(\overline{f})$ with $i<d$, and let $(i_0,i_1,\ldots,i_{\ell-1})$ be the cycle of $i=i_0$ under $\overline{f_{\per}}$. We extend the notation $i_t$ to arbitrary $t\in\IZ$ by reducing $t$ modulo $\ell$ (so that, for instance, $i_{\ell}=i_0$). For $t=0,1,\ldots,\ell-1$, say $A_{i_t}:z\mapsto \alpha_{i_t}z+\beta_{i_t}$\phantomsection\label{not128}\phantomsection\label{not129} is the affine map of $\IZ/s\IZ$ that describes how $f_{\per}$ (or, equivalently, $f$) maps from $C_{i_t}$ to $C_{i_{t+1}}$, and let $\varphi_{i_t}:=\mu_{\alpha_{i_t}}:z\mapsto \alpha_{i_t}z$\phantomsection\label{not130}, be the associated group endomorphism of $\IZ/s\IZ$. Then the following holds for any positive integer $k$. If $x\in C_i$ has at least $k$ successor generations in $\Gamma_{\per}^{\ast}$ (we note that those successor generations need not be entirely contained in $C_i$), then
\[
\proc^{(\Gamma_{\per}^{\ast})}_k(x)=\left|\ker\left(\prod_{j=0}^{k-1}{\varphi_{i_{-k+j}}}\right):\ker\left(\prod_{j=0}^{k-2}{\varphi_{i_{-k+j}}}\right)\right|=\frac{\gcd\left(\prod_{j=0}^{k-1}{\alpha_{i_{-k+j}}},s\right)}{\gcd\left(\prod_{j=0}^{k-2}{\alpha_{i_{-k+j}}},s\right)},
\]
independently of $x$.
\end{theoremm}

\begin{proof}
This theorem can be seen as a generalization of Theorem \ref{rigidTheo} (which corresponds to the case $\ell=1$), and likewise, its proof is a generalization of that of Theorem \ref{rigidTheo}. We proceed by induction on $k$. For $k=1$, we observe that $C_{\overline{f_{\per}}^{-1}(i)}=C_{i_{-1}}$ is the unique coset which $f_{\per}$ maps to $C_i$. Hence
\[
\proc_1^{(\Gamma_{\per}^{\ast})}(x)=\#\text{ children of }x\text{ in }\Gamma_{\per}^{\ast}=|\{y\in\IZ/s\IZ: A_{i_{-1}}(y)=x\}|=|\ker(\varphi_{i_{-1}})|,
\]
which implies the statement for $k=1$ since an empty product of group endomorphisms is by definition the identity function $\id$.

Now we assume that $k\geq2$ and that the statement holds up to $k-1$ for points in $C_{i_t}$ where $t\in\IZ$ is arbitrary. For $h=1,2,\ldots,k-1$ and $t\in\IZ$, we denote by $\proc_{i_t,h}$\phantomsection\label{not131} the common procreation number $\proc^{(\Gamma_{\per}^{\ast})}_h(y)$ for all vertices $y\in C_{i_t}$ with at least $h$ successor generations in $\Gamma_{\per}^{\ast}$. For each $h=1,2,\ldots,k-1$, the number of endpoints of paths of length $h$ in $\Gamma_{\per}^{\ast}$ starting at a vertex in $C_{i_t}$ with at least $h$ successor generations in $\Gamma_{\per}^{\ast}$ is $|\ker(\prod_{j=0}^{h-1}{\varphi_{i_{t-h+j}}})|$, and an easy induction on $h$ shows that it is also equal to $\prod_{j=0}^{h-1}{\proc_{i_{t-j},h-j}}$. Using this, it follows that for each $x\in C_i=C_{i_0}$ with at least $k$ successor generations, one has
\begin{align*}
&\proc^{(\Gamma_{\per}^{\ast})}_k(x)\cdot\left|\ker\left(\prod_{j=0}^{k-2}{\varphi_{i_{-k+j}}}\right)\right|=\proc^{(\Gamma_{\per}^{\ast})}_k(x)\cdot\prod_{j=0}^{k-2}{\proc_{i_{-j-1},k-1-j}}= \\
&(\#\text{ endpoints of paths of length }k\text{ starting at }x)=\left|\ker\left(\prod_{j=0}^{k-1}{\varphi_{i_{-k+j}}}\right)\right|,
\end{align*}
from which the asserted formula for $\proc^{(\Gamma_{\per}^{\ast})}_k(x)$ follows readily.
\end{proof}

The following example highlights some properties that do \emph{not} hold in general.

\begin{exxample}\label{cosetPermEx}
Let $q=13$, $d=2$ (thus $s=6$), and $\overline{f}=(0,1)(2)$, so that $\Gamma_f=\Gamma_{\per}$. Moreover, we assume that $A_0(z)=z$ and $A_1(z)=2z$. Then $f_{\mid\IF_q^{\ast}}$ has the following functional graph, in which we denote the point in $C_i$ corresponding to $z\in\IZ/6\IZ$ by $(z,i)$:
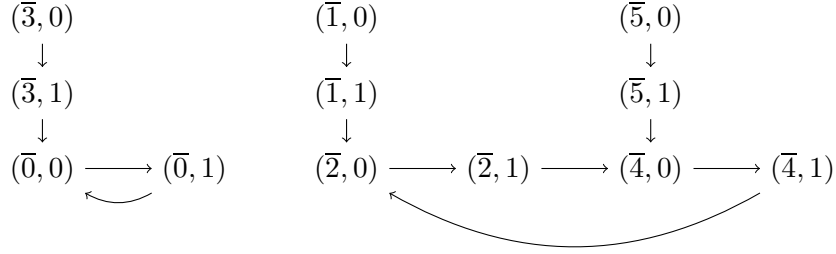
\begin{center}
\begin{tikzpicture}
\node (p1) at (0,0) {$(\overline{0},0)$};
\node (p2) at (2,0) {$(\overline{0},1)$};
\node (c11) at (0,1) {$(\overline{3},1)$};
\node (c12) at (0,2) {$(\overline{3},0)$};
\node (p3) at (4,0) {$(\overline{2},0)$};
\node (p4) at (6,0) {$(\overline{2},1)$};
\node (p5) at (8,0) {$(\overline{4},0)$};
\node (p6) at (10,0) {$(\overline{4},1)$};
\node (c31) at (4,1) {$(\overline{1},1)$};
\node (c32) at (4,2) {$(\overline{1},0)$};
\node (c51) at (8,1) {$(\overline{5},1)$};
\node (c52) at (8,2) {$(\overline{5},0)$};
\draw[->]
(p1) edge (p2)
(p2) edge[bend left=30] (p1)
(p3) edge (p4)
(p4) edge (p5)
(p5) edge (p6)
(p6) edge[bend left=30] (p3)
(c11) edge (p1)
(c12) edge (c11)
(c31) edge (p3)
(c32) edge (c31)
(c51) edge (p5)
(c52) edge (c51);
\end{tikzpicture}
\end{center}
We note the following.
\begin{itemize}
\item The rooted trees above periodic vertices in $C_0$ are \emph{not} isomorphic to the rooted trees above periodic vertices in $C_1$.
\item Transient vertices in $C_1$ have strictly larger tree height in $\Gamma_f$ than periodic vertices in $C_1$. More specifically, the transient vertices in $C_1$ are just those with tree height $1$, the periodic vertices are those with tree height $0$ in $\Gamma_f$.
\item The set of possible tree heights in $\Gamma_f$ above vertices in $C_0$ is $\{0,2\}$, which is \emph{not} an integer interval.
\end{itemize}
\end{exxample}

\subsection{The rooted trees}\label{subsec3P3}

We describe a recursive approach for understanding $\Tree_{\Gamma_f}(x)$ for each vertex $x\in\IF_q=\V(\Gamma_f)$. We proceed in three steps, according to the unique $i\in\{0,1,\ldots,d\}$ such that $x\in C_i$. Unless $i=d$, our goal is to find an arithmetic partition $\Pcal_i$ of $\IZ/s\IZ$, corresponding to a partition of $C_i$ via the bijection $\iota_i$ (that we also call an \emph{arithmetic partition of $C_i$}), such that for vertices $x\in C_i$ from a common block $\Bcal(\Pcal_i,\vec{\nu}^{(\Pcal_i)})$ of that partition, the isomorphism type of $\Tree_{\Gamma_f}(x)$ is constant, denoted by $\Tree_i(\Pcal_i,\vec{\nu}^{(\Pcal_i)})$. We also want to understand $\Tree_i(\Pcal_i,\vec{\nu}^{(\Pcal_i)})$ in terms of $\vec{\nu}^{(\Pcal_i)}$ explicitly, and verify that for fixed $d$, the maximum (arithmetic) complexity of $\Pcal_i$ (in the sense of Definition \ref{sArithDef}(3)) is in $O(d^2\mpe(q-1))\subseteq O(d^2\log{q})$, where $\mpe(m):=\max_p{\nu_p(m)}$ for $m\in\IN^+$ (and $p$ ranges over all primes). First, we introduce a few notations.
\begin{itemize}
\item For $M\subseteq\IF_q$\phantomsection\label{not132} and $x\in\IF_q$, the notation $\Tree_{\Gamma_f}(x,M)$\phantomsection\label{not133} denotes the digraph isomorphism type of the subgraph of $\Gamma_f$ that is a rooted tree with root $x$, obtained by attaching to that root all rooted trees $\Tree_{\Gamma_f}(y)$ where $y$ is an $f$-transient pre-image of $x$ with $y\in M$. We note that $\Tree_{\Gamma_f}(x,\IF_q)=\Tree_{\Gamma_f}(x)$.
\item Let $M\subseteq\IF_q$, and let us assume that $\Pcal$ is an arithmetic partition of $C_i$ with a fixed sequence of spanning congruences such that for $x\in C_i$, the rooted tree isomorphism type $\Tree_{\Gamma_f}(x,M)$ only depends on the block $\Bcal(\Pcal,\vec{\nu})$ of $\Pcal$ in which $x$ lies (but not on $x$ itself). Then we denote that isomorphism type by $\Tree_i(\Pcal,M,\vec{\nu})$\phantomsection\label{not134}. We also set $\Tree_i(\Pcal,\vec{\nu}):=\Tree_i(\Pcal,\IF_q,\vec{\nu})$\phantomsection\label{not135}.
\item If $\Ifrak_1,\Ifrak_2,\ldots,\Ifrak_N$ are isomorphism types of rooted trees, then their \emph{sum}\phantomsection\label{term49} $\Ifrak_1+\Ifrak_2+\cdots+\Ifrak_N$ is defined as the rooted tree isomorphism type obtained by glueing disjoint copies of the $\Ifrak_j$ together at their roots. This addition turns the class of rooted tree isomorphism types into a class-sized monoid, the neutral element of which is the trivial rooted tree isomorphism type (a single vertex without arcs).
\item If $\Ifrak$ is a rooted tree isomorphism type, we denote by $\Ifrak^+$\phantomsection\label{not136} the rooted tree isomorphism type obtained by connecting a copy of $\Ifrak$ to a new root via an arc from the old to the new root. For example, iterating this operation starting from the trivial rooted tree isomorphism type, one obtains those finite digraphs that are directed paths.
\item If $\Ifrak$ is a rooted tree isomorphism type and $n$ is a non-negative integer, we define the \emph{multiple}\phantomsection\label{term50} $n\cdot\Ifrak=n\Ifrak$ as follows recursively. $0\Ifrak$ is the trivial rooted tree isomorphism type, and $(n+1)\Ifrak:=n\Ifrak+\Ifrak$.
\item In view of the previous two bullet points, non-negative integer linear combinations $n_1\Ifrak_1+n_2\Ifrak_2+\cdots+n_N\Ifrak_N$ of rooted tree isomorphism types are well-defined.
\item If $\Xcal_1,\ldots\Xcal_n$\phantomsection\label{not137} are arithmetic partitions of $\IZ/m\IZ$, then $\bigwedge_{k=1}^n{\Xcal_k}$\phantomsection\label{not137P5} denotes the infimum\phantomsection\label{term51} of the $\Xcal_k$ in the lattice of all partitions of $\IZ/m\IZ$ (i.e., the roughest common refinement of the $\Xcal_k$). Equivalently, if a spanning $m$-congruence sequence is fixed for each $\Xcal_k$, then $\bigwedge_{k=1}^n{\Xcal_k}$ is the arithmetic partition of $\IZ/m\IZ$ that is spanned by the concatenation of those sequences.
\end{itemize}

For the first step of our approach, we consider the case where $i\not=d$ and $C_i$ is \emph{not} a periodic coset (let us call such cosets \emph{transient}). For this case, we proceed by recursion on the height of $\Tree_{\Gamma_{\overline{f}}}(i)$. The base of the recursion is when $i$ is a leaf in $\Gamma_{\overline{f}}$. Then $x\in C_i$ is a leaf in $\Gamma_f$, i.e., $x\notin\im(f)$\phantomsection\label{not138}. This means that $\Tree_{\Gamma_f}(x)$ consists of the single vertex $x$ and has no arcs. Therefore, we may choose $\Pcal_i$ as the trivial partition $\Pcal(\emptyset)=\{\IZ/m\IZ\}$ of $\IZ/m\IZ$, and $\Tree_i(\Pcal_i,\emptyset)$ as the trivial rooted tree isomorphism type. We note that $\AC(\Pcal_i)$, the complexity of $\Pcal_i$, is $0=\vfrak_i$\phantomsection\label{not139} where, for a general $j\in\{0,1,\ldots,d\}$, we set $\vfrak_j:=|\V(\Tree_{\Gamma_{\overline{f}}}(j))|-1$, the number of vertices that are strictly above $j$ in the corresponding subtree of $\Gamma_{\overline{f}}$.

Now we assume that $C_i$ is transient, that the height of $\Tree_{\Gamma_{\overline{f}}}(i)$ is $h\geq1$, and that all transient cosets where that height is less than $h$ have been \enquote{taken care of} via arithmetic partitions $\Pcal_j$ such that $\AC(\Pcal_j)\leq\vfrak_j$. In particular, if $\overline{f}^{-1}(\{i\})=\{j_1,j_2,\ldots,j_K\}$, then for each $t=1,2,\ldots,K$, we have an arithmetic partition $\Pcal_{j_t}$ of $C_{j_t}$ with an explicit sequence of spanning $s$-congruences of length $m_{j_t}\leq\vfrak_{j_u}$\phantomsection\label{not140} such that the isomorphism type $\Tree_{\Gamma_f}(y)$ is the same for all vertices $y\in C_{j_t}$ chosen from a common block $\Bcal(\Pcal_{j_t},\vec{\nu}^{(\Pcal_{j_t})})$ of $\Pcal_{j_t}$, and we understand each such isomorphism type $\Tree_{j_t}(\Pcal_{j_t},\vec{\nu}^{(\Pcal_{j_t})})$ explicitly. Now, because each pre-image $y$ of $x\in C_i$ under $f$ (i.e., each child of $x$ in $\Gamma_f^{\ast}$) must lie in one of the cosets $C_{j_t}$ for $t=1,2,\ldots,K$, it follows that
\[
\Tree_{\Gamma_f}(x)=\sum_{t=1}^K{\Tree_{\Gamma_f}(x,C_{j_t})}.
\]
Moreover, for fixed $t\in\{1,2,\ldots,K\}$, we can write
\[
\Tree_{\Gamma_f}(x,C_{j_t})=\sum_{\vec{\nu}^{(\Pcal_{j_t})}\in\{\emptyset,\neg\}^{m_{j_t}}}{(|f^{-1}(\{x\})\cap\Bcal(\Pcal_{j_t},\vec{\nu}^{(\Pcal_{j_t})})|\cdot\Tree_{j_t}(\Pcal_{j_t},\vec{\nu}^{(\Pcal_{j_t})})^+)}.
\]
Let us consider the arithmetic partition $\Pcal'_{j_t}:=\Pfrak'(\Pcal_{j_t},A_{j_t})$\phantomsection\label{not141} of $C_i$ with its explicit spanning $s$-congruence sequence of length $m_{j_t}+1$ from Lemma \ref{masterLem}. If $x$ lies in the block $\Bcal(\Pcal'_{j_t},\vec{\nu}^{(\Pcal'_{j_t})})$ for some fixed $\vec{\nu}^{(\Pcal'_{j_t})}\in\{\emptyset,\neg\}^{m_{j_t}+1}$, then
\[
|f^{-1}(\{x\})\cap\Bcal(\Pcal_{j_t},\vec{\nu}^{(\Pcal_{j_t})})|=\sigma_{\Pcal_{j_t},A_{j_t}}(\vec{\nu}^{(\Pcal_{j_t})},\vec{\nu}^{(\Pcal'_{j_t})})
\]
by Lemma \ref{masterLem}, and thus
\[
\Tree_{\Gamma_f}(x,C_{j_t})=\sum_{\vec{\nu}^{(\Pcal_{j_t})}\in\{\emptyset,\neg\}^{m_{j_t}}}{(\sigma_{\Pcal_{j_t},A_{j_t}}(\vec{\nu}^{(\Pcal_{j_t})},\vec{\nu}^{(\Pcal'_{j_t})})\cdot\Tree_{j_t}(\Pcal_{j_t},\vec{\nu}^{(\Pcal_{j_t})})^+)},
\]
independently of $x$ itself. Now, let us set $\Pcal_i:=\bigwedge_{t=1}^K{\Pcal'_{j_t}}$, viewed as an arithmetic partition of $C_i$ with a spanning sequence of length $m_i:=\sum_{t=1}^K{(m_{j_t}+1)}$. We can view each logical sign tuple $\vec{\nu}^{(\Pcal_i)}\in\{\emptyset,\neg\}^{m_i}$ as a concatenation of logical sign tuples $\vec{\nu}^{(\Pcal'_{j_t})}\in\{\emptyset,\neg\}^{m_{j_t}+1}$ for $t=1,2,\ldots,K$, and if $x\in C_i$ lies in the block $\Bcal(\Pcal_i,\vec{\nu}^{(\Pcal_i)})$ of $\Pcal_i$, then $x$ also lies in the block $\Bcal(\Pcal_{j_t},\vec{\nu}^{(\Pcal'_{j_t})})$ of $\Pcal'_{j_t}$ for each $t=1,2,\ldots,K$, whence
\[
\Tree_{\Gamma_f}(x)=\sum_{t=1}^K\sum_{\vec{\nu}^{(\Pcal_{j_t})}\in\{\emptyset,\neg\}^{m_{j_t}}}{(\sigma_{\Pcal_{j_t},A_{j_t}}(\vec{\nu}^{(\Pcal_{j_t})},\vec{\nu}^{(\Pcal'_{j_t})})\cdot\Tree_{j_t}(\Pcal_{j_t},\vec{\nu}^{(\Pcal_{j_t})})^+)},
\]
independently of $x$ itself, as required. Moreover, we note that
\[
\AC(\Pcal_i)\leq\sum_{t=1}^K{\AC(\Pcal'_{j_t})}\leq\sum_{t=1}^K{(\vfrak_{j_t}+1)}=\vfrak_i.
\]
In summary, we obtain the following result.

\begin{propposition}\label{transientCosetsProp}
For each $\overline{f}$-transient $i\in\{0,1,\ldots,d-1\}$, the arithmetic partition $\Pcal_i$ of $C_i$ together with an explicit spanning sequence of $s$-congruences of length $m_i$ and associated rooted tree isomorphism types $\Tree_i(\Pcal_i,\vec{\nu}^{(\Pcal_i)})$ for $\vec{\nu}^{(\Pcal_i)}\in\{\emptyset,\neg\}^{m_i}$ can be defined as follows by recursion on $h_i:=\height(\Tree_{\Gamma_{\overline{f}}}(i))$\phantomsection\label{not142}.
\begin{enumerate}
\item If $h_i=0$, we may set
\begin{enumerate}
\item $\Pcal_i:=\Pcal(\emptyset)$,
\item $m_i:=0$, and
\item $\Tree_i(\Pcal_i,\emptyset)$ to be the trivial rooted tree isomorphism type.
\end{enumerate}
\item If $h_i\geq1$, we let $\overline{f}^{-1}(\{i\})=\{j_1,j_2,\ldots,j_K\}$ and set $\Pcal'_{j_t}:=\Pfrak'(\Pcal_{j_t},A_{j_t})$ for $t=1,2,\ldots,K$. Then we define
\begin{enumerate}
\item $\Pcal_i:=\bigwedge_{t=1}^K{\Pcal'_{j_t}}$,
\item $m_i:=\sum_{t=1}^K{(m_{j_t}+1)}$, and
\item for $\vec{\nu}^{(\Pcal_i)}\in\{\emptyset,\neg\}^{m_i}$, viewed as the concatenation of the logical sign tuples $\vec{\nu}^{(\Pcal'_{j_t})}\in\{\emptyset,\neg\}^{m_{j_t}+1}$ for $t=1,2,\ldots,K$,
\begin{align*}
&\Tree_i(\Pcal_i,\vec{\nu}^{(\Pcal_i)}):= \\
&\sum_{t=1}^K\sum_{\vec{\nu}^{(\Pcal_{j_t})}\in\{\emptyset,\neg\}^{m_{j_t}}}{(\sigma_{\Pcal_{j_t},A_{j_t}}(\vec{\nu}^{(\Pcal_{j_t})},\vec{\nu}^{(\Pcal'_{j_t})})\cdot\Tree_{j_t}(\Pcal_{j_t},\vec{\nu}^{(\Pcal_{j_t})})^+)}.
\end{align*}
\end{enumerate}
\end{enumerate}
With this choice of $\Pcal_i$, we have
\[
\AC(\Pcal_i)\leq\vfrak_i:=|\V(\Tree_{\Gamma_{\overline{f}}}(i))|-1\leq d-1\in O(d)\subseteq O(d^2\mpe(q-1)).
\]
\end{propposition}

The second step is to describe the isomorphism type of $\Tree_{\Gamma_f}(0_{\IF_q})$, which is similar to the recursion step for transient cosets above. Let $\overline{f}^{-1}(\{d\})=\{d,j_1,j_2,\ldots,j_K\}$ (we note that $K=0$ unless some coefficient $a_i$ in the cyclotomic form (\ref{cyclotomicFormEq}) of $f$ is $0$). The children of $0$ in $\Tree_{\Gamma_f}(0)^{\ast}$ are just the \emph{nonzero} children of $0$ in $\Gamma_f^{\ast}$, and each such child must lie in $C_{j_t}$ for some $t\in\{1,2,\ldots,K\}$. We observe that each $C_{j_t}$ is a transient coset, so by Proposition \ref{transientCosetsProp}, we already know a suitable arithmetic partition $\Pcal_{j_t}$ of $C_{j_t}$, with an explicit spanning sequence of length $m_{j_t}$, and have an explicit understanding of the rooted tree isomorphism types $\Tree_{j_t}(\Pcal_{j_t},\vec{\nu}^{(\Pcal_{j_t})})$ for $\vec{\nu}^{(\Pcal_{j_t})}\in\{\emptyset,\neg\}^{m_{j_t}}$. Moreover, all vertices in $C_{j_t}$ map to $0_{\IF_q}$ under $f$, so
\[
f^{-1}(\{0_{\IF_q}\})\cap\Bcal(\Pcal_{j_t},\vec{\nu}^{(\Pcal_{j_t})})=\Bcal(\Pcal_{j_t},\vec{\nu}^{(\Pcal_{j_t})})=\mathbf{0}^{-1}(\{0_{\IZ/s\IZ}\})\cap\Bcal(\Pcal_{j_t},\vec{\nu}^{(\Pcal_{j_t})})
\]
for all $\vec{\nu}^{(\Pcal_{j_t})}\in\{\emptyset,\neg\}^{m_{j_t}}$, where $\mathbf{0}:\IZ/s\IZ\rightarrow\IZ/s\IZ, z\mapsto 0=0z+0$. Hence, if we set $\Pcal'_{j_t}:=\Pfrak'(\Pcal_{j_t},\mathbf{0})$, which is, in its standard form from Lemma \ref{masterLem}, spanned by the single $s$-congruence $x\equiv0\Mod{s}$ repeated $m_{j_t}+1$ times, then
\[
|f^{-1}(\{0_{\IF_q}\})\cap\Bcal(\Pcal_{j_t},\vec{\nu}^{(\Pcal_{j_t})})|=\sigma_{\Pcal_{j_t},\mathbf{0}}(\vec{\nu}^{(\Pcal_{j_t})},(\emptyset,\ldots,\emptyset)).
\]
Because $\Tree_{\Gamma_f}(0_{\IF_q})$ is obtained by attaching $|f^{-1}(\{0_{\IF_q}\})\cap\Bcal(\Pcal_{j_t},\vec{\nu}^{(\Pcal_{j_t})})|$ copies of $\Tree_{j_t}(\Pcal_{j_t},\vec{\nu}^{(\Pcal_{j_t})})$ to a common root for each $t=1,2,\ldots,K$ and each $\vec{\nu}^{(\Pcal_{j_t})}\in\{\emptyset,\neg\}^{m_{j_t}}$, we obtain the following proposition.

\begin{propposition}\label{zeroBlockProp}
Let $\overline{f}^{-1}(\{d\})=\{d,j_1,j_2,\ldots,j_K\}$. For $t=1,2,\ldots,K$, let $\Pcal_{j_t}$, $m_{j_t}$ and $\Tree_{j_t}(\Pcal_{j_t},\vec{\nu}^{(\Pcal_{j_t})})$ for $t=1,2,\ldots,K$ be as in Proposition \ref{transientCosetsProp}. Moreover, we denote by $\mathbf{0}$ the constant 0 function $\IZ/s\IZ\rightarrow\IZ/s\IZ$. Then
\[
\Tree_{\Gamma_f}(0_{\IF_q})=\sum_{t=1}^K\sum_{\vec{\nu}^{(\Pcal_{j_t})}\in\{\emptyset,\neg\}^{m_{j_t}}}{(\sigma_{\Pcal_{j_t},\mathbf{0}}(\vec{\nu}^{(\Pcal_{j_t})},(\emptyset,\ldots,\emptyset))\cdot\Tree_{j_t}(\Pcal_{j_t},\vec{\nu}^{(\Pcal_{j_t})})^+)}.
\]
\end{propposition}

In the third and final step, we consider vertices $x$ from a \emph{periodic} coset $C_i$ (with $i<d$). Let us write $\overline{f}^{-1}(\{i\})=\{i',j_1,j_2,\ldots,j_K\}$, where $i'$\phantomsection\label{not143} is the unique $\overline{f}$-periodic pre-image of $i$ under $\overline{f}$. Hence, $C_{j_t}$ for $t=1,2,\ldots,K$ is a transient coset; let $\Pcal_{j_t}$ be the arithmetic partition of $C_{j_t}$ defined in Proposition \ref{transientCosetsProp}. Moreover, let $(i_0,i_1,\ldots,i_{\ell-1})$ be the cycle of $i=i_0$ under $\overline{f}$, and let us define $i_k:=i_{k\bmod{\ell}}$ for $k\in\IZ$ (in particular, $i'=i_{-1}$). Moreover, let $A_{i_k}:\IZ/s\IZ\rightarrow\IZ/s\IZ$, $x\mapsto \alpha_{i_k}x+\beta_{i_k}$, be the affine map that describes how $f$ maps from $C_{i_k}$ to $C_{i_{k+1}}$.

This case is more complicated, and we need to make a recursion by another parameter. As in Subsection \ref{subsec3P2}, we denote by $\Gamma_{\per}$ the induced subgraph of $\Gamma_f$ on the union of all periodic blocks $C_i$, i.e., the functional graph of $f_{\per}$. We remind the reader that we explicitly understand the trees above periodic vertices in $\Gamma_{\per}$ thanks to Theorem \ref{cosetPermTheo}. The idea is to proceed by recursion on a parameter called $\hfrak(x)$\phantomsection\label{not144}, which can range from $0$ up to the maximum height $H_i$\phantomsection\label{not145} of the rooted trees in $\Gamma_{\per}$ above periodic vertices in one of the cosets $C_{i_t}=C_{\overline{f}^t(i)}$ for $t=0,1,\ldots,\ell-1$. This parameter is defined as follows:
\begin{equation}\label{hDefEq}
\hfrak(x)=
\begin{cases}
\height(\Tree_{\Gamma_{\per}}(x))\in\{0,1,\ldots,H_i-1\}, & \text{if }x\text{ is }f\text{-transient}, \\
H_i, & \text{if }x\text{ is }f\text{-periodic}.
\end{cases}
\end{equation}
By Example \ref{cosetPermEx}, there is in general no relation between $\hfrak(x)$ and $\height(\Tree_{\Gamma_{\per}}(x))$ when $\hfrak(x)=H_i$ (i.e., when $x$ is $f$-periodic). Also, $\hfrak$ need not assume all values in $\{0,1,\ldots,H_i\}$ on a given coset $C_{i_t}$; in fact, $\hfrak(C_{i_t})$ need not even be an integer interval. None of this will be an issue for our approach, though.

We observe that $H_i+1$ is the smallest positive integer $k$ such that for all $t\in\IZ$, the common procreation number
\[
\proc_k^{(\Gamma_{\per}^{\ast})}(x)=\frac{\gcd(\prod_{j=0}^{k-1}{\alpha_{i_{t-k+j}}},s)}{\gcd(\prod_{j=0}^{k-2}{\alpha_{i_{t-k+j}}},s)}
\]
of all $f$-periodic vertices $x\in C_{i_t}$ is equal to $1$, which is (for each given $t$) equivalent to
\[
\gcd(\prod_{j=0}^{k-1}{\alpha_{i_{t-k+j}}},s)=\gcd(\prod_{j=0}^{k-2}{\alpha_{i_{t-k+j}}},s),
\]
and further to
\[
\prod_{p\mid\gcd(\alpha_{i_{t-1}},s)}{p^{\nu_p(s)}} \mid \prod_{j=0}^{k-2}{\alpha_{i_{t-k+j}}}.
\]
Setting $\overline{\alpha}_i:=\prod_{t=0}^{\ell-1}{\alpha_{i_t}}$\phantomsection\label{not146} (the linear coefficient of $\Acal_i$ in the notation of Subsection \ref{subsec3P1}), it is not difficult to see from this that
\begin{equation}\label{hiBoundEq}
H_i\leq\ell\cdot\max_{p\mid\gcd(\overline{\alpha}_i,s)}{\left\lceil\frac{\nu_p(s)}{\nu_p(\overline{\alpha}_i)}\right\rceil}\leq\ell\mpe(s)\leq d\mpe(q-1).
\end{equation}

Let us set $\Pcal'_{j_t}:=\Pfrak'(\Pcal_{j_t},A_{j_t})$ for $t=1,2,\ldots,K$, and $\Rcal_i:=\bigwedge_{t=1}^K{\Pcal'_{j_t}}$\phantomsection\label{not147}. We denote by $n_i$\phantomsection\label{not148} the length of the spanning congruence sequence for $\Rcal_i$ which we use (in general, $n_i=\sum_{t=1}^K{(m_{j_t}+1)}$, but in a concrete example, there may be repetitions among those congruences, allowing us to delete some of them). A simple observation is that as far as the transient coset contribution $\Tree_{\Gamma_f}(x,\bigcup_{t=1}^K{C_{j_t}})$ to $\Tree_{\Gamma_f}(x)$ is concerned, everything is as in Step 1.

\begin{propposition}\label{periodicCosetsTransientProp}
Let $i\in\{0,1,\ldots,d-1\}$ be $\overline{f}$-periodic, and let $j_1,j_2,\ldots,j_K$ be the $\overline{f}$-transient pre-images of $i$ under $\overline{f}$. Moreover, let $\Pcal'_{j_t}:=\Pfrak'(\Pcal_{j_t},A_{j_t})$ for $t=1,2,\ldots,K$ and $\Rcal_i:=\bigwedge_{t=1}^K{\Pcal'_{j_t}}$. Then the following hold.
\begin{enumerate}
\item For $x\in C_i$ and $t\in\{1,2,\ldots,K\}$, the isomorphism type $\Tree_{\Gamma_f}(x,C_{j_t})$ only depends on the $\Pcal'_{j_t}$-block $\Bcal(\Pcal'_{j_t},\vec{\nu}^{(\Pcal'_{j_t})})$ (for some $\vec{\nu}^{(\Pcal'_{j_t})}\in\{\emptyset,\neg\}^{m_{j_t}+1}$) in which $x$ lies. That isomorphism type is denoted by $\Tree_i(\Pcal'_{j_t},C_{j_t},\vec{\nu}^{(\Pcal'_{j_t})})$\phantomsection\label{not149} and can be computed according to the formula
\begin{align*}
&\Tree_i\left(\Pcal'_{j_t},C_{j_t},\vec{\nu}^{(\Pcal'_{j_t})}\right)= \\
&\sum_{\vec{\nu}^{(\Pcal_{j_t})}\in\{\emptyset,\neg\}^{m_{j_t}}}{\sigma_{\Pcal_{j_t},A_{j_t}}\left(\vec{\nu}^{(\Pcal_{j_t})},\vec{\nu}^{(\Pcal'_{j_t})}\right)\Tree_{j_t}\left(\Pcal_{j_t},\vec{\nu}^{(\Pcal_{j_t})}\right)^+}.
\end{align*}
\item For $x\in C_i$, the rooted tree isomorphism type of $\Tree_{\Gamma_f}(x,\bigcup_{t=1}^K{C_{j_t}})$ only depends on the $\Rcal_i$-block $\Bcal\left(\Rcal_i,\vec{\nu}^{(\Rcal_i)}\right)$ (for some $\vec{\nu}^{(\Rcal_i)}\in\{\emptyset,\neg\}^{n_i}$) in which $x$ lies. That isomorphism type is denoted by $\Tree_i(\Rcal_i,\bigcup_{t=1}^K{C_{j_t}},\vec{\nu}^{(\Rcal_i)})$\phantomsection\label{not150} and can be computed according to the formula
\[
\Tree_i\left(\Rcal_i,\bigcup_{t=1}^K{C_{j_t}},\vec{\nu}^{(\Rcal_i)}\right)=\sum_{t=1}^K{\Tree_i\left(\Pcal'_{j_t},C_{j_t},\vec{\nu}^{(\Pcal'_{j_t})}\right)^+}
\]
where $\vec{\nu}^{(\Pcal'_{j_t})}\in\{\emptyset,\neg\}^{m_{j_t}+1}$, for $t=1,2,\ldots,K$, is the unique logical sign tuple such that $\Bcal\left(\Rcal_i,\vec{\nu}^{(\Rcal_i)}\right)\subseteq\Bcal\left(\Pcal'_{j_t},\vec{\nu}^{(\Pcal'_{j_t})}\right)$; in the standard situation, where $n_i=\sum_{t=1}^K{(m_{j_t}+1)}$, the tuple $\vec{\nu}^{(\Rcal_i)}$ is simply the concatenation of the $\vec{\nu}^{(\Pcal'_{j_t})}$ for $t=1,2,\ldots,K$.
\end{enumerate}
\end{propposition}

We remind the reader that we wish to proceed by recursion on the parameter $\hfrak(x)$ defined in (\ref{hDefEq}). This is motivated by Proposition \ref{periodicCosetsTransientProp}, because if $\hfrak(x)=0$, then $x$ has no transient children in $C_{i'}$, whence $\Tree_{\Gamma_f}(x)=\Tree_{\Gamma_f}(x,\bigcup_{t=1}^K{C_{j_t}})$. In general, we wish to construct an arithmetic partition $\Scal_{i,h}$\phantomsection\label{not151} of $C_i$ such that for vertices $x\in C_i$ with $\hfrak(x)=h$, the isomorphism type of $\Tree_{\Gamma_f}(x,C_{i'})$ only depends on the $\Scal_{i,h}$-block containing $x$ and is explicitly understood. Then we are basically done, because $\Tree_{\Gamma_f}(x)=\Tree_{\Gamma_f}(x,\bigcup_{t=1}^K{C_{j_t}})+\Tree_{\Gamma_f}(x,C_{i'})$.

In order to construct $\Scal_{i,h}$ and prove that it has the desired property, we need to introduce quite a few notations.
\begin{itemize}
\item For $h\in\{0,1,\ldots,H_i\}$, a vertex $x\in C_i=C_{i_0}$ has at least $h$ successor generations in $\Gamma_{\per}^{\ast}$ if and only if $x$ lies in the image of $A_{i_{-h}}A_{i_{-h+1}}\cdots A_{i_{-1}}$, which is the affine map $\Acal_{i,h}:z\mapsto \overline{\alpha}_{i,h}z+\overline{\beta}_{i,h}$\phantomsection\label{not152}\phantomsection\label{not153}\phantomsection\label{not154} of $\IZ/s\IZ$ where
\[
\overline{\alpha}_{i,h}=\prod_{t=1}^h{\alpha_{i_{-t}}}\text{ and }\overline{\beta}_{i,h}=\sum_{t=1}^h{\beta_{i_{-t}}\prod_{k=1}^{t-1}{\alpha_{i_{-k}}}}.
\]
That is, $x$ has at least $h$ successor generations in $\Gamma_{\per}^{\ast}$ if and only if it satisfies the following $s$-congruence, which we denote by $\theta_{i,h}(x)$\phantomsection\label{not155}:
\[
x\equiv\overline{\beta}_{i,h}\Mod{\gcd\left(\overline{\alpha}_{i,h},s\right)}.
\]
We observe that the modulus in $\theta_{i,0}(x)$ is $1$, so that congruence is trivial.
\item Next, we describe, for each $h\in\{0,1,\ldots,H_i\}$, a simple system $\Theta_{i,h}(x)$\phantomsection\label{not156} of at most two $s$-CCs such that for all $x\in C_i$, the equality $\hfrak(x)=h$ holds if and only if $\Theta_{i,h}(x)$ holds. If $h<H_i$, then the vertices $x\in C_i$ with $\hfrak(x)=h$ are just those $f$-transient $x\in C_i$ that have exactly $h$ successor generations in $\Gamma_{\per}^{\ast}$. It follows that for such $h$, one has $\hfrak(x)=h$ if and only if $\theta_{i,h}(x)$ and $\neg\theta_{i,h+1}(x)$ both hold. Now we assume that $h=H_i$. By definition of $\hfrak$, this happens if and only if $x$ is $f$-periodic, which is (by definition of $H_i$) equivalent to $x$ having at least $H_i$ successor generations in $\Gamma_{\per}^{\ast}$. Therefore, the condition $\theta_{i,H_i}(x)$ alone provides the desired characterization in this case.

Now, noting once more that $\theta_{i,0}(x)$ is trivial and may thus be omitted from any system of conditions containing it, we may define $\Theta_{i,h}(x)$ as follows:
\[
\Theta_{i,h}(x):=
\begin{cases}
\emptyset, & \text{if }H_i=(h=)0, \\
\neg\theta_{i,1}(x), & \text{if }h=0<H_i, \\
\theta_{i,h}(x)\wedge(\neg\theta_{i,h+1}(x)), & \text{if }0<h<H_i, \\
\theta_{i,H_i}(x), & \text{if }h=H_i>0.
\end{cases}
\]
\item If $\Theta(x)$\phantomsection\label{not157} is a system of $m$-CCs, then $\Pfrak(\Theta(x))$\phantomsection\label{not158} denotes the arithmetic partition of $\IZ/m\IZ$ spanned by the non-negated versions of the conditions in $\Theta(x)$. For example,
\[
\Pfrak
\begin{pmatrix}
x\equiv 4\Mod{6} \\
x\not\equiv 3\Mod{9} \\
x\equiv 0\Mod{2}
\end{pmatrix}
=
\Pfrak
\begin{pmatrix}
x\equiv 4\Mod{6} \\
x\equiv 3\Mod{9} \\
x\equiv 0\Mod{2}
\end{pmatrix}.
\]
We note that if $x\in\IZ/m\IZ$ is chosen from a fixed block of $\Pfrak(\Theta(x))$, then the truth value of each condition in $\Theta(x)$ is independent of $x$, and so is the truth value of $\Theta(x)$ itself. We also observe that for each $h\in\{0,1,\ldots,H_i\}$, the following equality holds for the systems $\Theta_{i,k}(x)$ of $s$-CCs defined in the previous bullet point:
\begin{equation}\label{systemsWedgeEq}
\bigwedge_{k=0}^h{\Pfrak(\Theta_{i,k}(x))}=\Pfrak(\theta_{i,j}(x): j=1,2,\ldots,\min\{h+1,H_i\}).
\end{equation}
We set $\Ucal_i:=\Pfrak(\theta_{i,k}(x): k=1,2,\ldots,H_i)$\phantomsection\label{not159}.
\item For $k\in\{0,1,\ldots,H_i\}$, we denote by $\vec{\xi}_{i,k}$\phantomsection\label{not160} the logical sign tuple $(\nu_1,\nu_2,\ldots,\nu_{H_i})$ with $\nu_t=\emptyset$ if and only if $t\leq k$. With this definition, if we view $\Ucal_i$ as an arithmetic partition of $C_i$, then the block $\Bcal(\Ucal_i,\vec{\xi}_{i,k})$ consists precisely of those $x\in C_i$ such that $\hfrak(x)=k$ (and every block of $\Ucal_i$ is of this form for some $k$).
\item Let $\Pcal=\Pfrak(x\equiv\bfrak_j\Mod{\afrak_j}: j=1,2,\ldots,K)$ be an arithmetic partition of $\IZ/m\IZ$, and let $A:z\mapsto az+b$, be an affine map of $\IZ/m\IZ$. In dependency of $\Pcal$ (actually, of the fixed sequence of spanning congruences for $\Pcal$, rather than $\Pcal$ itself) and $A$, we define another arithmetic partition $\lambda(\Pcal,A)$\phantomsection\label{not161} of $\IZ/m\IZ$ as follows: $\lambda(\Pcal,A):=\Pfrak(x\equiv a\bfrak_j+b\Mod{\gcd(a\afrak_j,m)}: j=1,2,\ldots,K)$. Here is a list of important facts concerning this notation (for which we have $m=s$ throughout):
\begin{itemize}
\item If $\Pcal$ is an arithmetic partition of $C_{i'}$, then
\[
\Pfrak'(\Pcal,A_{i'})=\lambda(\Pcal,A_{i'})\wedge\Pfrak(\theta_{i,1}(x)).
\]
\item For $h=0,1,\ldots,H_{i'}=H_i$, we have
\[
\lambda(\Pfrak(\theta_{i',h}(x)),A_{i'})=
\begin{cases}
\Pfrak(\theta_{i,h+1}(x)), & \text{if }h<H_i, \\
\Pfrak(\theta_{i,H_i}(x)), & \text{if }h=H_i.
\end{cases}
\]
Indeed, for $h<H_i$, this is immediate by the definition of $\theta_{i,h}(x)$ and the facts that $A_{i'}(\overline{\beta}_{i',h})=\overline{\beta}_{i,h+1}$ and $\gcd(a\gcd(a',s),s)=\gcd(aa',s)$ for all $a,a'\in\IZ$. Moreover, noting that our definition of $\theta_{i,h}(x)$ also makes sense if $h>H_i$, we have $\lambda(\Pfrak(\theta_{i',H_{i'}}(x)))=\Pfrak(\theta_{i,H_i+1}(x))$. However, for $x\in C_i$, the condition $\theta_{i,H_i+1}(x)$ holds if and only if $x$ has at least $H_i+1$ successor generations in $\Gamma_{\per}^{\ast}$, which is the case if and only if $x$ is periodic, i.e., if and only if $\theta_{i,H_i}(x)$ holds. Hence, the congruences $\theta_{i,H_i+1}(x)$ and $\theta_{i,H_i}(x)$ have the same solution set in $\IZ/s\IZ$, whence $\Pfrak(\theta_{i,H_i+1}(x))=\Pfrak(\theta_{i,H_i}(x))$. This concludes the proof of the above formulas for $\lambda(\Pfrak(\theta_{i',h}(x)),A_{i'})$.
\item By the previous two bullet points and equality (\ref{systemsWedgeEq}), applied with $i'$ in place of $i$, we have that
\[
\Pfrak'(\bigwedge_{k=0}^h{\Pfrak(\Theta_{i',k}(x))},A_{i'})=\bigwedge_{k=0}^{\min\{h+1,H_i\}}{\Theta_{i,k}(x)},
\]
and, in particular, $\Pfrak'(\Ucal_{i'},A_{i'})=\Ucal_i$.
\end{itemize}
\item Let $\Pcal$ be an arithmetic partition of $C_i$. We define the notation $\lambda_i^t(\Pcal)$, where $t$ is a non-negative integer, as follows recursively: $\lambda_i^0(\Pcal):=\Pcal$, and for $t\geq1$, we set $\lambda_i^t(\Pcal):=\lambda(\lambda_i^{t-1}(\Pcal),A_{t-1})$\phantomsection\label{not162}. In other words, $\lambda_i^t(\Pcal)$ is the arithmetic partition of $C_{i_t}$ obtained by pushing $\Pcal$ forward $t$ times along the $\overline{f}$-cycle of $i$ via the operation $\lambda$, using the appropriate affine function $A_{i_k}$ in each step.
\item For $h\in\{0,1,\ldots,H_i\}$, we introduce the following arithmetic partitions of $C_i$:
\begin{itemize}
\item $\Scal_{i,h}:=\bigwedge_{t=1}^h{\lambda_{i_{-t}}^t(\Rcal_{i_{-t}})}$;
\item $\Pcal_{i,h}:=\Rcal_i\wedge\Scal_{i,h}=\bigwedge_{t=0}^h{\lambda_{i_{-t}}^t(\Rcal_{i_{-t}})}$\phantomsection\label{not163};
\item $\Tcal_{i,h}:=\Scal_{i,h}\wedge\Ucal_i=\bigwedge_{t=1}^h{\lambda_{i_{-t}}^t(\Rcal_{i_{-t}})}\wedge\Ucal_i$\phantomsection\label{not164};
\item $\Qcal_{i,h}:=\Rcal_i\wedge\Tcal_{i,h}=\bigwedge_{t=0}^h{\lambda_{i_{-t}}^t(\Rcal_{i_{-t}})}\wedge\Ucal_i$\phantomsection\label{not165}.
\end{itemize}
The motivation for considering $\Scal_{i,h}$ and $\Pcal_{i,h}$ is that their blocks control the rooted tree isomorphism type of $\Tree_{\Gamma_f}(x,C_{i'})$ and $\Tree_{\Gamma_f}(x)$, respectively, for vertices $x\in C_i$ \emph{of $\hfrak$-value $h$} contained in them -- see Proposition \ref{periodicCosetsPeriodicProp} below. An explicit formula for $\Tree_{\Gamma_f}(x,C_{i'})$ in terms of the $\Scal_{i,h}$-block containing $x$ is also given in Proposition \ref{periodicCosetsPeriodicProp}, and that formula involves distribution numbers (as in Lemma \ref{masterLem} -- see the sentence after that lemma) of the partitions $\Qcal_{i',k}$ for $0\leq k<h$. The partitions of the form $\Tcal_{i,h}$ are not mentioned in the statement of Proposition \ref{periodicCosetsPeriodicProp}, but they play an important role in its proof due to the fact that for $h\geq1$, one has $\Tcal_{i,h}=\Pfrak'(\Qcal_{i',h-1},A_{i'})$.
\item We denote the concatenation of logical sign tuples $\vec{\nu}$ and $\vec{\nu}'$ by $\vec{\nu}\diamond\vec{\nu}'$\phantomsection\label{not166}.
\end{itemize}
We are now in a position to formulate in detail how the blocks $B$ of $\Scal_{i,h}$ affect the rooted trees above vertices $x\in B$ with $\hfrak(x)=h$.

\begin{propposition}\label{periodicCosetsPeriodicProp}
Let $h\in\{0,1,\ldots,H_i\}$, and let $x\in C_i$ with $\hfrak(x)=h$. Then the following hold.
\begin{enumerate}
\item The isomorphism type $\Tree_{\Gamma_f}(x,C_{i'})$ only depends on the block $\Bcal(\Scal_{i,h},\vec{\nu}^{(\Scal_{i,h})})$ of $\Scal_{i,h}$ in which $x$ lies and is denoted by $\Tree^{(h)}_i(\Scal_{i,h},C_{i'},\vec{\nu}^{(\Scal_{i,h})})$\phantomsection\label{not167}.
\item The isomorphism type $\Tree_{\Gamma_f}(x)$ only depends on the $\Pcal_{i,h}$-block $\Bcal(\Pcal_{i,h},\vec{\nu}^{(\Pcal_{i,h})})$ in which $x$ lies and is denoted by $\Tree^{(h)}_i(\Pcal_{i,h},\vec{\nu}^{(\Pcal_{i,h})})$\phantomsection\label{not168}.
\end{enumerate}
More specifically, for $h=0$, where $\Scal_{i,h}=\Pcal(\emptyset)$ and $\Pcal_{i,h}=\Rcal_i$, the rooted tree $\Tree^{(0)}_i(\Scal_{i,0},C_{i'},\emptyset)$ is trivial, and
\[
\Tree^{(0)}_i(\Pcal_{i,0},\vec{\nu}^{(\Pcal_{i,0})})=\Tree_i\left(\Rcal_i,\bigcup_{t=1}^K{C_{j_t}},\vec{\nu}^{(\Pcal_{i,0})}\right).
\]
For $h\geq1$, writing $\vec{\nu}^{(\Scal_{i,h})}=\diamond_{j=1}^h{\vec{o'_t}}$ with $\vec{o'_t}\in\{\emptyset,\neg\}^{n_{i_{-t}}}$, we have the following, where $\vec{o_t}$ for $t=0,1,\ldots,h-1$ is a variable ranging over $\{\emptyset,\neg\}^{n_{i_{-t-1}}}$:
\begin{align*}
&\Tree^{(h)}_i(\Scal_{i,h},C_{i'},\vec{\nu}^{(\Scal_{i,h})})= \\
&\sum_{k=0}^{h-1}\sum_{\vec{o_0},\ldots,\vec{o_k}}{\sigma_{\Qcal_{i',k},A_{i'}}(\diamond_{t=0}^k{\vec{o_t}}\diamond\vec{\xi}_{i',k},\diamond_{t=1}^{k+1}{\vec{o'_t}}\diamond\vec{\xi_{i,h}})\Tree^{(k)}_{i'}(\Pcal_{i',k},\diamond_{t=0}^k{\vec{o_t}})^+},
\end{align*}
and, viewing $\vec{\nu}^{(\Pcal_{i,h})}$ as the concatenation of $\vec{o'_0}\in\{\emptyset,\neg\}^{n_i}$ and $\vec{\nu}^{(\Scal_{i,h})}$, we have
\[
\Tree_i^{(h)}(\Pcal_{i,h},\vec{\nu}^{(\Pcal_{i,h})})=\Tree_i\left(\Rcal_i,\bigcup_{t=1}^K{C_{j_t}},\vec{o'_0}\right)+\Tree_i^{(h)}\left(\Scal_{i,h},C_{i'},\vec{\nu}^{(\Scal_{i,h})}\right).
\]
\end{propposition}

\begin{proof}
The formulas for $h=0$ are clear because vertices $x\in C_i$ with $\hfrak(x)=0$ have no $f$-transient children in $C_{i'}$. We may thus assume that $h\geq1$. With regard to $\Tree^{(h)}_i(\Scal_{i,h},C_{i'},\vec{\nu}^{(\Scal_{i,h})})$, we note that if $x\in C_i$ lies in the block $\Bcal(\Scal_{i,h},\vec{\nu}^{(\Scal_{i,h})})$ of $\Scal_{i,h}$ and satisfies $\hfrak(x)=h$, then for each $k\in\{0,1,\ldots,h-1\}$, we have that $x$ lies in the block $\Bcal(\Tcal_{i,k+1},\diamond_{j=1}^{k+1}{\vec{o'_j}}\diamond\vec{\xi}_{i,h})$ of $\Tcal_{i,k+1}$. By definition, $\Tcal_{i,k+1}=\Pcal'(\Qcal_{i',k},A_{i'})$. Each ($f$-transient) pre-image $y$ of $x$ in $C_{i'}$ with $\hfrak(y)=k$ lies in some block of $\Qcal_{i',k}$ of the form $\Bcal(\Qcal_{i',k},\diamond_{j=0}^k{\vec{o_j}}\diamond\vec{\xi}_{i',k})$ where $\vec{o_j}\in\{\emptyset,\neg\}^{n_{i_{-j-1}}}$ for $j=0,1,\ldots,k$. Moreover, by Lemma \ref{masterLem}, the number of such pre-images in that block is
\[
\sigma_{\Qcal_{i',k},A_{i'}}(\diamond_{j=0}^k{\vec{o_j}}\diamond\vec{\xi}_{i',k},\diamond_{j=1}^{k+1}{\vec{o'_j}}\diamond\vec{\xi}_{i,h}).
\]
We observe further that each pre-image $y$ of $x$ in $C_{i'}$ which is contained in $\Bcal(\Qcal_{i',k},\diamond_{j=0}^k{\vec{o_j}}\diamond\vec{\xi}_{i',k})$ is also contained in $\Bcal(\Pcal_{i',k},\diamond_{j=0}^k{\vec{o_j}})$, whence $\Tree_{\Gamma_f}(y)\cong\Tree^{(k)}_{i'}(\Pcal_{i',k},\diamond_{j=0}^k{\vec{o_j}})$. This concludes the proof of the formula for $\Tree^{(h)}_i(\Scal_{i,h},C_{i'},\vec{\nu}^{(\Scal_{i,h})})$.

The formula for $\Tree_i^{(h)}(\Pcal_{i,h},\vec{\nu}^{(\Pcal_{i,h})})$ is clear because
\[
\Tree_{\Gamma_f}(x)=\Tree_{\Gamma_f}(x,\bigcup_{t=1}^K{C_{j_t}})+\Tree_{\Gamma_f}(x,C_{i'})
\]
for all $x\in C_i$.
\end{proof}

Now, let us set $\Pcal_i:=\Qcal_{i,H_i}=\Pcal_{i,H_i}\wedge\Ucal_i$. Putting everything together, we obtain the following concluding result for this subsection.

\begin{propposition}\label{periodicCosetsProp}
Let $\Bcal(\Pcal_i,\vec{\nu}^{(\Pcal_i)})$ be a block of $\Pcal_i$. We can view $\vec{\nu}^{(\Pcal_i)}$ as the concatenation $\diamond_{t=0}^{H_i}{\vec{o'_t}}\diamond\vec{\xi}$ where $\vec{o'_t}\in\{\emptyset,\neg\}^{n_{i_{-t}}}$ for $t=0,1,\ldots,H_i$, and $\vec{\xi}=\vec{\xi}_{i,h}$ for a unique $h\in\{0,1,\ldots,H_i\}$. Then for $x\in\Bcal(\Pcal_i,\vec{\nu}^{(\Pcal_i)})$, the isomorphism type of $\Tree_{\Gamma_f}(x)$ does not depend on $x$, is denoted by $\Tree_i(\Pcal_i,\vec{\nu}^{(\Pcal_i)})$ and given by the formula
\[
\Tree_i(\Pcal_i,\vec{\nu}^{(\Pcal_i)})=\Tree_i^{(h)}(\Pcal_{i,h},\diamond_{t=0}^h{\vec{o'_t}}).
\]
Moreover, we have
\[
\AC(\Pcal_i)\leq d^2\mpe(q-1)+d-1\in O(d^2\mpe(q-1)).
\]
\end{propposition}

\begin{proof}
The block $\Bcal(\Pcal_i,\vec{\nu}^{(\Pcal_i)})$ is contained in $\Bcal(\Ucal_i,\xi_{i,h})$, whence $\hfrak(v)=h$. Additionally, $\Bcal(\Pcal_i,\vec{\nu}^{(\Pcal_i)})$ is contained in $\Bcal(\Pcal_{i,h},\diamond_{t=0}^h{\vec{o'_t}})$, so the asserted formula for $\Tree_i(\Pcal_i,\vec{\nu}^{(\Pcal_i)})$ is clear by Proposition \ref{periodicCosetsPeriodicProp}(2). For the complexity bound, we note that by definition of $\Pcal_i=\Pcal_{i,H_i}\wedge\Ucal_i$, we have
\begin{align*}
\AC(\Pcal_i) &\leq\AC(\Pcal_{i,H_i})+\AC(\Ucal_i)\leq\sum_{t=0}^{H_i}{\AC(\lambda_{i_{-t}}^t(\Rcal_{i_{-t}}))}+H_i\leq(H_i+1)(d-1)+H_i \\
&=H_id+d-1\leq d^2\mpe(q-1)+d-1,
\end{align*}
as required.
\end{proof}

\begin{remmark}\label{trivialBoundRem}
Omitting all explicit details which we worked out in this subsection, we basically proved that on each coset $C_i$, there is an arithmetic partition $\Pcal_i$ with $\AC(\Pcal_i)\in O(d^2\mpe(q-1))\subseteq O(d^2\log{q})$ which \enquote{controls} the rooted trees above vertices in $C_i$. Now, the trivial partition $\Tcal_s$\phantomsection\label{not169} of $C_i\cong\IZ/s\IZ$ which consists entirely of singleton blocks also \enquote{controls} the trees above its blocks, for trivial reasons. While it is preferable for our effective purposes to subsume as many isomorphic rooted trees under a common block as possible (and thus $\Pcal_i$ is in general preferable over $\Tcal_s$), it is an interesting question whether $\Tcal_s$ could \enquote{beat} $\Pcal_i$ at least as far as arithmetic complexity is concerned.

Let us discuss this problem for a general modulus $m\in\IN^+$ (not just $s=(q-1)/d$). We consider the factorization $m=p_1^{v_1}p_2^{v_2}\cdots p_K^{v_K}$ of $m$ into pairwise coprime prime powers. By adding logical signs to the $m$-CCs in the system consisting of $x\equiv b\Mod{p_j^{v_j}}$ for $j=1,2,\ldots,K$ and $b\in\{0,1,\ldots,p_j^{v_j}-2\}$, one can obtain each singleton subset of $\IZ/m\IZ$ as a block of the corresponding arithmetic partition. This shows that $\AC(\Tcal_m)\leq p_1^{v_1}+\cdots+p_K^{v_K}-K$. Of course, if $m$ is a prime power, then this bound is just $m-1$. On the other hand, if $m=p_1p_2\cdots p_K=p_K\#$ is a primorial, then because $p_K\#=\exp((1+o(1))K\log{K})$, and thus $\log{m}\sim K\log{K}$, i.e., recalling from Remark \ref{sArithRem} that $W$ denotes the Lambert W function, we have
\[
K\sim\frac{\log{m}}{W(\log{m})}\sim\frac{\log{m}}{\log\log{m}},
\]
and our bound implies that $\AC(\Tcal_m)$ is at most
\begin{align*}
p_1+\cdots+p_K-K &\sim p_1+\cdots+p_K\sim\frac{1}{2}n^2\log{n} \\
&\sim\frac{1}{2}\frac{\log^2{m}}{\log\log^2{m}}(\log\log{m}-\log\log\log{m})\sim\frac{1}{2}\frac{\log^2{m}}{\log\log{m}},
\end{align*}
which does not beat $O(d^2\log{m})$ for fixed $d$, and even less so $O(d^2\mpe(m))=O(d^2)$, noting that $\mpe(m)=\mpe(p_K\#)=1$. It is an interesting open question whether
\[
\liminf_{m\to\infty}{\frac{\AC(\Tcal_m)}{\log^2(m)/\log\log{m}}}>0,
\]
see also Question \ref{liminfQues}.
\end{remmark}

\subsection{Understanding the connected components}\label{subsec3P4}

Now we want to combine the theory developed thus far to understand the connected components of $\Gamma_f$ in their entirety. From the introduction, we recall our approach of associating a necklace of rooted tree isomorphism types with each connected component of $\Gamma_f$, which characterizes the digraph isomorphism type of that component.

Let $\Lcal$ be a CRL-list of $f$ (see Subsection \ref{subsec3P1} on how to construct $\Lcal$). We remind the reader that the first entries of the pairs in $\Lcal$ are representatives not only for the cycles of $f$, but also for the connected components of $\Gamma_f$. Let us fix $(r,l)\in\Lcal$. We aim to give a neat description of the cyclic sequence of rooted tree isomorphism types associated with the connected component of $\Gamma_f$ containing $r$.

We note that by our construction of $\Lcal$, if $r=0_{\IF_q}$, then $l=1$, and the connected component consists of a single rooted tree attached to the looped vertex $0_{\IF_q}$. We can determine that tree, $\Tree_{\Gamma_f}(0_{\IF_q})$, as described in Subsection \ref{subsec3P3}. The length $1$ cyclic sequence $[\Tree_{\Gamma_f}(0_{\IF_q})]$ determines the connected component of $0_{\IF_q}$ as a whole, and so we may henceforth assume that $r\not=0_{\IF_q}$, contained in a unique coset $C_i$ of $C$ in $\IF_q^{\ast}$.

We recall that $i$ is necessarily a periodic point of $\overline{f}$, and let $(i_0,i_1,\ldots,i_{\ell-1})$ with $i_0=i$ be its cycle. For general $t\in\IZ$, we set $i_t:=i_{t\bmod{\ell}}$. By Subsection \ref{subsec3P3}, on each coset $C_j$ of $C$ in $\IF_q^{\ast}$, we have an arithmetic partition $\Pcal_j$ such that the isomorphism type of $\Tree_{\Gamma_f}(x)$ is constant and explicitly understood for vertices $x\in C_j$ chosen from a common block $\Bcal(\Pcal_j,\vec{\nu})$ of $\Pcal_j$, and we denote the said isomorphism type by $\Tree_j(\Pcal_j,\vec{\nu})$.

Let us fix $t\in\{0,1,\ldots,\ell-1\}$ and recall (from the previous subsection) the notation $\Acal_{i_t}=A_{i_t}A_{i_{t+1}}\cdots A_{i_{t+\ell-1}}$ for the product of all affine maps along the cycle of $i_t$. Also, we recall that $f^{\ell}$ stabilizes $C_{i_t}$, and that the restriction $(f^{\ell})_{\mid C_{i_t}}$ corresponds to the affine map $\Acal_{i_t}$ of $\IZ/s\IZ$. Let us set $r'_{i_t}:=f^t(r)$\phantomsection\label{not170}. We have $\ell\mid l$, and the cycle of $r'_{i_t}$ under $\Acal_{i_t}$ is
\[
(r'_{i_t},\Acal_{i_t}(r'_{i_t}),\Acal_{i_t}^2(r'_{i_t}),\ldots,\Acal_{i_t}^{l/\ell-1}(r'_{i_t}))=(f^t(r),f^{t+\ell}(r),f^{t+2\ell}(r),\ldots,f^{t+((l/\ell)-1)\ell}(r)).
\]
We denote by $\Bcal_{i_t}$\phantomsection\label{not171} the function defined on $C_{i_t}$ that maps $x\in C_{i_t}$ to the unique block of $\Pcal_{i_t}$ containing $x$. Our next goal is to understand the block sequence
\[
(\Bcal_{i_t}(r'_{i_t}),\Bcal_{i_t}(\Acal_{i_t}(r'_{i_t})),\ldots,\Bcal_{i_t}(\Acal_{i_t}^{l/\ell-1}(r'_{i_t})))
\]
along the $\Acal_{i_t}$-cycle of $r'_{i_t}$ for each $t=0,1,\ldots,\ell-1$, because those sequences combine to the block sequence $(\Bcal_{i_n}(f^n(r)))_{n=0,1,\ldots,l-1}$, from which one can read off the cyclic sequence of rooted tree isomorphism types for the connected component of $\Gamma_f$ containing $r$.

Now, let us assume that $\Pcal_{i_t}=\Pfrak(x\equiv \bfrak_{i_t,j}\Mod{\afrak_{i_t,j}}: j=1,2,\ldots,m_{i_t})$\phantomsection\label{not173}\phantomsection\label{not174}. Understanding the block sequence $(\Bcal_{i_t}(\Acal_{i_t}^m(r'_{i_t})))_{m=0,1,\ldots,l/\ell-1}$ means understanding the truth values of the congruences $x\equiv \bfrak_{i_t,j}\Mod{\afrak_{i_t,j}}$ as $x$ ranges over the cycle of $r'_{i_t}$ under $\Acal_{i_t}$. We recall from the previous subsection that $\overline{\alpha}_{i_t}:=\prod_{t=0}^{\ell-1}{\alpha_{i_t}}$ is the linear coefficient of $\Acal_{i_t}$. Lemma \ref{periodicCharLem} implies that all periodic points of $\Acal_{i_t}$ (in particular all points on the cycle of $r'_{i_t}$ under $\Acal_{i_t}$) have one particular, explicitly known value modulo $\prod_{p\mid\gcd(\overline{\alpha}_{i_t},s)}{p^{\nu_p(s)}}$. This may cause some of the congruences $x\equiv \bfrak_{i_t,j}\Mod{\afrak_{i_t,j}}$ for $j\in\{1,2,\ldots,m_{i_t}\}$ to have constant truth value on all periodic points of $\Acal_{i_t}$. The remaining congruences can be dealt with as follows.

We compute $\log_{\Acal_{i_t}}^{(\afrak_{i_t,j})}(r'_{i_t},\bfrak_{i_t,j})=:\lfrak_{i_t,j}$\phantomsection\label{not175} (see Subsection \ref{subsec2P4} for this discrete log notation) and the cycle length $l_{i_t,j}$\phantomsection\label{not176} of $r'_{i_t}$ under $\Acal_{i_t}$ modulo $\afrak_{i_t,j}$. If $\lfrak_{i_t,j}=\infty$ (i.e., $\bfrak_{i_t,j}$ does not lie on the cycle of $r'_{i_t}$ under $\Acal_{i_t}$ modulo $\afrak_{i_t,j}$), then the congruence $x\equiv \bfrak_{i_t,j}\Mod{\afrak_{i_t,j}}$ is false for all $x$ on the $\Acal_{i_t}$-cycle of $r'_{i_t}$ modulo $s$. Otherwise, that congruence is true precisely for those $x=\Acal_{i_t}^y(r'_{i_t})$ for which $y\equiv \lfrak_{i_t,j}\Mod{l_{i_t,j}}$. Let us denote by $I_{i_t}$\phantomsection\label{not177} the set of those $j\in\{1,2,\ldots,m_{i_t}\}$ for which the truth value of $x\equiv \bfrak_{i_t,j}\Mod{\afrak_{i_t,j}}$ is constant along the $\Acal_{i_t}$-cycle of $r'_{i_t}$ modulo $s$. The block sequence $(\Bcal_{i_t}(\Acal_{i_t}^m(r'_{i_t})))_{m=0,1,\ldots,l/\ell-1}$ is determined by
\begin{itemize}
\item the information about the constant truth values along the $\Acal_{i_t}$-cycle of $r'_{i_t}$ modulo $s$ of the congruences $x\equiv \bfrak_{i_t,j}\Mod{\afrak_{i_t,j}}$ for $j\in I_{i_t}$, and
\item the arithmetic partition $\Pcal^{(i_t)}:=\Pfrak(y\equiv \lfrak_{i_t,j}\Mod{l_{i_t,j}}: j\notin I_{i_t})$\phantomsection\label{not178} of $\IZ/(l/\ell)\IZ$, which encodes the behavior of the truth values of the remaining congruences $x\equiv \bfrak_{i_t,j}\Mod{\afrak_{i_t,j}}$, for $j\notin I_{i_t}$, along the cycle.
\end{itemize}
Once the block sequence $(\Bcal_{i_t}(\Acal_{i_t}^m(r'_{i_t})))_{m=0,1,\ldots,l/\ell-1}$ has been understood that way for each $t$, the actual necklace that encodes the isomorphism type of the connected component of $\Gamma_f$ containing $r$ is given by the cyclic sequence
\begin{equation}\label{cyclicSeqEq}
[\Tree_{i_n}(\Pcal_{i_n},\Bcal_{i_n}(\Acal_{i_n}^{(n-(n\bmod{\ell}))/\ell}(r'_{i_n})))]_{n=0,1,\ldots,l-1}
\end{equation}
where, by abuse of notation, $\Tree_i(\Pcal_i,B)$\phantomsection\label{not179} is to be understood as $\Tree_i(\Pcal_i,\vec{\nu})$ for $B=\Bcal(\Pcal_i,\vec{\nu})$.

The connected components of $\Gamma_f$ associated with two different choices for $r$ are isomorphic if and only if the corresponding cyclic sequences (\ref{cyclicSeqEq}) are equal. We do note, however, that it does not appear obvious how to check this efficiently, just as it does not seem clear how to check efficiently whether two given arithmetic partitions are equal -- see Problems \ref{labeledProb1} and \ref{labeledProb2}.

\section{Computations and examples}\label{sec4}

The goal of this section is to illustrate the theory developed thus far through some concrete computations and examples/special cases. We start by introducing a useful notation that is used in Subsections \ref{subsec4P1} and \ref{subsec4P2}. We consider finite directed rooted trees (with all arcs oriented toward the root) that have non-negative integers as edge weights. An \emph{isomorphism}\phantomsection\label{term52} of such graphs is one of the underlying non-edge-weighted directed graphs that preserves the weights of arcs. With each isomorphism type $\Ifrak$ of such trees, we associate an isomorphism type $\Expand(\Ifrak)$\phantomsection\label{not180} of non-edge-weighted, finite directed rooted trees as follows. If $y_1,y_2,\ldots,y_n$ are the neighbors of the root $x$ of $\Ifrak$, and they have the (isomorphism types of) edge-weighted rooted trees $\Ifrak_1,\Ifrak_2,\ldots,\Ifrak_n$ attached to them and the edge joining $x$ and $y_j$ has weight $\wfrak_j\in\IN_0$\phantomsection\label{not181}, then $\Expand(\Ifrak)$ is defined recursively by taking a new root and attaching $\wfrak_j$ copies of $\Expand(\Ifrak_j)$ to it for each $j=1,2,\ldots,n$. For example, if $\Ifrak$ is
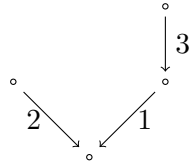
\begin{center}
\begin{tikzpicture}
\node (r) at (0,0) {};
\draw (r) circle [radius=1pt];
\node (c1) at (-1,1) {};
\draw (c1) circle [radius=1pt];
\node (c2) at (1,1) {};
\draw (c2) circle [radius=1pt];
\node (v) at (1,2) {};
\draw (v) circle [radius=1pt];
\path
(c1) edge[->] node[left] {2} (r)
(c2) edge[->] node[right] {1} (r)
(v) edge[->] node[right] {3} (c2);
\end{tikzpicture}
\end{center}
then $\Expand(\Ifrak)$ is
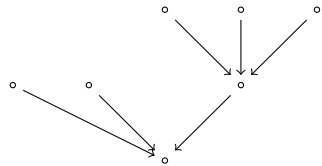
\begin{center}
\begin{tikzpicture}
\node (r) at (0,0) {};
\draw (r) circle [radius=1pt];
\node (c1) at (-2,1) {};
\draw (c1) circle [radius=1pt];
\node (c2) at (-1,1) {};
\draw (c2) circle [radius=1pt];
\node (c3) at (1,1) {};
\draw (c3) circle [radius=1pt];
\node (v1) at (0,2) {};
\draw (v1) circle [radius=1pt];
\node (v2) at (1,2) {};
\draw (v2) circle [radius=1pt];
\node (v3) at (2,2) {};
\draw (v3) circle [radius=1pt];
\path
(c1) edge[->] node {} (r)
(c2) edge[->] node {} (r)
(c3) edge[->] node {} (r)
(v1) edge[->] node {} (c3)
(v2) edge[->] node {} (c3)
(v3) edge[->] node {} (c3);
\end{tikzpicture}
\end{center}
By recursion on the height of $\Ifrak$, we define $\Ifrak$ to be \emph{simplified}\label{term53} as follows. The trivial isomorphism type $\Ifrak$ is simplified, and if $\Ifrak$ is of positive height, then $\Ifrak$ is simplified if all isomorphism types $\Ifrak'$ attached to the root of $\Ifrak$ (which are all of smaller height than $\Ifrak$) are simplified, pairwise distinct, and none of them is attached to the root with edge weight $0$. The simplified isomorphism types of finite edge-weighted directed rooted trees are in bijection with the isomorphism types of finite directed rooted trees via $\Expand$ (as can be easily proved by induction on the height). This allows us to define the \emph{simplified form}\phantomsection\label{term54} $\SF(\Ifrak)$\phantomsection\label{not182} of an arbitrary isomorphism type $\Ifrak$ of finite edge-weighted directed rooted trees as the unique simplified isomorphism type such that $\Expand(\SF(\Ifrak))=\Expand(\Ifrak)$. We write $\Ifrak\sim\Ifrak'$\phantomsection\label{not183} for $\SF(\Ifrak)=\SF(\Ifrak')$ (equivalently, $\Expand(\Ifrak)=\Expand(\Ifrak')$). The simplified form of $\Ifrak$ can be constructed explicitly from $\Ifrak$ in a simple recursion on the tree height (going through the $\Ifrak'$ attached to the root of $\Ifrak$, computing their simplified forms, and adding up edge weights that belong to the same $\SF(\Ifrak')$).

In Subsection \ref{subsec3P3}, we introduced a sum of isomorphism types of non-edge-weighted finite directed rooted trees (turning their class into a class-sized monoid), and there is a unique way to define a \emph{sum of simplified edge-weighted rooted tree isomorphism types}\phantomsection\label{term55} such that $\Expand$ becomes a monoid isomorphism (i.e., $\Expand(\Ifrak_1+\Ifrak_2)=\Expand(\Ifrak_1)+\Expand(\Ifrak_2)$). Explicitly, $\Ifrak_1+\Ifrak_2$ may be defined as follows. Pick a new root $x$ and consider the edge-weighted rooted trees $\Ifrak'$ that are attached to the root in $\Ifrak_1$ or $\Ifrak_2$. Let $\wfrak_j$ for $j=1,2$ be the weight of the arc that attaches $\Ifrak'$ to the root in $\Ifrak_j$ (treating $\wfrak_j$ as $0$ if such an arc does not exist). For each such $\Ifrak'$, attach a copy of $\Ifrak'$ to $x$ through an arc with weight $\wfrak_1+\wfrak_2$. For example,
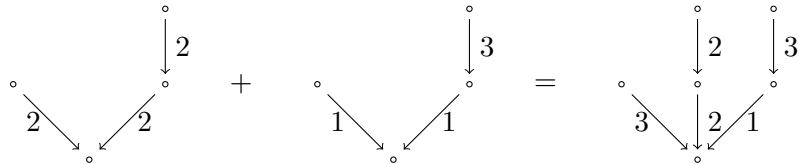
\begin{center}
\begin{tikzpicture}
\node (r1) at (-3,0) {};
\draw (r1) circle [radius=1pt];
\node (a1) at (-4,1) {};
\draw (a1) circle [radius=1pt];
\node (b1) at (-2,1) {};
\draw (b1) circle [radius=1pt];
\node (c1) at (-2,2) {};
\draw (c1) circle [radius=1pt];
\path
(a1) edge[->] node[left] {$2$} (r1)
(b1) edge[->] node[right] {$2$} (r1)
(c1) edge[->] node[right] {$2$} (b1);
\node at (-1,1) {$+$};
\node (r2) at (1,0) {};
\draw (r2) circle [radius=1pt];
\node (a2) at (0,1) {};
\draw (a2) circle [radius=1pt];
\node (b2) at (2,1) {};
\draw (b2) circle [radius=1pt];
\node (c2) at (2,2) {};
\draw (c2) circle [radius=1pt];
\path
(a2) edge[->] node[left] {$1$} (r2)
(b2) edge[->] node[right] {$1$} (r2)
(c2) edge[->] node[right] {$3$} (b2);
\node at (3,1) {$=$};
\node (r3) at (5,0) {};
\draw (r3) circle [radius=1pt];
\node (a3) at (4,1) {};
\draw (a3) circle [radius=1pt];
\node (b3) at (6,1) {};
\draw (b3) circle [radius=1pt];
\node (c3) at (6,2) {};
\draw (c3) circle [radius=1pt];
\node (d3) at (5,1) {};
\draw (d3) circle [radius=1pt];
\node (e3) at (5,2) {};
\draw (e3) circle [radius=1pt];
\path
(a3) edge[->] node[left] {$3$} (r3)
(b3) edge[->] node[right] {$1$} (r3)
(c3) edge[->] node[right] {$3$} (b3)
(d3) edge[->] node[right] {$2$} (r3)
(e3) edge[->] node[right] {$2$} (d3);
\end{tikzpicture}
\end{center}
This addition can be extended to arbitrary isomorphism types of finite edge-weighted directed rooted trees by setting $\Ifrak_1+\Ifrak_2:=\SF(\Ifrak_1)+\SF(\Ifrak_2)$. Henceforth, we frequently drop the word \enquote{isomorphism type} (thus identifying a finite (edge-weighted) directed rooted tree with its isomorphism type) for the sake of simplicity.

\subsection{Rooted trees under rigid procreation}\label{subsec4P1}

Let $\Gamma=\Gamma_g$ be a finite functional graph such that $\Gamma^{\ast}$ has rigid procreation (see Definition \ref{childProcDef}(4)). Moreover, let $x\in\V(\Gamma)$ be $g$-periodic, and let $(\proc_k(x))_{k\geq1}$ be the sequence of procreation numbers of $x$ in $\Gamma^{\ast}$ (which is independent of $x$ due to rigid procreation). Proposition \ref{rigidProp} states that the isomorphism type of the rooted tree $\Tree_{\Gamma}(x)$, which also does not depend on the choice of $g$-periodic vertex $x$, is entirely determined by this sequence of procreation numbers. We would like to understand explicitly how that isomorphism type can be derived from $(\proc_k(x))_{k\geq1}$.

Since $\Gamma$ is finite, so is $H:=\height(\Tree_{\Gamma}(x))$. We observe that $x$ has $g$-transient children with at least $H-1$ successor generations in $\Gamma^{\ast}$, but no such children with at least $H$ successor generations. This means that $\proc_h(x)>1$ for $h=1,2,\ldots,H$, but $\proc_{H+1}(x)=1$; the unique $g$-periodic child of $x$ in $\Gamma^{\ast}$ has infinitely many successor generations, thus providing a contribution of $1$ to all procreation numbers. This allows us to read off $H$ from the sequence $(\proc_k(x))_{k\geq1}$ alone.

Now, Lemma \ref{rigidLem} implies that for all $g$-transient vertices $y\in\V(\Gamma)$ with a fixed tree height $h\in\{0,1,\ldots,H-1\}$ above them in $\Gamma$, the rooted tree isomorphism type $\Tree_{\Gamma}(y)$ is always the same. We recursively define a (not necessarily simplified) edge-weighted directed rooted tree $\Ifrak_h=\Ifrak_h((\proc_k(x))_{k\geq1})$ such that the said isomorphism type is $\Expand(\Ifrak_h)$. Clearly, the only choice for $\Ifrak_0$ is a single vertex without arcs. If $h\in\{1,2,\ldots,H-1\}$, then we define $\Ifrak_h$ as follows. We fix a new root, and
\begin{itemize}
\item for $k=0,1,\ldots,h-2$, we attach a copy of $\Ifrak_k$ to the new root with edge weight $\proc_{k+1}(x)-\proc_{k+2}(x)=:\wfrak_k$; and
\item we attach a copy of $\Ifrak_{h-1}$ to the new root with edge weight $\proc_h(x)$.
\end{itemize}
Here is a visual version of this definition.
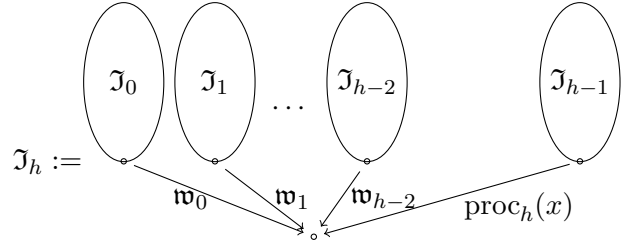
\begin{center}
\begin{tikzpicture}
\node (r) at (0,0) {};
\draw (r) circle [radius=1pt];
\node at (-3.5,1) {$\Ifrak_h:=$};
\node (c0) at (-2.5,1) {};
\draw (c0) circle [radius=1pt];
\node (m0) at (-2.5,2.05) {$\Ifrak_0$};
\draw (m0) circle [x radius=15pt, y radius=30pt];
\node (c1) at (-1.3,1) {};
\draw (c1) circle [radius=1pt];
\node (m1) at (-1.3,2.05) {$\Ifrak_1$};
\draw (m1) circle [x radius=15pt, y radius=30pt];
\node (dots) at (-0.3,1.7) {$\cdots$};
\node (chM2) at (0.7,1) {};
\draw (chM2) circle [radius=1pt];
\node (mhM2) at (0.7,2.05) {$\Ifrak_{h-2}$};
\draw (mhM2) circle [x radius=15pt, y radius=30pt];
\node (chM1) at (3.5,1) {};
\draw (chM1) circle [radius=1pt];
\node (mhM1) at (3.5,2.05) {$\Ifrak_{h-1}$};
\draw (mhM1) circle [x radius=15pt, y radius=30pt];
\path
(c0) edge[->] node[left] {$\wfrak_0$} (r)
(c1) edge[->] node[right] {$\wfrak_1$} (r)
(chM2) edge[->] node[right] {$\wfrak_{h-2}$} (r)
(chM1) edge[->] (r);
\node (lastLabel) at (2.7,0.4) {$\proc_h(x)$};
\end{tikzpicture}
\end{center}
This definition of $\Ifrak_h$ does the job, because by the proof of Lemma \ref{rigidLem}, for each $k\in\{0,1,\ldots,h-1\}$, the number of children $z$ of $y$ in $\Gamma^{\ast}$ such that $\Tree_{\Gamma}(z)$ has height exactly $k-1$ (and hence is isomorphic to $\Ifrak_{k-1}$ by induction) is equal to
\[
\proc_{k+1}(y)-\proc_{k+2}(y)=
\begin{cases}
\proc_{k+1}(x)-\proc_{k+2}(x)=\wfrak_k, & \text{if }k<h-1, \\
\proc_h(x)-0=\proc_h(x), & \text{if }k=h-1.
\end{cases}
\]
For our fixed periodic vertex $x$, the determination of $\Tree_{\Gamma}(x)$ is analogous, but one must take into account that $x$ has a (unique) $g$-periodic child in $\Gamma^{\ast}$, which does not appear in $\Tree_{\Gamma}(x)$. This means that the weight with which $\Ifrak_k$ for $k\in\{0,1,\ldots,H-2\}$ is attached to the root of $\Ifrak_H$ is
\[
(\proc_{k+1}(x)-1)-(\proc_{k+2}(x)-1)=\wfrak_k,
\]
whereas the weight with which $\Ifrak_{H-1}$ is attached is $\proc_H(x)-1=\proc_H(x)-\proc_{H+1}(x)=:\wfrak_{H-1}$. In short, we obtain the following definition of $\Ifrak_H$ such that $\Expand(\Ifrak_H)\cong\Tree_{\Gamma}(x)$:
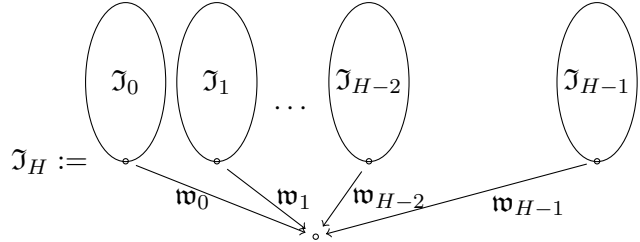
\begin{center}
\begin{tikzpicture}
\node (r) at (0,0) {};
\draw (r) circle [radius=1pt];
\node at (-3.5,1) {$\Ifrak_H:=$};
\node (c0) at (-2.5,1) {};
\draw (c0) circle [radius=1pt];
\node (m0) at (-2.5,2.05) {$\Ifrak_0$};
\draw (m0) circle [x radius=15pt, y radius=30pt];
\node (c1) at (-1.3,1) {};
\draw (c1) circle [radius=1pt];
\node (m1) at (-1.3,2.05) {$\Ifrak_1$};
\draw (m1) circle [x radius=15pt, y radius=30pt];
\node (dots) at (-0.3,1.7) {$\cdots$};
\node (chM2) at (0.7,1) {};
\draw (chM2) circle [radius=1pt];
\node (mhM2) at (0.7,2.05) {$\Ifrak_{H-2}$};
\draw (mhM2) circle [x radius=15pt, y radius=30pt];
\node (chM1) at (3.7,1) {};
\draw (chM1) circle [radius=1pt];
\node (mhM1) at (3.7,2.05) {$\Ifrak_{H-1}$};
\draw (mhM1) circle [x radius=15pt, y radius=30pt];
\path
(c0) edge[->] node[left] {$\wfrak_0$} (r)
(c1) edge[->] node[right] {$\wfrak_1$} (r)
(chM2) edge[->] node[right] {$\wfrak_{H-2}$} (r)
(chM1) edge[->] (r);
\node (lastLabel) at (2.8,0.4) {$\wfrak_{H-1}$};
\end{tikzpicture}
\end{center}
We can use similar ideas to describe, for each index $d$ generalized cyclotomic mapping $f$ of $\IF_q$, the rooted trees above non-zero periodic vertices in $\Gamma_{\per}$, the induced subgraph of $\Gamma_f$ on the union of all periodic blocks $C_i$ (in particular, we can obtain such a description for $\Gamma_f$ as a whole in case $\overline{f}$ is a permutation). Let $i\in\{0,1,\ldots,d-1\}$ be $\overline{f}$-periodic, with $\overline{f}$-cycle $(i_0,i_1,\ldots,i_{\ell-1})$ where $i_0=i$. For $t\in\IZ$, we set $i_t:=i_{t\bmod{\ell}}$. Theorem \ref{cosetPermTheo} states that for fixed $t\in\IZ$ and $h\in\IN^+$, if $x,y\in C_{i_t}$ each have at least $h$ successor generations in $\Gamma_{\per}^{\ast}$ (i.e., if $\min\{\proc^{(\Gamma_{\per}^{\ast})}_h(x),\proc^{(\Gamma_{\per}^{\ast})}_h(y)\}>0$), then $\proc_h^{(\Gamma_{\per}^{\ast})}(x)=\proc_h^{(\Gamma_{\per}^{\ast})}(y)$. This allows us to set $\proc_{i_t,h}:=\proc_h^{(\Gamma_{\per}^{\ast})}(x)$ for any $x\in C_{i_t}$ with at least $h$ successor generations in $\Gamma_{\per}^{\ast}$ (such as an $f$-periodic $x$); this notation agrees with the one used in the proof of Theorem \ref{cosetPermTheo}.

According to the comment before Theorem \ref{cosetPermTheo}, for periodic $x\in C_{i_t}$, the isomorphism type of $\Tree_{\Gamma_{\per}}(x)$ only depends on $i_t$ and the numbers $\proc_{i_{t'},h}$ for $h\geq1$ and $t'\in\IZ$ (i.e., it is independent of the choice of $x$). We describe how to read off this rooted tree from the data it depends on. We set
\[
\Hcal_{i_t}:=\min\{h\in\IN^+: \proc_{i_t,h}=1\}-1,
\]
the\phantomsection\label{not184} common height of the rooted trees in $\Gamma_{\per}$ above periodic vertices in $C_{i_t}$. Moreover, as in Subsection \ref{subsec3P3}, we let $H_i:=\max\{\Hcal_{i_t}: t=0,1,\ldots,\ell-1\}$, the maximum such tree height along the cycle of $i$. Then, in generalization of what was stated above, for all vertices $x\in C_{i_t}$, the isomorphism type of $\Tree_{\Gamma_{\per}}(x)$ only depends on $i_t$, the numbers $\proc_{i_{t'},h}$ and the $\hfrak$-value of $x$ (see formula (\ref{hDefEq}) in Subsection \ref{subsec3P3} for the definition of $\hfrak$). It should be noted that $\hfrak$ does not necessarily assume all of its possible values $0,1,\ldots,H_i$ on each coset $C_{i_t}$ (see Example \ref{cosetPermEx}), but this is not an issue for our construction.

We\phantomsection\label{latterHalfRef} recursively define edge-weighted rooted trees $\Ifrak_{i_t,h}$ such that the rooted tree in $\Gamma_{\per}$ above any $x\in C_{i_t}$ with $\hfrak(x)=h$ is isomorphic to $\Expand(\Ifrak_{i_t,h})$, a property that is certainly satisfied whenever there are no $x\in C_{i_t}$ of that $\hfrak$-value. For $k\in\IN_0$, we set $\wfrak_{i_t,k}:=\proc_{i_t,k+1}-\proc_{i_t,k+2}$. We define $\Ifrak_{i_t,0}$ to be the trivial rooted tree. If $h\in\{1,2,\ldots,H_i-1\}$ (we observe that vertices in $C_{i_t}$ of such an $\hfrak$-value are $f$-transient), then we set
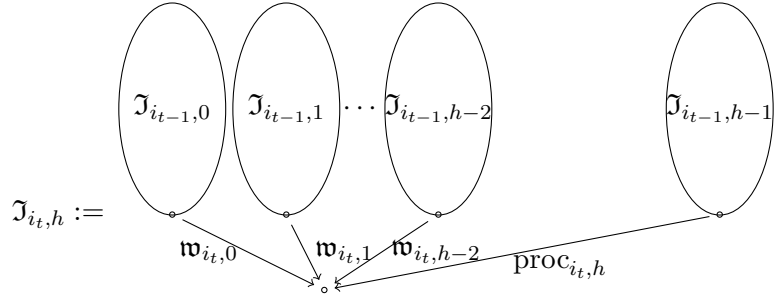
\begin{center}
\begin{tikzpicture}
\node (r) at (0,0) {};
\draw (r) circle [radius=1pt];
\node at (-3.5,1) {$\Ifrak_{i_t,h}:=$};
\node (c0) at (-2,1) {};
\draw (c0) circle [radius=1pt];
\node (m0) at (-2,2.4) {$\Ifrak_{i_{t-1},0}$};
\draw (m0) circle [x radius=20pt, y radius=40pt];
\node (c1) at (-0.5,1) {};
\draw (c1) circle [radius=1pt];
\node (m1) at (-0.5,2.4) {$\Ifrak_{i_{t-1},1}$};
\draw (m1) circle [x radius=20pt, y radius=40pt];
\node (dots) at (0.5,2.4) {$\cdots$};
\node (chM2) at (1.5,1) {};
\draw (chM2) circle [radius=1pt];
\node (mhM2) at (1.5,2.4) {$\Ifrak_{i_{t-1},h-2}$};
\draw (mhM2) circle [x radius=20pt, y radius=40pt];
\node (chM1) at (5.2,1) {};
\draw (chM1) circle [radius=1pt];
\node (mhM1) at (5.2,2.4) {$\Ifrak_{i_{t-1},h-1}$};
\draw (mhM1) circle [x radius=20pt, y radius=40pt];
\path
(c0) edge[->] node[left] {$\wfrak_{i_t,0}$} (r)
(c1) edge[->] node[right] {$\wfrak_{i_t,1}$} (r)
(chM2) edge[->] node[right] {$\wfrak_{i_t,h-2}$} (r)
(chM1) edge[->] (r);
\node (lastLabel) at (3.1,0.3) {$\proc_{i_t,h}$};
\end{tikzpicture}
\end{center}
Finally, the rooted tree above any vertex in $C_{i_t}$ that is $f$-periodic (equivalently, which has $\hfrak$-value $H_i$) may be constructed as
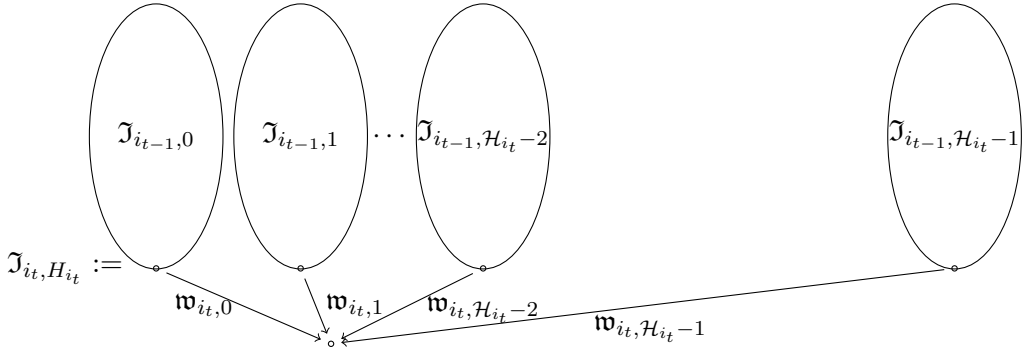
\begin{center}
\begin{tikzpicture}
\node (r) at (0,0) {};
\draw (r) circle [radius=1pt];
\node at (-3.5,1) {$\Ifrak_{i_t,H_{i_t}}:=$};
\node (c0) at (-2.3,1) {};
\draw (c0) circle [radius=1pt];
\node (m0) at (-2.3,2.75) {$\Ifrak_{i_{t-1},0}$};
\draw (m0) circle [x radius=25pt, y radius=50pt];
\node (c1) at (-0.4,1) {};
\draw (c1) circle [radius=1pt];
\node (m1) at (-0.4,2.75) {$\Ifrak_{i_{t-1},1}$};
\draw (m1) circle [x radius=25pt, y radius=50pt];
\node (dots) at (0.8,2.75) {$\cdots$};
\node (chM2) at (2,1) {};
\draw (chM2) circle [radius=1pt];
\node (mhM2) at (2,2.75) {$\Ifrak_{i_{t-1},\Hcal_{i_t}-2}$};
\draw (mhM2) circle [x radius=25pt, y radius=50pt];
\node (chM1) at (8.2,1) {};
\draw (chM1) circle [radius=1pt];
\node (mhM1) at (8.2,2.75) {$\Ifrak_{i_{t-1},\Hcal_{i_t}-1}$};
\draw (mhM1) circle [x radius=25pt, y radius=50pt];
\path
(c0) edge[->] node[left] {$\wfrak_{i_t,0}$} (r)
(c1) edge[->] node[right] {$\wfrak_{i_t,1}$} (r)
(chM2) edge[->] (r)
(chM1) edge[->] (r);
\node (penultimateLabel) at (2,0.45) {$\wfrak_{i_t,\Hcal_{i_t}-2}$};
\node (lastLabel) at (4.2,0.2) {$\wfrak_{i_t,\Hcal_{i_t}-1}$};
\end{tikzpicture}
\end{center}

\subsection{An illustrative example}\label{subsec4P2}

In this subsection, we follow the approach from Section \ref{sec3} to derive the cyclic sequences of rooted tree isomorphism types that characterize the connected components of the functional graph of the following generalized cyclotomic mapping $f$ of $\IF_{2^8}$ of index $d=5$:
\[
f(x)=
\begin{cases}
0, & \text{if }x=0, \\
\omega^5x^9, & \text{if }x\in C_0, \\
x^3, & \text{if }x\in C_1, \\
x^{17}, & \text{if }x\in C_2, \\
\omega^3x^{34}, & \text{if }x\in C_3, \\
\omega^4x^9, & \text{if }x\in C_4,
\end{cases}
\]
where $\omega$ is any fixed primitive element of $\IF_{2^8}$ (the minimal polynomial of $\omega$ over $\IF_2$ is not relevant here). These cyclic sequences were also derived in our introduction from a drawing of $\Gamma_f$ (see the text passage between Definitions \ref{treeAboveDef} and \ref{sArithDef}), but the approach of Section \ref{sec3} is usually more computationally efficient (see Section \ref{sec5}, especially Theorem \ref{complexitiesTheo}).

We observe that if a generalized cyclotomic mapping of a finite field of known index is not given in the above cyclotomic form, but in polynomial form, then one must first convert it into cyclotomic form before one can apply our methods. An algorithm for doing so is \cite[Algorithm 1]{BW22b}.

Because $d=5$, we have $s=(2^8-1)/5=51=3\cdot17$. We view each coset $C_i$ as a copy of $\IZ/51\IZ$ via the bijection $\iota_i:\IZ/51\IZ\rightarrow C_i$, $k\mapsto \omega^{i+5k}$. Let us work out what the monomial formulas for the values of $f$ in the different cases become under this identification. For example, if $x\in C_4$, then $x=\omega^{4+5k}$ for some $k\in\IZ$, and
\[
f(x)=\omega^4x^9=\omega^{4+9\cdot 4+9\cdot 5k}=\omega^{0+5\cdot(9k+8)},
\]
which shows that $f$ maps $C_4$ to $C_0$ via the affine map $x\mapsto 9x+8$. In total, we obtain the following picture describing the mapping behavior of $f$ between the cosets $C_i$ when viewing them as copies of $\IZ/51\IZ$.
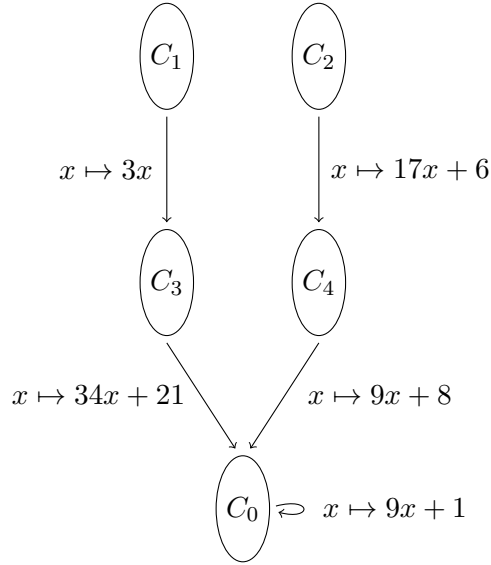
\begin{center}
\begin{tikzpicture}
\draw (0,0) circle [x radius=10pt, y radius=20pt];
\node at (0,0) {$C_0$};
\draw (-1,3) circle [x radius=10pt, y radius=20pt];
\node at (-1,3) {$C_3$};
\draw (1,3) circle [x radius=10pt, y radius=20pt];
\node at (1,3) {$C_4$};
\draw (-1,6) circle [x radius=10pt, y radius=20pt];
\node at (-1,6) {$C_1$};
\draw (1,6) circle [x radius=10pt, y radius=20pt];
\node at (1,6) {$C_2$};
\draw[->] (-1,5.2) -- (-1,3.8);
\node at (-1.8,4.5) {$x\mapsto 3x$};
\draw[->] (1,5.2) -- (1,3.8);
\node at (2.2,4.5) {$x\mapsto 17x+6$};
\draw[->] (-1,2.2) -- (-0.1,0.8);
\node at (-1.9,1.5) {$x\mapsto 34x+21$};
\draw[->] (1,2.2) -- (0.1,0.8);
\node at (1.8,1.5) {$x\mapsto 9x+8$};
\node (v) at (0.3,0) {};
\path
(v) edge [loop right] node {} (v);
\node at (2,0) {$x\mapsto 9x+1$};
\end{tikzpicture}
\end{center}
Until further notice, we put the concrete function $f$ from above aside and assume that, more generally, we have a finite field $\IF_q$ with $5\mid q-1$ and an index $5$ generalized cyclotomic mapping $f$ of $\IF_q$ which maps as follows between the five cosets of $C$ in $\IF_q^{\ast}$, viewed as copies of $\IZ/s\IZ$ (where $s=(q-1)/5$).
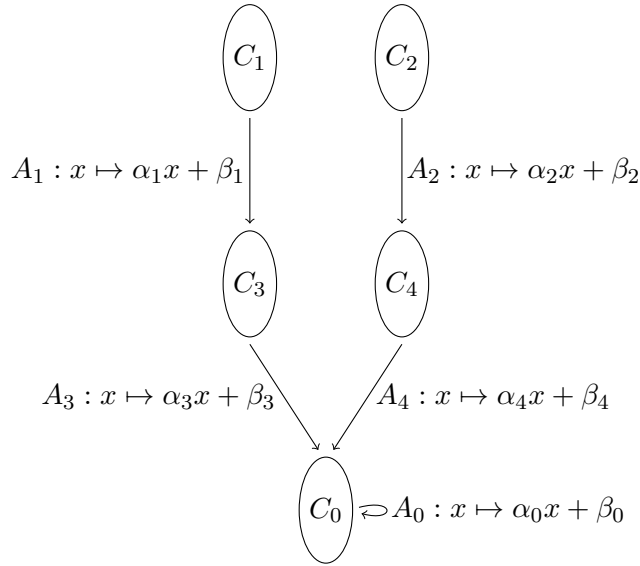
\begin{center}
\begin{tikzpicture}
\draw (0,0) circle [x radius=10pt, y radius=20pt];
\node at (0,0) {$C_0$};
\draw (-1,3) circle [x radius=10pt, y radius=20pt];
\node at (-1,3) {$C_3$};
\draw (1,3) circle [x radius=10pt, y radius=20pt];
\node at (1,3) {$C_4$};
\draw (-1,6) circle [x radius=10pt, y radius=20pt];
\node at (-1,6) {$C_1$};
\draw (1,6) circle [x radius=10pt, y radius=20pt];
\node at (1,6) {$C_2$};
\draw[->] (-1,5.2) -- (-1,3.8);
\node at (-2.6,4.5) {$A_1:x\mapsto \alpha_1x+\beta_1$};
\draw[->] (1,5.2) -- (1,3.8);
\node at (2.6,4.5) {$A_2:x\mapsto \alpha_2x+\beta_2$};
\draw[->] (-1,2.2) -- (-0.1,0.8);
\node at (-2.2,1.5) {$A_3:x\mapsto \alpha_3x+\beta_3$};
\draw[->] (1,2.2) -- (0.1,0.8);
\node at (2.2,1.5) {$A_4:x\mapsto \alpha_4x+\beta_4$};
\node (v) at (0.3,0) {};
\path
(v) edge [loop right] node {} (v);
\node at (2.4,0) {$A_0:x\mapsto \alpha_0x+\beta_0$};
\end{tikzpicture}
\end{center}
This allows us to describe $\Gamma_f$ in terms of those general coefficients $\alpha_i$ and $\beta_i$, which is more instructive; one can actually see the structure of formulas for relevant parameters, such as the moduli $\afrak_{i,j}$ and right-hand sides $\bfrak_{i,j}$ of the spanning congruences of $\Pcal_i$. Just as for our concrete generalized cyclotomic mapping from above, we assume that $\gcd(\alpha_0,s)=\gcd(\alpha_0^2,s)>1$, which means that the rooted trees attached to periodic vertices in the induced subgraph of $\Gamma_f$ on $C_0$ are of height $H_0=1$. We describe the arithmetic partitions $\Pcal_i=\Pfrak(x\equiv \bfrak_{i,j}\Mod{\afrak_{i,j}}: j=1,2,\ldots,m_i)$ and the associated rooted tree isomorphism type $\Tree_i(\Pcal_i,\vec{\nu}^{(\Pcal_i)})$ for each block $\Bcal(\Pcal_i,\vec{\nu}^{(\Pcal_i)})$ of $\Pcal_i$ for $\vec{\nu}^{(\Pcal_i)}\in\{\emptyset,\neg\}^{m_i}$.

The partitions $\Pcal_i$ and associated rooted trees are easily determined for $i=1,2,3,4$.
\begin{itemize}
\item For $i\in\{1,2\}$, every vertex in $C_i$ is a leaf in $\Gamma_f$, and so we may choose $\Pcal_1=\Pcal_2=\Pcal(\emptyset)$ (trivial partition with only one block). There is only one isomorphism type of rooted tree here, $\Tree_i(\Pcal_i,\emptyset)$ (with $\emptyset$ representing an empty sequence of logical signs, not the positive logical sign), and it consists of a single vertex without edges.
\item For $i\in\{3,4\}$, since $C_i$ is a transient coset (i.e., it does not lie on a cycle of cosets under $f$), the discussion in Subsection \ref{subsec3P3} shows that one can obtain $\Pcal_i$ simply as the lift $\Pfrak'(\Pcal_{i-2},A_{i-2})$. According to Lemma \ref{masterLem}, that lift is of the form $\Pcal_i=\Pfrak(x\equiv \beta_{i-2}\Mod{\gcd(\alpha_{i-2},s)})$. The significance of this single congruence is that it characterizes when $x\in C_i$ has at least one pre-image under $f$ in $C_{i-2}$. We note that if this is the case, then $x$ has exactly $\gcd(\alpha_{i-2},s)$ such pre-images, as they form a coset of the kernel of $z\mapsto \alpha_{i-2}z$ in $\IZ/s\IZ$. Hence, $\Tree_i(\Pcal_i,(\neg))$ is a single vertex without arcs, and $\Tree_i(\Pcal_i,(\emptyset))$ consists of a root with $\gcd(\alpha_{i-2},s)$ vertices attached to it.
\end{itemize}
In our discussion for $\Pcal_0$, rather than specify the rooted tree $\Tree_0(\Pcal_0,\vec{\nu}^{(\Pcal_0)})$ associated with a block $\Bcal(\Pcal_0,\vec{\nu}^{(\Pcal_0)})$ of $\Pcal_0$ itself, we specify a (not necessarily simplified isomorphism type of) finite edge-weighted directed rooted tree(s) $\Ifrak=\Ifrak(\Pcal_0,\vec{\nu}^{(\Pcal_0)})$ such that $\Tree_0(\Pcal_0,\vec{\nu}^{(\Pcal_0)})=\Expand(\Ifrak)$. But first, let us determine $\Pcal_0$ itself. We recall that $H_0=1$ by assumption. According to the general definition of $\Pcal_i$ for periodic $i$, which is just before Proposition \ref{periodicCosetsProp}, we have
\[
\Pcal_0=\Qcal_{0,1}=\Pcal_{0,1}\wedge\Ucal_0.
\]
Moreover, noting that $i=0$ lies on a cycle of $\overline{f}$ of length $1$ (so that $i_t=0$ for all $t\in\IZ$ in the notation of Subsection \ref{subsec3P3}), we conclude that
\[
\Pcal_{0,1}=\lambda_0^0(\Rcal_0)\wedge\lambda_0^1(\Rcal_0)=\Rcal_0\wedge\lambda(\Rcal_0,A_0).
\]
Now, $\Rcal_0$ is obtained as the infimum of the $\Pfrak'$-lifts of $\Pcal_3$ and $\Pcal_4$ to $C_0$. Using the notation $(n,m)$ in place of $\gcd(n,m)$ for simplicity, we conclude that
\begin{align}\label{r0Eq}
\notag \Rcal_0&=\Pfrak'(\Pcal_3,A_3)\wedge\Pfrak'(\Pcal_4,A_4)=
\Pfrak
\begingroup
\setlength\arraycolsep{2pt}
\left(
\begin{array}{lll}
x & \equiv & \alpha_3\beta_1+\beta_3\Mod{(\alpha_3(\alpha_1,s),s)} \\
x &\equiv & \beta_3\Mod{(\alpha_3,s)} \\
x & \equiv & \alpha_4\beta_2+\beta_4\Mod{(\alpha_4(\alpha_2,s),s)} \\
x & \equiv & \beta_4\Mod{(\alpha_4,s)}
\end{array}
\right)
\endgroup \\
&=\Pfrak
\begingroup
\setlength\arraycolsep{2pt}
\left(
\begin{array}{lll}
x & \equiv & \alpha_3\beta_1+\beta_3\Mod{(\alpha_1,s)(\alpha_3,\frac{s}{(\alpha_1,s)})} \\
x & \equiv & \beta_3\Mod{(\alpha_3,s)} \\
x & \equiv & \alpha_4\beta_2+\beta_4\Mod{(\alpha_2,s)(\alpha_4,\frac{s}{(\alpha_2,s)})} \\
x & \equiv & \beta_4\Mod{(\alpha_4,s)}
\end{array}
\right)
\endgroup,
\end{align}
and thus
\begin{align*}
\lambda(\Rcal_0,A_0)
&=\Pfrak
\begingroup
\setlength\arraycolsep{2pt}
\left(
\begin{array}{lll}
x & \equiv & \alpha_0\alpha_3\beta_1+\alpha_0\beta_3+\beta_0\Mod{(\alpha_0(\alpha_1,s)(\alpha_3,\frac{s}{(\alpha_1,s)}),s)} \\
x & \equiv & \alpha_0\beta_3+\beta_0\Mod{(\alpha_0(\alpha_3,s),s)} \\
x & \equiv & \alpha_0\alpha_4\beta_2+\alpha_0\beta_4+\beta_0\Mod{(\alpha_0(\alpha_2,s)(\alpha_4,\frac{s}{(\alpha_2,s)}),s)} \\
x & \equiv & \alpha_0\beta_4+\beta_0\Mod{(\alpha_0(\alpha_4,s),s)}
\end{array}
\right)
\endgroup \\
&=\Pfrak
\begingroup
\setlength\arraycolsep{2pt}
\left(
\begin{array}{lll}
x & \equiv & \alpha_0\alpha_3\beta_1+\alpha_0\beta_3+\beta_0\Mod{(\alpha_1,s)(\alpha_0(\alpha_3,\frac{s}{(\alpha_1,s)}),\frac{s}{(\alpha_1,s)})} \\
x & \equiv & \alpha_0\beta_3+\beta_0\Mod{(\alpha_0(\alpha_3,s),s)} \\
x & \equiv & \alpha_0\alpha_4\beta_2+\alpha_0\beta_4+\beta_0\Mod{(\alpha_2,s)(\alpha_0(\alpha_4,\frac{s}{(\alpha_2,s)}),\frac{s}{(\alpha_2,s)})} \\
x & \equiv & \alpha_0\beta_4+\beta_0\Mod{(\alpha_0(\alpha_4,s),s)}
\end{array}
\right)
\endgroup
\end{align*}
Moreover, by formula (\ref{systemsWedgeEq}) and the definition of $\Ucal_i$ just after it, we have
\[
\Ucal_0=\Pfrak(\theta_{0,1})=\Pfrak(x\equiv \beta_0\Mod{(\alpha_0,s)}).
\]
It follows that
\begin{equation}\label{p0Eq}
\Pcal_0=\Pfrak
\begingroup
\setlength\arraycolsep{2pt}
\left(
\begin{array}{lll}
x & \equiv & \alpha_3\beta_1+\beta_3\Mod{(\alpha_1,s)(\alpha_3,\frac{s}{(\alpha_1,s)})} \\
x & \equiv & \beta_3\Mod{(\alpha_3,s)} \\
x & \equiv & \alpha_4\beta_2+\beta_4\Mod{(\alpha_2,s)(\alpha_4,\frac{s}{(\alpha_2,s)})} \\
x & \equiv & \beta_4\Mod{(\alpha_4,s)} \\
x & \equiv & \alpha_0\alpha_3\beta_1+\alpha_0\beta_3+\beta_0\Mod{(\alpha_1,s)(\alpha_0(\alpha_3,\frac{s}{(\alpha_1,s)}),\frac{s}{(\alpha_1,s)})} \\
x & \equiv & \alpha_0\beta_3+\beta_0\Mod{(\alpha_0(\alpha_3,s),s)} \\
x & \equiv & \alpha_0\alpha_4\beta_2+\alpha_0\beta_4+\beta_0\Mod{(\alpha_2,s)(\alpha_0(\alpha_4,\frac{s}{(\alpha_2,s)}),\frac{s}{(\alpha_2,s)})} \\
x & \equiv & \alpha_0\beta_4+\beta_0\Mod{(\alpha_0(\alpha_4,s),s)} \\
x & \equiv & \beta_0\Mod{(\alpha_0,s)}
\end{array}
\right)
\endgroup.
\end{equation}

Now we turn to the determination of the rooted trees above vertices in any given block $B$ of $\Pcal_0$. More specifically, we have $B=\Bcal(\Pcal_0,\vec{\nu}^{(\Pcal_0)})$ where $\vec{\nu}^{(\Pcal_0)}=(\nu_1,\ldots,\nu_9)\in\{\emptyset,\neg\}^9$ is a tuple of logical signs for the nine spanning congruences of $\Pcal_0$. It is helpful to split $\vec{\nu}^{(\Pcal_0)}$ into segments; namely, in the notation of Proposition \ref{periodicCosetsProp}, we write $\vec{\nu}^{(\Pcal_0)}=\vec{o'_0}\diamond\vec{o'_1}\diamond\vec{\xi}$ where
\begin{itemize}
\item $\vec{o'_0}=(\nu_1,\nu_2,\nu_3,\nu_4)$ controls in which block $\Bcal(\Rcal_0,\vec{o'_0})$ of $\Rcal_0$ the block $\Bcal(\Pcal_0,\vec{\nu}^{(\Pcal_0)})$ of $\Pcal_0$ is contained. By Proposition \ref{periodicCosetsTransientProp}, knowing the logical signs in $\vec{o'_0}$ is enough to understand, uniformly for all $x\in\Bcal(\Pcal_0,\vec{\nu}^{(\Pcal_0)})$, the contribution $\Tree_{\Gamma_f}(x,C_3\cup C_4)=\Tree_0(\Rcal_0,C_3\cup C_4,\vec{o'_0})$ to $\Tree_{\Gamma_f}(x)$ that comes from those pre-images of $x$ that lie in $C_3\cup C_4$ (the union of all transient cosets that map to $C_0$).
\item $\vec{o'_1}=(\nu_5,\nu_6,\nu_7,\nu_8)$ controls in which block $\Bcal(\Scal_{0,1},\vec{o'_1})$ of $\Scal_{0,1}=\lambda(\Rcal_0,A_0)$ the $\Pcal_0$-block $\Bcal(\Pcal_0,\vec{\nu}^{(\Pcal_0)})$ is contained. By Proposition \ref{periodicCosetsPeriodicProp}, knowing the logical signs in $\vec{o'_1}$ is enough to understand, uniformly for all $x\in\Bcal(\Pcal_0,\vec{\nu}^{(\Pcal_0)})$ \emph{of $\hfrak$-value $1$}, i.e., which are $f$-periodic (or, equivalently here, which are non-leaves in $\Gamma_{\per}$), the contribution $\Tree_{\Gamma_f}(x,C_0)=\Tree_0^{(1)}(\Scal_{0,1},C_0,\vec{o'_1})$ to $\Tree_{\Gamma_f}(x)$ that comes from those pre-images of $x$ that lie in $C_0$ (the unique periodic coset that maps to $C_0$). We note that if $x\in C_0$ has $\hfrak$-value $0$, i.e., if $x$ is $f$-transient (or, equivalently here, if $x$ is a leaf in $\Gamma_{\per}$), then $\Tree_{\Gamma_f}(x,C_0)$ is trivial, because $x$ has no $f$-transient pre-images in $C_0$.
\item $\vec{\xi}=(\nu_9)$ controls the $\hfrak$-value of the vertices in $\Bcal(\Pcal_0,\vec{\nu}^{(\Pcal_0)})$; if $\nu_9=\neg$, then all of those vertices are leaves in $\Gamma_{\per}$ (i.e., their $\hfrak$-value is $0$), otherwise they all are periodic vertices (i.e., their $\hfrak$-value is $1=H_0$).
\end{itemize}

Let us be more specific about these different contributions to $\Tree_{\Gamma_f}(x)$ for $x\in\Bcal(\Pcal_0,\vec{\nu}^{(\Pcal_0)})$. We recall that by definition,
\[
\Rcal_0=\Pcal'_3\wedge\Pcal'_4=\Pfrak'(\Pcal_3,A_3)\wedge\Pfrak'(\Pcal_4,A_4),
\]
and note that $\vec{o'_0}$ can be written as the concatenation $\vec{\nu}^{(\Pcal'_3)}\diamond\vec{\nu}^{(\Pcal'_4)}$, with $\vec{\nu}^{(\Pcal'_3)}=(\nu_1,\nu_2)$, respectively $\vec{\nu}^{(\Pcal'_4)}=(\nu_3,\nu_4)$, controlling the containment of $\Bcal(\Pcal_0,\vec{\nu}^{(\Pcal_0)})$ in $\Pcal'_3$-blocks, respectively in $\Pcal'_4$-blocks. For $i\in\{3,4\}$, knowing the logical signs in $\vec{\nu}^{(\Pcal'_i)}$ is enough to understand, uniformly for all $x\in\Bcal(\Pcal_0,\vec{\nu}^{(\Pcal_0)})$, the contribution $\Tree_{\Gamma_f}(x,C_i)=\Tree_0(\Pcal'_i,C_i,\vec{\nu}^{(\Pcal'_i)})$ to $\Tree_{\Gamma_f}(x)$ that comes from those pre-images of $x$ that lie in $C_i$.

Of course, for each given $x\in C_0$, we have
\[
\Tree_{\Gamma_f}(x)=\Tree_{\Gamma_f}(x,C_0\cup C_3\cup C_4)=\sum_{i\in\{0,3,4\}}{\Tree_{\Gamma_f}(x,C_i)}.
\]
In view of what was said above about these three different contributions to $\Tree_{\Gamma_f}(x)$, we have the following formulas (which can also be derived from Propositions \ref{periodicCosetsTransientProp}(2) and \ref{periodicCosetsProp} as well as the last formula in Proposition \ref{periodicCosetsPeriodicProp}):
\begin{align}\label{tree0Eq}
\notag &\Tree_0(\Pcal_0,\vec{\nu}^{(\Pcal_0)})= \\
&\begin{cases}
\sum_{i=3}^4{\Tree_0(\Pcal'_i,C_i,\vec{\nu}^{(\Pcal'_i)})}, & \text{if }\nu_9=\neg, \\
\sum_{i=3}^4{\Tree_0(\Pcal'_i,C_i,\vec{\nu}^{(\Pcal'_i)})}+\Tree_0^{(1)}(\Scal_{0,1},C_0,\vec{o'_1}), & \text{if }\nu_9=\emptyset.
\end{cases}
\end{align}
In particular, the logical signs $\nu_5,\nu_6,\nu_7,\nu_8$ in $\vec{o'_1}$ are irrelevant for the value of $\Tree_0(\Pcal_0,\vec{\nu}^{(\Pcal_0)})$ if $\nu_9=\neg$.

Formula (\ref{tree0Eq}) allows us to split the task of determining $\Tree_0(\Pcal_0,\vec{\nu}^{(\Pcal_0)})$ into subtasks. First, we determine $\Tree_0(\Pcal'_i,C_i,\vec{\nu}^{(\Pcal'_i)})$ for $i=3,4$, which can be done uniformly. We note that
\[
\Pcal'_i=
\Pfrak
\begingroup
\setlength\arraycolsep{2pt}
\left(
\begin{array}{lll}
x & \equiv & \alpha_i\beta_{i-2}+\beta_i\Mod{(\alpha_i(\alpha_{i-2},s),s)} \\
x & \equiv & \beta_i\Mod{(\alpha_i,s)} \\
\end{array}
\right)
\endgroup.
\]
The two entries of $\vec{\nu}^{(\Pcal'_i)}=(\nu_{2i-5},\nu_{2i-4})$ are logical signs for those two congruences, and we need to distinguish cases according to their truth values. We could just work out $\Tree_0(\Pcal'_i,C_i,\vec{\nu}^{(\Pcal'_i)})$ in each case \enquote{mechanically} following Proposition \ref{periodicCosetsTransientProp}(1), using the formula for $\sigma_{\Pcal_i,A_i}(\vec{\nu}^{(\Pcal_i)},\vec{\nu}^{(\Pcal'_i)})$ from Lemma \ref{masterLem}. However, it is more instructive to derive them with direct arguments (inclined readers may still follow the formulaic approach themselves and compare).
\begin{itemize}
\item If $\nu_{2i-4}=\neg$, i.e., if $x\not\equiv \beta_i\Mod{(\alpha_i,s)}$, then $x$ simply has no pre-images in $C_i$, whence $\Tree_0(\Pcal'_i,C_i,\vec{\nu}^{(\Pcal'_i)})$ is a single vertex without arcs (the value of $\nu_{2i-5}$ is irrelevant here).
\item If $\nu_{2i-4}=\emptyset$, i.e., if $x\equiv \beta_i\Mod{(\alpha_i,s)}$, then $x$ has exactly $(\alpha_i,s)$ children in $C_i$ (which form a coset of the kernel of $z\mapsto \alpha_iz$). We need to determine the distribution of those children over the two blocks of $\Pcal_i=\Pfrak(x\equiv \beta_{i-2}\Mod{(\alpha_{i-2},s)})$, and that distribution is controlled by the truth value $\nu_{2i-5}$ of
\begin{equation}\label{theOtherEq}
x\equiv \alpha_i\beta_{i-2}+\beta_i\Mod{(\alpha_i(\alpha_{i-2},s),s)}.
\end{equation}
Indeed, following the proof of Lemma \ref{masterLem}, the pre-images $y$ of $x$ in $C_i$ that satisfy $y\equiv \beta_{i-2}\Mod{(\alpha_{i-2},s)}$ (we note that they are exactly those pre-images of $x$ in $C_i$ which are \emph{not} leaves in $\Gamma_f$) are characterized by the system of congruences
\begin{align*}
y &\equiv \beta_{i-2}\Mod{(\alpha_{i-2},s)} \\
\alpha_iy+\beta_i&\equiv x\Mod{s},
\end{align*}
which is (according to the proof of Lemma \ref{masterLem}) consistent if and only if congruence (\ref{theOtherEq}) holds, in which case the system is equivalent to a single congruence modulo
\[
\lcm\left((\alpha_{i-2},s),\frac{s}{(\alpha_i,s)}\right).
\]
Hence, if $\nu_{2i-5}=\emptyset$, i.e., if congruence (\ref{theOtherEq}) holds, then $x$ has exactly
\[
\frac{s}{\lcm\left((\alpha_{i-2},s),\frac{s}{(\alpha_i,s)}\right)}=\left(\frac{s}{(\alpha_{i-2},s)},(\alpha_i,s)\right)
\]
pre-images $y\in C_i$ with $y\equiv \beta_{i-2}\Mod{(\alpha_{i-2},s)}$, which are exactly those pre-images that lie in $\Bcal(\Pcal_i,(\emptyset))$. Otherwise, all pre-images of $x$ in $C_i$ are incongruent to $\beta_{i-2}$ modulo $(\alpha_{i-2},s)$ and thus lie in $\Bcal(\Pcal_i,(\neg))$. In view of the known value of $\Tree_i(\Pcal_i,(\nu))$ in terms of $\nu\in\{\emptyset,\neg\}$, we find that $\Tree_0(\Pcal'_i,C_i,\vec{\nu}^{(\Pcal'_i)})$ is the expanded version of the (not necessarily simplified) edge-weighted directed rooted tree specified in Table \ref{transientTreesTable}.
\end{itemize}

\begin{longtable}[h]{|c|c|}\hline
block of $\Pcal'_i=\Pfrak'(\Pcal_i,A_i)$ & associated $\Tree_{\Gamma_f}(x,C_i)$ \\ \hline
\thead{$\nu(x\equiv \alpha_i\beta_{i-2}+\beta_i\Mod{(\alpha_i(\alpha_{i-2},s),s)})$ \\
$x\not\equiv \beta_i\Mod{(\alpha_i,s)}$ \\
$\nu\in\{\emptyset,\neg\}$}
&
\begin{tikzpicture}
\node (r) at (0,0) {};
\draw (r) circle [radius=1pt];
\end{tikzpicture}
\\ \hline
\thead{$x\not\equiv \alpha_i\beta_{i-2}+\beta_i\Mod{(\alpha_i(\alpha_{i-2},s),s)}$ \\
$x\equiv \beta_i\Mod{(\alpha_i,s)}$ \\ \\}
&
\begin{tikzpicture}
\node (r) at (0,0) {};
\draw (r) circle [radius=1pt];
\node (c) at (0,1) {};
\draw (c) circle [radius=1pt];
\path
(c) edge[->] node[right] {$(\alpha_i,s)$} (r);
\end{tikzpicture}
\\ \hline
\thead{$x\equiv \alpha_i\beta_{i-2}+\beta_i\Mod{(\alpha_i(\alpha_{i-2},s),s)}$ \\
$x\equiv \beta_i\Mod{(\alpha_i,s)}$ \\ \\ \\}
&
\begin{tikzpicture}
\node (r) at (0,0) {};
\draw (r) circle [radius=1pt];
\node (c1) at (-1,1) {};
\draw (c1) circle [radius=1pt];
\node (c2) at (1,1) {};
\draw (c2) circle [radius=1pt];
\node (v) at (1,2) {};
\draw (v) circle [radius=1pt];
\path
(c1) edge[->] node[left] {$(\alpha_i,s)-(\frac{s}{(\alpha_{i-2},s)},(\alpha_i,s))$} (r)
(c2) edge[->] node[right] {$(\frac{s}{(\alpha_{i-2},s)},(\alpha_i,s))$} (r)
(v) edge[->] node[right] {$(\alpha_{i-2},s)$} (c2);
\end{tikzpicture}
\\ \hline
\caption{Rooted trees using only pre-images in the transient pre-image coset $C_i$ with $i\in\{3,4\}$.}
\label{transientTreesTable}
\end{longtable}
This settles the first two summands of $\Tree_0(\Pcal_0,\vec{\nu}^{(\Pcal_0)})$ in each of the two cases in formula (\ref{tree0Eq}). If $\nu_9=\neg$ (i.e., if the $\Pcal_0$-block in question consists of $f$-transient points), then these are all the summands in the formula, and one can obtain (an edge-weighted version of) $\Tree_0(\Pcal_0,\vec{\nu}^{(\Pcal_0)})$ simply by adding (the edge-weighted versions of) $\Tree_i(\Pcal'_i,\vec{\nu}^{(\Pcal'_i)})$ for $i\in\{3,4\}$, read off from Table \ref{transientTreesTable}. For example, if $\nu_j=\emptyset$ for $j=1,2,\ldots,8$ but $\nu_9=\neg$, then (an edge-weighted version of) $\Tree_0(\Pcal_0,\vec{\nu}^{(\Pcal_0)})$ is as follows, setting $\wfrak(i):=(\frac{s}{(\alpha_{i-2},s)},(\alpha_i,s))$ for $i=3,4$:
\begin{center}
\begin{tikzpicture}
\node (r1) at (-6,0) {};
\draw (r1) circle [radius=1pt];
\node (c11) at (-7,1) {};
\draw (c11) circle [radius=1pt];
\node (c12) at (-5,1) {};
\draw (c12) circle [radius=1pt];
\node (v1) at (-5,2) {};
\draw (v1) circle [radius=1pt];
\path
(c11) edge[->] node[left] {$(\alpha_3,s)-\wfrak(3)$} (r1)
(c12) edge[->] node[right] {$\wfrak(3)$} (r1)
(v1) edge[->] node[right] {$(\alpha_1,s)$} (c12);
\node at (-3,1) {$+$};
\node (r2) at (1,0) {};
\draw (r2) circle [radius=1pt];
\node (c21) at (0,1) {};
\draw (c21) circle [radius=1pt];
\node (c22) at (2,1) {};
\draw (c22) circle [radius=1pt];
\node (v2) at (2,2) {};
\draw (v2) circle [radius=1pt];
\path
(c21) edge[->] node[left] {$(\alpha_4,s)-\wfrak(4)$} (r2)
(c22) edge[->] node[right] {$\wfrak(4)$} (r2)
(v2) edge[->] node[right] {$(\alpha_2,s)$} (c22);
\node at (4,1) {$\sim$};
\node (r3) at (-2,-3) {};
\draw (r3) circle [radius=1pt];
\node (c31) at (-5,-2) {};
\draw (c31) circle [radius=1pt];
\node (c32) at (-2,-2) {};
\draw (c32) circle [radius=1pt];
\node (c33) at (1,-2) {};
\draw (c33) circle [radius=1pt];
\node (v31) at (-2,-1) {};
\draw (v31) circle [radius=1pt];
\node (v32) at (1,-1) {};
\draw (v32) circle [radius=1pt];
\path
(c31) edge[->] node[left] {$(\alpha_3,s)+(\alpha_4,s)-\wfrak(3)-\wfrak(4)$} (r3)
(c32) edge[->] node[right] {$\wfrak(3)$} (r3)
(c33) edge[->] node[below] {$\wfrak(4)$} (r3)
(v31) edge[->] node[right] {$(\alpha_1,s)$} (c32)
(v32) edge[->] node[right] {$(\alpha_2,s)$} (c33);
\end{tikzpicture}
\end{center}

On the other hand, if $\nu_9=\emptyset$ (so that all vertices in $\Bcal(\Pcal_0,\vec{\nu}^{(\Pcal_0)})$ are $f$-periodic), then we also need to compute the third summand in formula (\ref{tree0Eq}), $\Tree_0^{(1)}(\Scal_{0,1},C_0,\vec{o'_1})$, which expresses the contribution coming from $f$-transient pre-images in $C_0$. Following Proposition \ref{periodicCosetsPeriodicProp}, this can be done by studying the distribution of pre-images of any given point $x\in\Bcal(\Pcal_0,\vec{\nu}^{(\Pcal_0)})$ over certain blocks of the partition $\Qcal_{0,0}=\Rcal_0\wedge\Pfrak(x\equiv \beta_0\Mod{(\alpha_0,s)})$. More specifically, we note that each $f$-transient pre-image of $x$ is contained in a block of $\Qcal_{0,0}$ of the form $\Bcal(\Qcal_{0,0},\vec{o_0}\diamond(\neg))$ for some $\vec{o_0}\in\{\emptyset,\neg\}^4$ (and we also observe that each block of $\Qcal_{0,0}$ of this form consists entirely of $f$-transient points, due to the last logical sign being $\neg$). Being able to count the number of pre-images of $x$ in each such block of $\Qcal_{0,0}$ is enough to understand $\Tree_0^{(1)}(\Scal_{0,1},C_0,\vec{o'_1})$, because $\Bcal(\Qcal_{0,0},\vec{o_0}\diamond(\neg))\subseteq\Bcal(\Rcal_0,\vec{o_0})$ and we already understand the rooted trees above $f$-transient vertices in a given block $\Bcal(\Rcal_0,\vec{o'_0})=\Bcal(\Rcal_0,\vec{\nu}^{(\Pcal'_3)}\diamond\vec{\nu}^{(\Pcal'_4)})$ of $\Rcal_0$.

Now, let us observe that the distribution of the pre-images of any $f$-periodic point $x\in C_0$ over the blocks of $\Qcal_{0,0}$ is controlled by the values of $\nu_j$ for $j\in\{5,6,7,8\}$, i.e., by the block $\Bcal(\Scal_{0,1},\vec{o'_1})$ of $\Scal_{0,1}$ in which $x$ is contained. This is because $x\in\Bcal(\Tcal_{0,1},\vec{o'_1}\diamond(\emptyset))$ where $\Tcal_{0,1}=\Scal_{0,1}\wedge\Ucal_0=\Pfrak'(\Qcal_{0,0},A_0)$, a partition which does indeed control the distribution of pre-images of $x$ over the blocks of $\Qcal_{0,0}$ according to Lemma \ref{masterLem}. Applying this lemma here leads to the formula for $\Tree_0^{(1)}(\Scal_{0,1},C_0,\vec{o'_1})$ from Proposition \ref{periodicCosetsPeriodicProp}.

For example, the logical sign tuple $\vec{o_0}=(\neg,\neg,\emptyset,\emptyset)$ corresponds to the block $B:=\Bcal(\Rcal_0,\vec{o_0})$ of $\Rcal_0$. If we wish to count how many transient pre-images a vertex $x\in C_i$ with $x\equiv \alpha_0\Mod{(\alpha_0,s)}$ stemming from, say, the block $B':=\Bcal(\Scal_{0,1},\vec{o'_1})$ of $\Scal_{0,1}$ with $\vec{o'_1}=(\emptyset,\neg,\emptyset,\emptyset)$ has, 
then we need to compute
\[
\sigma_{\Qcal_{0,0},A_0}(\vec{o_0}\diamond(\neg),\vec{o'_1}\diamond(\emptyset)),
\]
which we do now to illustrate the method. To avoid confusion among readers, we note that the above expression does not perfectly match the notation used in Lemma \ref{masterLem}. Indeed, here we use spanning congruence sequences of length $5$ both for $\Qcal_{0,0}=\Rcal_0\wedge\Tcal_{0,0}=\Rcal_0\wedge\Ucal_0$ and for $\Tcal_{0,1}=\Pfrak'(\Qcal_{0,0},A_0)$. However, in Lemma \ref{masterLem}, it is assumed that we use the \enquote{standard format} of the spanning congruence sequence for $\Pfrak'(\Qcal_{0,0},A_0)$, which contains one congruence more than the sequence for $\Qcal_{0,0}$. This discrepancy occurs because we write $\Tcal_{0,1}$ as $\lambda(\Rcal_0,A_0)\wedge\Ucal_0$ -- in the \enquote{standard format}, it would instead be
\begin{align*}
&\lambda(\Qcal_{0,0},A_0)\wedge\Pfrak(x\equiv \beta_0\Mod{(\alpha_0,s)})= \\
&\lambda(\Rcal_0,A_0)\wedge\Pfrak(x\equiv \beta_0(1+\alpha_0)\Mod{(\alpha_0^2,s)})\wedge\Ucal_0,
\end{align*}
but we can omit the congruence $x\equiv \beta_0(1+\alpha_0)\Mod{(\alpha_0^2,s)}$, which is $\theta_{0,2}(x)$ in the notation of Subsection \ref{subsec3P3}, because (using that $H_0=1$) it is equivalent to $x\equiv \beta_0\Mod{(\alpha_0,s)}$, the unique spanning congruence $\theta_{0,1}(x)$ of $\Ucal_0$. In order to apply Lemma \ref{masterLem}, we put $\Tcal_{0,1}$ into the less concise standard format, which requires us to replace the logical sign sequence $\vec{o'_1}\diamond(\emptyset)=(\emptyset,\neg,\emptyset,\emptyset,\emptyset)$ for the block of $\Tcal_{0,1}$ by $\vec{\nu'}:=\vec{o'_1}\diamond(\emptyset,\emptyset)=(\emptyset,\neg,\emptyset,\emptyset,\emptyset,\emptyset)$ (i.e., we double the $\emptyset$ at the end), the $j$-th entry of which we denote by $\nu'_j$. The logical sign sequence for the block of $\Qcal_{0,0}$ remains $\vec{\nu}:=\vec{o_0}\diamond(\neg)=(\neg,\neg,\emptyset,\emptyset,\neg)$. Our goal now is to compute
\[
\sigma_{\Qcal_{0,0},A_0}(\vec{\nu},\vec{\nu'})
\]
strictly following Lemma \ref{masterLem}. For $j=1,\ldots,5$, we denote by $\overline{\afrak}_{j}$\phantomsection\label{not185}, respectively $\overline{\bfrak}_{j}$\phantomsection\label{not186}, the modulus, respectively right-hand side, of the $j$-th spanning congruence of $\Qcal_{0,0}=\Rcal_0\wedge\Ucal_0$. That is, for $j\in\{1,2,3,4\}$, the congruence $x\equiv\overline{\bfrak}_j\Mod{\overline{\afrak}_j}$ is the $j$-th displayed congruence in the formula for $\Pcal_0$, (\ref{p0Eq}). Moreover, $\overline{\afrak}_5=(\alpha_0,s)$ and $\overline{\bfrak}_5=\beta_0$. Using the notation from Lemma \ref{masterLem}, we observe that
\begin{itemize}
\item $J_-(\vec{\nu})=\{1,2,5\}$ (the set of indices $j\in\{1,\ldots,5\}$ such that the $j$-th entry of $\vec{o_0}\diamond(\neg)$ is $\neg$);
\item $J_+(\vec{\nu})=\{1,2,3,4,5\}\setminus J_-(\vec{\nu})=\{3,4\}$;
\item $J_-(\vec{\nu'})=\{2\}$;
\item for $J\subseteq J_-(\vec{\nu})=\{1,2,5\}$, the condition $E(\vec{\nu},J)$ demands: \enquote{For all $j_1,j_2\in\{3,4\}\cup J$: $\gcd(\overline{\afrak}_{j_1},\overline{\afrak}_{j_2})\mid \overline{\bfrak}_{j_1}-\overline{\bfrak}_{j_2}$}.
\end{itemize}
According to Lemma \ref{masterLem}, we have
\[
\sigma_{\Qcal_{0,0},A_0}(\vec{\nu},\vec{\nu'})=\sum_{J\subseteq J_-(\vec{\nu})}{(-1)^{|J|}\kappa_{\Qcal_{0,0},A_0}(\vec{\nu},\vec{\nu'},J)}
\]
where
\[
\kappa_{\Qcal_{0,0},A_0}(\vec{\nu},\vec{\nu'},J)=\delta_{\nu'_6=\emptyset}\cdot\delta_{E(\vec{\nu},J)}\cdot\delta_{(J_+(\vec{\nu})\cup J)\cap J_-(\vec{\nu'})=\emptyset}\cdot\frac{s}{\lcm(\frac{s}{\gcd(\alpha_0,s)},\overline{\afrak}_{0,j}: j\in J_+(\vec{\nu})\cup J)}.
\]
In this formula, the first Kronecker delta checks whether the last entry of $\vec{\nu'}$ is $\emptyset$, which is the case. The third Kronecker delta is $1$ if and only if $2\notin J$, which leaves the four possibilities $\emptyset,\{1\},\{5\},\{1,5\}$ for $J\subseteq J_-(\vec{\nu})=\{1,2,5\}$ for which $\kappa_{\Qcal_{0,0},A_0}(\vec{\nu},\vec{\nu'},J)$ is potentially nonzero. We conclude that
\begin{align*}
&\sigma_{\Qcal_{0,0},A_0}(\vec{\nu},\vec{\nu'})= \\
&\delta_{E(\vec{\nu},\emptyset)}\frac{s}{\lcm(\frac{s}{(\alpha_0,s)},\overline{\afrak}_3,\overline{\afrak}_4)}-\delta_{E(\vec{\nu},\{1\})}\frac{s}{\lcm(\frac{s}{(\alpha_0,s)},\overline{\afrak}_1,\overline{\afrak}_3,\overline{\afrak}_4)} \\
&-\delta_{E(\vec{\nu},\{5\})}\frac{s}{\lcm(\frac{s}{(\alpha_0,s)},\overline{\afrak}_3,\overline{\afrak}_4,\overline{\afrak}_5)}+\delta_{E(\vec{\nu},\{1,5\})}\frac{s}{\lcm(\frac{s}{(\alpha_0,s)},\overline{\afrak}_1,\overline{\afrak}_3,\overline{\afrak}_4,\overline{\afrak}_5)}
\end{align*}
is the number of $f$-transient children in $B=\Bcal(\Rcal_0,\vec{o_0})$ of each given $x\in B'=\Bcal(\Scal_{0,1},\vec{o'_1})$. Each of these children provides a copy of
\[
\Tree_0^{(0)}(\Pcal_{0,0},\vec{o_0})=\Tree_0(\Rcal_0,C_3\cup C_4,\vec{o_0})
\]
that is attached to the root of $\Tree_0^{(1)}(\Scal_{0,1},C_0,\vec{o'_1})$. If we carry this computation out for fixed $\vec{o'_1}$ and all possible values of $\vec{o_0}\in\{\emptyset,\neg\}^4$, then we obtain a \enquote{complete picture} of $\Tree_0^{(1)}(\Scal_{0,1},C_0,\vec{o'_1})$.

It is time to particularize the gained explicit understanding of the rooted trees back to the concrete example we started from. First, we deal with the rooted trees $\Tree_{\Gamma_f}(x,C_3\cup C_4)$ in terms of the blocks of $\Rcal_0$. By substituting $s=51,\alpha_0=9,\alpha_1=3,\alpha_2=17,\alpha_3=34,\alpha_4=9,\beta_0=1,\beta_1=0,\beta_2=6,\beta_3=21,\beta_4=8$ into formula (\ref{r0Eq}), we get
\[
\Rcal_0=
\Pfrak
\begingroup
\setlength\arraycolsep{2pt}
\left(
\begin{array}{lll}
x & \equiv & 21\Mod{51} \\
x & \equiv & 4\Mod{17} \\
x & \equiv & 11\Mod{51} \\
x & \equiv & 2\Mod{3}
\end{array}
\right)
\endgroup.
\]
There are dependencies between these congruences. For example, the first implies the second as well as the negations of the third and fourth. Table \ref{transientTreesConcreteTable} lists all $\vec{o'_0}\in\{\emptyset,\neg\}^4$ such that $\Bcal(\Rcal_0,\vec{o'_0})$ is nonempty, along with a description of the set $\Bcal(\Rcal_0,\vec{o'_0})$ and the (simplified edge-weighted form of the) corresponding rooted tree $\Tree_0(\Rcal_0,C_3\cup C_4,\vec{o'_0})$ above each point in $\Bcal(\Rcal_0,\vec{o'_0})$, obtained by adding the (edge-weighted forms of the) rooted trees $\Tree_0(\Pcal'_i,C_i,\vec{\nu'_i})$ for $i=3,4$ read off from Table \ref{transientTreesTable}.
\begin{longtable}[h]{|c|c|c|}\hline
$\vec{o'_0}$ & $B:=\Bcal(\Rcal_0,\vec{o'_0})$ & $\Tree_{\Gamma_f}(x,C_3\cup C_4)$ for $x\in B$ \\ \hline
$\begin{pmatrix}\neg \\ \neg \\ \neg \\ \neg\end{pmatrix}$
&
$\{x\in\IZ/51\IZ: x\not\equiv2\Mod{3}\}\setminus\{4,21\}$
&
\begin{tikzpicture}
\node (r) at (0,0) {};
\draw (r) circle [radius=1pt];
\end{tikzpicture}
\\ \hline
$\begin{pmatrix}\neg \\ \emptyset \\ \neg \\ \neg\end{pmatrix}$
&
$\{4\}$
&
\begin{tikzpicture}
\node (r) at (0,0) {};
\draw (r) circle [radius=1pt];
\node (c) at (0,1) {};
\draw (c) circle [radius=1pt];
\path
(c) edge[->] node[right] {$17$} (r);
\end{tikzpicture}
\\ \hline
$\begin{pmatrix}\emptyset \\ \emptyset \\ \neg \\ \neg\end{pmatrix}$
&
$\{21\}$
&
\begin{tikzpicture}
\node (r) at (0,0) {};
\draw (r) circle [radius=1pt];
\node (c) at (0,1) {};
\draw (c) circle [radius=1pt];
\node (v) at (0,2) {};
\draw (v) circle [radius=1pt];
\path
(c) edge[->] node[right] {$17$} (r)
(v) edge[->] node[right] {$3$} (c);
\end{tikzpicture}
\\ \hline
$\begin{pmatrix}\neg \\ \neg \\ \neg \\ \emptyset\end{pmatrix}$
&
$\{x\in\IZ/51\IZ: x\equiv2\Mod{3}\}\setminus\{11,38\}$
&
\begin{tikzpicture}
\node (r) at (0,0) {};
\draw (r) circle [radius=1pt];
\node (c) at (0,1) {};
\draw (c) circle [radius=1pt];
\path
(c) edge[->] node[right] {$3$} (r);
\end{tikzpicture}
\\ \hline
$\begin{pmatrix}\neg \\ \neg \\ \emptyset \\ \emptyset\end{pmatrix}$
&
$\{11\}$
&
\begin{tikzpicture}
\node (r) at (0,0) {};
\draw (r) circle [radius=1pt];
\node (c) at (0,1) {};
\draw (c) circle [radius=1pt];
\node (v) at (0,2) {};
\draw (v) circle [radius=1pt];
\path
(c) edge[->] node[right] {$3$} (r)
(v) edge[->] node[right] {$17$} (c);
\end{tikzpicture}
\\ \hline
$\begin{pmatrix}\neg \\ \emptyset \\ \neg \\ \emptyset\end{pmatrix}$
&
$\{38\}$
&
\begin{tikzpicture}
\node (r) at (0,0) {};
\draw (r) circle [radius=1pt];
\node (c) at (0,1) {};
\draw (c) circle [radius=1pt];
\path
(c) edge[->] node[right] {$20$} (r);
\end{tikzpicture}
\\ \hline
\caption{Rooted trees using only pre-images in transient pre-image cosets.}
\label{transientTreesConcreteTable}
\end{longtable}

Now we turn to the description of $\Tree_{\Gamma_f}(x,C_0)$ for $f$-periodic $x\in C_0$ in terms of the $\Scal_{0,1}$-block $\Bcal(\Scal_{0,1},\vec{o'_1})$ in which $x$ lies. First, we substitute our concrete values of $s$ and the $\alpha_i$ and $\beta_i$ into the four spanning congruences for $\Scal_{0,1}$ to get that
\[
\Scal_{0,1}=
\Pfrak
\begingroup
\setlength\arraycolsep{2pt}
\left(
\begin{array}{lll}
x & \equiv & 37\Mod{51} \\
x & \equiv & 37\Mod{51} \\
x & \equiv & 49\Mod{51} \\
x & \equiv & 1\Mod{3}
\end{array}
\right)
\endgroup.
\]
It is not necessary to strictly follow the computations described in Lemma \ref{masterLem} (as outlined above) to gain a complete understanding of the trees here -- we give a conceptual argument instead.

We note that by Lemma \ref{periodicCharLem}, the points in $\IZ/51\IZ$ that are periodic under $A_0:x\mapsto 9x+1$ are just those that are congruent to $1$ modulo $3$ (the unique fixed point of $A_0$ modulo $3=\gcd(\alpha_0^L,s)$ where $L$ is as in Lemma \ref{periodicCharLem}). Hence, the last spanning congruence of $\Scal_{0,1}$ is always true for those $x$ we are considering. We observe that it is just a coincidence that the spanning sequence for $\Scal_{0,1}$ contains the characterizing congruence for periodic vertices -- in general, this information needs to be added \enquote{externally} if one wants to control it via the partition blocks, using the partition $\Tcal_{0,1}=\Scal_{0,1}\wedge\Ucal_0$ instead.

Since the children of $x$ in $C_0$ form a coset of the kernel $17\IZ/51\IZ$ of $z\mapsto 9z$, it follows that $x$ has precisely three pre-images in $C_0$, one in each congruence class modulo $3$. But the pre-image $y$ of $x$ in $C_0$ with $y\equiv1\Mod{3}$ is $f$-periodic and hence does not occur in $\Tree_{\Gamma_f}(x,C_0)$. We also note the following.
\begin{itemize}
\item For $x=37$, the remaining two, $f$-transient pre-images are $21$ and $38$.
\item For $x=49$ the $f$-transient pre-images are $11$ and $45$.
\item For all other $x\equiv1\Mod{3}$, there is one $f$-transient pre-image each in the two \enquote{generic} blocks $\Bcal(\Rcal_0,(\neg,\neg,\neg,\neg))$ and $\Bcal(\Rcal_0,(\neg,\neg,\neg,\emptyset))$ of $\Rcal_0$; this is because all other blocks except $\{4\}$ have already been \enquote{used up}, and $4\equiv1\Mod{3}$.
\end{itemize}
From Table \ref{transientTreesConcreteTable}, we can read off $\Tree_{\Gamma_f}(y,C_3\cup C_4)=\Tree_{\Gamma_f}(y)$ for each of the two $f$-transient pre-images $y$ of $x$ in $C_0$, thus obtaining the shape of $\Tree_{\Gamma_f}(x,C_0)$ specified in Table \ref{periodicTreesConcreteTable}.

\begin{longtable}[h]{|c|c|c|}\hline
$\vec{o'_1}$ & $B':=\Bcal(\Scal_{0,1},\vec{o'_1})$ & $\Tree_{\Gamma_f}(x,C_0)$ for periodic $x\in B'$ \\ \hline
$\begin{pmatrix}\neg \\ \neg \\ \neg \\ \neg\end{pmatrix}$
&
$\{x\in\IZ/51\IZ: x\not\equiv1\Mod{3}\}$
&
n/a (no periodic $x$ in this block)
\\ \hline
$\begin{pmatrix}\neg \\ \neg \\ \neg \\ \emptyset\end{pmatrix}$
&
$\{x\in\IZ/51\IZ: x\equiv1\Mod{3}\}\setminus\{37,49\}$
&
\begin{tikzpicture}
\node (r) at (0,0) {};
\draw (r) circle [radius=1pt];
\node (c1) at (-1,1) {};
\draw (c1) circle [radius=1pt];
\node (c2) at (1,1) {};
\draw (c2) circle [radius=1pt];
\node (v) at (1,2) {};
\draw (v) circle [radius=1pt];
\path
(c1) edge[->] node[left] {$1$} (r)
(c2) edge[->] node[right] {$1$} (r)
(v) edge[->] node[right] {$3$} (c2);
\end{tikzpicture}
\\ \hline
$\begin{pmatrix}\emptyset \\ \emptyset \\ \neg \\ \emptyset\end{pmatrix}$
&
$\{37\}$
&
\begin{tikzpicture}
\node (r) at (0,0) {};
\draw (r) circle [radius=1pt];
\node (c1) at (-1,1) {};
\draw (c1) circle [radius=1pt];
\node (c2) at (1,1) {};
\draw (c2) circle [radius=1pt];
\node (v1) at (-1,2) {};
\draw (v1) circle [radius=1pt];
\node (v2) at (1,2) {};
\draw (v2) circle [radius=1pt];
\node (w) at (-1,3) {};
\draw (w) circle [radius=1pt];
\path
(c1) edge[->] node[left] {$1$} (r)
(c2) edge[->] node[right] {$1$} (r)
(v1) edge[->] node[left] {$17$} (c1)
(v2) edge[->] node[right] {$20$} (c2)
(w) edge [->] node[left] {$3$} (v1);
\end{tikzpicture}
\\ \hline
$\begin{pmatrix}\neg \\ \neg \\ \emptyset \\ \emptyset\end{pmatrix}$
&
$\{49\}$
&
\begin{tikzpicture}
\node (r) at (0,0) {};
\draw (r) circle [radius=1pt];
\node (c1) at (-1,1) {};
\draw (c1) circle [radius=1pt];
\node (c2) at (1,1) {};
\draw (c2) circle [radius=1pt];
\node (v) at (1,2) {};
\draw (v) circle [radius=1pt];
\node (w) at (1,3) {};
\draw (w) circle [radius=1pt];
\path
(c1) edge[->] node[left] {$1$} (r)
(c2) edge[->] node[right] {$1$} (r)
(v) edge[->] node[right] {$3$} (c2)
(w) edge [->] node[left] {$17$} (v);
\end{tikzpicture}
\\ \hline
\caption{Rooted trees using only pre-images in $C_0$.}
\label{periodicTreesConcreteTable}
\end{longtable}

Finally, our formula (\ref{tree0Eq}) leads us to a tabular list of $\Tree_{\Gamma_f}(x)$, which we specify in Table \ref{fullTreesConcreteTable}, where we also introduce the notation $\Ifrak_n$ for $n=1,2,3,4$ to denote the different isomorphism types of $\Tree_{\Gamma_f}(x)$ for periodic $x$. Additionally, we define $\Ifrak_0$ to denote the isomorphism type of the trivial rooted tree, consisting of a single vertex without arcs.

\begin{longtable}[h]{|c|c|}\hline
block $B$ of $\Pcal_0$ & $\Tree_{\Gamma_f}(x)$ for $x\in B$ \\ \hline
$\{x\in\IZ/51\IZ: x\equiv1\Mod{3}\}\setminus\{4,37,49\}$
& 
\begin{tikzpicture}
\node (r) at (0,0) {};
\draw (r) circle [radius=1pt];
\node (c1) at (-1,1) {};
\draw (c1) circle [radius=1pt];
\node (c2) at (1,1) {};
\draw (c2) circle [radius=1pt];
\node (v) at (1,2) {};
\draw (v) circle [radius=1pt];
\path
(c1) edge[->] node[left] {$1$} (r)
(c2) edge[->] node[right] {$1$} (r)
(v) edge[->] node[right] {$3$} (c2);
\node at (2,1) {$=:\Ifrak_1$};
\end{tikzpicture}
\\ \hline
$\{4\}$
&
\begin{tikzpicture}
\node (r) at (0,0) {};
\draw (r) circle [radius=1pt];
\node (c1) at (-1,1) {};
\draw (c1) circle [radius=1pt];
\node (c2) at (1,1) {};
\draw (c2) circle [radius=1pt];
\node (v) at (1,2) {};
\draw (v) circle [radius=1pt];
\path
(c1) edge[->] node[left] {$18$} (r)
(c2) edge[->] node[right] {$1$} (r)
(v) edge[->] node[right] {$3$} (c2);
\node at (2,1) {$=:\Ifrak_2$};
\end{tikzpicture}
\\ \hline
$\{37\}$
&
\begin{tikzpicture}
\node (r) at (0,0) {};
\draw (r) circle [radius=1pt];
\node (c1) at (-1,1) {};
\draw (c1) circle [radius=1pt];
\node (c2) at (1,1) {};
\draw (c2) circle [radius=1pt];
\node (v1) at (-1,2) {};
\draw (v1) circle [radius=1pt];
\node (v2) at (1,2) {};
\draw (v2) circle [radius=1pt];
\node (w) at (-1,3) {};
\draw (w) circle [radius=1pt];
\path
(c1) edge[->] node[left] {$1$} (r)
(c2) edge[->] node[right] {$1$} (r)
(v1) edge[->] node[left] {$17$} (c1)
(v2) edge[->] node[right] {$20$} (c2)
(w) edge [->] node[left] {$3$} (v1);
\node at (2,1) {$=:\Ifrak_3$};
\end{tikzpicture}
\\ \hline
$\{49\}$
&
\begin{tikzpicture}
\node (r) at (0,0) {};
\draw (r) circle [radius=1pt];
\node (c1) at (-1,1) {};
\draw (c1) circle [radius=1pt];
\node (c2) at (1,1) {};
\draw (c2) circle [radius=1pt];
\node (v) at (1,2) {};
\draw (v) circle [radius=1pt];
\node (w) at (1,3) {};
\draw (w) circle [radius=1pt];
\path
(c1) edge[->] node[left] {$1$} (r)
(c2) edge[->] node[right] {$1$} (r)
(v) edge[->] node[right] {$3$} (c2)
(w) edge [->] node[left] {$17$} (v);
\node at (2,1) {$=:\Ifrak_4$};
\end{tikzpicture}
\\ \hline
$B\subseteq\{x\in\IZ/51\IZ:x\not\equiv1\Mod{3}\}$
&
$\Tree_{\Gamma_f}(x,C_3\cup C_4)$ (see Table \ref{transientTreesConcreteTable})
\\ \hline
\caption{The rooted trees $\Tree_{\Gamma_f}(x)$ for $x\in C_0$.}
\label{fullTreesConcreteTable}
\end{longtable}

Now that we have a full understanding of the rooted trees above vertices in $\Gamma_f$, let us turn to the determination of the $f$-periodic points and to the construction of a CRL-list for $f$. Following the discussion in Subsection \ref{subsec3P1}, the periodic points of $f$ are the field element $0$ as well as all periodic points of $f$ in the unique periodic coset $C_0$, which we already identified above (using Lemma \ref{periodicCharLem}) to be precisely those $x\in\IZ/51\IZ$ with $x\equiv1\Mod{3}$, so there are $51/3=17$ periodic points in $C_0$. With such a small number of periodic points, it would be easy to just determine the cycle structure and a CRL-list by brute force, but we would still like to proceed as described in Subsection \ref{subsec3P1} (and Subsection \ref{subsec2P3}, on which Subsection \ref{subsec3P1} builds) to illustrate the method.

First, we observe that the nontrivial prime powers of the form $p^{\nu_p(51)}$ are just $3$ and $17$. By the approach from Subsection \ref{subsec2P3}, we need to determine a CRL-list of each \emph{bijective} reduction of $A_0$ modulo $p^{\nu_p(51)}$, i.e., here only for the reduction of $A_0$ modulo $17$. We can read off such a CRL-list from Table \ref{crlListPrimaryTable}. More specifically, since $\nu_{17}^{(1)}(1)=0\geq 0=\nu_{17}^{(1)}(9-1)$, and since $\ord(9)$, the multiplicative order of $9$ modulo $17$, is $8$, case 1 in that table with
\[
\rfrak:=3\text{ and }\ffrak:=-\frac{1}{17^0}\cdot\inv_{17}\left(\frac{8}{17^0}\right)=-15=2
\]
tells us that
\[
\{(3^017^0+2,8),(3^117^0+2,8),(3^017^1+2,1)\}=\{(3,8),(5,8),(2,1)\}
\]
is a CRL-list of $A_0$ modulo $17$. Modulo $3$, the only periodic point of $A_0$ is $1$, so in order to get a CRL-list for $A_0$ modulo $51$, we just map the first entries of the above CRL-list modulo $17$ under the function $\Lambda:\IZ/17\IZ\rightarrow\IZ/51\IZ$ with $\Lambda(x)\equiv x\Mod{17}$ and $\Lambda(x)\equiv1\Mod{3}$, which leads to the following CRL-list of $A_0$ modulo $51$:
\[
\{(37,8),(22,8),(19,1)\}.
\]
We can now describe the isomorphism types of the four connected components of $\Gamma_f$ as cyclic sequences (necklaces, isomorphism types of necklace graphs) of finite directed rooted trees simply by enumerating the elements on the cycles of $f$ by iteration, then looking up the associated rooted tree isomorphism types in Table \ref{fullTreesConcreteTable}. We note that this is a brute-force approach that is not viable when the number of $f$-periodic points is large and should then be replaced by the approach described in Subsection \ref{subsec3P4} instead.
\begin{itemize}
\item The connected component of the field element $0$ is a single vertex with a loop, corresponding to the following cyclic sequence of rooted tree isomorphism types: $[\Ifrak_0]$.
\item Because the cycle of $22\in\IZ/51\IZ$ under $A_0$ is $(22,46,7,13,16,43,31,25)$, the connected component of the field element $\iota_0(22)=\omega^{5\cdot 22}=\omega^{110}$ is represented by the cyclic sequence $[\Ifrak_1,\Ifrak_1,\Ifrak_1,\Ifrak_1,\Ifrak_1,\Ifrak_1,\Ifrak_1,\Ifrak_1]$.
\item Because the cycle of $37\in\IZ/51\IZ$ under $A_0$ is $(37,28,49,34,1,10,40,4)$, the connected component of the field element $\omega^{185}$ is represented by the cyclic sequence $[\Ifrak_3,\Ifrak_1,\Ifrak_4,\Ifrak_1,\Ifrak_1,\Ifrak_1,\Ifrak_1,\Ifrak_2]$.
\item Finally, because the cycle of $19$ is simply $(19)$, the connected component of the field element $\omega^{95}$ is represented by $[\Ifrak_1]$.
\end{itemize}
As a quick sanity check, we note that $|\V(\Ifrak_1)|=6$, $|\V(\Ifrak_2)|=23$, $|\V(\Ifrak_3)|=91$, and $|\V(\Ifrak_4)|=57$, so the vertex numbers of the three connected components in $\IF_{2^8}^{\ast}$ add up to $14\cdot 6+23+91+57=255$, as they should. One may also verify that these are the same cyclic sequences that were given in our introduction.

\subsection{Special case: All \texorpdfstring{$A_i$}{Ai} are permutations}\label{subsec4P3}

Let $f$ be an index $d$ generalized cyclotomic mapping of $\IF_q$, given in cyclotomic form (\ref{cyclotomicFormEq}). Let us assume that for each $i=0,1,\ldots,d-1$, we have
\[
\gcd(r_i,s)=\gcd(r_i,\frac{q-1}{d})=1.
\]
An important class of functions to which this applies are the index $d$ cyclotomic mappings of $\IF_q$ \emph{of first order} (i.e., those generalized cyclotomic mappings for which all $r_i$ are equal to $1$).

By our comments between Remark \ref{crlListRem} and Definition \ref{treeAboveDef}, the affine map  $A_i$ of $\IZ/s\IZ$, which encodes the restriction $f_{\mid C_i}$ in case $a_i\not=0$, is of the form $z\mapsto r_iz+\text{const}$. Our assumption on the $r_i$ is therefore equivalent to demanding that for each $i$ such that $a_i\not=0$ (and thus $A_i$ is well-defined), the function $A_i$ is an affine \emph{permutation} of $\IZ/s\IZ$.

Our goal is to describe the functional graph $\Gamma_f$, which turns out to be particularly easy. Let us start with the rooted trees.

\begin{lemmma}\label{allPermutationsLem}
Let $x\in\IF_q=\V(\Gamma_f)$.
\begin{enumerate}
\item If $x\not=0$, and if $i$ denotes the unique index in $\{0,1,\ldots,d-1\}$ such that $x\in C_i$, then $\Tree_{\Gamma_f}(x)$ is isomorphic to $\Ifrak_i:=\Tree_{\Gamma_{\overline{f}}}(i)$\phantomsection\label{not187}.
\item If $x=0$, then $\Tree_{\Gamma_f}(x)=\sum_{i\in\overline{f}^{-1}(\{d\})\setminus\{d\}}{s\Ifrak_i^+}$.
\end{enumerate}
\end{lemmma}

\begin{proof}
Statement (1) can be proved by induction on $h(x):=\height(\Tree_{\Gamma_{\overline{f}}}(i))$. If $h(x)=0$, then all $\overline{f}$-pre-images of $i$ (if any) are $\overline{f}$-periodic. In particular, $x$ has no $f$-transient pre-images under $f$, because each such pre-image would need to lie in a coset $C_j$ where $j$ is an $\overline{f}$-transient pre-image of $i$ under $\overline{f}$. Indeed, otherwise, $i$, having an $\overline{f}$-periodic pre-image under $\overline{f}$, is $\overline{f}$-periodic itself. By assumption, we can pick an $f$-transient pre-image $y$ of $x$ under $f$ in $C_{i'}$, where $i'$ is the unique $\overline{f}$-periodic pre-image of $i$ under $\overline{f}$. If $\ell$ denotes the cycle length of $i$ under $\overline{f}$, then $\Acal_i=A_{i_0}A_{i_1}\cdots A_{i_{\ell-1}}$ represents the restriction of $f^{\ell}$ to $C_i$. Because each $A_{i_t}$ is bijective, so is $\Acal_i$; in other words, every point in $C_i$ is periodic under $\Acal_i$ and thus under $f$ (following the discussion in Subsection \ref{subsec3P1}). In particular, $x$ is $f$-periodic, say with cycle length $l$. Therefore, $f^{l-1}(x)$ is an $f$-pre-image of $x$ in $C_{i'}$, as is $y$. Because $A_{i'}$, which represents the restriction $f_{\mid C_{i'}}:C_{i'}\rightarrow C_i$, is injective, it follows that $f^{l-1}(x)=y$, whence $y$ is $f$-periodic, contradicting our assumption. The upshot of this discussion is that if $h(x)=0$, then $\Tree_{\Gamma_f}(x)$ is trivial, as is $\Ifrak_i$.

Now we assume that $h(x)\geq1$. Let $j_1,j_2,\ldots,j_K$ be the distinct $\overline{f}$-transient pre-images of $i$ under $\overline{f}$. By the argument from the previous paragraph, each $f$-transient pre-image of $x$ under $f$ must lie in one of the cosets $C_{j_t}$ for $t=1,2,\ldots,K$, and since $A_{j_t}$ is bijective for each $t$, it follows that $x$ has precisely one (transient) pre-image $c_{j_t}\in C_{j_t}$ for each $t=1,2,\ldots,K$. Therefore, using the induction hypothesis,
\[
\Tree_{\Gamma_f}(x)=\sum_{t=1}^K{\Tree_{\Gamma_f}(c_{j_t})^+}=\sum_{t=1}^K{\Ifrak_{j_t}^+}=\sum_{t=1}^K{\Tree_{\Gamma_{\overline{f}}}(j_t)^+}=\Tree_{\Gamma_{\overline{f}}}(i).
\]

For statement (2), let $j_1,j_2,\ldots,j_K$ be the distinct $\overline{f}$-transient children of $d$ in $\Gamma^{\ast}_{\overline{f}}$. Equivalently, the $j_t$ are the distinct elements of $\overline{f}^{-1}(\{d\})\setminus\{d\}$. The $f$-transient children of $0_{\IF_q}$ in $\Gamma_f^{\ast}$ are precisely the points in $\bigcup_{t=1}^K{C_{j_t}}$. Using statement (1), it follows that
\[
\Tree_{\Gamma_f}(0_{\IF_q})=\sum_{t=1}^K\sum_{y\in C_{j_t}}{\Tree_{\Gamma_f}(y)^+}\cong\sum_{t=1}^K\sum_{y\in C_{j_t}}{\Ifrak_{j_t}^+}=\sum_{t=1}^K{s\Ifrak_{j_t}^+},
\]
as required.
\end{proof}

Because $\Tree_{\Gamma_f}(x)$ for $x\not=0$ only depends on the coset $C_i$ in which $x$ lies and can be read off directly from $\Gamma_{\overline{f}}$, we only need to know $\Gamma_{\overline{f}}$ and the cycle structure of $f$ on each coset union $U_i:=\bigcup_{t=0}^{\ell-1}{C_{i_t}}$\phantomsection\label{not188}, where $(i_0,i_1,\ldots,i_{\ell-1})$ with $i=i_0$ is the $\overline{f}$-cycle of $i$, in order to understand the isomorphism type of $\Gamma_f$. This can be achieved using analogous ideas to the ones for the determination of CRL-lists in Subsection \ref{subsec3P1}.

Let us set $\Acal_i:=A_{i_0}A_{i_1}\cdots A_{i_{\ell-1}}$. Then $\Acal_i$ is an affine permutation of $\IZ/s\IZ$, and its cycle type $\CT(\Acal_i)$ can be read off from \cite[Tables 3 and 4]{BW22b}. Moreover,
\[
\CT(f_{\mid U_i})=\BU_{\ell}(\CT(\Acal_i))
\]
where $\BU_{\ell}$\phantomsection\label{not189}, the so-called \emph{$\ell$-blow-up function}\phantomsection\label{term56}, is the unique $\IQ$-algebra endomorphism of $\IQ[x_n: n\in\IN^+]$ with $\BU_{\ell}(x_n)=x_{\ell n}$ for all $n\in\IN^+$. Say
\[
\CT(f_{\mid U_i})=x_1^{e_1}x_2^{e_2}\cdots x_{\ell s}^{e_{\ell s}}.
\]
Then, viewing isomorphism types of functional graphs as multisets of cyclic sequences (necklaces) of isomorphism types of finite directed rooted trees (with each such sequence encoding the isomorphism type of one connected component), we have the following:
\[
\Gamma_f^{(i)}:=\Gamma_{f\mid U_i}=\bigsqcup_{1\leq l\leq \ell s, \ell\mid l}\bigsqcup_{n=1}^{e_l}{\{\diamond_{m=1}^{l/\ell}{[\Ifrak_{i_0},\Ifrak_{i_1},\ldots,\Ifrak_{i_{\ell-1}}]}\}},
\]
and\phantomsection\label{not190}, if $\overline{\Lcal}$ is a CRL-list for $\overline{f}$, then
\[
\Gamma_f=\bigsqcup_{(i,\ell)\in\overline{\Lcal}}{\Gamma_f^{(i)}}\sqcup\{[\sum_{j\in\overline{f}^{-1}(\{d\})\setminus\{d\}}{s\Ifrak_j}]\}.
\]

\section{Algorithmic complexity analysis}\label{sec5}

The aim of this section is to describe algorithms for understanding important aspects of the structure of functional graphs of generalized cyclotomic mappings of finite fields in detail and analyze their complexities. In Subsection \ref{subsec5P1}, we set the ground by describing our computational model, the so-called dual model, in detail and introducing some important auxiliary concepts and results. We note that this dual model consists of carefully keeping track of three distinct parameters -- the bit operations, elementary quantum gates and conversions from bits to qubits and vice versa -- separately, which is, to the authors' knowledge, a novel approach and may be of independent, wider interest for readers working in quantum complexity analysis. Subsection \ref{subsec5P2} consists of the proof of Theorem \ref{complexitiesTheo}, which provides complexity bounds for three fundamental algorithmic problems and may be considered the main result of this section. As mentioned in the introduction, it is an open problem how to encode the overall structure of the functional graph of a generalized cyclotomic mapping compactly; in particular, these results do \emph{not} provide an efficient general algorithm for deciding whether the functional graphs of two given generalized cyclotomic mappings are isomorphic. However, in Subsection \ref{subsec5P3}, we discuss four special cases in which this isomorphism problem can be solved efficiently.

\subsection{Framework and auxiliary results}\label{subsec5P1}

Throughout this section, we assume that $f$ is an index $d$ generalized cyclotomic mapping of $\IF_q$, given in cyclotomic form (\ref{cyclotomicFormEq}), where either
\begin{itemize}
\item each $a_i$ is specified as the field element $0$ or as a power of a common, unknown primitive element $\omega$ of $\IF_q$, or
\item we explicitly know the minimal polynomial $P(T)$\phantomsection\label{not190P5}\phantomsection\label{not190P75} over the prime subfield $\IF_p$ of such an $\omega$, and the $a_i$ are represented as elements of $\IF_p[T]/(P(T))$.
\end{itemize}
The main goal in this section is to analyze the complexities of the following algorithmic problems.
\begin{itemize}
\item Problem 1: Given $f$, compute a compact parametrization of a CRL-list $\Lcal$ of $f$ (we note that $|\Lcal|$ equals the number of cycles of $f$ on its periodic points, which may be superpolynomial in $\log{q}$, so we want to avoid listing $\Lcal$ element-wise).
\item Problem 2: Given $f$, compute a partition-tree register of $f$ in the sense of Definition \ref{partTreeRegDef} below.
\item Problem 3: Given $f$, a partition-tree register of $f$, and a pair $(r,l)$ such that $r\in\IF_q$ if $f$-periodic and $l$ is the cycle length of $r$ under $f$, compute a compact description of the cyclic sequence of rooted tree isomorphism types from formula (\ref{cyclicSeqEq}) (which characterizes the digraph isomorphism type of the connected component of $\Gamma_f$ that contains $r$).
\end{itemize}

A partition-tree register of $f$ is a standardized way of storing information about the arithmetic partitions $\Pcal_i$ constructed in Subsection \ref{subsec3P3} and the rooted trees associated with their blocks. To define it, we first introduce the following auxiliary concept.

\begin{deffinition}\label{recTreeDescListDef}
A \emph{recursive tree description list}\phantomsection\label{term57} is a finite sequence $(\Dfrak_n)_{n=0,1,\ldots,N}$\phantomsection\label{not191} of sets that has an associated (unique) ordered sequence $(\Ifrak_n)_{n=0,1,\ldots,N}$ of pairwise distinct, finite rooted tree isomorphism types such that the following hold.
\begin{enumerate}
\item $\Ifrak_0$ is the trivial rooted tree isomorphism type, and $\Dfrak_0=\emptyset$.
\item For $n\geq1$, each rooted tree attached in $\Ifrak_n$ to the root of $\Ifrak_n$ is isomorphic to $\Ifrak_m$ for some $m\in\{0,1,\ldots,n-1\}$. Moreover, $\Dfrak_n$ is the set of all pairs $(m,k_m)$ where $m\in\{0,1,\ldots,n-1\}$ is an index for which $\Ifrak_m$ is attached to the root of $\Ifrak_n$ at least once, and $k_m$ is the multiplicity with which it is attached.
\end{enumerate}
\end{deffinition}

In a recursive tree description list, each set $\Dfrak_n$ can be viewed as a compact description of $\Ifrak_n$, referring to the rooted trees attached to the root of $\Ifrak_n$ with their (earlier) indices $m$, rather than their full descriptions. The idea of encoding isomorphism types of rooted trees via numbers (\enquote{tree indices}) to get more compact descriptions of larger rooted trees is not new; it appears, for example, in the decision algorithm for isomorphism of directed rooted trees described in \cite[Example 3.2 on p.~84]{AHU75a}. In contrast to that algorithm, which is linear in the number of vertices, we do not list tree indices $m$ repeatedly, but rather, we specify their multiplicities $k_m$. In situations such as ours, where entire sets (here: arithmetic partition blocks) of vertices can be dealt with simultaneously, this modification is crucial to ensure the efficiency of our algorithms relative to their smaller input length (which lies in $O(d\log{q})$). In implementations, we assume that each $\Dfrak_n$ is represented by an array (ordered list) of pairs $(m,k_m)$, sorted by increasing $m$. Moreover, $m$ and $k_m$, both of which are at most $q$, are to be represented by bit strings of length $\lfloor\log_2{q}\rfloor+1$ (please note, however, that we use other conventions for the related notion of a type-I tree register, introduced in Definition \ref{treeRegDef}(1)). We observe that with these conventions, all bit strings representing an element $(m,k_m)$ of $\Dfrak_n$ (for some $n$) have the same length, and the ordering of the elements of $\Dfrak_n$ by increasing $m$ corresponds to the lexicographic ordering of those bit string encodings.

Equipped with the concept of a recursive tree description list, we can define partition-tree registers of generalized cyclotomic mappings of finite fields as follows, using notations introduced in Subsection \ref{subsec3P3}. We note that in this algorithmic section, we frequently identify arithmetic partitions with specific spanning congruence sequences of them.

\begin{deffinition}\label{partTreeRegDef}
Let $f$ be an index $d$ generalized cyclotomic mapping of $\IF_q$. If $i\in\{0,1,\ldots,d-1\}$ is $\overline{f}$-periodic, we recall that $i_t$ for $t\in\IZ$ denotes the unique $\overline{f}$-periodic index in $\{0,1,\ldots,d-1\}$ such that $(\overline{f}_{\mid\per(\overline{f})})^t(i)=i_t$. A \emph{partition-tree register of $f$}\phantomsection\label{term58} is an ordered pair of the form
\[
((\Zcal_i)_{i=0,1,\ldots,d-1},((\Dfrak_n,(S_{n,i})_{i=0,1,\ldots,d}))_{n=0,1,\ldots,N})
\]
such\phantomsection\label{not192}\phantomsection\label{not193} that the following hold.
\begin{enumerate}
\item For each $i=0,1,\ldots,d-1$, $\Zcal_i$ is the following.
\begin{enumerate}
\item If $i$ is $\overline{f}$-transient, then $\Zcal_i=\Pcal_i$, given through a spanning congruence sequence of length $m_i\in\IN_0$.
\item If $i$ is $\overline{f}$-periodic, then $\Zcal_i$ is an $(H_i+2)$-tuple $(\Xcal_{i,h})_{h=-1,0,\ldots,H_i}$ such that
\begin{enumerate}
\item $\Xcal_{i,-1}=(\theta_{i,h}(x))_{h=1,2,\ldots,H_i}$, and
\item $\Xcal_{i,h}$ for $h=0,1,\ldots,H_i$ is (a spanning congruence sequence for) the arithmetic partition $\lambda_{i_{-h}}^h(\Rcal_{i_{-h}})$, of length $n_{i_{-h}}\in\IN_0$.
\end{enumerate}
\end{enumerate}
\item The sequence $(\Dfrak_n)_{n=0,1,\ldots,N}$ is a recursive tree description list, with associated rooted tree isomorphism type sequence $(\Ifrak_n)_{n=0,1,\ldots,N}$, such that the $\Ifrak_n$ are just those rooted tree isomorphism types that are of one of the forms
\begin{enumerate}
\item $\Tree_i(\Pcal_i,\vec{\nu}^{(\Pcal_i)})$ for some $\overline{f}$-transient $i$ and some $\vec{\nu}^{(\Pcal_i)}\in\{\emptyset,\neg\}^{m_i}$ such that $\Bcal(\Pcal_i,\vec{\nu}^{(\Pcal_i)})\not=\emptyset$;
\item $\Tree_{\Gamma_f}(0_{\IF_q})$; or
\item $\Tree_i^{(h)}(\Pcal_{i,h},\vec{\nu}^{(\Pcal_{i,h})})$ for some $\overline{f}$-periodic $i\not=d$, some $h\in\{0,1,\ldots,H_i\}$ and some $\vec{\nu}^{(\Pcal_{i,h})}\in\{\emptyset,\neg\}^{n_{i_0}+n_{i_{-1}}+\cdots+n_{i_{-h}}}$ such that $\Bcal(\Qcal_{i,h},\vec{\nu}^{(\Pcal_{i,h})}\diamond\vec{\xi}_{i,h})$, which is the set of all points in $\Bcal(\Pcal_{i,h},\vec{\nu}^{(\Pcal_{i,h})})$ of $\hfrak$-value $h$, is nonempty.
\end{enumerate}
\item The objects $S_{n,i}$ satisfy the following.
\begin{enumerate}
\item If $i$ is $\overline{f}$-transient, then
\[
S_{n,i}=\{\vec{\nu}^{(\Pcal_i)}\in\{\emptyset,\neg\}^{m_i}: \Bcal(\Pcal_i,\vec{\nu}^{(\Pcal_i)})\not=\emptyset\text{ and }\Tree_i(\Pcal_i,\vec{\nu}^{(\Pcal_i)})\cong \Ifrak_n\}.
\]
\item If $i=d$, then $S_{n,i}=S_{n,d}\in\{\emptyset,\neg\}$ is the logical sign associated with the truth value of the isomorphism relation $\Tree_{\Gamma_f}(0_{\IF_q})\cong \Ifrak_n$.
\item If $i\not=d$ is $\overline{f}$-periodic, then $S_{n,i}=(S_{n,i,h})_{h=0,1,\ldots,H_i}$ where
\begin{align*}
S_{n,i,h}=\{&\vec{\nu}^{(\Pcal_{i,h})}\in\{\emptyset,\neg\}^{n_{i_0}+n_{i_{-1}}+\cdots+n_{i_{-h}}}: \\
&\Bcal(\Qcal_{i,h},\vec{\nu}^{(\Pcal_{i,h})}\diamond\vec{\xi}_{i,h})\not=\emptyset\text{ and }\Tree_i^{(h)}(\Pcal_{i,h},\vec{\nu}^{(\Pcal_{i,h})})\cong \Ifrak_n\}.
\end{align*}
\end{enumerate}
\end{enumerate}
\end{deffinition}

In implementations, we assume that each set $S_{n,i}$ for $\overline{f}$-transient $i$, as well as each set $S_{n,i,h}$ for $\overline{f}$-periodic $i<d$ and $h\in\{0,1,\ldots,H_i\}$, is represented by a lexicographically ordered array of bit strings, where a bit $0$ stands for $\emptyset$ and a bit $1$ stands for $\neg$. We note that while a partition-tree register for $f$ does not explicitly mention the partitions $\Pcal_i=\Qcal_{i,H_i}$ for $\overline{f}$-periodic indices $i<d$, it is easy to read off their spanning congruence sequences and associated rooted trees from the register. Namely,
\begin{itemize}
\item the concatenation of the congruence sequences in $\Zcal_i$ spans $\Pcal_i$; and
\item by Proposition \ref{periodicCosetsProp}, the rooted tree associated with a block $\Bcal(\Pcal_i,\vec{\nu}^{(\Pcal_i)})$ of $\Pcal_i$ is of the form $\Tree_i^{(h)}(\Pcal_{i,h},\vec{\nu}^{(\Pcal_{i,h})})$ for suitable $h\in\{0,1,\ldots,H_i\}$ and $\vec{\nu}^{(\Pcal_{i,h})}\in\{\emptyset,\neg\}^{n_{i_0}+n_{i_{-1}}+\cdots+n_{i_{-h}}}$. The relevant parameters $h$ and $\vec{\nu}^{(\Pcal_{i,h})}$ can be read off from the logical sign tuple $\vec{\nu}^{(\Pcal_i)}$ that characterizes the block of $\Pcal_i$.
\end{itemize}

Before we proceed with the actual complexity analysis of Problems 1--3, we make some comments, starting with a discussion of our computational model.

As was already hinted at in Subsection \ref{subsec2P4}, in order to even stand a chance of achieving polynomial runtime for our algorithms, we need quantum computers at least for certain subtasks, such as whenever a modular multiplicative order or a discrete logarithm needs to be computed. That being said, we only relegate certain well-defined tasks, for which efficient quantum algorithms are already known, to quantum computers, while the rest of our algorithms can be performed on a classical computer. Therefore, we use the following two computational models:
\begin{itemize}
\item a bit operation model with several kinds of queries for the tasks for which no efficient algorithms are known on a classical computer (such as integer factorizations or discrete logarithm computations). In this model, which we henceforth refer to as the \emph{query model}\phantomsection\label{term59} (and an algorithm in that model is a \emph{query algorithm}\phantomsection\label{term60}), the complexity is measured as a tuple that tracks the amount of bit operations used outside queries and the amount of times each kind of query is called for;
\item a model in which classical and quantum computers are used in tandem and can \enquote{feed} their outputs to each other; we refer to algorithms built like that as \emph{dual algorithms}\phantomsection\label{term61}, and to the model as the \emph{dual model}\phantomsection\label{term62}.
\end{itemize}
For the quantum side of the computations in the dual model, we specifically use the quantum \emph{circuit} model, so whenever we speak of \emph{quantum complexity}\phantomsection\label{term63}, we mean quantum \emph{(elementary) gate} complexity. All quantum algorithms which we use are based on Shor's seminal paper \cite{Sho94a}, and all of them are \emph{Las Vegas algorithms}\phantomsection\label{term64}, i.e., their runtime on a given input varies randomly, and their specified bit operation cost, gate complexity and number of conversions from bits to qubits and vice versa are to be understood as \emph{expected} values. This also means that as a whole, all of our dual algorithms are Las Vegas algorithms, and all parts of their specified complexities are expected values only.

On the other hand, for the \enquote{classical} side of the computations in either model, we use a bit operation model based on random memory access in the vein of \cite[Section 1.2]{AHU75a}, in which memory access takes $O(N)$ bit operations if $N$ is the bit length of the address (index) of the memory register that needs to be accessed. For example, accessing the stored value of a variable $y_k$ where $k$ is a non-negative integer takes $O(\log{k})$ bit operations -- the entire memory address consists of a bit encoding for the letter \enquote{$y$} (which is assumed to be of length $O(1)$), concatenated with the standard binary representation of $k$. We thus assume that \enquote{jumping} to a place in memory after its address has been scanned is free. In addition to accessing memory registers by reading in their addresses, we also assume that we can save certain positions within a register through placing pointers (of which we have a finite amount, though we do not specify a concrete bound on their number), which enables us to jump back to that specific position (bit) in memory at a cost of $O(1)$ bit operations. Moreover, we assume that it takes $O(1)$ bit operations to move to a neighboring position in memory, including to the next entry of an array. We refer to the classical part of our complexity as \emph{classical complexity}\phantomsection\label{term65} or (synonymously) \emph{bit operations}\phantomsection\label{term66}.

Now, it is well-known (see e.g.~\cite[Subsection IV.3]{Wat12a}) that each classical circuit has an equivalent quantum circuit in which the number of elementary gates is only larger by at most a constant factor. Based on this, it may seem tempting to just use circuits for both kinds of computations in the dual model, so that the classical part could be subsumed (without changing the Landau $O$-class of the gate complexity) in the quantum part, and it appears that this is the usual approach for quantum complexity analysis. For example, in \cite[Sections 7.3 and 7.4]{KLM07a}, the complexity analysis of specific quantum algorithms (i.e., those that do not involve operations in a black-box group) only consists of counting the involved number of quantum gates, while the bit operation (or classical gate) cost of pre- and post-processing is ignored. For the algorithms in \cite[Sections 7.3 and 7.4]{KLM07a}, this is perfectly fine, as that classical cost is a big-$O$ of the quantum gate count regardless of whether bit operations or classical gates are used for measurement. However, our algorithms do involve a significantly larger classical cost than quantum gates, as can already be seen in the complexity bounds from Lemma \ref{complexitiesLem2}; we note that while these are essentially algorithms from \cite[Sections 7.3 and 7.4]{KLM07a}, they do end up with a relatively large classical cost if one wants them to be Las Vegas algorithms (due to the use of the AKS primality test). Hence, keeping track of the classical cost (whether bit operations or classical gates) and quantum gates separately seems natural, especially since the actual time cost of each quantum gate in a large-scale physical implementation of a quantum computer is not known at this point.

As for why we use bit operations (and not gates) in the classical part of our computations, we note that our algorithms for solving Problems 2 and 3 involve a copious amount of \enquote{bookkeeping}, i.e., memory access, and in any of the two circuit models, memory access is generally costly. Indeed, let us assume that, say, in the classical circuit model, we wish to access the value of a previously computed variable $y_k\in\{0,1\}$, where the index $k\in\{1,\ldots,N\}$ is also a result of an earlier computation. When building the circuit, we do not know a priori which of the associated $N$ wires carries the relevant information, and so this needs to be processed via a subcircuit that takes as input those $N$ wires and the $O(\log{N})$ wires carrying (the bit representation of) $k$. But each elementary gate only accepts $O(1)$ input bits, so the said subcircuit performing the memory access must consist of at least $cN$ elementary gates for some constant $c>0$, as opposed to the $O(\log{N})$ cost of memory access for the analogous problem in our chosen bit operation model.

When communication between the classical and quantum part of a dual algorithm happens, classical bit strings $\vec{\xfrak}\in\{0,1\}^N$ need to be converted into the corresponding qubit registers $|\vec{\xfrak}\rangle$\phantomsection\label{not194} and vice versa. In our algorithms dealing with generalized cyclotomic mappings of $\IF_q$, the bit length $N$ is in $O(\log{q})$ for each such conversion. Because it is not clear how costly such conversions are, it is of interest to count them separately (for both conversion directions together) in what we call the \emph{conversion complexity}\phantomsection\label{term67} of the corresponding dual algorithm. The copying of converted information over to the next classical computer or quantum circuit respectively, as well as the measurement taken at the end of a quantum circuit, are considered a part of the respective conversion process, and we do not track their cost separately. For standardization purposes, we assume that both the original input and final output of a dual algorithm are classical bit strings. In particular, the complexity of a pure quantum algorithm that is viewed as a dual algorithm involves two conversions (one each at the beginning and end of the algorithm) in addition to the quantum gate count.

Let us also talk about Grover's quantum algorithm for unstructured database search from \cite{Gro96a}. This algorithm is famous for providing a quadratic speedup over the classical linear search algorithm, and given the aforementioned copious amount of bookkeeping in our algorithms, it seems natural to use it. However, there are some subtleties to take into account here, which ultimately led the authors to decide against the inclusion of Grover's algorithm in our analysis. The usual complexity analysis for Grover's algorithm assumes that the list to be searched (or rather, the associated characteristic function $\chi$\phantomsection\label{not195} for the piece of information we want to find in the list) is given as a certain unitary operator $U_{\chi}$\phantomsection\label{not196}, called \emph{phase inversion}\phantomsection\label{term68}, which is subsequently used as a part of the quantum circuit for the algorithm and treated as an oracle. The celebrated Grover complexity of $O(\sqrt{N})$ for searching a list of length $N$ refers to the number of times $U_{\chi}$ (and another, so-called phase shift operator) is applied before the final measurement. However, we are interested in gate complexities, so the gate complexity of $U_{\chi}$ needs to be included as an additional factor. Now, $\chi$ could be any function $\{0,1,\ldots,N-1\}\rightarrow\{0,1\}$ (the oft-used assumption that $\chi(j)=1$ for a unique index $j$ does not apply to our case), which we may also view as a (partial) Boolean function in $n=\lfloor\log_2{N}\rfloor$ variables. This means that in order for the quantum (gate) complexity of Grover's algorithm to \enquote{beat} the bit operation complexity of linear search, the worst-case quantum gate complexity of an $n$-variable Boolean function would need to be in $o(2^{n/2})$, and it is not clear whether this holds. We do note that it is known that the worst-case \emph{classical} gate complexity of a Boolean function in $n$ variables is of order of magnitude $2^n/n$ (Shannon, \cite[Theorems 6 and 7 on pp.~76f.]{Sha49a}), and that it is only a certain power away from the worst-case quantum gate complexity of an $n$-variable Boolean function (Beals et al., \cite{BBCMW01a}).

We observer that our treatment of quantum algorithms in the dual model is idealized in the sense that we ignore the possibility of errors due to hardware failure and quantum noise. Like many authors, we do so relying on the celebrated Quantum Threshold Theorem, the morale of which is that once the failure rate per elementary gate can be pushed beneath a certain, constant threshold, arbitrarily robust quantum algorithms can be constructed at little extra cost compared to their idealized counterparts. This theorem dates back to a paper of Shor \cite{Sho96a}, though the version stated there is weaker than what the theorem is known as today. Several variants of the stronger version (depending on the error model used) were proved independently by Aharonov and Ben-Or \cite{AB08a}, Knill, Laflamme and Zurek \cite{KLZ98a}, and Kitaev \cite{Kit03a}, respectively. The survey \cite{Got10a}, in which the theorem is stated as Theorem 10, provides a unified proof of it.

The preceding discussion motivates the following definition of the notions of algorithmic complexity which our results in this section refer to.

\begin{deffinition}\label{complexitiesDef}
We introduce the following concepts and notations.
\begin{enumerate}
\item We denote by $\{0,1\}^{<\infty}$\phantomsection\label{not197} the set of all finite bit strings. Formally,
\[
\{0,1\}^{<\infty}=\bigcup_{n\in\IN_0}{\{0,1\}^n}.
\]
\item An \emph{algorithmic problem}\phantomsection\label{term69} is a function $\Lfrak$\phantomsection\label{not198} defined on a subset $\Lfrak_{\inn}$\phantomsection\label{not199} of $\{0,1\}^{<\infty}$ and mapping each bit string $\vec{\xfrak}\in\Lfrak_{\inn}$ to some non-empty finite subset $\Lfrak(\vec{\xfrak})\subseteq\{0,1\}^{<\infty}$.
\item In the situation of statement (2), the elements of $\Lfrak_{\inn}$ are called the \emph{admissible inputs for $\Lfrak$}\phantomsection\label{term70}, and for each $\vec{\xfrak}\in\Lfrak_{\inn}$, the elements of $\Lfrak(\vec{\xfrak})$ are called the \emph{admissible outputs for $\xfrak$ (with respect to $\Lfrak$)}\phantomsection\label{term71}.
\item Let $\Lfrak$ be an algorithmic problem, and let $y,y_1,y_2,\ldots,y_n$ be non-negative real parameters associated with the admissible inputs for $\Lfrak$; formally, $y$ and the $y_j$ are functions $\Lfrak_{\inn}\rightarrow\left[0,\infty\right)$.
\begin{enumerate}
\item A tuple $\vec{\Ccal}^{(\qry)}=(\Ccal_{\class},\Ccal_{\fdl},\Ccal_{\mdl},\Ccal_{\mord},\Ccal_{\prt})$\phantomsection\label{not200}\phantomsection\label{not201}\phantomsection\label{not202}\phantomsection\label{not203}\phantomsection\label{not204}\phantomsection\label{not205}\phantomsection\label{not206} each entry of which is a function $\left[0,\infty\right)^n\rightarrow\left[0,\infty\right)$ is called a \emph{$y$-bounded query complexity of $\Lfrak$ (with respect to $y_1,\ldots,y_n$)}\phantomsection\label{term72} if there is a query algorithm which on each input $\vec{\xfrak}\in\Lfrak_{\inn}$ produces an admissible output for $\vec{\xfrak}$ using
\begin{itemize}
\item $O(\Ccal_{\class}(y_1(\vec{\xfrak}),y_2(\vec{\xfrak}),\ldots,y_n(\vec{\xfrak})))$ bit operations outside the queries listed below;
\item $O(\Ccal_{\fdl}(y_1(\vec{\xfrak}),y_2(\vec{\xfrak}),\ldots,y_n(\vec{\xfrak})))$ queries to compute a discrete logarithm in a finite field of size at most $y(\vec{\xfrak})$;
\item $O(\Ccal_{\mdl}(y_1(\vec{\xfrak}),y_2(\vec{\xfrak}),\ldots,y_n(\vec{\xfrak})))$ queries to compute, for given $x,z\in(\IZ/m\IZ)^{\ast}$ where $m<y(\vec{\xfrak})^2$, the modular discrete logarithms $\log_x^{(k)}(z)$ where $k\in\{m,p^{\nu_p(m)}: p\mid m\}$, outputting a list consisting of the pair $(m,\log_x^{(m)}(z))$ and the quadruples $(p,\nu_p(m),p^{\nu_p(m)},\log_x^{(p^{\nu_p(m)})}(z))$ for all primes $p\mid m$;
\item $O(\Ccal_{\mord}(y_1(\vec{\xfrak}),y_2(\vec{\xfrak}),\ldots,y_n(\vec{\xfrak})))$ queries to compute, for a given unit $x\in(\IZ/m\IZ)^{\ast}$ where $m<y(\vec{\xfrak})^2$, the multiplicative orders $\ord_k(x)$ where $k\in\{m,p^{\nu_p(m)}: p\mid m\}$, outputting a list consisting of the pair $(m,\ord_m(x))$ and the quadruples $(p,\nu_p(m),p^{\nu_p(m)},\ord_{p^{\nu_p(m)}}(x))$ for all primes $p\mid m$;
\item $O(\Ccal_{\prt}(y_1(\vec{\xfrak}),y_2(\vec{\xfrak}),\ldots,y_n(\vec{\xfrak})))$ queries to find a primitive root $\rfrak^{(p)}$ modulo each odd prime power divisor $p^{\nu_p(m)}>1$ for some integer $m<y(\vec{\xfrak})$, outputting the corresponding list of quadruples $(p,\nu_p(m),p^{\nu_p(m)},\rfrak^{(p)})$.
\end{itemize}
\item A \emph{$y$-bounded Las Vegas dual complexity for $\Lfrak$}\phantomsection\label{term73} is a triple
\[
\vec{\Ccal}^{(\LV)}=(\Ccal_{\class},\Ccal_{\quant},\Ccal_{\conv})
\]
each\phantomsection\label{not207}\phantomsection\label{not208}\phantomsection\label{not209} entry of which is a function $\left[0,\infty\right)^n\rightarrow\left[0,\infty\right)$ such that there is an (idealized) dual algorithm which on each input $\vec{\xfrak}\in\Lfrak_{\inn}$ terminates after an expected number of
\begin{itemize}
\item $O(\Ccal_{\class}(y_1(\vec{\xfrak}),y_2(\vec{\xfrak}),\ldots,y_n(\vec{\xfrak})))$ bit operations,
\item $O(\Ccal_{\quant}(y_1(\vec{\xfrak}),y_2(\vec{\xfrak}),\ldots,y_n(\vec{\xfrak})))$ elementary quantum gates, and
\item $O(\Ccal_{\conv}(y_1(\vec{\xfrak}),y_2(\vec{\xfrak}),\ldots,y_n(\vec{\xfrak})))$ conversions of bit strings of length in $O(\log{y(\vec{\xfrak})})$ into qubit registers and of length $O(\log{y(\vec{\xfrak})})$ qubit registers into bit strings,
\end{itemize}
producing an admissible output for $\vec{\xfrak}$.
\end{enumerate}
\end{enumerate}
\end{deffinition}

In our algorithms, the value of $y(\vec{\xfrak})$ from Definition \ref{complexitiesDef} is always equal to the corresponding field size $q$. While our definition of the query model does not explicitly include integer factorization queries, they are subsumed in either of modular discrete logarithm queries or multiplicative order queries. Indeed, in order to factor $m\in\IN^+$ with $m<y(\vec{\xfrak})^2$, one can simply make the query to compute the multiplicative order modulo $m$ of $x:=1\in(\IZ/m\IZ)^{\ast}$. The resulting output consists of the pairs $(k,1)$ with $k\in\{m,p^{\nu_p(m)}: p\mid m\}$, from which it is straightforward to read off the prime factorization of $m$. Likewise, one could make a modular discrete logarithm query with $x:=y:=1$.

The assumption from Definition \ref{complexitiesDef}(2) that each admissible input for $\Lfrak$ should only have finitely many admissible outputs is without loss of generality for our analysis. It simplifies the formulation of Lemma \ref{queryComplexityLem} below, which is straightforward to prove and basically states that query complexities behave additively with respect to composition of algorithmic problems, which is defined as follows. If $\Lfrak$ and $\Lfrak'$ are algorithmic problems such that $\Lfrak(\vec{\xfrak})\subseteq\Lfrak'_{\inn}$ for each $\vec{\xfrak}\in\Lfrak_{\inn}$, then the \emph{composition of $\Lfrak$ and $\Lfrak'$}\phantomsection\label{term74}, written $\Lfrak\Lfrak'$\phantomsection\label{not210} or $\Lfrak'\circ\Lfrak$\phantomsection\label{not211}, is the algorithmic problem with input set $\Lfrak_{\inn}$ that is defined via $(\Lfrak'\circ\Lfrak)(\vec{\xfrak}):=\Lfrak'(\Lfrak(\vec{\xfrak}))$ (the element-wise image of the set $\Lfrak(\vec{\xfrak})$ under the function $\Lfrak'$).

\begin{lemmma}\label{queryComplexityLem}
Let $\Lfrak$ and $\Lfrak'$ be algorithmic problems such that $\Lfrak(\vec{\xfrak})\subseteq\Lfrak'_{\inn}$ for each $\vec{\xfrak}\in\Lfrak_{\inn}$, and let $y,y_1,y_2,\ldots,y_n$, respectively $y',y'_1,y'_2,\ldots,y'_{n'}$ be non-negative real parameters that are associated with the admissible inputs of $\Lfrak$, respectively $\Lfrak'$. For $j'=1,2,\ldots,n'$, we define
\begin{align*}
y_{n+j'}:\Lfrak_{\inn}&\rightarrow\left[0,\infty\right), \\
\vec{\xfrak}&\mapsto\max\{y'_{j'}(\vec{\yfrak}): \vec{\yfrak}\in\Lfrak(\vec{\xfrak})\}.
\end{align*}
Because the composition $\Lfrak\Lfrak'$ has input set $\Lfrak_{\inn}$, we may view each $y_j$ for $1\leq j\leq n+n'$ as a parameter for $\Lfrak\Lfrak'$. Moreover, we define
\begin{align*}
y^+:\Lfrak_{\inn}&\rightarrow\left[0,\infty\right), \\
\vec{\xfrak}&\mapsto\max\{y(\vec{\xfrak}),y'(\vec{\yfrak}): \vec{\yfrak}\in\Lfrak(\vec{\xfrak})\}.
\end{align*}
Finally, we let $\vec{\Ccal}^{(\qry)}$, respectively $\vec{\Ccal'}^{(\qry)}$, be a $y$-bounded query complexity for $\Lfrak$ with respect to $y_1,\ldots,y_n$, respectively a $y'$-bounded query complexity for $\Lfrak'$ with respect to $y'_1,\ldots,y'_{n'}$. For $k=1,\ldots,5$, we denote by $\Ccal_k$, respectively $\Ccal'_k$, the $k$-th entry of $\vec{\Ccal}^{(\qry)}$, respectively $\vec{\Ccal'}^{(\qry)}$. Then, defining
\begin{align*}
\Ccal^+_k:\left[0,\infty\right)^{n+n'}&\rightarrow\left[0,\infty\right), \\
(z_1,\ldots,z_{n+n'})&\mapsto\Ccal_k(z_1,\ldots,z_n)+\Ccal'_k(z_{n+1},\ldots,z_{n+n'}),
\end{align*}
the tuple $\vec{\Ccal^+}^{(\qry)}:=(\Ccal^+_1,\ldots,\Ccal^+_5)$ is a $y^+$-bounded query complexity for $\Lfrak\Lfrak'$ with respect to $y_1,\ldots,y_{n+n'}$.
\end{lemmma}

On the other hand, for each given algorithmic problem $\Lfrak$ and non-negative real parameter $y$ associated with the admissible inputs for $\Lfrak$, a $y$-bounded Las Vegas dual complexity for $\Lfrak$ can be derived from a $y$-bounded query complexity for $\Lfrak$ as long as $y$ can be bounded in terms of the other parameters $y_j$; see Lemma \ref{LVComplexityLem} below.

As usual, when specifying complexities in a concrete situation, we identify functions with their defining terms. For example, if $d$ and $q$ are the relevant parameters associated with our inputs, we may specify a $q$-bounded Las Vegas dual complexity as
\[
(d^3\log{q},d\log^{2+o(1)}{q},\log^{1+o(1)}{q}),
\]
rather than introduce names for the functions in the three components. In this context, we also note that $\log^k{x}$ always denotes the arithmetic power $(\log{x})^k$, \emph{not} the function value at $x$ of the $k$-fold iterate of $\log$. We always spell iterated logarithms out ($\log\log$, $\log\log\log$, etc.).

The following lemma, which is used throughout this section, provides the complexities of some fundamental algorithmic problems.

\begin{lemmma}\label{complexitiesLem}
The following hold.
\begin{enumerate}
\item Addition and subtraction of integers of absolute value less than $m$, as well as addition and subtraction modulo $m$ cost $O(\log{m})$ bit operations each.
\item Addition and subtraction in the finite field $\IF_q$ cost $O(\log{q})$ bit operations each.
\item Multiplication of positive integers less than $m$, multiplication modulo $m$ and division of positive integers less than $m$ with remainder each cost $O(\log^{1+o(1)}{m})$ bit operations.
\item Let $x\in(\IZ/m\IZ)^{\ast}$. The computation of $\inv_m(x)$, the multiplicative inverse of $x$ modulo $m$, costs $O(\log^{1+o(1)}{m})$ bit operations.
\item Multiplication, multiplicative inversion and division in the finite field $\IF_q$ each cost $O(\log^{1+o(1)}{q})$ bit operations.
\item Let $x\in\IZ/m\IZ$ and $e\in\IN_0$. The computation of the power $x^e$ modulo $m$ costs $O(\log(e)\log^{1+o(1)}{m})$ bit operations.
\item Let $x\in\IF_q$ and $e\in\IN_0$. The computation of $x^e$ costs $O(\log(e)\log^{1+o(1)}{q})$ bit operations.
\item The computations of the $\gcd$ and $\lcm$ of two positive integers that are at most $m$ cost $O(\log^{1+o(1)}{m})$ bit operations each.
\item Checking deterministically whether a given positive integer $m$ is a prime costs $O(\log^{6+o(1)}{m})$ bit operations.
\item An array of $n$ bit strings, each of length $k$, can be lexicographically sorted within $O(kn\log{n})$ bit operations.
\item We assume given two lexicographically sorted arrays of bit strings (not necessarily all of the same length) and consider the algorithmic problem of finding the lexicographically sorted version of their concatenation (i.e., the problem of \emph{merging}\phantomsection\label{termMerge} those sorted arrays). For $j=1,2$, say the $j$-th array has $N_j$ entries, and the sum of the bit lengths of the strings stored in it is $l^{(j)}_{\total}$. Then those two sorted arrays can be merged within $O(N_1+N_2+l^{(1)}_{\total}+l^{(2)}_{\total})$ bit operations (and thus within $O(l^{(1)}_{\total}+l^{(2)}_{\total})$ bit operations if all bit strings in question are non-empty).
\end{enumerate}
\end{lemmma}

\begin{proof}
For statement (1), it is well-known (and easy to check) that using the schoolbook algorithms for addition and subtraction yields the specified complexities.

For statement (2), let $q=p^m$. We refer to \cite[Table 2.8 on p.~84]{MOV97a}, and note that the only $(\IZ/p\IZ)$-operations involved in an addition/subtraction in $\IF_q$ are modular additions/subtractions. Therefore, statement (1) implies that the cost of addition and subtraction in $\IF_q$ is in $O(m\log{p})=O(\log{q})$, as required.

For statement (3), it follows from the Sch{\"o}nhage-Strassen algorithm \cite{SS71a} or the (slightly faster) algorithm \cite{HH21a} by Harvey and van der Hoeven that the multiplication of two positive integers less than $m$ costs $O(\log^{1+o(1)}{m})$ bit operations. Moreover, integer division with remainder also costs $O(\log^{1+o(1)}{m})$ bit operations if the Newton-Raphson algorithm is used for it; see \cite[Subsection 1.3]{Agr09a}. Multiplication modulo $m$ can be done by performing a (non-modular) multiplication of the two integers in question (resulting in a number with $O(\log\left(m^2\right))=O(\log{m})$ bits), followed by a modular reduction (which is a part of division with remainder). In total, multiplication modulo $m$ thus also only requires $O(\log^{1+o(1)}{m})$ bit operations.

For statement (4), we note that inversion modulo $m$ can be done with the Extended Euclidean Algorithm, which takes $O(\log^{1+o(1)}{m})$ bit operations when using an accelerated variant of it due to Sch{\"o}nhage (based on earlier ideas of Knuth) \cite{Sch71a}; see also \cite{WP03a}, which provides a generalization of this and may be more accessible due to being written in English.

For statement (5), we note that $\IF_q$ is given as $(\IZ/p\IZ)[T]/(P(T))$ where $P(T)$ is a monic primitive irreducible polynomial of degree $m:=\log_p{q}$. In order to multiply two elements $Q_1(T)+(P(T))$ and $Q_2(T)+(P(T))$ of $\IF_q$, one computes the polynomial product $Q_1(T)Q_2(T)\in(\IZ/p\IZ)[T]$, then determines its remainder upon division by $P(T)$. As observed in \cite[second paragraph in Section 2]{GP01a}, fast methods for multiplication of polynomials over $\IZ/p\IZ$ of degree at most $n$, as well as for divisions with remainder of such polynomials, take $O(n^{1+o(1)})$ operations (additions, subtractions, multiplications, multiplicative inversions) in $\IZ/p\IZ$. This corresponds to a bit operation cost of $O(n^{1+o(1)}\log^{1+o(1)}{p})$ by statements (1) and (3). It follows that the computation of $Q_1(T)Q_2(T)$, and the subsequent computation of its remainder modulo $P(T)$, both take $O(m^{1+o(1)}\log^{1+o(1)}{p})=O(\log^{1+o(1)}{q})$ bit operations, as needed for the asserted complexity bound on multiplication in $\IF_q$ to hold.

Now, because a division in $\IF_q$ consists of a multiplicative inversion followed by a multiplication, it suffices to argue that the bit operation cost of multiplicative inversion in $\IF_q$ is in $O(\log^{1+o(1)}{q})$ to conclude the proof of this statement. Assuming that $P(T)\nmid Q(T)$, the multiplicative inversion of $Q(T)$ modulo $P(T)$ may be performed by writing $1=\gcd(P(T),Q(T))$ as a $(\IZ/p\IZ)[T]$-linear combination of $P(T)$ and $Q(T)$, and reducing the scalar of $Q(T)$ in this linear combination by $P(T)$. The algorithm described in \cite[Section 8.9]{AHU75a} uses $O(\log^{1+o(1)}(p^m)\cdot\log{m})=O(\log^{1+o(1)}{q})$ bit operations to compute $\gcd(P(T),Q(T))$ (according to \cite[Theorem 8.19]{AHU75a}). In the process, one may store the $(2\times 2)$-matrices $R_1,R_2,\ldots,R_K$ (with coefficients in $(\IZ/p\IZ)[T]$) that are output (in the listed order) by the $K\in O(\log{m})$ calls of the HGCD procedure from \cite[Fig.~8.7 on p.~304]{AHU75a}. Let $P_0(T)=P(T),P_1(T)=Q(T),P_2(T),\ldots,P_N(T)=\gcd(P(T),Q(T))=1$ be the successive remainders appearing in the classical, \enquote{slow} version of the Euclidean algorithm applied to $(P(T),Q(T))$, and for $t\in\{0,1,\ldots,N-1\}$, let $Q_t(T)$ be the quotient of the polynomial division of $P_t(T)$ by $P_{t+1}(T)$, which satisfies $\deg{Q_t(T)}=\deg{P_t(T)}-\deg{P_{t+1}(T)}$. By \cite[statement on p.~303 that the output of HGCD is of the form $R_{0j}$, and definition of $R_{ij}$ before Example 8.10 on p.~302]{AHU75a}, each of the matrices $R_j$ is a product of matrices of the form
\[
\begin{pmatrix}
0 & 1 \\
1 & -Q_t(T)
\end{pmatrix}
\]
for pairwise distinct $t$. Therefore, each entry of $R_j(T)$ is a polynomial in $(\IZ/p\IZ)[T]$ of degree at most
\[
\sum_{t=0}^{N-1}{\deg{Q_t(T)}}=\sum_{t=0}^{N-1}{(\deg{P_t(T)}-\deg{P_{t+1}(T)})}=\deg{P(T)}-\deg{1}=\deg{P(T)}.
\]
Moreover, by \cite[Lemma 8.5(a)]{AHU75a}, for each $k\in\{1,2,\ldots,K\}$, one has
\[
R_kR_{k-1}\cdots R_1\cdot
\begin{pmatrix}
P(T) \\
Q(T)
\end{pmatrix}
=
\begin{pmatrix}
P_j(T) \\
P_{j+1}(T)
\end{pmatrix}
\]
for some $j=j(k)\in\{0,1,\ldots,N-1\}$, and specifically
\[
R_KR_{K-1}\cdots R_1\cdot
\begin{pmatrix}
P(T) \\
Q(T)
\end{pmatrix}
=
\begin{pmatrix}
P_{N-1}(T) \\
P_N(T)
\end{pmatrix}
=
\begin{pmatrix}
P_{N-1}(T) \\
1
\end{pmatrix}.
\]
This latter equality yields an expression of $1$ as a linear combination of $P(T)$ and $Q(T)$, in which the (reduction modulo $P(T)$ of the) scalar of $Q(T)$ is equal to the (reduction modulo $P(T)$ of the) lower right coefficient of the $(2\times 2)$-matrix $R_KR_{K-1}\cdots R_1$. It takes $O(K)\subseteq O(\log{m})$ additions, multiplications and divisions with remainder in $(\IZ/p\IZ)[T]$ to compute this matrix product, which also corresponds to a bit operation cost of $O(\log{m}\cdot\log^{1+o(1)}{p^m})=O(\log^{1+o(1)}{q})$, as required.

For statement (6), we note that using \enquote{Square and Multiply}, the power $x^e\mod{m}$ can be computed with $O(\log{e})$ multiplications modulo $m$, so statement (3) yields the claim.

For statement (7), the proof is analogous to the one for statement (6), but using statement (5) in place of statement (3).

For statement (8), one may use the Euclidean Algorithm to compute a greatest common divisor and refer to \cite
{Sch71a} or \cite{WP03a}. Moreover, $\lcm(x,y)=xy/\gcd(x,y)$, so statement (3) completes the proof of this claim.

For statement (9), the asserted complexity is achieved by a variant of the AKS primality test devised by Lenstra and Pomerance, see \cite{AKS04a} and \cite{LP11a}.

For statement (10), we refer the reader to \cite[Algorithm 3.1 on pp.~78f.]{AHU75a}, observing that the variable $m$ from that algorithm has the value $2$ in our situation. We note that while \cite[Theorem 3.1 on p.~79]{AHU75a} states that this algorithm costs $O(kn)$ bit operations, this uses an assumption which our computational model does not share. Namely, \cite[Algorithm 3.1 on pp.~78f.]{AHU75a} uses a \enquote{pointer} for each bit string, which must not be confused with the way we use that word. In our model, a pointer is a short-cut to jump to a previously saved point in memory using only $O(1)$ bit operations, and we may only use $O(1)$ of these pointers (i.e., the number of pointers used must not tend to $\infty$ as the input length tends to $\infty$). On the other hand, in \cite[Algorithm 3.1 on pp.~78f.]{AHU75a}, the word \enquote{pointer} appears to denote what we would call the memory address of the respective bit string. It is stated explicitly in \cite[Algorithm 3.1 on pp.~78f.]{AHU75a} that the authors of that book assume that a pointer in their sense can be processed (i.e., stored and used to jump to the respective bit string) within $O(1)$ bit operations. However, in our model, since $n$ of these memory addresses are needed, it takes $O(\log{n})$ bit operations to process an address, which leads to the additional factor $\log{n}$ in our cost.

For statement (11), we observe that it is easy to prove that the merging algorithm \cite[Algorithm M on p.~158]{Knu98a} achieves this complexity, as long as pointers (in our sense of the word) are used to immediately jump back to saved positions in the arrays, which avoids additional logarithmic factors in the complexity. In fact, in the discussion from \cite[p.~159]{Knu98a}, it is stated that the achieved complexity is in $O(N_1+N_2)$, but this uses the assumption that each stored bit string has constantly bounded bit length (causing $l^{(j)}_{\total}\in O(N_j)$ for $j=1,2$).
\end{proof}

The next lemma essentially provides the Las Vegas dual complexities of the query problems from our query model. It is used to translate query complexities into Las Vegas dual complexities; see Lemma \ref{LVComplexityLem} below.

\begin{lemmma}\label{complexitiesLem2}
The following hold.
\begin{enumerate}
\item The prime factorization of the positive integer $m$ can be performed with a Las Vegas dual algorithm with $m$-bounded complexity $(\log^{7+o(1)}{m},\log^{3+o(1)}{m},\log{m})$.
\item The computation of the multiplicative order of $x\in(\IZ/m\IZ)^{\ast}$ can be performed with a Las Vegas dual algorithm with (expected) $m$-bounded dual complexity $(\log^{7+o(1)}{m},\log^{3+o(1)}{m},\log{m})$.
\item Let $m$ be a positive integer. For $x,y\in(\IZ/m\IZ)^{\ast}$, the modular discrete logarithm $\log_x^{(m)}(y)$ can be computed with a Las Vegas dual algorithm with $m$-bounded complexity $(\log^{7+o(1)}{m},\log^{3+o(1)}{m},\log{m})$.
\item Let $q$ be a prime power. For $x,y\in\IF_q^{\ast}$, the discrete logarithm $\log_x(y)$ can be computed with a Las Vegas dual algorithm with $q$-bounded complexity
\[
(\log^{3+o(1)}{q},\log^{3+o(1)}{q},\log{q}).
\]
\item Let $p$ be an odd prime, and $k$ a positive integer. On input $(p,k)$, a primitive root modulo $p^k$ can be found with a Las Vegas dual algorithm with $p$-bounded complexity $(\log^{7+o(1)}{p},\log^{3+o(1)}{p},\log{p})$.
\end{enumerate}
\end{lemmma}

\begin{proof}
For the proofs of statements (1) and (2), we follow the approach described in \cite[Section 7.3]{KLM07a}, which is originally due to Miller \cite{Mil75a} and Shor \cite{Sho94a}. For this, we need to first analyze the complexity of the order-finding algorithm from \cite[p.~137]{KLM07a}, which uses quantum circuits combined with some classical post-processing. It should be noted that this algorithm admits absolute bounds on (i.e., not just expected values of) the different parts of its complexity, but the output is only correct with probability at least $0.399$, making this a \emph{Monte Carlo algorithm}\phantomsection\label{term75}, not the Las Vegas algorithm we wish to construct in the end.

First, we need to analyze the complexity of the continued fractions algorithm (it is mentioned in \cite[Theorem 7.1.7]{KLM07a} that this algorithm has polynomial complexity without specifying the degree). Let us assume given a floating point number (in binary format) that represents the rational number $k/2^n$ where $k\in\{0,1,\ldots,2^n-1\}$. Then there is a sequence of fractions
\[
\conv_j(k,2^n)=\frac{\num_j(k,2^n)}{\den_j(k,2^n)}
\]
for\phantomsection\label{not212}\phantomsection\label{not213}\phantomsection\label{not214} $j=1,2,\ldots,N\in O(n)$, the \emph{(principal) convergents of $k/2^n$}\phantomsection\label{term76}, which are optimal approximations of $k/2^n$ relative to the size of their denominators $\den_j(k,2^n)\leq 2^n$; for details, we refer the reader to \cite[Theorem 7.1.7 and Exercise 7.1.7]{KLM07a} and \cite[Chapter I]{Lan95a}. An important property which we need later is that any reduced integer fraction $y/z$ such that
\[
\left|\frac{k}{2^n}-\frac{y}{z}\right|\leq\frac{1}{2z^2}
\]
is one of the convergents of $k/2^n$; see \cite[Corollary 2 on p.~11]{Lan95a}. By \cite[Theorem 1 on p.~2]{Lan95a}, each of the two sequences $(\num_j(k,2^n))_{j=1,\ldots,N}$ and $(\den_j(k,2^n))_{j=1,\ldots,N}$ is defined through a simple recursion (involving $O(1)$ integer additions and multiplications in each recursion step) in terms of the so-called \emph{continued fractions coefficients of $k/2^n$}\phantomsection\label{term77}, which are just the integer quotient values in the divisions that occur upon applying the Euclidean algorithm to $(2^n,k)$. In view of \cite[Problem 31-2 posed on p.~937]{CLRS09a} (see also \cite{Che23a} for a worked out solution of this problem using a telescopic sum argument), one can compute and store the continued fraction coefficients of $k/2^n$ using $O(n^2)$ bit operations. Following that, the computation of $\num_j(k,2^n)$ and $\den_j(k,2^n)$ for all relevant $j$ takes another $O(n\cdot n^{1+o(1)})\subseteq O(n^{2+o(1)})$ bit operations by statements (1) and (3) of Lemma \ref{complexitiesLem}.

Having analyzed the continued fractions algorithm, let us now turn to the order-finding algorithm described in \cite[p.~137]{KLM07a}. In accordance with our notation, we assume that this algorithm is used to find the multiplicative order of $x$ modulo $m$ (in \cite{KLM07a}, the variable $a$, respectively $N$, is used in place of $x$, respectively $m$). The algorithm starts by computing $m':=\lceil 2\log{m}\rceil$, which takes $O(\log{m})$ bit operations. We observe that $\ord(x)$, the multiplicative order of $x$ modulo $m$, is at most $2^{(m'-1)/2}$. After this, we need to initialize two $m'$-qubit registers, which in our dual model formally takes $O(\log{m})$ bit operations for printing the length $m'$ bit strings $0\cdots 0$ and $0\cdots 01$, followed by $2\in O(1)$ conversions of these strings into the corresponding qubit registers $|0\rangle^{\otimes m'}$ and $|0\cdots 01\rangle$. Steps 4--6 of the algorithm in \cite[p.~137]{KLM07a} are applications of quantum circuits to those registers, which according to the analysis in \cite[pp.~138f.]{KLM07a} consist of $O(\log^{2+o(1)}{m})$ elementary gates. Next, the measurement described in Step 7 of \cite[p.~137]{KLM07a} corresponds to one more conversion in our model (this time from qubits to classical bits), and with high probability, it leads to a \enquote{good estimate} (see below) $k_1/2^{m'}$ of a random integer multiple $t/\ord_m(x)$ of $1/\ord_m(x)$, with $t\in\{0,1,\ldots,\ord_m(x)-1\}$. Considering Step 8 of \cite[p.~137]{KLM07a} next, we believe that there is a mistake in the formulation of this step, the first sentence of which should in our opinion read (using the notation from there) \enquote{Use the continued fractions algorithm to obtain integers $c_1\geq0$ and $r_1$ with $1\leq r_1\leq 2^{(n-1)/2}$ such that $|x_1/2^n-c_1/r_1|\leq 1/2^{n+1}$.} In any case, this is a formulation that works. Indeed, switching back to our notation, as long as the output $k_1/2^{m'}$ of Step 7 is a good estimate of $t/\ord(x)$ for some $t\in\{0,1,\ldots,\ord(x)-1\}$, it follows (by the definition of \enquote{good estimate} from \cite[beginning of Subsection 7.1.1]{KLM07a}, see in particular \cite[Theorem 7.1.4]{KLM07a}) that
\[
\left|\frac{t}{\ord(x)}-\frac{k_1}{2^{m'}}\right|\leq\frac{1}{2^{m'+1}}<\frac{1}{2^{m'}}\leq\frac{1}{2\ord(x)^2}
\]
whence, as noted above, we have $t/\ord(x)=\conv_j(k_1,2^{m'})$ for some $j$ by \cite[Corollary 2 on p.~11]{Lan95a}. By our above analysis of the continued fractions algorithm, one can thus find, using $O(\log^{2+o(1)}{m})$ bit operations, an index $j$ and associated values $\num_j(k_1,2^{m'})$ and $\den_j(k_1,2^{m'})$ with $\den_j(k_1,2^{m'})\leq 2^{(m'-1)/2}$ such that
\[
\left|\conv_j(k_1,2^{m'})-\frac{k_1}{2^{m'}}\right|=\left|\frac{\num_j(k_1,2^{m'})}{\den_j(k_1,2^{m'})}-\frac{k_1}{2^{m'}}\right|\leq\frac{1}{2^{m'+1}},
\]
unless we had bad luck with regard to the output of Step 7. If so, it makes sense to abandon the computations and output \enquote{FAIL}, as specified in \cite[Step 8 on p.~137]{KLM07a}. Now, it follows that
\begin{align*}
&\left|\frac{t}{\ord(x)}-\conv_j(k_1,2^{m'})\right|\leq\left|\frac{t}{\ord(x)}-\frac{k_1}{2^{m'}}\right|+\left|\frac{k_1}{2^{m'}}-\conv_j(k_1,2^{m'})\right| \\
\leq&\frac{1}{2^{m'+1}}+\frac{1}{2^{m'+1}}=\frac{1}{2^{m'}}\leq\min\left(\frac{1}{2\ord(x)^2},\frac{1}{2\den_j(k_1,2^{m'})^2}\right).
\end{align*}
Using \cite[Exercise 7.1.7(b)]{KLM07a}, this implies that $t/\ord(x)=\conv_j(k_1,2^{m'})$, as required for the correctness of the algorithm from \cite[p.~137]{KLM07a}. Step 9 of \cite[p.~137]{KLM07a} is just a repetition of Steps 1--8, hence does not make a difference for the $O$-class of the complexity. Finally, Steps 10 and 11 of \cite[p.~137]{KLM07a} take $O(\log^{2+o(1)}{m})$ bit operations by statements (6) and (8) of Lemma \ref{complexitiesLem}. In summary, the order-finding algorithm from \cite[p.~137]{KLM07a} may be viewed as a dual algorithm which, on input $x\in(\IZ/m\IZ)^{\ast}$ and $m$, outputs the multiplicative order of $x$ modulo $m$ with probability at least $39.9\%$ (see \cite[Theorem 7.3.2]{KLM07a}), and does so taking $O(\log^{2+o(1)}{m})$ bit operations and elementary quantum gates each, as well as $O(1)$ conversions of $O(\log{m})$-bit strings to $O(\log{m})$-qubit registers or vice versa. As noted in \cite[Theorem 7.3.2]{KLM07a} (and is clear from Step 11), unless the output of that algorithm is \enquote{FAIL}, it always outputs at least an integer multiple of $\ord(x)$. This concludes the preparation for the proofs of statements (1) and (2), which we tackle next.

For statement (1), we assume given a positive integer $m$. We wish to obtain the prime factorization of $m$. Formally, we wish to output the list of pairs $(p,\nu_p(m))$ where $p$ ranges over the prime divisors of $m$. First, we describe and analyze a deterministic (classical) algorithm that decides whether $m$ is a power of a single prime $p$ and, if so, outputs $(p,\nu_p(m))$; see also \cite[Exercise 7.3.3]{KLM07a}. This algorithm is, in turn, based on a deterministic routine that decides whether a given $m\in\IN^+$ is a power $n^k$ of a positive integer $n<m$ and, if so, outputs $(n,k)$ for the smallest possible value of $k\geq2$. We note that if $m=n^k$ for some $n\in\{1,2,\ldots,m-1\}$ and some $k\in\IN^+$, then $k\leq\lfloor\log_2{m}\rfloor$. Therefore, we loop over $k=2,3,\ldots,\lfloor\log_2{m}\rfloor$, and for each fixed value of $k$, we perform a binary search for $n$, using the strict monotonicity of the function $x\mapsto x^k$. More specifically, we initialize $n:=2$, and as long as $n^k<m$, we double $n$ until $n^k=m$ or $n^k>m$. In the latter case, we start a binary search between $\frac{n}{2}$ and $n$. With this approach, the values of the powers $n^k$ which we compute never exceed $2^{\lfloor\log_2{m}\rfloor}m\leq m^2$, whence each individual power computation in the process takes
\[
O(\log(k)\log^{1+o(1)}(m^2))=O(\log\log{m}\log^{1+o(1)}{m})=O(\log^{1+o(1)}{m})
\]
bit operations by statement (6) of Lemma \ref{complexitiesLem}. Because we loop over $O(\log{m})$ values of $k$, and for each $k$, the binary search for $n$ has $O(\log{m})$ iterations, it follows that it takes
\[
O(\log{m}\cdot\log{m}\cdot\log^{1+o(1)}{m})=O(\log^{3+o(1)}{m})
\]
bit operations to find the minimal working value of $k$ and associated $n=\sqrt[k]{m}$, or see that they do not exist.

As mentioned above, we can use this root-finding routine to check whether a given $m\in\IN^+$ is a prime power and, if so, write it as such -- let us describe how. First, by iterating the power-finding routine, one can write $m=n^k$ for the \emph{maximal} $k\in\{1,2,\ldots,\lfloor\log_2{m}\rfloor\}$ such that $m$ has an integer $k$-th root; this takes
\[
O\left(\sum_{t=0}^{\infty}{\log^{3+o(1)}\left(m^{1/2^t}\right)}\right)=O\left(\sum_{t=0}^{\infty}{\frac{1}{2^{3t}}}\cdot\log^{3+o(1)}{m}\right)=O(\log^{3+o(1)}{m})
\]
bit operations. The problem is then reduced to checking whether $n$ is a prime, which takes $O(\log^{6+o(1)}{n})\subseteq O(\log^{6+o(1)}{m})$ bit operations by statement (9) of Lemma \ref{complexitiesLem}. In summary, we have a deterministic routine with complexity in $O(\log^{6+o(1)}{m})$ for deciding whether $m$ is a prime power and, if so, writing it as such.

Shor's general Las Vegas dual approach for factoring $m\in\IN^+$ using reduction ideas of Miller is outlined in \cite[pp.~132f.]{KLM07a}. We start by splitting off the factor $2^{\nu_2(m)}$ from $m$. Because $m$ is given in its binary representation, this only takes $O(\log^{1+o(1)}{m})$ bit operations, accounting for $O(\nu_2(m))\subseteq O(\log{m})$ increases of a counter that remains in $O(\log{m})$ throughout (and thus has $O(\log\log{m})\subseteq O(\log^{o(1)}{m})$ bits). Now, we set $m':=m/2^{\nu_2(m)}$. In the rest of our proof of statement (1), we will only be dealing with \emph{odd} positive integers. We describe a Las Vegas routine that decides whether a given odd positive integer $n$ is a prime power, then does the following:
\begin{itemize}
\item if $n$ is a prime power, it writes $n$ as $p^{\nu_p(n)}$;
\item if $n$ is not a prime power, it finds a factorization of $n$ of the form $n=n'\cdot n''$ where $1<n',n''<n$.
\end{itemize}
We already described above how to decide whether $n=p^{\nu_p(n)}$ and, if so, write it as such using $O(\log^{6+o(1)}{n})$ bit operations, so we start by applying that routine and may henceforth assume that it returned that $n$ is \emph{not} a prime power. Following \cite[p.~133]{KLM07a}, we wish to draw an integer $x\in\{2,3,\ldots,n-1\}$ uniformly at random. Letting $N:=\lfloor\log_2(n-3)\rfloor+1$, we aim to draw the $N$-bit integer $y\in\{0,1,\ldots,n-3\}$ uniformly at random, then set $x:=y+2$. To draw $y$, we initialize an $N$-qubit register to $|0\rangle^{\otimes N}$, then pass it through an $N$-dimensional Hadamard circuit (with elementary gate complexity $N\in O(\log{n})$) to get a uniform superposition of all $N$-bit strings, so that a simple measurement returns (the binary representation of) a random integer in $\{0,1,\ldots,2^N-1\}$. The probability that this integer lies in the range for $y$ is at least $1/2$, so we only need to iterate this procedure an expected number of $O(1)$ times until we get a suitable value for $y$, using $O(\log{n})$ bit operations and elementary quantum gates as well as $O(1)$ conversions to and from $O(\log{n})$-bit strings. Following that, we compute $x=y+2$ and $\gcd(x,n)$, taking $O(\log^{1+o(1)}{n})$ bit operations by statements (1) and (8) of Lemma \ref{complexitiesLem}. If $\gcd(x,n)>1$, we may output the factorization $n=n'\cdot n''$ with $n'=\gcd(x,n)$, taking just another $O(\log^{1+o(1)}{n})$ bit operations to compute $n''$ by division, and are done. Otherwise, $x$ is a (uniformly random) unit modulo $n$, and we proceed to apply the order-finding routine from \cite[p.~137]{KLM07a} to get an output $o$ which is either a number or the string \enquote{FAIL}, and is equal to $\ord(x)$ with probability at least $0.399$. If $o$ is \enquote{FAIL}, we repeat this routine on the same value of $x$ until we get an output that is actually a number (only $O(1)$ repetitions needed by expectancy). Then, if $o$ is \emph{not} even, we abandon this value of $x$ and choose $y$ anew (because we want $2\mid\ord(x)$, and even if $o$ may not be equal to $\ord(x)$, it is an integer multiple of $\ord(x)$, as was noted above) until we get an $x$ such that either $\gcd(x,n)>1$ or the associated alleged multiplicative order $o\in\IN^+$ is even. This, too, only requires an expected number of $O(1)$ attempts, because for a randomly selected $x\in(\IZ/n\IZ)^{\ast}$, the order of $x$ is even with probability at least $1/2$. We then compute $z:=x^{o/2}\bmod{n}$, taking $O(\log^{2+o(1)}{n})$ bit operations by statement (6) of Lemma \ref{complexitiesLem}. Because $n$ is not a prime power, we have $\gcd(z-1,n)>1$ with probability at least $1/2$, so after expectedly $O(1)$ more tries, we will indeed have found a nontrivial factorization of $n$. Taking into account the complexity of the order-finding routine from \cite[p.~137]{KLM07a} which we analyzed above, this process expectedly takes $O(\log^{6+o(1)}{n})$ bit operations, $O(\log^{2+o(1)}{n})$ elementary quantum gates, and $O(1)$ conversions to and from $O(\log{n})$-bit strings.

Let us now return to our problem of factoring the odd positive integer $m'=m/2^{\nu_2(m)}$. Through iteratively applying the factor-finding routine we just described, which needs to be applied $O(\log{m})$ times, statement (1) follows (the complexity of some necessary deterministic post-processing, such as adding the exponents of primes appearing in multiple obtained factors, is clearly subsumed under $O(\log^{7+o(1)}{m})$).

For statement (2), we assume given a modulus $m$ and a unit $x\in(\IZ/m\IZ)^{\ast}$, and we wish to give a \emph{Las Vegas} algorithm that computes $\ord(x)$. For this, we first factor $m$, using the $m$-bounded Las Vegas dual complexity $(\log^{7+o(1)}{m},\log^{3+o(1)}{m},\log{m})$ by statement (1). Knowing the factorization of $m$ allows us to compute the Euler totient function value $\phi(m)$ and a factorization thereof within the same $m$-bounded Las Vegas dual complexity. Now, for each prime $p\mid\phi(m)$, we can work out $\nu_p(\ord(x))$ as the smallest $v_p\in\{0,1,\ldots,\nu_p(\phi(m))\}$ such that
\[
x^{p^{v_p}\phi(m)_{p'}}\equiv1\Mod{m},
\]
where $\phi(m)_{p'}:=\phi(m)/p^{\nu_p(\phi(m))}$\phantomsection\label{not215}. More specifically, we can use a binary search for finding $\nu_p(\ord(x))$, which according to statement (6) of Lemma \ref{complexitiesLem} results in a cost of
\[
O(\log\log{m}\cdot\log^{2+o(1)}{m})=O(\log^{2+o(1)}{m})
\]
bit operations for finding $\nu_p(\ord(x))$ for a single $p$, hence of $O(\log^{3+o(1)}{m})$ bit operations for finding all of these valuations. Finally, we compute $\ord(x)$ itself as the product of all prime powers $p^{\nu_p(\ord(x))}$ where $p$ ranges over the prime divisors of $\phi(m)\in\ord(x)\IZ$. This takes another $O(\log^{3+o(1)}{m})$ bit operations, thus proving statement (2).


For the proofs of statements (3) and (4), we need some preparations again. In \cite[Section 7.4]{KLM07a}, a general approach for computing discrete logarithms, working for elements chosen from any black-box group $G$ with a unique encoding of each element, is discussed. For given $x,y\in G$ such that $y=x^t$ for some $t\in\IZ$ and the order of $x$ in $G$ is known, this approach returns with high probability the unique $t\in\{0,1,\ldots,\ord(x)-1\}$ such that $y=x^t$, which is called the \emph{discrete logarithm (in $G$) of $y$ with base $x$}\phantomsection\label{term78}, written $\log_x(y)$\phantomsection\label{not216}. However, if $y$ is \emph{not} a power of $x$ in $G$, it seems that this approach does not provide a means of confirming this \emph{with certainty}, as is required for the Las Vegas algorithms we desire. This means that in addition to the discrete logarithm algorithm from \cite[Section 7.4]{KLM07a}, we need a Las Vegas routine for checking whether $y$ is a power of $x$ in the first place. We analyze these algorithms one after the other, starting with the routine for computing $\log_x(y)$ in case $y$ is a power of $x$ for $G=(\IZ/m\IZ)^{\ast}$ or $G=\IF_q^{\ast}$. As in Subsection \ref{subsec2P4}, we extend the notation $\log_x(y)$ to arbitrary $x,y\in G$ by setting $\log_x(y):=\infty$ if $y$ is \emph{not} a power of $x$.

An important observation is that the discrete logarithm algorithm described in \cite[p.~144]{KLM07a} only works if $\ord(x)$ is a prime (see \cite[second paragraph after formula (7.4.2.) on p.~143]{KLM07a}). For the general case, we follow \enquote{Method 1} from \cite[pp.~244f., starting after Corollary A.2.2]{KLM07a}. We start by setting $m_0:=1$ and $\rem_0:=0$. Then, certainly, $\log_x(y)\equiv\rem_0\Mod{m_0}$. The aim is to recursively define integers $m_1\mid m_2\mid\cdots\mid m_N=\ord(x)$ and $\rem_1,\rem_2,\ldots,\rem_N$ with $\rem_N=\log_x(y)$ such that $\log_x(y)\equiv\rem_k\Mod{m_k}$ throughout. As noted in \cite[p.~244, right after Corollary A.2.2]{KLM07a}, running the algorithm from \cite[p.~144]{KLM07a} allows us to work out $m_{k+1}$ and $\rem_{k+1}$ from $m_k<\ord(x)$ and $\rem_k$ with high probability (and no risk of getting incorrect values for them, only \enquote{FAIL}). In each step, this involves
\begin{itemize}
\item $O(1)$ arithmetic operations covered in statements (1)--(8) of Lemma \ref{complexitiesLem}, which account for $O(\log^{2+o(1)}{m})$ bit operations if $G=(\IZ/m\IZ)^{\ast}$, respectively for $O(\log^{2+o(1)}{q})$ bit operations if $G=\IF_q^{\ast}$;
\item $O(\log^{2+o(1)}{m})$, respectively $O(\log^{2+o(1)}{q})$, elementary quantum gates; and
\item $O(1)$ conversions to and from $O(\log{m})$-bit, respectively $O(\log{q})$-bit, strings.
\end{itemize}
As noted in \cite[p.~245]{KLM07a}, the number $N$ of iterations of this loop is in $O(\log{\ord(x)})$, and thus in $O(\log{m})$, respectively $O(\log{q})$. Therefore, we can compute $\log_x(y)$ in case it is not $\infty$ and $\ord(x)$ is known using the following $m$-bounded, respectively $q$-bounded, Las Vegas dual complexity:
\begin{itemize}
\item $(\log^{3+o(1)}{m},\log^{3+o(1)}{m},\log{m})$ if $G=(\IZ/m\IZ)^{\ast}$;
\item $(\log^{3+o(1)}{q},\log^{3+o(1)}{q},\log{q})$ if $G=\IF_q^{\ast}$.
\end{itemize}
This concludes our preparation for the proofs of statements (3) and (4).

For statement (3), we assume that $x,y\in(\IZ/m\IZ)^{\ast}$ are given. In order to compute $\log_x(y)$, we first check whether $y$ is a power of $x$ in the first place. We start by factoring $m$, which takes $m$-bounded Las Vegas dual complexity
\[
(\log^{7+o(1)}{m},\log^{3+o(1)}{m},\log{m})
\]
by statement (1). For a prime divisor $p$ of $m$, we set $v_p:=\nu_p(m)$ and $m_p:=p^{v_p}$\phantomsection\label{not217}. We wish to compute $\log_x^{(m_p)}(y)$, the discrete logarithm modulo $m_p$ of $y$ with base $x$, for each prime $p\mid m$, and in the following two paragraphs, we describe how to do so. We use the notation $\ord_n(z)$\phantomsection\label{not218} to denote the multiplicative order of $z$ modulo $n$.

First, we assume that $p>2$. We compute $\ord_{m_p}(x)$ and $\ord_{m_p}(y)$, taking $m_p$-bounded Las Vegas dual complexity $(\log^{7+o(1)}{m_p},\log^{3+o(1)}{m_p},\log{m_p})$ by statement (2). Because the unit group $(\IZ/m_p\IZ)^{\ast}$ is cyclic, we have that $y$ is a power of $x$ modulo $m_p$ if and only if $\ord_{m_p}(y)$ divides $\ord_{m_p}(x)$. By statement (3) of Lemma \ref{complexitiesLem}, it only takes $O(\log^{1+o(1)}{m_p})$ bit operations to check this. If this divisibility does \emph{not} hold, then $\log_x^{(m_p)}(y)$ is $\infty$, and so is $\log_x^{(m)}(y)$, so we are done. Otherwise, we compute $\log_x^{(m_p)}(y)$ using the Las Vegas routine from \cite[p.~144 and Appendix A.2]{KLM07a}, which takes $m_p$-bounded Las Vegas dual complexity $(\log^{3+o(1)}{m_p},\log^{3+o(1)}{m_p},\log{m_p})$; we note that at this point, we do know $\ord_{m_p}(x)$ because it was computed beforehand.

Now we assume that $p=2$. We proceed in a similar manner to when $p>2$, namely by first checking whether $\log_x^{(m_2)}(y)$ is $\infty$ and, if not, computing its precise integer value at the cost of an $m_2$-bounded Las Vegas dual complexity of
\[
(\log^{3+o(1)}{m_2},\log^{3+o(1)}{m_2},\log{m_2}),
\]
or $(\log^{7+o(1)}{m_2},\log^{3+o(1)}{m_2},\log{m_2})$ if $\ord_{m_2}(x)$ has not been computed at that point. Checking whether $\log_x^{(m_2)}(y)=\infty$ is a bit more complicated than for $p>2$, though, because $(\IZ/m_2\IZ)^{\ast}$ is not necessarily cyclic. We may assume that $m_2>2$ (otherwise, $\log_x^{(m_2)}(y)$ is simply equal to $1$), and we distinguish some cases.
\begin{itemize}
\item If $x\equiv y\equiv1\Mod{4}$, which only takes $O(1)$ bit operations to check because $x$ and $y$ are given in binary, then $x$ and $y$ both lie in the cyclic subgroup of $(\IZ/m_2\IZ)^{\ast}$ generated by the unit $5$ (which is equal to the unit $1$ if $m_2=4$). Therefore, just as for $p>2$, we have that $y$ is a power of $x$ if and only if $\ord_{m_2}(y)\mid\ord_{m_2}(x)$.
\item If $x\equiv1\Mod{4}$ and $y\equiv3\Mod{4}$, then $y$ cannot be a power of $x$ modulo $m_2$.
\item If $x\equiv3\Mod{4}$ and $y\equiv1\Mod{4}$, then $y$ is a power of $x$ modulo $m_2$ if and only if $y$ is a power of $x^2$ modulo $m_2$. Because $x^2\equiv1\Mod{4}$, we conclude that $y$ is a power of $x$ modulo $m_2$ if and only if $\ord_{m_2}(y)\mid\ord_{m_2}(x^2)$.
\item If $x\equiv y\equiv3\Mod{4}$, then $y$ is a power of $x$ modulo $m_2$ if and only if $-y$ is a power with odd exponent of $-x$ modulo $m_2$. Therefore, in order to check whether $y$ is a power of $x$ modulo $m_2$, we first check whether $\ord_{m_2}(-y)\mid\ord_{m_2}(-x)$. If not, then $y$ is certainly \emph{not} a power of $x$ modulo $m_2$. Otherwise, we compute $\log_{-x}^{(m_2)}(-y)$ and check whether it is odd.
\end{itemize}
In summary, since
\[
\sum_{p\mid m}{\log^k{m_p}}\in O(\log^k{m})
\]
for all real exponents $k\geq1$, we conclude that computing $\log_x^{(m_p)}(y)$ for each prime $p\mid m$ takes $m$-bounded Las Vegas dual complexity $(\log^{7+o(1)}{m},\log^{3+o(1)}{m},\log{m})$ for all $p$ together. As noted above, if any of these \enquote{primary discrete logarithms} is $\infty$, then $y$ is \emph{not} a power of $x$ modulo $m$, i.e., $\log_x^{(m)}(y)=\infty$, and we are done. Otherwise, noting that for each integer $t$, we have
\[
y\equiv x^t\Mod{m}\text{ if and only if }y\equiv x^t\Mod{m_p}\text{ for all primes }p\mid m,
\]
we find that $\log_x^{(m)}(y)<\infty$ if and only if the system of congruences
\[
t\equiv\log_x^{(m_p)}(y)\Mod{\ord_{m_p}(x)}\text{ for all primes }p\mid m
\]
in the variable $t$ is consistent, in which case $\log_x^{(m)}(y)$ is its unique solution between $0$ and $\lcm\left(\left\{\ord_{m_p}(x): p\mid m\right\}\right)-1=\ord_m(x)-1$. We already computed the right-hand sides and moduli of this system of congruences, except possibly $\ord_{m_2}(x)$, which takes $m$-bounded Las Vegas dual complexity $(\log^{7+o(1)}{m},\log^{3+o(1)}{m},\log{m})$ to compute. Following that, we can check the consistency of this system using the equivalence of statements (1) and (2) in Proposition \ref{sCongProp}, which requires $O(\log{m})$ subtractions, gcd computations and divisions of integers less than $m$, hence can be done using $O(\log^{2+o(1)}{m})$ bit operations. Should the system be consistent, the integer value of $\log_x^{(m)}(y)$ may be determined by solving the system (which is deterministic and can certainly be done with $O(\log^{7+o(1)}{m})$ bit operations, but we do not go into the details of this here), or by computing $\ord_m(x)$ and running the aforementioned routine from \cite[p.~144 and Appendix A.2]{KLM07a} for another $m$-bounded Las Vegas dual complexity of $(\log^{7+o(1)}{m},\log^{3+o(1)}{m},\log{m})$. This concludes the proof of statement (3).

For statement (4), we assume that $x,y\in\IF_q^{\ast}$ are given. In accordance with the two formats (specified at the beginning of this section) in which generalized cyclotomic mappings may be given, we consider two distinct versions of the computational problem of finding $\log_x(y)$:
\begin{itemize}
\item version 1: $x,y$ are given as powers of a common, unspecified primitive element $\omega$ of $\IF_q$;
\item version 2: a primitive irreducible polynomial $P(T)$ over $\IZ/p\IZ$ of degree $\log_p{q}$ is known, and $\IF_q$ is to be viewed as $(\IZ/p\IZ)[T]/(P(T))$, with $x,y$ given as elements of this quotient ring in standard form (i.e., as polynomials over $\IZ/p\IZ$ in the variable $T$ and of degree less than $\log_p{q}$).
\end{itemize}
In either scenario, we show that one can check whether $y$ is a power of $x$ and, if so, work out the integer value of $\log_x(y)$ using $q$-bounded Las Vegas dual complexity $(\log^{3+o(1)}{q},\log^{3+o(1)}{q},\log{q})$ altogether. Indeed, we note that for any given primitive element $\omega\in\IF_q$, the function $\log_{\omega}:\IF_q^{\ast}\rightarrow\IZ/(q-1)\IZ$ is a group isomorphism. In the first version of the computational problem, an (unspecified) value for $\omega$ was already fixed, and in the second version, we set $\omega:=T$ (viewed as an element of $(\IZ/p\IZ)[T]/(P(T))$). In either case, we can work out $\log_{\omega}(x)$ and $\log_{\omega}(y)$; in the first scenario, $x$ and $y$ are literally given as powers of $\omega$, and in the second scenario, we apply the routine from \cite[p.~144 and Appendix A.2]{KLM07a} to compute $\log_{\omega}(x)$ and $\log_{\omega}(y)$, which works because $x$ and $y$ are powers of $\omega$ and we know that $\ord(\omega)=q-1$. In either case, those two discrete logarithms can be computed with $q$-bounded Las Vegas dual complexity $(\log^{3+o(1)}{q},\log^{3+o(1)}{q},\log{q})$ (in fact, in the first scenario, they only need to be copied from the input, requiring $O(\log{q})$ bit operations only).

After finding $\log_{\omega}(x)$ and $\log_{\omega}(y)$, the problem is reduced to checking whether $\log_{\omega}(y)$ is a multiple of $\log_{\omega}(x)$ in $\IZ/(q-1)\IZ$, i.e., whether
\begin{equation}\label{logarithmizedEq}
\log_{\omega}(y)\equiv k\cdot\log_{\omega}(x)\Mod{q-1}
\end{equation}
for some $k\in\IZ$. But this is the case if and only if $\gcd(\log_{\omega}(x),q-1)\mid\log_{\omega}(y)$, which can be checked using $O(\log^{1+o(1)}{q})$ bit operations. If so, then $\log_x(y)$ is the unique solution $k\in\{0,1,\ldots,q-2\}$ of congruence (\ref{logarithmizedEq}), that is, modulo $q-1$
\[
\log_x(y)=\frac{\log_{\omega}(y)}{\gcd(\log_{\omega}(x),q-1)}\cdot\inv_{\frac{q-1}{\gcd(\log_{\omega}(x),q-1)}}\left(\frac{\log_{\omega}(x)}{\gcd(\log_{\omega}(x),q-1)}\right),
\]
which can be evaluated with a final batch of $O(\log^{1+o(1)}{q})$ bit operations. This concludes the proof of statement (4).

For statement (5), we first aim to find a primitive root modulo $p$. A polynomial-time probabilistic classical algorithm for doing so is \cite[Algorithm 1]{DD06a}, which is a refinement of an earlier algorithm by Bach \cite{Bac97a}, itself based on Itoh's idea of using partial factorizations of $p-1$ to find a primitive root with high probability. We follow this approach, but since we can factor $p-1$ completely, our situation is easier. We start by factoring $p-1$, taking $p$-bounded Las Vegas dual complexity $(\log^{7+o(1)}{p},\log^{3+o(1)}{p},\log{p})$ by statement (1). Say $p-1=\prod_{j=1}^K{p_j^{v_j}}$ is the said factorization. For each prime divisor $p_j$ of $p-1$, we wish to find a unit $y_j\in(\IZ/p\IZ)^{\ast}=\{1,2,\ldots,p-1\}$ such that $p_j^{v_j}$ divides $\ord_p(y_j)$. Equivalently, $y_j$ should \emph{not} be a $p_j$-th power in $\IZ/p\IZ$. The proportion of units that satisfy this is $1-\frac{1}{p_j}\geq\frac{1}{2}$, so if we pick $y_j\in(\IZ/p\IZ)^{\ast}$ at random, then check whether $y_j^{(p-1)/p_j}\not\equiv1\Mod{p}$, it only takes an expected number of $O(1)$ tries until we succeed at finding $y_j$. As in the proof of statement (1), we perform this random drawing of $y_j$ using a simple quantum circuit consisting of a $(\lfloor\log_2(p-2)\rfloor+1)$-qubit Hadamard gate. This means that the expected $p$-bounded Las Vegas dual complexity of finding $y_j$ is $(\log^{2+o(1)}{p},\log{p},1)$ for a single $j$, and $(\log^{3+o(1)}{p},\log^2{p},\log{p})$ for all $j=1,2,\ldots,K$ together. Once the $y_j$ have been found, the unit
\[
\rfrak:=\prod_{j=1}^K{y_j^{(p-1)/p_j^{v_j}}}\in(\IZ/p\IZ)^{\ast},
\]
which can be computed with an additional $O(\log^{3+o(1)}{p})$ bit operations, is a primitive root modulo $p$. Indeed, the $j$-th factor in this product has order $p_j^{v_j}$, and because the numbers $p_j^{v_j}$ are pairwise coprime and the group $(\IZ/p\IZ)^{\ast}$ is abelian, this entails that $\ord_p(\rfrak)=\prod_{j=1}^K{p_j^{v_j}}=p-1$.

From $\rfrak$, it is not difficult to construct a primitive root $\rfrak^+$ modulo $p^k$. Indeed, if $k=1$, we just set $\rfrak^+:=\rfrak$. Moreover, a primitive root modulo $p^2$ is a primitive root modulo $p^k$ for each $k\geq2$, and either $\rfrak$ or $\rfrak+p$ is a primitive root modulo $p^2$. Hence, if $k>1$, we simply check whether $\rfrak^{p-1}\equiv1\Mod{p^2}$, taking $O(\log^{2+o(1)}{p^2})=O(\log^{2+o(1)}{p})$ bit operations. If so, we set $\rfrak^+:=\rfrak+p$, otherwise we set $\rfrak^+:=\rfrak$. This concludes the proof of statement (5) and of Lemma \ref{complexitiesLem2} as a whole.
\end{proof}

Our first application of Lemma \ref{complexitiesLem2} is the following aforementioned result on converting a query complexity into a Las Vegas dual complexity.

\begin{lemmma}\label{LVComplexityLem}
Let $\Lfrak$ be an algorithmic problem, and let $y,y_1,\ldots,y_n$ be non-negative real parameters associated with the admissible inputs for $\Lfrak$ such that
\[
y(\vec{\xfrak})\leq h(y_1(\vec{\xfrak}),\ldots,y_n(\vec{\xfrak}))
\]
for all $\vec{\xfrak}\in\Lfrak_{\inn}$, where $h$ is a fixed function $\left[0,\infty\right)^n\rightarrow\left[0,\infty\right)$. Moreover, let
\[
\vec{\Ccal}^{\qry}=(\Ccal_{\class},\Ccal_{\fdl},\Ccal_{\mdl},\Ccal_{\mord},\Ccal_{\prt})
\]
be a $y$-bounded query complexity of $\Lfrak$ with respect to $y_1,\ldots,y_n$. Then
\[
\vec{\Ccal}=(\Ccal'_{\class},\Ccal_{\quant},\Ccal_{\conv}),
\]
with $\Ccal'_{\class},\Ccal_{\quant},\Ccal_{\conv}:\left[0,\infty\right)^n\rightarrow\left[0,\infty\right)$ as defined below, is a $y$-bounded Las Vegas dual complexity of $\Lfrak$. We write $\vec{z}$ shorthand for $(z_1,z_2,\ldots,z_n)\in\left[0,\infty\right)^n$.
\begin{align*}
&\Ccal'_{\class}(\vec{z}):=\Ccal_{\class}(\vec{z})+\log^{7+o(1)}(h(\vec{z}))\cdot(\Ccal_{\mdl}(\vec{z})+\Ccal_{\mord}(\vec{z})+\Ccal_{\prt}(\vec{z})) \\
&+\log^{3+o(1)}(h(\vec{z}))\Ccal_{\fdl}(\vec{z}); \\
&\Ccal_{\quant}(\vec{z}):=\log^{3+o(1)}(h(\vec{z}))\cdot(\Ccal_{\mdl}(\vec{z})+\Ccal_{\fdl}(\vec{z})+\Ccal_{\mord}(\vec{z})+\Ccal_{\prt}(\vec{z})); \\
&\Ccal_{\conv}(\vec{z}):=\log(h(\vec{z}))\cdot(\Ccal_{\mdl}(\vec{z})+\Ccal_{\fdl}(\vec{z})+\Ccal_{\mord}(\vec{z})+\Ccal_{\prt}(\vec{z})).
\end{align*}
\end{lemmma}

\begin{proof}
This follows easily from Lemma \ref{complexitiesLem2} and the definitions of the involved concepts. For example, when computing $\Ccal'_{\class}(\vec{z})$, we not only have to take into account the bit operations spent outside the special queries, which are represented by the summand $\Ccal_{\class}(\vec{z})$, but also those coming from the queries, using the algorithms discussed in the proofs of Lemma \ref{complexitiesLem2} to fulfill those queries. For example, on input $\vec{\xfrak}$, each of the $O(\Ccal_{\fdl}(y_1(\vec{\xfrak}),\ldots,y_n(\vec{\xfrak})))$ finite field discrete logarithm queries needed in the course of computing an admissible output for $\vec{\xfrak}$ is about computing a discrete logarithm in a finite field of size at most $h(y_1(\vec{\xfrak}),\ldots,y_n(\vec{\xfrak}))$. Therefore, by Lemma \ref{complexitiesLem2}(4), these \enquote{fdl queries} together account for
\[
O(\Ccal_{\fdl}(y_1(\vec{\xfrak}),\ldots,y_n(\vec{\xfrak}))\cdot\log^{3+o(1)}(h(y_1(\vec{\xfrak}),\ldots,y_n(\vec{\xfrak}))))
\]
bit operations, whence the inclusion of the summand $\Ccal_{\fdl}(\vec{z})\cdot\log^{3+o(1)}(h(\vec{z}))$ in the definition of $\Ccal'_{\class}(\vec{z})$. For the other three kinds of queries, one can use Lemma \ref{complexitiesLem2} together with the fact that
\[
\sum_{p\mid m}{\log^t\left(p^{\nu_p(m)}\right)}\in O(\log^t{m})
\]
for each positive integer $m$ and each real exponent $t\geq1$. For example, for a \enquote{prt query}, we need to find a primitive root modulo $p^{\nu_p(m)}$ for each odd prime $p$ dividing $m$. We do so by first factoring $m$, then applying the algorithm from Lemma \ref{complexitiesLem2}(5). By Lemma \ref{complexitiesLem2}(1,5), the number of bit operations needed in the process is in
\begin{align*}
&O\left(\log^{7+o(1)}{m}+\sum_{2<p\mid m}{\log^{7+o(1)}{p}}\right)\subseteq O\left(\log^{7+o(1)}{m}+\sum_{p\mid m}{\log^{7+o(1)}\left(p^{\nu_p(m)}\right)}\right) \\
=\,&O(\log^{7+o(1)}{m})\subseteq O(\log^{7+o(1)}{y(\vec{\xfrak})})\subseteq O(\log^{7+o(1)}{h(y_1(\vec{\xfrak}),\ldots,y_n(\vec{\xfrak}))}),
\end{align*}
which also subsumes the cost of computing $p^{\nu_p(m)}$ from $p$ and $\nu_p(m)$ for all $p$ by Lemma \ref{complexitiesLem}(6). The number of elementary quantum gates, respectively of bit-qubit conversions, needed in the process may be dealt with analogously. Moreover, an analogous approach works for \enquote{mdl queries} and \enquote{mord queries}, where one also needs to factor $m$ (see Lemma \ref{complexitiesLem2}(1)) and (due to the upper bound of $y(\vec{\xfrak})^2$ on $m$ from Definition \ref{complexitiesDef}(4,a)) ends up with an argument of $h(y_1(\vec{\xfrak}),\ldots,y_n(\vec{\xfrak}))^2$ in the logarithm power, but this may be replaced by $h(y_1(\vec{\xfrak}),\ldots,y_n(\vec{\xfrak}))$ without changing the $O$-class of the overall expression.
\end{proof}

In view of Lemma \ref{LVComplexityLem}, we mostly work with query complexities from here on, only converting them to Las Vegas dual complexities in some main results. In order to solve the three algorithmic problems on index $d$ generalized cyclotomic mappings $f$ from the beginning of this section using the theory developed in this paper, we first need to compute the induced function $\overline{f}:\{0,1,\ldots,d\}\rightarrow\{0,1,\ldots,d\}$ and, for each $i\in\{0,1,\ldots,d-1\}$ such that the coefficient $a_i$ in the cyclotomic form (\ref{cyclotomicFormEq}) of $f$ is non-zero, we need to compute the affine map $A_i$ of $\IZ/s\IZ$ that encodes the restriction $f_{\mid C_i}:C_i\rightarrow C_{\overline{f}(i)}$ under the identification of $C_j$ with $\IZ/s\IZ$ via the bijection $\iota_j$ described in our introduction. Our next goal is to analyze the query complexity of these tasks.

\begin{propposition}\label{faiProp}
Given $f$, one can compute the induced function $\overline{f}$ and the associated affine maps $A_i$ with $q$-bounded query complexity
\[
(d\log^{1+o(1)}{q},d,0,0,0).
\]
\end{propposition}

\begin{proof}
With regard to $\overline{f}$, we know that $\overline{f}(d)=d$, so only the values $\overline{f}(i)$ for $i\in\{0,1,\ldots,d-1\}$ need to be computed. These are $O(d)$ cases. By our discussion in the introduction, we have $\overline{f}(i)=(e_i+r_ii)\bmod{d}$, where $e_i=\log_{\omega}(a_i)$, and by our assumptions from the beginning of this section, this discrete logarithm is either directly specified with $a_i$, or we compute it with a (finite) field discrete logarithm (fdl) query. After computing $e_i$, it takes another $O(\log^{1+o(1)}{q})$ bit operations by Lemma \ref{complexitiesLem}(1,3) to evaluate $(e_i+r_ii)\bmod{d}$ and thus compute $\overline{f}(i)$. In total, a $q$-bounded query complexity of computing $\overline{f}$ is
\[
(d\log^{1+o(1)}{q},d,0,0,0).
\]
Once $\overline{f}$ has been determined, the computation of the $A_i$ is easy; for each of the $O(d)$ values of $i$ in question, we note that $A_i(x)=\alpha_ix+\beta_i$ for all $x\in\IZ/s\IZ$, where $\alpha_i,\beta_i\in\IZ/s\IZ$ are constants. Computing $A_i$ just means computing $\alpha_i$ and $\beta_i$, and by the discussion in our introduction, we have
\[
\alpha_i=r_i,\quad \beta_i=\frac{e_i+r_ii-\overline{f}(i)}{d}\bmod{s}.
\]
We can directly read off $r_i$ from the definition (\ref{cyclotomicFormEq}) of $f$, and computing $\beta_i$ takes $O(\log^{1+o(1)}{q})$ bit operations by Lemma \ref{complexitiesLem}(1,3). Therefore, computing all $\alpha_i$ and $\beta_i$ after $\overline{f}$ has been worked out takes $O(d\log^{1+o(1)}{q})$ bit operations, and the result follows.
\end{proof}

With regard to Problem 3 from the beginning of this section, we note that solving this problem efficiently provides us with a quick understanding of each given connected component of $\Gamma_f$. While it would be more desirable to have an efficient algorithm that achieves a \emph{global} understanding, of the isomorphism types of all connected components of $\Gamma_f$, it is not even clear what the output of such an algorithm would look like. In Definition \ref{necklaceListDef}, we introduce the concept of a tree necklace list, which is a way to list the isomorphism types of all connected components of $\Gamma_f$ with their multiplicities. While such a tree necklace list is a compact encoding of the isomorphism type of $\Gamma_f$ in \emph{some} cases (e.g., the ones considered in Subsubsections \ref{subsubsec5P3P2} and \ref{subsubsec5P3P3}), it is not clear whether that is always the case; see also the discussion after Remark \ref{necklaceListRem}.

The following main result of this section provides $q$-bounded query and Las Vegas dual complexities of the three algorithmic problems from the beginning of this section. We recall that $\mpe(q-1)=\max_{p\mid q-1}{\nu_p(q-1)}$ denotes the maximum exponent of a prime in the prime factorization of $q-1$.

\begin{theoremm}\label{complexitiesTheo}
The following hold with regard to the three algorithmic problems from the beginning of this section.
\begin{enumerate}
\item Problem 1 has $q$-bounded query complexity
\[
(d^2\log^2{d}+d^2\log^{1+o(1)}{q}+d\log^{2+o(1)}{q},d,d\log{q},d,1)
\]
and $q$-bounded Las Vegas dual complexity
\[
(d^2\log^2{d}+d^2\log^{1+o(1)}{q}+d\log^{8+o(1)}{q},d\log^{4+o(1)}{q},d\log^2{q}).
\]
\item Problem 2 has $q$-bounded query complexity
\[
(d^3\mpe(q-1)2^{(3d^2+d)\mpe(q-1)+2d}\log^{1+o(1)}{q},d,0,0,0)
\]
and $q$-bounded Las Vegas dual complexity
\[
(d^3\mpe(q-1)2^{(3d^2+d)\mpe(q-1)+2d}\log^{1+o(1)}{q}+d\log^{3+o(1)}{q},d\log^{3+o(1)}{q},d\log{q}).
\]
Moreover, the computed tree-partition register can be chosen such that its underlying recursive tree description list is of length in $O(\min\{d2^{d^2\mpe(q-1)+d},q\})$, with each tree description from the list being itself a list of length in
\[
O(\min\{d2^{d^2\mpe(q-1)+d},q\}),
\]
each entry of which is an ordered pair of bit length in $O(\log{q})$.
\item Problem 3 has $q$-bounded query complexity
\begin{align*}
(&d\log^{2+o(1)}{q}+d^3\mpe(q-1)\log^{1+o(1)}{q}+d^3\mpe(q-1)2^{d^2\mpe(q-1)+d}\log{q}, \\
&d,d^3\mpe(q-1),d^3\mpe(q-1),0)
\end{align*}
and $q$-bounded Las Vegas dual complexity
\begin{align*}
(&d^3\mpe(q-1)\log^{7+o(1)}{q}+d^3\mpe(q-1)2^{d^2\mpe(q-1)+d}\log{q}, \\
&d^3\mpe(q-1)\log^{3+o(1)}{q},d^3\mpe(q-1)\log{q}).
\end{align*}
\end{enumerate}
\end{theoremm}

We prove Theorem \ref{complexitiesTheo} in Subsection \ref{subsec5P2}. Before that, we make some more comments on the complexities of the three algorithmic problems in Theorem \ref{complexitiesTheo}. A glaring question is how the inclusion of the parameter $\mpe(q-1)$ in statements (2) and (3) affects the complexity. At first glance, this seems rather bad, because generally $\mpe(q-1)\leq\lfloor\log_2(q-1)\sim\log_2{q}$, and this bound is attained whenever $q$ is a Fermat prime. Since in statements (2) and (3) of Theorem \ref{complexitiesTheo}, $\mpe(q-1)$ occurs in the exponent of a power with base $2$, this means that in the worst case, the given complexities for Problems 2 and 3 are exponential in the input length (which lies in $O(d\log{q})$) for fixed $d$.

That being said, it turns out that \enquote{most of the time}, $\mpe(q-1)$ is actually bounded from above by a suitably large constant, as the following result states. This result and its proof was kindly pointed out by MathOverflow user \enquote{Dr.~Pi} in a response to a question posted by the first author on MathOverflow\footnote{see \url{https://mathoverflow.net/questions/436134/average-value-of-the-prime-omega-function-omega-on-predecessors-of-prime-powe}}.

\begin{propposition}\label{mpeAvProp}
There is an absolute constant $c_{\mpe}>0$ such that for all $x\geq2$, one has
\[
\left(\sum_{q\leq x}{1}\right)^{-1}\sum_{q\leq x}{\mpe(q-1)}\leq c_{\mpe}
\]
where the variable $q$ ranges over prime powers. In particular, the following hold.
\begin{enumerate}
\item For each $\epsilon>0$, there is a constant $c_{\epsilon}>0$ such that for all prime powers $q$ except an asymptotic fraction of less than $\epsilon$, one has $\mpe(q-1)<c_{\epsilon}$.
\item Let $h:\left[0,\infty\right)\rightarrow\left[0,\infty\right)$ be a function such that $h(x)\to\infty$ as $x\to\infty$. Then for asymptotically almost all prime powers $q$, one has $\mpe(q-1)\leq h(q)$.
\end{enumerate}
\end{propposition}

\begin{proof}
We start by observing that the number of proper (i.e., non-prime) prime powers up to $x$ is asymptotically equivalent to $2x^{1/2}/\log{x}$. Indeed, the number of prime \emph{squares} up to $x$ is
\[
\pi\left(x^{1/2}\right)\sim\frac{x^{1/2}}{\log\left(x^{1/2}\right)}=\frac{2x^{1/2}}{\log{x}}.
\]
Moreover, a prime power $p^k\leq x$ with $k\geq3$ satisfies $k\leq\lfloor\log_2{x}\rfloor$, and for each fixed $k$, the number of such prime powers is at most $x^{1/k}\leq x^{1/3}$. Hence the number of all proper prime powers up to $x$ is
\[
\pi\left(x^{1/2}\right)+O\left(x^{1/3}\log{x}\right)\sim\frac{2x^{1/2}}{\log{x}}.
\]
This entails the following two things.
\begin{enumerate}
\item The number $\sum_{q\leq x}{1}$ of \emph{all} prime powers up to $x$ is asymptotically equivalent to $x/\log{x}$, same as $\pi(x)$.
\item In the sum $\sum_{q\leq x}{\mpe(q-1)}$, the total contribution stemming from \emph{proper} prime powers is at most
\[
O\left(\log{x}\cdot\frac{2x^{1/2}}{\log{x}}\right)=O\left(x^{1/2}\right)\subseteq o\left(\sum_{q\leq x}{1}\right).
\]
We may thus focus on the contribution $\sum_{p\leq x}{\mpe(p-1)}$ stemming from primes.
\end{enumerate}
For each $v=1,2,\ldots,\lfloor\log_2{x}\rfloor$, we give a $O$-bound on the number of primes $p\leq x$ with $\mpe(p-1)=v$. For $v=1$, we use the trivial bound
\[
O(\pi(x))=O\left(\frac{x}{\log{x}}\right)=O\left(\frac{x}{2\log{x}}\right).
\]
Now we assume that $v\geq 2$. In order to derive a bound for such $v$, we use the Brun-Titchmarsh Theorem in its stronger form proved by Montgomery and Vaughan \cite[Theorem 2]{MV73a}. This result states that for $\afrak\in\IN^+$, $\bfrak\in\IZ$ and each real $x>\afrak$, the number of primes $p\leq x$ with $p\equiv\afrak\Mod{\bfrak}$ is at most
\[
\frac{2x}{\phi(\afrak)\log(x/\afrak)}.
\]

Now, a prime $p\leq x$ with $\mpe(p-1)=v$ is congruent to $1$ modulo $\pfrak^v$ for some prime $\pfrak<x^{1/v}$. If $\pfrak^v\leq x^{1/2}$ is fixed, then the Brun-Titchmarsh Theorem implies that the number of primes $p\leq x$ with $p\equiv1\Mod{\pfrak^v}$ is at most
\[
\frac{2x}{\phi(\pfrak^v)\log(x/\pfrak^v)}\leq \frac{4x}{\pfrak^{v-1}(\pfrak-1)\log{x}}\in O\left(\frac{x}{\pfrak^v\log{x}}\right).
\]
On the other hand, if $x^{1/2}<\pfrak^v<x$, then the number of primes $p\leq x$ with $p\equiv1\Mod{\pfrak^v}$ is at most $x/\pfrak^v<x^{1/2}$. It follows that the number of primes $p\leq x$ with $\mpe(p-1)=v=2$ is in
\[
O\left(\frac{x}{\log{x}}\cdot\sum_{\pfrak\leq x^{1/4}}{\frac{4}{\pfrak(\pfrak-1)}}+x^{1/2}\cdot\frac{2x^{1/2}}{\log{x}}\right)=O\left(\frac{x}{\log{x}}\right)=O\left(\frac{x}{4\log{x}}\right)
\]
and, if $v>2$, that number is in
\[
O\left(\frac{x}{\log{x}}\cdot\sum_{\pfrak\leq x^{1/(2v)}}{\frac{4}{\pfrak^{v-1}(\pfrak-1)}}+x^{1/2}\cdot x^{1/3}\right)=O\left(\frac{x}{2^v\log{x}}\right).
\]
In summary, we have shown that for each $v=1,2,\ldots,\lfloor\log_2{x}\rfloor$, the number of primes $p\leq x$ with $\mpe(p-1)=v$ is in $O\left(x/(2^v\log{x})\right)$, and so
\[
\sum_{p\leq x}{\mpe(p-1)}\in O\left(\sum_{v=1}^{\lfloor\log_2{x}\rfloor}{\frac{vx}{2^v\log{x}}}\right)=O\left(\frac{x}{\log{x}}\sum_{v=1}^{\lfloor\log_2{x}\rfloor}{\frac{v}{2^v}}\right)=O\left(\frac{x}{\log{x}}\right),
\]
whence
\begin{align*}
\left(\sum_{q\leq x}{1}\right)^{-1}\sum_{q\leq x}{\mpe(q-1)} &\sim \left(\sum_{p\leq x}{1}\right)^{-1}\sum_{p\leq x}{\mpe(p-1)} \\
&\in O\left(\frac{1}{x/\log{x}}\cdot\frac{x}{\log{x}}\right)=O(1),
\end{align*}
which is the main statement of this proposition. The first \enquote{In particular} statement follows readily from this by observing that the quantity $\left(\sum_{q\leq x}{1}\right)^{-1}\sum_{q\leq x}{\mpe(q-1)}$ is the average value of $\mpe(q-1)$ on prime powers $q\leq x$. Finally, the second \enquote{In particular} statement is an easy consequence of the first.
\end{proof}

For applications, finite fields of characteristic $2$ are of particular interest. The authors are not aware of any rigorous results concerning the asymptotic behavior of $\mpe(2^v-1)$ as $v\to\infty$, but in Table \ref{mpeMersenneTable}, we provide an overview of the maximum and average values of $\mpe(2^v-1)$ for $v\in\{1,2,\ldots,K\}$, where $K\in\{100,200,\ldots,1000\}$. This was obtained using GAP \cite{GAP4} and information from the Cunningham project \cite{Wag22a}. More specifically, GAP appeared to have difficulties factoring $2^v-1$ for $v\in\{929,947,991\}$, but a quick consultation of the Cunningham factorization tables reveals that $\mpe(2^v-1)=1$ for each of these three values of $v$.

\begin{longtable}[h]{|c|c|c|}\hline
$K$ & $\max\{\mpe(2^v-1):1\leq v\leq K\}$ & $K^{-1}\sum_{v=1}^K{\mpe(2^v-1)}$ rounded \\ \hline
100 & 4 & 1.28 \\ \hline
200 & 5 & 1.325 \\ \hline
300 & 5 & 1.3267 \\ \hline
400 & 5 & 1.3325 \\ \hline
500 & 6 & 1.336 \\ \hline
600 & 6 & 1.3383 \\ \hline
700 & 6 & 1.3371 \\ \hline
800 & 6 & 1.34 \\ \hline
900 & 6 & 1.3389 \\ \hline
1000 & 6 & 1.341 \\ \hline
\caption{Maximum and average values of $\mpe(2^v-1)$.}
\label{mpeMersenneTable}
\end{longtable}

Based on this, we conjecture that the average value of $\mpe(2^v-1)$ for $1\leq v\leq x$ is always less than $2$, see Conjecture \ref{mpeConj}.

\subsection{Proof of Theorem \ref{complexitiesTheo}}\label{subsec5P2}

We give detailed descriptions of algorithms for solving Problems 1--3 and analyze their query complexities (their Las Vegas dual complexities specified in Theorem \ref{complexitiesTheo} follow readily using Lemma \ref{LVComplexityLem}). The amount of details we give should make it easy to implement these algorithms. In all three cases, we first need to compute $\overline{f}$ and the affine maps $A_i$, which takes query complexity $(d\log^{1+o(1)}{q},d,0,0,0)$ by Proposition \ref{faiProp}. We assume that this has already been done at the start of the discussion of each individual problem. Whenever a positive integer needs to be factored, we subsume this under an mdl query (counted in the third entry of a query complexity).

\subsubsection{Proof of statement (1)}\label{subsubsec5P2P1}

Quite a lot of notations are needed to provide this algorithm in full detail. For the reader's convenience, we print the names of those notations that are newly introduced in this discussion, as well as those of a few notations introduced earlier but rarely used since, in underlined form at the beginning of the respective paragraph where they first appear in this discussion. For the reading flow, these underlined parts need to be ignored. Of course, these notations are also catalogued in Table \ref{termNotTable} in the Appendix.

\underline{$\overline{\Lcal}$.} Before computing $\Lcal$ properly, we need to compute a CRL-list $\overline{\Lcal}$ for $\overline{f}$. Because $\overline{f}$ can be any function $\{0,1,\ldots,d\}\rightarrow\{0,1,\ldots,d\}$ with $\overline{f}(d)=d$, we use a general, brute-force algorithm for this, assuming that the indices $i\in\{0,1,\ldots,d\}$ are processed as non-negative integers in binary representation, with $\lfloor\log{d}\rfloor+1$ digits each. Going\phantomsection\label{beginningRef} through them to compute $\im(\overline{f})$ as a (not necessarily repetition-free) list of its elements takes $O(d\log{d})$ bit operations, and checking whether $\im(\overline{f})=\{0,1,\ldots,d\}$ uses $O(d\log^2{d})$ bit operations, for sorting $\im(\overline{f})$ according to Lemma \ref{complexitiesLem}(10), then checking that there are no repeated entries. If $\im(\overline{f})\not=\{0,1,\ldots,d\}$, we continue by computing $\im(\overline{f}^2)$ and checking whether $\im(\overline{f})=\im(\overline{f}^2)$, using another $O(d\log^2{d})$ bit operations, and so on. After $O(d)$ iterations of this, and thus after $O(d^2\log^2{d})$ bit operations in total, we have found the periodic point set $\per(\overline{f})$ as the first iterated image $\im(\overline{f}^n)$ such that $\im(\overline{f}^{n+1})=\im(\overline{f}^n)$. Finally, it takes another $O(d\log{d})$ bit operations to compute $\overline{\Lcal}$ through iteration of $\overline{f}$ on $\per(\overline{f})$ by brute force, and in the process, we can actually store each cycle of $\overline{f}$ in full, which will be useful shortly. In total, these computations require $O(d^2\log^2{d})$ bit operations.

\underline{$\Lcal_i, U_i, \para_i, Y_i$.} To compute the desired parametrization of $\Lcal$, we go through the elements $(i,\ell)\in\overline{\Lcal}$, with associated $\overline{f}$-cycle $(i_0,i_1,\ldots,i_{\ell-1})$, and compute a parametrization of a CRL-list $\Lcal_i$ of the restriction $f_{\mid U_i}$, where $U_i=\bigcup_{t=0}^{\ell-1}{C_{i_t}}$. This works because by Proposition \ref{crlListConstructProp}, $\Lcal$ is simply the (disjoint) union of those $\Lcal_i$. Specifically, we compute a formula that defines a bijective function $\para_i:Y_i\rightarrow\Lcal_i$\phantomsection\label{not220}\phantomsection\label{not221}, where $Y_i$ is a \enquote{simple} set depending on $i$. For $(i,\ell)=(d,1)$, which is dealt with outside the loop for the other pairs $(i,\ell)$, we have $\Lcal_d=\{(0_{\IF_q},1)\}$, and we set $Y_d:=\{(\emptyset,\emptyset)\}$ (to conform with the format the sets $Y_i$ for $i<d$ have -- each of the two $\emptyset$ is to be viewed as an empty tuple) and define $\para_d(\emptyset,\emptyset):=(0_{\IF_q},1)$. This only takes $O(\log{d})$ bit operations (not $O(1)$, because the index $d$ on the left-hand side of the definition needs to be spelled out).

\underline{$\rfrak_p, \Acal_i, \overline{\alpha}_i, \overline{\beta}_i, \para'_i, \Lcal'_i, \Pfrak_i$.} Next, we factor $s=(q-1)/d$ in a single $q$-bounded modular discrete logarithm (mdl) query (we remind the reader that we subsume factorizations under mdl queries). We also find a primitive root $\rfrak_p$\phantomsection\label{not236} modulo $p^{\nu_p(s)}$ for each odd prime divisor $p$ of $s$ using a single $q$-bounded primitive root (prt) query. Following that, we loop over the elements $(i,\ell)\in\overline{\Lcal}$ with $i<d$, and for each of them, we do the following. We compute
\[
\Acal_i:=A_{i_0}A_{i_1}\cdots A_{i_{\ell-1}},\quad \Acal_i(z)=\overline{\alpha}_iz+\overline{\beta}_i.
\]
This\phantomsection\label{not222} takes $O(d)$ multiplications of already computed affine maps, each of which costs $O(\log^{1+o(1)}{q})$ bit operations by Lemma \ref{complexitiesLem}(1,3) using the formula
\[
(z\mapsto \alpha z+\beta)(z\mapsto \alpha'z+\beta')=(z\mapsto \alpha\alpha'z+\alpha'\beta+\beta').
\]
Hence, in total, the computation of $\Acal_i$ takes $O(d\log^{1+o(1)}{q})$ bit operations. Our next goal is to compute a parametrization $\para'_i:Y_i\rightarrow\Lcal'_i$\phantomsection\label{not223} of a CRL-list $\Lcal'_i$ for $\Acal_i$, from which $\para_i:Y_i\rightarrow\Lcal_i$ is obtained simply by stretching all second entries (cycle lengths) of images of $\para'_i$ by the factor $\ell$. As preparations for an upcoming loop over the prime divisors of $s$, we initialize $\Pfrak_i:=\emptyset$ (ultimately, $\Pfrak_i$\phantomsection\label{not228} will be a list of those prime divisors of $s$ that do \emph{not} divide $\overline{\alpha}_i$). We also compute $\ord_{p^{\nu_p(s)}}(\overline{\alpha}_i)$ for each prime divisor $p$ of $s$, requiring a single $q$-bounded multiplicative order (mord) query.

\underline{$\kappa_p, \overline{\Acal}_{i,p}, \para'_{i,p}, Y_{i,p}, \Lcal'_{i,p}$} Next, we loop over the prime divisors $p$ of $s$, and for each of them, we do the following. First, we check whether $p\mid\overline{\alpha}_i$, and if so, we skip to the next value of $p$. Otherwise, we add $p$ to $\Pfrak_i$, then read off $\kappa_p:=\nu_p(s)$\phantomsection\label{not229} and $p^{\kappa_p}$ from the factorization of $s$ computed earlier. Following that, we compute $\overline{\Acal}_{i,p}:=\Acal_i\bmod{p^{\kappa_p}}$\phantomsection\label{not224} (that is, we compute $\overline{\alpha}_i\bmod{p^{\kappa_p}}$ and $\overline{\beta}_i\bmod{p^{\kappa_p}}$), which takes $O(\log^{1+o(1)}{q})$ bit operations by Lemma \ref{complexitiesLem}(3). We note that since $p$ does not divide $\overline{\alpha}_i$, the function $\overline{\Acal}_{i,p}$ is an affine \emph{permutation} of $\IZ/p^{\kappa_p}\IZ$, and from our Table \ref{crlListPrimaryTable}, we can read off a compact parametrization $\para'_{i,p}:Y_{i,p}\rightarrow\Lcal'_{i,p}$\phantomsection\label{not225}\phantomsection\label{not226}\phantomsection\label{not227} of a CRL-list $\Lcal'_{i,p}$ of $\overline{\Acal}_{i,p}$ in which all specified cycle lengths are fully factored. The details of this are given in Table \ref{crlListParTable} below; each numbered row of that table corresponds to the case with the same number in Table \ref{crlListPrimaryTable}. The following paragraph introduces some more notation, which is used in Table \ref{crlListPrimaryTable} and needs to be computed before one is able to print a description of $\para'_{i,p}$.

\underline{$\kappa_{i,p}, \nfrak_p, \pfrak_{p,k}, v_{p,k}, v'_{i,p,k}, v'_{i,p}, v''_{i,2}, \rfrak_p, \ffrak_{i,p}$.} Recalling that $\nu_p^{(v)}(m):=\min\{\nu_p(m),v\}$, we set
\[
\kappa_{i,p}:=\nu_p^{(\kappa_p)}(\overline{\beta}_i)=\nu_p^{(\kappa_p)}(\overline{\beta}_i\bmod{p^{\kappa_p}}),
\]
which\phantomsection\label{not230} can be computed using $O(\kappa_p)=O(\log{p^{\kappa_p}})$ integer divisions by $p$, resulting in a bit operation cost of $O(\log^{2+o(1)}{p^{\kappa_p}})$. If $p>2$, we next compute factorizations of $p-1$ and of $\ord_{p^{\kappa_p}}(\overline{\alpha}_i)$ using $2\in O(1)$ mord queries. We spell these factorizations out as follows:
\[
p-1=\prod_{k=1}^{\nfrak_p}{\pfrak_{p,k}^{v_{p,k}}}
\]
and\phantomsection\label{not231}\phantomsection\label{not232}\phantomsection\label{not233}
\[
\ord_{p^{\kappa_p}}(\overline{\alpha}_i)=\prod_{k=1}^{\nfrak_p}{\pfrak_{p,k}^{v'_{i,p,k}}}\cdot p^{v'_{i,p}}.
\]
We\phantomsection\label{not234}\phantomsection\label{not235} note that some of the exponents $v'_{i,p,k}$ or $v'_{i,p}$ may be $0$. On the other hand, if $p=2$ (where $\nfrak_p=0$), we write $\ord_{2^{\kappa_2}}(\overline{\alpha}_i)=2^{v'_{i,2}}$, which matches with the notation for $p>2$ above, and $\ord_{2^{\kappa_2}}(-\overline{\alpha}_i)=2^{v''_{i,2}}$\phantomsection\label{not235P5}. Usually, $v''_{i,2}=v'_{i,2}$ as $\ord_{2^{\kappa_2}}(\overline{\alpha}_i)=\ord_{2^{\kappa_2}}(-\overline{\alpha}_i)$, but if $\overline{\alpha}_i\equiv\pm1\Mod{2^{\kappa_2}}$, then $v'_{i,2}\in\{0,1\}$ and $v''_{i,2}=1-v_{i,2}$. Finally, regardless of whether or not $p>2$, we check whether $\overline{\Acal}_{i,p}$ has a fixed point, i.e., whether
\[
\gcd(\overline{\alpha}_i-1,p^{\kappa_p})=\gcd((\overline{\alpha}_i\bmod{p^{\kappa_p}})-1,p^{\kappa_p}) \mid \overline{\beta_i}\bmod{p^{\kappa_p}},
\]
which can be done using $O(\log^{1+o(1)}{p^{\kappa_p}})$ bit operations. We store this information, and whenever $\overline{\Acal}_{i,p}$ has a fixed point, we compute one, denoted by $\ffrak_{i,p}$\phantomsection\label{not237}, via the formula in Proposition \ref{crlListPrimaryProp}, taking another $O(\log^{1+o(1)}{p^{\kappa_p}})$ bit operations.

\underline{$u, u', \vec{u}$.} We are now ready to give the tabular definition of the bijective parametrization $\para'_{i,p}:Y_{i,p}\rightarrow\Lcal'_{i,p}$ of a CRL-list $\Lcal'_{i,p}$ of $\Acal'_{i,p}$. We note that the set $Y_{i,p}$ always has one of the following two forms, which will be important later on.
\begin{itemize}
\item $Y_{i,p}$ is an integer interval, a general element of which is denoted by $u$\phantomsection\label{not238}; or
\item the elements of $Y_{i,p}$ are pairs $(u,u')$ of integers, where $u$ ranges over an integer interval, and for each fixed value of $u$, the second entry $u'$\phantomsection\label{not239} also ranges over an integer interval.
\end{itemize}
To have a uniform notation, we may also denote an element of $Y_{i,p}$ by $\vec{u}$\phantomsection\label{not240} in either case. For example, to derive the formulas in the first case of Table \ref{crlListParTable}, we apply the first case in Table \ref{crlListPrimaryTable}, with $v:=\kappa_p$, $t:=u$, $j:=u'$, $a:=\overline{\alpha}_i$ and $b:=\overline{\beta}_i$. Then the range for $u$ is clear from the the range for $t$ in Table \ref{crlListPrimaryTable}. Concerning the asserted range for $u'$, we note that
\[
\frac{\phi(p^{\kappa_p})}{\ord_{p^{\kappa_p}}(\overline{\alpha}_i)}=\prod_{k=1}^{\nfrak_p}{\pfrak_{p,k}^{v_{p,k}-v'_{i,p,k}}}\cdot p^{\kappa_p-1-v'_{i,p}}
\]
and
\[
\phi(p^{\kappa_p-u})=\left((p-1)p^{\kappa_p-u-1}\right)^{\delta_{[u<\kappa_p]}}=\left(\prod_{k=1}^{\nfrak_p}{\pfrak_{p,k}^{v_{p,k}}}\cdot p^{\kappa_p-u-1}\right)^{\delta_{[u<\kappa_p]}},
\]
from which it can be deduced that
\[
\gcd\left(\frac{\phi(p^{\kappa_p})}{\ord_{p^{\kappa_p}}(\overline{\alpha}_i)},\phi(p^{\kappa_p-u})\right)=\left(\prod_{k=1}^{\nfrak_p}{\pfrak_{p,k}^{v_{p,k}-v'_{i,p,k}}}\right)^{\delta_{[u<\kappa_p]}}\cdot p^{\min\{\kappa_p-1-v'_{i,p},\delta_{[u<\kappa_p]}(\kappa_p-u-1)\}},
\]
as required. Finally, the formula for the cycle lengths (second entries of $\para'_{i,p}(\vec{u})$) in case 1 holds because
\begin{align*}
&\frac{\phi(p^{\kappa_p-u})}{\gcd\left(\frac{\phi(p^{\kappa_p})}{\ord_{p^{\kappa_p}}(\overline{\alpha}_i)},\phi(p^{\kappa_p-u})\right)}=\frac{\left(\prod_{k=1}^{\nfrak_p}{\pfrak_{p,k}^{v_{p,k}}}\cdot p^{\kappa_p-u-1}\right)^{\delta_{[u<\kappa_p]}}}{\left(\prod_{k=1}^{\nfrak_p}{\pfrak_{p,k}^{v_{p,k}-v'_{i,p,k}}}\right)^{\delta_{[u<\kappa_p]}}\cdot p^{\min\{\kappa_p-1-v'_{i,p},\delta_{[u<\kappa_p]}(\kappa_p-u-1)\}}} \\
&=\left(\prod_{k=1}^{\nfrak_p}{\pfrak_{p,k}^{v'_{i,p,k}}}\right)^{\delta_{[u<\kappa_p]}}\cdot p^{\delta_{[u<\kappa_p]}(\kappa_p-u-1-\min\{\kappa_p-1-v'_{i,p},\kappa_p-u-1\})} \\
&=\left(\prod_{k=1}^{\nfrak_p}{\pfrak_{p,k}^{v'_{i,p,k}}}\right)^{\delta_{[u<\kappa_p]}}\cdot p^{\delta_{[u<\kappa_p]}\max\{v'_{i,p}-u,0\}}
\end{align*}
The other cases in Table \ref{crlListParTable} can be dealt with analogously. To prevent confusion among readers, we note that in Cases 7 and 8 of Table \ref{crlListPrimaryTable}, the specified CRL-list consists of several disjoint parts with different formulas. Because we want $u$ to range over an integer interval, these have been slightly rearranged and \enquote{glued together} here. For example, in Case 7 here, the ranges $\{0,1,\ldots,\kappa_2-2\}$ and $\{-\kappa_2+1,-\kappa_2+2,\ldots,-1\}$ for $u$ correspond, respectively, to the parts with representative elements $5^j2^t+\ffrak$ and $-5^j2^t+\ffrak$ in Case 7 of Table \ref{crlListPrimaryTable}. In the latter of the two segments, the range for $u$ is not equal to the corresponding range for $t$ in Table \ref{crlListPrimaryTable}, which explains the variable substitution $u\rightarrow -u-1$ although the corresponding formulas for cycle lengths in Table \ref{crlListPrimaryTable} are the same. In Case 10, we use $1$ (rather than $0$) as an admissible value for $u$, also to turn the range for $u$ into an integer interval.
\begin{longtable}[h]{|c|c|c|c|}\hline
\centering
No. & $u$ & $u'$ & $\para'_{i,p}(\vec{u})$ \\ \hline
1 & $0,\ldots,\kappa_p$ & \thead{$0,\ldots,\left(\prod_{k=1}^{\nfrak_p}{\pfrak_{p,k}^{v_{p,k}-v'_{i,p,k}}}\right)^{\delta_{[u<\kappa_p]}}\cdot$ \\ $p^{\min\{\kappa_p-1-v'_{i,p},\delta_{[u<\kappa_p]}(\kappa_p-u-1)}\}-1$} & \thead{$(\rfrak_p^{u'}p^u+\ffrak_{i,p},$ \\ $\left(\prod_{k=1}^{\nfrak_p}{\pfrak_{p,k}^{v'_{i,p,k}}}\right)^{\delta_{[u<\kappa_p]}}\cdot$ \\ $p^{\delta_{[u<\kappa_p]}\max(v'_{i,p}-u,0)})$} \\ \hline
2 & $0,\ldots,p^{\kappa_{i,p}}-1$ & n/a & $(u,p^{\kappa_p-\kappa_{i,p}})$ \\ \hline
3 & $0,\ldots,2^{\kappa_{i,2}}-1$ & n/a & $(u,2^{\kappa_2-\kappa_{i,2}})$ \\ \hline
4 & $0,1,2$ & n/a & $(u,-u^2+2u+1)$ \\ \hline
5 & $0,1,2$ & n/a & $(2^u-1,\frac{1}{2}u^2-\frac{3}{2}u+2)$ \\ \hline
6 & $0,1$ & n/a & $(2u,2)$ \\ \hline
7 & $-\kappa_2,\ldots,\kappa_2-1$ & \thead{if $u\in\{-\kappa_2,\kappa_2-1\}$: $0$; \\ otherwise: $0,\ldots,2^{\kappa_2-2-\max(v'_{i,2},u)}-1$} & \thead{if $u=\kappa_2-1$: $(\ffrak_{i,2},1)$; \\ if $u=-\kappa_2$: $(2^{\kappa_2-1}+\ffrak_{i,2},1)$; \\ if $0\leq u<\kappa_2-1$: \\ $(5^{u'}2^u+\ffrak_{i,2},2^{\max(v'_{i,2}-u,0)})$; \\ if $-\kappa_2<u<0$: \\ $(-5^{u'}2^u+\ffrak_{i,2},2^{\max(v'_{i,2}+u+1,0)})$.} \\ \hline
8 & $-v''_{i,2},\ldots,v''_{i,2}$ & \thead{if $u=v''_{i,2}$: $0,\ldots,2^{\kappa_2-v''_{i,2}-1}$; \\ otherwise: $0,\ldots,2^{\kappa_2-v''_{i,2}-2}-1$.} & \thead{if $u=v''_{i,2}$ and $u'\in\{0,2^{\kappa_2-v''_{i,2}-1}\}$: \\ $(u'2^{v''_{i,2}}+\ffrak_{i,2},1)$; \\ if $u=v''_{i,2}$ and $0<u'<2^{\kappa_2-v''_{i,2}-1}$: \\ $(u'2^{v''_{i,2}}+\ffrak_{i,2},2)$; \\ if $0\leq u<v''_{i,2}$: \\ $(5^{u'}2^u+\ffrak_{i,2},2^{v''_{i,2}-u})$; \\ if $u<0$: \\ $(5^{u'}2^{-u-1}+\ffrak_{i,2},2^{v''_{i,2}+u+1})$.} \\ \hline
9 & $0,\ldots,2^{\kappa_{i,2}}-1$ & n/a & $(u,2^{\kappa_2-\kappa_{i,2}})$ \\ \hline
10 & $1,\ldots,2^{\kappa_2-v''_{i,2}-1}$ & n/a & \thead{if $u=1$: $(0,2^{v''_{i,2}+1})$; \\ otherwise: $(\overline{\beta}_iu,2^{v''_{i,2}+1})$.} \\ \hline
\caption{Explicit parametrizations of CRL-lists of affine permutations of finite primary cyclic groups.}
\label{crlListParTable}
\end{longtable}
Our algorithm prints and stores the parametric description of $\para'_{i,p}(\vec{u})$ for all primes $p\mid s$ with $p\nmid\overline{\alpha}_i$. For a given $p$, this parametric description takes $O(\log{p^{\kappa_p}})$ bits to store (as follows by observing that it takes $O(\log{n})$ bits to print the prime factorization of $n\in\IN^+$), and so all descriptions together can be stored using $O(\log{s})\subseteq O(\log{q})$ bits.

\underline{$\para''_i, \Lcal''_i, \Acal'_i, s'_i, \vec{\overline{u}}, \vec{u}_p, u_p, u'_p, \overline{Y}_i, \proj_j, \vec{r}_i(\vec{\overline{u}}), r_{i,p}(\vec{u}_p), B_{\vec{r}_i(\vec{\overline{u}})}, \Ical_{i,\vec{\overline{u}}}$.} Next, based on the parametrizations $\para'_{i,p}$ of the CRL-lists $\Lcal'_{i,p}$ of $\overline{\Acal}_{i,p}$ for $p\in\Pfrak_i$, we construct a parametrization $\para''_i:Y_i\rightarrow\Lcal''_i$\phantomsection\label{not241}\phantomsection\label{not242} of a CRL-list $\Lcal''_i$ for $\Acal'_i:=\Acal_i\bmod{\prod_{p\in\Pfrak_i}{p^{\kappa_p}}}$\phantomsection\label{not243}. We start by setting $s'_i:=\prod_{p\in\Pfrak_i}{p^{\kappa_p}}$\phantomsection\label{not244}, which takes $O(\log^{2+o(1)}{q})$ bit operations to compute, carrying out $|\Pfrak_i|\in O(\log{q})$ integer multiplications, each with a bit operation cost in $O(\log^{1+o(1)}{q})$ (we note that the powers $p^{\kappa_p}$ themselves do not need to be computed, as they are specified, alongside the pairs $(p,\kappa_p)$, in the output of the mdl query that gave the factorization of $s$). We follow the approach described at the end of Subsection \ref{subsec2P3}. More specifically, we identify $\IZ/s'_i\IZ$ with $\prod_{p\in\Pfrak_i}{\IZ/p^{\kappa_p}\IZ}$, and $\Acal'_i$ with $\bigotimes_{p\in\Pfrak_i}{\overline{\Acal}_{i,p}}$. We consider tuples $\vec{\overline{u}}=(\vec{u}_p)_{p\in\Pfrak_i}\in\prod_{p\in\Pfrak_i}{Y_{i,p}}=:\overline{Y}_i$\phantomsection\label{not245}\phantomsection\label{not246}\phantomsection\label{not247}. We can either write $\vec{u}_p=u_p$\phantomsection\label{not248} or $\vec{u}_p=(u_p,u'_p)$\phantomsection\label{not249}. For $j=1,2$, we denote by $\proj_j$\phantomsection\label{not250} the (class-sized) function that maps an ordered pair to its $j$-th entry. Associated with each parameter tuple $\vec{\overline{u}}\in \overline{Y}_i$, we have the tuple
\[
\vec{r}_i(\vec{\overline{u}}):=(r_{i,p}(\vec{u}_p))_{p\in\Pfrak_i}:=(\proj_1(\para'_{i,p}(\vec{u}_p)))_{p\in\Pfrak_i}
\]
of\phantomsection\label{not251}\phantomsection\label{not252} associated cycle representatives of the $\overline{\Acal}_{i,p}$. By our discussion at the end of Subsection \ref{subsec2P3}, these tuples $\vec{r}_i(\vec{\overline{u}})$ parametrize the blocks $B_{\vec{r}_i(\vec{\overline{u}})}$ of a certain partition of $\prod_{p\in\Pfrak_i}{\IZ/p^{\kappa_p}\IZ}$, each block of which is a union of cycles of $\Acal'_i$. In terms of $\vec{\overline{u}}$, we wish to explicitly describe a CRL-list for the restriction of $\Acal'_i$ to $B_{\vec{r}_i(\vec{\overline{u}})}$. For this, we need to exhibit an $\vec{r}_i(\vec{\overline{u}})$-admissible indexing function $\Ical_{i,\vec{\overline{u}}}$\phantomsection\label{not253} in the sense of Definition \ref{admissibleGoodDef}(1) and understand its associated set of good tuples (in the sense of Definition \ref{admissibleGoodDef}(2)).

\underline{$l_{i,p,\vec{u}_p}, l_{i,\vec{\overline{u}}}, \Pfrak'_i$.} Now, following the definition of an admissible indexing function, the domain of definition of $\Ical_{i,\vec{\overline{u}}}$ is the set of all primes that divide at least one of the component cycle lengths $l_{i,p,\vec{u}_p}:=\proj_2(\para'_{i,p}(\vec{u}_p))$\phantomsection\label{not254} for $p\in\Pfrak_i$, or, equivalently, that divide $l_{i,\vec{\overline{u}}}:=\lcm\{l_{i,p,\vec{u}_p}: p\in\Pfrak_i\}$\phantomsection\label{not255}, of which we compute a parametric definition of bit length in $O(\log^{1+o(1)}{q})$ for later use by scanning the displayed parametric factorizations of the $l_{i,p,\vec{u}_p}$, taking $O(\log^{2+o(1)}{q})$ bit operations. By definition, the domain of $\Ical_{i,\vec{\overline{u}}}$ is a subset of $\Pfrak'_i:=\Pfrak_i\cup\pi(\prod_{p\in\Pfrak_i}{(p-1)})$\phantomsection\label{not256}. We compute $\Pfrak'_i$ as a list (with $O(\log{q})$ entries), and this computation consists of $O(\log{q})$ containment checks each involving $O(1)$ copying processes of bit strings of length in $O(\log{q})$, and $O(\log{q})$ bit comparisons and scans of memory addresses each of length in $O(\log\log{q})$. Hence, we may compute $\Pfrak'_i$ using $O(\log^{2+o(1)}{q})$ bit operations. In our algorithmic approach, we treat $\Ical_{i,\vec{\overline{u}}}$ as a function whose domain of definition is all of $\Pfrak'_i$; the additional primes $\pfrak$\phantomsection\label{not257} are those which do not divide any component cycle length, hence occur with valuation $0$ in each component, and the value $\Ical_{i,\vec{\overline{u}}}(\pfrak)$ may be chosen arbitrarily in $\Pfrak_i$. We note that this change does not affect the associated notion of good tuples and ensures that the domain of $\Ical_{i,\vec{\overline{u}}}$ does not depend on $\vec{\overline{u}}$.

For each $\pfrak\in\Pfrak'_i$, the value $\Ical_{i,\vec{\overline{u}}}(\pfrak)$ is a prime $p'\in\Pfrak_i$ (thought of as an index for a component of $\vec{\overline{u}}$) such that $\nu_{\pfrak}(l_{i,p',\vec{\overline{u}}})$ is maximal among all $\nu_{\pfrak}(l_{i,p,\vec{u}_p})$ for $p\in\Pfrak_i$. We recall that we have already worked out explicit factorizations of the positive integers $l_{i,p,\vec{u}_p}$ in terms of $\vec{u}_p$ (see Table \ref{crlListParTable}). If $\pfrak\notin\Pfrak_i$, then for each $p\in\Pfrak_i$, the value of $\nu_{\pfrak}(l_{i,p,\vec{u}_p})$ is constant, not depending on $\vec{u}_p$, and a scan along the length $O(\log{q})$ parametric description, combined with comparisons of the relevant exponents $\nu_{\pfrak}(l_{i,p,\vec{u}_p})$, each of which has bit length in $O(\log\log{q})$, lets us pick a suitable value for $\Ical_{i,\vec{\overline{u}}}(\pfrak)$. For a given $\pfrak\in\Pfrak'_i\setminus\Pfrak_i$, this process requires $O(\log^{1+o(1)}{q})$ bit operations, and carrying it out for all $\pfrak\in\Pfrak'_i\setminus\Pfrak_i$ takes $O(\log^{2+o(1)}{q})$ bit operations.

\underline{$\Jcal_{i,p,k}, \mfrak_{i,p}, \Jcal'_{i,p,k}$.} We still need to discuss the approach when $\pfrak\in\Pfrak_i$. Even then, $\nu_{\pfrak}(l_{i,p,\vec{u}_p})$ does not depend on $\vec{u}_p$ unless $p=\pfrak$, in which case one of the following applies.
\begin{itemize}
\item $\nu_p(l_{i,p,\vec{\overline{u}}})$ also does not depend on $\vec{u}_p$ (see e.g.~case 2 in Table \ref{crlListParTable}), and we can compute a constant value for $\Ical_{i,\vec{\overline{u}}}(p)$ as described above.
\item $\nu_p(l_{i,p,\vec{\overline{u}}})$ does depend on $\vec{u}_p$, in the following way: It only depends on $u_p$ (not $u'_p$), and one can partition the range for $u_p$ into at most five subintervals $\Jcal_{i,p,1},\ldots,\Jcal_{i,p,\mfrak_{i,p}}$\phantomsection\label{not258}\phantomsection\label{not259} (case 8 in Table \ref{crlListParTable} does require $\mfrak_{i,p}=5$) such that in case $u_p\in \Jcal_{i,p,k}$ for a fixed $k\in\{1,\ldots,\mfrak_{i,p}\}$, the value of $\nu_p(l_{i,p,\vec{\overline{u}}})$ is either constant, or given by a linear expression in $u_p$, or given by an expression that is the maximum among a linear expression in $u_p$ and $0$. This allows us to specify a subinterval (in fact, an initial or terminal segment) $\Jcal'_{i,p,k}$\phantomsection\label{not260} of $\Jcal_{i,p,k}$ (with constant boundary points) such that $\Ical_{i,\vec{\overline{u}}}(p)$ may be chosen as $p$ if $u_p\in \Jcal'_{i,p,k}$, whereas $\Ical_{i,\vec{\overline{u}}}(p)$ must be chosen as a different constant value in $\Pfrak_i$ (the same for each $k$) if $u_p\in \Jcal_{i,p,k}\setminus\Jcal'_{i,p,k}$. For each given $p$, writing down an explicit definition of $\Ical_{i,\vec{\overline{u}}}(p)$ (which consists of a case distinction with at most two cases) requires us to scan the parametric descriptions of the component images $\para'_{i,p}(\vec{u}_p)$ and perform some low-cost computations such as additions or subtractions between exponents of primes (which are numbers of bit length in $O(\log\log{q})$). For all relevant values of $\pfrak$ together, this can be done using $O(\log^{2+o(1)}{q})$ bit operations.
\end{itemize}
We note that for each given $\pfrak\in\Pfrak'_i$, the parametric definition of $\Ical_{i,\vec{\overline{u}}}$ which we just derived has bit length in $O(\log{q})$. Therefore, and because the domain $\Pfrak'_i$ of $\Ical_{i,\vec{\overline{u}}}$ has size in $O(\log{q})$, it takes $O(\log^2{q})$ bits to store the parametric definitions of all function values of $\Ical_{i,\vec{\overline{u}}}$.

\underline{$\Ical_i, \Pfrak_{i,p,\vec{\overline{u}}}, \vec{k}, k_p, \dfrak_{i,p,\vec{\overline{u}}}, \vec{k'}, k'_p, K'_{i,\vec{\overline{u}}}$.} Before we proceed with our argument, we need to introduce another notation. For $\pfrak\in\Pfrak'_i$, if $\Ical_{i,\vec{\overline{u}}}(\pfrak)$ only assumes one distinct value as $\vec{\overline{u}}$ ranges over $\overline{Y}_i$, we set $\Ical_i(\pfrak):=\Ical_{i,\vec{\overline{u}}}(\pfrak)$ for any $\vec{\overline{u}}\in\overline{Y}_i$. On the other hand, if $\Ical_{i,\vec{\overline{u}}}(\pfrak)$ assumes two distinct values, one of which is $\pfrak$, we let $\Ical_i(\pfrak)$ be the unique element of $\{\Ical_{i,\vec{\overline{u}}}(\pfrak): \vec{\overline{u}}\in\overline{Y}_i\}$ that is distinct from $\pfrak$. This defines a function $\Ical_i:\Pfrak'_i\rightarrow\Pfrak_i$\phantomsection\label{not261} that is independent of $\vec{\overline{u}}$ and can be easily derived from the parametric definitions of the function values $\Ical_{i,\vec{\overline{u}}}(\pfrak)$ (taking $O(\log^{2+o(1)}{q})$ bit operations). Using the function $\Ical_i$, we can give the following compact parametric definition of the pre-image of a singleton subset of $\Pfrak_i$ under $\Ical_{i,\vec{\overline{u}}}$:
\begin{align}\label{PfrakipuEq}
\notag &\Pfrak_{i,p,\vec{\overline{u}}}:=\Ical^{-1}_{i,\vec{\overline{u}}}(\{p\})= \\
&\left(\Ical_i^{-1}(\{p\})\setminus\{\pfrak\in\Pfrak_i\setminus\{p\}: u_{\pfrak}\in\bigcup_{k=1}^{\mfrak_{i,\pfrak}}{\Jcal'_{i,p,k}}\}\right)\cup\{\pfrak\in\{p\}: u_{\pfrak}\in\bigcup_{k=1}^{\mfrak_{i,\pfrak}}{\Jcal'_{i,p,k}}\}.
\end{align}
We\phantomsection\label{not262} note that our algorithm is merely producing this defining formula for $\Pfrak_{i,p,\vec{\overline{u}}}$ for each $p\in\Pfrak_i$, which is harmless complexity-wise -- even when spelling $\Ical_i^{-1}(\{p\})$ out explicitly in each case, this can be done using $O(\log^{1+o(1)}{q})$ bit operations and storage space per $p$, hence $O(\log^{2+o(1)}{q})$ bit operations and storage space altogether. One could also try to provide a case-distinction definition of $\Pfrak_{i,p,\vec{\overline{u}}}$ where each case corresponds to a constant value of $\Pfrak_{i,p,\vec{\overline{u}}}$, but this breaks the complexity, as one needs to go through $2^{O(\log{q})}$ cases in general. Likewise, it is easy to check that producing each parametric definition described in the rest of this argument takes $O(\log^{2+o(1)}{q})$ bit operations if one is careful enough about how to spell those parametrizations out.

Having these explicit definitions of the pre-images $\Pfrak_{i,p,\vec{\overline{u}}}$ is important because they are needed to set up a parametrization of the $\Ical_{i,\vec{\overline{u}}}$-good tuples. We recall from above the notation $l_{i,p,\vec{u}_p}$ for the cycle length of the representative $r_{i,p}(\vec{u}_p)$ in the $p$-indexed component of $\vec{r}_i(\vec{\overline{u}})$. An $\Ical_{i,\vec{\overline{u}}}$-good tuple is a tuple $\vec{k}=(k_p)_{p\in\Pfrak_i}$\phantomsection\label{not262P5}\phantomsection\label{not262P75} with $k_p\in\IZ/l_{i,p,\vec{u}_p}\IZ=\{0,1,\ldots,l_{i,p,\vec{u}_p}-1\}$ such that $k_p$ is divisible by
\[
\dfrak_{i,p,\vec{\overline{u}}}:=\prod_{\pfrak\in \Pfrak_{i,p,\vec{\overline{u}}}}{\pfrak^{\nu_{\pfrak}(l_{i,\vec{\overline{u}}})}}.
\]
We\phantomsection\label{not263} can compute a parametric definition of $\dfrak_{i,p,\vec{\overline{u}}}$ and $l_{i,p,\vec{u}_p}/\dfrak_{i,p,\vec{\overline{u}}}$ using $O(\log^{1+o(1)}{q})$ bit operations and storage space per $p$, hence $O(\log^{2+o(1)}{q})$ bit operations and storage space altogether. Moreover, we can parametrize the set of $\Ical_{i,\vec{\overline{u}}}$-good tuples as follows\phantomsection\label{not264}\phantomsection\label{not264P5}\phantomsection\label{not265}:
\[
\Good_{\vec{r}_i(\vec{\overline{u}})}(\Ical_{i,\vec{\overline{u}}})=\{(k'_p\dfrak_{i,p,\vec{\overline{u}}})_{p\in\Pfrak_i}: \vec{k'}=(k'_p)_{p\in\Pfrak_i}\in\prod_{p\in\Pfrak_i}{\IZ/\frac{l_{i,p,\vec{u}_p}}{\dfrak_{i,p,\vec{\overline{u}}}}\IZ}=:K'_{i,\vec{\overline{u}}}\}.
\]

\underline{$r'_{i,\vec{\overline{u}}}(\vec{k'})$.} Now, for each $\vec{k'}=(k'_p)_{p\in\Pfrak_i}\in K'_{i,\vec{\overline{u}}}$ and its associated $\Ical_{i,\vec{\overline{u}}}$-good tuple $(k'_p\dfrak_{i,p,\vec{\overline{u}}})_{p\in\Pfrak_i}$, we have the cycle representative $\left(\overline{\Acal}_{i,p}^{k'_p\dfrak_{i,p,\vec{\overline{u}}}}(r_{i,p}(\vec{u}_p))\right)_{p\in\Pfrak_i}$ of $\Acal'_i$, or rather, of the permutation $\bigotimes_{p\in\Pfrak_i}{\overline{\Acal}_{i,p}}$ identified with it, in $\prod_{p\in\Pfrak_i}{\IZ/p^{\kappa_p}\IZ}$. Literally, $\Acal'_i$ is defined as an affine permutation of $\IZ/s'_i\IZ=\IZ/\prod_{p\in\Pfrak_i}{p^{\kappa_p}}\IZ$. Therefore, the actual cycle representative of $\Acal'_i$ associated with $\vec{k'}$ is
\[
r'_{i,\vec{\overline{u}}}(\vec{k'}):=\sum_{p\in\Pfrak_i}{\overline{\Acal}_{i,p}^{k'_p\dfrak_{i,p,\vec{\overline{u}}}}(r_{i,p}(\vec{u}_p))\frac{s'_i}{p^{\kappa_p}}\inv_{p^{\kappa_p}}\left(\frac{s'_i}{p^{\kappa_p}}\right)},
\]
the\phantomsection\label{not266} unique element of $\IZ/s'_i\IZ$ that is congruent to $\overline{\Acal}_{i,p}^{k'_p\dfrak_{i,p,\vec{\overline{u}}}}(r_{i,p}(\vec{u}_p))$ modulo $p^{\kappa_p}$ for each $p\in\Pfrak_i$. We note that the expression $\overline{\Acal}_{i,p}^{k'_p\dfrak_{i,p,\vec{\overline{u}}}}(r_{i,p}(\vec{u_p}))$ can be spelled out explicitly as follows.
\[
\overline{\Acal}_{i,p}^{k'_p\dfrak_{i,p,\vec{\overline{u}}}}(r_{i,p}(\vec{u}_p))=
\begin{cases}
r_{i,p}(\vec{u}_p)+k'_p\dfrak_{i,p,\vec{\overline{u}}}\overline{\beta}_i, & \text{if }\overline{\alpha}_i=1, \\
\overline{\alpha}_i^{k'_p\dfrak_{i,p,\vec{\overline{u}}}}r_{i,p}(\vec{u}_p)+\overline{\beta}_i\frac{\overline{\alpha}_i^{k'_p\dfrak_{i,p,\vec{\overline{u}}}}-1}{\overline{\alpha}_i-1}, & \text{otherwise},
\end{cases}
\]
where the fraction in the second case is to be understood as an integer division, but the expression as a whole represents an element of $\IZ/p^{\kappa_p}\IZ$ (one needs to identify the integer value with its reduction modulo $p^{\kappa_p}$). It takes $O(\log^{2+o(1)}{q})$ bit operations and storage space to compute and store the parametric definition of $r'_{i,\vec{\overline{u}}}$.

\underline{$L_i$.} At last, we can now provide the parametric definitions for the CRL-list $\Lcal''_i$ of $\Acal'_i$ and, subsequently, for the CRL-list $\Lcal'_i$ of $\Acal_i$. Namely, $Y_i:=\bigcup_{\vec{\overline{u}}\in\overline{Y}_i}{\left(\{\vec{\overline{u}}\}\times K'_{i,\vec{\overline{u}}}\right)}$, and for $(\vec{\overline{u}},\vec{k'})\in Y_i$, we set
\[
\para''_i(\vec{\overline{u}},\vec{k'}):=(r'_{i,\vec{\overline{u}}}(\vec{k'}),l_{i,\vec{\overline{u}}}).
\]
Then $\Lcal''_i=\{\para''_i(\vec{\overline{u}},\vec{k'}): (\vec{\overline{u}},\vec{k'})\in Y_i\}$. In order to obtain $\para'_i$ and $\Lcal'_i$, we simply need to lift the first entries of elements of $\Lcal''_i$ (images of $\para''_i$) from $\IZ/s'_i\IZ$ to $\IZ/s\IZ$ such that the reduction modulo $s/s'_i$ of each lift is the unique periodic point of $\Acal_i\bmod{(s/s'_i)}$ in $\IZ/(s/s'_i)\IZ$. By Lemma \ref{periodicCharLem}, we can compute that periodic point as follows. Let $L_i$\phantomsection\label{not267} denote the smallest non-negative integer such that
\[
\gcd\left(\alpha_i^{L_i},s\right)=\prod_{p\mid\gcd(\alpha_i,s)}{p^{\kappa_p}},
\]
which satisfies
\[
L_i=\max_{p\mid\gcd(\overline{\alpha}_i,s)}{\left\lceil\frac{\nu_p(s)}{\nu_p(\overline{\alpha}_i)}\right\rceil}=\max_{p\mid\gcd(\overline{\alpha}_i,s)}{\left\lceil\frac{\nu_p(s)}{\nu_p(\overline{\alpha}_i\bmod{p^{\nu_p(s)}})}\right\rceil}\leq\mpe(s)\in O(\log{q})
\]
and may be found by computing, for each $p\mid\gcd(\overline{\alpha}_i,s)$, the value $\overline{\alpha}_i\bmod{p^{\nu_p(s)}}$ (for $\nu_p(s)$, one consults the factorization of $s$ computed above), then finding $\nu_p(\overline{\alpha}_i\bmod{p^{\nu_p(s)}})$ with a binary search between $0$ and $\nu_p(s)$ (each step of which involves a power and a gcd computation). Altogether, this costs
\[
O\left(\log{q}\cdot\left(\log^{1+o(1)}{q}+\sum_{p\mid\gcd(\overline{\alpha}_i,s)}{\left(\log\log{q}\cdot\log^{2+o(1)}{p^{\nu_p(s)}}\right)}\right)\right)=O(\log^{2+o(1)}{q})
\]
bit operations by Lemma \ref{complexitiesLem}(6,8). The unique periodic point of $\Acal_i\bmod{(s/s'_i)}$ is the reduction of
\[
\sum_{z=0}^{L_i-1}{\overline{\alpha}_i^z\overline{\beta}_i}=
\begin{cases}
L_i\overline{\beta}_i, & \text{if }\overline{\alpha}_i=1, \\
\frac{\overline{\alpha}_i^{L_i}-1}{\overline{\alpha}_i-1}\overline{\beta}_i, & \text{otherwise},
\end{cases}
\]
modulo $s/s'_i$ and may be computed in $O(\log^{2(1+o(1))}{q})=O(\log^{2+o(1)}{q})$ bit operations using that $\overline{\alpha}_i^{L_i}$ is of bit length in $O(\log^2{q})$. We obtain the following formula for $\para'_i$ (which has the domain of definition $Y_i$, same as $\para''_i$) such that $\Lcal'_i=\{\para'_i(\vec{\overline{u}},\vec{k'}): (\vec{\overline{u}},\vec{k'})\in Y_i\}$:
\[
\para'_i(\vec{\overline{u}},\vec{k'})=\left(r'_{i,\vec{\overline{u}}}(\vec{k'})\frac{s}{s'_i}\inv_{s'_i}\left(\frac{s}{s'_i}\right)+\sum_{z=0}^{L_i-1}{\overline{\alpha}_i^z\overline{\beta}_i}s'_i\inv_{s/s'_i}(s'_i),l_{i,\vec{\overline{u}}}\right).
\]
Printing this parametric definition of $\Lcal'_i$ takes $O(\log^{2+o(1)}{q})$ bits of storage space. As mentioned before, the (bijective) parametrization $\para_i:Y_i\rightarrow\Lcal_i$ of the CRL-list $\Lcal_i$ of $f_{\mid U_i}$ can be obtained by stretching the second entries of the images of $\para'_i$ by the factor $\ell=\ell_i$ (the $\overline{f}$-cycle length of $i$). That is,
\[
\para_i(\vec{\overline{u}},\vec{k'})=\left(r'_{i,\vec{\overline{u}}}(\vec{k'})\frac{s}{s'_i}\inv_{s'_i}\left(\frac{s}{s'_i}\right)+\sum_{z=0}^{L_i-1}{\overline{\alpha}_i^z\overline{\beta}_i}s'_i\inv_{s/s'_i}(s'_i),\ell\cdot l_{i,\vec{\overline{u}}}\right)
\]
and $\Lcal_i=\{\para_i(\vec{\overline{u}},\vec{k'}): (\vec{\overline{u}},\vec{k'})\in Y_i\}$.

Finally, the expression $\para_i(\vec{\overline{u}},\vec{k'})$ where
\begin{itemize}
\item $i\in\proj_1(\overline{\Lcal})$;
\item $\vec{\overline{u}}\in\overline{Y}_i$ (with $\overline{Y}_d:=\{\emptyset\}$); and
\item $\vec{k'}\in K'_{i,\vec{\overline{u}}}$ (with $K'_{d,\emptyset}:=\{\emptyset\}$)
\end{itemize}
forms the desired bijective parametrization of a CRL-list $\Lcal$ of $f$, which can be pasted together from the results of earlier computations using $O(d\log^{2+o(1)}{q})$ bit operations.

In what follows, we conclude this subsubsection with an overview of the steps of this algorithm. At the end of the description of each step, we specify its query complexity (QC); in the case of a loop, this is obtained component-wise by computing the sum of the entries in the corresponding components of the query complexities of the iteration steps of the loop, if applicable replacing the resulting expression by a simpler one that generates the same $O$-class, and multiplying it with a $O$-bound on the number of iterations of the loop. It follows from this overview that the query complexity of Problem 1 is as specified in statement (1) of Theorem \ref{complexitiesTheo}, and the formula for the Las Vegas dual complexity follows from this and Lemma \ref{LVComplexityLem}.

\begin{denumerate}[label=\arabic*]
\item Compute the induced function $\overline{f}$ on $\{0,1,\ldots,d\}$ and the affine maps $A_i$ of $\IZ/s\IZ$.

QC: $(d\log^{1+o(1)}{q},d,0,0,0)$.
\item Compute a CRL-list $\overline{\Lcal}$ for $\overline{f}$, storing the cycles of $\overline{f}$ in full in the process.

QC: $(d^2\log^2{d},0,0,0,0)$.
\item Compute and store the parametrization $\para_d:Y_d\rightarrow\Lcal_d$ where $Y_d=\{(\emptyset,\emptyset)\}$ and $\para_d(\emptyset,\emptyset)=(0_{\IF_q},1)$.

QC: $(\log{d},0,0,0,0)$.
\item Compute and factor $s=(q-1)/d$.

QC: $(\log^{1+o(1)}{q},0,1,0,0)$.
\item Find a primitive root $\rfrak_p$ modulo $p^{\nu_p(s)}$ for each odd prime $p\mid s$.

QC: $(\log{q},0,0,0,1)$.
\item For each $(i,\ell)\in\overline{\Lcal}\setminus\{(d,1)\}$, with associated $\overline{f}$-cycle $(i_0,i_1,\ldots,i_{\ell-1})$ which was already computed in Step 2, do the following.

QC: $(d^2\log^{1+o(1)}{q}+d\log^{2+o(1)}{q},0,d\log{q},d,0)$.
\begin{denumerate}[label=6.\arabic*]
\item Compute the forward cycle product $\Acal_i=A_{i_0}A_{i_1}\cdots A_{i_{\ell-1}}: z\mapsto\overline{\alpha}_iz+\overline{\beta}_i$.

QC: $(d\log^{1+o(1)}{q},0,0,0,0)$.
\item Initialize $\Pfrak_i:=\emptyset$.

QC: $(\log{d},0,0,0,0)$.
\item Compute $\ord_{p^{\nu_p(s)}}(\overline{\alpha}_i)$ for each prime $p\mid s$.

QC: $(\log{q},0,0,1,0)$.
\item For each prime $p\mid s$, do the following.

QC: $(\log^{2+o(1)}{q},0,\log{q},0,0)$.
\begin{denumerate}[label=6.4.\arabic*]
\item Check whether $p\mid\overline{\alpha}_i$, and if not, skip to the next $p$.

QC: $(\log^{1+o(1)}{q},0,0,0,0)$.
\item Add $p$ to $\Pfrak_i$ as a new element.

QC: $(\log{q},0,0,0,0)$.
\item Read off $\kappa_p=\nu_p(s)$ and $p^{\kappa_p}$ from the factorization of $s$ computed in Step 4.

QC: $(\log{q},0,0,0,0)$.
\item Compute $\overline{\Acal}_{i,p}=\Acal_i\mod{p^{\kappa_p}}$, i.e., compute $\overline{\alpha}_i\bmod{p^{\kappa_p}}$ and $\overline{\beta}_i\bmod{p^{\kappa_p}}$.

QC: $(\log^{1+o(1)}{q},0,0,0,0)$.
\item Compute $\kappa_{i,p}:=\nu_p^{(\kappa_p)}(\overline{\beta}_i)=\nu_p^{(\kappa_p)}(\overline{\beta}_i\bmod{p^{\kappa_p}})$, using $O(\kappa_p)$ divisions by $p$ and increasing a counter.

QC: $(\log^{2+o(1)}{p^{\kappa_p}},0,0,0,0)$.
\item If $p>2$ then do the following.

QC: $(\log{q}+\log^2{p^{\kappa_p}},0,1,0,0)$.
\begin{denumerate}[label=6.4.6.\arabic*]
\item Compute factorizations of $p-1$ and of $\ord_{p^{\kappa_p}}(\overline{\alpha}_i)$:
\[
p-1=\prod_{k=1}^{\nfrak_p}{\pfrak_{p,k}^{v_{p,k}}}\text{ and }\ord_{p^{\kappa_p}}(\overline{\alpha}_i)=\prod_{k=1}^{\nfrak_p}{\pfrak_{p,k}^{v'_{i,p,k}}}\cdot p^{v'_{i,p}}.
\]
QC: $(\log{q},0,1,0,0)$.
\item Check whether $\overline{\Acal}_{i,p}$ has a fixed point and, if so, store this information and compute a fixed point $\ffrak_{i,p}$ of it according to Proposition \ref{crlListPrimaryProp}. The check can be done by testing whether $\gcd\left(\left(\overline{\alpha}_i\bmod{p^{\kappa_p}}\right)-1,p^{\kappa_p}\right)$ divides $\overline{\beta}_i\bmod{p^{\kappa_p}}$.

QC: $(\log^{1+o(1)}{p^{\kappa_p}},0,0,0,0)$.
\end{denumerate}
\item Else do the following.
\begin{denumerate}[label=6.4.7.\arabic*]
\item Compute $v'_{i,2}=\nu_2(\ord_{2^{\kappa_2}}(\overline{\alpha}_i))$ and $v''_{i,2}=\nu_2(\ord_{2^{\kappa_2}}(-\overline{\alpha}_i))$. To avoid making another mord query, we note that $v''_{i,2}=v'_{i,2}$ unless $\overline{\alpha}_i\equiv\pm1\Mod{2^{\kappa_2}}$, in which case $v''_{i,2}=1-v'_{i,2}$.

QC: $(\log^{2+o(1)}{2^{\kappa_2}},0,0,0,0)$.
\item Check whether $\overline{\Acal}_{i,2}$ has a fixed point and, if so, store this information and compute a fixed point $\ffrak_{i,2}$ of it (cf.~Step 6.4.6.2).

QC: $(\log^{1+o(1)}{2^{\kappa_2}},0,0,0,0)$.
\end{denumerate}
\item Spell out a definition of the bijective parametrization $\para'_{i,p}:Y_{i,p}\rightarrow\Lcal'_{i,p}$ of a CRL-list $\Lcal'_{i,p}$ of $\overline{\Acal}_{i,p}$ in which all specified cycle lengths are fully factored, referring to Table \ref{crlListParTable}. This requires checking which of the cases from Table \ref{crlListPrimaryTable} applies, and we stored part of the information relevant for this in Steps 6.4.6.2 and 6.4.7.2. We note that a general element of $Y_{i,p}$ is denoted by $\vec{u}_p$ and is either equal to $u_p$ or $(u_p,u'_p)$ where $u_p$ and $u'_p$ are integer parameters, with $u_p$ ranging over a fixed interval, and $u'_p$ ranging over an interval for each fixed value of $u_p$ (with explicit formulas for the interval bounds in terms of $u_p$).

QC: $(\log{p^{\kappa_p}},0,0,0,0)$.
\end{denumerate}
\item Compute $s'_i=\prod_{p\in\Pfrak_i}{p^{\kappa_p}}$.

QC: $(\log^{2+o(1)}{q},0,0,0,0)$.
\item Compute a parametric definition of $l_{i,\vec{\overline{u}}}$, the cycle length of $\Acal'_i$ (or rather, of the permutation $\bigotimes_{p\in\Pfrak_i}{\overline{\Acal}_{i,p}}$ on $\prod_{p\in\Pfrak_i}{\IZ/p^{\kappa_p}\IZ}$ identified with it) on the point $\vec{r}_i(\vec{\overline{u}})$ represented by $\vec{\overline{u}}=(\vec{u}_p)_{p\in\Pfrak_i}$. This can be done through scanning the parametric definitions of the fully factored cycle lengths
\[
l_{i,p,\vec{u}_p}=\proj_2(\para'_{i,p}(\vec{u}_p)),
\]
of which $l_{i,\vec{\overline{u}}}$ is the least common multiple, and performing low-cost operations on numbers of bit length in $O(\log\log{q})$.

QC: $(\log^{2+o(1)}{q},0,0,0,0)$.
\item Set $\Pfrak'_i:=\Pfrak_i\cup\pi(\prod_{p\in\Pfrak_i}{(p-1)})$, using $O(\log{q})$ containment checks each involving $O(1)$ copying processes of bit strings of length $O(\log{q})$, and $O(\log{q})$ bit comparisons and scans of memory addresses each of length $O(\log\log{q})$.

QC: $(\log^{2+o(1)}{q},0,0,0,0)$.
\item For $\pfrak\in\Pfrak'_i$, compute a parametric definition of the function value $\Ical_{i,\vec{\overline{u}}}(\pfrak)$ of the $\vec{r}_i(\vec{\overline{u}})$-admissible indexing function $\Ical_{i,\vec{\overline{u}}}$. This definition consists of a case distinction with at most two cases (and constant value of $\Ical_{i,\vec{\overline{u}}}(\pfrak)$ in each case). For $\pfrak\notin\Pfrak_i$, there is always only one case, and for $\pfrak\in\Pfrak_i$, the cases depend on the containment of $u_{\pfrak}$ in a union of certain intervals (at most five such intervals per $\pfrak$). Moreover, whenever there are two cases, one of them corresponds to $\Ical_{i,\vec{\overline{u}}}(\pfrak)=\pfrak$. Whenever there is only one case, set $\Ical_i(\pfrak):=\Ical_{i,\vec{\overline{u}}}(\pfrak)$ for any given $\vec{\overline{u}}$, otherwise let $\Ical_i(\pfrak)$ be the unique element of $\Ical_{i,\vec{\overline{u}}}(\pfrak)$ that is distinct from $\pfrak$.

QC: $(\log^{2+o(1)}{q},0,0,0,0)$.
\item For $p\in\Pfrak_i$, compute a parametric description of  the pre-image set $\Pfrak_{i,p,\vec{\overline{u}}}:=\Ical^{-1}_{i,\vec{\overline{u}}}(\{p\})$, using formula (\ref{PfrakipuEq}). In this parametric description, the inclusion of primes $\pfrak\in\Pfrak'_i\setminus\Pfrak_i$ in $\Pfrak_{i,p,\vec{\overline{u}}}$ is independent of $\vec{\overline{u}}$, whereas primes $p'\in\Pfrak_i$ each have a condition, in terms of a disjunction of bounds on $u_{p'}$ corresponding to the intervals mentioned in Step 6.8, for whether $p'\in\Pfrak_{i,p,\vec{\overline{u}}}$.

QC: $(\log^{2+o(1)}{q},0,0,0,0)$.
\item Based on Step 6.9, compute parametric descriptions of
\[
\dfrak_{i,p,\vec{\overline{u}}}=\prod_{\pfrak\in\Pfrak_{i,p,\vec{\overline{u}}}}{\pfrak^{\nu_{\pfrak}(l_{i,\vec{\overline{u}}})}}=\prod_{\pfrak\in\Pfrak_{i,p,\vec{\overline{u}}}}{\pfrak^{\nu_{\pfrak}(l_{i,p,\vec{u}_p})}}
\]
and $l_{i,p,\vec{\overline{u}}}/\dfrak_{i,p,\vec{\overline{u}}}$ for each $p\in\Pfrak_i$. In view of Step 6.9, this can be achieved using suitable Kronecker deltas in the exponents.

QC: $(\log^{2+o(1)}{q},0,0,0,0)$.
\item Compute the parametric description
\[
r_{i,\vec{\overline{u}}}(\vec{k'})=\sum_{p\in\Pfrak_i}{\overline{\Acal}_{i,p}^{k'_p\dfrak_{i,p,\vec{\overline{u}}}}(r_{i,p}(\vec{u}_p))\frac{s'_i}{p^{\kappa_p}}\inv_{p^{\kappa_p}}\left(\frac{s'_i}{p^{\kappa_p}}\right)}
\]
of the cycle representative $r_{i,\vec{\overline{u}}}(\vec{k'})$ of $\Acal_{i'}$ associated with the parameter tuple $(\vec{\overline{u}},\vec{k'})$ where $\vec{\overline{u}}\in\overline{Y}_i$ and
\[
\vec{k'}=(k'_p)_{p\in\Pfrak_i}\in K'_{i,\vec{\overline{u}}}=\prod_{p\in\Pfrak_i}{\IZ/(l_{i,p,\vec{u}_p}/\dfrak_{i,p,\vec{\overline{u}}})\IZ}.
\]
In this expression, the (inexplicit) affine map iterate value $\overline{\Acal}_{i,p}^{k'_p\dfrak_{i,p,\vec{\overline{u}}}}(r_{i,p}(\vec{u}_p))$ is to be substituted with the explicit formula
\[
\overline{\Acal}_{i,p}^{k'_p\dfrak_{i,p,\vec{\overline{u}}}}(r_{i,p}(\vec{u}_p))=
\begin{cases}
r_{i,p}(\vec{u}_p)+k'_p\dfrak_{i,p,\vec{\overline{u}}}\overline{\beta}_i, & \text{if }\overline{\alpha}_i=1, \\
\overline{\alpha}_i^{k'_p\dfrak_{i,p,\vec{\overline{u}}}}r_{i,p}(\vec{u}_p)+\overline{\beta}_i\frac{\overline{\alpha}_i^{k'_p\dfrak_{i,p,\vec{\overline{u}}}}-1}{\overline{\alpha}_i-1}, & \text{otherwise}.
\end{cases}
\]
QC: $(\log^{2+o(1)}{q},0,0,0,0)$.
\item Compute
\[
L_i=\max_{p\mid\gcd(\overline{\alpha}_i,s)}{\left\lceil\frac{\nu_p(s)}{\nu_p(\overline{\alpha}_i)}\right\rceil}=\max_{p\mid\gcd(\overline{\alpha}_i,s)}{\left\lceil\frac{\nu_p(s)}{\nu_p(\overline{\alpha}_i\bmod{p^{\nu_p(s)}}})\right\rceil},
\]
the smallest non-negative integer such that $\gcd(\overline{\alpha}_i^{L_i},s)=\prod_{p\mid\gcd(\overline{\alpha}_i,s)}{p^{\kappa_p}}$. To do so, for each $p\mid\gcd(\overline{\alpha}_i,s)$, compute $\overline{\alpha}_i\bmod{p^{\nu_p(s)}}$ with a division, then find $\nu_p(\overline{\alpha}_i\bmod{p^{\nu_p(s)}})$ with a binary search between $0$ and $\nu_p(s)$.

QC: $(\log^{2+o(1)}{q},0,0,0,0)$
\item Compute the parametric description
\[
\para_i(\vec{\overline{u}},\vec{k'})=\left(r'_{i,\vec{\overline{u}}}(\vec{k'})\frac{s}{s'_i}\inv_{s'_i}\left(\frac{s}{s'_i}\right)+\sum_{z=0}^{L_i-1}{\overline{\alpha}_i^z\overline{\beta}_i}s'_i\inv_{s/s'_i}(s'_i),\ell\cdot l_{i,\vec{\overline{u}}}\right)
\]
of the element of $\Lcal_i$ (a CRL-list of $f_{\mid U_i}$ where $U_i=\bigcup_{t=0}^{\ell-1}{C_{i_t}}$) associated with $(\vec{\overline{u}},\vec{k'})$.

QC: $(\log^{2+o(1)}{q},0,0,0,0)$.
\end{denumerate}
\item Output the parametric description $\para_i(\vec{\overline{u}},\vec{k'})$ of the element of $\Lcal$ (a CRL-list of $f$) associated with $(i,\vec{\overline{u}},\vec{k'})$ where $i\in\proj_1(\overline{\Lcal})$, $\vec{\overline{u}}\in\overline{Y}_i$ and $\vec{k'}\in K'_{i,\vec{\overline{u}}}$ (with the convention that $\overline{Y}_d=\{\emptyset\}$ and $K'_{d,\emptyset}=\{\emptyset\}$), then halt.

QC: $(d\log^{2+o(1)}{q},0,0,0,0)$.
\end{denumerate}

\subsubsection{Proof of statement (2)}\label{subsubsec5P2P2}

We note that the only part of our algorithm for Problem 2 where a quantum computer is required is at the beginning, when $\overline{f}$ and the $A_i$ need to be computed. The rest of the algorithm, which we describe henceforth, uses bit operations only.

In addition to computing $\overline{f}$ and the $A_i$, and as at the beginning of the proof of statement (1), we need to compute the different \enquote{layers} of indices $i\in\{0,1,\ldots,d-1\}$ according to their containment in the iterated images of $\overline{f}$, requiring $O(d^2\log^2{d})$ bit operations overall.

We follow the approach from Subsection \ref{subsec3P3}, proceeding in three successive steps.

\underline{Step 1: transient $i$.} We aim to compute
\begin{itemize}
\item $\Zcal_i=\Pcal_i$ for all $\overline{f}$-transient $i\in\{0,1,\ldots,d-1\}$;
\item a list of rooted tree descriptions $(\Dfrak_0,\Dfrak_1,\ldots,\Dfrak_{N_1})$ that covers all isomorphism types of rooted trees of the form $\Tree_i(\Pcal_i,\vec{\nu}^{(\Pcal_i)})$ for $\overline{f}$-transient $i\in\{0,1,\ldots,d-1\}$ and logical sign tuples $\vec{\nu}^{(\Pcal_i)}\in\{\emptyset,\neg\}^{m_i}$ such that $\Bcal(\Pcal_i,\vec{\nu}^{(\Pcal_i)})\not=\emptyset$; and
\item the corresponding logical sign tuple data $S_{n,i}$.
\end{itemize}

At any given point in the algorithm (not just in this step), the set of all $n\in\IN_0$ for which $\Dfrak_n$ is defined is an initial segment $\{0,1,\ldots,N'\}$ of $\IN_0$, denoted by $\Ncal$\phantomsection\label{not268} (a variable that gets updated throughout the process).

In order to carry out the computations listed above, we proceed by recursion on $h_i=\height(\Tree_{\Gamma_{\overline{f}}}(i))$. First, we assume that $h_i=0$. Then, in accordance with Subsection \ref{subsec3P3}, we set $\Pcal_i:=\Pfrak(\emptyset)$ for all such $i$, introduce the trivial rooted tree isomorphism type $\Ifrak_0$ via its description $\Dfrak_0:=\emptyset$, and set $S_{0,i}:=\{\emptyset\}$ (with $\emptyset$ to be viewed as the empty logical sign tuple), while all $S_{n,i}$ for values $n>0$ introduced later will be defined as the empty set. This settles the case $h_i=0$.

Now we assume that $h_i=h>0$, and that all transient indices $j$ with $h_j<h$ have been taken care of. The first thing we need to do for each given $i$ is to find the $\overline{f}$-preimages $j_1,j_2,\ldots,j_K$ of $i$, which requires $O(d\log{d})$ bit operations. Following that, we compute a spanning congruence sequence for $\Pcal_i=\bigwedge_{t=1}^K{\Pfrak'(\Pcal_{j_t},A_{j_t})}$. This involves simple arithmetic operations (including gcd computations) and requires
\[
O\left(\sum_{t=1}^K{m_{j_t}\log^{1+o(1)}{q}}\right)\subseteq O(d\log^{1+o(1)}{q})
\]
bit operations. Subsequently, we go through the logical sign tuples $\vec{\nu}^{(\Pcal_i)}\in\{\emptyset,\neg\}^{m_i}$ in lexicographic order, check whether $\Bcal(\Pcal_i,\vec{\nu}^{(\Pcal_i)})\not=\emptyset$, and if so, compute a compact description $\Dfrak$ of $\Tree_i(\Pcal_i,\vec{\nu}^{(\Pcal_i)})$. For checking whether the block $\Bcal(\Pcal_i,\vec{\nu}^{(\Pcal_i)})$ is non-empty, we note that by the argument before Proposition \ref{zeroBlockProp}, the cardinality $|\Bcal(\Pcal_i,\vec{\nu}^{(\Pcal_i)})|$ is equal to the distribution number $\sigma_{\Pcal_i,\mathbf{0}}(\vec{\nu}^{(\Pcal_i)},(\emptyset,\ldots,\emptyset))$ where $\mathbf{0}$ is the constant zero function $\IZ/s\IZ\rightarrow\IZ/s\IZ$. To see how costly the computation of this distribution number is, we refer to the following lemma.

\begin{lemmmma}\label{distNumLem}
Let $\Pcal$ be an arithmetic partition of $\IZ/m\IZ$, given by an explicit spanning $m$-congruence sequence of length $k\in\IN^+$. Moreover, let $A$ be an affine function $\IZ/m\IZ\rightarrow\IZ/m\IZ$. Then for any given logical sign tuples $\vec{\nu}^{(\Pcal)}$ and $\vec{\nu}^{(\Pcal')}$, of length $k$ and $k+1$, respectively, it takes $O(k2^k\log^{1+o(1)}{m})$ bit operations to compute the single distribution number $\sigma_{\Pcal,A}(\vec{\nu}^{(\Pcal)},\vec{\nu}^{(\Pcal')})$.
\end{lemmmma}

\begin{proof}
According to the formula in Lemma \ref{masterLem}, computing $\sigma_{\Pcal,A}(\vec{\nu}^{(\Pcal)},\vec{\nu}^{(\Pcal')})$ requires us to add up the summands $(-1)^{|J|}\kappa_{\Pcal,A}(\vec{\nu}^{(\Pcal)},\vec{\nu}^{(\Pcal')},J)$ for all $J\subseteq J_-(\vec{\nu}^{(\Pcal)})$, and there are $O(2^k)$ such summands. Computing a single such summand consists of
\begin{itemize}
\item a simple look-up of the last component of $\vec{\nu}^{(\Pcal')}$ (bit operation cost: $O(\log{k})$ for scanning the corresponding memory address);
\item $O(k)$ integer divisibility checks following a gcd computation and subtraction, of total bit operation cost $O(k\log^{1+o(1)}{m})$ by Lemma \ref{complexitiesLem}(1,3,8);
\item checking whether the two subsets $J_+(\vec{\nu}^{(\Pcal)})\cup J$ and $J_-(\vec{\nu}^{(\Pcal')})$ of $\{1,2,\ldots,k\}$ are disjoint, which involves look-ups of entries of $\vec{\nu}^{(\Pcal)}$ and $\vec{\nu}^{(\Pcal')}$ and takes $O(k)$ bit operations in total if pointers are used; and
\item performing $O(k)$ gcd computations, integer divisions and lcm computations, of total complexity $O(k\log^{1+o(1)}{m})$.
\end{itemize}
Thus, computing all summands $(-1)^{|J|}\kappa_{\Pcal,A}(\vec{\nu}^{(\Pcal_{j_t})},\vec{\nu}^{(\Pcal'_{j_t})},J)$ takes $O(k2^k\log^{1+o(1)}{m})$ bit operations, which majorizes the cost of adding these summands up and is thus also the complexity of computing $\sigma_{\Pcal,A}(\vec{\nu}^{(\Pcal)},\vec{\nu}^{(\Pcal')})$.
\end{proof}

In particular, computing $|\Bcal(\Pcal_i,\vec{\nu}^{(\Pcal_i)})|$ to check whether that block of $\Pcal_i$ is empty costs $O(m_i2^{m_i}\log^{1+o(1)}{q})\subseteq O(d2^d\log^{1+o(1)}{q})$ bit operations.

Let us now assume that $\Bcal(\Pcal_i,\vec{\nu}^{(\Pcal_i)})$ turned out to be non-empty. Then we wish to compute a compact description $\Dfrak$ of $\Tree_i(\Pcal_i,\vec{\nu}^{(\Pcal_i)})$. To do so, we write $\vec{\nu}^{(\Pcal_i)}=\diamond_{t=1}^K{\vec{\nu}^{(\Pcal'_{j_t})}}$ with $\vec{\nu}^{(\Pcal'_{j_t})}\in\{\emptyset,\neg\}^{m_{j_t}+1}$. By Proposition \ref{transientCosetsProp}, we may set
\[
\Dfrak:=\left\{\left(n,\sum_{t=1}^K\sum_{\vec{\nu}^{(\Pcal_{j_t})}\in S_{n,j_t}}{\sigma_{\Pcal_{j_t},A_{j_t}}(\vec{\nu}^{(\Pcal_{j_t})},\vec{\nu}^{(\Pcal'_{j_t})})}\right): n\in\Ncal\right\}\setminus(\IN\times\{0\}).
\]
We note that the range of the summation index in Proposition \ref{transientCosetsProp} includes logical sign tuples $\vec{\nu}^{(\Pcal_{j_t})}$ for which $\Bcal(\Pcal_{j_t},\vec{\nu}^{(\Pcal_{j_t})})$ is empty, but these may be ignored (as we do here), because all corresponding distribution numbers $\sigma_{\Pcal_{j_t},A_{j_t}}\left(\vec{\nu}^{(\Pcal_{j_t})},\vec{\nu}^{(\Pcal'_{j_t})}\right)$ are $0$. According to Lemma \ref{distNumLem}, computing a single one of the distribution numbers $\sigma_{\Pcal_{j_t},A_{j_t}}\left(\vec{\nu}^{(\Pcal_{j_t})},\vec{\nu}^{(\Pcal'_{j_t})}\right)$ takes
\[
O(2^{m_{j_t}}m_{j_t}\log^{1+o(1)}{q}),
\]
bit operations. Observing that for each fixed $t\in\{1,2,\ldots,K\}$, one has
\[
\bigcup\{S_{n,j_t}: n\in\Ncal\}\subseteq\{\emptyset,\neg\}^{m_{j_t}},
\]
we end up with a total bit operation cost of
\[
O\left(\sum_{t=1}^K{(2^{m_t}\cdot 2^{m_{j_t}}m_{j_t}\log^{1+o(1)}{q})}\right)\subseteq O(d4^d\log^{1+o(1)}{q})
\]
for computing $\Dfrak$, which majorizes the cost of checking whether $\Bcal(\Pcal_i,\vec{\nu}^{(\Pcal_i)})\not=\emptyset$. Also, if we go through the numbers $n\in\Ncal$ in increasing order when computing $\Dfrak$, the array representing $\Dfrak$ has its elements ordered by increasing $n$, as it should.

Next, we need to check whether the rooted tree $\Ifrak$ described by $\Dfrak$ occurs among the rooted trees $\Ifrak_n$, described by $\Dfrak_n$, which have already been introduced. The number of those trees is at most the total number of distinct (non-empty) blocks in all $\Pcal_j$ where $j$ is $\overline{f}$-transient, and that number is in $O(\min\{d2^d,q\})$. Since each $\Dfrak_n$ as well as $\Dfrak$ is a lexicographically sorted list of length in $O(\min\{d2^d,q\})$ every entry of which is a bit string of length in $O(\log{q})$, it takes $O(d^24^d\log{q})$ bit operations to check whether $\Dfrak=\Dfrak_n$ for some $n$. Should that be the case, we add $\vec{\nu}^{(\Pcal_i)}$ to $S_{n,i}$ as a new element (at the end of the array, which leads to that array being lexicographically ordered). Otherwise, we create $\Dfrak$ as a new tree description $\Dfrak_{n'}$, where $n'=\max{\Ncal}+1$, and initialize
\[
S_{n',j}:=
\begin{cases}
\{\vec{\nu}^{(\Pcal_i)}\}, & \text{if }j=i, \\
\emptyset, & \text{otherwise}.
\end{cases}
\]
For a given $\overline{f}$-transient $i$ such that $h_i=h$, this loop takes
\[
O(2^{m_i}(d4^d\log^{1+o(1)}{q}+d^24^d\log{q}))
\]
bit operations. Now, distinct indices $i$ with $h_i=h$ have disjoint iterated pre-image sets under $\overline{f}$, whence the sum of the numbers $m_i$ for all such $i$ is at most $d$. Therefore, we get a total bit operation cost of
\[
O(d8^d\log^{1+o(1)}{q}+d^28^d\log{q})
\]
for dealing with all $i$ such that $h_i$ has a given value. Dealing with all $\overline{f}$-transient $i$ in total takes
\[
O(d^28^d\log^{1+o(1)}{q}+d^38^d\log{q})
\]
bit operations.

\underline{Step 2: $i=d$.} Now that the $\overline{f}$-transient indices $i$ have been taken care of, one can compute a description $\Dfrak$ of $\Tree_{\Gamma_f}(0_{\IF_q})$ following Proposition \ref{zeroBlockProp}, which is similar to a single iteration of the loop in Step 1 and takes $O(d4^d\log^{1+o(1)}{q})$ bit operations. Afterward, we check whether $\Dfrak$ occurs among the existing descriptions $\Dfrak_n$ (introduced in Step 1), which (analogously to Step 1) takes $O(d^24^d\log{q})$ bit operations. If so, we set $S_{n,d}:=\emptyset$ (positive logical sign) for the corresponding unique $n$, and $S_{m,d}:=\neg$ for all other $m$. If not, we introduce $\Dfrak$ as a new tree description $\Dfrak_{n'}$ where $n'=\max{\Ncal}+1$, and set $S_{n',d}:=\emptyset$ and $S_{m,d}:=\neg$ for all $m<n'$. The overall complexity of this step is majorized by the one of Step 1.

\underline{Step 3: $\overline{f}$-periodic $i<d$.} Finally, we discuss how to deal with $\overline{f}$-periodic indices $i\in\{0,1,\ldots,d-1\}$. Let us use the notation from Step 3 in Subsection \ref{subsec3P3}. For instance, $(i_0,i_1,\ldots,i_{\ell-1})$ is the $\overline{f}$-cycle of $i=i_0$, and we have $i_t=i_{t\bmod{\ell}}$ for $t\in\IZ$ as well as $i'=i_{-1}$.

We begin by computing $H_i$, the maximum tree height in $\Gamma_f$ above periodic vertices in cosets of the form $C_{i_t}$ for $t\in\IZ$, for each $\overline{f}$-periodic $i$. Following the argument in Subsection \ref{subsec3P3}, we recall that $H_i\leq\ell\mpe(s)\leq\ell\lfloor\log_2{s}\rfloor$. Moreover, if we denote for fixed $k\in\IZ$ by $h'_{i,k}$\phantomsection\label{not269} the smallest positive integer $h'$ such that
\begin{equation}\label{twoGcdsEq}
\gcd\left(\prod_{t=0}^{h'-1}{\alpha_{i_{k-h'+t}}},s\right)=\gcd\left(\prod_{t=0}^{h'-2}{\alpha_{i_{k-h'+t}}},s\right),
\end{equation}
then $H_i=\max\{h'_{i,k}: k=0,1,\ldots,\ell-1\}-1$. Before we enter a loop over $k$ to find $h'_{i,k}$, we compute $\overline{\alpha}_i=\prod_{t=0}^{\ell-1}{\alpha_{i_t}}\bmod{s}$, taking $O(\ell\log^{1+o(1)}{s})\subseteq O(d\log^{1+o(1)}{q})$ bit operations. We then enter the loop over $k=0,1,\ldots,\ell-1$. For each $k$, we aim to find the correct value of $h'_{i,k}$ using a binary search in the range between $1$ and $\ell\lfloor\log_2{s}\rfloor+1$. Let us assume that we fixed a tentative value $h'$. Then for $H'\in\{h'-1,h'-2\}$, we have
\[
\prod_{t=0}^{H'}{\alpha_{i_{k-h'+t}}}=\left(\overline{\alpha}_i\right)^{\lfloor(H'+1)/\ell\rfloor}\cdot\prod_{t=0}^{(H'+1)\bmod{\ell}}{\alpha_{i_{k-h'+t}}},
\]
which can be computed modulo $s$ for both values of $H'$ using
\[
O(\log\log{s}\log^{1+o(1)}{s}+\ell\log^{1+o(1)}{s})=O(\ell\log^{1+o(1)}{s})\subseteq O(d\log^{1+o(1)}{q})
\]
bit operations. Following that, we compute and compare the two gcds from Equation (\ref{twoGcdsEq}), which takes $O(\log^{1+o(1)}{s})\subseteq O(\log^{1+o(1)}{q})$ bit operations. This binary search has $O(\log(\ell\log{s}))\subseteq O(\log{d}+\log\log{q})$ iterations per $k$, and there are $\ell\in O(d)$ values of $k$ to deal with. In total, the computation of $H_i$ takes
\[
O(d\cdot(\log{d}+\log\log{q})\cdot d\log^{1+o(1)}{q})=O(d^2\log{d}\log^{1+o(1)}{q}).
\]
bit operations for each individual $i$, and
\[
O(d^3\log{d}\log^{1+o(1)}{q})
\]
bit operations for all $\overline{f}$-periodic $i<d$ together.

After finding all $H_i$, we aim to compute $\Zcal_i=(\Xcal_{i,h})_{h=-1,0,\ldots,H_i}$ for each $\overline{f}$-periodic $i$. We do so by computing $\Xcal_{i,h}$ for all $\overline{f}$-periodic $i$ together successively for $h=-1,0,\ldots,\Hfrak:=\max\{H_i: i\in\per(\overline{f})\setminus\{d\}\}$\phantomsection\label{not270} (for each fixed value of $h$, we skip those $i$ such that $h>H_i$). Now, $\Xcal_{i,-1}=(\theta_{i,h}(x))_{h=1,2,\ldots,H_i}$ consists of $H_i$ congruences, and according to the definition of $\theta_{i_t,h}(x)$, these congruences can be computed recursively using simple arithmetic in each step. Per $i$, this takes
\[
O((H_i+1)\log^{1+o(1)}{q})\subseteq O((d\mpe(s)+1)\log^{1+o(1)}{q})\subseteq O(d\mpe(q-1)\log^{1+o(1)}{q})
\]
bit operations, and for all $\overline{f}$-periodic $i$ together, it takes $O(d^2\mpe(q-1)\log^{1+o(1)}{q})$ bit operations to compute $\Xcal_{i,-1}$. The computation of $\Xcal_{i,0}=\Rcal_i$ is analogous to the one of $\Zcal_j=\Pcal_j$ for $\overline{f}$-transient $j$ (see Step 1), taking $O(d\log^{1+o(1)}{q})$ bit operations per $i$, and $O(d^2\log^{1+o(1)}{q})$ for all $\overline{f}$-periodic $i$ together. Following that, the spanning congruence sequence for
\[
\Xcal_{i,h}=\lambda_{i_{-h}}^h(\Rcal_{i_{-h}})=\lambda(\lambda_{i_{-h}}^{h-1}(\Rcal_{i_{-h}}),A_{i_{-1}})=\lambda(\Xcal_{i_{-1},h-1},A_{i_{-1}})
\]
is obtained recursively for $h=1,2,\ldots,H_i$ through processing the one for $\Xcal_{i_{-1},h-1}$ using simple arithmetic, again taking complexity $O(d^2\log^{1+o(1)}{q})$ in each step for all $i$ together. Overall, the bit operation cost of computing $\Zcal_i$ after each $H_i$ has been worked out is in
\begin{align*}
&O(d^2\mpe(q-1)\log^{1+o(1)}{q}+(d\mpe(s)+1)\cdot d^2\log^{1+o(1)}{q}) \\
\subseteq &O(d^3\mpe(q-1)\log^{1+o(1)}{q})
\end{align*}
for all $\overline{f}$-periodic $i<d$ together.

Finally, we need to
\begin{itemize}
\item extend the list of rooted tree descriptions $\Dfrak_n$ produced in Steps 1 and 2 to its final version, which additionally contains descriptions of all rooted tree isomorphism types of the form $\Tree_i^{(h)}(\Pcal_{i,h},\vec{\nu}^{(\Pcal_{i,h})})$ where $i<d$ is $\overline{f}$-periodic, $h\in\{0,1,\ldots,H_i\}$, $\vec{\nu}^{(\Pcal_{i,h})}\in\{\emptyset,\neg\}^{n_{i_{-h}}+n_{i_{-h+1}}+\cdots+n_{i_0}}$, and $\Bcal(\Qcal_{i,h},\vec{\nu}^{(\Pcal_{i,h})}\diamond\vec{\xi}_{i,h})\not=\emptyset$; and
\item compute the associated logical sign sets $S_{n,i,h}$.
\end{itemize}

We do so recursively in the same manner as before, i.e., computing successively for $h=0,1,\ldots,\Hfrak$ the relevant data for \emph{all} corresponding $\overline{f}$-periodic $i<d$ together. For $h=0$, where $\Pcal_{i,h}=\Rcal_i$, this is basically identical to the corresponding computations in Step 1 and has the same overall bit operation cost, in $O(d^28^d\log^{1+o(1)}{q}+d^38^d\log{q})$.

Now we assume that $h>0$ and that all smaller values have been taken care of. We loop (in lexicographic order) over the logical sign tuples $\vec{\nu}^{(\Pcal_{i,h})}$ with $\sum_{\tfrak=0}^h{n_{i_{-\tfrak}}}$ entries, noting that there are $O(2^{(h+1)d})$ such tuples. For each tuple $\vec{\nu}^{(\Pcal_{i,h})}$, we first check whether $\Bcal(\Qcal_{i,h},\vec{\nu}^{(\Pcal_{i,h})}\diamond\vec{\xi}_{i,h})$ is non-empty. Because the length of the spanning congruence sequence for $\Qcal_{i,h}$ which we use here is in $O((h+1)d+H_i)$, Lemma \ref{distNumLem} implies that this check takes
\begin{align*}
&O(((h+1)d+H_i)2^{(h+1)d+H_i}\log^{1+o(1)}{q}) \\
\subseteq &O(d^2\mpe(q-1)2^{d^2\mpe(q-1)+d\mpe(q-1)+d}\log^{1+o(1)}{q})
\end{align*}
bit operations. If this block of $\Qcal_{i,h}$ is indeed non-empty, we need to compute a compact description $\Dfrak$ of $\Tree_i(\Pcal_{i,h},\vec{\nu}^{(\Pcal_{i,h})})$. Let $j_1,j_2,\ldots,j_K$ be the $\overline{f}$-transient pre-images of $i$ under $\overline{f}$ (which take $O(d\log{d})$ bit operations to determine). Writing $\vec{\nu}^{(\Pcal_{i,h})}=\diamond_{\tfrak=0}^h{\vec{o'_{\tfrak}}}$ with $\vec{o'_{\tfrak}}\in\{\emptyset,\neg\}^{n_{i_{-\tfrak}}}$, and $\vec{o'_0}=\diamond_{t=1}^K{\vec{\nu}^{(\Pcal'_{j_t})}}$ with $\vec{\nu}^{(\Pcal'_{j_t})}\in\{\emptyset,\neg\}^{m_{j_t}+1}$, we may choose $\Dfrak$ as follows according to Propositions \ref{periodicCosetsTransientProp} and \ref{periodicCosetsPeriodicProp}:
\begin{align*}
\Dfrak:=&\{(m,\sum_{t=1}^K\sum_{\vec{\nu}^{(\Pcal_{j_t})}\in S_{m,j_t}}{\sigma_{\Pcal_{j_t},A_{j_t}}\left(\vec{\nu}^{(\Pcal_{j_t})},\vec{\nu}^{(\Pcal'_{j_t})}\right)} \\
&+\sum_{k=0}^{h-1}\sum_{\vec{\nu}^{(\Pcal_{i',k})}\in S_{m,i',k}}{\sigma_{\Qcal_{i',k},A_{i'}}\left(\vec{\nu}^{(\Pcal_{i',k})}\diamond\vec{\xi}_{i',k},\diamond_{\tfrak=1}^{k+1}{\vec{o'_{\tfrak}}}\diamond\vec{\xi_{i,h}}\right)}):m\in\Ncal\} \\
&\setminus(\IN\times\{0\}).
\end{align*}
Computing the first sum,
\[
\sum_{t=1}^K\sum_{\vec{\nu}^{(\Pcal_{j_t})}\in S_{m,j_t}}{\sigma_{\Pcal_{j_t},A_{j_t}}\left(\vec{\nu}^{(\Pcal_{j_t})},\vec{\nu}^{(\Pcal'_{j_t})}\right)},
\]
\emph{for all $m\in\Ncal$ together} is analogous to the corresponding argument in Step 1 (also applying Lemma \ref{distNumLem} with $k:=m_{j_t}$ and $m:=q$) and takes $O(d4^d\log^{1+o(1)}{q})$ bit operations. As for the complexity of computing the second sum,
\[
\sum_{k=0}^{h-1}\sum_{\vec{\nu}^{(\Pcal_{i',k})}\in S_{m,i',k}}{\sigma_{\Qcal_{i',k},A_{i'}}\left(\vec{\nu}^{(\Pcal_{i',k})}\diamond\vec{\xi}_{i',k},\diamond_{\tfrak=1}^{k+1}{\vec{o'_{\tfrak}}}\diamond\vec{\xi}_{i,h}\right)},
\]
we note that for fixed $k$ and $\vec{\nu}^{(\Pcal_{i',k})}$, Lemma \ref{distNumLem} with $k:=(k+1)d+H_i$ and $m:=q$ implies that computing the single distribution number
\[
\sigma_{\Qcal_{i',k},A_{i'}}(\vec{\nu}^{(\Pcal_{i',k})}\diamond\vec{\xi}_{i',k},\diamond_{\tfrak=1}^{k+1}{\vec{o'_{\tfrak}}}\diamond\vec{\xi}_{i,h})
\]
takes
\[
O(2^{(k+1)d+H_i}((k+1)d+H_i)\log^{1+o(1)}{q})
\]
bit operations, whence
\begin{align*}
&O(2^{(k+1)d}\cdot 2^{(k+1)d+H_i}((k+1)d+H_i)\log^{1+o(1)}{q}) \\
=&O(2^{2(k+1)d+H_i}((k+1)d+H_i)\log^{1+o(1)}{q})
\end{align*}
bit operations are needed for computing all of these numbers for a fixed $k$ \emph{and all $m$ together}. Computing the second sum in its entirety for all $m$ together takes
\[
O\left(\sum_{k=0}^{h-1}{2^{2(k+1)d+H_i}((k+1)d+H_i)\log^{1+o(1)}{q}}\right)\subseteq O(2^{2hd+H_i}(hd+H_i)\log^{1+o(1)}{q}),
\]
bit operations, which majorizes the overall bit operation cost for computing the first sum, and thus is also the cost of computing $\Dfrak$.

Next, we need to check if $\Dfrak$ occurs among the already introduced descriptions $\Dfrak_n$ (for $n\in\Ncal$). The number $|\Ncal|$ of these descriptions is at most the sum of the numbers of distinct (nonempty) blocks in arithmetic partitions of one of the forms
\begin{itemize}
\item $\Pcal_j$ where $j$ is $\overline{f}$-transient, or
\item $\Pcal_{j,k}$ where $j<d$ is $\overline{f}$-periodic and $k\leq h$,
\end{itemize}
and that sum is in
\[
O\left(d2^d+d\sum_{k=0}^h{2^{(k+1)d}}\right)=O(d2^{(h+1)d}).
\]
Moreover, each of the descriptions $\Dfrak_n$ for $n\in\Ncal$ as well as $\Dfrak$ is a lexicographically sorted list of length in $O(\min\{d2^{(h+1)d},q\})$ each entry of which is a bit string of length in $O(\log{q})$, so it takes $O(d^24^{(h+1)d}\log{q})$ bit operations to check if $\Dfrak=\Dfrak_n$ for some $n\in\Ncal$. If so, we add $\vec{\nu}^{(\Pcal_{i,h})}$ to $S_{n,i,h}$ as a new element. Otherwise, we create $\Dfrak$ as a new tree description $\Dfrak_{n'}$, where $n'=\max{\Ncal}+1$, and initialize
\[
S_{n',j,h}:=
\begin{cases}
\{\vec{\nu}^{(\Pcal_{i,h})}\}, & \text{if }j=i, \\
\emptyset, & \text{otherwise}.
\end{cases}
\]

For a given $\overline{f}$-periodic $i<d$, this loop has a bit operation cost in
\begin{align*}
&O(2^{(h+1)d}\cdot(2^{2hd+H_i}(hd+H_i)\log^{1+o(1)}{q}+4^{(h+1)d}d^2\log{q})) \\
&\subseteq O(2^{3hd+H_i+d}(hd+H_i)\log^{1+o(1)}{q}+8^{(h+1)d}d^2\log{q}),
\end{align*}
and doing this for all such $i$ for a fixed value of $h$ costs
\[
O(d(hd+\Hfrak)2^{3hd+\Hfrak+d}\log^{1+o(1)}{q}+d^38^{(h+1)d}\log{q}).
\]
bit operations. In total, the bit operation cost of Step 3 is in
\begin{align*}
&O(d^3\mpe(q-1)\log^{1+o(1)}{q}+d^28^d\log^{1+o(1)}{q}+d^38^d\log{q}\\
&+\sum_{h=1}^{\Hfrak}{(d(hd+\Hfrak)2^{3hd+\Hfrak+d}\log^{1+o(1)}{q}+d^38^{(h+1)d}\log{q})}) \\
\subseteq &O(d^3\mpe(q-1)\log^{1+o(1)}{q}+d^28^d\log^{1+o(1)}{q}+d^38^d\log{q} \\
&+d(d+1)\Hfrak 2^{3\Hfrak d+\Hfrak+d}\log^{1+o(1)}{q}+d^38^{(\Hfrak+1)d}\log{q}) \\
\subseteq &O(d^3\mpe(q-1)2^{(3d^2+d)\mpe(q-1)+d}\log^{1+o(1)}{q}+d^38^{(d\mpe(q-1)+1)d}\log{q}) \\
\subseteq &O(d^3\mpe(q-1)2^{(3d^2+d)\mpe(q-1)+2d}\log^{1+o(1)}{q}),
\end{align*}
which majorizes the costs of Steps 1 and 2 and thus is the overall bit operation cost of the algorithm for Problem 2, as asserted. The claims on the length $|\Ncal|$ of the constructed recursive tree description list, as well as on the memory costs of the individual rooted tree descriptions, can be deduced as follows. Through applying the above bound on $|\Ncal|$ that holds throughout the $h$-th iteration of the loop with $h=\Hfrak\leq d\mpe(q-1)$, we get
\[
|\Ncal|\in O(d2^{(\Hfrak+1)d})\subseteq O(d2^{(d\mpe(q-1)+1)d})=O(d2^{d^2\mpe(q-1)+d}).
\]
Moreover, $|\Ncal|\leq q$ because by construction, each rooted tree isomorphism type in the associated recursive tree description list is of the form $\Tree_{\Gamma_f}(x)$ for some $x\in\IF_q=\V(\Gamma_f)$. Each individual tree description $\Dfrak_n$ is a set consisting of pairs of the form $(m,k)$ where the first entries $m$ are pairwise distinct elements of $\Ncal$, whence the length of $\Dfrak_n$ as a list is at most $|\Ncal|\in O(\min\{d2^{d^2\mpe(q-1)+d},q\})$. Finally, the bit cost of storing an individual pair $(m,k)$ is in $O(\log{q})$, as required.

We conclude this subsubsection with a detailed overview of the steps of this algorithm in the form of pseudocode, using the same format as at the end of Subsubsection \ref{subsubsec5P2P1}.

\begin{denumerate}[label=\arabic*]
\item Compute the induced function $\overline{f}$ on $\{0,1,\ldots,d\}$ and the affine maps $A_i$ of $\IZ/s\IZ$.

QC: $(d\log^{1+o(1)}{q},d,0,0,0)$.
\item For $i\in\{0,1,\ldots,d-1\}$, let
\[
h_i:=
\begin{cases}
\height(\Tree_{\Gamma_{\overline{f}}}(i)), & \text{if }i\text{ is }\overline{f}\text{-transient}, \\
\infty, & \text{otherwise}.
\end{cases}
\]
For each attainable value $h$ of $h_i$, compute the associated list $\Layer_h$\phantomsection\label{not271} of indices $i$.

QC: $(d^2\log^2{d},0,0,0,0)$.
\item Set $\Ncal:=\emptyset$.

QC: $(1,0,0,0,0)$.
\item Set $\vec{\Dfrak}:=\emptyset$.

QC: $(1,0,0,0,0)$.
\item Set $\vec{\imath}:=\emptyset$.

QC: $(1,0,0,0,0)$.
\item For $h=0,1,\ldots,\max\{h_i: i\in\{0,1,\ldots,d-1\}\setminus\per(\overline{f})\}$, do the following.

QC: $(d^28^d\log^{1+o(1)}{q}+d^38^d\log{q},0,0,0,0)$.
\begin{denumerate}[label=6.\arabic*]
\item If $h=0$, then do the following.
\begin{denumerate}[label=6.1.\arabic*]
\item Set $N':=0$.

QC: $(1,0,0,0,0)$.
\item Add $N'$ to $\Ncal$ as a new element.

QC: $(1,0,0,0,0)$.
\item Set $\Dfrak_0:=\emptyset$, and add it to $\vec{\Dfrak}$ as a new element.

QC: $(1,0,0,0,0)$.
\item For each $i\in\Layer_0$, do the following.

QC: $(d\log{d},0,0,0,0)$.
\begin{denumerate}[label=6.1.4.\arabic*]
\item Set $\Zcal_i:=\Pcal_i:=\Pfrak(\emptyset)$ and $m_i:=0$.

QC: $(\log{d},0,0,0,0)$.
\item Set $S_{0,i}:=\{\emptyset\}$.

QC: $(\log{d},0,0,0,0)$.
\item Add $i$ to $\vec{\imath}$ as a new element.

QC: $(\log{d},0,0,0,0)$.
\end{denumerate}
\end{denumerate}
\item Else do the following.
\begin{denumerate}[label=6.2.\arabic*]
\item For each $i\in\Layer_h$, do the following.

QC: $(d8^d\log^{1+o(1)}{q}+d^28^d\log{q},0,0,0,0)$.
\begin{denumerate}[label=6.2.1.\arabic*]
\item Compute the $\overline{f}$-pre-images $j_1,j_2,\ldots,j_K$ of $i$.

QC: $(d\log{d},0,0,0,0)$.
\item Compute a spanning congruence sequence, of length $m_i$, for $\Zcal_i:=\Pcal_i:=\bigwedge_{t=1}^K{\Pfrak'(\Pcal_{j_t},A_{j_t})}$.

QC: $(d\log^{1+o(1)}{q},0,0,0,0)$.
\item For each $n\in\Ncal$, do the following.

QC: $(d2^d\log{q},0,0,0,0)$.
\begin{denumerate}[label=6.2.1.3.\arabic*]
\item Initialize $S_{n,i}:=\emptyset$.

QC: $(\log{q},0,0,0,0)$.
\end{denumerate}
\item For each $\vec{\nu}^{(\Pcal_i)}\in\{\emptyset,\neg\}^{m_i}$, do the following.

QC: $(2^{m_i}(d4^d\log^{1+o(1)}{q}+d^24^d\log{q}),0,0,0,0)$.
\begin{denumerate}[label=6.2.1.4.\arabic*]
\item Check whether $|\Bcal(\Pcal_i,\vec{\nu}^{(\Pcal_i)})|=\sigma_{\Pcal_i,\mathbf{0}}(\vec{\nu}^{(\Pcal_i)},(\emptyset,\ldots,\emptyset))=0$, and if so, skip to the next tuple $\vec{\nu}^{(\Pcal_i)}$.

QC: $(d2^d\log^{1+o(1)}{q},0,0,0,0)$.
\item Writing $\vec{\nu}^{(\Pcal_i)}=\diamond_{t=1}^K{\vec{\nu}^{(\Pcal'_{j_t})}}$ with $\vec{\nu}^{(\Pcal'_{j_t})}\in\{\emptyset,\neg\}^{m_{j_t}+1}$, compute
\[
\Dfrak:=\{(n,\sum_{t=1}^K\sum_{\vec{\nu}^{(\Pcal_{j_t})}\in S_{n,j_t}}{\sigma_{\Pcal_{j_t},A_{j_t}}(\vec{\nu}^{(\Pcal_{j_t})},\vec{\nu}^{(\Pcal'_{j_t})})}): n\in\Ncal\}\setminus(\IN\times\{0\}),
\]
a compact description of $\Tree_i(\Pcal_i,\vec{\nu}^{(\Pcal_i)})$.

QC: $(d4^d\log^{1+o(1)}{q},0,0,0,0)$.
\item Check whether $\Dfrak=\Dfrak_n$ for some (unique) $n\in\Ncal$, and store this information (the truth value and, if applicable, $n$).

QC: $(d^24^d\log{q},0,0,0,0)$.
\item If $\Dfrak=\Dfrak_n$ for some $n\in\Ncal$, then do the following.
\begin{denumerate}[label=6.2.1.4.4.\arabic*]
\item Add $\vec{\nu}^{(\Pcal_i)}$ to $S_{n,i}$ as a new element.

QC: $(\log{q}+d,0,0,0,0)$.
\end{denumerate}
\item Else do the following.
\begin{denumerate}[label=6.2.1.4.5.\arabic*]
\item Set $N':=N'+1$, and add it to $\Ncal$ as a new element.

QC: $(\log{q},0,0,0,0)$.
\item Set $\Dfrak_{N'}:=\Dfrak$, and add it to $\vec{\Dfrak}$ as a new element.

QC: $(d2^d\log{q},0,0,0,0)$.
\item For $j\in\vec{\imath}$, do the following.

QC: $(d\log{q},0,0,0,0)$.
\begin{denumerate}[label=6.2.1.4.5.3.\arabic*]
\item Set $S_{N',j}:=\emptyset$.

QC: $(\log{q},0,0,0,0)$.
\end{denumerate}
\item Set $S_{N',i}:=\{\vec{\nu}^{(\Pcal_i)}\}$.

QC: $(\log{q}+d,0,0,0,0)$.
\end{denumerate}
\end{denumerate}
\item Add $i$ to $\vec{\imath}$ as a new element.

QC: $(\log{d},0,0,0,0)$.
\end{denumerate}
\end{denumerate}
\end{denumerate}
\item Compute the $\overline{f}$-transient $\overline{f}$-pre-images $j_1,j_2,\ldots,j_K$ of $d$.

QC: $(d\log{d},0,0,0,0)$.
\item Compute
\[
\Dfrak:=\{(n,\sum_{t=1}^K\sum_{\vec{\nu}^{(\Pcal_{j_t})}\in S_{n,j_t}}{\sigma_{\Pcal_{j_t},\mathbf{0}}(\vec{\nu}^{(\Pcal_{j_t})},(\emptyset,\ldots,\emptyset))}): n\in\Ncal\}\setminus(\IN\times\{0\}),
\]
a compact description of $\Tree_{\Gamma_f}(0_{\IF_q})$.

QC: $(d4^d\log^{1+o(1)}{q},0,0,0,0)$.
\item Check whether $\Dfrak=\Dfrak_n$ for some (unique) $n\in\Ncal$, and store this information (the truth value and, if applicable, $n$).

QC: $(d^24^d\log{q},0,0,0,0)$.
\item If $\Dfrak=\Dfrak_n$ for some $n\in\Ncal$, then do the following.
\begin{denumerate}[label=10.\arabic*]
\item Set $S_{n,d}:=\emptyset$, and $S_{m,d}:=\neg$ for all $m\in\Ncal\setminus\{n\}$.

QC: $(d2^d\log{q},0,0,0,0)$.
\end{denumerate}
\item Else do the following.
\begin{denumerate}[label=11.\arabic*]
\item Set $N':=N'+1$, and add it to $\Ncal$ as a new element.

QC: $(\log{q},0,0,0,0)$.
\item Set $\Dfrak_{N'}:=\Dfrak$, and add it to $\vec{\Dfrak}$ as a new element.

QC: $(d2^d\log{q},0,0,0,0)$.
\item For $j\in\vec{\imath}=\{0,1,\ldots,d-1\}\setminus\per(\overline{f})$, do the following.

QC: $(d\log{q},0,0,0,0)$.
\begin{denumerate}[label=11.3.\arabic*]
\item Set $S_{N',j}:=\emptyset$.

QC: $(\log{q},0,0,0,0)$.
\end{denumerate}
\item Set $S_{N',d}:=\emptyset$, and $S_{n,d}:=\neg$ for all $n\in\Ncal\setminus\{N'\}$.

QC: $(d2^d\log{q},0,0,0,0)$.
\end{denumerate}
\item For $i\in\per(\overline{f})\setminus\{d\}$, do the following.

QC: $(d^3\log{d}\log^{1+o(1)}{q},0,0,0,0)$.
\begin{denumerate}[label=12.\arabic*]
\item Compute $H_i$, the maximum tree height in $\Gamma_f$ above periodic vertices in cosets of the form $C_{i_t}$ where $t\in\IZ$. See the discussion above for details.

QC: $(d^2\log{d}\log^{1+o(1)}{q},0,0,0,0)$.
\end{denumerate}
\item Compute $\Hfrak:=\max\{H_i: i\in\per(\overline{f})\setminus\{d\}\}$.

QC: $(d\log{d},0,0,0,0)$.
\item For $h=-1,0,1,\ldots,\Hfrak$, do the following.

QC: $(d^3\mpe(q-1)\log^{1+o(1)}{q},0,0,0,0)$.
\begin{denumerate}[label=14.\arabic*]
\item If $h=-1$, then do the following.
\begin{denumerate}[label=14.1.\arabic*]
\item For $i\in\per(\overline{f})\setminus\{d\}$, do the following.

QC: $(d^2\mpe(q-1)\log^{1+o(1)}{q},0,0,0,0)$.
\begin{denumerate}[label=14.1.1.\arabic*]
\item Compute and store the $\overline{f}$-pre-images of $i$, including the periodic one, $i'$.

QC: $(d\log{d},0,0,0,0)$.
\item Compute $\Xcal_{i,-1}:=(\theta_{i,k}(x))_{k=1,2,\ldots,H_i}$.

QC: $(d\mpe(q-1)\log^{1+o(1)}{q},0,0,0,0)$.
\end{denumerate}
\end{denumerate}
\item Else do the following.
\begin{denumerate}[label=14.2.\arabic*]
\item If $h=0$, then do the following.
\begin{denumerate}[label=14.2.1.\arabic*]
\item For $i\in\per(\overline{f})\setminus\{d\}$, do the following.

QC: $(d^2\log^{1+o(1)}{q},0,0,0,0)$.
\begin{denumerate}[label=14.2.1.1.\arabic*]
\item Compute $\Xcal_{i,0}=\Rcal_i=\bigwedge_{t=1}^K{\Pfrak'(\Pcal_{j_t},A_{j_t})}$, with a spanning sequence of length $n_i$, where $j_1,j_2,\ldots,j_K$ are the $\overline{f}$-transient $\overline{f}$-pre-images of $i$ (this can be handled analogous to Steps 6.2.1.1 and 6.2.1.2).

QC: $(d\log^{1+o(1)}{q},0,0,0,0)$.
\end{denumerate}
\end{denumerate}
\item Else do the following.
\begin{denumerate}[label=14.2.2.\arabic*]
\item For $i\in\per(\overline{f})\setminus\{d\}$, do the following.

QC: $(d^2\log^{1+o(1)}{q},0,0,0,0)$.
\begin{denumerate}[label=14.2.2.1.\arabic*]
\item If $h\leq H_i$, then compute $\Xcal_{i,h}=\lambda(\Xcal_{i',h-1},A_{i'})$ (we observe that as a spanning congruence sequence, $\Xcal_{i,h}$ has length $n_{i'}$, same as $\Xcal_{i',h-1}$). Otherwise, skip to the next value of $i$.

QC: $(d\log^{1+o(1)}{q},0,0,0,0)$.
\end{denumerate}
\end{denumerate}
\end{denumerate}
\end{denumerate}
\item For $i\in\per(\overline{f})\setminus\{d\}$, do the following.

QC: $(d^3\mpe(q-1)\log{q},0,0,0,0)$.
\begin{denumerate}[label=15.\arabic*]
\item Set $\Zcal_i:=(\Xcal_{i,h})_{h=-1,0,\ldots,H_i}$.

QC: $(d^2\mpe(q-1)\log{q},0,0,0,0)$.
\end{denumerate}
\item For $h=0,1,\ldots,\Hfrak$, do the following.

QC: $(d^3\mpe(q-1)2^{(3d^2+d)\mpe(q-1)+2d}\log^{1+o(1)}{q},0,0,0,0)$.
\begin{denumerate}[label=16.\arabic*]
\item Initialize all sets $S_{n,i,h}$, where $n\in\Ncal$ and $i\in\per(\overline{f})\setminus\{d\}$ such that $H_i\geq h$, to be $\emptyset$.

QC: $(d^22^{hd},0,0,0,0)$.
\item If $h=0$, then do the following.
\begin{denumerate}[label=16.2.\arabic*]
\item Extend $\Ncal$ and the associated list $\vec{\Dfrak}$ of rooted tree descriptions $\Dfrak_n$ such that all rooted tree isomorphism types of the form $\Tree_i(\Pcal_{i,0},\vec{\nu}^{(\Pcal_{i,0})})=\Tree_i(\Rcal_i,\bigcup_{t=1}^K{C_{j_t}},\vec{\nu}^{(\Pcal_{i,0})})$ for $i\in\per(\overline{f})\setminus\{d\}$ are covered, and compute the associated logical sign tuple sets $S_{n,i,0}$. This is analogous to Step 6.2.1.4, but carried out for the $O(d)$ values of $i\in\per(\overline{f})\setminus\{d\}$ together.

QC: $(d^28^d\log^{1+o(1)}{q}+d^38^d\log{q},0,0,0,0)$.
\end{denumerate}
\item Else do the following.
\begin{denumerate}[label=16.3.\arabic*]
\item For each $i\in\per(\overline{f})\setminus\{d\}$, do the following.

QC: $(d(hd+\Hfrak)2^{3hd+\Hfrak+d}\log^{1+o(1)}{q}+d^38^{(h+1)d}\log{q},0,0,0,0)$.
\begin{denumerate}[label=16.3.1.\arabic*]
\item Check whether $h\leq H_i$. If not, skip to the next $i$.

QC: $(\log{q},0,0,0,0)$.
\item Recalling that $\Pcal_{i,h}=\bigwedge_{t=0}^h{\Xcal_{i,t}}$, do the following for each $\vec{\nu}^{(\Pcal_{i,h})}\in\{\emptyset,\neg\}^{n_{i_0}+n_{i_{-1}}+\cdots+n_{i_{-h}}}$.

QC: $((hd+H_i)2^{3hd+H_i+d}\log^{1+o(1)}{q}+d^28^{(h+1)d}\log{q},0,0,0,0)$.
\begin{denumerate}[label=16.3.1.2.\arabic*]
\item Recalling that $\Qcal_{i,h}=\Pcal_{i,h}\wedge\Ucal_i$, where $\Ucal_i=\Pfrak'(\theta_{i,k}(x): k=1,2,\ldots,H_i)$ (and the definition of $\vec{\xi}_{i,h}\in\{\emptyset,\neg\}^{H_i}$ from page \pageref{not160}), check whether
\[
|\Bcal(\Qcal_{i,h},\vec{\nu}^{(\Pcal_{i,h})}\diamond\vec{\xi}_{i,h})|=\sigma_{\Qcal_{i,h},\mathbf{0}}(\vec{\nu}^{(\Pcal_{i,h})}\diamond\vec{\xi}_{i,h},(\emptyset,\ldots,\emptyset))=0,
\]
and if so, skip to the next tuple $\vec{\nu}^{(\Pcal_{i,h})}$.

QC: $(d^2\mpe(q-1)2^{d^2\mpe(q-1)+d\mpe(q-1)+d}\log^{1+o(1)}{q},0,0,0,0)$.
\item Writing $\vec{\nu}^{(\Pcal_{i,h})}=\diamond_{t=0}^h{\vec{o'_t}}$ with $\vec{o'_t}\in\{\emptyset,\neg\}^{n_{i_{-t}}}$, and $\vec{o'_0}=\diamond_{t=1}^K{\vec{\nu}^{(\Pcal'_{j_t})}}$ with $\vec{\nu}^{(\Pcal'_{j_t})}\in\{\emptyset,\neg\}^{n_{j_t}+1}$ (where $j_1,j_2,\ldots,j_K$ are the $\overline{f}$-transient $\overline{f}$-pre-images of $i$, computed in Step 14.1.1.1), compute
\begin{align*}
&\Dfrak:= \\
&\{(m,\sum_{t=1}^K\sum_{\vec{\nu}^{(\Pcal_{j_t})}\in S_{m,j_t}}{\sigma_{\Pcal_{j_t},A_{j_t}}(\vec{\nu}^{(\Pcal_{j_t})},\vec{\nu}^{(\Pcal'_{j_t})})} \\
&+\sum_{k=0}^{h-1}\sum_{\vec{\nu}^{(\Pcal_{i',k})}\in S_{m,i',k}}{\sigma_{\Qcal_{i',k},A_{i'}}(\vec{\nu}^{(\Pcal_{i',k})}\diamond\vec{\xi}_{i',k},\diamond_{\tfrak=1}^{k+1}{\vec{o'_{\tfrak}}}\diamond\vec{\xi_{i,h}})}): \\
&m\in\Ncal\}\setminus(\IN\times\{0\}).
\end{align*}
QC: $(2^{2hd+H_i}(hd+H_i)\log^{1+o(1)}{q},0,0,0,0)$.
\item Check whether $\Dfrak=\Dfrak_n$ for some (unique) $n\in\Ncal$, and store this information (the truth value and, if applicable, $n$).

QC: $(d^24^{(h+1)d}\log{q},0,0,0,0)$.
\item If $\Dfrak=\Dfrak_n$ for some $n\in\Ncal$, then do the following.
\begin{denumerate}[label=16.3.1.2.4.\arabic*]
\item Add $\vec{\nu}^{(\Pcal_{i,h})}$ to $S_{n,i,h}$ as a new element.

QC: $(\log{q}+(h+1)d,0,0,0,0)$.
\end{denumerate}
\item Else do the following.
\begin{denumerate}[label=16.3.1.2.5.\arabic*]
\item Set $N':=N'+1$, and add it to $\Ncal$ as a new element.

QC: $(\log{q},0,0,0,0)$.
\item Set $\Dfrak_{N'}:=\Dfrak$, and add it to $\vec{\Dfrak}$ as a new element.

QC: $(d2^{(h+1)d}\log{q},0,0,0,0)$.
\item For $j\in\vec{\imath}=\{0,1,\ldots,d-1\}\setminus\per(\overline{f})$, do the following.

QC: $(d\log{q},0,0,0,0)$.
\begin{denumerate}[label=16.3.1.2.5.3.\arabic*]
\item Set $S_{N',j}:=\emptyset$.

QC: $(\log{q},0,0,0,0)$.
\end{denumerate}
\item Set $S_{N',d}:=\neg$.

QC: $(\log{q},0,0,0,0)$.
\item For $k=0,1,\ldots,h-1$, do the following.

QC: $(dh\log{q},0,0,0,0)$.
\begin{denumerate}[label=16.3.1.2.5.5.\arabic*]
\item For $j\in\per(\overline{f})\setminus\{d\}$, do the following.

QC: $(d\log{q},0,0,0,0)$.
\begin{denumerate}[label=16.3.1.2.5.5.1.\arabic*]
\item Set $S_{N',j,k}:=\emptyset$.

QC: $(\log{q},0,0,0,0)$.
\end{denumerate}
\end{denumerate}
\item For $j\in\per(\overline{f})\setminus\{d\}$, do the following.

QC: $(d\log{q}+(h+1)d^2,0,0,0,0)$.
\begin{denumerate}[label=16.3.1.2.5.6.\arabic*]
\item Set
\[
S_{N',j,h}:=
\begin{cases}
\{\vec{\nu}^{(\Pcal_{i,h})}\}, & \text{if }j=i, \\
\emptyset, & \text{otherwise}.
\end{cases}
\]
QC: $(\log{q}+(h+1)d,0,0,0,0)$.
\end{denumerate}
\end{denumerate}
\end{denumerate}
\end{denumerate}
\end{denumerate}
\end{denumerate}
\item For $n\in\Ncal$, do the following.

QC: $(d^4\mpe(q-1)4^{d^2\mpe(q-1)+d},0,0,0,0)$; please note that although the QC term for Step 16 only involves $d^3$, not $d^4$, it still majorizes this.
\begin{denumerate}[label=17.\arabic*]
\item For $i\in\per(\overline{f})\setminus\{d\}$, do the following.

QC: $(d^3\mpe(q-1)2^{d^2\mpe(q-1)+d},0,0,0,0)$.
\begin{denumerate}[label=17.1.\arabic*]
\item Set $S_{n,i}:=(S_{n,i,h})_{h=0,1,\ldots,H_i}$.

QC: $(d^2\mpe(q-1)2^{d^2\mpe(q-1)+d},0,0,0,0)$ for copying, using that
\begin{align*}
&\sum_{h=0}^{H_i}{2^{(h+1)d}(h+1)d}\in O((H_i+1)d2^{(H_i+1)d}) \\
\subseteq &O(d^2\mpe(q-1)2^{d^2\mpe(q-1)+d})
\end{align*}
\end{denumerate}
\end{denumerate}
\item Output the partition-tree register $((\Zcal_i)_{i=0,1,\ldots,d-1},((\Dfrak_n,(S_{n,i})_{i=0,1,\ldots,d}))_{n\in\Ncal})$, and halt.

QC: $(d^24^{d^2\mpe(q-1)+d}\log{q}+d^4\mpe(q-1)4^{d^2\mpe(q-1)+d},0,0,0,0)$ for copying, using that the bit storage cost of $\Xcal_{i,-1}=(\theta_{i,h}(x))_{h=1,2,\ldots,H_i}$ for fixed $i$ is in $O(H_i\log{q})\subseteq O(d\mpe(q-1)\log{q})$, the bit storage cost of $\Xcal_{i,h}$ for fixed $i$ and $h$ is in $O((h+1)d\log{q})$, the total bit storage cost of the recursive tree description list $(\Dfrak_n)_{n\in\Ncal}$ is in $O(|\Ncal|^2\log{q})\subseteq O(d^24^{d^2\mpe(q-1)+d}\log{q})$, the total bit storage cost of the $S_{n,i}$ for $n\in\Ncal$ and $i\in\{0,1,\ldots,d\}$ is in $O(d^4\mpe(q-1)4^{d^2\mpe(q-1)+d})$ (the bit operation cost of Step 17; the combined storage cost of the $S_{n,i}$ for $i\notin\per(\overline{f})$ or $i=d$ is majorized by that for $i\in\per(\overline{f})\setminus\{d\}$), and
\begin{align*}
O(&dH_i\log{q}+d\sum_{h=0}^{H_i}{(h+1)d\log{q}}+d^24^{d^2\mpe(q-1)+d}\log{q} \\
&+d^4\mpe(q-1)4^{d^2\mpe(q-1)+d}) \\
&\subseteq O(d^24^{d^2\mpe(q-1)+d}\log{q}+d^4\mpe(q-1)4^{d^2\mpe(q-1)+d}).
\end{align*}
\end{denumerate}

\subsubsection{Proof of statement (3)}\label{subsubsec5P2P3}

We follow the approach of Subsection \ref{subsec3P4}. If $r=0_{\IF_q}$, then we search for the unique $n\in\Ncal=\{0,1,\ldots,N\}$ such that $S_{n,d}=\emptyset$ (as opposed to $\neg$). This process takes $O(N)\subseteq O(d2^{d^2\mpe(q-1)+d})$ bit operations -- see statement (2) for this bound on $N$. Once $n$ has been found, one may simply read off the description $\Dfrak_n$ of $\Ifrak_n=\Tree_{\Gamma_f}(0_{\IF_q})$ and output the description $[\Dfrak_n]$ of the cyclic sequence in question, which takes $O(d2^{d^2\mpe(q-1)+d}\log{q})$ bit operations for copying.

Henceforth, we assume that $r\not=0_{\IF_q}$. We recall from the proof of statement (4) of Lemma \ref{complexitiesLem2} that there is a natural choice $\omega$ for a primitive element of $\IF_q$. We compute $\lfrak:=\log_{\omega}(r)$\phantomsection\label{not272} in a single fdl query (if $(r,l)$ is obtained as an element of the CRL-list of $f$ that is parametrized by the algorithm for Problem 1, then $r$ is literally specified as a power of $\omega$, and this computation may be skipped). Then we set $i:=\lfrak\bmod{d}$ (taking $O(\log^{1+o(1)}{q})$ bit operations to compute by Lemma \ref{complexitiesLem}(2)), so that $r\in C_i$, and determine the cycle $(i_0,i_1,\ldots,i_{\ell-1})$ of $i=i_0$ under $\overline{f}$, which costs $O(d\log{d})$ bit operations (see also the beginning of the argument for statement (1)). Moreover, we set
\[
r'_{i_0}=r'_i:=\iota_i^{-1}(r)=\frac{\lfrak-i}{d}\in\IZ/s\IZ\,\,\text{ and }\,\,r'_{i_t}:=A_{i_{t-1}}(r'_{i_{t-1}})\text{ for }t=1,2,\ldots,\ell-1,
\]
which takes $O(d\log^{1+o(1)}{q})$ bit operations altogether. Finally, for $t=0,1,\ldots,\ell-1$, we compute
\[
\Acal_{i_t}:=A_{i_t}A_{i_{(t+1)\bmod{\ell}}}\cdots A_{i_{(t+\ell-1)\bmod{\ell}}},
\]
which takes $O(d^2\log^{1+o(1)}{q})$ bit operations for all $t$ together (see also the beginning of the argument for statement (1)).

Until further notice, we assume that $t$ is fixed. In the notation of Subsection \ref{subsec3P4}, we have $\Acal_{i_t}(x)=\overline{\alpha}_{i_t}x+\overline{\beta}_{i_t}$. We compute $\gcd(\overline{\alpha}_{i_t},s)$, find the smallest non-negative integer $L_{i_t}$ such that $\gcd(\overline{\alpha}_{i_t}^{L_{i_t}},s)=\gcd(\overline{\alpha}_{i_t}^{L_{i_t}+1},s)=:s''_{i_t}$\phantomsection\label{not273}, compute $s'_{i_t}:=s/s''_{i_t}$ and $\ffrak_{i_t}:=\sum_{z=0}^{L_{i_t}-1}{\overline{\alpha}_{i_t}^z\overline{\beta}_{i_t}}\bmod{s''_{i_t}}$\phantomsection\label{not274} (the unique periodic point of the reduction of $\Acal_{i_t}$ modulo $s''_{i_t}$). Together, these tasks can be performed within $O(\log^{2+o(1)}{q})$ bit operations if binary search is used to find $L_{i_t}$ (see also the paragraph in the proof of statement (1) where $L_i$ is introduced).

As we observed in Subsection \ref{subsec3P4} already, because each point on the cycle of $r'_{i_t}$ under $\Acal_{i_t}$ (which corresponds to the cycle of $f^t(r)$ under $f^{\ell}$) is congruent to $\ffrak_{i_t}$ modulo $s''_{i_t}$, some of the truth values of the spanning congruences of $\Pcal_{i_t}=\Pfrak(x\equiv\bfrak_{i_t,j}\Mod{\afrak_{{i_t},j}}: j=1,2,\ldots,m_{i_t})$ may be constant along that cycle. More specifically,
\begin{itemize}
\item if $\ffrak_{i_t}\not\equiv\bfrak_{i_t,j}\Mod{\gcd(\afrak_{i_t,j},s''_{i_t})}$, then the $j$-th spanning congruence of $\Pcal_{i_t}$ is always false along the cycle. For such $j$, we set $\nu_j:=\neg$ and say that $j$ is \emph{of type I}\phantomsection\label{term79}.
\item if $\afrak_{i_t,j}\mid s''_{i_t}$ and $\ffrak_{i_t}\equiv\bfrak_{i_t,j}\Mod{\afrak_{i_t,j}}$, then the $j$-th spanning congruence of $\Pcal_{i_t}$ is always true along the cycle. For such $j$, we set $\nu_j:=\emptyset$ and say that $j$ is \emph{of type II}\phantomsection\label{term80}.
\end{itemize}
Going through the spanning congruences of $\Pcal_{i_t}$ and creating complete lists of those $j\in\{1,2,\ldots,m_{i_t}\}$ that are of type I, respectively of type II, takes $O(m_{i_t}\log^{1+o(1)}{q})\subseteq O(d^2\mpe(q-1)\log^{1+o(1)}{q})$ bit operations. For the bound $m_{i_t}\in O(d^2\mpe(q-1))$, we recall from Subsection \ref{subsec3P3} that $\Pcal_{i_t}=\Qcal_{i_t,H_{i_t}}=\bigwedge_{k=0}^{H_{i_t}}{\lambda_{i_{t-k}}^k(\Rcal_{i_{t-k}})}\wedge\Ucal_{i_t}$, and that the length of the \enquote{standard} spanning congruence sequence of $\lambda_{i_{t-k}}^k(\Rcal_{i_{t-k}})$ (stored in the partition-tree register as $\Xcal_{i_t,k}$) is $n_{i_{t-k}}\in O(d)$, while the one of $\Ucal_{i_t}=\Pfrak(\theta_{i_t,h}(x): h=0,1,\ldots,H_{i_t})$, which is stored as $\Xcal_{i_t,-1}$, has length $H_i$. Moreover, $H_i\in O(d\mpe(q-1))$.

Next, we compute the affine discrete logarithm $\lfrak_{i_t,j}:=\log^{(\afrak_{i_t,j})}_{\Acal_{i_t}}(r'_{i_t},\bfrak_{i_t,j})$ (as defined in Subsection \ref{subsec2P4}) for those $j\in\{1,2,\ldots,m_{i_t}\}$ that are neither of type I nor of type II, which takes at worst $O(m_{i_t})\subseteq O(d^2\mpe(q-1))$ mdl queries and $O(m_{i_t}\log^{1+o(1)}{q})\subseteq O(d^2\mpe(q-1)\log^{1+o(1)}{q})$ bit operations by the reduction argument in Subsection \ref{subsec2P4}. If $j\in\{1,2,\ldots,m_{i_t}\}$ is not of type I or II and $\lfrak_{i_t,j}=\infty$, then the $j$-th spanning congruence of $\Pcal_{i_t}$ is always false along the cycle of $r'_{i_t}$ under $\Acal_{i_t}$; for such $j$, we set $\nu_j:=\neg$ and say that $j$ is \emph{of type III}\phantomsection\label{term81}.

In the notation of Subsection \ref{subsec3P4}, the set $I_{i_t}$ consists precisely of those $j\in\{1,2,\ldots,m_{i_t}\}$ that are of one of the types I--III. Now, for each $j\in\{1,2,\ldots,m_{i_t}\}$ such that $j\notin I_{i_t}$, we compute the cycle length $l_{i_t,j}$ of $r'_{i_t}$ under $\Acal_{i_t}$ modulo $\afrak_{i_t,j}$ -- by the reduction argument of Subsection \ref{subsec2P4}, this takes $O(m_{i_t})\subseteq O(d^2\mpe(q-1))$ mord queries and $O(m_{i_t}\log^{1+o(1)}{q})\subseteq O(d^2\mpe(q-1)\log^{1+o(1)}{q})$ bit operations.

We recall that the $f$-cycle length $l$ of $r$ is given to us as part of the input. Associated with each $j\in\{1,2,\ldots,m_{i_t}\}\setminus I_{i_t}$, we have the $(l/\ell)$-congruence $y\equiv\lfrak_{i_t,j}\Mod{l_{i_t,j}}$, which holds precisely for those $y\in\IZ/(l/\ell)\IZ$ such that the $j$-th spanning congruence of $\Pcal_{i_t}$ becomes true when substituting $x:=\Acal_{i_t}^y(r'_{i_t})$. As in Subsection \ref{subsec3P4}, we set
\[
\Pcal^{(i_t)}:=\Pfrak(y\equiv\lfrak_{i_t,j}\Mod{l_{i_t,j}}: j\in\{1,2,\ldots,m_{i_t}\}\setminus I_{i_t}),
\]
an arithmetic partition of $\IZ/(l/\ell)\IZ$. The logical sign tuples which parametrize the blocks of $\Pcal^{(i_t)}$ are from the set $\{\emptyset,\neg\}^{\{1,2,\ldots,m_{i_t}\}\setminus I_{i_t}}$ and are of the form $\vec{\nu}=(\nu_j)_{j\in\{1,\ldots,m_{i_t}\}\setminus I_{i_t}}$. For such a tuple $\vec{\nu}$, we set $\vec{\nu}^+:=(\nu_j)_{j=1,2,\ldots,m_{i_t}}$\phantomsection\label{not275} (the positions indexed by $j\in I_{i_t}$ are filled with the constant logical signs for such $j$ that were defined above). Conversely, for $\vec{\nu'}\in\{\emptyset,\neg\}^{m_{i_t}}$, we denote by $\vec{\nu'}^-$\phantomsection\label{not276} the projection of $\vec{\nu'}$ to $\{\emptyset,\neg\}^{\{1,2,\ldots,m_{i_t}\}\setminus I_{i_t}}$.

It follows that if $y\in\IZ/(l/\ell)\IZ$ is contained in the block $\Bcal(\Pcal^{(i_t)},\vec{\nu})$ of $\Pcal^{(i_t)}$, then $\Acal_{i_t}^y(r'_{i_t})$ is always contained in the block $\Bcal(\Pcal_{i_t},\vec{\nu}^+)$ of $\Pcal_{i_t}$, whence $\Tree_{\Gamma_f}(f^{t+\ell y}(r))=\Ifrak_{n'_{i_t}(\vec{\nu})}$ where $n'_{i_t}(\vec{\nu})$\phantomsection\label{not277} is the unique $n\in\{0,1,\ldots,N\}$ such that $\vec{\nu}^+\in S_{i_t,n}$.

We remind the reader that we wish to output a compact description of the cyclic sequence of rooted tree isomorphism types that characterizes the connected component of $\Gamma_f$ containing $r$. In order to do so, we compute sets $\overline{S}_{n,i_t}$\phantomsection\label{not277P5} for $n\in\{0,1,\ldots,N\}$ (and our still fixed $t$) that are defined as follows. The elements of $\overline{S}_{n,i_t}$ are precisely those logical sign tuples $\vec{\nu}\in\{\emptyset,\neg\}^{I_{i_t}}$ for which $n'_{i_t}(\vec{\nu})=n$.

Let us describe how to compute these sets $\overline{S}_{n,i_t}$. We go through the $N+1\in O(d2^{d^2\mpe(q-1)+d})$ values for $n$ (see also the argument for $r=0_{\IF_q}$ at the beginning of this subsubsection), and for each $n$, we go through the $|S_{n,i_t,H_{i_t}}|$ logical sign tuples $\vec{\nu''}$ in $S_{n,i_t,H_{i_t}}$ (the last entry of $S_{n,i_t}$) in their listed order, which is lexicographic. For each $\vec{\nu''}$, we set $\vec{\nu'}:=\vec{\nu''}\diamond\vec{\xi}_{i_t,H_{i_t}}=(\nu'_j)_{j=1,2,\ldots,m_{i_t}}$. We check whether $\nu'_j=\nu_j$ for all $j\in I_{i_t}$ (which takes $O(m_{i_t}\log{q})\subseteq O(d^2\mpe(q-1)\log{q})$ bit operations). If not, we move on to the next $\vec{\nu''}$ and associated $\vec{\nu'}$, otherwise we add $\vec{\nu'}^-$ to $\overline{S}_{n,i_t}$ as a new element (again taking $O(m_{i_t}\log{q})\subseteq O(d^2\mpe(q-1)\log{q})$ bit operations, for copying). Because $N\in O(d2^{d^2\mpe(q-1)+d})$ and $\sum_{n=0}^N{|S_{n,i_t,H_{i_t}}|}\leq 2^{(H_{i_t}+1)d}\leq 2^{d^2\mpe(q-1)+d}$, and we only need $O(1)$ bit operations to deal with an $n$ for which $S_{n,i_t}=\emptyset$, we conclude that the overall bit operation cost of computing the sets $\overline{S}_{n,i_t}$ (for our fixed $t$) is in
\[
O(d2^{d^2\mpe(q-1)+d}+m_{i_t}2^{(H_{i_t}+1)d}\log{q})\subseteq O(d^2\mpe(q-1)2^{d^2\mpe(q-1)+d}\log{q}),
\]
and the constructed arrays representing those sets are lexicographically ordered.

Let us now finally \enquote{unfix} $t$ again. For all the $O(d)$ values of $t\in\{0,1,\ldots,\ell-1\}$ together, a $q$-bounded query complexity of all computations we described after agreeing to fix $t$ is
\begin{align}\label{lastOEq}
\notag (&d\log^{2+o(1)}{q}+d^3\mpe(q-1)\log^{1+o(1)}{q}+d^3\mpe(q-1)2^{d^2\mpe(q-1)+d}\log{q}, \\
&0,d^3\mpe(q-1),d^3\mpe(q-1),0).
\end{align}
Here, the first summand in the bit operation component, as well as the specified queries, cover the cost of everything except the computation of the sets $\overline{S}_{n,i_t}$ themselves, which is instead covered by the second summand in the bit operation component (without needing any queries). Considering the query complexities of all other involved computations (see the detailed steps of the algorithm below), we find that only the second entry, $0$, in formula (\ref{lastOEq}) needs to be replaced by $d$ in order to obtain a valid $q$-bounded query complexity of the entire algorithm for solving Problem 3.

As in the previous two subsubsections, we conclude with some pseudocode for this algorithm, which includes $q$-bounded query complexities for each step.

\begin{denumerate}[label=\arabic*]
\item If $r=0_{\IF_q}$, then search for the unique $n\in\Ncal=\{0,1,\ldots,N\}$ such that $S_{n,d}=\emptyset$, and output the following and halt: \enquote{The cyclic sequence of rooted tree isomorphism types which encodes the digraph isomorphism type of the connected component in question is $[\Dfrak_n]$.}

QC: $(d2^{d^2\mpe(q-1)+d}\log{q},0,0,0,0)$.
\item Compute the induced function $\overline{f}$ on $\{0,1,\ldots,d\}$ and the affine maps $A_i$ of $\IZ/s\IZ$.

QC: $(d\log^{1+o(1)}{q},d,0,0,0)$.
\item Let $\omega$ be the \enquote{natural} choice of primitive element of $\IF_q$ (see the proof of Lemma \ref{complexitiesLem2}(4)), and determine $\lfrak:=\log_{\omega}(r)$.

QC: $(\log{q},1,0,0,0)$.
\item Set $i:=\lfrak\bmod{d}$.

QC: $(\log^{1+o(1)}{q},0,0,0,0)$.
\item Determine the cycle $(i_0,i_1,\ldots,i_{\ell-1})$ of $i=i_0$ under $\overline{f}$.

QC: $(d\log{d},0,0,0,0)$.
\item Compute $r'_{i_0}:=\frac{\lfrak-i}{d}$ and $r'_{i_t}:=A_{i_{t-1}}(r'_{i_{t-1}})$ for $t=1,2,\ldots,\ell-1$.

QC: $(d\log^{1+o(1)}{q},0,0,0,0)$.
\item For $t=0,1,\ldots,\ell-1$, compute $\Acal_{i_t}:=A_{i_t}A_{i_{(t+1)\bmod{\ell}}}\cdots A_{i_{(t+\ell-1)\bmod{\ell}}}: x\mapsto\overline{\alpha}_{i_t}x+\overline{\beta}_{i_t}$.

QC: $(d^2\log^{1+o(1)}{q},0,0,0,0)$.
\item For each $t=0,1,\ldots,\ell-1$, do the following. QC:
\begin{align*}
(&d\log^{2+o(1)}{q}+d^3\mpe(q-1)\log^{1+o(1)}{q}+d^3\mpe(q-1)2^{d^2\mpe(q-1)+d}\log{q}, \\
&0,d^3\mpe(q-1),d^3\mpe(q-1),0).
\end{align*}
\begin{denumerate}[label=8.\arabic*]
\item Compute $\gcd(\overline{\alpha}_{i_t},s)$.

QC: $(\log^{1+o(1)}{q},0,0,0,0)$.
\item Find the smallest non-negative integer $L_{i_t}$ such that
\[
\gcd(\overline{\alpha}_{i_t}^{L_{i_t}},s)=\gcd(\overline{\alpha}_{i_t}^{L_{i_t}+1},s)=:s''_{i_t}.
\]
QC: $(\log^{2+o(1)}{q},0,0,0,0)$.
\item Compute $s'_{i_t}:=s/s''_{i_t}$.

QC: $(\log^{1+o(1)}{q},0,0,0,0)$.
\item Compute $\ffrak_{i_t}:=\sum_{z=0}^{L_{i_t}-1}{\overline{\alpha}_{i_t}^z\overline{\beta}_{i_t}}\bmod{s''_{i_t}}$. In doing so, do \emph{not} add up all the summands, but use the geometric sum formula
\[
\sum_{z=0}^{L_{i_t}-1}{\overline{\alpha}_{i_t}^z}=
\begin{cases}
\frac{\overline{\alpha}_{i_t}^{L_{i_t}}-1}{\overline{\alpha}_{i_t}-1}, & \text{if }\overline{\alpha}_{i_t}\not=1, \\
L_{i_t}\overline{\alpha}_{i_t}, & \text{if }\overline{\alpha}_{i_t}=1,
\end{cases}
\]
and the fact that $L_{i_t}\in O(\log{q})$.

QC: $(\log^{2+o(1)}{q},0,0,0,0)$.
\item Read off $\Pcal_{i_t}=\Pfrak(x\equiv\bfrak_{i_t,j}\Mod{\afrak_{i_t,j}}: j=1,2,\ldots,m_{i_t})$ from the given partition-tree register.

QC: $(d^2\mpe(q-1)\log{q},0,0,0,0)$.
\item For $j=1,2,\ldots,m_{i_t}$, do the following.

QC: $(d^2\mpe(q-1)\log^{1+o(1)}{q},0,d^2\mpe(q-1),d^2\mpe(q-1),0)$.
\begin{denumerate}[label=8.6.\arabic*]
\item Check whether $\ffrak_{i_t}\not\equiv\bfrak_{i_t,j}\Mod{\gcd(\afrak_{i_t,j},s''_{i_t})}$. If so, set $\nu_j:=\neg$, and $\Type_{\mathrm{I}}(j):=\mathrm{true}$, $\Type_{\mathrm{II}}(j):=\mathrm{false}$ and $\Type_{\mathrm{III}}(j):=\mathrm{false}$. If not, just set $\Type_{\mathrm{I}}(j):=\mathrm{false}$.

QC: $(\log^{1+o(1)}{q},0,0,0,0)$.
\item If $\Type_{\mathrm{I}}(j)=\mathrm{false}$, then check whether
\[
\afrak_{i_t,j}\mid s''_{i_t}\text{ and }\ffrak_{i_t}\equiv\bfrak_{i_t,j}\Mod{\afrak_{i_t,j}}.
\]
If so, set $\nu_j:=\emptyset$, and $\Type_{\mathrm{II}}(j):=\mathrm{true}$ and $\Type_{\mathrm{III}}(j):=\mathrm{true}$. If not, just set $\Type_{\mathrm{II}}(j):=\mathrm{false}$.

QC: $(\log^{1+o(1)}{q},0,0,0,0)$.
\item If $\Type_{\mathrm{I}}{j}=\Type_{\mathrm{II}}(j)=\mathrm{false}$, then compute $\lfrak_{i_t,j}:=\log_{\Acal_{i_t}}^{(\afrak_{i_t,j})}(r'_{i_t},\bfrak_{i_t,j})$, and check whether $\lfrak_{i_t,j}=\infty$. If so, set $\nu_j:=\neg$, and $\Type_{\mathrm{III}}(j):=\mathrm{true}$. Otherwise, just set $\Type_{\mathrm{III}}(j):=\mathrm{false}$.

QC: $(\log^{1+o(1)}{q},0,1,0,0)$.
\item If $\Type_{\mathrm{I}}(j)=\Type_{\mathrm{II}}(j)=\Type_{\mathrm{III}}(j)=\mathrm{false}$, then compute the cycle length $l_{i_t,j}$ of $r'_{i_t}$ under $\Acal_{i_t}$ modulo $\afrak_{i_t,j}$.

QC: $(\log^{1+o(1)}{q},0,0,1,0)$.
\end{denumerate}
\item Set
\begin{align*}
\Pcal^{(i_t)}:=\Pfrak(&y\equiv l_{i_t,j}\Mod{\lfrak_{i_t,j}}: 1\leq j\leq m_{i_t}, \\
&\Type_{\mathrm{I}}(j)=\Type_{\mathrm{II}}(j)=\Type_{\mathrm{III}}(j)=\mathrm{false}).
\end{align*}
QC: $(d^2\mpe(q-1)\log{q},0,0,0,0)$.
\item For each $n\in\Ncal=\{0,1,\ldots,N\}$, do the following.

QC: $(d^2\mpe(q-1)2^{d^2\mpe(q-1)+d}\log{q},0,0,0,0)$.
\begin{denumerate}[label=8.8.\arabic*]
\item Initialize the set $\overline{S}_{n,i_t}$ to be $\emptyset$.

QC: $(\log{q},0,0,0,0)$.
\item For each $\vec{\nu''}\in S_{n,i_t,H_{i_t}}$ (last entry of $S_{n,i_t}$ from the partition-tree register from the input), do the following.

QC: $(d^2\mpe(q-1)2^{d^2\mpe(q-1)+d}\log{q},0,0,0,0)$.
\begin{denumerate}[label=8.8.2.\arabic*]
\item Set $\vec{\nu'}:=\vec{\nu''}\diamond\vec{\xi}_{i_t,H_{i_t}}=(\nu'_j)_{j=1,2,\ldots,m_{i_t}}$.

QC: $(d\mpe(q-1)\log{q},0,0,0,0)$.
\item Check whether $\nu'_j=\nu_j$ for all $j\in\{1,2,\ldots,m_{i_t}\}$ such that one of $\Type_{\mathrm{I}}(j)$, $\Type_{\mathrm{II}}(j)$ or $\Type_{\mathrm{III}}(j)$ is $\mathrm{true}$. If so, add $\vec{\nu'}^-$ (which is $\vec{\nu'}$ with all components corresponding to one of the three types deleted) to $\overline{S}_{n,i_t}$ as a new element.

QC: $(d^2\mpe(q-1)\log{q},0,0,0,0)$.
\end{denumerate}
\end{denumerate}
\end{denumerate}
\item Output the following: \enquote{Consider an iterate $f^z(r)$ where $z\in\IZ/l\IZ$. Let $t:=z\bmod{\ell}$ and $y:=(z-t)/\ell\in\IZ/(l/\ell)\IZ$. Depending on $t$, we have the arithmetic partition $\Pcal^{(i_t)}$ of $\IZ/(l/\ell)\IZ$, and the block of $\Pcal^{(i_t)}$ in which $y$ is contained controls the isomorphism type of $\Tree_{\Gamma_f}(f^z(r))$. More precisely, for each $n\in\Ncal$, we have the set $\overline{S}_{n,i_t}$ of logical sign tuples such that $\Tree_{\Gamma_f}(f^z(r))=\Ifrak_n$ if and only if $y\in\Bcal(\Pcal^{(i_t)},\vec{\nu})$ for some $\vec{\nu}\in\overline{S}_{n,i_t}$. The arithmetic partitions $\Pcal^{(i_t)}$ and associated logical sign tuple sets $\overline{S}_{n,i_t}$ are as follows.}, followed by printing (the computed spanning congruence sequence of) $\Pcal^{(i_t)}$ and the sets $\overline{S}_{n,i_t}$ for each $t=0,1,\ldots,\ell-1$. Then halt.

QC: $(d^3\mpe(q-1)2^{d^2\mpe(q-1)+d}\log{q},0,0,0,0)$.
\end{denumerate}

\subsection{The isomorphism problem for functional graphs of generalized cyclotomic mappings}\label{subsec5P3}

In this subsection, we consider the algorithmic problem of deciding whether the functional graphs of two given generalized cyclotomic mappings of $\IF_q$ (each of a fixed index, but not necessarily both of the same index) are isomorphic digraphs. We note that thanks to Babai \cite{Bab18a} (see also Helfgott's expository article \cite{Hel19a}, or its English translation \cite{HBD17a}), a general algorithm for deciding the isomorphism of (di)graphs is known that is quasipolynomial in the number of vertices. Assuming that the graphs in question are given by specifying their edges (as ordered or unordered pairs of vertices) individually, this algorithm is quasipolynomial in the input length (one can assume without loss of generality that there are no isolated vertices, and thus that the number of vertices and the number of edges are within quadratic bounds of each other).

However, we do not specify our functional graphs $\Gamma_f$ in this form; rather, we specify $f$ in its cyclotomic form (\ref{cyclotomicFormEq}), so our input length is only in $O(d\log{q})$, while the vertex number is $|\IF_q|=q$. We believe that it is a hard problem to decide whether the functional graphs of two generalized cyclotomic mappings of $\IF_q$, each of an index that is at most $d$, are isomorphic (i.e., that this problem is not generally solvable in polynomial time in $\log{q}$ for fixed $d$). However, for some special cases, efficient algorithms can be developed, and it is the purpose of this subsection to present such special cases and the associated decision algorithms.

\subsubsection{Special case: Index 1}\label{subsubsec5P3P1}

Generalized cyclotomic mappings of $\IF_q$ of index $1$ are the same as monomial mappings $x\mapsto ax^r$, where $a\in\IF_q$ and $r\in\{1,\ldots,q-1\}$ (if one has $r_0=0$ in formula (\ref{cyclotomicFormEq}) with $d=1$, then one may replace it by $q-1$ to get a formula that works on all of $\IF_q$). To get the class of all monomial mappings of $\IF_q$, one must include the case \enquote{$r=0$} (in which $0_{\IF_q}$ is not necessarily fixed), and we do so in this subsubsection. Thanks to earlier work of Deng \cite{Den13a}, it is easy to decide whether two monomial mappings of $\IF_q$, say $f:x\mapsto ax^r$ and $f':x\mapsto a'x^{r'}$, have isomorphic functional graphs, as we explain below.

First, we note that if $f$ is constant, which happens if and only if $r=0$ or $a=0_{\IF_q}$, then $\Gamma_f\cong\Gamma_{f'}$ if and only if $f'$ is constant as well, i.e., if and only if $r'=0$ or $a'=0_{\IF_q}$. Checking whether this special case applies only takes $O(\log{q})$ bit operations (for scanning the values of $a,a',r,r'$).

We may thus assume that $r,r',a,a'$ all are non-zero. Because $a,a'\not=0_{\IF_q}$, we have $f(\IF_q^{\ast})\cup f'(\IF_q^{\ast})\subseteq\IF_q^{\ast}$. Following the procedure described in our introduction, we can associate with $f$, respectively $f'$, an affine map $A$, respectively $A'$, of $\IZ/(q-1)\IZ$ such that $\Gamma_f\cong\Gamma_{f'}$ if and only if $\Gamma_A\cong\Gamma_{A'}$. Computing $A$ and $A'$ has $q$-bounded query complexity $(\log^{1+o(1)}{q},1,0,0,0)$; beside some simple arithmetic, it requires the computation of the discrete logarithms of $a$ and $a'$ with base $\omega$ (the \enquote{natural choice} of primitive element of $\IF_q$ -- see the proof of Lemma \ref{complexitiesLem2}(4)).

In order to decide whether $\Gamma_A\cong\Gamma_{A'}$, we use the following result, which is a variant of \cite[Theorem 11]{Den13a}.

\begin{theoremmm}\label{dengTheo}
Let $m=p_1^{v_1}\cdots p_K^{v_K}$ be a positive integer with its prime factorization displayed. Let $a,a',b,b'\in\IZ/m\IZ$, and let us denote by $A$, respectively $A'$, the affine map $x\mapsto ax+b$, respectively $x\mapsto a'x+b'$, of $\IZ/m\IZ$. Moreover, we define
\begin{itemize}
\item $\Yfrak:=\{p_j: j\in\{1,2,\ldots,K\}, p_j\nmid a, A\bmod{p_j^{v_j}}\text{ has a fixed point}\}$\phantomsection\label{not278}, and analogously for $\Yfrak'$, with $a'$ and $A'$ in place of $a$ and $A$; and
\item $l$, respectively $l'$, to be the minimal cycle length of $A$, respectively $A'$, on its respective periodic points.
\end{itemize}
Then $\Gamma_A\cong\Gamma_{A'}$ if and only if both of the following hold:
\begin{enumerate}
\item $\gcd(a,m)=\gcd(a',m)$, and
\item $\Yfrak=\Yfrak'$, $l=l'$, $\ord_{p^{\nu_p(m)}}(a^l)=\ord_{p^{\nu_p(m)}}((a')^l)$ for all $p\in\Yfrak$, and if $2\in\Yfrak$ and $\nu_2(m)>1$, then $\ord_4(a^l)=\ord_4((a')^l)$.
\end{enumerate}
\end{theoremmm}

\begin{proof}
This is the same as \cite[Theorem 11]{Den13a} except that in condition (2), we do not demand that $\ord_t(a^l)=\ord_t((a')^l)$ for \emph{all} divisors $t$ of $m$ whose prime divisors are in $\Yfrak$. However, our condition (2) is enough, because if it holds and $t=\prod_{p\in\Yfrak}{p^{v'_p}}$ divides $m$, then for all $p\in\Yfrak$, one has
\begin{equation}\label{dengEq}
\ord_{p^{v'_p}}(a^l)=\ord_{p^{v'_p}}((a')^l),
\end{equation}
and thus
\[
\ord_t(a^l)=\lcm_{p\in\Yfrak}{\ord_{p^{v'_p}}(a^l)}=\lcm_{p\in\Yfrak}{\ord_{p^{v'_p}}((a')^l)}=\ord_t((a')^l).
\]
To see that formula (\ref{dengEq}) holds, we make a case distinction.
\begin{itemize}
\item If $p>2$, let us fix a primitive root $\rfrak$ of $\IZ/p^{\nu_p(m)}\IZ$. Because $\ord_{p^{\nu_p(m)}}(a^l)=\ord_{p^{\nu_p(m)}}((a')^l)$, it follows that modulo $p^{\nu_p(m)}$, we have $a^l=\rfrak^k$ and $(a')^l=\rfrak^{k'}$ where $\gcd(k,\phi(p^{\nu_p(m)}))=\gcd(k',\phi(p^{\nu_p(m)}))$. Since $\phi(p^{v'_p})$ divides $\phi(p^{\nu_p(m)})$, it follows that $\gcd(k,\phi(p^{v'_p}))=\gcd(k',\phi(p^{v'_p}))$, whence $a^l$ and $(a')^l$ are also of the same multiplicative order modulo $p^{v'_p}$, as required.
\item If $p=2$, then formula (\ref{dengEq}) holds by assumption if $\nu_2(m)\leq 2$, so let us assume that $\nu_2(m)\geq3$. We may also assume that $v'_2<\nu_2(m)$, because there is nothing to show if $v'_2=\nu_2(m)$. By the structure of $\IZ/2^{\nu_2(m)}\IZ$, modulo $2^{\nu_2(m)}$ we can write $a^l=(-1)^{\epsilon}5^k$ and $(a')^l=(-1)^{\epsilon'}5^{k'}$ with $\epsilon,\epsilon'\in\{0,1\}$. Since $a^l$ and $(a')^l$ have (by assumption) the same multiplicative order modulo $4$, we infer that $\epsilon=\epsilon'$. It is not hard to see that if $\ord_{2^{\nu_2(m)}}(a^l)=\ord_{2^{\nu_2(m)}}((a')^l)\in\{1,2\}$, then $\ord_{2^{v'_2}}(a^l)=\ord_{2^{v'_2}}((a')^l)=1$. Indeed, this is clear if both orders modulo $2^{\nu_2(m)}$ are equal to $1$, and if both orders modulo $2^{\nu_2(m)}$ are equal to $2$, then $\ord_{2^{\nu_2(m)}}(5^k),\ord_{2^{\nu_2(m)}}(5^{k'})$ both divide $2$, whence due to $v'_2<\nu_2(m)$, one has $5^k\equiv 5^{k'}\equiv1\Mod{2^{v'_2}}$, as follows by comparing the multiplicative orders of $5$ modulo $2^{\nu_2(m)}$ and modulo $2^{v'_2}$, respectively. We may thus assume that the common multiplicative order modulo $2^{\nu_2(m)}$ of $a^l$ and $(a')^l$ is strictly greater than $2$. Then
\[
\ord_{2^{\nu_2(m)}}(5^k)=\ord_{2^{\nu_2(m)}}(a^l)=\ord_{2^{\nu_2(m)}}((a')^l)=\ord_{2^{\nu_2(m)}}(5^{k'}),
\]
and with an analogous argument to the one for \enquote{$p>2$}, we conclude that $\ord_{2^{v'_2}}(5^k)=\ord_{2^{v'_2}}(5^{k'})$. Therefore,
\begin{align*}
\ord_{2^{v'_2}}(a^l) &=\lcm(\ord_{2^{v'_2}}((-1)^{\epsilon}),\ord_{2^{v'_2}}(5^k))=\lcm(\ord_{2^{v'_2}}((-1)^{\epsilon}),\ord_{2^{v'_2}}(5^{k'})) \\
&=\ord_{2^{v'_2}}((a')^l),
\end{align*}
as required.
\end{itemize}
\end{proof}

Using Theorem \ref{dengTheo}, we can prove the following.

\begin{corrrollary}\label{dengCor}
Let $m$ be a positive integer, and let $A$ and $A'$ be affine maps of $\IZ/m\IZ$. Deciding whether $\Gamma_A\cong\Gamma_{A'}$ takes $m$-bounded query complexity
\[
(\log^{2+o(1)}{m},0,\log{m},1,0)
\]
and $m$-bounded Las Vegas dual complexity
\[
(\log^{8+o(1)}{m},\log^{4+o(1)}{m},\log^2{m}).
\]
\end{corrrollary}

\begin{proof}
We use the notation from Theorem \ref{dengTheo}, including that $A(x)=ax+b$ and $A'(x)=a'x+b'$. We argue that the conditions given in Theorem \ref{dengTheo} can be verified within the specified $m$-bounded query complexity (the asserted Las Vegas dual complexity then follows readily using Lemma \ref{LVComplexityLem}). First, we compute and compare $\gcd(a,m)$ and $\gcd(a',m)$, which takes $O(\log^{1+o(1)}{m})$ bit operations by Lemma \ref{complexitiesLem}(8). Next, we factor $m=p_1^{v_1}\cdots p_K^{v_K}$ using a single $m$-bounded mdl query (see also the beginning of Subsection \ref{subsec5P2}).

Following that, we compute the sets $\Yfrak$ and $\Yfrak'$ and compare them. For this, we observe that $A\bmod{p_j^{v_j}}$, respectively $A'\bmod{p_j^{v_j}}$, has a fixed point if and only if the congruence $(a-1)x\equiv-b\Mod{p_j^{v_j}}$, respectively $(a'-1)x\equiv-b'\Mod{p_j^{v_j}}$, is solvable, which holds if and only if $\gcd(a-1,p_j^{v_j})\mid b$, respectively $\gcd(a'-1,p_j^{v_j})\mid b'$. To check whether this holds, we read off the binary representation of $p_j^{v_j}$ from the the output of the above mdl query, then carry out the relevant arithmetic in either case, which takes $O(\log^{1+o(1)}{m})$ bit operations for a single $j$ by Lemma \ref{complexitiesLem}(1,3,8). Because there are $O(\log{m})$ distinct values of $j$, it takes $O(\log^{2+o(1)}{m})$ bit operations altogether to compute $\Yfrak$ and $\Yfrak'$ and check whether they are equal.

If $\Yfrak=\Yfrak'$, we next compute the minimal cycle lengths $l$ and $l'$ and check if they are equal. It follows from our Table \ref{crlListPrimaryTable} (or \cite[Tables 3 and 4]{BW22b}, in which cycle types, not CRL-lists, of affine maps of finite primary cyclic groups are displayed and from which the cycle lengths can be read off more directly) that modulo a prime power, the cycle lengths of an affine permutation are linearly ordered under divisibility. Therefore, $l$, respectively $l'$, is the least common multiple of the smallest cycle lengths of $A$, respectively of $A'$, modulo the $p_j^{v_j}$ for those $j\in\{1,2,\ldots,K\}$ such that $p_j\nmid a$, respectively $p_j\nmid a'$. Those minimal cycle lengths can be computed according to our Table \ref{crlListPrimaryTable} (or \cite[Tables 3 and 4]{BW22b}). More specifically, we go through $j=1,2,\ldots,K$, and for each of these values, we do the following.
\begin{itemize}
\item We check whether $p_j\nmid a$, respectively $p_j\nmid a'$, taking $O(\log^{1+o(1)}{m})$ bit operations.
\item If so, we check which case in Table \ref{crlListPrimaryTable} (or \cite[Table 3 or 4 respectively]{BW22b}) applies, which also takes $O(\log^{1+o(1)}{m})$ bit operations.
\item Finally, we compute the minimal cycle length according to the case-specific formula and factor it, using one mdl query and $O(\log^{1+o(1)}{m})$ bit operations.
\end{itemize}
For all $j$ together, this process takes $m$-bounded query complexity
\[
(\log^{2+o(1)}{m},0,\log{m},0,0).
\]
Following this, we determine the least common multiple of the computed cycle lengths. These lengths are already factored, so one only needs to go through the $O(\log{m})$ primes dividing at least one of those cycle lengths (we note that each of those primes is a divisor of $p(p-1)$ for some prime $p\mid m$), and for each of them, we compute the largest exponent with which it occurs. For this, we need to scan the obtained factorization of the minimal cycle length modulo $p_j^{v_j}$, which is a bit string of length in $O(\log{p_j^{v_j}})$, for each $j$, and we need to compare the stored intermediate maximum with one of the prime exponents in it, which is a bit string of length in $O(\log\log{p_j^{v_j}})$. Altogether, the computation of $l$ and $l'$ takes $O(\log^{2+o(1)}{m})$ bit operations, and checking whether $l$ and $l'$ are equal takes a mere $O(\log{m})$ bit operations.

If $l=l'$, we next compute $a^l\bmod{m}$ and $(a')^l\bmod{m}$, taking $O(\log^{2+o(1)}{m})$ bit operations. Following that, for each $p\in\Yfrak$, we compute $\ord_{p^{\nu_p(m)}}(a^l)$ and $\ord_{p^{\nu_p(m)}}((a')^l)$, which can be done for all $p$ together with just two mord queries, and check if they are equal, which takes $O(\log{m})$ bit operations for all $p$ together. Finally, if necessary, we compute $\ord_4(a^l)$ and $\ord_4((a')^l)$ and check if they are equal, which can be done with $O(\log^{1+o(1)}{m})$ bit operations (no queries necessary).
\end{proof}

With regard to our application to generalized cyclotomic mappings of degree $1$, we note the following consequence of Corollary \ref{dengCor}.

\begin{corrrollary}\label{dengCor2}
Let $f$ and $f'$ be monomial mappings of $\IF_q$, given in polynomial form. Within $q$-bounded query complexity
\[
(\log^{2+o(1)}{q},1,\log{q},1,0),
\]
or $q$-bounded Las Vegas dual complexity
\[
(\log^{8+o(1)}{q},\log^{4+o(1)}{q},\log^2{q})
\]
one can decide whether $\Gamma_f\cong\Gamma_{f'}$.
\end{corrrollary}

\begin{proof}
As explained at the beginning of this subsubsection, one first deals with the case where at least one of $f$ or $f'$ is constant through simple scans of the input, taking $O(\log{q})$ bit operations. If not, then $f(x)=ax^r$ and $f'(x)=a'x^{r'}$ where $a,a',r,r'$ all are non-zero. Under logarithmization, the restriction of $f$, respectively $f'$, to $\IF_q^{\ast}$ corresponds to the affine map $A(x)=rx+\log_{\omega}(a)$, respectively $A'(x)=r'x+\log_{\omega}(a')$, of $\IZ/(q-1)\IZ$, and computing these affine maps takes $q$-bounded query complexity $(\log{q},1,0,0,0)$. Finally, one applies the algorithm from the proof of Corollary \ref{dengCor} to check within $q$-bounded query complexity $(\log^{2+o(1)}{q},0,\log{q},1,0)$ whether $\Gamma_A\cong\Gamma_{A'}$, which is equivalent to $\Gamma_f\cong\Gamma_{f'}$.
\end{proof}

Comparing our Theorem \ref{dengTheo} with Deng's original version \cite[Theorem 11]{Den13a}, we note that the additional simplification of essentially only having to check the equality $\ord_t(a^l)=\ord_t((a')^l)$ for those divisors $t$ of $m$ that are of the form $p^{\nu_p(m)}$, as opposed to all divisors of $m$, is essential to achieve a polynomial complexity in the associated algorithm. This is because the number $\tau(m)$\phantomsection\label{not278P5} of distinct (positive) divisors of $m$ may be superpolynomial in $\log{m}$.

On the other hand, Dirichlet proved a result which implies that the average value of $\tau$ over the initial segment $\{1,2,\ldots,N\}$ of $\IN^+$ is asymptotically equivalent to $\log{N}$ \cite[Theorem 3.3]{Apo76a}, so in particular, the set of positive integers $m$ for which $\tau(m)\geq\log^{1+\epsilon}{m}$ has asymptotic density $0$ for each constant $\epsilon>0$, because otherwise, if the said asymptotic density is $\delta>0$, there are infinitely many $N\in\IN^+$ such that
\begin{align*}
N^{-1}\sum_{m\leq N}{\tau(m)}&\geq N^{-1}\sum_{m=1}^{\lfloor\delta N\rfloor}{\log^{1+\epsilon}{m}}\geq (2N)^{-1}\frac{\delta}{2}N\log^{1+\epsilon}\left(\frac{\delta}{2}N\right)=\frac{\delta}{4}\log^{1+\epsilon}\left(\frac{\delta}{2}N\right) \\
&\geq\frac{\delta}{8}\log^{1+\epsilon}{N}>>\log{N},
\end{align*}
a contradiction. Concerning the average value of $\tau(q-1)$ where $q$ ranges over prime powers, we have the following result, which is used in Subsubsection \ref{subsubsec5P3P2}. In this context, the authors would like to thank Ofir Gorodetsky, who kindly pointed out Halberstam's crucial paper \cite{Hal56a} and parts of the proof of Proposition \ref{tauAvProp} in an answer to a question posted by the first author on MathOverflow\footnote{see \url{https://mathoverflow.net/questions/436134/average-value-of-the-prime-omega-function-omega-on-predecessors-of-prime-powe}}.

\begin{proppposition}\label{tauAvProp}
For each $\epsilon>0$, there are constants $c_{\epsilon},c'_{\epsilon}>0$ such that the following hold for all but an asymptotic fraction of less than $\epsilon$ of all prime powers $q$:
\begin{enumerate}
\item $\mpe(q-1)<c_{\epsilon}$; and
\item the number of distinct prime divisors of $q-1$ is less than $c'_{\epsilon}\log\log{q}$.
\end{enumerate}
In particular, for all such prime powers $q$, one has $\tau(q-1)<\log^{c''_{\epsilon}}{q}$ where $c''_{\epsilon}:=\log(c_{\epsilon}+1)c'_{\epsilon}$.
\end{proppposition}

\begin{proof}
Throughout this proof, the variable $q$ ranges over prime powers, while $p$ ranges over primes. Statement (1) is the same as Proposition \ref{mpeAvProp}(1). For statement (2), as usual, we denote by $\omega(m)$ the number of distinct prime divisors of $m\in\IN^+$. Halberstam proved that
\begin{equation}\label{halberstamEq}
\left(\sum_{p\leq x}{1}\right)^{-1}\sum_{p\leq x}{\omega(p-1)}\sim\log\log{x}
\end{equation}
as $x\to\infty$, see \cite[Theorem 1]{Hal56a}. Now, let $\epsilon>0$ be fixed, and let us assume that for some constant $c>0$, one has $\omega(q-1)\geq c\log\log{q}$ for an asymptotic fraction of at least $\epsilon$ of all prime powers $q$. We need to bound $c$ in terms of $\epsilon$ in order to prove statement (2). Now, because proper prime powers are a density $0$ subset of all prime powers (see the proof of Proposition \ref{mpeAvProp}), we conclude that also for an asymptotic fraction of at least $\epsilon$ of all \emph{primes} $p$, one has $\omega(p-1)\geq c\log\log{p}$. We fix a large enough $x\geq2$ such that for a fraction of at least $\epsilon/2$ of all primes $p\leq x$, one has $\omega(p-1)\geq c\log\log{p}$. Let us denote by $p_m$ the $m$-th prime number for $m\in\IN^+$. Because $p_m\geq\frac{1}{2}m\log{m}$ for large enough $m$, and the total number of primes $p\leq x$ is at least $x/(2\log{x})$, it follows that
\begin{align*}
\sum_{p\leq x}{\omega(p-1)} &\geq\sum_{m=1}^{\lfloor\epsilon x/(4\log{x})\rfloor}{c\log\log{p_m}}\geq\sum_{m=\lceil\epsilon x/(8\log{x})\rceil}^{\lfloor\epsilon x/(4\log{x})\rfloor}{c\log\log\left(\frac{1}{2}m\log{m}\right)} \\
&\geq\frac{\epsilon x}{16\log{x}}\cdot c\log\log\left(\frac{1}{2}\cdot\frac{\epsilon x}{8\log{x}}\cdot\log\left(\frac{\epsilon x}{8\log{x}}\right)\right),
\end{align*}
whence
\begin{align*}
&\left(\sum_{p\leq x}{1}\right)^{-1}\sum_{p\leq x}{\omega(p-1)} \\
\geq&\left(2\frac{x}{\log{x}}\right)^{-1}\cdot\frac{\epsilon x}{16\log{x}}\cdot c\log\log\left(\frac{1}{2}\cdot\frac{\epsilon x}{8\log{x}}\cdot\log\left(\frac{\epsilon x}{8\log{x}}\right)\right) \\
\sim&\frac{\epsilon c}{32}\log\log{x},
\end{align*}
which implies
\[
\left(\sum_{p\leq x}{1}\right)^{-1}\sum_{p\leq x}{\omega(p-1)}\geq\frac{\epsilon c}{64}\log\log{x}
\]
if $x$ is large enough. Hence, in order to not contradict Halberstam's (\ref{halberstamEq}), we must have $c\leq 64/\epsilon$, an upper bound on $c$ in terms of $\epsilon$, as required in order for statement (2) to hold.

Finally, for the \enquote{In particular} statement, we note that because $\tau(q-1)=(v_1+1)(v_2+1)\cdots(v_K+1)$ if $q-1=p_1^{v_1}p_2^{v_2}\cdots p_K^{v_K}$ is the prime factorization of $q-1$, one can bound $\tau(q-1)$ from above as follows:
\begin{align*}
\tau(q-1) &\leq(\mpe(q-1)+1)^{\omega(q-1)}\leq(c_{\epsilon}+1)^{c'_{\epsilon}\log\log{q}}=\exp(c'_{\epsilon}\log\log{q}\log(c_{\epsilon}+1)) \\
&=\left(\log{q}\right)^{\log(c_{\epsilon}+1)c'_{\epsilon}}.
\end{align*}
\end{proof}

We conclude this subsubsection with two batches of pseudocode. First, we give pseudocode for the algorithm that checks whether $\Gamma_A\cong\Gamma_{A'}$ where $A:x\mapsto ax+b$ and $A':x\mapsto a'x+b'$ are affine maps of $\IZ/m\IZ$ (see Corollary \ref{dengCor} and its proof). As in the previous subsection, we specify the ($m$-bounded) query complexity (QC) of each step.

\begin{denumerate}[label=\arabic*]
\item Compute $\gcd(a,m)$ and $\gcd(a',m)$, and check whether they are equal. If not, output \enquote{false} and halt.

QC: $(\log^{1+o(1)}{m},0,0,0,0)$.
\item Factor $m=p_1^{v_1}\cdots p_K^{v_k}$.

QC: $(\log{m},0,1,0,0)$.
\item For each $j=1,2,\ldots,K$, do the following.

QC: $(\log^{2+o(1)}{m},0,0,0,0)$.
\begin{denumerate}[label=3.\arabic*]
\item Check whether it is the case that $p_j\nmid a$ and $\gcd(a-1,p_j^{v_j})\mid b$. If so, set $\test_j:=\mathrm{true}$, otherwise set $\test_j:=\mathrm{false}$.

QC: $(\log^{1+o(1)}{m},0,0,0,0)$.
\item Check whether it is the case that $p_j\nmid a'$ and $\gcd(a'-1,p_j^{v_u})\mid b'$. If so, set $\test'_j:=\mathrm{true}$, otherwise set $\test'_j:=\mathrm{false}$.

QC: $(\log^{1+o(1)}{m},0,0,0,0)$.
\item If $\test_j\not=\test'_j$, then output \enquote{false} and halt.

QC: $(\log\log{m},0,0,0,0)$.
\end{denumerate}
\item For each $j=1,2,\ldots,K$, do the following.

QC: $(\log^{2+o(1)}{m},0,\log{m},0,0)$.
\begin{denumerate}[label=4.\arabic*]
\item Check whether $p_j\nmid a$. If not, set $\Test_j:=\mathrm{false}$ and skip to Step 4.4. Otherwise, set $\Test_j:=\mathrm{true}$.

QC: $(\log^{1+o(1)}{m},0,0,0,0)$.
\item Check which case in Table \ref{crlListPrimaryTable} (or \cite[Table 3 or 4 respectively]{BW22b}) applies to $A\bmod{p_j^{v_j}}$, using simple arithmetic.

QC: $(\log^{1+o(1)}{m},0,0,0,0)$.
\item Determine the minimal cycle length $l_j$ of $A\bmod{p_j^{v_j}}$ according to Table \ref{crlListPrimaryTable} (or \cite[Table 3 or 4 respectively]{BW22b}) and factor it.

QC: $(\log^{1+o(1)}{m},0,1,0,0)$.
\item Check whether $p_j\nmid a'$. If not, set $\Test'_j:=\mathrm{false}$ and skip to the next $j$. Otherwise, set $\Test'_j:=\mathrm{true}$.

QC: $(\log^{1+o(1)}{m},0,0,0,0)$.
\item Check which case in Table \ref{crlListPrimaryTable} (or \cite[Table 3 or 4 respectively]{BW22b}) applies to $A'\bmod{p_j^{v_j}}$, using simple arithmetic.

QC: $(\log^{1+o(1)}{m},0,0,0,0)$.
\item Determine the minimal cycle length $l'_j$ of $A'\bmod{p_j^{v_j}}$ according to Table \ref{crlListPrimaryTable} (or \cite[Table 3 or 4 respectively]{BW22b}) and factor it.

QC: $(\log^{1+o(1)}{m},0,1,0,0)$.
\end{denumerate}
\item Compute
\begin{align*}
l&:=\lcm(l_j: 1\leq j\leq K, \Test_j=\mathrm{true})\text{ and} \\
l'&:=\lcm(l'_j: 1\leq j\leq K, \Test'_j=\mathrm{true})
\end{align*}
using the factorizations of the $l_j$ and $l'_j$ computed in Steps 4.3 and 4.6 above.

QC: $(\log^{2+o(1)}{m},0,0,0,0)$.
\item Check whether $l=l'$. If not, output \enquote{false} and halt.

QC: $(\log{m},0,0,0,0)$.
\item Compute $a^l\bmod{m}$ and $(a')^l\bmod{m}$.

QC: $(\log^{2+o(1)}{m},0,0,0,0)$.
\item Compute $\ord_{p^{\nu_p(m)}}(a^l)$, respectively $\ord_{p^{\nu_p(m)}}((a')^l)$, for all primes $p\mid m$ such that $p\nmid a$, respectively $p\nmid a'$, in particular for all $p\in\Yfrak=\Yfrak'=\{p_j: \test_j=\mathrm{true}\}$.

QC: $(\log{m},0,0,1,0)$.
\item Check whether $\ord_{p^{\nu_p(m)}}(a^l)=\ord_{p^{\nu_p(m)}}((a')^l)$ for all $p\in\Yfrak$. If not, output \enquote{false} and halt.

QC: $(\log{m},0,0,0,0)$.
\item If $4\mid m$, do the following.
\begin{denumerate}[label=10.\arabic*]
\item Check whether $\ord_4(a^l)=\ord_4((a')^l)$. If not, output \enquote{false} and halt.

QC: $(\log^{1+o(1)}{m},0,0,0,0)$.
\item Output \enquote{true} and halt.

QC: $(1,0,0,0,0)$.
\end{denumerate}
\item Else do the following.
\begin{denumerate}[label=11.\arabic*]
\item Output \enquote{true} and halt.

QC: $(1,0,0,0,0)$.
\end{denumerate}
\end{denumerate}

Finally, we provide pseudocode for the algorithm that checks whether $\Gamma_f\cong\Gamma_{f'}$ for monomial mappings $f:x\mapsto ax^r$ and $f':x\mapsto a'x^{r'}$ of $\IF_q$ (where $a,a'\in\IF_q$ and $r,r'\in\{0,1,\ldots,q-1\}$).

\begin{denumerate}[label=\arabic*]
\item If $a=0_{\IF_q}$ or $r=0$, then do the following.
\begin{denumerate}[label=1.\arabic*]
\item Check whether it is the case that $a'=0_{\IF_q}$ or $r'=0$. If so, output \enquote{true} and halt. Otherwise, output \enquote{false} and halt.

QC: $(\log{q},0,0,0,0)$.
\end{denumerate}
\item Else do the following.
\begin{denumerate}[label=2.\arabic*]
\item Check whether it is the case that $a'=0_{\IF_q}$ or $r'=0$. If so, output \enquote{false} and halt.

QC: $(\log{q},0,0,0,0)$.
\item Compute $b:=\log_{\omega}(a)$ and $b':=\log_{\omega}(a')$, where $\omega$ is the \enquote{natural} choice of primitive element of $\IF_q$ (see the proof of Lemma \ref{complexitiesLem2}(4)). Denote by $A$, respectively $A'$, the affine map of $\IZ/(q-1)\IZ$ given by the formula $A(x)=rx+b$, respectively $A'(x)=r'x+b'$.

QC: $(\log{q},1,0,0,0)$.
\item Use the algorithm from above to check whether $\Gamma_A\cong\Gamma_{A'}$. If so, output \enquote{true} and halt. Otherwise, output \enquote{false} and halt.

QC: $(\log^{2+o(1)}{q},0,\log{q},1,0)$.
\end{denumerate}
\end{denumerate}

\subsubsection{Special case: Trees only depend on the coset}\label{subsubsec5P3P2}

In this subsubsection, we discuss two special classes of generalized cyclotomic mappings of $\IF_q$ such that if $f_1$ and $f_2$ each belong to one of those classes (not necessarily both to the same) and are of index $d_1$ and $d_2$ respectively, then it can be decided whether $\Gamma_{f_1}\cong\Gamma_{f_2}$ in a $q$-bounded query complexity each entry of which is polynomial in the parameters $\max\{d_1,d_2\}$, $\log{q}$ and $\tau(q-1)$ (the number of divisors of $q-1$). In particular, the $q$-bounded Las Vegas dual complexity of this problem is always subexponential in the input size $O(\max\{d_1,d_2\}\log{q})$, as $\tau(m)\in o(m^{\epsilon})$ for each $\epsilon>0$ \cite[p.~296]{Apo76a}. Moreover, in view of Proposition \ref{tauAvProp}, for each $\epsilon>0$, one has that for all but an asymptotic fraction of less than $\epsilon$ of all finite fields $\IF_q$, the said $q$-bounded Las Vegas dual complexity is polynomial (of a degree depending on $\epsilon$) in the input size.

The two classes of generalized cyclotomic mappings $f$ of $\IF_q$, say of index $d$, which we consider are as follows.
\begin{itemize}
\item \underline{Class 1}: $f$ maps each coset $C_i$ for $i\in\{0,1,\ldots,d-1\}$ either to $C_d=\{0_{\IF_q}\}$ or bijectively to $C_{\overline{f}(i)}$ (i.e., whenever the affine function $A_i$ of $\IZ/s\IZ$ is well-defined, it is a permutation of $\IZ/s\IZ$). If this happens, we say that $f$ is \emph{of special type I}\phantomsection\label{term82}. This is the same situation as in Subsection \ref{subsec4P3}.
\item \underline{Class 2}: The induced function $\overline{f}$ is a permutation of $\{0,1,\ldots,d\}$ (i.e., $f$ permutes the cosets of $C$). If this happens, we say that $f$ is \emph{of special type II}\phantomsection\label{term83}. Using the notation of Subsection \ref{subsec3P2}, this means that $\Gamma_f=\Gamma_{\per}$, and so the discussion from that subsection applies.
\end{itemize}
The crucial property which these two cases share is that the rooted trees above periodic vertices in $\Gamma_f$ only depend on the block $C_i$, for $i\in\{0,1,\ldots,d\}$, in which these vertices lie, as follows from Lemma \ref{allPermutationsLem} and Theorem \ref{cosetPermTheo}, respectively. This allows us to produce a comparatively compact description of the digraph isomorphism type of $\Gamma_f$. To that end, it is helpful to adapt the notion of a partition-tree register, introduced in Definition \ref{partTreeRegDef}, as follows.

\begin{defffinition}\label{treeRegDef}
Let $f$ be an index $d$ generalized cyclotomic mapping of $\IF_q$.
\begin{enumerate}
\item We assume that $f$ is of special type I, so that \emph{all} vertices (not just $f$-periodic ones) in a given block $C_i$ have isomorphic rooted trees above them in $\Gamma_f$. A \emph{type-I tree register for $f$}\phantomsection\label{term84} is an ordered sequence $((\Dfrak_n,S_n))_{n=0,1,\ldots,N}$\phantomsection\label{not279} such that
\begin{enumerate}
\item the sets $\Dfrak_n$ form a recursive tree description list, with associated rooted tree isomorphism types $\Ifrak_n$ (see Definition \ref{recTreeDescListDef}) and
\item for each $n$, the set $S_n$ is nonempty and consists precisely of those $i\in\{0,1,\ldots,d\}$ such that $\Tree_{\Gamma_f}(x)\cong\Tree_{\Gamma_{\overline{f}}}(i)\cong\Ifrak_n$ for any $x\in C_i$.
\end{enumerate}
\item We assume that $f$ is of special type II, so that for each $i\in\{0,1,\ldots,d-1\}$, the rooted tree isomorphism type $\Tree_{\Gamma_f}(x)$ for $x\in C_i$ only depends on $i$ and the $\hfrak$-value of $x$ (see Subsection \ref{subsec4P1}, page \pageref{latterHalfRef} onward). A \emph{type-II tree register for $f$}\phantomsection\label{term85} is an ordered sequence $((\Dfrak_n,S_n))_{n=0,1,\ldots,N}$ such that
\begin{enumerate}
\item the sets $\Dfrak_n$ form a recursive tree description list, with associated rooted tree isomorphism types $\Ifrak_n$;
\item the $\Ifrak_n$ are just those isomorphism types that occur among the rooted trees of the form $\Expand(\Ifrak_{i,h})$ for $i\in\{0,1,\ldots,d-1\}$ and $h\in\{0,1,\ldots,H_i\}$ (see Subsection \ref{subsec4P1}, page \pageref{latterHalfRef} onward, for the definition of the $\Ifrak_{i,h}$); and
\item for each $n$, one has $S_n=(\height(\Ifrak_n),S_{n,\trans},S_{n,\per})$\phantomsection\label{not280}\phantomsection\label{not281} where
\[
S_{n,\trans}=\{i\in\{0,1,\ldots,d-1\}: \height(\Ifrak_n)<H_i\text{ and }\Ifrak_n=\Expand(\Ifrak_{i,\height(\Ifrak_n)})\}
\]
and
\[
S_{n,\per}=\{i\in\{0,1,\ldots,d-1\}: \Ifrak_n=\Expand(\Ifrak_{i,H_i})\}.
\]
\end{enumerate}
\end{enumerate}
\end{defffinition}

In an implementation, we assume that the sets $S_n$ from Definition \ref{treeRegDef}(1), as well as the sets $S_{n,\trans}$ and $S_{n,\per}$ from Definition \ref{treeRegDef}(2), are represented by \emph{sorted} arrays each entry of which is a binary digit representation of a number $i\in\{0,1,\ldots,d\}$ with bit length exactly $\lfloor\log_2{d}\rfloor+1$. Moreover, in a type-I tree register, only one of the descriptions $\Dfrak_n$ corresponds to $\Tree_{\Gamma_f}(0_{\IF_q})$, and it is the only description in which the second entries of elements may be larger than $d$. For the sake of efficiency, we make the convention that all descriptions $\Dfrak_n$ except that one use $\lfloor\log_2{d}\rfloor+1$ digits for representing each entry $m$ or $k_m$ of an element $(m,k_m)$ of $\Dfrak_n$. On the other hand, in the description corresponding to $\Tree_{\Gamma_f}(0_{\IF_q})$, we use $\lfloor\log_2{d}\rfloor+1$ digits for the first entry $m$, and $\lfloor\log_2{q}\rfloor+1$ digits for the second entry $k_m$. In contrast to this, in a type-II tree register, we know that $m\leq d^2\mpe(q-1)+d\leq d^2\lfloor\log_2{q}\rfloor+d$ (see the first paragraph in the proof of Lemma \ref{treeRegLem}(4) below), while there is a priori no upper bound on $k_m$ other than the trivial one, $q$. Hence, in such a register, we use $\lfloor\log_2(d^2\lfloor\log_2{q}\rfloor+d)\rfloor+1$ digits for representing $m$, and $\lfloor\log_2{q}\rfloor+1$ digits for representing $k_m$. In either case, the entries of a given description $\Dfrak_n$ all have the same bit length, and different descriptions $\Dfrak_n$ use the same bit length for the first entries of their elements. As before, we assume that the array representing a given description $\Dfrak_n$ is lexicographically ordered (corresponding to the ordering of the elements $(m,k_m)$ of $\Dfrak_n$ by increasing $m$).

We observe that in Definition \ref{treeRegDef}(2,c), the set $S_{n,\trans}$ consists precisely of those $i\in\{0,1,\ldots,d-1\}$ such that $\Ifrak_n=\Tree_{\Gamma_f}(x)$ for all $x\in C_i$ with $\hfrak(x)=\height(\Ifrak_n)<H_i$, regardless of whether such $x$ exist; we recall from Example \ref{cosetPermEx} that not necessarily all values in $\{0,1,\ldots,H_i-1\}$ are assumed by $\hfrak$ on the $f$-transient points in a given coset $C_i$. We also remind the reader that $f$-transient $x\in C_i$ are characterized by the inequality $\hfrak(x)<H_i$, and that for all such $x$, one has $\height(\Tree_{\Gamma_f}(x))=\height(\Ifrak_{i,\hfrak(x)})=\hfrak(x)$; see the recursive definition of the $\Ifrak_{i,h}$ in Subsection \ref{subsec4P1}, page \pageref{latterHalfRef} onward. Moreover, the set $S_{n,\per}$ consists of those $i\in\{0,1,\ldots,d-1\}$ such that $\Ifrak_n$ is the unique isomorphism type $\Tree_{\Gamma_f}(x)=\Expand(\Ifrak_{i,H_i})$ for $f$-periodic $x\in C_i$ (and $\Expand(\Ifrak_{i,H_i})$ is \emph{not} necessarily of height $H_i$, but it is of height $\Hcal_i$).

Next, we discuss the following important lemma.

\begin{lemmmma}\label{treeRegLem}
Let $f$ be an index $d$ generalized cyclotomic mapping of $\IF_q$, given in cyclotomic form (\ref{cyclotomicFormEq}).
\begin{enumerate}
\item Checking whether $f$ is of special type I takes $q$-bounded query complexity
\[
(d\log^{1+o(1)}{q},d,0,0,0).
\]
\item If $f$ is of special type I, then a type-I tree register $\Rfrak=((\Dfrak_n,S_n))_{n=0,1,\ldots,N}$\phantomsection\label{not282} for $f$ with $N\in O(d)$ can be computed within $q$-bounded query complexity
\[
(d^3\log^2{d}+d\log^{1+o(1)}{q},d,0,0,0).
\]
\item Checking whether $f$ is of special type II has $q$-bounded query complexity
\[
(d\log^2{d}+d\log^{1+o(1)}{q},d,0,0,0).
\]
\item If $f$ is of special type II, then a type-II tree register $\Rfrak=((\Dfrak_n,S_n))_{n=0,1,\ldots,N}$ for $f$ with $N\in O(d^2\mpe(q-1))$ can be computed within $q$-bounded query complexity
\[
(d^5\log^2{d}\mpe(q-1)^3+d^5\mpe(q-1)^3\log{q}+d^2\mpe(q-1)\log^{1+o(1)}{q},d,0,0,0).
\]
In particular, it can be computed within $q$-bounded query complexity
\[
(d^{5+o(1)}\log^4{q},d,0,0,0).
\]
\end{enumerate}
\end{lemmmma}

\begin{proof}
For statement (1), we first compute $\overline{f}$ and the affine maps $A_i:x\mapsto\alpha_ix+\beta_i$, which requires $q$-bounded query complexity $(d\log^{1+o(1)}{q},d,0,0,0)$ by Proposition \ref{faiProp}. We note that $f$ is of special type I if and only if $\gcd(\alpha_i,s)=1$ for all $i$ such that $A_i$ is well-defined (i.e., such that the coefficient $a_i\in\IF_q$ in the cyclotomic form (\ref{cyclotomicFormEq}) of $f$ is non-zero), which one can check with $O(d\log^{1+o(1)}{q})$ bit operations by Lemma \ref{complexitiesLem}(3,8).

For statement (2), we start by computing $\overline{f}$, the number $s=(q-1)/d$ and (as in Subsubsection \ref{subsubsec5P2P1}, page \pageref{beginningRef}) the \enquote{layers}
\[
\Layer_h:=
\begin{cases}
\im(\overline{f}^h)\setminus\im(\overline{f}^{h+1}), & \text{if }h\in\{0,1,\ldots,\overline{H}-1\}, \\
\im(\overline{f}^{\overline{H}})=\per(\overline{f}), & \text{if }h=\infty,
\end{cases}
\]
of $\overline{f}$ with respect to iteration in sorted form and without multiple entries, where $\overline{H}$\phantomsection\label{not283} is the maximum tree height in $\Gamma_{\overline{f}}$. Altogether, this takes $q$-bounded query complexity $(d\log^{1+o(1)}{q}+d^2\log^2{d},d,0,0,0)$. After this, we start building the register. At any given point during that process, we have at least a \enquote{partial register} as an intermediate result, which includes definitions of descriptions $\Dfrak_n$ of rooted trees $\Ifrak_n$ for all $n\in\Ncal$, an initial segment of $\IN_0$. Because each $n\in\Ncal$ has a non-empty subset $S_n$ of $\{0,1,\ldots,d\}$ associated with it and those sets are pairwise disjoint, we conclude that $|\Ncal|\leq d+1\in O(d)$. In particular, $N\in O(d)$ in the end, as asserted.

Now, to build the register, we do the following successively for $h=0,1,\ldots,\overline{H}-1,\infty$. We go through the indices $i\in\Layer_h$, and for each of them, we compute the pre-image set $\overline{f}^{-1}(\{i\})$ in sorted form, taking $O(d\log{d})$ bit operations per $i$. Let $j_1,j_2,\ldots,j_K$ be the \emph{$\overline{f}$-transient} pre-images of $i$ under $\overline{f}$; if $h<\infty$, those are simply all pre-images of $i$, and if $h=\infty$, one can determine them by additionally identifying the unique $\overline{f}$-pre-image of $i$ in $\per(\overline{f})=\Layer_{\infty}$, which takes $O(d\cdot|\overline{f}^{-1}(\{i\})|\cdot\log^2{d})\subseteq O(d^2\log^2{d})$ bit operations for all $i\in\Layer_{\infty}$ together, using binary search thanks to $\Layer_{\infty}$ being sorted.

In what follows, we assume that $i$ is fixed. Each $j_t$ lies in a unique layer $\Layer_{h_{j_t}}$ with $h_{j_t}<h_i=h$, and so there is a unique non-negative integer $\overline{n}_{j_t}\in\Ncal$\phantomsection\label{not284} such that $j_t\in S_{\overline{n}_{j_t}}$. Computing $\overline{n}_{j_t}$ takes $O(|\Ncal|\cdot\log^2{d})\subseteq O(d\log^2{d})$ bit operations for a single $t$, hence $O(d^2\log^2{d})$ bit operations altogether (for this fixed value of $i$). Now, by Lemma \ref{allPermutationsLem}, the rooted tree above any $x\in C_i$ is isomorphic to
\[
\begin{cases}
\Tree_{\Gamma_{\overline{f}}}(i)\cong\sum_{t=1}^K{\Ifrak_{\overline{n}_{j_t}}^+}, & \text{if }i<d, \\
\sum_{t=1}^K{s\Ifrak_{\overline{n}_{j_t}}^+}, & \text{if }i=d,
\end{cases}
\]
and so we may choose the following compact description $\Dfrak$ for this tree.
\begin{itemize}
\item If $i<d$, we set
\[
\Dfrak:=\{(n,m): n\in\{\overline{n}_{j_t}: 1\leq t\leq K\}, m=|\{t\in\{1,\ldots,K\}: \overline{n}_{j_t}=n\}|>0\}.
\]
\item If $i=d$, we set
\[
\Dfrak:=\{(n,sm): n\in\{\overline{n}_{j_t}: 1\leq t\leq K\}, m=|\{t\in\{1,\ldots,K\}: \overline{n}_{j_t}=n\}|>0\}.
\]
\end{itemize}
Computing $\Dfrak$ after the numbers $\overline{n}_{j_t}$ have been determined requires us to create a list of the \emph{distinct} values of the $\overline{n}_{j_t}$ and their multiplicities, which can be done in $O(K\log{K}\log{d})\subseteq O(d\log^2{d})$ bit operations when using the sorting algorithm from Lemma \ref{complexitiesLem}(10). If $i<d$, this is also the overall complexity of computing $\Dfrak$ for that $i$, whereas if $i=d$, the complexity of computing $\Dfrak$ is in $O(d\log^2{d}+d\log^{1+o(1)}{q})$, since it also involves integer multiplications. After $\Dfrak$ has been computed, we check whether there is an $n\in\Ncal$ such that $\Dfrak=\Dfrak_n$, which takes $O(|\Ncal|\cdot d\log{d})\subseteq O(d^2\log{d})$ bit operations regardless of the value of $i$. Indeed, if $i=d$, for which the bit length of the second entries of elements of $\Dfrak$ is not necessarily in $O(\log{d})$, one can proceed as follows. Observing that those second entries can only be that large for this one value of $i$, one first checks whether $s>d$, which can be done with a mere $O(\log{d})$ bit operations (we note that $s$ itself was already computed at the beginning). If so, one knows that $\Dfrak\not=\Dfrak_n$ for any $n\in\Ncal$; otherwise, the bit length of the second entries of $\Dfrak$ is in $O(\log{d^2})=O(\log{d})$ even for $i=d$, and one can proceed as for $i<d$. In any case, if $\Dfrak=\Dfrak_n$ for some $n\in\Ncal$, then we add $i$ to $S_n$ as a new element by merging the sorted lists corresponding to the sets $S_n$ and $\{i\}$, which takes $O(d\log{d})$ bit operations by Lemma \ref{complexitiesLem}(11). Otherwise, we create $\Dfrak$ as a new description $\Dfrak_{n'}$, where $n'=\max{\Ncal}+1$, and initialize $S_{n'}:=\{i\}$.

Since we need to carry out the computations described after declaring $i$ to be fixed for all such $i$, the total bit operation cost of these computations is in $O(d^3\log^2{d}+d\log^{1+o(1)}{q})$, and so the overall $q$-bounded query complexity of computing a type-I tree register for $f$ is as asserted.

For statement (3), we simply compute $\overline{f}$ and check whether $\im(\overline{f})=\{0,1,\ldots,d\}$. The former of these two tasks takes $q$-bounded query complexity
\[
(d\log^{1+o(1)}{q},d,0,0,0)
\]
by Proposition \ref{faiProp}, and the latter takes $O(d\log^2{d})$ bit operations (see the beginning of the argument in Subsubsection \ref{subsubsec5P2P1}).

For statement (4), we first compute $\overline{f}$, the affine maps $A_i:x\mapsto\alpha_ix+\beta_i$ and (as in Subsubsection \ref{subsubsec5P2P1}) a CRL-list $\overline{\Lcal}$ of $\overline{f}$ and the cycles of $\overline{f}$, taking $q$-bounded query complexity in $(d\log^{1+o(1)}{q}+d^2\log^2{d},d,0,0,0)$. Following that, we start building the tree register, and as in the proof of statement (2), in dependency of a given point in that process, we denote by $\Ncal$ the initial segment of $\IN_0$ consisting of all $n$ for which $\Dfrak_n$ is defined at that point. Because the associated rooted tree isomorphism types $\Ifrak_n$ are pairwise distinct and are elements of the set $\{\Expand(\Ifrak_{i,h}): i\in\{0,1,\ldots,d-1\}, h\in\{0,1,\ldots,H_i\}\}$, we have
\[
|\Ncal|\leq d\cdot(\max_i{H_i}+1)\leq d^2\mpe(q-1)+d\in O(d^2\mpe(q-1))
\]
(for the bound on $H_i$, see formula (\ref{hiBoundEq}) in Subsection \ref{subsec3P3}). In particular, $N\in O(d^2\mpe(q-1))$ in the end, as asserted.

To build the register, we go through the elements $(i,\ell)\in\overline{\Lcal}$ with $i<d$ (we note that $\Tree_{\Gamma_f}(0_{\IF_q})$ is simply trivial and is not even recorded in the register by definition), and for each of them, we do the following. First, we compute the exact value of $H_i$, the maximum tree height above a periodic vertex in $\bigcup_{t=0}^{\ell-1}{C_{i_t}}$, where $(i_0,i_1,\ldots,i_{\ell-1})$ with $i=i_0$ is the $\overline{f}$-cycle of $i$. We note that by Theorem \ref{cosetPermTheo} and the paragraph before it (which was worked out in detail in Subsection \ref{subsec4P1}), we have the following. For each $t\in\{0,1,\ldots,\ell-1\}$, the trees above periodic vertices in $C_{i_t}$ are pairwise isomorphic and thus of a common height $\Hcal_{i_t}$. We have $H_i=\max\{\Hcal_{i_t}: t=0,1,\ldots,\ell-1\}$, so we compute the numbers $\Hcal_{i_t}$ in order to get $H_i$. At this point, we note that in fact, in Subsubsection \ref{subsubsec5P2P2}, we already described how to find $H_i$ through a binary search. However, here we are also interested in storing the procreation numbers $\proc_{i_t,k}$ for $t\in\{0,1,\ldots,\ell-1\}$ and $k\in\{1,2,\ldots,\Hcal_{i_t}\}$, whence we do \emph{not} use binary search to skip steps.

We remind the reader that for arbitrary $t\in\IZ$, the notation $i_t$ is shorthand for $i_{t\bmod{\ell}}$. For $t=0,1,\ldots,\ell-1$ and successively for $k=1,2,\ldots$, we compute (and store) the procreation number (see Theorem \ref{cosetPermTheo})
\[
\proc_{i_t,k}=\frac{\gcd(\prod_{r=0}^{k-1}{\alpha_{i_{t-k+r}}},s)}{\gcd(\prod_{r=0}^{k-2}{\alpha_{i_{t-k+r}}},s)}
\]
until $\proc_{i_t,k}=1$ for the first time for a given $t$, which happens precisely when $k=\Hcal_{i_t}+1$. If we store the values of the two products appearing in the formula for $\proc_{i_t,k}$, then the computation of each $\proc_{i_t,k}$ only involves $O(1)$ multiplications and thus has a bit operation cost in $O(\log^{1+o(1)}{q})$ by Lemma \ref{complexitiesLem}(3,8). Therefore, and because $\Hcal_{i_t}\in O(\ell\mpe(q-1))$, we can compute each individual $\Hcal_{i_t}$ using $O(\ell\mpe(q-1)\log^{1+o(1)}{q})$ bit operations, whence the computation of $H_i$ in total takes $O(\ell^2\mpe(q-1)\log^{1+o(1)}{q})$ bit operations.

Once $H_i$ has been computed, we start adding the information associated with the $\overline{f}$-cycle of $i$ to our tree register. More precisely, we do the following successively for $h=0,1,\ldots,H_i$. If $h=0$, then $\Expand(\Ifrak_{i_t,h})$ is trivial for all $t$, so in case $\Ncal=\emptyset$ (which only happens for the first pair $(i,\ell)\in\overline{\Lcal}$ we consider), we set $\Dfrak_0:=\emptyset$, causing $\Ifrak_0$ to be the trivial rooted tree (in particular, $\height(\Ifrak_0)=0$), and we initialize some variables as follows.
\begin{itemize}
\item We set
\[
S_{0,\trans}:=
\begin{cases}
\{i_0,i_1,\ldots,i_{\ell-1}\}, & \text{if }H_i>0, \\
\emptyset, & \text{otherwise}.
\end{cases}
\]
\item We set $S_{0,\per}:=\{i_t: t\in\{0,1,\ldots,\ell-1\}, \Hcal_{i_t}=0\}$.
\end{itemize}
We remind the reader that we want the arrays representing $S_{0,\trans}$ and $S_{0,\per}$ to be sorted, so one should apply the sorting algorithm from Lemma \ref{complexitiesLem}(10), which takes $O(d\log^2{d})$ bit operations for each array.

In the other case, where $\Ncal\not=\emptyset$, we do the following.
\begin{itemize}
\item If $H_i>0$, we add $i_0,i_1,\ldots,i_{\ell-1}$ to the already defined set $S_{0,\trans}$ as new elements (technically speaking, we sort $\{i_0,i_1,\ldots,i_{\ell-1}\}$ and merge it with $S_{0,\trans}$).
\item We also add all indices $i_t$, for $t\in\{0,1,\ldots,\ell-1\}$, such that $\Hcal_{i_t}=0$ to the already defined set $S_{0,\per}$ as new elements.
\end{itemize}
Using Lemma \ref{complexitiesLem}(10,11), one sees that dealing with the case $h=0$ as a whole only takes $O(\ell\log^2{d}+d\log{d})$ bit operations (for copying information and sorting/merging, as well as simple look-ups of the values $\Hcal_{i_t}$).

Now we assume that $h\geq1$. The edge-weighted rooted tree (isomorphism type) $\Ifrak_{i_t,h}$ is drawn at the end of Subsection \ref{subsec4P1} (we draw the reader's attention to the case distinction between $h<H_i$ and $h=H_i$), and we compute the description $\Dfrak=\Dfrak(i_t,h)$ of $\Expand(\Ifrak_{i_t,h})$ as follows. First, we set
\[
h':=
\begin{cases}
h, & \text{if }h<H_i, \\
\Hcal_{i_t}, & \text{if }h=H_i,
\end{cases}
\]
which is the height of $\Expand(\Ifrak_{i_t,h})$. It is also the number of edges in $\Ifrak_{i_t,h}$ that have the root as their terminal vertex (i.e., the unweighted in-degree of that root). We do note that some of these edges may have weight $0$. For $k=0,1,\ldots,h'-1$, we compute
\[
\wfrak_k:=
\begin{cases}
\wfrak_{i_t,k}=\proc_{i_t,k+1}-\proc_{i_t,k+2}, & \text{if }h=H_i,\text{ or }h<H_i\text{ and }k<h'-1, \\
\proc_{i_t,h}, & \text{if }h<H_i\text{ and }k=h'-1,
\end{cases}
\]
which\phantomsection\label{not285} is the weight of the $(k+1)$-th edge in $\Ifrak_{i_t,h}$ (counted from the left in the drawing) that has the root as its terminal vertex. These computations only require
\[
O(\ell\cdot h'\cdot\log{q})\subseteq O(\ell^2\mpe(q-1)\log{q})
\]
bit operations for all $t$ together. We may then set
\[
\Dfrak:=\{(\overline{n}_{i_t,k},\wfrak_k): k=0,1,\ldots,h'-1\}
\]
where $\overline{n}_{i_t,k}$\phantomsection\label{not286} is the unique $n\in\Ncal$ such that $\Ifrak_n=\Expand(\Ifrak_{i_t,k})$, i.e., such that $i_t\in S_{n,\trans}$ and $\height(\Ifrak_n)=k$. Assuming that the heights of the various $\Ifrak_n$ are stored whenever the register is updated, the computation of $\Dfrak$ takes
\begin{align*}
&O(h'\cdot(|\Ncal|\cdot\log^2{d}+\log\log{q})) \\
\subseteq &O(\ell\mpe(q-1)\cdot(d^2\mpe(q-1)\cdot\log^2{d}+\log\log{q})) \\
\subseteq &O(\ell d^2\log^2{d}\mpe(q-1)^2+\ell\mpe(q-1)\log\log{q})
\end{align*}
bit operations for a single $t$ (needed for determining the $\overline{n}_{i_t,k}$), hence
\[
O(\ell^2d^3\log^2{d}\mpe(q-1)^3+\ell^2d\mpe(q-1)^2\log\log{q})
\]
bit operations for all $t$ and $h$ together.

Once $\Dfrak$ has been computed, we need to check whether it already occurs among the $\Dfrak_n$ for $n\in\Ncal$ (and update the register accordingly). If we sort $\Dfrak$ lexicographically, we may compare it with a given $\Dfrak_n$ through linear comparison of entries, and so checking whether $\Dfrak=\Dfrak_n$ for some $n\in\Ncal$ takes
\begin{align*}
&O(h'\log{h'}\log{q}+|\Ncal|h'\log{q})\subseteq O(d^2\mpe(q-1)\cdot\ell\mpe(q-1)\cdot\log{q}) \\
=&O(\ell d^2\mpe(q-1)^2\log{q})
\end{align*}
bit operations for a single $t$, hence
\[
O(\ell H_i\cdot\ell d^2\mpe(q-1)^2\log{q})\subseteq O(\ell^2d^3\mpe(q-1)^3\log{q})\subseteq O(\ell d^4\mpe(q-1)^3\log{q})
\]
bit operations for all $t$ and $h$ together. This last $O$-expression dominates every other bit operation cost given in this complexity analysis except the cost
\[
O(\ell^2\mpe(q-1)\log^{1+o(1)}{q})\subseteq O(\ell d\mpe(q-1)\log^{1+o(1)}{q})
\]
of computing $H_i$, and the total cost $O(\ell^2d^3\log^2{d}\mpe(q-1)^3)$ of computing the descriptions $\Dfrak$. Therefore, the bit operation cost of these computations is in
\[
O(\ell d\mpe(q-1)\log^{1+o(1)}{q}+\ell d^4\mpe(q-1)^3\log{q}+\ell^2d^3\log^2{d}\mpe(q-1)^3)
\]
for a given $(i,\ell)$ and $h$. Using that $\sum_{(i,\ell)\in\overline{\Lcal}}{\ell}=d$, the total bit operation cost of computing the type-II tree register for $f$ is in
\begin{align*}
&O(d\log^{1+o(1)}{q}+d^2\log^2{d} \\
&+\sum_{(i,\ell)\in\overline{\Lcal}}{(\ell d\mpe(q-1)\log^{1+o(1)}{q}+\ell d^4\mpe(q-1)^3\log{q}+\ell^2d^3\log^2{d}\mpe(q-1)^3)}) \\
&\subseteq O(d^2\mpe(q-1)\log^{1+o(1)}{q}+d^5\mpe(q-1)^3\log{q}+d^5\log^2{d}\mpe(q-1)^3),
\end{align*}
as asserted. The \enquote{In particular} statement holds because $\mpe(q-1)\in O(\log{q})$.
\end{proof}

So far, everything has been of a $q$-bounded query complexity that is polynomial in $\log{q}$ and $d$. The quantity $\tau(q-1)$, which is generally superpolynomial in $\log{q}$, enters through the following auxiliary result.

\begin{proppposition}\label{ctTauProp}
Let $m$ be a positive integer, and let $A:x\mapsto ax+b$ be an affine map of $\IZ/m\IZ$. The cycle type of $A_{\mid\per(A)}$, denoted by $\CT(A_{\mid\per(A)})$, can be computed within $m$-bounded query complexity
\[
(\tau(m)\log^{2+o(1)}{m}+\tau(m)^2\log{m},0,\log{m},1,0).
\]
\end{proppposition}

\begin{proof}
Using a single $m$-bounded mord query (i.e., $m$-bounded query complexity $(\log{m},0,0,1,0)$), we factor $m$ and compute $\ord_{p^{\nu_p(m)}}(a)$ for all primes $p\mid m$ such that $p\nmid a$.
Letting $m'':=\prod_{p\mid\gcd(a,m)}{p^{\nu_p(m)}}$ and $m':=m/m''$ (which we do not need to actually compute), we observe the following. Because $A\bmod{m''}$ has a unique periodic point (see Lemma \ref{periodicCharLem}), we find that $\CT(A_{\mid\per(A)})=\CT(A\bmod{m'})$, and so we compute the latter. This allows us to assume without loss of generality that $A$ is a permutation of $\IZ/m\IZ$. For each prime $p\mid m$, we set $A_{(p)}:=A\bmod{p^{\nu_p(m)}}$. Since $A$ is given via its coefficients $a$ and $b$, computing $A_{(p)}$ takes a mere $O(\log^{1+o(1)}{m})$ bit operations per $p$ for obtaining the remainders of $a$ and $b$ upon division by $p^{\nu_p(m)}$. Hence, computing all reductions $A_{(p)}$ takes $O(\log^{2+o(1)}{m})$ bit operations.

Formulas for $\CT(A_{(p)})$ were given in \cite[Tables 3 and 4]{BW22b}, and since we know $\ord_{p^{\nu_p(m)}}(a)$ for all $p\mid m$ from our initial mord query, these formulas allow us to compute $\CT(A_{(p)})$ for a given $p$ within $m$-bounded query complexity
\[
(\log^{2+o(1)}(p^{\nu_p(m)})+\nu_p(m)\log^{1+o(1)}(p^{\nu_p(m)}),0,1,0,0).
\]
Indeed, taking a closer look at those formulas, we see that we initially need to factor $\ord_{p^{\nu_p(m)}}(a)$ and compute a single power, combined with simpler arithmetic (taking $p^{\nu_p(m)}$-bounded query complexity $(\log^{2+o(1)}(p^{\nu_p(m)}),0,1,0,0)$), followed by $O(\nu_p(m))$ iterations of a loop, each consisting of $O(1)$ basic arithmetic operations taking $O(\log^{1+o(1)}(p^{\nu_p(m)}))$ bit operations each. For all $p$ together, computing $\CT(A_{(p)})$ has $m$-bounded query complexity $(\log^{2+o(1)}{m},0,\log{m},0,0)$. We also note that each $\CT(A_{(p)})$ is a monomial with at most $\nu_p(m)+1$ factors, and that our computation process allows us to store $\CT(A_{(p)})$ with all cycle lengths fully factored.

Now, we may compute $\CT(A)$ via the formula $\CT(A)=\divideontimes_{p\mid m}{\CT(A_{(p)})}$, where $\divideontimes$ denotes the Wei-Xu product from \cite[Definition 2.2 on pp.~182f.]{WX93a}. This can be done by looping over the $O(\prod_{p\mid m}{(\nu_p(m)+1)})\subseteq O(\tau(m))$ tuples formed by choosing one variable power in the factorization of each $\CT(A_{(p)})$ and computing the Wei-Xu product of those variable powers (which is itself a variable power) according to \cite[formula (2.9) in Lemma 2.3(b)]{WX93a}. Doing so requires us to compute the least common multiple of the involved cycle lengths, which takes $O(\log^{2+o(1)}{m})$ bit operations because those cycle lengths are already fully factored (see also the paragraph on the computation of $l$ in the proof of Corollary \ref{dengCor}), followed by $O(\log{m})$ integer multiplications and divisions for computing the exponent, which also take $O(\log^{2+o(1)}{m})$ bit operations together. In total, the process of computing all relevant Wei-Xu products of variable powers takes $O(\tau(m)\log^{2+o(1)}{m})$ bit operations. Once this is done, we need to compute the product of those variable powers, which means that $O(\tau(m))$ times, we need to multiply a monic monomial with $O(\tau(m))$ distinct variable power factors, each with index and exponent in $\{1,2,\ldots,m\}$, with a single such variable power. Each such multiplication takes $O(\tau(m)\log{m})$ bit operations, so the overall complexity of these computations, which result in $\CT(A)$, is in $O(\tau(m)^2\log{m})$.
\end{proof}

\begin{remmmark}\label{ctTauRem}
By our proof of Proposition \ref{ctTauProp}, the cycle type of an affine map $A$ of $\IZ/m\IZ$ is a product of at most $\tau(m)$ variable powers, and so $A$ has at most $\tau(m)$ distinct cycle lengths.

As far as lower bounds on the maximum number of distinct cycle lengths of an affine map of $\IZ/m\IZ$ are concerned, let us fix a positive integer $K$ and primes $2<p_1<p_2<\cdots<p_K$ such that $\gcd(p_j,p_{j'}-1)=1$ for $1\leq j<j'\leq K$ (such primes exist for each $K$ by Dirichlet's theorem on primes in arithmetical progressions, see \cite[Chapter 7]{Apo76a}). For variable positive integers $v_1,v_2,\ldots,v_K$, we set $m:=p_1^{v_1}\cdots p_K^{v_K}$ and consider an automorphism $A:x\mapsto ax$ of $\IZ/m\IZ$ such that $a$ is a primitive root modulo $p_j^{v_j}$ for each $j$ (it is possible to choose $a$ like this because of the Chinese Remainder Theorem). By \cite[Table 3]{BW22b}, the cycle lengths of $A\bmod{p_j^{v_j}}$ are just the numbers of the form $(p_j-1)p_j^{v'_j}$ where $v'_j\in\{0,1,\ldots,v_j-1\}$. Therefore, the cycle lengths of $A$ are just the numbers of the form $\prod_{j=1}^K{(p_j-1)}\cdot\prod_{j=1}^K{p_j^{v'_j}}$, whence $A$ has $v_1\cdots v_K$ distinct cycle lengths. We observe that this cycle length count is asymptotically equivalent to $(v_1+1)\cdots(v_K+1)=\tau(m)$ if $\min\{v_1,\ldots,v_K\}\to\infty$. Moreover, we note that $\log{m}=v_1\log{p_1}+\cdots+v_K\log{p_K}\leq(v_1+\cdots+v_K)\log{p_K}$. Now, let us assume that $v_1=v_2=\cdots=v_K=:v\to\infty$. Then the number of distinct cycle lengths of $A$ is
\[
v^K=\left(\frac{\log{p_K}\cdot K\cdot v}{\log{p_K}\cdot K}\right)^K\geq\left(\frac{\log{m}}{\log{p_K}\cdot K}\right)^K=\frac{\log^K{m}}{(\log{p_K}\cdot K)^K}=c(K,p_K)\cdot\log^K{m},
\]
where $c(K,p_K):=(K\log{p_K})^{-K}$. Because we can construct such a class of examples for each $K$, the maximum number of distinct cycle lengths of an affine map of $\IZ/m\IZ$ is in general not bounded from above by a polynomial in $\log{m}$.
\end{remmmark}

Before we proceed further, we need another auxiliary concept and result.

\begin{defffinition}\label{periodDef}
Let $\vec{\xfrak}=(\xfrak_1,\xfrak_2,\ldots,\xfrak_n)$ be a finite sequence. A \emph{period length of $\vec{\xfrak}$}\phantomsection\label{term86} is a positive divisor $m$ of $n$ such that $\vec{\xfrak}=\diamond_{t=1}^{n/m}{(\xfrak_1,\xfrak_2,\ldots,\xfrak_m)}$, where $\diamond$ denotes concatenation (as in Subsection \ref{subsec3P3}). The smallest positive integer that is a period length of $\vec{\xfrak}$ is denoted by $\minperl(\vec{\xfrak})$\phantomsection\label{not287}.
\end{defffinition}

\begin{remmmark}\label{periodRem}
We note the following concerning Definition \ref{periodDef}.
\begin{enumerate}
\item The number $\minperl(\vec{\xfrak})$ is well-defined because at the very least, the length $n$ of $\vec{\xfrak}$ is a period of it.
\item All elements in $[\vec{\xfrak}]$, the cyclic equivalence class of the sequence $\vec{\xfrak}$ (see the second paragraph after Definition \ref{treeAboveDef}) have the same period lengths, in particular the same $\minperl$-value, as $\vec{\xfrak}$. We denote this common $\minperl$-value by $\minperl([\vec{\xfrak}])$.
\end{enumerate}
\end{remmmark}

We can bound the complexity of computing $\minperl(\vec{\xfrak})$ as follows.

\begin{lemmmma}\label{periodLem}
Let $\vec{\xfrak}\in\{0,1,\ldots,N-1\}^n$, given as a length $n$ list of non-negative integers of bit length $l_{\bit}$. Then $\minperl(\vec{\xfrak})$ can be computed using
\[
O(n\log{n}\log\log{n}(\log{n}+l_{\bit}))\subseteq O(n^{1+o(1)}l_{\bit})
\]
bit operations.
\end{lemmmma}

\begin{proof}
We start by determining $n$, the number of entries of the tuple $\vec{\xfrak}$, in its binary representation with $\lfloor\log_2{n}\rfloor+1$ bits, which takes $O(n\log{n})$ bit operations. Following that, we factor $n$ deterministically. The current record for the bit operation cost of this is $O(n^{1/5+o(1)})$ due to Harvey \cite{Har21a}, building on an ealier breakthrough of Hittmeir \cite{Hit21a} (we could also use an mdl query for this factorization, but then the algorithm is not entirely classical, and $n^{1/5+o(1)}$ is majorized by other terms in this analysis anyway). The rest of the algorithm is analogous to the deterministic procedure for computing a modular multiplicative order described in the proof of Lemma \ref{complexitiesLem2}(2). More specifically, if $n=p_1^{v_1}\cdots p_K^{v_K}$ is the obtained factorization of $n$, then for $j=1,2,\ldots,K$, we perform a binary search to find the smallest $v'_j\in\IN_0$ such that $p_j^{v'_j}\prod_{k\not=j}{p_k^{v_k}}$ is a period length of $\vec{\xfrak}$, which implies that $v'_j=\nu_{p_j}(\minperl(\vec{\xfrak}))$. For each given $j$, it takes $O(\log{v_j})\subseteq O(\log\log{n})$ iterations of the binary search loop until $v'_j$ is found, and each iteration costs $O(n(\log{n}+l_{\bit}))$ bit operations. Because $K\in O(\log{n})$, this means that the total bit operation cost of computing the numbers $v'_j$ is in $O(n\log{n}\log\log{n}(\log{n}+l_{\bit}))$, and $\minperl(\vec{\xfrak})=\prod_{j=1}^K{p_j^{v'_j}}$ takes $O(\log{n}\cdot\log^{2+o(1)}{n})=O(\log^{3+o(1)}{n})$ bit operations to compute by Lemma \ref{complexitiesLem}(3,6).
\end{proof}

We now give the precise definition of the compact description of the isomorphism type of $\Gamma_f$ we aim to obtain.

\begin{defffinition}\label{necklaceListDef}
Let $f$ be a function $X\rightarrow X$, where $X$ is some finite set, and let $\vec{\Ifrak}=(\Ifrak_n)_{n=0,1,\ldots,N}$ be a sequence of pairwise distinct finite rooted tree isomorphism types that contains all isomorphism types of the form $\Tree_{\Gamma_f}(x)$ for $x\in\per(f)$. The \emph{tree necklace list for $f$ relative to $\vec{\Ifrak}$}\phantomsection\label{term87} is the unique set $\{([\vec{\nfrak}_k],l_k,\mfrak_k): k=1,2,\ldots,N'\}$\phantomsection\label{not288} of triples such that the following hold.
\begin{enumerate}
\item $[\vec{\nfrak}_k]=[\nfrak_{k,1},\nfrak_{k,2},\ldots,\nfrak_{k,\ell'_k}]$ is a cyclic sequence of numbers in $\{0,1,\ldots,N\}$ such that $\minperl([\vec{\nfrak}_k])=\ell'_k$.
\item $l_k$ and $\mfrak_k$ are positive integers, and $l_k$ is a multiple of $\ell'_k$.
\item In $\Gamma_f$, there are precisely $\mfrak_k$ connected components corresponding to the cyclic sequence of rooted tree isomorphism types $[\diamond_{t=1}^{l_k/\ell'_k}{(\Ifrak_{\nfrak_{k,1}},\Ifrak_{\nfrak_{k,2}},\ldots,\Ifrak_{\nfrak_{k,\ell'_k}})}]$.
\item For each connected component of $\Gamma_f$, there is a $k\in\{1,2,\ldots,N'\}$ such that the said connected component corresponds to $[\diamond_{t=1}^{l_k/\ell'_k}{(\Ifrak_{\nfrak_{k,1}},\Ifrak_{\nfrak_{k,2}},\ldots,\Ifrak_{\nfrak_{k,\ell'_k}})}]$.
\end{enumerate}
If $f$ is an index $d$ generalized cyclotomic mapping of the finite field $\IF_q$ such that $f$ is of special type I or II respectively, and if $\Rfrak=((\Dfrak_n,S_n))_{n=0,1,\ldots,N}$ is a type-I or -II tree register for $f$, then associated with $\Rfrak$, we have the sequence $(\Ifrak_n)_{n=0,1,\ldots,N}$ of rooted tree isomorphism types where $\Ifrak_n$ has the compact description $\Dfrak_n$. In that case, the tree necklace list for $f$ relative to $\vec{\Ifrak}$ is also called one \emph{relative to $\Rfrak$}\phantomsection\label{term88}.
\end{defffinition}

\begin{remmmark}\label{necklaceListRem}
We make the following comments concerning Definition \ref{necklaceListDef}.
\begin{enumerate}
\item The uniqueness of the tree necklace list for $f$ relative to $\vec{\Ifrak}$ is not hard to prove, but it does require that $\minperl([\vec{\nfrak}_k])=\ell'_k$ for all $k$. For example, without this property, for any $\vec{\Ifrak}$ of length $N+1\geq2$, both $\{([0,1],4,1)\}$ and $\{([0,1,0,1],4,1)\}$ would be tree necklace lists relative to $\vec{\Ifrak}$ for a suitable function $f$.
\item For $j=1,2$, let $X_j$ be a finite set and $f_j$ a function $X_j\rightarrow X_j$. Moreover, let $\vec{\Ifrak}$ be a finite sequence of pairwise distinct, finite rooted tree isomorphism types such that for $j=1,2$, each $\Tree_{\Gamma_{f_j}}(x)$ for $x\in\per(f_j)$ occurs in $\vec{\Ifrak}$. For $j=1,2$, let $\Nfrak_j$\phantomsection\label{not289} be the unique tree necklace list for $f_j$ relative to $\vec{\Ifrak}$. It is not hard to prove that $\Nfrak_1=\Nfrak_2$ if and only if $\Gamma_{f_1}\cong\Gamma_{f_2}$.
\item Let $f$ be an index $d$ generalized cyclotomic mapping of $\IF_q$ that is of special type I or II. We need to fix suitable bit string encodings of the elements of a tree necklace list for $f$, making the following conventions. By assumption, if $i\in\{0,1,\ldots,d\}$ has cycle length $\ell$ under $\overline{f}$, then the cyclic sequence of rooted tree isomorphism types encoding the connected component of $\Gamma_f$ containing any $f$-periodic vertex from $C_i$ has minimal period some divisor of $\ell$. In particular, the said minimal period is always at most $d$. Moreover, we assume that $\vec{\Ifrak}$ stems from a recursive tree description list $\vec{\Dfrak}$ that is part of a type-I or -II tree register $\Rfrak$ for $f$. In $\vec{\Dfrak}$, there is a common bit length to represent numbers from $\{0,1,\ldots,N\}$ (see the remarks after Definition \ref{treeRegDef}); we denote that bit length by $l_{\bit}$ and observe that
\[
l_{\bit}=
\begin{cases}
\lfloor\log_2{d}\rfloor+1\in O(\log{d}), & \text{if }\Rfrak\text{ has type I}, \\
\lfloor\log_2(d^2\lfloor\log_2{q}\rfloor+d)\rfloor+1\in O(\log{d}+\log\log{q}), & \text{if }\Rfrak\text{ has type II}.
\end{cases}
\]
A cyclic sequence $[\vec{\nfrak}]=[\nfrak_1,\nfrak_2,\ldots,\nfrak_{\ell'}]$ as above is the first entry of an element of the tree necklace list for $f$ relative to $\Rfrak$; we assume that the associated ordered sequence $\vec{\nfrak}=(\nfrak_1,\nfrak_2,\ldots,\nfrak_{\ell'})$ is lexicographically minimal among all ordered sequences in its cyclic equivalence class $[\vec{\nfrak}]$. We encode $[\vec{\nfrak}]$ as follows. We take the ordered sequence $\vec{\nfrak}$ and fill it up with entries $-1$ (a dummy value) until it has $d$ entries. We then print a bit string that is a concatenation of encodings of the entries of this extended sequence. We use $l_{\bit}+1$ bits to denote each entry, where an entry other than $-1$ is represented by an ancillary bit $1$, followed by the length $l_{\bit}$ binary digit representation of that entry from $\Rfrak$. On the other hand, an entry $-1$ is represented by a string of $l_{\bit}+1$ zeroes. For example, if $l_{\bit}=3$ and $d=5$, then the bit string encoding of $[6,3,4]=[3,4,6]$ in the corresponding tree necklace list is
\[
10111100111000000000.
\]
On the other hand, the second and third entries of elements of any tree necklace list for $f$ are positive integers that are at most $q$, and we simply use their standard binary representations with $\lfloor\log_2{q}\rfloor+1$ digits to represent them; these may be directly concatenated with the bit string encoding of $[\vec{\nfrak}]$. With these conventions, all elements of a given tree necklace list for $f$ are bit strings of the same bit length, namely
\[
d(l_{\bit}+1)+2(\lfloor\log_2{q}\rfloor+1)\in
\begin{cases}
O(d\log{d}+\log{q}), & \text{if }\Rfrak\text{ has type I}, \\
O(d\log{d}+d\log\log{q}+\log{q}), & \text{if }\Rfrak\text{ has type II},
\end{cases}
\]
which allows us to use the sorting algorithm from Lemma \ref{complexitiesLem}(10) to sort them lexicographically. Moreover, the lexicographic ordering of those bit strings corresponds to the \enquote{natural} lexicographic ordering of the elements of the associated (abstract) tree necklace list (using the lexicographic ordering among lexicographically minimal representatives of cyclic sequences in the first component, and the usual linear ordering of integers in the second and third component). It should be noted that our approach involves some padding, and this could be avoided through using \cite[Algorithm 3.2 on pp.~89f.]{AHU75a}, which is a more general lexicographic sorting algorithm that does not require the bit strings from the input to be of a common length. However, in terms of the $O$-class of the complexity bounds we derive, it does not make a difference.
\end{enumerate}
\end{remmmark}

Remark \ref{necklaceListRem}(2) guarantees that tree necklace lists are \emph{injective} descriptions of digraph isomorphism types of functional graphs, but they may not always be compact. Indeed, they contain as many elements as there are distinct isomorphism types of connected components of the said functional graph, and in the case of the functional graph $\Gamma_f$ of a generalized cyclotomic mapping $f$ of $\IF_q$ of a fixed index $d$, the maximum number of such connected components is at least the maximum number of distinct cycle lengths which an affine permutation of $\IZ/s\IZ$ can achieve. That latter number can, a priori, be superpolynomial in $s=(q-1)/d$ (and thus in $q$ if $d$ is fixed), see Remark \ref{ctTauRem}. We note however, that the moduli considered in Remark \ref{ctTauRem} are of a special form, and it is not clear whether the construction from Remark \ref{ctTauRem} can still be carried out if, additionally, all constructed moduli must be of the form $(q-1)/d$ for some prime power $q$ with $d\mid q-1$, where $d\in\IN^+$ is fixed. Moreover, by our Proposition \ref{tauAvProp}, as long as one is willing to exclude a small positive asymptotic fraction of prime powers, then $\tau(q-1)$, which is an upper bound on the number of distinct cycle lengths of an affine map of $\IZ/s\IZ$ (see the proof of Proposition \ref{ctTauProp}), \emph{is} polynomial in $\log{q}$. Even for such prime powers $q$, the number of distinct isomorphism types of connected components of $\Gamma_f$ itself could be superpolynomial in $\log{q}$, however. We leave the problem of finding a concrete infinite class of examples that confirms this open; see also Problems \ref{isomorphismTypesProb} and \ref{isomorphismTypesProb2}.

Nonetheless, if $f$ is of special type I or II, then the following key result implies that one may compute a tree register $\Rfrak$ for $f$ and, subsequently, compute and print the tree necklace list for $f$ relative to $\Rfrak$ within a $q$-bounded query complexity that is polynomial in $\log{q}$, $d$ and $\tau(q-1)$.

\begin{theoremmm}\label{necklaceListTheo}
Let $f$ be an index $d$ generalized cyclotomic mapping of $\IF_q$.
\begin{enumerate}
\item We assume that $f$ is of special type I. Then one can compute within a $q$-bounded query complexity of
\begin{align*}
(&d^3\log^2{d}+d^3\tau(q-1)^2\log{q}+d^2\log^{1+o(1)}{q}+d\tau(q-1)\log^{2+o(1)}{q}, \\
&d,d\log{q},d,0),
\end{align*}
or a $q$-bounded Las Vegas dual complexity of
\begin{align*}
(&d^3\log^2{d}+d^3\tau(q-1)^2\log{q}+d^2\log^{1+o(1)}{q}+d\tau(q-1)\log^{2+o(1)}{q} \\
&+d\log^{8+o(1)}{q},d\log^{4+o(1)}{q},d\log^2{q}),
\end{align*}
a type-I tree register $\Rfrak=((\Dfrak_n,S_n))_{n=0,1,\ldots,N}$ for $f$ with $N\in O(d)$, as well as the tree necklace list of $f$ relative to $\Rfrak$.
\item We assume that $f$ is of special type II. Then one can compute within a $q$-bounded query complexity of
\begin{align*}
(&d^5\log^2{d}\mpe(q-1)+(d^5\mpe(q-1)^3+d^3\tau(q-1)^2)\log{q} \\
&+d^2\mpe(q-1)\log^{1+o(1)}{q}+d\tau(q-1)\log^{2+o(1)}{q},d,d\log{q},d,0),
\end{align*}
or a $q$-bounded Las Vegas dual complexity of
\begin{align*}
(&d^5\log^2{d}\mpe(q-1)+(d^5\mpe(q-1)^3+d^3\tau(q-1)^2)\log{q} \\
&+d^2\mpe(q-1)\log^{1+o(1)}{q}+d\tau(q-1)\log^{2+o(1)}{q}+d\log^{8+o(1)}{q}, \\
&d\log^{4+o(1)}{q},d\log^2{q}),
\end{align*}
a type-II tree register $\Rfrak=((\Dfrak_n,S_n))_{n=0,1,\ldots,N}$ of $f$ with $N\in O(d^2\mpe(q-1))$, as well as the tree necklace list of $f$ relative to $\Rfrak$.
\end{enumerate}
\end{theoremmm}

\begin{proof}
We prove both statements simultaneously, referring with \enquote{case I}, respectively \enquote{case II}, to the situation described in statement (1), respectively (2). First, we compute $\overline{f}$, the affine maps $A_i$, and a tree register $\Rfrak$ for $f$ of the desired type and with the asserted bound on $N$, taking $q$-bounded query complexity
\begin{itemize}
\item $(d^3\log^2{d}+d\log^{1+o(1)}{q},d,0,0,0)$ in case I,
\item $(d^2\mpe(q-1)\log^{1+o(1)}{q}+d^5\mpe(q-1)^3\log{q}+d^5\log^2{d}\mpe(q-1)^3,d,0,0,0)$ in case II
\end{itemize}
by Proposition \ref{faiProp} and Lemma \ref{treeRegLem}(2,4). Then we compute a CRL-list $\overline{\Lcal}$ of $\overline{f}$ together with the cycles of $\overline{f}$, taking $O(d^2\log^2{d})$ bit operations by the argument at the beginning of Subsubsection \ref{subsubsec5P2P1}. For each $(i,\ell)\in\overline{\Lcal}$, letting $(i_0,i_1,\ldots,i_{\ell-1})$ with $i=i_0$ denote the $\overline{f}$-cycle of $i$ determined earlier, we compute the tree necklace list $\Nfrak_i$, relative to $\Rfrak$, for the restriction of $f$ to $\bigcup_{t=0}^{\ell-1}{C_{i_t}}$ as follows.

If $i=d$, we simply set $\Nfrak_i:=\{([\nfrak],1,1)\}$ where $\nfrak\in\{0,1,\ldots,N\}$ is the positive integer that represents $\Tree_{\Gamma_f}(0_{\IF_q})$ in $\Rfrak$. We observe that in case II, where $\Tree_{\Gamma_f}(0_{\IF_q})$ is trivial, one has $\nfrak=0$ necessarily, whereas in case I, the number $\nfrak$ is uniquely characterized by the inclusion $d\in S_{\nfrak}$, and can thus be determined with $O(d\log{d})$ bit operations (using that each $\nfrak$ is represented by a bit string of length in $O(\log{d})$).

Now we assume that $i<d$. Then we compute $\Acal_i=A_{i_0}A_{i_1}\cdots A_{i_{\ell-1}}$, taking $O(d\log^{1+o(1)}{q})$ bit operations (see Subsubsection \ref{subsubsec5P2P1}, page \pageref{not222}), as well as the cycle type $\CT((\Acal_i)_{\mid\per(\Acal_i)})=x_1^{e_{i,1}}x_2^{e_{i,2}}\cdots x_s^{e_{i,s}}$ (where some of the $e_{i,j}$ may be $0$), taking $q$-bounded query complexity
\[
(\tau(q-1)\log^{2+o(1)}{q}+\tau(q-1)^2\log{q},0,\log{q},1,0)
\]
by Proposition \ref{ctTauProp}. Moreover, we compute the number sequence
\[
\vec{\nfrak}_i=(\nfrak_{i_0},\nfrak_{i_1},\ldots,\nfrak_{i_{\ell-1}})
\]
where $\nfrak_{i_t}$ is uniquely characterized by the inclusion
\[
\begin{cases}
i_t\in S_{\nfrak_{i_t}}, & \text{in case I}, \\
i_t\in S_{\nfrak_{i_t},\per}, & \text{in case II},
\end{cases}
\]
i.e., $\nfrak_{i_t}$ is the positive integer that represents, in the register $\Rfrak$, the unique rooted tree isomorphism type above periodic vertices of $\Gamma_f$ that are contained in $C_{i_t}$. The number of bit operations it takes to determine (and set the value of) $\nfrak_{i_t}$ for all $t=0,1,\ldots,\ell-1$ can be bounded as follows:
\begin{itemize}
\item in case I, it is in
\[
O\left(d\cdot\sum_{n=0}^N{|S_n|}\cdot\log{d}+dl_{\bit}\right)\subseteq O(d^2\log{d});
\]
\item in case II, it is in
\[
O\left(d\left(N+\sum_{n=0}^N{|S_{n,\per}|}\right)\log{d}+dl_{\bit}\right)\subseteq O(d^3\log{d}\mpe(q-1)+d\log\log{q}).
\]
\end{itemize}
Next, we overwrite $\vec{\nfrak}_i$ with the unique lexicographically minimal number sequence in the same cyclic equivalence class. We can do do by spelling the $O(\ell)\subseteq O(d)$ cyclic shifts of $\vec{\nfrak}_i$ out, then ordering them lexicographically as in Lemma \ref{complexitiesLem}(10) and taking the first sequence in the sorted list. This process takes
\[
O(d\log{d}\cdot dl_{\bit})=O(d^2\log{d}l_{\bit})\subseteq
\begin{cases}
O(d^2\log^2{d}), & \text{in case I}, \\
O(d^2\log^2{d}+d^2\log{d}\log\log{q}), & \text{in case II}
\end{cases}
\]
bit operations. Following that, we compute $\minperl(\vec{\nfrak}_i)$, which by Lemma \ref{periodLem} takes the following amount of bit operations:
\begin{align*}
&O(d\log{d}\log\log{d}\cdot(\log{d}+l_{\bit})) \\
\subseteq
&\begin{cases}
O(d\log^2{d}\log\log{d}), & \text{in case I}, \\
O(d\log{d}\log\log{d}(\log{d}+\log\log{q})), & \text{in case II}.
\end{cases}
\end{align*}

We observe that the following is a valid choice for $\Nfrak_i$:
\[
\Nfrak_i:=\{([\nfrak_{i_0},\nfrak_{i_1},\ldots,\nfrak_{i_{\minperl(\vec{\nfrak}_i)-1}}],\ell\cdot l',e_{i,l'}): l'\in\{1,2,\ldots,s\}, e_{i,l'}>0\}.
\]
For further processing, rather than store $\Nfrak_i$ literally as this list, it is more advantageous to store $\vec{\nfrak}'_i:=[\nfrak_{i_0},\nfrak_{i_1},\ldots,\nfrak_{i_{\minperl(\vec{\nfrak}_i)-1}}]$ and the list $\Nfrak'_i:=\{(\ell\cdot l',e_{i,l'}): l'\in\{1,2,\ldots,s\}, e_{i,l'}>0\}$ separately, which takes
\begin{align*}
&O(\minperl(\vec{\nfrak}_i)l_{\bit}+\tau(q-1)\log^{1+o(1)}{q}) \\
&\subseteq
\begin{cases}
O(d\log{d}+\tau(q-1)\log^{1+o(1)}{q}), & \text{in case I}, \\
O(d(\log{d}+\log\log{q})+\tau(q-1)\log^{1+o(1)}{q}), & \text{in case II}
\end{cases}
\end{align*}
bit operations for carrying out the multiplications $\ell\cdot l'$ and copying data. This concludes our analysis of how to compute $\Nfrak_i$, which in summary takes $q$-bounded query complexity
\[
(\tau(q-1)\log^{2+o(1)}{q}+(d+\tau(q-1))\log^{1+o(1)}{q}+\tau(q-1)^2\log{q}+E(d,q),0,\log{q},1,0)
\]
where
\begin{align*}
&E(d,q):=Nd\log{d}+d^2\log{d}l_{\bit} \\
\in
&\begin{cases}
O(d^2\log^2{d}), & \text{in case I}, \\
O(d^3\log{d}\mpe(q-1)+d^2\log^2{d}+d^2\log{d}\log\log{q}), & \text{in case II}
\end{cases}
\end{align*}
for each given $i$, and the (component-wise) $d$-fold of that complexity for all $i$ together.

Finally, we need to compute the actual tree necklace list $\Nfrak$ for $f$ relative to $\Rfrak$. We start by setting $\Mfrak:=\Nfrak:=\emptyset$ and $\vec{\nfrak}':=\{(\vec{\nfrak}'_i,i,\ell): (i,\ell)\in\overline{\Lcal}\}$ (where $\vec{\nfrak}'_d:=[\nfrak]$, the first entry of the unique triple in $\Nfrak_d$), followed by sorting $\vec{\nfrak}'$ lexicographically. Altogether, this takes $O(d^2\log{d}l_{\bit})$ bit operations. Throughout the subsequently described process, $\Nfrak$ is a (lexicographically sorted) initial segment of the tree necklace list that will be output, and $\Mfrak$ is an initial segment of the list obtained by deleting repeated entries in the multiset $\{\vec{\nfrak}'_i: (i,\ell)\in\overline{\Lcal}\}$. In particular, $\Mfrak$ has $O(d)$ elements, each of which is a cyclic sequence of length in $O(d)$ each entry of which has bit length in $O(l_{\bit})$ and which is given by its lexicographically minimal representative. We observe that two such cyclic sequences are equal if and only if their representatives are equal, so it takes $O(dl_{\bit})$ bit operations to verify such an equality.

Now, we go through the triples $(\vec{\nfrak}'_i,i,\ell)\in\vec{\nfrak}'$, and for each of them, we do the following. First, we check whether $\vec{\nfrak}'_i\in\Mfrak$, which takes $O(d^2l_{\bit})$ bit operations. If so, we skip to the next triple in $\vec{\nfrak}'$, otherwise we proceed as follows. We add $\vec{\nfrak}'_i$ to $\Mfrak$ as a new element (which takes $O(dl_{\bit})$ bit operations for copying), and set $\Vfrak:=\Nfrak'_i$, taking $O(\tau(q-1)\log{q})$ bit operations.

Then, we go through the elements $(\vec{\nfrak}'_j,j,\ell')\in\vec{\nfrak}'$ that come after $(\vec{\nfrak}'_i,i,\ell)$, and for each of them, we do the following. We check whether $\vec{\nfrak}'_j=\vec{\nfrak}'_i$, which takes $O(dl_{\bit})$ bit operations. If not, we skip to the next value of $(\vec{\nfrak}'_j,j,\ell')$, otherwise we proceed as follows. We go through the elements $(l,k)$ of $\Nfrak'_j$, and for each of them, we check whether $l$ occurs as the first entry of some pair $(l,k')$ of $\Vfrak$; each such check takes $O(|\Vfrak|\log{q})\subseteq O(d\tau(q-1)\log{q})$ bit operations. If so, we replace the unique element of $\Vfrak$ of the form $(l,k')$ by $(l,k+k')$; otherwise, we add $(l,k)$ to $\Vfrak$ as a new element.

Overall, the described loop over the elements of $\Nfrak'_j$ takes
\[
O(|\Nfrak'_j|\cdot d\tau(q-1)\log{q})\subseteq O(d\tau(q-1)^2\log{q})
\]
bit operations, and thus the loop over $(\vec{\nfrak}'_j,j,\ell')$ takes $O(d^2l_{\bit}+d^2\tau(q-1)^2\log{q})$ bit operations. Once that loop is finished, we complete the loop over $(\vec{\nfrak}'_i,i,\ell)\in\vec{\nfrak}'$ by adding, for each $(l,k)\in\Vfrak$, the triple $([\vec{\nfrak}'_i],l,k)$ to $\Nfrak$ as a new element -- this copying process takes
\[
O(|\Vfrak|\cdot(dl_{\bit}+\log{q}))\subseteq O(d\tau(q-1)(dl_{\bit}+\log{q}))
\]
bit operations.

At the end of the loop over $(\vec{\nfrak}'_i,i,\ell)\in\vec{\nfrak}'$, the variable $\Nfrak$ has its desired value, and the overall bit operation cost of this loop is in
\[
O(d\cdot(d^2l_{\bit}+d^2\tau(q-1)l_{\bit}+d^2\tau(q-1)^2\log{q}))=O(d^3\tau(q-1)^2\log{q})
\]
\end{proof}

Finally, we discuss the complexity of the digraph isomorphism problem. Let $f_1$ and $f_2$ be generalized cyclotomic mappings of $\IF_q$, each of one of the special types I or II. We note that neither do $f_1$ and $f_2$ need to have the same index, nor are they necessarily both of the same special type. In order to decide whether $\Gamma_{f_1}\cong\Gamma_{f_2}$, we would like to compare a computed tree necklace list for $f_1$ with one for $f_2$. To that end, those tree necklace lists must be \enquote{synchronized}, so that each $n$ denotes the same rooted tree isomorphism type $\Ifrak_n$ in each case. Here is a precise definition.

\begin{defffinition}\label{synchronizeDef}
Let $\vec{\Dfrak}=(\Dfrak_n)_{n=0,1,\ldots,N}$ and $\vec{\Dfrak}'=(\Dfrak'_n)_{n=0,1,\ldots,N'}$ be recursive tree description lists, with associated rooted tree isomorphism type sequences $(\Ifrak_n)_{n=0,1,\ldots,N}$ and $(\Ifrak'_n)_{n=0,1,\ldots,N'}$. A \emph{synchronization of $\vec{\Dfrak}$ and $\vec{\Dfrak}'$}\phantomsection\label{term89} is a pair $(\vec{\Dfrak}^+,\ifrak)$\phantomsection\label{not290} such that the following hold.
\begin{enumerate}
\item $\vec{\Dfrak}^+=(\Dfrak^+_n)_{n=0,1,\ldots,N^+}$ is a recursive tree description list of which $\vec{\Dfrak}$ is an initial segment (in particular, $N\leq N^+$). We denote by $(\Ifrak^+_n)_{n=0,1,\ldots,N^+}$ the unique rooted tree isomorphism type sequence associated with $\vec{\Dfrak}^+$.
\item $\ifrak$ is a function $\{0,1,\ldots,N'\}\rightarrow\{0,1,\ldots,N^+\}$ with $\{N+1,N+2,\ldots,N^+\}\subseteq\im(\ifrak)$ such that for each $n\in\{0,1,\ldots,N'\}$, one has $\Ifrak'_n\cong\Ifrak^+_{\ifrak(n)}$.
\end{enumerate}
\end{defffinition}

In an implementation, we assume that each description $\Dfrak_n$ or $\Dfrak'_n$ is sorted by increasing first entries of its elements. We also assume that each of $\vec{\Dfrak}$ and $\vec{\Dfrak}'$ uses a common bit length, denoted by $l_{\bit}$ and $l'_{\bit}$ respectively, for the binary representations of the numbers $n$. Because $N^+\leq N+N'$, we use $\max\{l_{\bit},l'_{\bit}\}+1\in O(l_{\bit}+l'_{\bit})$ bits for the numbers $n$ in the synchronization $\vec{\Dfrak}^+$. We note that the definition of a synchronization is asymmetric in the sense that a synchronization of $\vec{\Dfrak}$ and $\vec{\Dfrak}'$ is not necessarily also one of $\vec{\Dfrak}'$ and $\vec{\Dfrak}$. Complexity-wise, the following lemma shows that it is slightly more advantageous to have $N'\leq N$.

\begin{lemmmma}\label{synchronizeLem}
Let $\vec{\Dfrak}=(\Dfrak_n)_{n=0,1,\ldots,N}$ and $\vec{\Dfrak}'=(\Dfrak'_n)_{n=0,1,\ldots,N'}$ be recursive tree description lists such that for all $n$,
\begin{itemize}
\item each second entry of an element of $\Dfrak_n$ or $\Dfrak'_n$ is represented by a bit string of length at most $m$ (a quantity that does not depend on $n$);
\item each first entry of each element of $\Dfrak_n$, respectively of $\Dfrak'_n$, is represented by a bit string of length exactly $l_{\bit}$, respectively $l'_{\bit}$, and
\item within a given description $\Dfrak_n$, respectively $\Dfrak'_n$, the second entries of elements of that description have a common bit length (so all pairs in $\Dfrak_n$, respectively in $\Dfrak'_n$, have the same bit length, which lies in $O(l_{\bit}+m)$, respectively in $O(l'_{\bit}+m)$).
\end{itemize}
It takes $O((NN'\min\{N,N'\}+(\max\{N,N'\})^2+(N')^2l'_{\bit})(l_{\bit}+l'_{\bit}+m))$ bit operations to compute a synchronization of $\vec{\Dfrak}$ and $\vec{\Dfrak}'$.
\end{lemmmma}

\begin{proof}
Let $\Ifrak_n$, respectively $\Ifrak'_n$, be the rooted tree isomorphism type described by $\Dfrak_n$, respectively by $\Dfrak'_n$. In order to compute the synchronization, we proceed in a loop over $n=0,1,\ldots,N'$. At each given point in the process, the description $\Dfrak^+_k$ (of the rooted tree isomorphism type $\Ifrak^+_k$) is defined for all $k\in\Ncal$, an initial segment of $\IN_0$ that starts out as $\{0,1,\ldots,N\}$, with $\Dfrak^+_k:=\Dfrak_k$ for each $k\in\{0,1,\ldots,N\}$, and will be $\{0,1,\ldots,N^+\}$ in the end. We note that it takes
\[
O(N^2(l_{\bit}+\log{m}))\subseteq O((\max\{N,N'\})^2(l_{\bit}+l'_{\bit}+m))
\]
bit operations (spent copying) to set $\Ncal$ and the $\Dfrak^+_k$ for $k\in\{0,1,\ldots,N\}$ up. At any given point in the algorithm, the descriptions $\Dfrak^+_k$ form a recursive tree description list denoted by $\vec{\Dfrak}^+$ (which will have the desired value in the end). We also keep updating the value of $\ifrak:\{0,1,\ldots,n\}\rightarrow\Ncal$, which starts out as the empty function $\emptyset$.

For $n=0$, where $\Dfrak'_n=\emptyset$ and its associated rooted tree is trivial, we simply set $\ifrak(0):=0$ without updating $\Ncal$. Now let us assume that $n\geq1$. Based on $\Dfrak'_n$, we compute a new rooted tree description $\Dfrak$ through replacing the first entry $k<n\leq N'$ of each given pair in $\Dfrak'_n$ by $\ifrak(k)$. Because $\Dfrak'_n$ contains at most $n+1\in O(n)$ distinct pairs and we are handling non-negative integers of bit length in $O(l_{\bit}+l'_{\bit})$ here, computing $\Dfrak$ as an unsorted list takes $O(n\cdot(l_{\bit}+l'_{\bit}))\subseteq O(N'(l_{\bit}+l'_{\bit}))$ bit operations overall (for a given $n$), and another $O(N'l'_{\bit}(l_{\bit}+l'_{\bit}+m))$ bit operations for sorting $\Dfrak$.

Once $\Dfrak$ has been computed in sorted form, we need to check whether the rooted tree $\Ifrak'_n$ described by it with respect to $\vec{\Dfrak}^+$ already occurs among the $\Ifrak^+_k$. If $k>N$, then $\Ifrak^+_k\cong\Ifrak'_l$ for some $l\in\{0,1,\ldots,n-1\}$, and thus $\Ifrak^+_k\not\cong\Ifrak'_n$, as the isomorphism types $\Ifrak'_t$ are pairwise distinct by assumption. Therefore, we only need to check the isomorphism $\Ifrak^+_k\cong\Ifrak'_n$ for $k\leq N$, where it is equivalent to $\Ifrak_k\cong\Ifrak'_n$ and further to $\Dfrak_k=\Dfrak$. For a given $k$, it takes $O(\min\{N,N'\}(l_{\bit}+l'_{\bit}+m))$ bit operations to check with a linear scan whether $\Dfrak_k=\Dfrak$ (using that both $\Dfrak_k$ and $\Dfrak$ are sorted), and so it can be checked with $O(N\min\{N,N'\}(l_{\bit}+l'_{\bit}+m))$ bit operations whether $\Ifrak'_n\cong\Ifrak^+_k$ for a (unique) $k\in\{0,1,\ldots,N\}$. If so, we set $\ifrak(n):=k$ without updating $\Ncal$; otherwise, we extend $\Ncal$ by the new element $n':=\max{\Ncal}+1$ and set $\Dfrak^+_{n'}:=\Dfrak$ and $\ifrak(n):=n'$.

The overall bit operation cost of the described loop is in $O((NN'\min\{N,N'\}+(N')^2l'_{\bit})(l_{\bit}+l'_{\bit}+m))$. We conclude the algorithm by outputting $(\vec{\Dfrak}^+,\ifrak)$, where $\vec{\Dfrak}^+:=(\Dfrak^+_k)_{k\in\Ncal}$. This takes
\begin{align*}
&O((N+N')\max\{N,N'\}(l_{\bit}+l'_{\bit}+m)) \\
=&O((\max\{N,N'\})^2(l_{\bit}+l'_{\bit}+m))
\end{align*}
bit operations for copying.
\end{proof}

\begin{corrrollary}\label{necklaceListCor}
Let $f_1$ and $f_2$ be generalized cyclotomic mappings of a common finite field $\IF_q$, say of index $d_1$ and $d_2$, respectively, and set $d:=\max\{d_1,d_2\}$.
\begin{enumerate}
\item If each $f_j$ is of special type I or II (not necessarily both of the same type), then it takes $q$-bounded query complexity
\begin{align*}
(&(d^6\mpe(q-1)^3+d^3\tau(q-1)^2)\log{q}+d^4\mpe(q-1)^2\log^{1+o(1)}{q} \\
&+d\tau(q-1)\log^{2+o(1)}{q},d,d\log{q},d,0),
\end{align*}
or $q$-bounded Las Vegas dual complexity
\begin{align*}
(&(d^6\mpe(q-1)^3+d^3\tau(q-1)^2)\log{q}+d^4\mpe(q-1)^2\log^{1+o(1)}{q} \\
&+d\tau(q-1)\log^{2+o(1)}{q}+d\log^{8+o(1)}{q},d\log^{4+o(1)}{q},d\log^2{q}),
\end{align*}
to decide whether $\Gamma_{f_1}\cong\Gamma_{f_2}$.
\item If both $f_j$ are of special type I, then it takes $q$-bounded query complexity
\begin{align*}
(&d^3\log^2(d)+d^3\tau(q-1)^2\log{q}+d^2\log^{1+o(1)}{q} \\
&+d\tau(q-1)\log^{2+o(1)}{q},d,d\log{q},d,0),
\end{align*}
or $q$-bounded Las Vegas dual complexity
\begin{align*}
(&d^3\log^2(d)+d^3\tau(q-1)^2\log{q}+d^2\log^{1+o(1)}{q} \\
&+d\tau(q-1)\log^{2+o(1)}{q}+d\log^{8+o(1)}{q},d\log^{4+o(1)}{q},d\log^2{q}),
\end{align*}
to decide whether $\Gamma_{f_1}\cong\Gamma_{f_2}$.
\end{enumerate}
\end{corrrollary}

\begin{proof}
We denote the situation in statement (1) by \enquote{case (1)}, and the one in statement (2) by \enquote{case (2)}; these must \emph{not} be confused with cases I and II from the proof of Theorem \ref{necklaceListTheo}. First, for $j=1,2$, we compute a suitable tree register $\Rfrak_j$ of $f_j$ with $N_j+1$ entries, together with an associated tree necklace list $\Nfrak_j$ for $f_j$. By Theorem \ref{necklaceListTheo}, this takes $q$-bounded query complexity
\begin{align*}
(&d^5\log^2{d}\mpe(q-1)+d^3\log{d}\tau(q-1)+d^3\tau(q-1)\log\log{q} \\
&+(d^5\mpe(q-1)^3+d^3\tau(q-1)^2)\log{q}+d^2\mpe(q-1)\log^{1+o(1)}{q} \\
&+d\tau(q-1)\log^{2+o(1)}{q},d,d\log{q},d,0),
\end{align*}
in case (1), or
\begin{align*}
(&d^3\log^2{d}+d^3\log{d}\tau(q-1)+d^3\tau(q-1)^2\log{q}+d^2\log^{1+o(1)}{q} \\
&+d\tau(q-1)\log^{2+o(1)}{q},d,d\log{q},d,0),
\end{align*}
in case (2).

Following that, we synchronize the underlying recursive tree description lists of the $\Rfrak_j$. By Lemma \ref{synchronizeLem}, applied with $m:=\log{q}$, and using the facts that $l_{\bit}\in O(\log{q})$ in either case and that
\[
\min\{N_1,N_2\}\leq\max\{N_1,N_2\}\in
\begin{cases}
O(d^2\mpe(q-1)), & \text{in case (1)}, \\
O(d), & \text{in case (2)},
\end{cases}
\]
we see that this can be done using
\begin{align*}
&O((d^6\mpe(q-1)^3+d^4\mpe(q-1)^2\log\log{q})(\log{d}+\log{\mpe(q-1)}+\log{q})) \\
=&O((d^6\mpe(q-1)^3+d^4\mpe(q-1)^2\log\log{q})\log{q})
\end{align*}
bit operations in case (1), or $O(d^3\log{q})$ bit operations in case (2). Following our convention on the bit lengths of indices $n$ of descriptions $\Dfrak^+_n$ in synchronizations, the computed synchronization uses a common bit length $l^+_{\bit}$, which lies in $O(\log{d}+\log\log{q})$ in case (1), and in $O(\log{d})$ in case (2).

Next, based on $\Nfrak_2$, we compute a modified tree necklace list $\Nfrak'_2$ for $f_2$ with respect to the computed synchronization $(\vec{\Dfrak}^+,\ifrak)$ through replacing each number $\nfrak$ occurring as an entry in one of the cyclic sequences in $\Nfrak_2$ by $\ifrak(\nfrak)$, then replacing the underlying ordered sequence with the lexicographically minimal representative in the same cyclic equivalence class. Computing $\Nfrak'_2$ in unsorted form takes
\[
O(|\Nfrak_2|\cdot (dl^+_{\bit}+\log{q})+d\cdot d\log{d}l^+_{\bit})\subseteq O(d^2\tau(q-1)\log{q}+d^2\log{d}\log{q})
\]
bit operations, and following that, we sort $\Nfrak'_2$ lexicographically, which takes $O(d^3\tau(q-1)\log{q})$ bit operations through successively merging the $O(d)$ segments corresponding to a common, rewritten first entry (see Lemma \ref{complexitiesLem}(11)).

Finally, we need to check whether $\Nfrak'_2=\Nfrak_1$, which only takes a linear scan thanks to $\Nfrak'_2$ and $\Nfrak_1$ both being lexicographically sorted (we note that the bit lengths of indices $n$ in $\Nfrak_1$ may not be the same as those in $\Nfrak'_2$, but that is of course not a problem). The bit operation cost of this is in
\[
O(\min\{|\Nfrak_1|,|\Nfrak'_2|\}\cdot(dl^+_{\bit}+\log{q}))\subseteq O(d\tau(q-1)\cdot d\log{q})=O(d^2\tau(q-1)\log{q}).
\]
\end{proof}

As we did throughout Subsection \ref{subsec5P2} and in Subsubsection \ref{subsubsec5P3P1}, we conclude this subsubsection with pseudocode for all relevant algorithms introduced in it, specifying the query complexity ($q$-bounded or $m$-bounded, depending on the context) of each step. We start with the algorithm from Lemma \ref{treeRegLem}(1), which decides whether a given index $d$ generalized cyclotomic mapping $f$ of $\IF_q$ is of special type I.

\begin{denumerate}[label=\arabic*]
\item Compute the affine maps $A_i:x\mapsto\alpha_ix+\beta_i$ of $\IZ/s\IZ$ associated with $f$.

QC: $(d\log^{1+o(1)}{q},d,0,0,0)$.
\item For each $i\in\{0,1,\ldots,d-1\}$, do the following.

QC: $(d\log^{1+o(1)}{q},0,0,0,0)$.
\begin{denumerate}[label=2.\arabic*]
\item If $a_i=0_{\IF_q}$, then skip to the next $i$.

QC: $(\log{d},0,0,0,0)$.
\item If $\gcd(\alpha_i,s)>1$, then output \enquote{false} and halt.

QC: $(\log^{1+o(1)}{q},0,0,0,0)$.
\end{denumerate}
\item Output \enquote{true} and halt.
\end{denumerate}

Next, we list the steps of the algorithm from Lemma \ref{treeRegLem}(2), which computes, for a given index $d$ generalized cyclotomic mapping $f$ of $\IF_q$ that is of special type I, a type-I tree register $\Rfrak=((\Dfrak_n,S_n))_{n=0,1,\ldots,N}$ for $f$ with $N\in O(d)$.

\begin{denumerate}[label=\arabic*]
\item Compute $s=(q-1)/d$, the induced function $\overline{f}$ on $\{0,1,\ldots,d\}$ and the layers $\Layer_h$, for $h\in\{0,1,\ldots,\overline{H}-1,\infty\}$, of $\overline{f}$ with respect to iteration.

QC: $(d^2\log^2{d}+d\log^{1+o(1)}{q},d,0,0,0)$.
\item Set $\Ncal:=\emptyset$.

QC: $(1,0,0,0,0)$.
\item For each $h=0,1,\ldots,\overline{H}-1,\infty$, do the following.

QC: $(d^3\log^2{d}+d\log^{1+o(1)}{q},0,0,0,0)$.
\begin{denumerate}[label=3.\arabic*]
\item If $h=0$ or $\overline{H}=0$, then do the following.
\begin{denumerate}[label=3.1.\arabic*]
\item Add $0$ to $\Ncal$ as a new (the first) element, set $\Dfrak_0:=\emptyset$ and $S_0:=\Layer_h$.

QC: $(d\log{d},0,0,0,0)$.
\item If $\overline{H}=0$, then output $\Rfrak:=((\Dfrak_0,S_0))$ and halt.

QC: $(d\log{d},0,0,0,0)$.
\end{denumerate}
\item Else do the following.
\begin{denumerate}[label=3.2.\arabic*]
\item For each $i\in\Layer_h$, do the following.

QC: $(|\Layer_h|d^2\log^2{d}+\delta_{d\in\Layer_h}d\log^{1+o(1)}{q},0,0,0,0)$.
\begin{denumerate}[label=3.2.1.\arabic*]
\item Determine the full pre-image set $\overline{f}^{-1}(\{i\})$ and the $\overline{f}$-transient $\overline{f}$-pre-images $j_1,j_2,\ldots,j_K$ of $i$.

QC: $(d\log{d},0,0,0,0)$ if $h<\infty$; $(d^2\log^2{d},0,0,0,0)$ if $h=\infty$.
\item For each $t=1,2,\ldots,K$, do the following.

QC: $(d^2\log^2{d},0,0,0,0)$.
\begin{denumerate}[label=3.2.1.2.\arabic*]
\item Determine the unique number $\overline{n}_{j_t}\in\Ncal$ such that $j_t\in S_{\overline{n}_{j_t}}$.

QC: $(d\log^2{d},0,0,0,0)$.
\end{denumerate}
\item If $i<d$, then do the following.
\begin{denumerate}[label=3.2.1.3.\arabic*]
\item Set
\begin{align*}
\Dfrak:=\{&(n,m): n\in\{\overline{n}_{j_t}: 1\leq t\leq K\}, \\
&m=|\{t\in\{1,\ldots,K\}: \overline{n}_{j_t}=n\}|>0\}.
\end{align*}
QC: $(d\log^2{d},0,0,0,0)$.
\end{denumerate}
\item Else do the following.
\begin{denumerate}[label=3.2.1.4.\arabic*]
\item Set
\begin{align*}
\Dfrak:=\{&(n,sm): n\in\{\overline{n}_{j_t}: 1\leq t\leq K\}, \\
&m=|\{t\in\{1,\ldots,K\}: \overline{n}_{j_t}=n\}|>0\}
\end{align*}
QC: $(d\log^2{d}+d\log^{1+o(1)}{q},0,0,0,0)$.
\end{denumerate}
\item Check whether there is a (unique) $n\in\Ncal$ such that $\Dfrak=\Dfrak_n$, and store this information (the truth value and $n$).

QC: $(d^2\log{d},0,0,0,0)$.
\item If $\Dfrak=\Dfrak_n$ for some $n\in\Ncal$, then do the following.
\begin{denumerate}[label=3.2.1.6.\arabic*]
\item Merge the sorted lists $\{i\}$ and $S_n$.

QC: $(d\log{d},0,0,0,0)$.
\end{denumerate}
\item Else do the following.
\begin{denumerate}
\item Set $n':=\max{\Ncal}+1$, and add $n'$ to $\Ncal$ as a new element.

QC: $(\log{d},0,0,0,0)$.
\item Set $\Dfrak_{n'}:=\Dfrak$, and initialize $S_{n'}:=\{i\}$.

QC: $(d\log{d},0,0,0,0)$ if $i<d$; $(d\log{q},0,0,0,0)$ if $i=d$.
\end{denumerate}
\end{denumerate}
\end{denumerate}
\end{denumerate}
\item Output $\Rfrak:=((\Dfrak_n,S_n))_{n=0,1,\ldots,\max{\Ncal}}$ and halt.

QC: $(d^2\log{d}+d\log{q},0,0,0,0)$.
\end{denumerate}

Next, we give the (simple) pseudocode for the algorithm from Lemma \ref{treeRegLem}(3), which checks whether a given index $d$ generalized cyclotomic mapping $f$ of $\IF_q$ is of special type II.

\begin{denumerate}[label=\arabic*]
\item Compute the induced function $\overline{f}:\{0,1,\ldots,d\}\rightarrow\{0,1,\ldots,d\}$.

QC: $(d\log^{1+o(1)}{q},d,0,0,0)$.
\item Sort $\im(\overline{f})$, and check whether it is equal to $\{0,1,\ldots,d\}$.

QC: $(d\log^2{d},0,0,0,0)$.
\end{denumerate}

The algorithm from Lemma \ref{treeRegLem}(4), for computing a type-II tree register for a given index $d$ generalized cyclotomic mapping $f$ of $\IF_q$ that is of special type II, has the following pseudocode.

\begin{denumerate}[label=\arabic*]
\item Compute the induced function $\overline{f}$ and the affine maps $A_i:x\mapsto\alpha_ix+\beta_i$.

QC: $(d\log^{1+o(1)}{q},d,0,0,0)$.
\item Compute a CRL-list $\overline{\Lcal}$ of $\overline{f}$ and, in the process, store the cycles of $\overline{f}$.

QC: $(d^2\log^2{d},0,0,0,0)$.
\item Set $\Ncal:=\emptyset$.

QC: $(1,0,0,0,0)$.
\item For each $(i,\ell)\in\overline{\Lcal}$, do the following.

QC: $(d^5\log^2{d}\mpe(q-1)^3+d^5\mpe(q-1)^3\log{q}+d^2\mpe(q-1)\log^{1+o(1)}{q},0,0,0,0)$.
\begin{denumerate}[label=4.\arabic*]
\item For each $t=0,1,\ldots,\ell-1$, do the following.

QC: $(\ell^2\mpe(q-1)\log^{1+o(1)}{q},0,0,0,0)$.
\begin{denumerate}[label=4.1.\arabic*]
\item Set $\product_{\up}:=\alpha_{i_{t-1}}$.

QC: $(\log{q},0,0,0,0)$.
\item Set $\product_{\low}:=1$.

QC: $(1,0,0,0,0)$.
\item For each $k=1,2,\ldots$ (no upper bound a priori), do the following.

QC: $(\ell\mpe(q-1)\log^{1+o(1)}{q},0,0,0,0)$.
\begin{denumerate}[label=4.1.3.\arabic*]
\item Compute
\[
\proc_{i_t,k}:=\frac{\gcd(\product_{\up},s)}{\gcd(\product_{\low},s)}.
\]
QC: $(\log^{1+o(1)}{q},0,0,0,0)$.
\item If $\proc_{i_t,k}=1$, then do the following.
\begin{denumerate}[label=4.1.3.2.\arabic*]
\item Set $\Hcal_{i_t}:=k-1$.

QC: $(\log{d}+\log\log{q},0,0,0,0)$.
\item Exit the loop for $k$, and skip to the next $t$.

QC: $(1,0,0,0,0)$.
\end{denumerate}
\item Else do the following.
\begin{denumerate}
\item Set $\product_{\low}:=\product_{\up}$.

QC: $(\log{q},0,0,0,0)$.
\item Set $\product_{\up}:=\product_{\up}\cdot\alpha_{i_{t-k-1}}$.

QC: $(\log^{1+o(1)}{q},0,0,0,0)$.
\end{denumerate}
\end{denumerate}
\end{denumerate}
\item Compute $H_i:=\max\{\Hcal_{i_t}: t=0,1,\ldots,\ell-1\}$.

QC: $(\ell(\log{d}+\log\log{q}),0,0,0,0)$.
\item For $h=0,1,\ldots,H_i$, do the following.

QC: $(\ell^2d^3\log^2{d}\mpe(q-1)^3+\ell d^4\mpe(q-1)^3\log{q},0,0,0,0)$.
\begin{denumerate}[label=4.3.\arabic*]
\item If $h=0$, then do the following.
\begin{denumerate}[label=4.3.1.\arabic*]
\item If $\Ncal=\emptyset$, then do the following.
\begin{denumerate}[label=4.3.1.1.\arabic*]
\item Set $\Dfrak_0:=\emptyset$ and $\height_0:=0$.

QC: $(1,0,0,0,0)$.
\item If $H_i>0$, then set $S_{0,\trans}:=\{i_0,i_1,\ldots,i_{\ell-1}\}$, sorted. Otherwise, set $S_{0,\trans}:=\emptyset$.

QC: $(\ell\log^2{d},0,0,0,0)$.
\item Set $S_{0,\per}:=\{i_t: \Hcal_{i_t}=0\}$, sorted.

QC: $(\ell\log^2{d},0,0,0,0)$.
\end{denumerate}
\item Else do the following.
\begin{denumerate}[label=4.3.1.2.\arabic*]
\item If $H_i>0$, then sort $\{i_0,i_1,\ldots,i_{\ell-1}\}$ and merge it with $S_{0,\trans}$.

QC: $(\ell\log^2{d}+d\log{d},0,0,0,0)$.
\item Create a list of all indices $i_t$, for $t=0,1,\ldots,\ell-1$, such that $\Hcal_{i_t}=0$, then sort it and merge it with $S_{0,\per}$.

QC: $(\ell\log^2{d}+d\log{d},0,0,0,0)$.
\end{denumerate}
\end{denumerate}
\item Else do the following.
\begin{denumerate}[label=4.3.2.\arabic*]
\item For $t=0,1,\ldots,\ell-1$, do the following.

QC: $(\ell^2d^2\log^2{d}\mpe(q-1)^2+\ell d^3\mpe(q-1)^2\log{q},0,0,0,0)$.
\begin{denumerate}[label=4.3.2.1.\arabic*]
\item If $h<H_i$, then set $h':=h$; otherwise, set $h':=\Hcal_{i_t}$.

QC: $(\log{d}+\log\log{q},0,0,0,0)$.
\item For $k=0,1,\ldots,h'-1$, set
\[
\wfrak_k:=
\begin{cases}
\proc_{i_t,k+1}-\proc_{i_t,k+2}, & \text{if }h=H_i,\text{ or }h<H_i\text{ and }k<h'-1, \\
\proc_{i_t,h}, & \text{if }h<H_i\text{ and }k=h'-1.
\end{cases}
\]
QC: $(\ell\mpe(q-1)\log{q},0,0,0,0)$.
\item For $k=0,1,\ldots,h'-1$, find $\overline{n}_{i_t,k}$, the unique $n\in\Ncal$ such that $i_t\in S_{n,\trans}$ and $\height_n=k$.

QC: $(\ell d^2\log^2{d}\mpe(q-1)^2+\ell\mpe(q-1)\log\log{q},0,0,0,0)$.
\item Set $\Dfrak:=\{(\overline{n}_{i_t,k},\wfrak_k): k=0,1,\ldots,h'-1\}$, sorted lexicographically.

QC: $(\ell\mpe(q-1)(\log{\ell}+\log{\mpe(q-1)})\log{q},0,0,0,0)$.
\item Check whether $\Dfrak=\Dfrak_n$ some (unique) $n\in\Ncal$, and if so, store this information (the truth value and $n$).

QC: $(\ell d^2\mpe(q-1)^2\log{q},0,0,0,0)$.
\item If $\Dfrak=\Dfrak_n$ for some $n\in\Ncal$, then do the following.
\begin{denumerate}[label=4.3.2.2.6.\arabic*]
\item If $h<H_i$, then merge the sorted lists $\{i_t\}$ and $S_{n,\trans}$. Otherwise, merge the sorted lists $\{i_t\}$ and $S_{n,\per}$.

QC: $(d\log{d},0,0,0,0)$.
\end{denumerate}
\item Else do the following.
\begin{denumerate}
\item Set $n':=\max{\Ncal}+1$, $\Dfrak_{n'}:=\Dfrak$, $\height_{n'}:=h'$, and add $n'$ to $\Ncal$ as a new element.

QC: $(d\mpe(q-1)\log{q},0,0,0,0)$.
\item Initialize $S_{n',\trans}:=\emptyset$ and $S_{n',\per}:=\emptyset$.

QC: $(\log{d}+\log\log{q},0,0,0,0)$
\item If $h<H_i$, then add $i_t$ to $S_{n',\trans}$ as a new element. Otherwise, add $i_t$ to $S_{n',\per}$ as a new element.

QC: $(\log{d}+\log\log{q},0,0,0,0)$.
\end{denumerate}
\end{denumerate}
\end{denumerate}
\end{denumerate}
\end{denumerate}
\item Compute $\Hfrak:=\max\{H_i: (i,\ell)\in\overline{\Lcal}\}$.

QC: $(d(\log{d}+\log\log{q}),0,0,0,0)$.
\item For $n\in\Ncal$, do the following.

QC: $(d^3\log{d}\mpe(q-1)+d^2\mpe(q-1)\log\log{q},0,0,0,0)$.
\begin{denumerate}[label=6.\arabic*]
\item Set $S_n:=(\height_n,S_{n,\trans},S_{n,\per})$.

QC: $(d\log{d}+\log\log{q},0,0,0,0)$.
\end{denumerate}
\item Output $\Rfrak:=((\Dfrak_n,S_n))_{n\in\Ncal}$ and halt.

QC: $(d^4\mpe(q-1)^2\log{q},0,0,0,0)$.
\end{denumerate}

Next, we give pseudocode for the algorithm from Proposition \ref{ctTauProp}, which computes the cycle type $\CT(A_{\mid\per(A)})$ for a given affine map $A:x\mapsto ax+b$ of $\IZ/m\IZ$.

\begin{denumerate}[label=\arabic*]
\item Factor $m$, and compute $\ord_{p^{\nu_p(m)}}(a)$ for all primes $p\mid m$ with $p\nmid a$.

QC: $(\log{m},0,1,0,0)$.
\item For each prime $p\mid m$ such that $p\nmid a$, do the following.

QC: $(\log^{2+o(1)}{m},0,\log{m},0,0)$.
\begin{denumerate}[label=2.\arabic*]
\item Compute $A_{(p)}:=A\bmod{p^{\nu_p(m)}}$.

QC: $(\log^{1+o(1)}{m},0,0,0,0)$.
\item Using \cite[Tables 3 and 4]{BW22b}, compute
\[
\CT(A_{(p)})=x_{\overline{l}_{p,1}}^{e_{p,1}}x_{\overline{l}_{p,2}}^{e_{p,2}}\cdots x_{\overline{l}_{p,K_p}}^{e_{p,K_p}}\text{ (all }e_{p,j}>0\text{)}
\]
with all cycle lengths $\overline{l}_{p,j}$ fully factored. This involves factoring $\ord_{p^{\nu_p(m)}}(a)$, a single power computation, and $O(\nu_p(m))$ instances of simpler arithmetic.

QC: $(\log^{2+o(1)}{p^{\nu_p(m)}}+\nu_p(m)\log^{1+o(1)}{p^{\nu_p(m)}},0,1,0,0)$.
\end{denumerate}
\item For each $\vec{\jmath}=(j_p)_{p\mid m,p\nmid a}\in\prod_{p\mid m,p\nmid a}{\{1,2,\ldots,K_p\}}$, do the following.

QC: $(\tau(m)\log^{2+o(1)}{m},0,0,0,0)$.
\begin{denumerate}[label=3.\arabic*]
\item Compute the Wei-Xu product of variable powers
\[
\WX(\vec{\jmath}):=\divideontimes_{p\mid m,p\nmid a}{x_{\overline{l}_{p,j_p}}^{e_{p,j_p}}}=x_{\lcm(\overline{l}_{p,j_p}: p\mid m, p\nmid a)}^{\prod_{p\mid m,p\nmid a}{(e_{p,j_p}\overline{l}_{p,j_p})}/\lcm(\overline{l}_{p,j_p}: p\mid m, p\nmid a)}.
\]
QC: $(\log^{2+o(1)}{m},0,0,0,0)$.
\end{denumerate}
\item Compute and output
\[
\CT(A)=\divideontimes_{p\mid m,p\nmid a}{\CT(A_{(p)})}=\prod_{\vec{\jmath}}{\WX(\vec{\jmath})},
\]
then halt.

QC: $(\tau(m)^2\log{m},0,0,0,0)$.
\end{denumerate}

The following is pseudocode for the algorithm from Lemma \ref{periodLem}, serving to compute $\minperl(\vec{\xfrak})$ for given $\vec{\xfrak}\in\{0,1,\ldots,N-1\}^n$, where each $n\in\{0,1,\ldots,N-1\}$ is given with bit length $l_{\bit}$.

\begin{denumerate}[label=\arabic*]
\item Compute the binary representation of $n$

QC: $(n\log{n},0,0,0,0)$.
\item Factor $n=p_1^{v_1}\cdots p_K^{v_K}$ deterministically.

QC: $(n^{1/5+o(1)},0,0,0,0)$.
\item For each $j=1,2,\ldots,K$, do the following.

QC: $(n\log{n}\log\log{n}(\log{n}+l_{\bit}))$.
\begin{denumerate}[label=3.\arabic*]
\item Using binary search, find $v'_j=\nu_{p_j}(\minperl(\vec{\xfrak}))$ as the smallest $v\in\{0,1,\ldots,v_j\}$ such that $p_j^v\prod_{k\not=j}{p_k^{v_k}}$ is a period length of $\vec{\xfrak}$.

QC: $(n\log\log{n}(\log{n}+l_{\bit}))$.
\end{denumerate}
\item Compute and output
\[
\minperl(\vec{\xfrak})=\prod_{j=1}^K{p_j^{v'_j}},
\]
then halt.

QC: $(\log^{3+o(1)}{n},0,0,0,0)$.
\end{denumerate}

Next, we give pseudocode for Theorem \ref{necklaceListTheo}, which is concerned with computing not only a tree register, but also an associated tree necklace list for a given index $d$ generalized cyclotomic mapping $f$ of $\IF_q$ that is of special type I or II. Because the procedures for the two cases are analogous, we just give one algorithm that deals with both simultaneously.

\begin{denumerate}[label=\arabic*]
\item Check whether $f$ is of special type I and store this information. In the process, also compute and store $\overline{f}$ and the affine maps $A_i$ for later use.

QC: $(d\log^{1+o(1)}{q},d,0,0,0)$.
\item If $f$ is \emph{not} of special type I, then check whether $f$ is of special type II and store this information.

QC: $(d\log^2{d},0,0,0,0)$ because $\overline{f}$ has already been computed.
\item If $f$ is neither of special type I nor II, then output \enquote{fail} and halt.

QC: $(1,0,0,0,0)$.
\item If $f$ is of special type I, then do the following.
\begin{denumerate}[label=4.\arabic*]
\item Compute a type-I tree register $\Rfrak=((\Dfrak_n,S_n))_{n=0,1,\ldots,N}$ for $f$ with $N\in O(d)$.

QC: $(d^3\log^2{d}+d\log^{1+o(1)}{q},0,0,0,0)$ because $\overline{f}$ and the $A_i$ have already been computed.
\end{denumerate}
\item Else do the following.
\begin{denumerate}[label=5.\arabic*]
\item Compute a type-II tree register $\Rfrak=((\Dfrak_n,S_n))_{n=0,1,\ldots,N}$ for $f$ with $N\in O(d^2\mpe(q-1))$, where $S_n=(\height_n,S_{n,\trans},S_{n,\per})$.

QC: $(d^2\mpe(q-1)\log^{1+o(1)}{q}+d^5\mpe(q-1)^3\log{q}+d^5\log^2{d}\mpe(q-1)^3,0,0,0,0)$ because $\overline{f}$ and the $A_i$ have already been computed.
\end{denumerate}
\item Compute a CRL-list $\overline{\Lcal}$ of $\overline{f}$ and, in the process, store the cycles of $\overline{f}$.

QC: $(d^2\log^2{d},0,0,0,0)$.
\item For each $(i,\ell)\in\overline{\Lcal}$, with associated $\overline{f}$-cycle $(i_0,i_1,\ldots,i_{\ell-1})$, do the following.

QC: $(d\tau(q-1)\log^{2+o(1)}{q}+d^2\log^{1+o(1)}{q}+d\tau(q-1)^2\log{q}+Nd^2\log{d}+d^3\log{d}l_{\bit},\linebreak[4]0,d\log{q},d,0)$.
\begin{denumerate}[label=7.\arabic*]
\item If $i=d$, then do the following.
\begin{denumerate}[label=7.1.\arabic*]
\item If $f$ is of special type I, then do the following.
\begin{denumerate}[label=7.1.1.\arabic*]
\item Set $\nfrak$ to be the unique $n\in\{0,1,\ldots,N\}$ such that $d\in S_n$.

QC: $(d\log{d},0,0,0,0)$.
\end{denumerate}
\item Else do the following.
\begin{denumerate}[label=7.1.2.\arabic*]
\item Set $\nfrak:=0$.

QC: $(1,0,0,0,0)$.
\end{denumerate}
\item Set $\Nfrak_d:=\{([\nfrak],1,1)\}$, $\vec{\nfrak}'_d:=[\nfrak]$ and $\Nfrak'_d:=\{(1,1)\}$, then skip to the next pair $(i,\ell)$.

QC: $(\log{d},0,0,0,0)$.
\end{denumerate}
\item Else do the following.
\begin{denumerate}[label=7.2.\arabic*]
\item Compute $\Acal_i:=A_{i_0}A_{i_1}\cdots A_{i_{\ell-1}}$.

QC: $(d\log^{1+o(1)}{q},0,0,0,0)$.
\item Compute $\CT((\Acal_i)_{\mid\per(\Acal_i)})=x_1^{e_{i,1}}x_2^{e_{i,2}}\cdots x_s^{e_{i,s}}$.

QC: $(\tau(q-1)\log^{2+o(1)}{q}+\tau(q-1)^2\log{q},0,\log{q},1,0)$.
\item If $f$ is of special type I, then do the following.
\begin{denumerate}[label=7.2.3.\arabic*]
\item For each $n=0,1,\ldots,N$, do the following.

QC: $(d^2\log{d},0,0,0,0)$.
\begin{denumerate}[label=7.2.3.1.\arabic*]
\item For each $t=0,1,\ldots,\ell-1$, do the following.

QC: $(|S_n|d\log{d},0,0,0,0)$.
\begin{denumerate}[label=7.2.3.1.1.\arabic*]
\item If $i_t\in S_n$, then set $\nfrak_{i_t}:=n$.

QC: $(|S_n|\log{d},0,0,0,0)$.
\end{denumerate}
\end{denumerate}
\end{denumerate}
\item Else do the following.
\begin{denumerate}[label=7.2.4.\arabic*]
\item For each $n=0,1,\ldots,N$, do the following.

QC: $(d^3\log{d}\mpe(q-1)+dl_{\bit},0,0,0,0)$.
\begin{denumerate}[label=7.2.4.1.\arabic*]
\item For each $t=0,1,\ldots,\ell-1$, do the following.

QC: $(\max\{1,|S_{n,\per}|\}d\log{d}+|S_{n,\per}\cap\{i_0,\ldots,i_{\ell-1}\}|l_{\bit},0,0,0,0)$.
\begin{denumerate}[label=7.2.4.1.1.\arabic*]
\item If $i_t\in S_{n,\per}$, then set $\nfrak_{i_t}:=n$.

QC: $(\max\{1,|S_{n,\per}|\}\log{d}+l_{\bit},0,0,0,0)$.
\end{denumerate}
\end{denumerate}
\end{denumerate}
\item Set $\vec{\nfrak}_i:=(\nfrak_{i_0},\nfrak_{i_1},\ldots,\nfrak_{i_{\ell-1}})$.

QC: $(dl_{\bit},0,0,0,0)$.
\item Overwrite $\vec{\nfrak}_i$ with the lexicographically smallest sequence in the same cyclic equivalence class.

QC: $(d^2\log{d}l_{\bit},0,0,0,0)$.
\item Compute $\minperl(\vec{\nfrak}_i)$.

QC: $(d\log{d}\log\log{d}(\log{d}+l_{\bit}),0,0,0,0)$.
\item Set
\begin{itemize}
\item $\vec{\nfrak}'_i:=[\nfrak_{i_0},\nfrak_{i_1},\ldots,\nfrak_{i_{\minperl(\vec{\nfrak}_i)-1}}]$,
\item $\Nfrak'_i:=\{(\ell\cdot l',e_{i,l'}): l'\in\{1,2,\ldots,s\}, e_{i,l'}>0\}$, and
\item $\Nfrak_i:=\{\vec{\nfrak}'_i\}\times\Nfrak'_i$.
\end{itemize}
QC: $(dl_{\bit}+\tau(q-1)\log^{1+o(1)}{q},0,0,0,0)$.
\end{denumerate}
\end{denumerate}
\item Set $\Mfrak:=\Nfrak:=\emptyset$.

QC: $(1,0,0,0,0)$.
\item Set $\vec{\nfrak}':=\{(\vec{\nfrak}'_i,i,\ell): (i,\ell)\in\overline{\Lcal}\}$, and sort it lexicographically.

QC: $(d^2\log{d}l_{\bit},0,0,0,0)$.
\item For each $(\vec{\nfrak}'_i,i,\ell)\in\vec{\nfrak}'$, do the following.

QC: $(d^3\tau(q-1)^2\log{q},0,0,0,0)$.
\begin{denumerate}[label=10.\arabic*]
\item Check whether $\vec{\nfrak}'_i\in\Mfrak$, and if so, skip to the next triple $(\vec{\nfrak}'_i,i,\ell)$.

QC: $(d^2l_{\bit},0,0,0,0)$.
\item Add $\vec{\nfrak}'_i$ to $\Mfrak$ as a new element.

QC: $(dl_{\bit},0,0,0,0)$.
\item Set $\Vfrak:=\Nfrak'_i$.

QC: $(\tau(q-1)\log{q},0,0,0,0)$.
\item For each $(\vec{\nfrak}'_j,j,\ell')\in\vec{\nfrak}'$ that comes after $(\vec{\nfrak}'_i,i,\ell)$, do the following.

QC: $(d^2l_{\bit}+d^2\tau(q-1)^2\log{q},0,0,0,0)$.
\begin{denumerate}[label=10.4.\arabic*]
\item Check whether $\vec{\nfrak}'_i=\vec{\nfrak}'_j$, and if not, skip to the next triple $(\vec{\nfrak}'_j,j,\ell')$.

QC: $(dl_{\bit},0,0,0,0)$.
\item For each $(l,k)\in\Nfrak'_j$, do the following.

QC: $(d\tau(q-1)^2\log{q},0,0,0,0)$.
\begin{denumerate}[label=10.4.2.\arabic*]
\item Check whether $l$ occurs as the first entry of some pair $(l,k')\in\Vfrak$, and store this information.

QC: $(d\tau(q-1)\log{q},0,0,0,0)$.
\item If $(l,k')\in\Vfrak$ for some $k'$, then do the following.
\begin{denumerate}[label=10.4.2.2.\arabic*]
\item Replace the unique element of $\Vfrak$ of the form $(l,k')$ by $(l,k+k')$.

QC: $(\log{q},0,0,0,0)$.
\end{denumerate}
\item Else do the following.
\begin{denumerate}[label=10.4.2.3.\arabic*]
\item Add $(l,k)$ to $\Vfrak$ as a new element.

QC: $(\log{q},0,0,0,0)$.
\end{denumerate}
\end{denumerate}
\end{denumerate}
\item For each $(l,k)\in\Vfrak$, do the following.

QC: $(d\tau(q-1)(dl_{\bit}+\log{q}),0,0,0,0)$.
\begin{denumerate}[label=10.5.\arabic*]
\item Add $([\vec{\nfrak}'_i],l,k)$ to $\Nfrak$ as a new element.

QC: $(dl_{\bit}+\log{q},0,0,0,0)$.
\end{denumerate}
\end{denumerate}
\item Output $\Rfrak$ and $\Nfrak$, and halt.

QC: $(N^2\log{q}+d\tau(q-1)(dl_{\bit}+\log{q}),0,0,0,0)$.
\end{denumerate}

The following is pseudocode for the algorithm from Lemma \ref{synchronizeLem}. For given recursive tree description lists $\vec{\Dfrak}=(\Dfrak_n)_{n=0,1,\ldots,N}$ and $\vec{\Dfrak}'=(\Dfrak'_n)_{n=0,1,\ldots,N'}$ satisfying the assumptions of Lemma \ref{synchronizeLem}, this algorithm computes a synchronization $(\vec{\Dfrak}^+,\ifrak)$ of $\vec{\Dfrak}$ and $\vec{\Dfrak}'$.

\begin{denumerate}[label=\arabic*]
\item Set $\Nfrak:=\{0,1,\ldots,N\}$, and for $k\in\Nfrak$, set $\Dfrak^+_k:=\Dfrak_k$. Moreover, let $\ifrak$ be the empty function $\emptyset$.

QC: $(N^2(l_{\bit}+m),0,0,0,0)$.
\item For each $n=0,1,\ldots,N'$, do the following.

QC: $((NN'\min\{N,N'\}+(N')^2l'_{\bit})(l_{\bit}+l'_{\bit}+m),0,0,0,0)$.
\begin{denumerate}[label=2.\arabic*]
\item If $n=0$, then do the following.
\begin{denumerate}[label=2.1.\arabic*]
\item Set $\ifrak(0):=0$.

QC: $(1,0,0,0,0)$.
\end{denumerate}
\item Else do the following.
\begin{denumerate}[label=2.2.\arabic*]
\item Let $\Dfrak$ be the set of pairs obtained from $\Dfrak'_n$ through replacing each first entry $k$ of each pair in $\Dfrak'_n$ by $\ifrak(k)$ (one may simply overwrite the corresponding entries of $\Dfrak'_n$, so one does not need to handle the second entries of bit length in $O(m)$).

QC: $(N'(l_{\bit}+l'_{\bit}),0,0,0,0)$.
\item Sort $\Dfrak$.

QC: $(N'l'_{\bit}(l_{\bit}+l'_{\bit}+m))$.
\item Check whether $\Dfrak=\Dfrak^+_k$ for some (unique) $k\in\{0,1,\ldots,N\}$, and store this information (the truth value and $k$).

QC: $(N\min\{N,N'\}(l_{\bit}+l'_{\bit}+m),0,0,0,0)$.
\item If $\Dfrak=\Dfrak^+_k$ for some $k\in\{0,1,\ldots,N\}$, then do the following.
\begin{denumerate}[label=2.2.4.\arabic*]
\item Set $\ifrak(n):=k$.

QC: $(l_{\bit}+l'_{\bit},0,0,0,0)$.
\end{denumerate}
\item Else do the following.
\begin{denumerate}[label=2.2.5.\arabic*]
\item Set $n':=\max{\Ncal}+1$, add $n'$ to $\Ncal$ as a new element, set $\Dfrak^+_{n'}:=\Dfrak$ and $\ifrak(n):=n'$.

QC: $(N'(l_{\bit}+l'_{\bit}+m),0,0,0,0)$.
\end{denumerate}
\end{denumerate}
\end{denumerate}
\item Set $\vec{\Dfrak}^+:=(\Dfrak^+_n)_{n\in\Ncal}$, output $(\vec{\Dfrak}^+,\ifrak)$ and halt.

QC: $((N+N')\max\{N,N'\}(l_{\bit}+l'_{\bit}+m),0,0,0,0)$.
\end{denumerate}

Finally, we provide pseudocode for the algorithm from Corollary \ref{necklaceListCor}. For given generalized cyclotomic mappings $f_1$ and $f_2$ of $\IF_q$, of index $d_1$ and $d_2$, respectively, such that each $f_j$ is of special type I or II (not necessarily both of the same special type), this algorithm decides whether $\Gamma_{f_1}\cong\Gamma_{f_2}$. Throughout this discussion, we have $d:=\max\{d_1,d_2\}$.

\begin{denumerate}[label=\arabic*]
\item For $j=1,2$, do the following.

QC:
\begin{itemize}
\item $(d\log^{1+o(1)}{q},d,0,0,0)$ if $f_1$ and $f_2$ are both of special type I;
\item $(d^2\log{d}+d\log^{1+o(1)}{q},d,0,0,0)$ otherwise.
\end{itemize}
\begin{denumerate}[label=1.\arabic*]
\item Check whether $f_j$ is of special type I, and store this information as well as the induced function $\overline{f_j}$ and the affine maps on $\IZ/((q-1)/d_j)\IZ$ associated with $f_j$.

QC: $(d\log^{1+o(1)}{q},d,0,0,0)$.
\item If $f_j$ is \emph{not} of special type I, then check whether $f_j$ is of special type II, and store this information.

QC: $(d\log^2{d},0,0,0,0)$, because $\overline{f_j}$ and the affine maps have already been computed.
\item If $f_j$ is neither of special type I nor II, then output \enquote{fail} and halt.

QC: $(1,0,0,0,0)$.
\end{denumerate}
\item For $j=1,2$, do the following.

QC:
\begin{itemize}
\item $(d^3\log^2{d}+d^3\log{d}\tau(q-1)+d^3\tau(q-1)^2\log{q}+d^2\log^{1+o(1)}{q}+d\tau(q-1)\log^{2+o(1)}{q},d,d\log{q},d,0)$ if $f_1$ and $f_2$ are both of special type I;
\item $(d^5\log^2{d}\mpe(q-1)+(d^5\mpe(q-1)^3+d^3\tau(q-1)^2)\log{q}+d^2\mpe(q-1)\log^{1+o(1)}{q}+d\tau(q-1)\log^{2+o(1)}{q},d,d\log{q},d,0)$ otherwise.
\end{itemize}
\begin{denumerate}[label=2.\arabic*]
\item If $f_j$ is of special type I, then do the following.
\begin{denumerate}[label=2.1.\arabic*]
\item Compute a type-I tree register $\Rfrak_j$ of $f_j$, and the tree necklace list $\Nfrak_j$ for $f_j$ relative to $\Rfrak_j$.

QC: $(d^3\log^2{d}+d^3\log{d}\tau(q-1)+d^3\tau(q-1)^2\log{q}+d^2\log^{1+o(1)}{q}+d\tau(q-1)\log^{2+o(1)}{q},d,d\log{q},d,0)$.
\end{denumerate}
\item Else do the following.
\begin{denumerate}[label=2.2.\arabic*]
\item Compute a type-II tree register $\Rfrak_j$ of $f_j$, and the tree necklace list $\Nfrak_j$ for $f_j$ relative to $\Rfrak_j$.

QC: $(d^5\log^2{d}\mpe(q-1)+(d^5\mpe(q-1)^3+d^3\tau(q-1)^2)\log{q}+d^2\mpe(q-1)\log^{1+o(1)}{q}+d\tau(q-1)\log^{2+o(1)}{q},d,d\log{q},d,0)$.
\end{denumerate}
\end{denumerate}
\item For $j=1,2$, let $\vec{\Dfrak}^{(j)}=(\Dfrak^{(j)}_n)_{n=0,1,\ldots,N_j}$ be the underlying recursive tree description list of $\Rfrak_j$. Compute a synchronization $(\vec{\Dfrak}^+,\ifrak)$ of $\vec{\Dfrak}^{(1)}$ and $\vec{\Dfrak}^{(2)}$.

QC:
\begin{itemize}
\item $(d^3\log{q},0,0,0,0)$ if $f_1$ and $f_2$ are both of special type I;
\item $(d^6\mpe(q-1)^3\log{q}+d^4\mpe(q-1)^2\log^{1+o(1)}{q},0,0,0,0)$ otherwise.
\end{itemize}
\item Create a modified version $\Nfrak'_2$ of $\Nfrak_2$ by replacing each entry $\nfrak$ of each first entry $[\vec{\nfrak}]$ of an element $([\vec{\nfrak}],l,\mfrak)\in\Nfrak_2$ by $\ifrak(\nfrak)$, then overwriting each of the resulting $O(d)$ distinct first entries of triples in the list with the lexicographically minimal number sequence in the same cyclic equivalence class, and finally sorting $\Nfrak'_2$ lexicographically by merging the $O(d)$ distinct segments corresponding to the same first entry of triples in $\Nfrak'_2$.

QC: $(d^3\tau(q-1)\log{q},0,0,0,0)$.
\item Check whether $\Nfrak'_2=\Nfrak_1$, output the corresponding truth value, and halt.

QC: $(d^2\tau(q-1)\log{q},0,0,0,0)$.
\end{denumerate}

\subsubsection{Short-term block behavior and the special case where all cycles are short}\label{subsubsec5P3P3}

Let $f$ be an index $d$ generalized cyclotomic mapping of $\IF_q$. For an $f$-periodic $x\in\IF_q$ and $t\in\IZ$, we set $x^{(t)}:=(f_{\mid\per(f)})^t(x)$\phantomsection\label{not291}, and we recall the notation $i_t:=(\overline{f}_{\mid\per(\overline{f})})^t(i)$ for $\overline{f}$-periodic $i\in\{0,1,\ldots,d-1\}$ and $t\in\IZ$, as well as $i':=i_{-1}$.

In Subsection \ref{subsec3P3}, for each $i\in\{0,1,\ldots,d-1\}$, we constructed an arithmetic partition $\Pcal_i$ of $C_i$ such that for $x\in C_i$, the isomorphism type of $\Tree_{\Gamma_f}(x)$ only depends on the $\Pcal_i$-block in which $x$ is contained. Here, we refine this construction. We recall that $\Pcal_i=\Qcal_{i,H_i}$ where $H_i$ is the maximum tree height in $\Gamma_{\per}$ (the induced subgraph of $\Gamma_f$ on $\bigcup_{j\in\per(\overline{f})}{C_j}$) above an $f$-periodic point in $\bigcup_{t\in\IZ}{C_{i_t}}$. The arithmetic partition $\Qcal_{i,h}$ of $C_i$ is defined for all $h\in\IN_0$ (even though we only considered it for $h\in\{0,1,\ldots,H_i\}$ in Subsection \ref{subsec3P3}) and satisfies
\begin{equation}\label{crucialEq}
\Qcal_{i,h+1}=\Rcal_i\wedge\Pfrak'(\Qcal_{i',h},A_{i'}).
\end{equation}
This formula is key to our construction. Indeed, the refined arithmetic partition of $C_i$ which we consider here is simply $\Qcal_{i,H_i+L-1}$ for some $L\in\IN^+$, as opposed to $\Pcal_i=\Qcal_{i,H_i}$. While the blocks of $\Pcal_i$ control the isomorphism types of rooted trees in $\Gamma_f$ above (periodic) vertices in $C_i$, the following more general statement holds for $\Qcal_{i,H_i+L-1}$.

\begin{lemmmma}\label{refinedLem}
Let $x\in C_i$ be $f$-periodic, and let $L\in\IN_0$. The $\Qcal_{i,H_i+L-1}$-block in which $x$ is contained uniquely determines the (length $L$) sequence
\[
(\Tree_{\Gamma_f}(x^{(t)}))_{t=0,-1,\ldots,-L+1}
\]
of rooted tree isomorphism types.
\end{lemmmma}

\begin{proof}
We can write the $\Qcal_{i,H_i+L-1}$-block of $x$ as $\Bcal(\Qcal_{i,H_i+L-1},\diamond_{t=0}^{H_i+L-1}{\vec{o'_t}}\diamond\vec{\xi}_{i,H_i})$ where $\vec{o'_t}\in\{\emptyset,\neg\}^{n_{i_{-t}}}$. Noting that
\[
\Bcal(\Qcal_{i,H_i+L-1},\diamond_{t=0}^{H_i+L-1}{\vec{o'_t}}\diamond\vec{\xi}_{i,H_i})\subseteq\Bcal(\Pcal_i,\diamond_{t=0}^{H_i}{\vec{o'_t}}\diamond\vec{\xi}_{i,H_i}),
\]
we see that $\Tree_{\Gamma_f}(x)=\Tree_i(\Pcal_i,\diamond_{t=0}^{H_i}{\vec{o'_t}}\diamond\vec{\xi}_{i,H_i})$ is uniquely determined.

By formula (\ref{crucialEq}), we know that for each block $\Bcal(\Qcal_{i',H_i+L-2},\diamond_{t=0}^{H_i+L-2}{\vec{o_t}}\diamond\vec{\xi})$ of $\Qcal_{i',H_i+L-2}$, the number of $f$-pre-images of $x$ in that block is the constant
\begin{equation}\label{qcalSigmaEq}
\sigma_{\Qcal_{i',H_i+L-2},A_{i'}}(\diamond_{t=0}^{H_i+L-2}{\vec{o_t}}\diamond\vec{\xi},\diamond_{t=1}^{H_i+L-1}{\vec{o'_t}}\diamond\vec{\xi}_{i,H_i}).
\end{equation}
Now we assume that $\vec{\xi}=\vec{\xi}_{i',H_i}$ (which is actually the same as $\vec{\xi}_{i,H_i}$). The union of all blocks of $\Qcal_{i',H_i+L-2}$ of the form $\Bcal(\Qcal_{i',H_i+L-2},\diamond_{t=0}^{H_i+L-2}{\vec{o_t}}\diamond\vec{\xi}_{i',H_i})$ (where $\vec{o_t}$ ranges over $\{\emptyset,\neg\}^{n_{i_{-t-1}}}$ for each $t\in\{0,1,\ldots,H_i+L-2\}$) is just the subset of $C_{i'}$ consisting of all $f$-periodic points in it. Since $x$ has precisely one $f$-periodic pre-image (which lies in $C_{i'}$), it follows that the value of (\ref{qcalSigmaEq}) for $\vec{\xi}=\vec{\xi}_{i',H_i}$ is $0$ for all $\diamond_{t=0}^{H_i+L-2}{\vec{o_t}}\in\{\emptyset,\neg\}^{n_{i_{-1}}+n_{i_{-2}}+\cdots+n_{i_{-H_i-L+1}}}$ except for one, for which the constant (\ref{qcalSigmaEq}) has value $1$. If $\diamond_{t=0}^{H_i+L-2}{\vec{o_t}}$ is that unique logical sign tuple, then the unique $f$-periodic pre-image $x^{(-1)}$ of $x\in\Bcal(\Qcal_{i,H_i+L-1},\diamond_{t=0}^{H_i+L-1}{\vec{o'_t}}\diamond\vec{\xi}_{i,H_i})$ always lies in $\Bcal(\Qcal_{i',H_i+L-2},\diamond_{t=0}^{H_i+L-2}{\vec{o_t}}\diamond\vec{\xi}_{i',H_i})$, and so $\Tree_{\Gamma_f}(x^{(-1)})\cong\Tree_{i'}(\Pcal_{i'},\diamond_{t=0}^{H_i}{\vec{o_t}}\diamond\vec{\xi}_{i',H_i})$ is also uniquely determined. Continuing this process inductively, we get the statement of the lemma.
\end{proof}

For the purposes of our later complexity analysis, we need a more explicit version of Lemma \ref{refinedLem}. To each $\overline{f}$-periodic $i\in\{0,1,\ldots,d-1\}$ and each $L\in\IN^+$, we associate the set
\begin{align*}
&\Ocal_{i,L}:= \\
&\{\diamond_{t=0}^{H_i+L-1}{\vec{o'_t}}\in\{\emptyset,\neg\}^{n_{i_0}+n_{i_{-1}}+\cdots+n_{i_{-H_i-L+1}}}: \Bcal(\Qcal_{i,H_i+L-1},\diamond_{t=0}^{H_i+L-1}{\vec{o'_t}}\diamond\vec{\xi}_{i,H_i})\not=\emptyset\}
\end{align*}
of\phantomsection\label{not292} logical sign tuples that correspond to a non-empty block of $\Qcal_{i,H_i+L-1}$ consisting of $f$-periodic points. The proof of Lemma \ref{refinedLem} shows that as long as $L\geq2$, we may implicitly define a (surjective) function $\ufrak_{i,L}:\Ocal_{i,L}\rightarrow\Ocal_{i',L-1}$\phantomsection\label{not293} via
\[
\sigma_{\Qcal_{i',H_i+L-2},A_{i'}}(\ufrak_{i,L}(\diamond_{t=0}^{H_i+L-1}{\vec{o'_t}})\diamond\vec{\xi}_{i',H_i},\diamond_{t=1}^{H_i+L-1}{\vec{o'_t}}\diamond\vec{\xi}_{i,H_i})=1.
\]
Then for each ($f$-periodic) $x\in\Bcal(\Qcal_{i,H_i+L-1},\diamond_{t=0}^{H_i+L-1}{\vec{o'_t}}\diamond\vec{\xi}_{i,H_i})$, the unique $f$-periodic pre-image $x^{(-1)}$ of $x$, which lies in $C_{i'}$, is contained in $\Bcal(\Qcal_{i',H_i+L-2},\ufrak_{i,L}(\diamond_{t=0}^{H_i+L-1}{\vec{o'_t}})\diamond\vec{\xi}_{i',H_i})$. Denoting by $\proj_{i,L}$\phantomsection\label{not294} the projection
\[
\Ocal_{i,L}\rightarrow\{\emptyset,\neg\}^{n_{i_0}+n_{i_{-1}}+\cdots+n_{i_{-H_i}}},
\diamond_{t=0}^{H_i+L-1}{\vec{o'_t}}\mapsto\diamond_{t=0}^{H_i}{\vec{o'_t}},
\]
we therefore have the following more explicit version of Lemma \ref{refinedLem}.

\begin{lemmmma}\label{refinedExplicitLem}
Let $L\in\IN^+$, and let $x\in C_i$ be $f$-periodic, say contained in $\Bcal(\Qcal_{i,H_i+L-1},\diamond_{t=0}^{H_i+L-1}{\vec{o'_t}})$. Then for each $k=0,-1,\ldots,-L+1$, we have
\begin{align*}
&\Tree_{\Gamma_f}(x^{(k)})\cong \\
&\Tree_{i_k}(\Pcal_{i_k},(\proj_{i_k,L+k}\circ\ufrak_{i_{k+1},L+k+1}\circ\ufrak_{i_{k+2},L+k+2}\circ\cdots\circ\ufrak_{i_0,L})(\diamond_{t=0}^{H_i+L-1}{\vec{o'_t}})).
\end{align*}
\end{lemmmma}

We note that the composition of functions of the form $\ufrak_{i_t,L+t}$ in the formula in Lemma \ref{refinedExplicitLem} is empty if $k=0$ (index-wise, it is supposed to ascend from $1$ to $0$, which is nonsensical). Specifically, Lemma \ref{refinedExplicitLem} for $k=0$ states that
\[
\Tree_{\Gamma_f}(x)=\Tree_{\Gamma_f}(x^{(0)})=\Tree_{i_0}(\Pcal_{i_0},\proj_{i_0,L}(\diamond_{t=0}^{H_i+L-1}{\vec{o'_t}})),
\]
for $k=-1$, it states that
\[
\Tree_{\Gamma_f}(x^{(-1)})=\Tree_{i_{-1}}(\Pcal_{i_{-1}},\proj_{i_{-1},L-1}\circ\ufrak_{i_0,L}(\diamond_{t=0}^{H_i+L-1}{\vec{o'_t}})),
\]
and so on.

In what follows, let us assume that all cycle lengths of $f$ are at most $L$. We consider an $\overline{f}$-periodic index $i$ of cycle length $\ell$. By Lemma \ref{refinedExplicitLem}, for each $x\in C_i$, the block of $\Qcal_{i,H_i+L-1}$ in which $x$ is contained together with the precise $f$-cycle length of $x$ completely determines the digraph isomorphism type of the connected component of $\Gamma_f$ containing $x$. By adding suitable $s$-congruences to the spanning congruences of $\Qcal_{i,H_i+L-1}$, we can construct a finer arithmetic partition, denoted by $\Wcal_{i,L}$\phantomsection\label{not295} below, each block of which consists of points of a common $f$-cycle length. Hence, for each $f$-periodic point $x\in C_i$, the digraph isomorphism type of the connected component of $\Gamma_f$ containing $x$ is completely determined by the $\Wcal_{i,L}$-block containing $x$.

Let us discuss the details of how to construct $\Wcal_{i,L}$. We recall from Subsection \ref{subsec3P3} that for each $l\in\IN^+$, the restriction of $f^l$ to $C_{i_{-l}}$, which maps to $C_i$, is represented by the affine map $\Acal_{i,l}:x\mapsto \overline{\alpha}_{i,l}x+\overline{\beta}_{i,l}$ (formulas for $\overline{\alpha}_{i,l}$ and $\overline{\beta}_{i,l}$ are given in the first bullet point after Proposition \ref{periodicCosetsTransientProp}). Therefore, a point $x\in C_i$, viewed as an element of $\IZ/s\IZ$, is a fixed point of $f^l$ if and only if $\ell$ divides $l$ (so that $i_{-l}=i$) and $\overline{\alpha}_{i,l}x+\overline{\beta}_{i,l}\equiv x\Mod{s}$. This congruence is solvable if and only if $\gcd(s,\overline{\alpha}_{i,l}-1)\mid\overline{\beta}_{i,l}$, in which case it is equivalent to the $s$-congruence
\[
x\equiv-\frac{\overline{\beta}_{i,l}}{\gcd(s,\overline{\alpha}_{i,l}-1)}\cdot\inv_{\frac{s}{\gcd(s,\overline{\alpha}_{i,l}-1)}}\left(\frac{\overline{\alpha}_{i,l}-1}{\gcd(s,\overline{\alpha}_{i,l}-1)}\right)\Mod{\frac{s}{\gcd(s,\overline{\alpha}_{i,l}-1)}},
\]
which we henceforth denote by $\eta_{i,l}(x)$\phantomsection\label{not296}. We observe that $\eta_{i,l}(x)$ is only well-defined when $\gcd(s,\overline{\alpha}_{i,l}-1)\mid\overline{\beta}_{i,l}$. Let us set
\[
\Cfrak_{i,L}:=\{l\in\{1,2,\ldots,L\}: \ell\mid l\text{ and }\gcd(s,\overline{\alpha}_{i,l}-1)\mid\overline{\beta}_{i,l}\}
\]
and\phantomsection\label{not297} define
\[
\Vcal_{i,L}:=\Pfrak(\eta_{i,l}(x): l\in\Cfrak_{i,L})\text{ and }\Wcal_{i,L}:=\Qcal_{i,H_i+L-1}\wedge\Vcal_{i,L}.
\]
Viewing\phantomsection\label{not298} $\Vcal_{i,L}$ as an arithmetic partition of $C_i$, we claim that its blocks are just those subsets of $C_i$ that consist of all points of any given $f$-cycle length. Indeed, let $l\in\{1,2,\ldots,L\}$. If $l\notin\Cfrak_{i,L}$, then $f^l$ has no fixed points in $C_i$ and, in particular, $f$ has no points of cycle length $l$ in $C_i$. On the other hand, if $l\in\Cfrak_{i,L}$, then the points $x\in C_i$ of $f$-cycle length exactly $l$ (if any) are just those that satisfy the congruence $\eta_{i,l'}(x)$ for precisely those $l'\in\Cfrak_{i,L}$ that are multiples of $l$. In other words, if for $l'\in\Cfrak_{i,L}$, we set
\[
\nu_{l,l'}:=
\begin{cases}
\emptyset, & \text{if }l\mid l', \\
\neg, & \text{otherwise},
\end{cases}
\]
and\phantomsection\label{not299} set $\vec{\nu}_{i,L,l}:=(\nu_{l,l'})_{l'\in\Cfrak_{i,L}}$\phantomsection\label{not300}, then the set $\Bcal(\Vcal_{i,L},\vec{\nu}_{i,L,l})$ (which may be empty) consists precisely of those $x\in C_i$ that are of $f$-cycle length $l$. In summary, we obtain the following result.

\begin{proppposition}\label{refinedExplicitProp}
Let $L\in\IN^+$ be such that all cycle lengths of $f$ are at most $L$, and let $i\in\{0,1,\ldots,d-1\}$ be $\overline{f}$-periodic. We view $\Vcal_{i,L}$ and $\Wcal_{i,L}$ as arithmetic partitions of $C_i$. Then the following hold.
\begin{enumerate}
\item Each block of $\Vcal_{i,L}$ is of one of the forms $\Bcal(\Vcal_{i,L},(\neg,\neg,\ldots,\neg))$, respectively $\Bcal(\Vcal_{i,L},\vec{\nu}_{i,L,l})$ for some $l\in\Cfrak_{i,L}$, and it consists precisely of the $f$-transient points in $C_i$, respectively of those $f$-periodic points in $C_i$ that have $f$-cycle length precisely $l$.
\item Each block of $\Wcal_{i,L}$ consists either entirely of $f$-periodic or entirely of $f$-transient points. Moreover, each block of $\Wcal_{i,L}$ whose elements are $f$-periodic is of the form
\[
\Bcal(\Wcal_{i,L},\diamond_{t=0}^{H_i+L-1}{\vec{o'_t}}\diamond\vec{\xi}_{i,H_i}\diamond\vec{\nu}_{i,L,l})
\]
for some $\vec{o'_t}\in\{\emptyset,\neg\}^{n_{i_{-t}}}$ and $l\in\Cfrak_{i,L}$, in which case for any given point $x$ in that block, the digraph isomorphism type of the connected component of $\Gamma_f$ containing $x$ is represented by the cyclic sequence
\begin{align*}
&[\Tree_{i_k}(\Pcal_{i_k}, \\
&(\proj_{i_k,L+k}\circ\ufrak_{i_{k+1},L+k+1}\circ\ufrak_{i_{k+2},L+k+2}\circ\cdots\circ\ufrak_{i_0,L})(\diamond_{t=0}^{H_i+L}{\vec{o'_t}}))]_{k=-l+1,-l+2,\ldots,0}
\end{align*}
of rooted tree isomorphism types.
\end{enumerate}
\end{proppposition}

Proposition \ref{refinedExplicitProp} is the basis for proving the following theorem.

\begin{theoremmm}\label{lengthsBoundedTheo}
Let $f$ be an index $d$ generalized cyclotomic mapping of $\IF_q$. Moreover, let $L\in\IN^+$ with $L\leq q-1$ be such that all cycle lengths of $f$ are at most $L$. Then, within $q$-bounded query complexity
\[
(8^{d^2\mpe(q-1)+dL}2^{d\mpe(q-1)}(d^3L\mpe(q-1)+d^2L^2)\log^{1+o(1)}{q},d,0,0,0),
\]
and thus within $q$-bounded Las Vegas dual complexity
\begin{align*}
(&8^{d^2\mpe(q-1)+dL}2^{d\mpe(q-1)}(d^3L\mpe(q-1)+d^2L^2)\log^{1+o(1)}{q}+d\log^{3+o(1)}{q}, \\
&d\log^{3+o(1)}{q},d\log{q}),
\end{align*}
one can compute
\begin{itemize}
\item a recursive tree description list $\vec{\Dfrak}=(\Dfrak_n)_{n=0,1,\ldots,N}$ with $N\in O(d2^{d^2\mpe(q-1)+d})$, whose associated sequence of rooted tree isomorphism types is denoted by $\vec{\Ifrak}=(\Ifrak_n)_{n=0,1,\ldots,N}$, such that indices $n\in\{0,1,\ldots,N\}$ as well as second entries of elements of a description $\Dfrak_n$ are represented by bit strings of length $l_{\bit}:=\lfloor\log_2{q}\rfloor+1$; and
\item the tree necklace list $\Nfrak$ of $f$ relative to $\vec{\Ifrak}$, in the sense of Definition \ref{necklaceListDef}, in lexicographically sorted form, which has $O(dL2^{d^2\mpe(q-1)+dL})$ distinct elements (triples) and, by convention,
\begin{itemize}
\item has the first entries $[\vec{\nfrak}]=[\nfrak_1,\nfrak_2,\ldots,\nfrak_{l'}]$ of its elements padded analogously to Remark \ref{necklaceListRem}(3), but with $L-l'$ dummy entries $-1$, so that the bit string representation of $[\vec{\nfrak}]$ always has the length $L(l_{\bit}+1)=L(\lfloor\log_2{q}\rfloor+2)$;
\item uses $\lfloor\log_2{L}\rfloor+1$ bits to represent the second entries of its elements; and
\item uses $\lfloor\log_2{q}\rfloor+1$ bits for the third entries of its elements.
\end{itemize}
\end{itemize}
\end{theoremmm}

\begin{proof}
First, we compute $\overline{f}$, the affine maps $A_i$ and a partition-tree register
\[
((\Zcal_i)_{i=0,1,\ldots,d-1},((\Dfrak_n,(S_{n,i})_{i=0,1,\ldots,d}))_{n=0,1,\ldots,N})
\]
for $f$ with $N\in O(d2^{d^2\mpe(q-1)+d})$; the desired recursive tree description list $\vec{\Dfrak}$ is a part of this. By Proposition \ref{faiProp} and Theorem \ref{complexitiesTheo}(2), these computations can be carried out within $q$-bounded query complexity
\[
(d^3\mpe(q-1)2^{(3d^2+d)\mpe(q-1)+2d}\log^{1+o(1)}{q},d,0,0,0),
\]
which is majorized by the asserted overall $q$-bounded query complexity for computing $\vec{\Dfrak}$ and $\Nfrak$ (it is this term which necessitates the inclusion of the factor $2^{d\mpe(q-1)}$ in the bound on the overall bit operation cost). Moreover, by the proof of Theorem \ref{complexitiesTheo}(2), the following are computed (and may be stored) as part of this:
\begin{itemize}
\item the cycles of $\overline{f}$ and a CRL-list $\overline{\Lcal}$ of $\overline{f}$;
\item the parameter $H_i$ for each $\overline{f}$-periodic index $i\in\{0,1,\ldots,d-1\}$.
\end{itemize}
Until further notice, we assume that $(i,\ell)\in\overline{\Lcal}$ with $i<d$ is fixed (the case $i=d$ is easy to deal with separately and will be \enquote{tacked on} at the end of this proof). As usual, we let $(i_0,i_1,\ldots,i_{\ell-1})$ with $i=i_0$ be the $\overline{f}$-cycle of $i$, and set $i_t:=i_{t\bmod{\ell}}$ for arbitrary $t\in\IZ$. We analyze the bit operation cost of counting the isomorphism types of connected components of $\Gamma_f$ that intersect $\bigcup_{t=0}^{d-1}{C_{i_t}}$ (i.e., that may be represented by a periodic vertex in one of the cosets $C_{i_t}$). We note that for each $t\in\{0,1,\ldots,\ell-1\}$, one can directly read off a spanning congruence sequence for $\Ucal_{i_t}$, of length $H_i\leq d\mpe(q-1)$, from the partition-tree register computed above. To proceed, we need to determine a spanning congruence sequence of $\Qcal_{i_{-t},H_i+L-t-1}$ for $t=0,1,\ldots,L-1$. By the definitions of $\Rcal_i$ and $\Qcal_{i,h}$ from Subsection \ref{subsec3P3}, we have
\[
\Qcal_{i_{-t},H_i+L-t-1}=\Rcal_{i_{-t}}\wedge\Pfrak'(\Qcal_{i_{-t-1},H_i+L-t-2},A_{i_{-t-1}}).
\]
Moreover, we observe that
\begin{align*}
&\Pfrak'(\Qcal_{i_{-t-1},H_i+L-t-2},A_{i_{-t-1}})
=\Pfrak'\left(\bigwedge_{j=0}^{H_i+L-t-2}{\lambda_{i_{-t-1-j}}^j(\Rcal_{i_{-t-1-j}})}\wedge\Ucal_{i_{-t-1}},A_{i_{-t-1}}\right) \\
&=\bigwedge_{j=0}^{H_i+L-t-2}{\lambda_{i_{-t-1-j}}^{j+1}(\Rcal_{i_{-t-1-j}})}\wedge\Ucal_{i_{-t}}
=\bigwedge_{j=1}^{H_i+L-t-1}{\lambda_{i_{-t-j}}^j(\Rcal_{i_{-t-j}})}\wedge\Ucal_{i_{-t}},
\end{align*}
whence
\[
\Qcal_{i_{-t},H_i+L-t-1}=\bigwedge_{j=0}^{H_i+L-t-1}{\lambda_{i_{-t-j}}^j(\Rcal_{i_{-t-j}})}\wedge\Ucal_{i_{-t}}.
\]
Now, from our partition-tree register, we can directly read off a spanning congruence sequence for $\lambda_{i_t}^j(\Rcal_{i_t})$, of length $n_{i_t}\leq d$, for each $t=0,1,\ldots,\ell-1$ and each $j=0,1,\ldots,H_i$, which altogether takes only $O(d^3\mpe(q-1)\log{q})$ bit operations for copying. Moreover, for any fixed $t\in\{0,1,\ldots,\ell-1\}$, we can compute a spanning congruence sequence for
\[
\lambda_{i_t}^j(\Rcal_{i_t})=\lambda(\lambda_{i_t}^{j-1}(\Rcal_{i_t}),A_{i_{t+j-1}})
\]
successively for $j=H_i+1,H_i+2,\ldots,H_i+L-1$. For each given $t$ and $j$, this takes $O(d\log^{1+o(1)}{q})$ bit operations, and thus for all $t$ and $j$ together, it takes $O(d^2L\log^{1+o(1)}{q})$ bit operations. Once all of these spanning congruence sequences have been computed, one can paste together such a sequence for a single partition of the form $\Qcal_{i_{-t},H_i+L-t-1}$ using $O((H_i+L)d\log{q})\subseteq O((d^2\mpe(q-1)+dL)\log{q})$ bit operations. Doing so for all $t=0,1,\ldots,L-1$ takes $O((d^2L\mpe(q-1)+dL^2)\log{q})$ bit operations.

Our next goal is to compute the function
\[
\ufrak_{i_{-t},L-t}:\Ocal_{i_{-t},L-t}\rightarrow\Ocal_{i_{-t-1},L-t-1}
\]
for $t=0,1,\ldots,L-2$. To that end, we first compute the set
\begin{align*}
&\Ocal_{i_{-t},L-t}= \\
&\{\diamond_{k=0}^{H_i+L-t-1}{\vec{o'_k}}\in\{\emptyset,\neg\}^{n_{i_{-t}}+n_{i_{-t+1}}+\cdots+n_{i_{-H_i-L+1}}}: \\
&\Bcal(\Qcal_{i_{-t},H_i+L-t-1},\diamond_{k=0}^{H_i+L-t-1}{\vec{o'_k}}\diamond\vec{\xi}_{i_{-t},H_i})\not=\emptyset\}
\end{align*}
for $t=0,1,\ldots,L-1$. To do so, we go through the
\[
O(2^{d(H_i+L-t)})\subseteq O(2^{d^2\mpe(q-1)+dL-dt})
\]
tuples $\diamond_{k=0}^{H_i+L-t-1}{\vec{o'_k}}$, and for each of them, we compute the cardinality of the block
\[
\Bcal(\Qcal_{i_{-t},H_i+L-t-1},\diamond_{k=0}^{H_i+L-t-1}{\vec{o'_k}}\diamond\vec{\xi}_{i_{-t},H_i}).
\]
Following the ideas leading to Proposition \ref{zeroBlockProp}, this cardinality is equal to the distribution number
\begin{equation}\label{specialDistNumEq}
\sigma_{\Qcal_{i_{-t},H_i+L-t-1},\mathbf{0}}(\diamond_{k=0}^{H_i+L-t-1}{\vec{o'_k}}\diamond\vec{\xi}_{i_{-t},H_i},(\emptyset,\ldots,\emptyset))
\end{equation}
of $\Qcal_{i_{-t},H_i+L-t-1}$ under the constantly zero affine function $\mathbf{0}$. By the proof of Lemma \ref{distNumLem} and the facts that
\begin{itemize}
\item the number of spanning congruences of $\Qcal_{i_{-t},H_i+L-t-1}$ we are using is at most $d(H_i+L)+H_i\leq d^2\mpe(q-1)+d\mpe(q-1)+dL\in O(d^2\mpe(q-1)+dL)$,
\item the subsets $J$ we need to loop over never contain any index corresponding to a logical sign for one of the $H_i$ spanning congruences of $\Ucal_{i_{-t}}$, because $\vec{\xi}_{i_{-t},H_i}$ consists only of positive logical signs, and
\item $d(H_i+L)\leq d^2\mpe(q-1)+dL$,
\end{itemize}
we conclude that the complexity of computing the distribution number (\ref{specialDistNumEq}) is in
\begin{equation}\label{specialDistNumCompEq}
O(2^{d^2\mpe(q-1)+dL}(d^2\mpe(q-1)+dL)\log^{1+o(1)}{q}).
\end{equation}
In summary, computing all sets $\Ocal_{i_{-t},L-t}$ for $t=0,1,\ldots,L-1$ takes
\begin{align*}
&O\left(\sum_{t=0}^{L-1}{4^{d^2\mpe(q-1)+dL}2^{-dt}(d^2\mpe(q-1)+dL)\log^{1+o(1)}{q}}\right) \\
\subseteq &O(4^{d^2\mpe(q-1)+dL}(d^2\mpe(q-1)+dL)\log^{1+o(1)}{q}).
\end{align*}
bit operations.

Concerning the computation of the functions $\ufrak_{i_{-t},L-t}$ themselves, we note that for a given $t$ and argument $\diamond_{k=0}^{H_i+L-t-1}{\vec{o'_k}}$ of $\ufrak_{i_{-t},L-t}$, the associated function value $\ufrak_{i_{-t},L-t}(\diamond_{k=0}^{H_i+L-t-1}{\vec{o'_k}})$ is the unique tuple $\diamond_{k=0}^{H_i+L-t-2}{\vec{o_k}}\in\Ocal_{i_{-t-1},L-t-1}$ such that the distribution number
\[
\sigma_{\Qcal_{i_{-t-1},H_i+L-t-2},A_{i_{-t-1}}}(\diamond_{k=0}^{H_i+L-t-2}{\vec{o_k}}\diamond\vec{\xi}_{i_{-t-1},H_i},\diamond_{k=1}^{H_i+L-t-1}{\vec{o'_k}}\diamond\vec{\xi}_{i_{-t},H_i})
\]
is equal to $1$ (and, according to the proof of Lemma \ref{refinedLem}, that distribution number is $0$ for all other choices of $\diamond_{k=0}^{H_i+L-t-2}{\vec{o_k}}$). Therefore, in order to compute each value of every function $\ufrak_{i_{-t},L-t}$ for $t=0,1,\ldots,L-2$, we need to carry out
\[
O\left(\sum_{t=0}^{L-2}{4^{d^2\mpe(q-1)+dL-dt}2^{-d}}\right)\subseteq O\left(4^{d^2\mpe(q-1)+dL}\right)
\]
computations of a distribution number of $\Qcal_{i_{-t-1},H_i+L-t-2}$ of the above form, and as above (this time using that all entries of $\vec{\xi}_{i_{-t-1},H_i}$ are the positive logical sign), the bit operation cost of each individual such distribution number computation is (\ref{specialDistNumCompEq}). Therefore, we end up with a total bit operation cost of
\[
O(8^{d^2\mpe(q-1)+dL}(d^2\mpe(q-1)+dL)\log^{1+o(1)}{q})
\]
for computing all functions $\ufrak_{i_{-t},L-t}$.

Next, we compute the set $\Cfrak_{i,L}$, and for each $l\in\Cfrak_{i,L}$, we compute the $s$-congruence $\eta_{i,l}(x)$. By Lemma \ref{complexitiesLem}(1,3,4,8), this takes $O((d+L)\log^{1+o(1)}{q})$ bit operations altogether if $\Acal_{i,l}=(A_{i_0}A_{i_1}\cdots A_{i_{\ell-1}})^{l/\ell}=\Acal_i^{l/\ell}$, with linear coefficient $\overline{\alpha}_{i,l}$ and constant coefficient $\overline{\beta}_{i,l}$, is stored for each $l$ such that $\ell\mid l$, then computed for the next larger relevant value, $l+\ell$, using the formula $\Acal_{i,l+\ell}=\Acal_{i,l}\Acal_i$ whenever $\ell\mid l$. With these computations, we have established a spanning congruence sequence for $\Wcal_{i,L}$, of length at most $d(H_i+L)+H_i+L\leq d^2\mpe(q-1)+d\mpe(q-1)+dL+L$. Now, we go through the $O(2^{d(H_i+L)}L)\subseteq O(2^{d^2\mpe(q-1)+dL}L)$ logical sign tuples $\diamond_{k=0}^{H_i+L-1}{\vec{o'_k}}\diamond\vec{\nu}_{i,L,l}$ that parametrize subsets of $\IZ/s\IZ$ of the form $\Bcal(\Wcal_{i,L},\diamond_{k=0}^{H_i+L-1}{\vec{o'_k}}\diamond\vec{\xi}_{i,H_i}\diamond\vec{\nu}_{i,L,l})$ -- we observe that the non-empty such sets are just those blocks of $\Wcal_{i,L}$ that consist of $f$-periodic points. For each such tuple, we compute
\[
\mfrak_i(\diamond_{k=0}^{H_i+L-1}{\vec{o'_k}}\diamond\vec{\nu}_{i,L,l}):=|\Bcal(\Wcal_{i,L},\diamond_{k=0}^{H_i+L-1}{\vec{o'_k}}\diamond\vec{\xi}_{i,H_i}\diamond\vec{\nu}_{i,L,l})|
\]
as\phantomsection\label{not301} the distribution number
\[
\sigma_{\Wcal_{i,L},\mathbf{0}}(\diamond_{k=0}^{H_i+L-1}{\vec{o'_k}}\diamond\vec{\xi}_{i,H_i}\diamond\vec{\nu}_{i,L,l},(\emptyset,\emptyset,\ldots,\emptyset)),
\]
which costs
\begin{align*}
&O(2^{d(H_i+L)+L}(d(H_i+L)+H_i+L)\log^{1+o(1)}{q}) \\
&\subseteq O(2^{d^2\mpe(q-1)+dL+L}(d^2\mpe(q-1)+dL)\log^{1+o(1)}{q})
\end{align*}
bit operations per cardinality to compute. If $\mfrak_i(\diamond_{k=0}^{H_i+L-1}{\vec{o'_k}}\diamond\vec{\nu}_{i,L,l})=0$, then we discard that case. Otherwise, we store the computed block size. Overall, this process takes
\begin{align*}
&O(4^{d^2\mpe(q-1)+dL}2^L(d^2L\mpe(q-1)+dL^2)\log^{1+o(1)}{q}) \\
\subseteq &O(8^{d^2\mpe(q-1)+dL}(d^2L\mpe(q-1)+dL^2)\log^{1+o(1)}{q})
\end{align*}
bit operations.

We recall that in view of Lemmas \ref{refinedLem} and \ref{refinedExplicitLem}, the blocks of $\Wcal_{i,L}$ that consist of $f$-periodic points control the digraph isomorphism type of the connected component of $\Gamma_f$ containing any given point in the block. We use Lemma \ref{refinedExplicitLem} to compute, for each tuple $\diamond_{k=0}^{H_i+L-1}{\vec{o'_k}}\diamond\vec{\nu}_{i,L,l}$ whose $\mfrak_i$-value (the associated block size) is non-zero, the unique cyclic sequence $[\nfrak_1,\nfrak_2,\ldots,\nfrak_{l'}]$ with entries in $\{0,1,\ldots,N\}$ and of minimal period $l'$ such that the cyclic sequence of rooted tree isomorphism types characterizing the corresponding digraph isomorphism type is equal to
\[
[\diamond_{m=1}^{l/l'}(\Ifrak_{\nfrak_1},\Ifrak_{\nfrak_2},\ldots,\Ifrak_{\nfrak_{l'}})].
\]
To that end, for fixed $\diamond_{k=0}^{H_i+L-1}{\vec{o'_k}}\diamond\vec{\nu}_{i,L,l}$, we compute the sequence
\[
(\nfrak^{(-l+1)},\nfrak^{(-l+2)},\ldots,\nfrak^{(0)})
\]
where $\nfrak^{(-t)}$ for $t\in\{0,1,\ldots,l-1\}$ is the unique index in $\{0,1,\ldots,N\}$ such that
\begin{align*}
&\Tree_{i_{-t}}(\Pcal_{i_{-t}},(\proj_{i_{-t,L-t}}\circ\ufrak_{i_{-t+1},L-t+1}\circ\ufrak_{i_{-t+2},L-t+2}\circ\cdots\circ\ufrak_{i_0,L})(\diamond_{k=0}^{H_i+L-1}{\vec{o'_k}})) \\
\cong &\Ifrak_{\nfrak^{(-t)}}.
\end{align*}
For each fixed $t$, this requires us to compute the logical sign tuple
\begin{equation}\label{logicalSignTupleEq}
(\proj_{i_{-t,L-t}}\circ\ufrak_{i_{-t+1},L-t+1}\circ\ufrak_{i_{-t+2},L-t+2}\circ\cdots\circ\ufrak_{i_0,L})(\diamond_{k=0}^{H_i+L-1}{\vec{o'_k}}).
\end{equation}
Now, we can compute
\[
(\ufrak_{i_{-t+1},L-t+1}\circ\ufrak_{i_{-t+2},L-t+2}\circ\cdots\circ\ufrak_{i_0,L})(\diamond_{k=0}^{H_i+L-1}{\vec{o'_k}})
\]
simply by looking up the pre-computed values of the functions $\ufrak_{i_{-j},L-j}$, which takes $O(t(H_i+L)d)\subseteq O(d^2L\mpe(q-1)+dL^2)$ bit operations. The $\proj_{i_t,L-t}$-value of this tuple is a projection onto an initial segment, which can be read off using $O(d(H_i+1))\subseteq O(d^2\mpe(q-1))$ bit operations. Finally, we need to look up the number $\nfrak^{(-t)}$ in our partition-tree register -- it is characterized by the inclusion
\[
(\proj_{i_{-t,L-t}}\circ\ufrak_{i_{-t+1},L-t+1}\circ\ufrak_{i_{-t+2},L-t+2}\circ\cdots\circ\ufrak_{i_0,L})(\diamond_{k=0}^{H_i+L-1}{\vec{o'_k}})\in S_{\nfrak^{(-t)},i_{-t},H_i}.
\]
Therefore, in order to determine and store $\nfrak^{(-t)}$, we need to go through the pairwise disjoint sets $S_{n,i_{-t},H_i}$ for $n=0,1,\ldots,N$ until we find the one that contains the logical sign tuple (\ref{logicalSignTupleEq}). Because $\sum_{n=0}^N{|S_{n,i_{-t},H_i}|}\leq 2^{d(H_i+1)}\in O(2^{d^2\mpe(q-1)+d})$, and each element of each set $S_{n,i_{-t},H_i}$ is a logical sign tuple of length in $O(d(H_i+1))\subseteq O(d^2\mpe(q-1))$, it takes
\[
O(N+2^{d^2\mpe(q-1)+d}d^2\mpe(q-1)+\log{q})\subseteq O(d^2\mpe(q-1)2^{d^2\mpe(q-1)+d}+\log{q})
\]
bit operations to determine $\nfrak^{(-t)}$ for our fixed value of $t$. In total, the computation of the sequence $\vec{\nfrak}=(\nfrak^{(-l+1)},\nfrak^{(-l+2)},\ldots,\nfrak^{(0)})$ for a fixed value of $\diamond_{k=0}^{H_i+L-1}{\vec{o'_k}}\diamond\vec{\nu}_{i,L,l}$ takes
\[
O(d^2L^2\mpe(q-1)+dL^3+d^2L\mpe(q-1)2^{d^2\mpe(q-1)+d}+L\log{q})
\]
bit operations. The cyclic sequence $[\nfrak_1,\nfrak_2,\ldots,\nfrak_{l'}]$ we are looking for is simply $[\nfrak^{(-l+1)},\nfrak^{(-l+2)},\ldots,\nfrak^{(-l+l')}]$ where $l'=\minper(\vec{\nfrak})$, which can be computed in
\[
O(l^{1+o(1)}l_{\bit})\subseteq O(L^{1+o(1)}\log{q})
\]
bit operations by Lemma \ref{periodLem}. Finally, in representing this cyclic sequence, we would like to replace $(\nfrak_1,\ldots,\nfrak_{l'})$ by the lexicographically minimal ordered sequence in the same cyclic equivalence class. This takes another $O(L^2\log{L}\log{q})$ bit operations per such sequence. In total, the computation of the cyclic sequence $[\nfrak_1,\nfrak_2,\ldots,\nfrak_{l'}]$ for all of the $O(2^{d(H_i+L)}L)\subseteq O(2^{d^2\mpe(q-1)+dL}L)$ logical sign tuples $\diamond_{k=0}^{H_i+L-1}{\vec{o'_k}}\diamond\vec{\nu}_{i,L,l}$ takes
\begin{align*}
&O(2^{d^2\mpe(q-1)+dL}L\cdot(d^2L^2\mpe(q-1)+dL^3+d^2L\mpe(q-1)2^{d^2\mpe(q-1)+d} \\
&+L^2\log{L}\log{q}))\subseteq O(8^{d^2\mpe(q-1)+dL}\log^{1+o(1)}{q})
\end{align*}
bit operations. In summary, the bit operation cost of the computations described after fixing $(i,\ell)\in\overline{\Lcal}\setminus\{(d,1)\}$ for all of those $O(d)$ pairs $(i,\ell)$ together is in
\[
O(8^{d^2\mpe(q-1)+dL}(d^3L\mpe(q-1)+d^2L^2)\log^{1+o(1)}{q}).
\]
At this point, we have computed, for each $(i,\ell)\in\overline{\Lcal}\setminus\{(d,1)\}$, a set $\Nfrak_i$ of the form
\begin{align*}
&\Nfrak_i= \\
&\{(\diamond_{k=0}^{H_i+L-1}{\vec{o'_k}}\diamond\vec{\nu}_{i,L,l},[\nfrak_1,\nfrak_2,\ldots,\nfrak_{l'}],\mfrak_i(\diamond_{k=0}^{H_i+L-1}{\vec{o'_k}}\diamond\vec{\nu}_{i,L,l})): \\
&\vec{o'_k}\in\{\emptyset,\neg\}^{n_{i_{-k}}}\text{ for }k=0,1,\ldots,H_i+L-1,l\in\Cfrak_{i,L},\text{ and} \\
&\mfrak_i(\diamond_{k=0}^{H_i+L-1}{\vec{o'_k}}\diamond\vec{\nu}_{i,L,l})>0\}.
\end{align*}
We note that for each element of $\Nfrak_i$, we have $l'\mid l$ and $\minper([\nfrak_1,\ldots,\nfrak_{l'}])=l'$. We also observe that while $\Nfrak_i$ is \emph{not} the tree necklace list, relative to $\vec{\Ifrak}$, for the restriction of $f$ to the union of all cosets $C_j$ where $j$ is a vertex of the connected component of $\Gamma_{\overline{f}}$ containing $i$, that tree necklace list could be easily derived from $\Nfrak_i$ as follows. If we fix $l$ and $[\nfrak_1,\ldots,\nfrak_{l'}]$ and add up all the third entries of the corresponding triples in $\Nfrak_i$ (i.e., of those triples where the first entry has terminal segment $\vec{\nu}_{i,L,l}$ and the second entry is $[\nfrak_1,\ldots,\nfrak_{l'}]$), then we end up with the exact number of connected components of $\Gamma_f$ that are characterized by the cyclic sequence of rooted tree isomorphism types $[\diamond_{m=1}^{l/l'}{(\Ifrak_{\nfrak_1},\ldots,\Ifrak_{\nfrak_{l'}})}]$ and are contained in the union of all cosets $C_j$ for $j$ in the connected component of $\Gamma_{\overline{f}}$ containing $i$. This observation also implies that we can compute the full tree necklace list $\Nfrak$ of $f$ relative to $\vec{\Ifrak}$ as follows.

We start by setting $\Nfrak:=\emptyset$. Throughout the process described below, $\Nfrak$ is a set of triples $([\nfrak_1,\nfrak_2,\ldots,\nfrak_{l'}],l,\mfrak)$ such that
\begin{itemize}
\item $|\Nfrak|\leq\sum_{(i,\ell)\in\overline{\Lcal}\setminus\{(d,1)\}}{|\Nfrak_i|}\in O(dL2^{d^2\mpe(q-1)+dL})$;
\item $\nfrak_1,\ldots,\nfrak_{l'}\in\{0,1,\ldots,N\}$;
\item $\minper([\nfrak_1,\ldots,\nfrak_{l'}])=l'$;
\item $[\diamond_{m=1}^{l/l'}{(\Ifrak_{\nfrak_1},\Ifrak_{\nfrak_2},\ldots,\Ifrak_{\nfrak_{l'}})}]$ is a cyclic sequence of rooted tree isomorphism types that characterizes at least one connected component of $\Gamma_f$ that is contained in $\IF_q^{\ast}$; and
\item $\mfrak\in\IN^+$ is the exact number of connected components of $\Gamma_f$ that are contained in $\IF_q^{\ast}$ and are characterized by $[\diamond_{m=1}^{l/l'}{(\Ifrak_{\nfrak_1},\Ifrak_{\nfrak_2},\ldots,\Ifrak_{\nfrak_{l'}})}]$.
\end{itemize}
We observe that the first entry of each triple in $\Nfrak$ is a cyclic sequence of length in $O(L)$ each entry of which has bit length in $O(\log{N})\subseteq O(d^2\mpe(q-1))$, and the second and third entries of such a triple may be represented by bit strings of length in $O(\log{L})$ and $O(\log{q})$, respectively. Now, to get an almost-final form of $\Nfrak$, we loop over the pairs $(i,\ell)\in\overline{\Lcal}\setminus\{(d,1)\}$, and for each such pair, we loop over the elements of $\Nfrak_i$; this double loop has $O(dL2^{d^2\mpe(q-1)+dL})$ individual iterations. In each iteration, we add at most one new element to $\Nfrak$, which explains the above bound on $|\Nfrak|$ that is valid throughout the process. Specifically, an iteration consists of the following steps. Associated with the triple in $\Nfrak_i$ we are considering, we have the parameters $l$, which can be read off from the first entry $\diamond_{k=0}^{H_i+L}{\vec{o'_k}}\diamond\vec{\nu}_{i,L,l}$ of the triple using $O(l\log{l})\subseteq O(L^{1+o(1)})$ bit operations (by scanning to find the first positive logical sign in $\vec{\nu}_{i,L,l}$ and incrementing a counter during that process), and $[\nfrak_1,\ldots,\nfrak_{l'}]$. We check whether these already occur as the first two entries of some element of $\Nfrak$, which takes
\[
O(dL2^{d^2\mpe(q-1)+dL}\cdot L\log{q})=O(dL^22^{d^2\mpe(q-1)+dL}\log{q})
\]
bit operations. If this is the case, then we end the current iteration, having added nothing to $\Nfrak$. Otherwise, we compute the sum $\mfrak$ of the third entries of all triples in $\bigcup_{(j,\ell')\in\overline{\Lcal}\setminus\{(d,1)\}}{\Nfrak_j}$ that have the parameters $l$ and $[\nfrak_1,\ldots,\nfrak_{l'}]$ associated with them. This takes
\begin{align*}
&O(dL2^{d^2\mpe(q-1)+dL}(L\log{L}+L\log{q})=O(dL^22^{d^2\mpe(q-1)+dL}\log{q})
\end{align*}
bit operations. Then we add $([\nfrak_1,\ldots,\nfrak_{l'}],l,\mfrak)$ to $\Nfrak$ as a new element and end the current iteration. Overall, this double loop takes
\begin{align*}
&O(dL2^{d^2\mpe(q-1)+dL}\cdot dL^22^{d^2\mpe(q-1)+dL}\log{q})=O(d^2L^34^{d^2\mpe(q-1)+dL}\log{q}) \\
\subseteq&\,O(8^{d^2\mpe(q-1)+dL}\log{q})
\end{align*}
bit operations.

At the end of the double loop, $\Nfrak$ is almost equal to the tree necklace list for $f$ relative to $\vec{\Ifrak}$; only the connected component of $\Gamma_f$ containing $0_{\IF_q}$, which is characterized by the cyclic sequence $[\Tree_{\Gamma_f}(0_{\IF_q})]$, has not yet been accounted for. To find the index number $\nfrak\in\{0,1,\ldots,N\}$ of $\Tree_{\Gamma_f}(0_{\IF_q})$ with respect to our partition-tree register, we note that $\nfrak$ is characterized by the equality $S_{\nfrak,d}=\emptyset$ (the positive logical sign). Therefore, we only need an additional $O(N)\subseteq O(d2^{d^2\mpe(q-1)+d})$ bit operations to find $\nfrak$. Then, we need to check whether $[\nfrak]$ and $1$ already occur as the first two entries of some triple in $\Nfrak$, which takes
\[
O(dL2^{d^2\mpe(q-1)+dL}\cdot\log{q})\subseteq O(8^{d^2\mpe(q-1)+dL}\log{q})
\]
bit operations. If so, we increase the third entry of that triple by $1$ and halt. Otherwise, we add $([\nfrak],1,1)$ to $\Nfrak$ as a new element.

Finally, we sort the computed array representing $\Nfrak$ lexicographically. By Lemma \ref{complexitiesLem}(10), since $|\Nfrak|\in O(dL2^{d^2\mpe(q-1)+dL})$ and each entry of the array has bit length in $O(L\log{q})$, this takes
\[
O(dL^2(d^2\mpe(q-1)+dL)2^{d^2\mpe(q-1)+dL}\log{q})\subseteq O(8^{d^2\mpe(q-1)+dL}\log{q})
\]
bit operations. We conclude by outputting $(\vec{\Dfrak},\Nfrak)$ and halting.
\end{proof}

Of course, Theorem \ref{lengthsBoundedTheo} is not useful in practice unless the maximum cycle length of a generalized cyclotomic mapping can be computed efficiently. The following proposition takes care of that.

\begin{proppposition}\label{lengthsBoundedProp}
Let $f$ be an index $d$ generalized cyclotomic mapping of $\IF_q$. The maximum cycle length of $f_{\mid\per(f)}$ can be computed within $q$-bounded query complexity
\[
(d\log^{2+o(1)}{q}+d^2\log^2{d},d,1,d,0).
\]
\end{proppposition}

\begin{proof}
This is similar in spirit to Proposition \ref{ctTauProp}, but we obtain a better bound than by using that proposition directly, because there is no need to spell out entire cycle types here. First, we compute $\overline{f}$, the affine maps $A_i$, the cycles of $\overline{f}$ and a CRL-list $\overline{\Lcal}$ of $\overline{f}$. By Proposition \ref{faiProp} and the beginning of Subsubsection \ref{subsubsec5P2P1}, this takes $q$-bounded query complexity
\[
(d\log^{1+o(1)}{q}+d^2\log^2{d},d,0,0,0).
\]
We also factor $s$, using a single $q$-bounded mdl query (i.e., spending $q$-bounded query complexity $(\log{q},0,1,0,0)$).

Now, we loop over the pairs $(i,\ell)\in\overline{\Lcal}\setminus\{(d,1)\}$, and for each of them, we do the following. First, we compute $\Acal_i=A_{i_0}A_{i_1}\cdots A_{i_{\ell-1}}:x\mapsto\overline{\alpha}_ix+\overline{\beta}_i$, taking $O(\ell\log^{1+o(1)}{q})$ bit operations (per $(i,\ell)$) by the beginning of Subsubsection \ref{subsubsec5P2P1}. Next, we determine the largest cycle length of $\Acal_i$ on $\per(\Acal_i)\subseteq\IZ/s\IZ$. To do so, we note that every cycle length of $\Acal_i$ is a least common multiple of cycle lengths of those primary components $\overline{\Acal}_{i,p}=\Acal_i\bmod{p^{\nu_p(s)}}$ where $p\mid s$ and $p\nmid\overline{\alpha}_i$. But by \cite[Tables 3 and 4]{BW22b} (or our Table \ref{crlListPrimaryTable}), the largest cycle length of $\overline{\Acal}_{i,p}$ is equal to the order of $\overline{\Acal}_{i,p}$ in $\Sym(\IZ/p^{\nu_p(s)}\IZ)$, i.e., to the least common multiple of all cycle lengths of $\overline{\Acal}_{i,p}$. It follows that the largest cycle length of $\Acal_i$ is equal to $\ord(\Acal_i\bmod{s'_i})$ where $s'_i=\prod_{p\mid s,p\nmid\overline{\alpha}_i}{p^{\nu_p(s)}}$. We compute this order as follows.

We set $s'_i:=1$ and loop over the $O(\log{q})$ primes $p$ dividing $s$ (which can be read off from the factorization of $s$ computed above). For each $p$, we check whether $p\mid\overline{\alpha}_i$, taking $O(\log^{1+o(1)}{q})$ bit operations by Lemma \ref{complexitiesLem}(3). If so, we skip to the next $p$; otherwise, we overwrite $s'_i:=s'_i\cdot p^{\nu_p(s)}$, taking $O(\log^{1+o(1)}{q})$ bit operations for the multiplication (there is no need to compute the power $p^{\nu_p(s)}$, because it is part of the output of the mdl query used to factor $s$). At the end of this loop over $p$, which has a bit operation cost in $O(\log^{2+o(1)}{q})$, the variable $s'_i$ has the desired value. Now, we compute $\overline{\alpha}'_i:=\overline{\alpha}_i\bmod{s'_i}$\phantomsection\label{not302} and $\overline{\beta}'_i:=\overline{\beta}_i\bmod{s'_i}$\phantomsection\label{not303}, taking $O(\log^{1+o(1)}{q})$ bit operations, to get the affine map $\Acal'_i=\Acal_i\bmod{s'_i}:x\mapsto\overline{\alpha}'_ix+\overline{\beta}'_i$ of $\IZ/s'_i\IZ$. We wish to compute $\ord(\Acal'_i)$. To that end, we first compute the (multiplicative) order $\ord_{s'_i}(\overline{\alpha}'_i)$ with a $q$-bounded mord query. We observe that $\ord_{s'_i}(\overline{\alpha}'_i)$ divides $\ord(\Acal'_i)$, because $(\Acal'_i)^t(x)=(\overline{\alpha}'_i)^tx+c(\overline{\alpha}'_i,\overline{\beta}'_i,t)$ for all $x\in\IZ/s'_i\IZ$ and all $t\in\IZ$. Therefore,
\[
\ord(\Acal'_i)=\ord_{s'_i}(\overline{\alpha}'_i)\cdot\ord((\Acal'_i)^{\ord_{s'_i}(\overline{\alpha}'_i)}).
\]
But $(\Acal'_i)^{\ord_{s'_i}(\overline{\alpha}'_i)}$ is the translation $x\mapsto x+\overline{\beta}''_i$\phantomsection\label{not304} such that
\[
\overline{\beta}''_i=
\begin{cases}
\overline{\beta}'_i, & \text{if }\overline{\alpha}'_i=1, \\
\overline{\beta}'_i\frac{(\overline{\alpha}'_i)^{\ord_{s'_i}(\overline{\alpha}'_i)}-1}{\overline{\alpha}'_i-1}, & \text{otherwise},
\end{cases}
\]
where the formula in the second case is to be evaluated in the ring $\IZ$, although the result is to be viewed as an element of $\IZ/s'_i\IZ$. Because $\ord((\Acal'_i)^{\ord_{s'_i}(\overline{\alpha}'_i)})$ is equal to the additive order of $\overline{\beta}''_i$ modulo $s'_i$, we can work out $\ord(\Acal'_i)$, which coincides with the largest cycle length of $\Acal_i$, using another $O(\log^{2+o(1)}{q})$ bit operations (for computing $\overline{\beta}''_i$, which may involve a power computation, and working out its additive order modulo $s'_i$ via a gcd computation and a division). We conclude this loop by setting $l_i:=\ell\cdot\ord(\Acal'_i)$, which takes $O(\log^{1+o(1)}{q})$ bit operations to compute and is the largest cycle length of $f$ on its periodic points in $\bigcup_{t=0}^{\ell-1}{C_{i_t}}$. This ends our description of the loop over $(i,\ell)\in\overline{\Lcal}\setminus\{(d,1)\}$, which overall takes $q$-bounded query complexity $(d\log^{2+o(1)}{q},0,0,d,0)$, using that $\sum_{(i,\ell)\in\overline{\Lcal}}{\ell}\leq d$.

Finally, the maximum cycle length of $f$ is simply the maximum value among the $l_i$ for $(i,\ell)\in\overline{\Lcal}$, where $l_d:=1$, which takes $O(d\log{q})$ bit operations to compute.
\end{proof}

To conclude, we give the following corollary of Theorem \ref{lengthsBoundedTheo}, which can be seen as the main result of this subsubsection.

\begin{corrrollary}\label{lengthsBoundedCor}
Let $f_1$ and $f_2$ be generalized cyclotomic mappings of $\IF_q$, of index $d_1$ and $d_2$ respectively, and set $d:=\max\{d_1,d_2\}$. Moreover, let $L\in\IN^+$, and denote by $L_1$, respectively $L_2$, the maximum cycle length of $f_1$, respectively of $f_2$, on its periodic points. Then, if $\min\{L_1,L_2\}\leq L$, it can be decided whether $\Gamma_{f_1}\cong\Gamma_{f_2}$ within $q$-bounded query complexity
\begin{align*}
(&8^{d^2\mpe(q-1)+dL}2^{d\mpe(q-1)}(d^3L\mpe(q-1)+d^2L^2)\log^{1+o(1)}{q} \\
&+d^24^{d^2\mpe(q-1)+d}\log^2{q}+d\log^{2+o(1)}{q}, \\
&d,1,d,0),
\end{align*}
and thus within $q$-bounded Las Vegas dual complexity
\begin{align*}
(&8^{d^2\mpe(q-1)+dL}2^{d\mpe(q-1)}(d^3L\mpe(q-1)+d^2L^2)\log^{1+o(1)}{q} \\
&+d^24^{d^2\mpe(q-1)+d}\log^2{q}+d\log^{2+o(1)}{q}+d\log^{7+o(1)}{q}, \\
&d\log^{3+o(1)}{q},d\log{q}).
\end{align*}
\end{corrrollary}

\begin{proof}
First, we compute $L_1$ and $L_2$, which takes $q$-bounded query complexity
\[
(d\log^{2+o(1)}{q}+d^2\log^2{d},d,1,d,0)
\]
by Proposition \ref{lengthsBoundedProp}. We check whether $L_1=L_2$, taking $O(\log{q})$ bit operations. If not, then $\Gamma_{f_1}\not\cong\Gamma_{f_2}$, so we may output \enquote{false} and halt. Otherwise, we continue by computing, for $j=1,2$,
\begin{itemize}
\item a recursive tree description list $\vec{\Dfrak}^{(j)}=(\Dfrak_n^{(j)})_{n=0,1,\ldots,N_j}$, with
\[
N_j\in O(d2^{d^2\mpe(q-1)+d})
\]
and associated rooted tree isomorphism type list $\vec{\Ifrak}^{(j)}$; and
\item the tree necklace list $\Nfrak_j$ of $f_j$ relative to $\vec{\Ifrak}^{(j)}$ such that
\[
|\Nfrak_j|\in O(dL2^{d^2\mpe(q-1)+dL}).
\]
\end{itemize}
By Theorem \ref{lengthsBoundedTheo}, this can be done within $q$-bounded query complexity
\[
(8^{d^2\mpe(q-1)+dL}2^{d\mpe(q-1)}(d^3L\mpe(q-1)+d^2L^2)\log^{1+o(1)}{q},d,0,0,0).
\]
Next, we compute a synchronization $(\vec{\Dfrak}^+,\ifrak)$ of $\vec{\Dfrak}^{(1)}$ and $\vec{\Dfrak}^{(2)}$, in the sense of Definition \ref{synchronizeDef}. By Lemma \ref{synchronizeLem}, this takes
\begin{align*}
&O((d^38^{d^2\mpe(q-1)+d}+d^24^{d^2\mpe(q-1)+d}\log{q})\cdot\log{q}) \\
=&\,O(d^38^{d^2\mpe(q-1)+d}\log{q}+d^24^{d^2\mpe(q-1)+d}\log^2{q})
\end{align*}
bit operations. Following that, we overwrite each first entry $[\nfrak_1,\nfrak_2,\ldots,\nfrak_{l'}]$ in each triple in $\Nfrak_2$ with $[\ifrak(\nfrak_1),\ifrak(\nfrak_2),\ldots,\ifrak(\nfrak_{l'})]$, using lexicographically minimal representatives of cyclic equivalence classes, which results in a modified, unsorted tree necklace list $\Nfrak'_2$. This takes
\[
O(dL2^{d^2\mpe(q-1)+dL}\cdot L^2\log{L}\log{q})=O(dL^3\log{L}2^{d^2\mpe(q-1)+dL}\log{q})
\]
bit operations overall. After this, we sort $\Nfrak'_2$ lexicographically, which takes
\[
O(dL^2(d^2\mpe(q-1)+dL)2^{d^2\mpe(q-1)+dL}\log{q})
\]
bit operations (see also the end of the proof of Theorem \ref{lengthsBoundedTheo}). Finally, we note that $\Gamma_{f_1}\cong\Gamma_{f_2}$ if and only if $\Nfrak_1=\Nfrak'_2$, so we determine the truth value of the latter, which takes $O(dL^22^{d^2\mpe(q-1)+dL}\log{q})$ bit operations using a linear scan. We then output that truth value and halt.
\end{proof}

Like the previous subsubsections, we conclude this subsubsection with some pseudocode for the discussed algorithms, specifying the $q$-bounded query complexity (QC) of each step. We begin with the algorithm from Theorem \ref{lengthsBoundedTheo}, which on input $(f,L)$, where $f$ is an index $d$ generalized cyclotomic mapping $f$ of $\IF_q$ such that all cycle lengths of $f$ are at most $L$, computes a pair $(\vec{\Dfrak},\Nfrak)$ such that $\vec{\Dfrak}=(\Dfrak_n)_{n=0,1,\ldots,N}$ is a recursive tree description list with $N\in O(d2^{d^2\mpe(q-1)+d})$ and associated sequence of rooted tree isomorphism types $\vec{\Ifrak}$, and $\Nfrak$ is a tree necklace list for $f$ relative to $\vec{\Ifrak}$ with $|\Nfrak|\in O(dL2^{d^2\mpe(q-1)+dL+d})$.

\begin{denumerate}[label=\arabic*]
\item Compute the induced function $\overline{f}$ on $\{0,1,\ldots,d\}$ and the affine maps $A_i$ of $\IZ/s\IZ$.

QC: $(d\log^{1+o(1)}{q},d,0,0,0)$.
\item Compute a partition-tree register
\[
((\Zcal_i)_{i=0,1,\ldots,d-1},((\Dfrak_n,(S_{n,i})_{i=0,1,\ldots,d}))_{n=0,1,\ldots,N})
\]
of $f$ with $N\in O(d2^{d^2\mpe(q-1)+d})$. In the process, store a CRL-list $\overline{\Lcal}$ of $\overline{f}$, the cycles of $\overline{f}$, and the parameter $H_i$ for each $\overline{f}$-periodic $i\in\{0,1,\ldots,d-1\}$.

QC: $(d^3\mpe(q-1)2^{(3d^2+d)\mpe(q-1)+2d}\log^{1+o(1)}{q},0,0,0,0)$ because $\overline{f}$ and the $A_i$ have already been computed.
\item Set $\vec{\Dfrak}:=(\Dfrak_n)_{n=0,1,\ldots,N}$.

QC: $(d^24^{d^2\mpe(q-1)+d}\log{q},0,0,0,0)$.
\item For each $(i,\ell)\in\overline{\Lcal}\setminus\{(d,1)\}$, do the following.

QC: $(8^{d^2\mpe(q-1)+dL}(d^3L\mpe(q-1)+d^2L^2)\log^{1+o(1)}{q},0,0,0,0)$.
\begin{denumerate}[label=4.\arabic*]
\item For each $t=0,1,\ldots,\ell-1$, do the following.

QC: $(d^3\mpe(q-1)\log{q}+d^2L\log^{1+o(1)}{q},0,0,0,0)$.
\begin{denumerate}[label=4.1.\arabic*]
\item From $\Zcal_{i_t}$, read off a spanning congruence sequence for $\Ucal_{i_t}$ of length $H_i\leq d\mpe(q-1)$.

QC: $(d\mpe(q-1)\log{q},0,0,0,0)$.
\item For each $j=0,1,\ldots,H_i+L-1$, do the following.

QC: $(d^2\mpe(q-1)\log{q}+dL\log^{1+o(1)}{q},0,0,0,0)$.
\begin{denumerate}[label=4.1.2.\arabic*]
\item If $j\leq H_i$, then do the following.
\begin{denumerate}[label=4.1.2.1.\arabic*]
\item From $\Zcal_{i_{t+j}}$, read off a spanning congruence sequence for $\lambda_{i_t}^j(\Rcal_{i_t})$ of length $n_{i_t}\leq d$.

QC: $(d\log{q},0,0,0,0)$.
\end{denumerate}
\item Else do the following.
\begin{denumerate}[label=4.1.2.2.\arabic*]
\item Compute a spanning congruence sequence for
\[
\lambda_{i_t}^j(\Rcal_{i_t})=\lambda(\lambda_{i_t}^{j-1}(\Rcal_{i_t}),A_{i_{t+j-1}})
\]
of length $n_{i_t}\leq d$.

QC: $(d\log^{1+o(1)}{q},0,0,0,0)$.
\end{denumerate}
\end{denumerate}
\end{denumerate}
\item For each $t=0,1,\ldots,L-1$, do the following.

QC: $((d^2L\mpe(q-1)+dL^2)\log{q},0,0,0,0)$.
\begin{denumerate}[label=4.2.\arabic*]
\item From the data stored in Step 4.1, paste together a spanning congruence sequence for
\[
\Qcal_{i_{-t},H_i+L-t-1}=\bigwedge_{j=0}^{H_i+L-t-1}{\lambda_{i_{-t-j}}^j(\Rcal_{i_{-t-j}})}\wedge\Ucal_{i_{-t}}.
\]
QC: $((d^2\mpe(q-1)+dL)\log{q},0,0,0,0)$.
\end{denumerate}
\item For each $t=0,1,\ldots,L-1$, do the following.

QC: $(4^{d^2\mpe(q-1)+dL}(d^2\mpe(q-1)+dL)\log^{1+o(1)}{q},0,0,0,0)$.
\begin{denumerate}[label=4.3.\arabic*]
\item Set $\Ocal_{i_{-t},L-t}:=\emptyset$.

QC: $(\log{d}+\log{L},0,0,0,0)$.
\item For each $\diamond_{k=0}^{H_i+L-t-1}{\vec{o'_k}}\in\{\emptyset,\neg\}^{n_{i_{-t}}+n_{i_{-t-1}}+\cdots+n_{i_{-H_i-L+1}}}$, do the following.

QC: $(4^{d^2\mpe(q-1)+dL}2^{-dt}(d^2\mpe(q-1)+dL)\log^{1+o(1)}{q},0,0,0,0)$.
\begin{denumerate}[label=4.3.2.\arabic*]
\item Check whether
\[
\sigma_{\Qcal_{i_{-t},H_i+L-t-1},\mathbf{0}}(\diamond_{k=0}^{H_i+L-t-1}{\vec{o'_k}}\diamond\vec{\xi}_{i_{-t},H_i},(\emptyset,\ldots,\emptyset))>0,
\]
and if so, add $\diamond_{k=0}^{H_i+L-t-1}{\vec{o'_k}}$ to $\Ocal_{i_{-t},L-t}$ as a new element.

QC: $(2^{d^2\mpe(q-1)+dL}(d^2\mpe(q-1)+dL)\log^{1+o(1)}{q},0,0,0,0)$.
\end{denumerate}
\end{denumerate}
\item For each $t=0,1,\ldots,L-2$, do the following.

QC: $(8^{d^2\mpe(q-1)+dL}(d^2\mpe(q-1)+dL)\log^{1+o(1)}{q},0,0,0,0)$.
\begin{denumerate}[label=4.4.\arabic*]
\item For each $\diamond_{k=0}^{H_i+L-t-1}{\vec{o'_k}}\in\Ocal_{i_{-t},L-t}$, do the following.

QC: $(8^{d^2\mpe(q-1)+dL}2^{-d(t+1)}2^{-dt}(d^2\mpe(q-1)+dL)\log^{1+o(1)}{q},0,0,0,0)$.
\begin{denumerate}[label=4.4.1.\arabic*]
\item For each $\diamond_{k=0}^{H_i+L-t-2}{\vec{o_k}}\in\Ocal_{i_{-t-1},L-t-1}$, do the following.

QC: $(4^{d^2\mpe(q-1)+dL}2^{-d(t+1)}(d^2\mpe(q-1)+dL)\log^{1+o(1)}{q},0,0,0,0)$.
\begin{denumerate}[label=4.4.1.1.\arabic*]
\item Check whether the distribution number
\begin{align*}
\sigma_{\Qcal_{i_{-t-1},H_i+L-t-2},A_{i_{-t-1}}}(&\diamond_{k=0}^{H_i+L-t-2}{\vec{o_k}}\diamond\vec{\xi}_{i_{-t-1},H_i}, \\
&\diamond_{k=1}^{H_i+L-t-1}{\vec{o'_k}}\diamond\vec{\xi}_{i_{-t},H_i})
\end{align*}
is equal to $1$. If so, set
\[
\ufrak_{i_{-t},L-t}(\diamond_{k=0}^{H_i+L-t-1}{\vec{o'_k}}):=\diamond_{k=0}^{H_i+L-t-2}{\vec{o_k}}.
\]
QC: $(2^{d^2\mpe(q-1)+dL}(d^2\mpe(q-1)+dL)\log^{1+o(1)}{q},0,0,0,0)$.
\end{denumerate}
\end{denumerate}
\end{denumerate}
\item Compute $\Acal_i:=A_{i_0}A_{i_1}\cdots A_{i_{\ell-1}}$.

QC: $(d\log^{1+o(1)}{q},0,0,0,0)$.
\item Set $\Cfrak_{i,L}:=\emptyset$.

QC: $(\log{d}+\log{L},0,0,0,0)$.
\item For each $l\in\{1,2,\ldots,L\}$, do the following.

QC: $(L\log^{1+o(1)}{q},0,0,0,0)$.
\begin{denumerate}[label=4.7.\arabic*]
\item Check whether $\ell\mid l$. If not, skip to the next $l$.

QC: $(\log^{1+o(1)}{q},0,0,0,0)$.
\item Compute and store the affine iterate
\[
\Acal_{i,l}=\Acal_i^{l/\ell}=
\begin{cases}
\Acal_i, & \text{if }l=\ell, \\
\Acal_{i,l-\ell}\Acal_i, & \text{if }\ell\mid l>\ell,
\end{cases}
\]
with linear coefficient $\overline{\alpha}_{i,l}$ and constant coefficient $\overline{\beta}_{i,l}$.

QC: $(\log{q},0,0,0,0)$ if $l=\ell$ (only needing to copy information from Step 4.5); $(\log^{1+o(q)}{q},0,0,0,0)$ otherwise.
\item Check whether $\gcd(s,\overline{\alpha}_{i,l}-1)\mid\overline{\beta}_{i,l}$. If so, add $l$ to $\Cfrak_{i,L}$ as a new element, and store $\eta_{i,l}(x)$ as the $s$-congruence
\begin{align*}
&x\equiv-\frac{\overline{\beta}_{i,l}}{\gcd(s,\overline{\alpha}_{i,l}-1)}\cdot\inv_{\frac{s}{\gcd(s,\overline{\alpha}_{i,l}-1)}}\left(\frac{\overline{\alpha}_{i,l}-1}{\gcd(s,\overline{\alpha}_{i,l}-1)}\right) \\
&\Mod{\frac{s}{\gcd(s,\overline{\alpha}_{i,l}-1)}}.
\end{align*}
QC: $(\log^{1+o(1)}{q},0,0,0,0)$.
\end{denumerate}
\item Set $\Nfrak_i:=\emptyset$.

QC: $(\log{d},0,0,0,0)$.
\item For each $\diamond_{k=0}^{H_i+L-1}{\vec{o'_k}}\in\Ocal_{i,L}$, do the following.

QC: $(4^{d^2\mpe(q-1)+dL}2^L(d^2L\mpe(q-1)+dL^2)\log^{1+o(1)}{q},0,0,0,0)$.
\begin{denumerate}[label=4.9.\arabic*]
\item For each $l\in\Cfrak_{i,L}$, do the following.

QC: $(2^{d^2\mpe(q-1)+dL+L}(d^2L\mpe(q-1)+dL^2)\log^{1+o(1)}{q},0,0,0,0)$.
\begin{denumerate}[label=4.9.1.\arabic*]
\item Compute the logical sign tuple $\vec{\nu}_{i,L,l}$ (see the paragraph before Proposition \ref{refinedExplicitProp}).

QC: $(L\log^{1+o(1)}{q},0,0,0,0)$.
\item Set
\[
\mfrak_i(\diamond_{k=0}^{H_i+L-1}{\vec{o'_k}}\diamond\vec{\nu}_{i,L,l}):=\sigma_{\Wcal_{i,L},\mathbf{0}}(\diamond_{k=0}^{H_i+L-1}{\vec{o'_k}}\diamond\vec{\xi}_{i,H_i}\diamond\vec{\nu}_{i,L,l},(\emptyset,\ldots,\emptyset)).
\]
QC: $(2^{d^2\mpe(q-1)+dL+L}(d^2\mpe(q-1)+dL)\log^{1+o(1)}{q},0,0,0,0)$.
\item If $\mfrak_i(\diamond_{k=0}^{H_i+L-1}{\vec{o'_k}}\diamond\vec{\nu}_{i,L,l})=0$, then skip to the next $l$.

QC: $(d^2\mpe(q-1)+dL,0,0,0,0)$.
\item For each $t=0,1,\ldots,l-1$, do the following.

QC: $(dL^3+d^2L\mpe(q-1)2^{d^2\mpe(q-1)+d}+L\log{q},0,0,0,0)$.
\begin{denumerate}[label=4.9.1.4.\arabic*]
\item Compute the logical sign tuple
\[
(\proj_{i_{-t},L-t}\circ\ufrak_{i_{-t+1},L-t+1}\circ\ufrak_{i_{-t+2},L-t+2}\circ\cdots\circ\ufrak_{i_0,L})(\diamond_{k=0}^{H_i+L-1}{\vec{o'_k}}).
\]
QC: $(d^2L\mpe(q-1)+dL^2,0,0,0,0)$.
\item For each $n=0,1,\ldots,N$, do the following.

QC: $(d^2\mpe(q-1)2^{d^2\mpe(q-1)+d}+\log{q},0,0,0,0)$.
\begin{denumerate}[label=4.9.1.4.2.\arabic*]
\item If the logical sign tuple computed in Step 4.9.1.4.1 is an element of $S_{n,i_{-t},H_i}$ (the last entry of the tuple $S_{n,i_{-t}}$ from the partition-tree register computed in Step 2), then set $\nfrak^{(-t)}:=n$ and skip to the next $t$.

QC: $(\max\{1,|S_{n,i_{-t},H_i}|d^2\mpe(q-1)\},0,0,0,0)$ if the condition is \emph{not} satisfied; $(\max\{1,|S_{n,i_{-t},H_i}|d^2\mpe(q-1)\}+\log{q},0,0,0,0)$ if the condition \emph{is} satisfied (the additional $O(\log{q})$ bit operations are from copying the value of $n$; we do not need to process the $O(\log{q})$-bit indices $n$ during the loop over them, because we may jump to a neighboring address in memory in $O(1)$ bit operations).
\end{denumerate}
\end{denumerate}
\item Set $\vec{\nfrak}:=(\nfrak^{(-l+1)},\nfrak^{(-l+2)},\ldots,\nfrak^{(0)})$.

QC: $(L\log{q},0,0,0,0)$.
\item Set $l'$ to be $\minperl(\vec{\nfrak})$.

QC: $(L^{1+o(1)}\log{q},0,0,0,0)$.
\item Set $\vec{\nfrak}'$ to be the lexicographically minimal ordered sequence in the same cyclic equivalence class as $(\nfrak^{(-l+1)},\nfrak^{(-l+2)},\ldots,\nfrak^{(-l+l')})$.

QC: $(L^2\log{L}\log{q},0,0,0,0)$.
\item Add
\[
(\diamond_{k=0}^{H_i+L-1}{\vec{o'_k}}\diamond\vec{\nu}_{i,L,l},[\vec{\nfrak}'],\mfrak_i(\diamond_{k=0}^{H_i+L-1}{\vec{o'_k}}\diamond\vec{\nu}_{i,L,l}))
\]
to $\Nfrak_i$ as a new element.

QC: $(d^2\mpe(q-1)+dL+L\log{q},0,0,0,0)$.
\end{denumerate}
\end{denumerate}
\end{denumerate}
\item Set $\Nfrak:=\emptyset$.

QC: $(1,0,0,0,0)$.
\item For each $(i,\ell)\in\overline{\Lcal}\setminus\{(d,1)\}$, do the following.

QC: $(d^4L^4\mpe(q-1)4^{d^2\mpe(q-1)+dL}+d^2L^24^{d^2\mpe(q-1)+dL}\log{q},0,0,0,0)$.
\begin{denumerate}[label=6.\arabic*]
\item For each $(\diamond_{k=0}^{H_i+L-1}{\vec{o'_k}}\diamond\vec{\nu}_{i,L,l},[\vec{\nfrak}],\mfrak')\in\Nfrak_i$, do the following.

QC: $(d^3L^4\mpe(q-1)4^{d^2\mpe(q-1)+dL}+dL^24^{d^2\mpe(q-1)+dL}\log{q},0,0,0,0)$.
\begin{denumerate}[label=6.1.\arabic*]
\item From the first entry, $\diamond_{k=0}^{H_i+L-1}{\vec{o'_k}}\diamond\vec{\nu}_{i,L,l}$, determine the binary representation of $l$.

QC: $(L\log{L},0,0,0,0)$.
\item Check whether $[\vec{\nfrak}]$ and $l$ already occur as the first two entries of some element of $\Nfrak$. If so, skip to the next element of $\Nfrak_i$.

QC: $(dL^22^{d^2\mpe(q-1)+dL}\log{q},0,0,0,0)$.
\item Compute the sum $\mfrak$ of the third entries of all triples in $\bigcup_{(j,\ell')\in\overline{\Lcal}\setminus\{(d,1)\}}{\Nfrak_j}$ that have the parameters $l$ and $[\vec{\nfrak}]$ associated with them.

QC: $(dL^22^{d^2\mpe(q-1)+dL}\log{q},0,0,0,0)$.
\item Add $([\vec{\nfrak}],l,\mfrak)$ to $\Nfrak$ as a new element.

QC: $(L\log{q},0,0,0,0)$.
\end{denumerate}
\end{denumerate}
\item For each $n\in\{0,1,\ldots,N\}$, do the following.

QC: $(d2^{d^2\mpe(q-1)+d}+\log{q},0,0,0,0)$.
\begin{denumerate}[label=7.\arabic*]
\item Check whether $S_{n,d}=\emptyset$. If so, set $\nfrak:=n$ and exit the loop.

QC: $(\log{q},0,0,0,0)$ if the condition is satisfied (which happens only once), $(1,0,0,0,0)$ otherwise.
\end{denumerate}
\item Check whether $[\nfrak]$ and $1$ already occur as the first two entries of some (unique) triple in $\Nfrak$, and store this information (the truth value and, if applicable, the position of that triple in $\Nfrak$).

QC: $(dL2^{d^2\mpe(q-1)+dL}\log{q},0,0,0,0)$.
\item If $[\nfrak]$ and $1$ occur as the first two entries of a triple $([\nfrak],1,\mfrak)$ in $\Nfrak$, then do the following.
\begin{denumerate}[label=9.\arabic*]
\item Overwrite the entry $([\nfrak],1,\mfrak)$ of the list $\Nfrak$ with $([\nfrak],1,\mfrak+1)$.

QC: $(\log{q},0,0,0,0)$.
\end{denumerate}
\item Else do the following.
\begin{denumerate}[label=10.\arabic*]
\item Add $([\nfrak],1,1)$ to $\Nfrak$ as a new element.

QC: $(L\log{q},0,0,0,0)$.
\end{denumerate}
\item Sort $\Nfrak$ lexicographically.

QC: $(dL^2(d^2\mpe(q-1)+dL)2^{d^2\mpe(q-1)+dL}\log{q},0,0,0,0)$.
\item Output $(\vec{\Dfrak},\Nfrak)$ and halt.

QC: $((d^24^{d^2\mpe(q-1)+d}+dL^22^{d^2\mpe(q-1)+dL})\log{q},0,0,0,0)$.
\end{denumerate}

Next, we give pseudocode for the algorithm from Proposition \ref{lengthsBoundedProp}, which for a given index $d$ generalized cyclotomic mapping $f$ of $\IF_q$ outputs the maximum cycle length of $f$.

\begin{denumerate}[label=\arabic*]
\item Compute the induced function $\overline{f}$ on $\{0,1,\ldots,d\}$ and the affine maps $A_i$ of $\IZ/s\IZ$.

QC: $(d\log^{1+o(1)}{q},d,0,0,0)$.
\item Compute a CRL-list $\overline{\Lcal}$ of $\overline{f}$ and the cycles of $\overline{f}$.

QC: $(d^2\log^2{d},0,0,0,0)$.
\item Factor $s=p_1^{v_1}p_2^{v_2}\cdots p_K^{v_K}$.

QC: $(\log{q},0,1,0,0)$.
\item For each $(i,\ell)\in\overline{\Lcal}\setminus\{(d,1)\}$, do the following.

QC: $(d\log^{2+o(1)}{q},0,0,d,0)$.
\begin{denumerate}[label=4.\arabic*]
\item Compute $\Acal_i=A_{i_0}A_{i_1}\cdots A_{i_{\ell-1}}:x\mapsto\overline{\alpha}_ix+\overline{\beta}_i$.

QC: $(\ell\log^{1+o(1)}{q},0,0,0,0)$.
\item Set $s'_i:=1$.

QC: $(\log{d},0,0,0,0)$.
\item For each $j=1,2,\ldots,K$, do the following.

QC: $(\log^{2+o(1)}{q},0,0,0,0)$.
\begin{denumerate}[label=4.3.\arabic*]
\item Check whether $p_j\mid\overline{\alpha}_i$, and if so, skip to the next $j$.

QC: $(\log^{1+o(1)}{q},0,0,0,0)$.
\item Set $s'_i:=s'_i\cdot p_j^{v_j}$.

QC: $(\log^{1+o(1)}{q},0,0,0,0)$.
\end{denumerate}
\item Set $\overline{\alpha}'_i:=\overline{\alpha}_i\bmod{s'_i}$ and $\overline{\beta}'_i:=\overline{\beta}_i\bmod{s'_i}$.

QC: $(\log^{1+o(1)}{q},0,0,0,0)$.
\item Compute $\ord_{s'_i}(\overline{\alpha}'_i)$.

QC: $(\log{q},0,0,1,0)$.
\item Compute
\[
\overline{\beta}''_i:=
\begin{cases}
\overline{\beta}'_i, & \text{if }\overline{\alpha}'_i=1, \\
\overline{\beta}'_i\frac{(\overline{\alpha}'_i)^{\ord_{s'_i}(\overline{\alpha}'_i)}-1}{\overline{\alpha}'_i-1}, & \text{otherwise}.
\end{cases}
\]
QC: $(\log^{2+o(1)}{q},0,0,0,0)$.
\item Compute the additive order of $\overline{\beta}''_i$ modulo $s'_i$, which is equal to $s'_i/\gcd(s'_i,\overline{\beta}''_i)$.

QC: $(\log^{1+o(1)}{q},0,0,0,0)$.
\item Set $l_i:=\ell\cdot\ord_{s'_i}(\overline{\alpha}'_i)\cdot s'_i/\gcd(s'_i,\overline{\beta}''_i)$.

QC: $(\log^{1+o(1)}{q},0,0,0,0)$.
\end{denumerate}
\item Set $L:=1$.

QC: $(\log{d},0,0,0,0)$.
\item For each $(i,\ell)\in\overline{\Lcal}\setminus\{(d,1)\}$, do the following.

QC: $(d\log{q},0,0,0,0)$.
\begin{denumerate}[label=6.\arabic*]
\item If $l_i>L$, then set $L:=l_i$.

QC: $(\log{q},0,0,0,0)$.
\end{denumerate}
\item Output $L$ and halt.

QC: $(\log{q},0,0,0,0)$.
\end{denumerate}

Finally, we give pseudocode for a variant of the algorithm from Corollary \ref{lengthsBoundedCor}. On input $(L,f_1,f_2)$ where $L\in\IN^+$ and $f_j$, for $j=1,2$, is a generalized cyclotomic mapping of $\IF_q$ of index $d_j$, this algorithm outputs \enquote{fail} if neither the largest cycle length of $f_1$ nor the largest cycle length of $f_2$ is at most $L$. Otherwise, it outputs the truth value of the digraph isomorphism relation $\Gamma_{f_1}\cong\Gamma_{f_2}$.

\begin{denumerate}[label=\arabic*]
\item For $j=1,2$, compute the largest cycle length $L_j$ of $f_j$.

QC: $(d\log^{2+o(1)}{q}+d^2\log^2{d},d,1,d,0)$.
\item Check whether $\min\{L_1,L_2\}\leq L$, and store this information.

QC: $(\log{q},0,0,0,0)$.
\item If $\min\{L_1,L_2\}>L$, then do the following.
\begin{denumerate}[label=3.\arabic*]
\item Output \enquote{fail} and halt.

QC: $(1,0,0,0,0)$.
\end{denumerate}
\item Else do the following.
\begin{denumerate}[label=4.\arabic*]
\item Check whether $L_1=L_2$. If not, output \enquote{false} and halt.

QC: $(\log{q},0,0,0,0)$.
\end{denumerate}
\item For $j=1,2$, compute
\begin{itemize}
\item a recursive tree description list $\vec{\Dfrak}^{(j)}=(\Dfrak^{(j)}_n)_{n=0,1,\ldots,N_j}$ with
\[
N_j\in O(d2^{d^2\mpe(q-1)+d})
\]
and associated rooted tree isomorphism type sequence $\vec{\Ifrak}^{(j)}$; and
\item the tree necklace list $\Nfrak_j$ of $f_j$ relative to $\vec{\Ifrak}^{(j)}$, with
\[
|\Nfrak_j|\leq O(dL2^{d^2\mpe(q-1)+dL}).
\]
\end{itemize}
QC: $(8^{d^2\mpe(q-1)+dL}2^{d\mpe(q-1)}(d^3L\mpe(q-1)+d^2L^2)\log^{1+o(1)}{q},d,0,0,0)$.
\item Compute a synchronization $(\vec{\Dfrak}^+,\ifrak)$ of $\vec{\Dfrak}^{(1)}$ and $\vec{\Dfrak}^{(2)}$.

QC: $(d^38^{d^2\mpe(q-1)+d}\log{q}+d^24^{d^2\mpe(q-1)}\log^2{q},0,0,0,0)$.
\item Set $\Nfrak'_2:=\emptyset$.

QC: $(1,0,0,0,0)$.
\item For each $([\nfrak_1,\nfrak_2,\ldots,\nfrak_{l'}],l,\mfrak)\in\Nfrak_2$, do the following.

QC: $(dL^3\log{L}2^{d^2\mpe(q-1)+dL}\log{q},0,0,0,0)$.
\begin{denumerate}[label=8.\arabic*]
\item Compute the lexicographically minimal representative of the cyclic equivalence class $[\nfrak_1,\nfrak_2,\ldots,\nfrak_{l'}]$, representing that class by it.

QC: $(L^2\log{L}\log{q},0,0,0,0)$.
\item Add $([\ifrak(\nfrak_1),\ifrak(\nfrak_2),\ldots,\ifrak(\nfrak_{l'})],l,\mfrak)$ to $\Nfrak'_2$ as a new element.

QC: $(L\log{q},0,0,0,0)$.
\end{denumerate}
\item Sort $\Nfrak'_2$ lexicographically.

QC: $(dL^2(d^2\mpe(q-1)+dL)2^{d^2\mpe(q-1)+dL}\log{q},0,0,0,0)$.
\item Check whether $\Nfrak_1=\Nfrak'_2$ as sets, output the truth value of this equality, and halt.

QC: $(dL^22^{d^2\mpe(q-1)+dL}\log{q},0,0,0,0)$.
\end{denumerate}

\section{Open problems}\label{sec6}

We conclude this paper with a discussion of open problems related to our results and methods.

\subsection{Asymptotic behavior of \texorpdfstring{$\mpe$}{mpe} and \texorpdfstring{$\tau$}{tau} over prime powers}\label{subsec6P1}

Our Proposition \ref{mpeAvProp} states that as $q$ ranges over an initial segment of all prime powers, the average value of $\mpe(q-1)$ (the maximum exponent of a prime in the full factorization of $q-1$) is bounded from above by a constant (independent of that segment). This led to the important observation that when $d$ is fixed, then for asymptotically almost all finite fields $\IF_q$, the complexities in Theorem \ref{complexitiesTheo}(2,3) are polynomial in $\log{q}$. Following that, at the end of Subsection \ref{subsec5P1}, we raised the analogous problem restricted to powers of $2$, and our computational evidence (gathered in the form of Table \ref{mpeMersenneTable}) leads to the following conjecture.

\begin{connjecture}\label{mpeConj}
The average value of $\mpe(2^k-1)$, where $k$ ranges over an initial segment of $\IN^+$, is always less than $2$. Formally, this conjecture asserts that for each $K\in\IN^+$, one has
\[
\frac{1}{K}\sum_{k=1}^K{\mpe(2^k-1)}<2.
\]
\end{connjecture}

In fact, looking at Table \ref{mpeMersenneTable}, one might even conjecture that the said average value is always less than $3/2$, which would imply that $\mpe(2^k-1)=1$ (i.e., that $2^k-1$ is square-free) for more than half of all $k\in\IN^+$.

The complexity bounds in Subsubsection \ref{subsubsec5P3P2} also involved $\tau(q-1)$, the number of divisors of $q-1$, and with our Proposition \ref{tauAvProp}, we were able to show that for all but an asymptotic fraction of less than $\epsilon$ of all prime powers $q$, one has $\tau(q-1)<\log^{c''_{\epsilon}}{q}$, where $c''_{\epsilon}$ is a suitable constant depending on $\epsilon$. This implies that for fixed $d$ and all such prime powers $q$, the complexity bounds from Subsubsection \ref{subsubsec5P3P2} are polynomial in $\log{q}$ of a degree depending on $\epsilon$. In the comments leading to Proposition \ref{tauAvProp}, we mentioned Dirichlet's result \cite[Theorem 3.3]{Apo76a}, which implies that as $x\to\infty$, one has
\[
\frac{1}{x}\sum_{n\leq x}{\tau(n)}\sim\log{x},
\]
where the summation index $n$ ranges over arbitrary positive integers (not just numbers of the form $q-1$ where $q$ is a prime power). In view of this result, we pose the following problem concerning a potential strengthening of Proposition \ref{tauAvProp}.

\begin{problemm}\label{tauProb}
Prove or disprove that as $x\to\infty$, one has
\[
\left(\sum_{q\leq x}{1}\right)^{-1}\sum_{q\leq x}{\tau(q-1)}\in O(\log{x}),
\]
where $q$ ranges over prime powers. Should this turn out to be false, is it at least the case that for some absolute constant $c>0$, one has
\[
\left(\sum_{q\leq x}{1}\right)^{-1}\sum_{q\leq x}{\tau(q-1)}\in O(\log^c{x})?
\]
\end{problemm}

We back Problem \ref{tauProb} up with the following table containing some computational evidence obtained with GAP \cite{GAP4}.

\begin{longtable}[h]{|c|c|}\hline
$x$ & $(\sum_{q\leq x}{1})^{-1}\sum_{q\leq x}{\tau(q-1)}/\log{x}$ rounded \\ \hline
$10^3$ & $1.56316$ \\ \hline
$10^4$ & $1.70011$ \\ \hline
$10^5$ & $1.74922$ \\ \hline
$10^6$ & $1.78351$ \\ \hline
$10^7$ & $1.8069$ \\ \hline
$10^8$ & $1.82506$ \\ \hline
\caption{Average value of $\tau(q-1)$ vs natural logarithm.}
\label{tauVsLogTable}
\end{longtable}

\subsection{Efficient comparison of arithmetic partitions}\label{subsec6P2}

We recall Problem 3 from the beginning of Section \ref{sec5}: given a generalized cyclotomic mapping $f$ of $\IF_q$ and a pair $(r,l)$ where $r\in\IF_q$ is $f$-periodic of cycle length $l$, the task is to obtain a compact description of the digraph isomorphism type of the connected component of $\Gamma_f$ containing $r$, viewed as a necklace of rooted tree isomorphism types. In Problem 3, it is also assumed that a partition-tree register for $f$ (in the sense of Definition \ref{partTreeRegDef}) is given, and we can use this to refer to the rooted trees with their numbers $n\in\{0,1,\ldots,N\}$ in this register, rather than spell each of them out completely.

Assuming that $r\not=0_{\IF_q}$ (the case \enquote{$r=0_{\IF_q}$} is easily dealt with separately), the main idea behind our algorithm from the proof of Theorem \ref{complexitiesTheo}(3) for tackling this problem is to identify the positions on the $f$-cycle of $r$ that lie in a given coset $C_j$ with the elements of $\IZ/(l/\ell)\IZ$, where $\ell$ is the associated coset cycle length (i.e., the $\overline{f}$-cycle length of $i$ for the unique $i\in\{0,1,\ldots,d-1\}$ such that $r\in C_i$), and to derive an arithmetic partition $\Pcal^{(i_t)}$ of $\IZ/(l/\ell)\IZ$ for $t\in\{0,1,\ldots,\ell-1\}$ such that vertices in $C_{i_t}$ on the cycle that lie in the same block of $\Pcal^{(i_t)}$ have the same tree above them in $\Gamma_f$. Formally, we may express this via a labeling function $\lab_{i_t}:\Pcal^{(i_t)}\rightarrow\{0,1,\ldots,N\}$\phantomsection\label{not305} that maps each block of $\Pcal^{(i_t)}$ to its associated rooted tree number.

Now, the description of the connected component of $\Gamma_f$ containing $r$ obtained this way is \emph{not} an injective encoding of its isomorphism type. That is, isomorphic connected components may end up getting different descriptions, and it appears to be a nontrivial computational problem to decide efficiently whether two given descriptions pertain to the same isomorphism type. The purpose of this subsection is to discuss this open problem in more detail, reducing it to some concrete questions to be answered.

Of course, in the actual implementation of our algorithm for Problem 3, each of the arithmetic partitions $\Pcal^{(i_t)}$ mentioned above is expressed through a spanning congruence sequence, of length $\overline{m}_{i_t}$\phantomsection\label{not306} say, and $\lab_{i_t}$ may be expressed through a function $\{\emptyset,\neg\}^{\overline{m}_{i_t}}\rightarrow\{-1,0,1,\ldots,N\}$, where $-1$ is a dummy value to be assigned to a logical sign tuple $\vec{\nu}$ if the associated set $\Bcal(\Pcal^{(i_t)},\vec{\nu})$ is empty. By abuse of notation, we also call this function $\lab_{i_t}$. We give a special name to ordered pairs such as $(\Pcal^{(i_t)},\lab_{i_t})$.

\begin{deffinition}\label{labeledDef}
Let $m,N\in\IN_0$ with $m>0$. An \emph{$N$-labeled arithmetic partition of $\IZ/m\IZ$}\phantomsection\label{term90} is a pair $(\Pfrak(x\equiv \bfrak_j\Mod{\afrak_j}:j=1,2,\ldots,K),\lab)$ consisting of an arithmetic partition of $\IZ/m\IZ$ with a fixed spanning congruence sequence and a so-called labeling function $\lab:\{\emptyset,\neg\}^K\rightarrow\{-1,0,1,\ldots,N\}$\phantomsection\label{not307} such that $\lab(\vec{\nu})=-1$ if and only if $\Bcal(\Pcal,\vec{\nu})=\emptyset$.
\end{deffinition}

In order to characterize when two compact descriptions obtained by the algorithm for Problem 3 represent isomorphic connected components, let us first consider the special case where $\ell=1$. Then we only need to worry about the coset $C_i$ and the associated $N$-labeled arithmetic partition $(\Pcal^{(i)},\lab_i)$ of $\IZ/l\IZ$.

We need to understand how applying a cyclic shift to the associated sequence of rooted tree isomorphism types affects the $N$-labeled arithmetic partition. Let $\tfrak$\phantomsection\label{not308} be the translation $x\mapsto x+1$ of $\IZ/l\IZ$. It generates a cyclic subgroup $\Tfrak$\phantomsection\label{not309} of order $l$ of $\Sym(\IZ/l\IZ)$, namely the image of the regular representation of $\IZ/l\IZ$ on itself. This group $\Tfrak$ also acts naturally on the power set of $\IZ/l\IZ$ via $M^t:=t(M)=\{y^t: y\in M\}$. In the same manner, this leads to an action of $\Tfrak$ on the power set of the power set of $\IZ/l\IZ$ (i.e., an action which transforms families of subsets of $\IZ/l\IZ$ into other such families), and this action restricts to one on the set of all arithmetic partitions of $\IZ/l\IZ$. Indeed, if $\Pcal=\Pfrak(x\equiv\bfrak_j\Mod{\afrak_j}:j=1,2,\ldots,K)$ is an arithmetic partition of $\IZ/l\IZ$, then $\Pcal^{\tfrak}$, the partition of $\IZ/l\IZ$ obtained by shifting all blocks of $\Pcal$ to the right by one unit, is just the arithmetic partition $\Pfrak(x\equiv\bfrak_j+1\Mod{\afrak_j}:j=1,2,\ldots,K)$ of $\IZ/l\IZ$.

If we assume that $\Pcal$ is $N$-labeled and that each block of $\Pcal^t$, where $t\in\Tfrak$, carries the same label as the block of $\Pcal$ it is shifted from, we finally get an action of $\Tfrak$ on the set of $N$-labeled arithmetic partitions of $\IZ/l\IZ$, which is useful for our characterization. Formally, this action is defined via
\begin{align*}
&(\Pfrak(x\equiv\bfrak_j\Mod{\afrak_j}: j=1,2,\ldots,K),\lab)^{\tfrak^n}:= \\
&(\Pfrak(x\equiv\bfrak_j+n\Mod{\afrak_j}: j=1,2,\ldots,K),\lab).
\end{align*}
In the special case \enquote{$\ell=1$} we are currently discussing, applying a right cyclic shift of $n$ units to the cyclic sequence of rooted tree isomorphism types associated with the $N$-labeled arithmetic partition $(\Pcal^{(i)},\lab_i)$ corresponds to replacing $(\Pcal^{(i)},\lab_i)$ by $(\Pcal^{(i)},\lab_i)^{\tfrak^n}$. This motivates the following definition.

\begin{deffinition}\label{labeledEquivDef}
Let $m,N\in\IN_0$ with $m>0$, and let $(\Pcal,\lab)$ and $(\Pcal',\lab')$ be $N$-labeled arithmetic partitions of $\IZ/m\IZ$. These partitions are \emph{equivalent}\phantomsection\label{term91} if there is a $t\in\Tfrak$ such that $(\Pcal',\lab')^t=(\Pcal,\lab)$. In that case, the smallest positive integer $n$ such that $(\Pcal',\lab')^{\tfrak^n}=(\Pcal,\lab)$ is the \emph{translation number of $(\Pcal,\lab)$ and $(\Pcal',\lab')$}\phantomsection\label{term92}. If $(\Pcal,\lab)$ and $(\Pcal',\lab')$ are not equivalent, then we define their translation number to be $\infty$.
\end{deffinition}

We observe that translation numbers are not symmetric in their two arguments. Rather, if $(\Pcal,\lab)$ and $(\Pcal',\lab')$ are equivalent, then the translation number of $(\Pcal',\lab')$ and $(\Pcal,\lab)$ is the difference of the stabilizer order
\[
|\Stab_{\Tfrak}((\Pcal,\lab))|=|\{t\in\Tfrak: (\Pcal,\lab)^t=(\Pcal,\lab)\}|=|\Stab_{\Tfrak}((\Pcal',\lab'))|
\]
and\phantomsection\label{not310} the translation number of $(\Pcal,\lab)$ and $(\Pcal',\lab')$. For example, if $(\Pcal,\lab)^{\tfrak^3}=(\Pcal,\lab)$ and $(\Pcal',\lab')^{\tfrak}=(\Pcal,\lab)$, then $(\Pcal,\lab)^{\tfrak^2}=(\Pcal',\lab')$.

As we mentioned just before Definition \ref{labeledEquivDef}, in case $\ell=1$, two $N$-labeled arithmetic partitions of the form $(\Pcal^{(i)},\lab_i)$ and $(\Pcal^{(j)},\lab_j)$ correspond to isomorphic connected components of $\Gamma_f$ if and only if they are equivalent. We thus pose the following algorithmic problem.

\begin{problemm}\label{labeledProb1}
Find an efficient algorithm which, for given $m,N\in\IN_0$ with $m>0$ and $N$-labeled arithmetic partitions $(\Pcal,\lab)$ and $(\Pcal',\lab')$ of $\IZ/m\IZ$, computes the translation number of $(\Pcal,\lab)$ and $(\Pcal',\lab')$.
\end{problemm}

We note that by our convention on translation numbers of inequivalent $N$-labeled arithmetic partitions, such an algorithm could in particular be used to efficiently decide whether $(\Pcal,\lab)$ and $(\Pcal',\lab')$ are equivalent in the first place. Moreover, it could be used to determine the stabilizer in $\Tfrak$ of a given $N$-labeled arithmetic partition $(\Pcal,\lab)$, because that stabilizer is generated by $\tfrak^n$ where $n$ is the translation number of $(\Pcal,\lab)$ with itself.

We now return from \enquote{$\ell=1$} to the general case. By the details of our identification of the $l$ positions on the $f$-cycle of $r$ with the elements in $\ell$ disjoint copies of $\IZ/(l/\ell)\IZ$ (see Subsubsection \ref{subsubsec5P2P3}), it is not hard to see that in general, applying a right cyclic shift by $n$ units to the sequence of rooted tree isomorphism types associated with the sequence $((\Pcal^{(i_t)},\lab_{i_t}))_{t=0,1,\ldots,\ell-1}$ corresponds to replacing that sequence with
\begin{align*}
\notag &(((\Pcal^{(i_j)},\lab_{i_j}))_{j=0,1,\ldots,\ell-1})^{\tfrak^n}:= \\
&((\Pcal^{(i_j)},\lab_{i_j})^{\tfrak^{n''+1}})_{j=\ell-n',\ell-n'+1,\ldots,\ell-1}\diamond((\Pcal^{(i_j)},\lab_{i_j})^{\tfrak^{n''}})_{j=0,1,\ldots,\ell-n'-1}
\end{align*}
where $n':=n\bmod{\ell}$ and $n'':=(n-n')/\ell$. This defines an action of $\Tfrak$ on the set of all length $\ell$ sequences of $N$-labeled arithmetic partitions of $\IZ/(l/\ell)\IZ$, and as for $\ell=1$, we call two such sequences \emph{equivalent}\phantomsection\label{term93} if they can be mapped to each other under this action. In order to decide in general whether two descriptions produced by our algorithm for Problem 3 correspond to isomorphic connected components, we need to decide whether these descriptions are equivalent in this more general sense. However, it turns out that this can be done efficiently if we have an algorithm as in Problem \ref{labeledProb1}. Let us explain why.

We assume that $(\Xcal_k,\lab^{(k)}))_{k=0,1,\ldots,\ell-1}$ and $(\Ycal_k,\Lab^{(k)}))_{k=0,1,\ldots,\ell-1}$\phantomsection\label{not311}\phantomsection\label{not312} are length $\ell$ sequences of $N$-labeled arithmetic partitions of $\IZ/m\IZ$ for some $m\in\IN^+$. They are equivalent if and only if there are $n'\in\{0,1,\ldots,\ell-1\}$ and $n''\in\IZ$ such that $((\Ycal_k,\Lab^{(k)}))_{k=0,1,\ldots,\ell-1}$ is equal to
\begin{equation}\label{generalShiftEq}
((\Xcal_t,\lab^{(t)})^{\tfrak^{n''+1}})_{t=\ell-n',\ell-n'+1,\ldots,\ell-1}\diamond((\Xcal_t,\lab^{(t)})^{\tfrak^{n''}})_{t=0,1,\ldots,\ell-n'-1}.
\end{equation}
To check whether this is the case, we assume that $n'$ is fixed (in the worst case, we need to try out $\ell\in O(d)$ values for $n'$). For $t=0,1,\ldots,\ell-1$, we compute the number
\[
\bfrak_t:=
\begin{cases}
\text{translation number of }(\Ycal_t,\Lab^{(t)})\text{ and }(\Xcal_{\ell-n'+t},\lab^{(\ell-n'+t)}), & \text{if }t<n', \\
\text{translation number of }(\Ycal_t,\Lab^{(t)})\text{ and }(\Xcal_{t-n'},\lab^{(t-n')}), & \text{otherwise}.
\end{cases}
\]
using the algorithm from Problem \ref{labeledProb1}. If any of these numbers is $\infty$, then the chosen value of $n'$ does not work. Otherwise, we compute $\afrak_t$, the group order of the stabilizer of $(\Ycal_t,\Lab^{(t)})$ in $\Tfrak$ (which here is a cyclic group of order $m$), using the said algorithm. The question is whether there exists $n''\in\IZ$ such that
\begin{align*}
n''+1&\equiv\bfrak_0\Mod{\afrak_0} \\
n''+1&\equiv\bfrak_1\Mod{\afrak_1} \\
&\vdots \\
n''+1&\equiv\bfrak_{n'-1}\Mod{\afrak_{n'-1}} \\
n''&\equiv\bfrak_{n'}\Mod{\afrak_{n'}} \\
n''&\equiv\bfrak_{n'+1}\Mod{\afrak_{n'+1}} \\
&\vdots \\
n''&\equiv\bfrak_{\ell-1}\Mod{\afrak_{\ell-1}}
\end{align*}
because these congruences characterize when $((\Ycal_t,\Lab^{(t)}))_{t=0,1,\ldots,\ell-1}$ is equal to (\ref{generalShiftEq}). Viewing this as a system of $m$-congruences in the single variable $n''$, the existence of $n''$ can easily be decided using Proposition \ref{sCongProp}.

We conclude this subsection by noting that the algorithm from Problem \ref{labeledProb1} can be used to decide whether two ($N$-labeled) arithmetic partitions are equal, i.e., have the same (labeled) blocks. This is because two $N$-labeled arithmetic partitions $(\Pcal,\lab)$ and $(\Pcal',\lab')$ of $\IZ/m\IZ$ are equal if and only if the translation number of $(\Pcal,\lab)$ and $(\Pcal',\lab')$ equals the translation number of $(\Pcal,\lab)$ with itself. Moreover, $\Pcal$ and $\Pcal'$ are equal if and only if $(\Pcal,\mathbf{0})$ and $(\Pcal',\mathbf{0})$ are equal, where (in each of the two cases) $\mathbf{0}$ denotes the constantly zero labeling function. Still, the algorithmic problem of verifying whether two given arithmetic partitions of $\IZ/m\IZ$ are equal is interesting in its own right, and it may admit an efficient algorithmic solution even if Problem \ref{labeledProb1} does not, so we pose it separately.

\begin{problemm}\label{labeledProb2}
Find an efficient algorithm which, for given $m\in\IN^+$ and (spanning $m$-congruence sequences of) arithmetic partitions $\Pcal$ and $\Qcal$ of $\IZ/m\IZ$, decides whether $\Pcal=\Qcal$.
\end{problemm}

\subsection{More problems concerning asymptotic growth rates}\label{subsec6P3}

Let $m=p_1^{v_1}\cdots p_K^{v_K}$ be a positive integer with its factorization displayed. We recall from Definition \ref{sArithDef}(3) that the minimal number of spanning $m$-congruences for an arithmetic partition $\Pcal$ of $\IZ/m\IZ$ is called the (arithmetic) complexity of $\Pcal$ and denoted by $\AC(\Pcal)$. In Remark \ref{trivialBoundRem}, we observed that the trivial partition $\Tcal_m$ of $\IZ/m\IZ$, all of whose blocks are singletons, satisfies $\AC(\Tcal_m)\leq\sum_{j=1}^K{p_j^{v_j}}-K$. While this bound is equal to $m-1$ when $m$ is a prime power, we also observed in Remark \ref{trivialBoundRem} that the bound is asymptotically equivalent to $\log^2{m}/(2\log\log{m})$ when $m$ is a primorial, which leads to the question whether the actual complexity of $\Tcal_m$ can be significantly smaller than that (on a suitable infinite class of values for $m$).

\begin{quesstion}\label{liminfQues}
Is it true that
\[
\liminf_{m\to\infty}{\frac{\AC(\Tcal_m)}{\log^2{m}/\log\log{m}}}>0?
\]
\end{quesstion}

We observe that $\Tcal_m$ is the unique (arithmetic) partition of $\IZ/m\IZ$ that achieves the maximum possible number of blocks, $m$. One may ask more generally for nontrivial bounds that relate the arithmetic complexity of an arithmetic partition $\Pcal$ of $\IZ/m\IZ$ with its number of blocks. Trivially, the number of distinct blocks of $\Pcal$ is at most $2^{\AC(\Pcal)}$, and this bound is attained if $\AC(\Pcal)\in\{0,1\}$. In fact, for any given value $k\in\IN_0$, there is an $m\in\IN^+$ and an arithmetic partition $\Pcal$ of $\IZ/m\IZ$ such that $\AC(\Pcal)=k$ and $\Pcal$ has $2^k$ distinct blocks: simply let $m$ be the $k$-th primorial $p_k\#=p_1p_2\cdots p_k$, where $p_j$ denotes the $j$-th smallest prime number, and $\Pcal:=\Pfrak(x\equiv0\Mod{p_j}:j=1,2,\ldots,k)$. However, once $\AC(\Pcal)$ becomes sufficiently large with respect to $m$, the number of blocks of $\Pcal$ falls behind $2^{\AC(\Pcal)}$; at latest, this happens once $\AC(\Pcal)>\log_2{m}$, because the block count of $\Pcal$ cannot be larger than $m$. In the example we just gave, where $m=p_k\#$, we have $\AC(\Pcal)=k\sim(\log{m}/\log\log{m})$, which motivates the following open problem.

\begin{problemm}\label{gammaBlocksProb}
Either find functions $f,g:\left[0,\infty\right)\rightarrow\left[0,\infty\right)$ such that
\begin{enumerate}
\item $(\log{x}/\log\log{x})\lesssim f(x)\in o(\log{x})$,
\item $g(x)\in o(2^x)$, and
\item for every positive integer $m$ and every arithmetic partition $\Pcal$ of $\IZ/m\IZ$ with $\AC(\Pcal)\geq f(m)$, the number of blocks of $\Pcal$ is at most $g(\AC(\Pcal))$,
\end{enumerate}
or prove that such functions do not exist.
\end{problemm}

In Subsubsection \ref{subsubsec5P3P2}, we took note of a potential obstacle to using tree necklace lists to give a general, efficient algorithm for deciding whether the functional graphs of two given generalized cyclotomic mappings of $\IF_q$ are isomorphic. Namely, it could be that even when their index is fixed, generalized cyclotomic mappings $f$ have too many distinct isomorphism types of connected components in their functional graphs. Specifically, we pose the following problem.

\begin{problemm}\label{isomorphismTypesProb}
Prove or disprove that for every $d\in\IN^+$, there is a constant $c=c(d)$ such that for every prime power $q$ and every index $d$ generalized cyclotomic mapping $f$ of $\IF_q$, the number of distinct isomorphism types of connected components of $\Gamma_f$ is in $O(\log^c{q})$.
\end{problemm}

We observe that for $d=1$, all rooted trees above \emph{non-zero} $f$-periodic points are isomorphic; see Theorem \ref{allTreesIsomorphicTheo}, noting that $\Gamma_{f_{\mid\IF_q^{\ast}}}\cong\Gamma_{A_0}$, unless $f$ is constantly zero, in which case the statement in question is vacuously true. Therefore, for $d=1$, Problem \ref{isomorphismTypesProb} is equivalent to proving or disproving that the number of distinct cycle lengths of an affine map of $\IZ/(q-1)\IZ$, where $q$ ranges over all prime powers, is bounded from above by some fixed polynomial in $\log{q}$. Even this appears to be an open problem, in spite of Remark \ref{ctTauRem}.

In view of Proposition \ref{tauAvProp}, we accept that an arbitrarily small but positive asymptotic fraction of prime powers $q$ needs to be excluded in order for the algorithms from Subsubsection \ref{subsubsec5P3P2} to be efficient. In this context, we note the following problem, which is harder than Problem \ref{isomorphismTypesProb}, but also more interesting.

\begin{problemm}\label{isomorphismTypesProb2}
For $d\in\IN^+$ and $c>0$, we denote by $\epsilon(d,c)$ the asymptotic proportion of all prime powers $q$ with $d\mid q-1$ for which there exists an index $d$ generalized cyclotomic mapping $f$ of $\IF_q$ such that $\Gamma_f$ has more than $\log^{c}{q}$ distinct isomorphism types of connected components. Prove or disprove that as $d$ is fixed and $c\to\infty$, one has $\epsilon(d,c)\to0$.
\end{problemm}

\subsection{Extension to other coset-wise affine functions}\label{subsec6P4}

An index $d$ generalized cyclotomic mapping $f$ of $\IF_q$, given in cyclotomic form (\ref{cyclotomicFormEq}), such that all $a_i$ and $r_i$ are non-zero restricts to a function $\IF_q^{\ast}\rightarrow\IF_q^{\ast}$, which is \enquote{coset-wise affine} in the sense that its restriction to any given coset $C_i$ of the index $d$ subgroup $C$ of $\IF_q^{\ast}$ maps to another coset $C_{\overline{f}(i)}$ via an affine map of the cyclic group $C$ (here, we are using the general, group-theoretic sense of the word \enquote{affine map}, as in Definition \ref{affineMapDef}). That we can split $f$ up into such smaller, easy to handle parts is crucial for the approach of understanding $\Gamma_f$ presented in this paper.

In this subsection, we aim to generalize this idea. More specifically, we replace $\IF_q^{\ast}$ by some group $G$ (usually, but not necessarily finite), and $C$ by a subgroup $H$ of $G$. We consider the following two notions of coset-wise affine functions.

\begin{deffinition}\label{cosetwiseAffineDef}
Let $G$ be a group, and let $H$ be a subgroup of $G$.
\begin{enumerate}
\item An \emph{affine function $H\rightarrow G$}\phantomsection\label{term94} is a function $H\rightarrow G$ of the form $h\mapsto h^{\varphi}b$ for some group homomorphism $\varphi:H\rightarrow G$ and some $b\in G$.
\item A function $f:G\rightarrow G$ is called \emph{$H$-coset-wise affine in the wide sense}\phantomsection\label{term95} if for every right coset $C=Hr_C$ of $H$ in $G$, there is an affine function $A_C:H\rightarrow G$ such that $f(hr_C)=A_C(h)$ for all $h\in H$.
\item A function $f:G\rightarrow G$ is called \emph{$H$-coset-wise affine in the narrow sense}\phantomsection\label{term96} if for every right coset $C=Hr_C$ of $H$ in $G$, there is an affine map $A_C$ of $H$ and a $t_C\in G$ such that $f(hr_C)=A_C(h)t_C$ for all $h\in H$.
\end{enumerate}
\end{deffinition}

In contrast to $H$-coset-wise affine functions in the narrow sense, an $H$-coset-wise affine function in the wide sense does not need to map each right coset of $H$ to a single such coset. This makes it hard to study the behavior of $H$-coset-wise affine functions in the wide sense under iteration. The most celebrated example of this is the Collatz function $g$, corresponding to $G=\IZ$ and $H=2\IZ$ and given by the coset-wise affine formula
\[
g(x)=
\begin{cases}
x/2, & \text{if }x\in2\IZ, \\
3x+1, & \text{if }x\in 2\IZ+1.
\end{cases}
\]
On the other hand, any function $g:\IZ\rightarrow\IZ$ such that $g$ agrees, on each coset $k+n\IZ$ of the index $n$ subgroup $n\IZ$, with an affine function $x\mapsto a^{(k)}x+b^{(k)}$ with \emph{integer} coefficients $a^{(k)}$ and $b^{(k)}$, is $n\IZ$-coset-wise affine in the narrow sense and thus amenable to the ideas mentioned in the first paragraph of this subsection. Henceforth, we restrict our attention to $H$-coset-wise affine functions in the narrow sense, which we simply call \emph{$H$-coset-wise affine functions}\phantomsection\label{term97} for short.

An important special case is when $G=(\IF_q,+)$ and $H$ is an $\IF_p$-subspace of $G$, for which this class of functions was already considered in \cite{BW22a}. We expect our approach for understanding functional graphs of generalized cyclotomic mappings to work mostly analogously for $H$-coset-wise affine functions, with one big caveat: in the proof of Lemma \ref{masterLem} (our \enquote{Master Lemma}), we made essential use of the equivalence of statements (1) and (3) in Proposition \ref{sCongProp}. In the more general group-theoretic context of the current subsection, this equivalence needs to be replaced by the following property of the group $H$.

\begin{deffinition}\label{pairwiseCongConsDef}
Let $H$ be a group. We say that $H$ is \emph{pairwise congruence-consistent}\phantomsection\label{term98} if any given system of congruences over $H$,
\begin{align*}
x&\equiv h_1\Mod{N_1} \\
x&\equiv h_2\Mod{N_2} \\
&\vdots \\
x&\equiv h_K\Mod{N_K}
\end{align*}
where $h_1,h_2,\ldots,h_K\in H$ and $N_1,N_2,\ldots,N_K$ are normal subgroups of $H$, is consistent if and only if each pair of congruences in the system is consistent.
\end{deffinition}

The equivalence of statements (1) and (3) in Proposition \ref{sCongProp} can be reformulated as \enquote{Finite cyclic groups are pairwise congruence-consistent.} In the special case \enquote{$G=(\IF_q,+)$} mentioned above, the group $H$ is of the form $\IF_p^n$, i.e., it is a finite elementary abelian $p$-group. If our approach is to work completely analogously for that case, we would need that finite elementary abelian groups are pairwise congruence-consistent. However, that is not the case, as the following result shows (noting that all abelian groups are nilpotent).

\begin{theoremm}\label{pairwiseCongConsTheo}
Let $G$ be a finite nilpotent group. The following are equivalent:
\begin{enumerate}
\item $G$ is pairwise congruence-consistent.
\item $G$ is cyclic.
\end{enumerate}
\end{theoremm}

We prove Theorem \ref{pairwiseCongConsTheo} at the end of this subsection. Before doing so, we make two more comments.

Firstly, we observe that non-cyclic finite pairwise congruence-consistent groups exist; there are both non-solvable examples, such as any non-abelian finite simple group (for trivial reasons), and solvable examples, such as $\AGL_1(q)=\IF_q\rtimes\IF_q^{\ast}$ for any prime power $q$. To see that the latter kind of groups are pairwise congruence-consistent, we note that $\AGL_1(q)$ has $\IF_q$ as its unique minimal, nontrivial normal subgroup, so any congruence over $\AGL_1(q)$ has an associated congruence over the cyclic group $\IF_q^{\ast}$ (i.e., a congruence in the classical, number-theoretic sense) such that the solution set of the congruence over $\AGL_1(q)$ is the full pre-image, under the canonical projection $\AGL_1(q)\rightarrow\IF_q^{\ast}$, of the solution set of the associated number-theoretic congruence. In particular, if any pair of congruences in a given system of congruences over $\AGL_1(q)$ is consistent, the same holds true for the associated system over $\IF_q^{\ast}$, whence that system is consistent by Proposition \ref{sCongProp}, and so the original system over $\AGL_1(q)$ must also have a solution.

Secondly, we pose the following two open problems, which are motivated by Theorem \ref{pairwiseCongConsTheo} and the discussion leading to it.

\begin{problemm}\label{pairwiseCongConsProb}
Classify the finite groups that are pairwise congruence-consistent.
\end{problemm}

\begin{problemm}\label{elementaryAbelianCongProb}
For important classes of finite groups that are \emph{not} contained in the class of finite pairwise congruence-consistent groups (such as the class of finite (elementary) abelian groups), devise efficient algorithms that decide whether a given system of congruences over a group in that class is consistent.
\end{problemm}

In order to prove Theorem \ref{pairwiseCongConsTheo}, we first consider the following property of groups.

\begin{deffinition}\label{pairwiseIntersectingDef}
Let $G$ be a group. We say that $G$ has the \emph{pairwise coset-intersection property}\phantomsection\label{term99} (or \emph{PCIP}\phantomsection\label{term100} for short) if the following holds: for any positive integer $m$ and any sequence $(C_1,C_2,\ldots,C_m)$ of (left or right) cosets of subgroups of $G$, if $C_j\cap C_k\not=\emptyset$ for all $1\leq j<k\leq m$, then $\bigcap_{j=1}^m{C_j}\not=\emptyset$. A group satisfying the PCIP is also called a \emph{PCIP-group}\phantomsection\label{term101} for short.
\end{deffinition}

The following proposition is immediate from observing that the solution set of the congruence $x\equiv g\Mod{N}$ over the group $G$ is the coset $gN=Ng$ of $N$.

\begin{propposition}\label{pairwiseIntersectingProp}
Let $G$ be a group. The following are equivalent.
\begin{enumerate}
\item $G$ is pairwise congruence-consistent.
\item For any positive integer $m$ and any sequence $(C_1,C_2,\ldots,C_m)$ of cosets of \underline{normal} subgroups of $G$, if $C_j\cap C_k\not=\emptyset$ for all $1\leq j<k\leq m$, then $\bigcap_{j=1}^m{C_j}\not=\emptyset$.
\end{enumerate}
\end{propposition}

Proposition \ref{pairwiseIntersectingProp} has two important consequences.

\begin{corrollary}\label{pairwiseIntersectingCor}
The following hold.
\begin{enumerate}
\item Every PCIP-group is pairwise congruence-consistent.
\item An abelian group satisfies the PCIP if and only if it is pairwise congruence-consistent.
\end{enumerate}
\end{corrollary}

In view of Corollary \ref{pairwiseIntersectingCor}, the following result proves Theorem \ref{pairwiseCongConsTheo} for \emph{abelian} groups.

\begin{theoremm}\label{pairwiseIntersectingTheo}
Let $G$ be a finite group. The following are equivalent.
\begin{enumerate}
\item $G$ satisfies the PCIP.
\item $G$ is cyclic.
\end{enumerate}
\end{theoremm}

\begin{proof}
The implication \enquote{(2) $\Rightarrow$ (1)} holds by Proposition \ref{sCongProp}, so we focus on \enquote{(1) $\Rightarrow$ (2)}.

It is not hard to show that all subgroups and quotients of a PCIP-group are PCIP-groups themselves. A minimal counterexample to the implication \enquote{(1) $\Rightarrow$ (2)} would thus be a finite, non-cyclic group $G$ all of whose proper subgroups are cyclic. These groups were classified by Miller and Moreno \cite{MM03a} to be one of the following.
\begin{enumerate}
\item $\IZ/p\IZ\times\IZ/p\IZ$, where $p$ is a prime;
\item the quaternion group $\Q_8$; or
\item the metacyclic group $\IZ/q^n\IZ\ltimes\IZ/p\IZ=\langle x,y: x^p=y^{q^n}=1, y^{-1}xy=x^r\rangle$, where $r\equiv1\Mod{q}$ and $r^q\equiv1\Mod{p}$, but $r\not\equiv1\Mod{p}$, since otherwise, the group is cyclic or isomorphic to $\IZ/p\IZ\times\IZ/p\IZ$.
\end{enumerate}
It suffices to show that none of these groups $G$ satisfies the PCIP, which we do now, by specifying three subsets $C_j\subseteq G$ for $j=1,2,3$, each of which is a left and right coset of some subgroup of $G$, such that the $C_j$ intersect pairwise while $\bigcap_{j=1}^3{C_j}=\emptyset$.
\begin{itemize}
\item For groups of the first type, where the elements are pairs $(x,y)$ with $x,y\in\IF_p$, let
\begin{align*}
&C_1:=\{(x,y)\in\IF_p^2: x=1\}, C_2:=\{(x,y)\in\IF_p^2: y=1\},\text{ and} \\
&C_3:=\{(x,y)\in\IF_p^2: x+y=1\}.
\end{align*}
\item For $\Q_8=\{\pm1,\pm i,\pm j,\pm k\}$, let
\begin{align*}
&C_1:=\langle i\rangle j=j\langle i\rangle=\{\pm j,\pm k\}, C_2:=\langle j\rangle i=i\langle j\rangle=\{\pm i,\pm k\},\text{ and} \\
&C_3:=\langle k\rangle j=j\langle k\rangle=\{\pm i,\pm j\}.
\end{align*}
\item We note that the central quotient of a group of the third type is of the same form but with $n=1$, which we may thus assume without loss of generality. Let
\begin{align*}
&C_1=\langle x\rangle y=y\langle x\rangle=\{y,yx,yx^2,\ldots,yx^{p-1}\}, C_2=\langle y\rangle=\{1,y,y^2,\ldots,y^{q-1}\}, \\ 
&\text{and }C_3=\langle yx\rangle=\{1,yx,y^2x^{1+r},y^3x^{1+r+r^2},\ldots,y^{q-1}x^{1+r+r^2+\cdots+r^{q-2}}\}.
\end{align*}
Since the $y$-exponent is $1$ in all elements of $C_1$, we have $C_1\cap C_2=\{y\}$ and $C_1\cap C_3=\{yx\}$. Moreover, $C_2\cap C_3=\{1\}$, since $C_2$ and $C_3$ are distinct subgroups of prime order. It follows that $C_1\cap C_2\cap C_3=\emptyset$, as required.
\end{itemize}
\end{proof}

\begin{proof}[Proof of Theorem \ref{pairwiseCongConsTheo}]
As in the proof of Theorem \ref{pairwiseIntersectingTheo}, the implication \enquote{(2) $\Rightarrow$ (1)} is clear by Proposition \ref{sCongProp}, so we focus on \enquote{(1) $\Rightarrow$ (2)}.

Let $G$ be a finite nilpotent group that is pairwise congruence-consistent. If $G$ is abelian, then $G$ must be cyclic by Corollary \ref{pairwiseIntersectingCor}(2) and Theorem \ref{pairwiseIntersectingTheo}, so we assume (aiming for a contradiction) that $G$ is non-abelian. Then for some prime $p$, the (unique) Sylow $p$-group $S_p$ of $G$ is non-abelian. Because $G$ is the direct product of its Sylow subgroups, we find that $S_p$ is a quotient of $G$. Moreover, it is not hard to prove that quotients of pairwise congruence-consistent groups are themselves pairwise congruence-consistent. Therefore, $S_p$ is pairwise congruence-consistent.

But $S_p$ is a non-abelian finite $p$-group, whence Burnside's Basis Theorem implies that its Frattini quotient $S_p/\Phi(S_p)$ is the finite elementary abelian $p$-group $\IF_p^n$ where $n>1$ is the minimal size of a generating set of $S_p$. Using again that the property of being pairwise congruence-consistent is preserved under passing to quotients, it follows that $\IF_p^n$ is pairwise congruence-consistent. Since $\IF_p^n$ is abelian, Corollary \ref{pairwiseIntersectingCor} thus implies that $\IF_p^n$ satisfies the PCIP, which contradicts Theorem \ref{pairwiseIntersectingTheo}.
\end{proof}

\subsection{Generalization to transformation graphs}\label{subsec6P5}

This paper is concerned with functional graphs, which are natural visualizations of individual functions on a set and are useful for understanding the long-term behavior of discrete dynamical systems. As a generalization, one may consider the situation where a dynamical system does not evolve deterministically, but for some $n\in\IN^+$, each system state $x$ has $n$ possibilities (possibly with repetitions) for its successor state, occurring with different probabilities and represented by the values $g_1(x),g_2(x),\ldots,g_n(x)$ of functions $g_j:X\rightarrow X$. Let us set $\Gcal:=\{g_1,g_2,\ldots,g_n\}$\phantomsection\label{not313}. A first step toward studying the behavior of such a system is to understand the so-called \emph{transformation graph}\phantomsection\label{term102} $\TRAG(X,\Gcal)$\phantomsection\label{not314}, which is defined as the edge-labeled digraph with vertex set $X$ whose arcs are of the form $x\xrightarrow{g_j}g_j(x)$ for $x\in X$ and $j=1,2,\ldots,n$. The terminology \enquote{transformation graph} is from Annexstein, Baumslag and Rosenberg's paper \cite{ABR90a}. The concept is also closely related to operands, which are actions of semigroups on sets \cite[Section 11.1]{CP67a} (and in analogy to the terminology \enquote{group action graph} from \cite{ABR90a}, one could also call transformation graphs \enquote{operand graphs}), and to deterministic finite automata \cite[Subsection 2.2.1]{HMU07a}. In fact, $\TRAG(X,\Gcal)$ is like a deterministic finite automaton with state set $X$ and input symbol set $\Gcal$, but without declared start and accept states.

We observe that except for the edge labels, $\TRAG(X,\{g\})$ is the same as the functional graph $\Gamma_g$ in our notation. As noted in \cite[beginning of Subsection 2.1]{ABR90a}, one may also consider a simple (i.e., no multiple arcs $x\rightarrow x'$ for given $x,x'\in X$), unlabeled version of $\TRAG(X,\Gcal)$, which we denote by $\STRAG(X,\Gcal)$\phantomsection\label{not315}. Of course, for a given set $X$, any digraph with vertex set $X$ in which each vertex has positive out-degree (including possibly $\infty$) is of the form $\STRAG(X,\Gcal)$ for a suitable non-empty $\Gcal\subseteq X^X$.

It would be interesting to know whether the methods developed in our paper could be extended to deal with graphs of the form $\TRAG(\IF_q,\Fcal)$ and $\STRAG(\IF_q,\Fcal)$ where $\Fcal$\phantomsection\label{not316} is a set of generalized cyclotomic mappings of $\IF_q$, say of a common, small index $d$ (which also covers some cases where the index is not uniform, because if $f_j$ for $j=1,2,\ldots,n$ is a generalized cyclotomic mapping of $\IF_q$ of index $d_j$, then each $f_j$ also has index $\lcm(d_1,d_2,\ldots,d_n)$). Specifically, we pose the following problems.

\begin{problemm}\label{tragProb1}
For a given prime power $q$ and set $\Fcal$ of index $d$ generalized cyclotomic mappings of $\IF_q$, devise efficient algorithms (say of runtime polynomial in $\log{q}$ for fixed $d$) that find
\begin{enumerate}
\item a compact parametrization of the connected components of $\TRAG(\IF_q,\Fcal)$ (equivalently, of $\STRAG(\IF_q,\Fcal)$) by representative vertices, and
\item a compact description of the isomorphism type of a connected component of $\TRAG(\IF_q,\Fcal)$, respectively of $\STRAG(\IF_q,\Fcal)$, given by a vertex in the image of the parametrization from point (1).
\end{enumerate}
\end{problemm}

To the authors' knowledge, this is an open problem even for $d=1$ and $|\Fcal|=2$ (i.e., when considering transformation graphs that are each based on two monomial functions over $\IF_q$).

\begin{problemm}\label{tragProb2}
For some classes of sets of index $d$ generalized cyclotomic mappings of $\IF_q$, devise efficient algorithms to decide, for sets $\Fcal_1,\Fcal_2$ in such a class, whether $\TRAG(\IF_q,\Fcal_1)\cong\TRAG(\IF_q,\Fcal_2)$, respectively $\STRAG(\IF_q,\Fcal_1)\cong\STRAG(\IF_q,\Fcal_2)$.
\end{problemm}

\setcounter{secnumdepth}{0}

\section{Appendix: Tabular overview of notation and terminology}

The following two tables contain all pieces of notation and terminology that appear in this paper. We start with a rather short list of notations based on mathematical symbols in Table \ref{termNotTableShort}, which would be hard to find in the much longer Table \ref{termNotTable}, the entries of which are listed in alphabetical order (placing Latin letters before Greek letters, lowercase letters before their capital counterparts, and letters in standard font before calligraphic letters, which are in turn placed before Fraktur letters).

\begin{longtable}[h]{|c|c|c|}\hline
Notation & page & additional comments \\ \hline
$\{0,1\}^{<\infty}$ & \pageref{not197} & the set of all finite bit strings \\ \hline
$\divideontimes$ & \pageref{not78} & \thead{a certain $\IQ$-bilinear product over $\IQ[x_n: n\in\IN^+]$, \\ originally defined by Wei and Xu, such that \\ $\CT(\psi_1\otimes\psi_2)=\CT(\psi_1)\divideontimes\CT(\psi_2)$} \\ \hline
$\bigotimes_{j\in I}{g_j}$ & \pageref{not63} & \thead{defined when each $g_j$ is a function $X_j\rightarrow X_j$; it is the function \\ on $\prod_{j\in I}{X_j}$ given by component-wise application of the $g_j$} \\ \hline
$gg'$ & \pageref{notComp} & \thead{composition of the functions $g$ and $g'$ \\ (first $g$, then $g'$); synonymous: $g'\circ g$} \\ \hline
$\langle g_1,\ldots,g_n\rangle$ & \pageref{notGenerate} & \thead{if $g_1,\ldots,g_n$ are elements of a group $G$, \\ this denotes the subgroup of $G$ generated by the $g_j$} \\ \hline
$G=H\ltimes N$ & \pageref{notSDP} & \thead{expresses that $G$ is the (internal) semidirect \\ product of $H$ and $N$} \\ \hline
$|G_1:G_2|$ & \pageref{not77P25} & the index of the subgroup $G_2$ in $G_1$ \\ \hline
$G_1\wr G_2$ & \pageref{not34} & imprimitive permutational wreath product \\ \hline
$H\leq G$ & \pageref{notSubgroup} & short for \enquote{$H$ is a subgroup of $G$} \\ \hline
$\Ifrak\sim\Ifrak'$ & \pageref{not183} & shorthand for $\SF(\Ifrak)=\SF(\Ifrak')$ \\ \hline
$\Pcal\wedge\Qcal$ & \pageref{not137P5} & \thead{the infimum (coarsest common refinement) of the partitions \\ $\Pcal$ and $\Qcal$; if these are arithmetic partitions, then $\Pcal\wedge\Qcal$ \\ can be obtained as the arithmetic partition spanned by the \\ concatenation of any choices of spanning congruence \\ sequences for $\Pcal$ and $\Qcal$} \\ \hline
$x^g$ & \pageref{notExp} & \thead{if $x$ is an argument of the function $g$, this \\ may denote the function value $g(x)$, especially \\ in conjunction with the composition notation $gg'$} \\\hline
$\bigotimes_{j\in I}{\Gamma_j}$ & \pageref{not61} & the digraph tensor product of the $\Gamma_j$ \\ \hline
$\vec{\nu}\diamond\vec{\nu}'$ & \pageref{not166} & the concatenation of $\vec{\nu}$ and $\vec{\nu}'$ \\ \hline
\caption{Notations based on mathematical symbols.}
\label{termNotTableShort}
\end{longtable}

\begin{longtable}[h]{|c|c|c|}\hline
Notation/terminology & page & additional comments \\ \hline
$a$ & \pageref{not58} & \thead{a variable denoting the linear coefficient of the affine map \\ $A:x\mapsto ax+b$ of $\IZ/m\IZ$} \\ \hline
$\afrak$ & \pageref{not46} & \thead{a variable used to denote the modulus of an (in)congruence} \\ \hline
$\afrak_{i,j}$ & \pageref{not174} & \thead{used in Subsection \ref{subsec3P4} and Subsubsection \ref{subsubsec5P2P3}; \\ the modulus of the $j$-th spanning congruence of $\Pcal_i$} \\ \hline
$\overline{\afrak}_j$ & \pageref{not185} & \thead{(used in the example in Subsection \ref{subsec4P2}); \\ the modulus of the $j$-th spanning congruence of $\Qcal_{0,0}$} \\ \hline
$A$ & \pageref{not57} & \thead{a variable denoting an affine map of a group (mostly $\IZ/m\IZ$)} \\ \hline
$\AC(\Pcal)$ & \pageref{not49} & arithmetic complexity of $\Pcal$ \\ \hline
\thead{admissible indexing \\ function} & \pageref{term47} & none \\ \hline
admissible input & \pageref{term70} & none \\ \hline
admissible output & \pageref{term71} & none \\ \hline
\thead{affine function \\ (between groups)} & \pageref{term94} & none \\ \hline
\thead{affine map \\ (of a group)} & \pageref{term42} & none \\ \hline
$a_i$ & \pageref{not22} & \thead{defined for $i\in\{0,1,\ldots,d-1\}$; an element of $\IF_q$ \\ associated with $f$; $f(x)=a_ix^{r_i}$ for all $x\in C_i$} \\ \hline
$A_i$ & \pageref{not33} & \thead{defined for each $i\in\{0,1,\ldots,d-1\}$ such that $a_i\not=0$; it is the \\ unique affine map of $\IZ/s\IZ$ corresponding to $f_{\mid C_i}:C_i\rightarrow C_{\overline{f}(i)}$ \\ under identifying $C_i,C_{\overline{f}(i)}$ with $\IZ/s\IZ$ via $\iota_i,\iota_{\overline{f}(i)}$, \\ and it is given by the formula $A_i(x)=r_ix+\frac{e_i+r_ii-\overline{f}(i)}{d}$} \\ \hline
$\Acal_i$ & \pageref{not123} & the affine map $A_{i_0}A_{i_1}\cdots A_{i_{\ell-1}}$ of $\IZ/s\IZ$ \\ \hline
$\Acal_{i,h}$ & \pageref{not152} & \thead{used in Subsection \ref{subsec3P3}; defined for $i\in\per(\overline{f})\setminus\{d\}$ \\ and $h\in\IN_0$; it is defined as $A_{i_{-h}}A_{i_{-h+1}}\cdots A_{i_{-1}}$, \\ an affine map of $\IZ/s\IZ$} \\ \hline
algorithmic problem & \pageref{term69} & none \\ \hline
$\overline{\Acal}_{i,p}$ & \pageref{not224} & \thead{used in Subsubsections \ref{subsubsec5P2P1} and \ref{subsubsec5P3P3}; it is the reduction \\ $\Acal_i\bmod{p^{\nu_p(s)}}$, an affine map of $\IZ/p^{\nu_p(s)}\IZ$} \\ \hline
$\Acal'_i$ & \pageref{not243} & the affine permutation $\Acal_i\bmod{s'_i}$ of $\IZ/s'_i\IZ$ \\ \hline
$\aord(x)$ & \pageref{not108} & \thead{additive order of $x\in\IZ/m\IZ$ ($m$ must be clear from context)} \\ \hline
arc & \pageref{term11} & none \\ \hline
(arithmetic) complexity & \pageref{term34} & none \\ \hline
arithmetic partition & \pageref{term32} & none \\ \hline
$\Aut(G)$ & \pageref{not107} & the automorphism group of the group $G$ \\ \hline
$b$ & \pageref{not59} & \thead{a variable denoting the constant coefficient of the affine map \\ $A:x\mapsto ax+b$ of $\IZ/m\IZ$; also used more generally for the \\ value at $1_G$ of an affine map of the group $G$} \\ \hline
$\bfrak$ & \pageref{not45} & \thead{variable used to denote the right-hand side of \\ a congruence or incongruence} \\ \hline
$\bfrak_{i,j}$ & \pageref{not173} & \thead{used in Subsection \ref{subsec3P4} and Subsubsection \ref{subsubsec5P2P3}; \\ the right-hand side of the $j$-th spanning congruence of $\Pcal_i$} \\ \hline
$\overline{\bfrak}_j$ & \pageref{not186} & \thead{(used in the example in Subsection \ref{subsec4P2}); \\ the right-hand side of the $j$-th spanning congruence of $\Qcal_{0,0}$} \\ \hline
$B$ & \pageref{not88} & \thead{variable denoting a block of a (usually arithmetic) partition; \\ cf.~the notation $B_{\vec{r}}$ introduced after Problem \ref{crlListTensorProb}} \\ \hline
$\Bcal_i$ & \pageref{not171} & \thead{used in Subsection \ref{subsec3P4}; the function that maps $x\in C_i$ \\ to the unique $\Pcal_i$-block containing $x$} \\ \hline
$\Bcal(\Pcal,\vec{\nu})$ & \pageref{not50} & \thead{the block of the arithmetic partition $\Pcal$ \\ associated with the logical sign tuple $\vec{\nu}$} \\ \hline
$B_m$ & \pageref{not51} & $m$-th Bell number \\ \hline
bit operation & \pageref{term66} & none \\ \hline
blow-up function & \pageref{term56} & none \\ \hline
\thead{bounded Las Vegas \\ dual complexity} & \pageref{term73} & none \\ \hline
bounded query complexity & \pageref{term72} & none \\ \hline
bracelet & \pageref{term29} & none \\ \hline
bracelet graph & \pageref{term27} & none \\ \hline
$\BU_m$ & \pageref{not189} & \thead{$m$-blow-up function; the unique $\IQ$-algebra endomorphism \\ of $\IQ[x_n: n\in\IN^+]$ such that $\BU_m(x_n)=x_{mn}$ for all $n\in\IN^+$} \\ \hline
$C$ & \pageref{not19} & the index $d$ subgroup of $\IF_q^{\ast}$ \\ \hline
child & \pageref{term37} & synonym: successor \\ \hline
$C_i$ & \pageref{not20} & \thead{a certain subset of $\IF_q$ defined for $i\in\{0,1,\ldots,d\}$; \\ namely, $C_i:=\omega^iC$ if $i\in\{0,1,\ldots,d-1\}$, and $C_d:=\{0_{\IF_q}\}$} \\ \hline
$\Ccal_{\class}$ & \pageref{not201} & \thead{a classical complexity (i.e., bit operation count); \\ used for denoting the first entry of $\vec{\Ccal}^{(\qry)}$ or $\vec{\Ccal}^{(\LV)}$} \\ \hline
$\Ccal_{\conv}$ & \pageref{not209} & \thead{a conversion complexity (i.e., count of conversions \\ from bits to qubits and vice versa); \\ used for denoting the third entry of $\vec{\Ccal}^{(\LV)}$} \\ \hline
$\Ccal_{\fact}$ & \pageref{not202} & \thead{a count of (integer) factorization queries; \\ used for denoting the second entry of $\vec{\Ccal}^{(\qry)}$} \\ \hline
$\Ccal_{\fdl}$ & \pageref{not204} & \thead{a count of (finite) field discrete logarithm queries; \\ used for denoting the fourth entry of $\vec{\Ccal}^{(\qry)}$} \\ \hline
$\Ccal_{\mdl}$ & \pageref{not203} & \thead{a count of modular discrete logarithm queries; \\ used for denoting the third entry of $\vec{\Ccal}^{(\qry)}$} \\ \hline
$\Ccal_{\mord}$ & \pageref{not205} & \thead{a count of modular multiplicative order queries; \\ used for denoting the fifth entry of $\vec{\Ccal}^{(\qry)}$} \\ \hline
$\Ccal_{\prt}$ & \pageref{not206} & \thead{a count of primitive root queries; \\ used for denoting the sixth entry of $\vec{\Ccal}^{(\qry)}$} \\ \hline
$\Ccal_{\quant}$ & \pageref{not208} & \thead{a quantum complexity (i.e., count of elementary quantum \\ gates); used for denoting the second entry of $\vec{\Ccal}^{(\LV)}$} \\ \hline
$\vec{\Ccal}^{(\LV)}$ & \pageref{not207} & \thead{variable denoting a Las Vegas dual complexity, which \\ is itself a triple of component complexities of different kinds} \\ \hline
$\vec{\Ccal}^{(\qry)}$ & \pageref{not200} & \thead{variable denoting a query complexity, which is itself \\ a $6$-tuple of component complexities of different kinds} \\ \hline
$\Cfrak_{i,L}$ & \pageref{not297} & \thead{the set of all $l\in\{1,2,\ldots,L\}$ such that $l$ is a multiple \\ of $\ell$ (the $\overline{f}$-cycle length of $i$) and $\gcd(s,\overline{\alpha}_{i,l}-1)\mid\overline{\beta}_{i,l}$; \\ these are just those $l\in\{1,2,\ldots,L\}$ for which $\eta_{i,l}(x)$ is \\ well-defined and characterizes, for $x\in C_i$, when $f^l(x)=x$} \\ \hline
classical complexity & \pageref{term65} & none \\ \hline
\thead{complexity \\ (of a partition)} & \pageref{term34} & none \\ \hline
\thead{composition (of \\ algorithmic problems)} & \pageref{term74} & none \\ \hline
$\conj(x)$ & \pageref{not86} & \thead{the conjugation by $x\in G$ on the group $G$, \\ i.e., the automorphism $y\mapsto x^{-1}yx$ of $G$} \\ \hline
\thead{connected component \\ (of a functional graph)} & \pageref{term12} & none \\ \hline
\thead{consistent \\ (system of $m$-CCs)} & \pageref{term44} & none \\ \hline
\thead{continued fractions \\ coefficient} & \pageref{term77} & none \\ \hline
convergent & \pageref{term76} & none \\ \hline
conversion complexity & \pageref{term67} & none \\ \hline
coset-wise affine function & \pageref{term97} & same as coset-wise affine function in the narrow sense \\ \hline
\thead{coset-wise affine function \\ in the narrow sense} & \pageref{term96} & none \\ \hline
\thead{coset-wise affine function \\ in the wide sense} & \pageref{term95} & none \\ \hline
CRL-list & \pageref{term18} & \thead{short for \enquote{cycle representatives and lengths list}} \\ \hline
$\CT(\psi)$ & \pageref{not13} & the cycle type of $\psi$ \\ \hline
cycle type & \pageref{term15} & none \\ \hline
\thead{cyclic sequence associated \\ with $\vec{\xfrak}$} & \pageref{term25} & none \\ \hline
cyclic sequence & \pageref{term26} & none \\ \hline
cyclotomic mapping & \pageref{term17} & none \\ \hline
$d$ & \pageref{not16} & an index of $f$ \\ \hline
$\dfrak$ & \pageref{not119} & a variable denoting a divisor \\ \hline
$\dfrak_{i,p,\vec{\overline{u}}}$ & \pageref{not263} & \thead{shorthand for $\prod_{\pfrak\in\Pfrak_{i,p,\vec{\overline{u}}}}{\pfrak^{\nu_{\pfrak}(l_{i,\vec{\overline{u}}})}}$, which is the same as \\ $\prod_{\pfrak\in\Pfrak_{i,p,\vec{\overline{u}}}}{\pfrak^{\nu_{\pfrak}(l_{i,p,\vec{u}_p})}}$ and always divides $l_{i,p,\vec{u}_p}$} \\ \hline
$\Dfrak$ & \pageref{not191} & \thead{a variable denoting a compact description of a \\ finite directed rooted tree isomorphism type with \\ respect to a fixed recursive tree description list} \\ \hline
dihedral sequence & \pageref{term28} & none \\ \hline
directed rooted tree & \pageref{term14} & \thead{we assume that all arcs are oriented toward the root} \\ \hline
\thead{discrete logarithm \\ (in a group)} & \pageref{term78} & none \\ \hline
discrete dynamical system & \pageref{term1} & none \\ \hline
distribution number & \pageref{term46} & none \\ \hline
dual algorithm & \pageref{term61} & none \\ \hline
dual digraph & \pageref{term36} & none \\ \hline
dual model & \pageref{term62} & none \\ \hline
$e$ & \pageref{not79} & \thead{used for exponents (degrees) of variables in cycle types; \\ needs to be distinguished from the notation $e_i$} \\ \hline
$e_i$ & \pageref{not32} & \thead{defined for those $i\in\{0,1,\ldots,d-1\}$ such that $a_i\not=0$; \\ it is the discrete logarithm of $a_i$ with base $\omega$, i.e., $a_i=\omega^{e_i}$} \\ \hline
$E$ & \pageref{not69} & variable denoting the edge (arc) set of a digraph \\ \hline
$E^{-1}$ & \pageref{not71} & the inverse relation of $E$ \\ \hline
$E(\vec{\nu},J)$ & \pageref{not93} & a certain conjunction of divisiblity conditions \\ \hline
\thead{equivalent labeled \\ arithmetic partitions} & \pageref{term91} & none \\\hline
\thead{equivalent sequences of \\ labeled arithmetic partitions} & \pageref{term93} & none \\ \hline
\thead{equivalent systems \\ of $m$-CCs} & \pageref{term45} & none \\ \hline
$\Expand(\Delta)$ & \pageref{not180} & \thead{the \enquote{expanded version} of the finite edge-weighted \\ directed rooted tree $\Delta$} \\ \hline
$f$ & \pageref{not17} & a generalized cyclotomic mapping of $\IF_q$ \\ \hline
$f_{\per}$ & \pageref{not126} & the restriction of $f$ to $\bigcup_{i\in\per(\overline{f})}{C_i}$ \\ \hline
$\overline{f}$ & \pageref{not29} & \thead{the function on $\{0,1,\ldots,d\}$ induced by $f$, \\ defined implicitly via $f(C_i)\subseteq C_{\overline{f}(i)}$} \\ \hline
$\overline{f_{\per}}$ & \pageref{not127} & \thead{the induced function of $f_{\per}$, which is equal to $\overline{f}_{\mid\per(\overline{f})}$} \\ \hline
$\ffrak$ & \pageref{not100} & \thead{a variable denoting a fixed point; in Proposition \ref{crlListPrimaryProp} \\ and discussions based on it, this is a specifically \\ defined fixed point of an affine map of a finite cyclic group} \\ \hline
$\ffrak_i$ & \pageref{not274} & the unique periodic point of $\Acal_i\bmod{s''_i}$ \\ \hline
$\ffrak_{i,p}$ & \pageref{not237} & \thead{defined whenever $\Acal_{i,p}$ has a fixed point; a certain fixed \\ point of $\Acal_{i,p}$, given by the formula in Proposition \ref{crlListPrimaryProp}} \\ \hline
$\Fcal$ & \pageref{not316} & \thead{a set of generalized cyclotomic mappings \\ over a common finite field} \\ \hline
$\IF_q$ & \pageref{not6} & the finite field of size $q$ \\ \hline
finite dynamical system & \pageref{term4} & none \\ \hline
functional graph & \pageref{term10} & none \\ \hline
$g$ & \pageref{not2} & a function $X\rightarrow X$ \\ \hline
$G$ & \pageref{not82} & variable denoting an abstract group \\ \hline
$\Gcal$ & \pageref{not313} & a set of functions $X\rightarrow X$ \\ \hline
\thead{generalized cyclotomic \\ mapping} & \pageref{term16} & none\\ \hline
$\Good_{\vec{r}}(\Ical)$ & \pageref{not116} & \thead{the set of all $\Ical$-good tuples (associated with $\vec{r}$ or, rather, \\ with the cycle length tuple $(l_j(\vec{r}))_{j=1,2,\ldots,n}$) for the \\ $\vec{r}$-admissible indexing function $\Ical$} \\ \hline
good tuple & \pageref{term48} & none \\ \hline
$g_{\mid Y}$ & \pageref{not28} & the restriction of $g$ to the subset $Y\subseteq X$ \\ \hline
$h_i$ & \pageref{not142} & \thead{defined for $i\in\{0,1,\ldots,d-1\}$; if $i$ is $\overline{f}$-transient, then \\ $h_i=\height(\Tree_{\Gamma_{\overline{f}}}(i))$; if $i$ is $\overline{f}$-periodic, then $h_i=\infty$} \\ \hline
$h'_{i,k}$ & \pageref{not269} & \thead{defined when $i<d$ is $\overline{f}$-periodic and $k\in\IZ$; \\ it is the smallest positive integer $h'$ such that \\ $\gcd(\prod_{t=0}^{h'-1}{\alpha_{i_{k-h'+t}}},s)=\gcd(\prod_{t=0}^{h'-2}{\alpha_{i_{k-h'+t}}},s)$; \\ one has $H_i=\max\{h'_{i,k}: k=0,1,\ldots,\ell-1\}$} \\ \hline
$\hfrak(x)$ & \pageref{not144} & \thead{a certain technical parameter from Subsection \ref{subsec3P3}; \\ when $x\in C_i$ for an $\overline{f}$-periodic $i<d$, then our \\ understanding of $\Tree_{\Gamma_f}(x)$ is gained by recursion on $\hfrak(x)$ \\ (after dealing with $\overline{f}$-transient indices $i$ and $i=d$ first)} \\ \hline
$H_i$ & \pageref{not145} & \thead{the maximum height of the rooted trees in $\Gamma_{\per}$ above \\ $f$-periodic vertices in one of the cosets $C_{i_t}$ for $t\in\IZ$} \\ \hline
$\overline{H}$ & \pageref{not283} & the maximum tree height in $\Gamma_{\overline{f}}$ \\ \hline
$\Hcal_i$ & \pageref{not184} & \thead{defined for $i<d$ if $\overline{f}$ is a permutation; \\ it is the common tree height above $f$-periodic vertices \\ in $C_i$; we note that $H_i=\max_{t\in\IZ}{\Hcal_{i_t}}$} \\ \hline
$\Hfrak$ & \pageref{not270} & shorthand for $\max\{H_i: i\in\per(\overline{f})\setminus\{d\}\}$ \\ \hline
height & \pageref{term41} & \thead{attribute of a finite directed rooted tree, \\ denoting the maximum length of a directed path in it} \\ \hline
$\height(\Delta)$ & \pageref{not75} & the height of $\Delta$ \\ \hline
hyperkernel & \pageref{term43} & none \\ \hline
$i$ & \pageref{not21} & an index that can range over $\{0,1,\ldots,d\}$ \\ \hline
$i_t$ &  \pageref{not122} & \thead{a notation defined for $i\in\per(\overline{f})$ and $t\in\IZ$; \\ shorthand for $(\overline{f}_{\mid\per(\overline{f})})^t(i)$; in particular, \\ $(i_0,i_1,\ldots,i_{\ell-1})$ is the $\overline{f}$-cycle of $i=i_0$} \\ \hline
$i'$ & \pageref{not143} & shorthand for $i_{-1}$ \\ \hline
$\ifrak$ & \pageref{not290} & a variable denoting an injective function \\ \hline
$I$ & \pageref{not60} & a variable denoting an index set \\ \hline
$I_i$ & \pageref{not177} & \thead{used in Subsection \ref{subsec3P4}; the set of those $j\in\{1,2,\ldots,m_i\}$ \\ for which the truth value of $x\equiv\bfrak_{i,j}\Mod{\afrak_{i,j}}$ \\ is constant along the $\Acal_i$-cycle of $r_i$ modulo $s$} \\ \hline
$\Ical$ & \pageref{not115} & \thead{variable denoting an $\vec{r}$-admissible \\ indexing function (for some $\vec{r}$)} \\ \hline
$\Ical_i$ & \pageref{not261} & \thead{a certain function $\Pfrak'_i\rightarrow\Pfrak_i$; if $\Ical_{i,\vec{\overline{u}}}(\pfrak)$  only has one \\ distinct value for $\vec{\overline{u}}\in\overline{Y}_i$, then $\Ical_i(\pfrak)$ is that value, \\ otherwise $\Ical_i(\pfrak)$ is the unique value distinct from \\ $\pfrak$ which $\Ical_{i,\vec{\overline{u}}}(\pfrak)$ assumes} \\ \hline
$\Ical_{i,\vec{\overline{u}}}$ & \pageref{not253} & \thead{a certain $\vec{r}_i(\vec{\overline{u}})$-admissible indexing function, \\ extended such that its domain of definition is $\Pfrak'_i$} \\ \hline
$\Ifrak$ & \pageref{not77} & \thead{variable denoting an isomorphism type of finite \\ directed rooted trees (either edge-weighted or not, \\ depending on the context)} \\ \hline
$\Ifrak^+$ & \pageref{not136} & \thead{the rooted tree isomorphism type obtained by attaching \\ a single copy of $\Ifrak$ to a new root} \\ \hline
$\Ifrak_i$ & \pageref{not187} & (used in Subsection \ref{subsec4P3}) shorthand for $\Tree_{\Gamma_{\overline{f}}}(i)$ \\ \hline
$\Ical$-good tuple & \pageref{term48} & none \\ \hline
induced by $f$ & \pageref{term19} & \thead{attribute referring to the function $\overline{f}$} \\ \hline
$\im(g)$ & \pageref{not138} & the image (synonymously, range) of the function $g$ \\ \hline
infimum (of partitions) & \pageref{term51} & none \\ \hline
$\inv_n(u)$ & \pageref{not97} & \thead{the multiplicative inverse modulo $n$ \\ of $u\in\IZ$ with $\gcd(u,n)=1$} \\ \hline
\thead{isomorphism of vertex- \\ labeled digraphs} & \pageref{term23} & none \\ \hline
\thead{isomorphism of finite \\ edge-weighted directed \\ rooted trees} & \pageref{term52} & none \\ \hline 
iterated pre-image & \pageref{term13} & none \\ \hline
$J$ & \pageref{not92} & variable denoting an index set \\ \hline
$J_+(\vec{\nu})$ & \pageref{not91} & the set of indices $j$ such that the $j$-th entry of $\vec{\nu}$ is $\emptyset$ \\ \hline
$J_-(\vec{\nu})$ & \pageref{not90} & the set of indices $j$ such that the $j$-th entry of $\vec{\nu}$ is $\neg$ \\ \hline
$\ker^{(k)}(\varphi)$ & \pageref{not77P5} & \thead{defined for $k\in\IN_0$ and a group endomorphism $\varphi$ of $G$; \\ it is the normal subgroup $\{x\in G: \varphi^k(x)=1_G\}$ of $G$} \\ \hline
$\Jcal_{i,p,k}$ & \pageref{not258} & \thead{one of up to five subintervals of the range for $u_p$ \\ such that $\nu_p(l_{i,p,\vec{u}_p})$ is given by a simple formula in $u_p$ \\ whenever $u_p$ lies in a fixed $\Jcal_{i,p,k}$} \\ \hline
$\Jcal'_{i,p,k}$ & \pageref{not260} & \thead{a subinterval of $\Jcal_{i,p,k}$ such that $\Ical_{i,\vec{\overline{u}}}(p)=p$ \\ if and only if $u_p\in\bigcup_{k=1}^{\mfrak_{i,p}}{\Jcal'_{i,p,k}}$} \\ \hline
$k_p$ & \pageref{not262P75} & \thead{(in Subsubsection \ref{subsubsec5P2P1}) the $p$-labeled entry of $\vec{k}$ \\ (an element of $\IZ/l_{i,\vec{\overline{u}}}\IZ$)} \\ \hline
$\vec{k}$ & \pageref{not262P5} & \thead{(in Subsubsection \ref{subsubsec5P2P1}) variable denoting an $\Ical_{i,\vec{\overline{u}}}$-good tuple} \\ \hline
$k'_p$ & \pageref{not264P5} & \thead{(in Subsubsection \ref{subsubsec5P2P1}) the $p$-labeled entry of $\vec{k'}$ \\ (an element of $\IZ/(l_{i,p,\vec{u}_p}/\dfrak_{i,p,\vec{\overline{u}}})\IZ$)} \\ \hline
$\vec{k'}$ & \pageref{not264} & \thead{a tuple, ranging over $\prod_{\pfrak\in\Pfrak_i}{\IZ/(l_{i,p,\vec{u}_p}/\dfrak_{i,p,\vec{\overline{u}}})\IZ}$, \\ which for a fixed $\vec{\overline{u}}\in\overline{Y}_i$ serves as the parameter \\ (argument) of a certain bijective parametrization \\ of $\Good_{\vec{r}_i(\vec{\overline{u}})}(\Ical_{i,\vec{\overline{u}}})$, the set of good tuples for \\ the $\vec{r}_i(\vec{\overline{u}})$-admissible indexing function $\Ical_{i,\vec{\overline{u}}}$} \\ \hline
$K'_{i,\vec{\overline{u}}}$ & \pageref{not265} & \thead{the set $\prod_{\pfrak\in\Pfrak_i}{\IZ/(l_{i,p,\vec{u}_p}/\dfrak_{i,p,\vec{\overline{u}}})\IZ}$, from which $\vec{k'}$ stems} \\ \hline
$l$ & \pageref{not27} & a cycle length of $f$ or $g$ \\ \hline
$l_{i,j}$ & \pageref{not176} & \thead{used in Subsection \ref{subsec3P4} and Subsubsection \ref{subsubsec5P2P3}; \\ the cycle length of $r'_i$ under $\Acal_i$ modulo $\afrak_{i,j}$} \\ \hline
$l_{i,\vec{\overline{u}}}$ & \pageref{not255} & \thead{the cycle length $\lcm\{l_{i,p,\vec{u}_p}: p\in\Pfrak_i\}$ \\ of $\bigotimes_{p\in\Pfrak_i}{\Acal_{i,p}}$ associated with $\vec{\overline{u}}$} \\ \hline
$l_{i,p,\vec{u}_p}$ & \pageref{not254} & the \enquote{component cycle length} $\proj_2(\para'_{i,p}(\vec{u}_p))$ \\ \hline
$\ell$ & \pageref{not121} & \thead{variable denoting a cycle length of $\overline{f}$, as opposed to a cycle \\ length of $f$, which is denoted by $l$ instead; if $i\in\{0,1,\ldots,d\}$ \\ has been fixed, then $\ell$ is \emph{always} the $\overline{f}$-cycle length of $i$} \\ \hline
$\lfrak$ & \pageref{not272} & \thead{shorthand for $\log_{\omega}(r)$, where $r\in\IF_q$ is an $f$-periodic point \\ representing a connected component of $\Gamma_f$} \\ \hline
$\lfrak_{i,j}$ & \pageref{not175} & \thead{used in Subsection \ref{subsec3P4} and Subsubsection \ref{subsubsec5P2P3}; \\ it denotes the affine discrete logarithm value $\log_{\Acal_i}^{(\afrak_{i,j})}(r_i,\bfrak_{i,j})$} \\ \hline
$L_i$ & \pageref{not267} & \thead{the smallest non-negative integer such that \\ $\gcd\left(\overline{\alpha}_i^{L_i},s\right)=\prod_{p\mid\gcd(\overline{\alpha}_i,s)}{p^{\kappa_p}}$, used in the \\ formula for the unique periodic point of $\Acal_i\bmod{(s/s'_i)}$} \\ \hline
$\Lcal$ & \pageref{not25} & variable denoting a CRL-list \\ \hline
$\overline{\Lcal}$ & \pageref{not120} & a CRL-list of $\overline{f}$ \\ \hline
$\Lcal^{(1)}$ & \pageref{not110} & the set of first entries of pairs in the CRL-list $\Lcal$ \\ \hline
$\Lcal_i$ & \pageref{not125} & a CRL-list of the restriction of $f$ to $\bigcup_{t=0}^{\ell-1}{C_{i_t}}$ \\ \hline
$\Lcal'_i$ & \pageref{not124} & a CRL-list of $\Acal_i$ \\ \hline
$\Lcal'_{i,p}$ & \pageref{not227} & \thead{a CRL-list of $\Acal_{i,p}$ in which all specified cycle lengths \\ are fully factored; for the precise definition, see Table \ref{crlListParTable}} \\ \hline
$\Lcal''_i$ & \pageref{not242} & a CRL-list of $\Acal'_i$ \\ \hline
$\Lcal(p^v,a)$ & \pageref{not105} & \thead{a certain CRL-list of the automorphism $\mu_a$ of $\IZ/p^v\IZ$, \\ defined in Table \ref{crlListAutomorphismTable} and (for $p>2$) depending on \\ a choice of primitive root $\rfrak$ modulo $p^v$} \\ \hline
$\Lcal(p^v,a,b)$ & \pageref{not109} & \thead{a certain CRL-list of the affine permutation $x\mapsto ax+b$ \\ of $\IZ/p^v\IZ$, defined in Table \ref{crlListPrimaryTable} and (for $p>2$) \\ depending on a choice of primitive root $\rfrak$ modulo $p^v$} \\ \hline
$\Lfrak$ & \pageref{not198} & variable denoting an algorithmic problem \\ \hline
$\Lfrak_{\inn}$ & \pageref{not199} & \thead{the set of admissible inputs for the algorithmic problem $\Lfrak$; \\ formally, this is the domain of definition of the function $\Lfrak$} \\ \hline
$\Lfrak\Lfrak'$ & \pageref{not210} & \thead{the composition of the algorithmic problems $\Lfrak$ and $\Lfrak'$ \\ (first $\Lfrak$, then $\Lfrak'$); synonymous notation: $\Lfrak'\circ\Lfrak$} \\ \hline
$\Lfrak'\circ\Lfrak$ & \pageref{not211} & see $\Lfrak\Lfrak'$ \\ \hline
$\lab$ & \pageref{not307} & \thead{variable denoting a labeling function (i.e., the \\ second entry of a labeled arithmetic partition)} \\ \hline
$\lab_i$ & \pageref{not305} & \thead{a certain labeling function for the arithmetic partition $\Pcal^{(i)}$, \\ used in two slightly different meanings \\ (see also the paragraph before Definition \ref{labeledDef})} \\ \hline
$\Lab$ & \pageref{not312} & variable denoting a labeling function (like $\lab$) \\ \hline
\thead{labeled arithmetic partition} & \pageref{term90} & none \\ \hline
Las Vegas algorithm & \pageref{term64} & none \\ \hline
$\Layer_h$ & \pageref{not271} & the set of those $i\in\{0,1,\ldots,d-1\}$ for which $h_i=h$ \\ \hline
\thead{left-regular \\ representation} & \pageref{termLeftReg} & none \\ \hline
$\log_a^{(m)}(x)$ & \pageref{not118} & the discrete logarithm modulo $m$ of $x$ with base $a$ \\ \hline
$\log_A^{(m)}(x,y)$ & \pageref{not117} & \thead{an \enquote{affine discrete logarithm}, \\ formally defined at the beginning of Subsection \ref{subsec2P4}} \\ \hline
$\log_x(y)$ & \pageref{not216} & \thead{if $x$ and $y$ are elements of a group $G$, this denotes \\ the discrete logarithm of $y$ with base $x$ \\ (defined to be $\infty$ if $y$ is not a power of $x$)} \\ \hline
$m$ & \pageref{not4} & a positive integer; used as a general modulus \\ \hline
$m_i$ & \pageref{not140} & \thead{defined for $i\in\{0,1,\ldots,d-1\}$; it denotes the length \\ of a fixed spanning $s$-congruence sequence of $\Pcal_i$} \\ \hline
$\overline{m}_i$ & \pageref{not306} & \thead{length of the standard spanning congruence sequence for $\Pcal^{(i)}$; \\ in the notation of Subsection \ref{subsec3P4}, one has $\overline{m}_i=m_i-|I_i|$} \\ \hline
$m_p$ & \pageref{not217} & the so-called $p$-part of $m\in\IN^+$; defined as $p^{\nu_p(m)}$ \\ \hline
$m_{p'}$ & \pageref{not215} & the so-called $p'$-part of $m\in\IN^+$; defined as $m/p^{\nu_p(m)}$ \\ \hline
$m'$ & \pageref{not98} & \thead{the product of all prime powers $p^{\nu_p(m)}$ where $p$ does \emph{not} \\ divide a certain other integer that is clear from context \\ (usually the linear coefficient of a certain affine map \\ of $\IZ/m\IZ$); not to be confused with the notation $i'$} \\ \hline
$m''$ &\pageref{not99} & the quotient $m/m'$ \\ \hline
$\mfrak$ & \pageref{not288} & \thead{variable denoting a count of isomorphism types \\ of connected components of a functional graph} \\ \hline
$\mfrak_i$ & \pageref{not301} & \thead{a function, used in the proof of Theorem \ref{lengthsBoundedTheo}, \\ which encodes the block sizes of $\Wcal_{i,L}$} \\ \hline
$\mfrak_{i,p}$ & \pageref{not259} & the number of intervals $\Jcal_{i,p,k}$ ($k=1,2,\ldots,\mfrak_{i,p}$) \\ \hline
$M$ & \pageref{not132} & \thead{variable denoting a subset of a universal set \\ (usually $\IZ/m\IZ$ or $\IF_q$)} \\ \hline
$M^c$ & \pageref{not96} & the complement set of $M$ (in its respective universal set) \\ \hline
$m$-CC & \pageref{term31} & short for \enquote{$m$-congruential condition} \\ \hline
$m$-(in)congruence & \pageref{term30} & an (in)congruence whose modulus divides $m$ \\ \hline
\thead{merging (of \\ sorted arrays)} & \pageref{termMerge} & none \\ \hline
$\minperl(\vec{\xfrak})$ & \pageref{not287} & minimal period length of the finite sequence $\vec{\xfrak}$ \\ \hline
Monte Carlo algorithm & \pageref{term75} & none \\ \hline
$\mpe(m)$ & \pageref{not81} & \thead{the maximum prime exponent of $m$, \\ i.e., $\mpe(m)=\max_p{\nu_p(m)}$} \\ \hline
\thead{multiple of a rooted \\ tree isomorphism type} & \pageref{term50} & none \\ \hline
$n_i$ & \pageref{not148} & \thead{the length of a fixed spanning $s$-congruence sequence for $\Rcal_i$} \\ \hline
$\overline{n}_i$ & \pageref{not284} & \thead{used in the proof of Lemma \ref{treeRegLem}(2); it denotes the \\ unique $n\in\Ncal$ such that $\Ifrak_n=\Tree_{\Gamma_f}(x)$ for each $x\in C_i$} \\ \hline
$\overline{n}_{i,k}$ & \pageref{not286} & \thead{used in the proof of Lemma \ref{treeRegLem}(4); it denotes the \\ unique $n\in\Ncal$ such that $\Expand(\Ifrak_{i,k})=\Ifrak_n$ (i.e., $i\in S_{n,\trans,k}$)} \\ \hline
$n'_i(\vec{\nu})$ & \pageref{not277} & \thead{a special notation used in Subsubsection \ref{subsubsec5P2P3}; \\ it denotes the unique $n\in\Ncal$ such that $\vec{\nu}^+\in S_{i,n}$} \\ \hline
$\nfrak_p$ & \pageref{not231} & \thead{used in Subsubsection \ref{subsubsec5P2P1}; it denotes \\ the number of distinct prime divisors of $p-1$} \\ \hline
$\Ncal$ & \pageref{not268} & \thead{used in Subsubsections \ref{subsubsec5P2P2} and \ref{subsubsec5P3P2}; it denotes \\ an initial segment of $\IN_0$, consisting of those $n$ for which \\ the compact tree description $\Dfrak_n$ has been defined \\ (at the respective point in the algorithm in question)} \\ \hline
$\Nfrak$ & \pageref{not289} & variable denoting a tree necklace list \\ \hline
$\IN^+$ & \pageref{not14} & the set of positive integers \\ \hline
$\IN_0$ & \pageref{not24} & the set of non-negative integers \\ \hline
necklace & \pageref{term24} & none \\ \hline
necklace graph & \pageref{term22} & none \\ \hline
$\nil(\varphi)$ & \pageref{not87} & the hyperkernel of $\varphi$ \\ \hline
$\Ocal_{i,L}$ & \pageref{not292} & \thead{defined for $\overline{f}$-periodic $i\in\{0,1,\ldots,d-1\}$ \\ and $L\in\IN_0$; it is the set of all logical sign tuples \\ $\diamond_{t=0}^{H_i+L}{\vec{o'_t}}\in\{\emptyset,\neg\}^{n_{i_0}+n_{i_{-1}}+\cdots+n_{i_{-H_i-L}}}$ such that the \\ associated block $\Bcal(\Qcal_{i,H_i+L},\diamond_{t=0}^{H_i+L}{\vec{o'_t}}\diamond\vec{\xi}_{i,H_i})$ of \\ $\Qcal_{i,H_i+L}$ consisting of $f$-periodic points is non-empty} \\ \hline
$\ord(x)$ & \pageref{not101P5} & \thead{the order of the group element $x\in G$; the group $G$ must \\ be clear from context; for elements of $\IZ/m\IZ$, this always \\ denotes the multiplicative order (hence is only well-defined \\ if $x$ is a unit modulo $m$) -- see also the notation $\aord$} \\ \hline
$\ord_n(x)$ & \pageref{not218} & the multiplicative order of $x$ modulo $n$ \\ \hline
$p$ & \pageref{not56P5} & \thead{variable denoting a prime, not necessarily the prime base of $q$; \\ using $p$ as a summation index implies that only prime indices \\ satisfying the explicitly stated constraints should be used} \\ \hline
$\pfrak$ & \pageref{not257} & \thead{variable used to denote a general element of $\Pfrak'_i$ \\ (to be distinguished from $p$ used for elements of $\Pfrak_i$)} \\ \hline
$\pfrak_{p,k}$ & \pageref{not232} & the $k$-th prime divisor of $p-1$ (in a fixed factorization) \\ \hline
$P(T)$ & \pageref{not190P5} & a polynomial in the variable $T$ \\ \hline
$\Pcal$ & \pageref{not48} & variable denoting a partition, usually an arithmetic one \\ \hline
$\Pcal_i$ & \pageref{not43} & \thead{a certain arithmetic partition of $C_i\cong\IZ/s\IZ$, defined in \\ Subsection \ref{subsec3P3}, which \enquote{controls} the trees above vertices in $C_i$} \\ \hline
$\Pcal^{(i)}$ & \pageref{not178} & \thead{used in Subsection \ref{subsec3P4} and Subsubsection \ref{subsubsec5P2P3}; a certain \\ arithmetic partition that encodes (part of) the cyclic \\ sequence of rooted tree isomorphism types that represents the \\ connected component of $\Gamma_f$ containing $r'_i$} \\ \hline
$\Pcal'_i$ & \pageref{not141} & shorthand for $\Pfrak'(\Pcal_i,A_i)$ \\ \hline
$\Pcal_{i,h}$ & \pageref{not163} & \thead{a certain arithmetic partition of $C_i$ such that for $x\in C_i$ \\ with $\hfrak(x)=h$, the isomorphism type $\Tree_{\Gamma_f}(x)$ \\ only depends on the $\Pcal_{i,h}$-block in which $x$ is contained} \\ \hline
$\Pfrak$ & \pageref{not47} & \thead{an operator, used in the notation \\ $\Pfrak(x\equiv\bfrak_k\Mod{\afrak_k}: k=1,2,\ldots,K)$, \\ which denotes the arithmetic partition of $\IZ/m\IZ$ (for \\ implicitly clear $m\in\IN^+$) spanned by those $m$-congruences} \\ \hline
$\Pfrak_i$ & \pageref{not228} & \thead{the set of all primes $p\mid s$ such that $p\nmid\overline{\alpha}_i$ (i.e., such that \\ $\Acal_{i,p}$ is a permutation of $\IZ/p^{\nu_p(s)}\IZ$)} \\ \hline
$\Pfrak_{i,p,\vec{\overline{u}}}$ & \pageref{not262} & the pre-image set $\Ical^{-1}_{i,\vec{\overline{u}}}(\{p\})$ \\ \hline
$\Pfrak'_i$ & \pageref{not256} & \thead{the extended domain of definition \\ $\Pfrak_i\cup\pi(\prod_{p\in\Pfrak_i}{(p-1)})$ of $\Ical_{i,\vec{\overline{u}}}$} \\ \hline
$\Pfrak(\Theta(x))$ & \pageref{not158} & \thead{the arithmetic partition of $\IZ/m\IZ$ spanned by \\ the non-negated versions of the $m$-CCs in $\Theta(x)$} \\ \hline
$\Pfrak'(\Pcal,A)$ & \pageref{not89} & \thead{defined when $\Pcal$ is an arithmetic partition of $\IZ/m\IZ$ with a \\ fixed spanning congruence sequence and $A$ is an affine map \\ of $\IZ/m\IZ$; it is a certain arithmetic partition of $\IZ/m\IZ$ \\ such that the intersection size $|A^{-1}(\{x\})\cap B|$, where \\ $B=\Bcal(\Pcal,\vec{\nu})$ is a fixed block of $\Pcal$, is constant (equal to \\ $\sigma_{\Pcal,A}(\vec{\nu},\vec{\nu'})$) when $x$ ranges over a fixed block $\Bcal(\Pcal',\vec{\nu'})$ \\ of $\Pcal'=\Pfrak'(\Pcal,A)$} \\ \hline
\thead{pairwise \\ congruence-consistent} & \pageref{term98} & none \\ \hline
\thead{pairwise coset-intersection \\ property} & \pageref{term99} & none \\ \hline
$\para_i$ & \pageref{not220} & \thead{a bijective parametrization of $\Lcal_i$; for $i<d$, this is obtained \\ from $\para'_i$ by stretching all cycle lengths by the factor $\ell$ \\ (the $\overline{f}$-cycle length of $i$)} \\ \hline
$\para'_i$ & \pageref{not223} & a bijective parametrization of $\Lcal'_i$ (for $i<d$) \\ \hline
$\para''_i$ & \pageref{not241} & a bijective parametrization of $\Lcal''_i$ \\ \hline
$\para'_{i,p}$ & \pageref{not225} & a bijective parametrization of $\Lcal'_{i,p}$ \\ \hline
partition-tree register & \pageref{term58} & none \\ \hline
PCIP & \pageref{term100} & short for \enquote{pairwise coset-intersection property} \\ \hline
PCIP-group & \pageref{term101} & none \\ \hline
$\per(g)$ & \pageref{not10} & the set of $g$-periodic points in $X$ \\ \hline
\thead{period (length) (of an \\ argument under a function)} & \pageref{term7} & none \\ \hline
periodic & \pageref{term8} & synonymous uses: periodic under $f$; $f$-periodic \\ \hline
\thead{period length (of a \\ finite sequence)} & \pageref{term86} & none \\ \hline
$\perl_g(x)$ & \pageref{not9} & the period (length) of $x\in X$ under $g$ \\ \hline
phase inversion & \pageref{term68} & none \\ \hline
$\pperl_g(x)$ & \pageref{not8} & the pre-period (length) of $x\in X$ under $g$ \\ \hline
pre-period (length) & \pageref{term6} & none \\ \hline
pre-periodic & \pageref{term5} & none \\ \hline
$\proc_{i,h}$ & \pageref{not131} & \thead{an auxiliary (procreation) number, used in \\ the proof of Theorem \ref{cosetPermTheo}} \\ \hline
$\proc_k(x)$ & \pageref{not72} & \thead{the $k$-th procreation number in $\Gamma$ of $x\in\V(\Gamma)$; \\ also written $\proc_k^{(\Gamma)}(x)$ for greater clarity} \\ \hline
procreation number & \pageref{term38} & none \\ \hline
$\proj_{i,L}$ & \pageref{not294} & \thead{the projection mapping each tuple in $\Ocal_{i,L}$ to \\ its initial segment of length $n_{i_0}+n_{i_{-1}}+\cdots+n_{i_{-H_i}}$} \\ \hline
$\proj_j$ & \pageref{not250} & \thead{defined for $j\in\{1,2\}$, it is the (class-sized) \\ function mapping an ordered pair to its $j$-th entry} \\ \hline
$q$ & \pageref{not5} & a prime power \\ \hline
$\Qcal_{i,h}$ & \pageref{not165} & \thead{the arithmetic partition $\Pcal_{i,h}\wedge\Ucal_i$ of $C_i$; plays an auxiliary \\ role in the proof of Proposition \ref{periodicCosetsPeriodicProp}; moreover, $\Pcal_i:=\Qcal_{i,H_i}$} \\ \hline
quantum complexity & \pageref{term63} & none \\ \hline
query algorithm & \pageref{term60} & none \\ \hline
query model & \pageref{term59} & none \\ \hline
$r$ & \pageref{not26} & \thead{a periodic point under $f$ or $g$ that represents a cycle of that \\ function; exception: the use as an exponent in \\ \enquote{$r$-th order cyclotomic mapping} (cf.~the notation $r_i$)} \\ \hline
$r_i$ & \pageref{not23} & \thead{defined for $i\in\{0,1,\ldots,d-1\}$; a non-negative integer \\ between $0$ and $q-2$ associated with $f$; one has \\ $f(x)=a_ix^{r_i}$ for all $x\in C_i$} \\ \hline
$r_{i,p}(\vec{u}_p)$ & \pageref{not252} & same as $\proj_1(\para'_{i,p}(\vec{u}_p))$, the $p$-labeled entry of $\vec{r}_i(\vec{\overline{u}})$ \\ \hline
$r'_i$ & \pageref{not170} & \thead{used in Subsection \ref{subsec3P4} and Subsubsection \ref{subsubsec5P2P3}; a certain \\ coset representative of $C_i$} \\ \hline
$r'_{i,\vec{\overline{u}}}(\vec{k'})$ & \pageref{not266} & \thead{the unique element of $\IZ/s'_i\IZ$ that is congruent to \\ $\Acal_{i,p}^{k'_p\dfrak_{i,p,\vec{\overline{u}}}}(r_{i,p}(\vec{u}_p))$ modulo $p^{\kappa_p}$ for each $p\in\Pfrak_i$; it is \\ given by the formula \\ $\sum_{p\in\Pfrak_i}{\Acal_{i,p}^{k'_p\dfrak_{i,p,\vec{\overline{u}}}}(r_{i,p}(\vec{u}_p))\frac{s'_i}{p^{\kappa_p}}\inv_{p^{\kappa_p}}\left(\frac{s'_i}{p^{\kappa_p}}\right)}$} \\ \hline
\thead{$\vec{r}$-admissible \\ indexing function} & \pageref{term47} & none \\ \hline
$\vec{r}_i(\vec{\overline{u}})$ & \pageref{not251} & \thead{the element $(\proj_1(\para'_{i,p}(\vec{u}_p)))_{p\in\Pfrak_i}$ of $\prod_{p\in\Pfrak_i}{\IZ/p^{\kappa_p}\IZ}$ \\ associated with $\vec{\overline{u}}$; these elements are representatives for \\ the orbits of the natural (component-wise) action of \\ $\langle\Acal_{i,p}: p\in\Pfrak_i\rangle$ on $\prod_{p\in\Pfrak_i}{\IZ/p^{\kappa_p}\IZ}$} \\ \hline
$r^{\langle\psi\rangle}$ & \pageref{not111} & \thead{the set of points on the $\psi$-cycle of $r$, a special case of the \\ notation $x^G$ for the orbit of $x\in X$ under the natural \\ action of $G\leq\Sym(X)$} \\ \hline
$\rfrak$ & \pageref{not103} & variable denoting a primitive root \\ \hline
$\rfrak_p$ & \pageref{not236} & \thead{a primitive root modulo $p^{\nu_p(m)}$ where $m\in\IN^+$ is clear from \\ context (in Subsubsection \ref{subsubsec5P2P1}, $m=s$)} \\ \hline
$\Rcal_i$ & \pageref{not147} & \thead{defined for $i\in\{0,1,\ldots,d-1\}$; it is the arithmetic partition \\ $\bigwedge_{t=1}^K{\Pcal_{j_t}}$ where $j_1,j_2,\ldots,j_K$ are the $\overline{f}$-transient pre-images \\ of $i$ under $\overline{f}$; if $i$ is $\overline{f}$-transient itself, then $\Rcal_i=\Pcal_i$} \\ \hline
$\Rfrak$ & \pageref{not282} & variable denoting a type-I or -II tree register \\ \hline
\thead{recursive tree \\ description list} & \pageref{term57} & none \\ \hline
\thead{right-regular \\ representation} & \pageref{termRightReg} & none \\ \hline
rigid procreation & \pageref{term40} & none \\ \hline
rooted tree & \pageref{term14} & for us, synonymous with \enquote{directed rooted tree} \\ \hline
$\rt(\Delta)$ & \pageref{not74} & the root of $\Delta$ \\ \hline
$s$ & \pageref{not30} & the group order of $C$, i.e., $s=(q-1)/d$ \\ \hline
$s'_i$ & \pageref{not244} & shorthand for $\prod_{p\in\Pfrak_i}{p^{\kappa_p}}$ \\ \hline
$s''_i$ & \pageref{not273} & shorthand for $s/s'_i$ \\ \hline
$\sfrak$ & \pageref{not113} & \thead{variable denoting a shift (additive translation) of a finite \\ abelian group; see in particular the notation $\sfrak_{\vec{r}}$ introduced \\ after Problem \ref{crlListTensorProb}} \\ \hline
$S_n$ & \pageref{not279} & \thead{notation used in type-I and -II tree registers; it contains the \\ information where the trees from the register occur in $\Gamma_f$} \\ \hline
$S_{n,i}$ & \pageref{not193} & \thead{appears in the definition of a partition-tree register; it is \\ either a specific logical sign tuple, or a set thereof, or a tuple \\ of such sets; together, the $S_{n,i}$ encode the information \\ which block of $\Pcal_i$ corresponds to which rooted tree} \\ \hline
$\overline{S}_{n,i}$ & \pageref{not277P5} & \thead{a special notation used in Subsubsection \ref{subsubsec5P2P3}; \\ the element-wise image of $S_{n,i}$ under $\vec{\nu'}\mapsto\vec{\nu'}^-$} \\ \hline
$S_{n,\per}$ & \pageref{not281} & \thead{part of $S_n$ in a type-II tree register; it stores the information \\ for which $i\in\{0,1,\ldots,d-1\}$ the rooted tree isomorphism \\ type $\Ifrak_n$ occurs above $f$-periodic vertices in $C_i$} \\ \hline
$S_{n,\trans}$ & \pageref{not280} & \thead{part of $S_n$ in a type-II tree register; it stores the information \\ for which $i\in\{0,1,\ldots,d-1\}$ the rooted tree isomorphism \\ type $\Ifrak_n$ occurs above $f$-transient vertices in $C_i$} \\ \hline
$\Scal_{i,h}$ & \pageref{not151} & \thead{defined when $i\in\per(\overline{f})\setminus\{d\}$ and $h\in\{0,1,\ldots,H_i\}$; it is a \\ certain arithmetic partition of $C_i$ such that for $x\in C_i$ and \\ $\hfrak(x)=h$, the isomorphism type $\Tree_{\Gamma_f}(x,C_{i'})$ only depends \\ on the block of $\Scal_{i,h}$ in which $x$ is contained} \\ \hline
semidirect product & \pageref{termSDP} & none \\ \hline
$\SF(\Ifrak)$ & \pageref{not182} & simplified form of $\Ifrak$ \\ \hline
simplified & \pageref{term53} & none \\ \hline
simplified form & \pageref{term54} & none \\ \hline
spanned by & \pageref{term33} & none \\ \hline
special type I & \pageref{term82} & not to be confused with \enquote{type I} \\ \hline
special type II & \pageref{term83} & not to be confused with \enquote{type II} \\ \hline
$\Stab_G(x)$ & \pageref{not310} & \thead{defined when $G$ is a group acting on a set $X$ and $x\in X$; \\ it denotes the stabilizer of $x$ in $G$} \\ \hline
state space & \pageref{term2} & none \\ \hline
$\STRAG(X,\Gcal)$ & \pageref{not315} & \thead{the underlying digraph of $\TRAG(X,\Gcal)$, which has no edge \\ labels and no multiple edges $x\rightarrow x'$ for fixed $x,x'\in X$} \\ \hline
successor & \pageref{term37} & synonym: child \\ \hline
successor generations & \pageref{term39} & none \\ \hline
\thead{sum of rooted tree \\ isomorphism types} & \pageref{term49} & none \\ \hline
\thead{sum of simplified \\ edge-weighted rooted \\ tree isomorphism types} & \pageref{term55} & none \\ \hline
$\Sym(n)$ & \pageref{not36} & same as $\Sym(\{0,1,\ldots,n-1\})$ \\ \hline
$\Sym(X)$ & \pageref{not35} & symmetric group on $X$ \\ \hline
synchronization & \pageref{term89} & none \\ \hline
$\tfrak$ & \pageref{not308} & \thead{the translation $x\mapsto x+1$ of the cyclic group $\IZ/m\IZ$ \\ ($m$ must be clear from context)} \\ \hline
$T$ & \pageref{not190P75} & a formal variable, used for polynomial rings \\ \hline
$\Tcal_m$ & \pageref{not169} & the trivial partition of $\IZ/m\IZ$ (all blocks are singletons) \\ \hline
$\Tcal_{i,h}$ & \pageref{not164} & \thead{the arithmetic partition $\Scal_{i,h}\wedge\Ucal_i$ of $C_i$; plays an auxiliary \\ role in the proof of Proposition \ref{periodicCosetsPeriodicProp}} \\ \hline
$\Tfrak$ & \pageref{not309} & the cyclic subgroup of $\Sym(\IZ/m\IZ)$ generated by $\tfrak$ \\ \hline
\thead{tensor product \\ (of digraphs)} & \pageref{term35} & none \\ \hline
$\TRAG(X,\Gcal)$ & \pageref{not314} & \thead{the transformation graph associated with the set $X$ and \\ set $\Gcal$ of functions $X\rightarrow X$} \\ \hline
transient & \pageref{term9} & synonymous uses: transient under $f$; $f$-transient \\ \hline
transition function & \pageref{term3} & none \\ \hline
translation number & \pageref{term92} & none \\ \hline
$\Tree_i(\Pcal_i,B)$ & \pageref{not179} & defined when $B=\Bcal(\Pcal_i,\vec{\nu})$; the same as $\Tree_i(\Pcal_i,\vec{\nu})$ \\ \hline
$\Tree_i(\Pcal_i,\vec{\nu})$ & \pageref{not55} & \thead{the rooted tree isomorphism type associated with the block \\ $\Bcal(\Pcal_i,\vec{\nu})$ of $\Pcal_i$} \\ \hline
$\Tree_i(\Pcal,M,\vec{\nu})$ & \pageref{not134} & \thead{defined when $\Tree_{\Gamma_f}(x,M)$ for $x\in C_i$ only depends on the \\ $\Pcal$-block in which $x$ is contained (under the identification \\ of $C_i$ with $\IZ/s\IZ$ via $\iota_i$); it denotes the common rooted \\ tree isomorphism type $\Tree_{\Gamma_f}(x,M)$ for $x\in\Bcal(\Pcal,\vec{\nu})$} \\ \hline
$\Tree_i(\Pcal,\vec{\nu})$ & \pageref{not135} & \thead{shorthand for $\Tree_i(\Pcal,\Bcal(\Pcal,\vec{\nu}))$; a natural extension of the \\ notation $\Tree_i(\Pcal_i,\vec{\nu})$ to arbitrary arithmetic \\ partitions of $C_i$} \\ \hline
$\Tree_i(\Pcal'_j,C_j,\vec{\nu}^{(\Pcal'_j)})$ & \pageref{not149} & \thead{defined when $i\in\{0,1,\ldots,d-1\}$ and $j$ is an $\overline{f}$-transient \\ pre-image of $i$ under $\overline{f}$; it denotes the common \\ isomorphism type $\Tree_{\Gamma_f}(x,C_j)$ for $x\in\Bcal(\Pcal'_j,\vec{\nu}^{(\Pcal'_j)})\subseteq C_i$} \\ \hline
$\Tree_i(\Rcal_i,\bigcup_{t=1}^K{C_{j_t}},\vec{\nu}^{(\Rcal_i)})$ & \pageref{not150} & \thead{defined when $i\in\{0,1,\ldots,d-1\}$ and the $j_t$ are the \\ $\overline{f}$-transient pre-images of $i$ under $\overline{f}$; it denotes the common \\ isomorphism type $\Tree_{\Gamma_f}(x,\bigcup_{t=1}^K{C_{j_t}})$ for \\ $x\in\Bcal(\Rcal_i,\vec{\nu}^{(\Rcal_i)})\subseteq C_i$} \\ \hline
$\Tree_i^{(h)}(\Pcal_{i,h},\vec{\nu}^{(\Pcal_{i,h})})$ & \pageref{not168} & \thead{defined when $i<d$ is $\overline{f}$-periodic and $h\in\{0,1,\ldots,H_i\}$; \\ a rooted tree isomorphism type to which all $\Tree_{\Gamma_f}(x)$ \\ for $x\in\Bcal(\Pcal_{i,h},\vec{\nu}^{(\Pcal_{i,h})})\subseteq C_i$ with $\hfrak(x)=h$ are equal} \\ \hline
$\Tree_i^{(h)}(\Scal_{i,h},C_{i'},\vec{\nu}^{(\Scal_{i,h})})$ & \pageref{not167} & \thead{defined when $i<d$ is $\overline{f}$-periodic and $h\in\{0,1,\ldots,H_i\}$; \\ a rooted tree isomorphism type to which all $\Tree_{\Gamma_f}(x,C_{i'})$ \\ for $x\in\Bcal(\Scal_{i,h},\vec{\nu}^{(\Scal_{i,h})})\subseteq C_i$ with $\hfrak(x)=h$ are equal} \\ \hline
$\Tree_{\Gamma}(x)$ & \pageref{not38} & \thead{the tree above $x$ in $\Gamma$} \\ \hline
$\Tree_{\Gamma}(x,M)$ & \pageref{not133} & \thead{like $\Tree_{\Gamma}(x)$, but deleting all subgraphs $\Tree_{\Gamma}(y)$ \\ where $y\rightarrow x$ and $y\notin M$} \\ \hline
tree above $x$ in $\Gamma$ & \pageref{term21} & none \\ \hline
\thead{tree necklace list (relative \\ to a sequence of rooted tree \\ isomorphism types)} & \pageref{term87} & none \\ \hline
\thead{tree necklace list (relative \\ to a tree register)} & \pageref{term88} & none \\ \hline
type I & \pageref{term79} & \thead{not to be confused with \enquote{special type I} \\ or \enquote{type-I tree register}} \\ \hline
type-I tree register & \pageref{term84} & none \\ \hline
type II & \pageref{term80} & \thead{not to be confused with \enquote{special type II} \\ or \enquote{type-II tree register}} \\ \hline
type-II tree register & \pageref{term85} & none \\ \hline
type III & \pageref{term81} & none \\ \hline
$u$ & \pageref{not238} & the first of up to two parameters making up $\vec{u}$ \\ \hline
$u_p$ & \pageref{not248} & the first of up to two parameters making up $\vec{u}_p$ \\ \hline
$u'$ & \pageref{not239} & the second of up to two parameters making up $\vec{u}$ \\ \hline
$u'_p$ & \pageref{not249} & the second of up to two parameters making up $\vec{u}_p$ \\ \hline
$\vec{u}$ & \pageref{not240} & \thead{general form of an element of the parameter set $Y_{i,p}$ (depends \\ on $p$, which is suppressed in this notation; see also $\vec{u}_p$)} \\ \hline
$\vec{u}_p$ & \pageref{not246} & the $p$-labeled component of $\vec{\overline{u}}$; it is an element of $Y_{i,p}$ \\ \hline
$\vec{\overline{u}}$ & \pageref{not245} & general form of an element of the parameter set $\overline{Y}_i$ \\ \hline
$\ufrak_{i,L}$ & \pageref{not293} & \thead{the function $\Ocal_{i,L}\rightarrow\Ocal_{i',L-1}$ mapping each tuple \\ $\diamond_{t=0}^{H_i+L}{\vec{o'_t}}\in\Ocal_{i,L}$ to the unique tuple $\diamond_{t=0}^{H_i+L-1}{\vec{o_t}}\in\Ocal_{i',L-1}$ \\ such that for each ($f$-periodic) point \\ $x\in\Bcal(\Qcal_{i,H_i+L},\diamond_{t=0}^{H_i+L}{\vec{o'_t}}\diamond\vec{\xi}_{i,H_i})$, the unique $f$-periodic \\ pre-image $x^{(-1)}\in C_{i'}$ of $x$ under $f$ lies in the block \\ $\Bcal(\Qcal_{i',H_i+L-1},\diamond_{t=0}^{H_i+L-1}{\vec{o_t}\diamond\vec{\xi}_{i',H_i}})$} \\ \hline
$U_i$ & \pageref{not188} & \thead{defined for $\overline{f}$-periodic $i$, say of cycle length $\ell$; it is the union \\ $\bigcup_{t=0}^{\ell-1}{C_{i_t}}$ of the blocks $C_j$ for all indices $j$ on the $\overline{f}$-cycle of $i$} \\ \hline
$U_{\chi}$ & \pageref{not196} & phase inversion operator associated with $\chi$ \\ \hline
$\Ucal_i$ & \pageref{not159} & \thead{defined when $i\in\{0,1,\ldots,d-1\}$ is $\overline{f}$-periodic; it is the \\ arithmetic partition $\Pfrak(\theta_{i,h}(x): h=1,2,\ldots,H_i)$ of $C_i$, \\ the blocks of which are the subsets of $C_i$ consisting of \\ points with a common $\hfrak$-value} \\ \hline
$v$ & \pageref{not64} & variable denoting an exponent in a prime factorization \\ \hline
$v_{p,k}$ & \pageref{not233} & shorthand for $\nu_{\pfrak_{p,k}}(p-1)$ \\ \hline
$v'_{i,p}$ & \pageref{not235} & shorthand for $\nu_p(\ord_{p^{\kappa_p}}(\overline{\alpha}_i))$ \\ \hline
$v'_{i,p,k}$ & \pageref{not234} & shorthand for $\nu_{\pfrak_{p,k}}(\ord_{p^{\kappa_p}}(\overline{\alpha}_i))$ \\ \hline
$v''_{i,2}$ & \pageref{not236} & shorthand for $\nu_2(\ord_{2^{\kappa_2}}(-\overline{\alpha}_i)$ \\ \hline
$\vfrak_i$ & \pageref{not139} & \thead{the number of vertices strictly above $i$ in $\Gamma_{\overline{f}}$; \\ in other words, $\vfrak_i=|\V(\Tree_{\Gamma_{\overline{f}}}(i))|-1$} \\ \hline
$V$ & \pageref{not68} & variable denoting the vertex set of a digraph \\ \hline
$\V(\Gamma)$ & \pageref{not38P5} & the vertex set of $\Gamma$ \\ \hline
$\Vcal_{i,L}$ & \pageref{not298} & \thead{the arithmetic partition $\Pfrak(\eta_{i,l}(x): l\in\Cfrak_{i,L})$ of $C_i$; \\ if all $f$-periodic points in $C_i$ have $f$-cycle length \\ at most $L$, then $\Vcal_{i,L}$ is the unique arithmetic partition \\ of $C_i$ such that one block of $\Vcal_{i,L}$ consists of all $f$-transient \\ points in $C_i$, while every other block of $\Vcal_{i,L}$ consists of \\ all $f$-periodic points of a common $f$-cycle length} \\ \hline
$w$ & \pageref{not80} & \thead{variable denoting an element of the same ground set \\ as $x,y,z$} \\ \hline
$\wfrak$ & \pageref{not181} & variable denoting an edge weight \\ \hline
$\wfrak_k$ & \pageref{not285} & \thead{notation used in the proof of Lemma \ref{treeRegLem}(4); for fixed \\ $i\in\{0,1,\ldots,d-1\}$ and $t\in\IZ$, it denotes the weight \\ of the $k$-th edge from the left in the drawing of $\Ifrak_{i_t,k}$ \\ in Subsection \ref{subsec4P1}} \\ \hline
$W$ & \pageref{not52} & Lambert W function \\ \hline
$\Wcal_{i,L}$ & \pageref{not295} & \thead{the arithmetic partition $\Vcal_{i,L}\wedge\Qcal_{i,H_i+L-1}$ of $C_i$; if all \\ $f$-periodic points in $C_i$ have $f$-cycle length at most $L$, then \\ for each $f$-periodic $x\in C_i$, the $\Wcal_{i,L}$-block in which \\ $x$ is contained controls the isomorphism type of the \\ connected component of $\Gamma_f$ containing $x$} \\ \hline
wreath product & \pageref{term20} & none \\ \hline
$x$ & \pageref{not7} & \thead{an element of a ground set such as $X$ or $\IF_q$, or a variable \\ in a congruence, depending on the context} \\ \hline
$x_n$ & \pageref{not15} & a formal variable, used in cycle types \\ \hline
$x^{(t)}$ & \pageref{not291} & defined for $x\in\per(f)$ and $t\in\IZ$; denotes $(f_{\mid\per(f)})^t(x)$ \\ \hline
$\xfrak$ & \pageref{not40} & an element of $\Xfrak$ \\ \hline
$[\vec{\xfrak}]$ & \pageref{not42} & \thead{the cyclic sequence associated with the finite sequence $\vec{\xfrak}$ \\ over $\Xfrak$; also written $[\xfrak_0,\ldots,\xfrak_{L-1}]$ if $\vec{\xfrak}=(\xfrak_0,\ldots,\xfrak_{L-1})$} \\ \hline
$|\vec{\xfrak}\rangle$ & \pageref{not194} & \thead{defined when $\vec{\xfrak}\in\{0,1\}^n$; it is the $n$-qubit register encoding $\vec{\xfrak}$; \\ physicists call this object a \emph{ket}} \\ \hline
$X$ & \pageref{not1} & a set (usually assumed to be finite) \\ \hline
$X^X$ & \pageref{not37} & the transformation semigroup of all functions $X\rightarrow X$ \\ \hline
$\Xcal$ & \pageref{not137} & \thead{variable denoting an arithmetic partition; used e.g. with \\ enumerating indices ($\Xcal_1,\Xcal_2,\ldots$) to avoid a clash with \\ the notation $\Pcal_i$} \\ \hline
$\Xfrak$ & \pageref{not41} & \thead{variable used to denote a set viewed as an alphabet, \\ from whose elements finite sequences are formed} \\ \hline
$y$ & \pageref{not39} & variable denoting an element of the same ground set as $x$ \\ \hline
$Y_i$ & \pageref{not221} & the domain of definition of $\para_i$, $\para'_i$ and $\para''_i$ \\ \hline
$Y_{i,p}$ & \pageref{not226} & the domain of definition of $\para'_{i,p}$ \\ \hline
$\overline{Y}_i$ & \pageref{not247} & the parameter set $\prod_{p\in\Pfrak_i}{Y_{i,p}}$ \\ \hline
$\Ycal$ & \pageref{not311} & variable denoting an arithmetic partition (like $\Xcal)$ \\ \hline
$\Yfrak$ & \pageref{not278} & (used in Theorem \ref{dengTheo}) a special set of primes \\ \hline
$z$ & \pageref{not62} & \thead{variable denoting an element of the same ground set as \\ $x$ or $y$} \\ \hline
$\Zcal_i$ & \pageref{not192} & \thead{appears in the definition of a partition-tree register; \\ a specifically defined object that encodes enough information \\ to reconstruct $\Pcal_i$ (sometimes more)} \\ \hline
$\IZ/m\IZ$ & \pageref{not3} & \thead{the ring of residues modulo $m$, with underlying set \\ $\{0,1,\ldots,m-1\}$} \\ \hline
$(\IZ/m\IZ)^{\ast}$ & \pageref{not106} & the multiplicative group of units of $\IZ/m\IZ$ \\ \hline
$\alpha_i$ & \pageref{not128} & the linear coefficient of $A_i$ \\ \hline
$\overline{\alpha}_i$ & \pageref{not146} & the linear coefficient of $\Acal_i$ \\ \hline
$\overline{\alpha}_{i,h}$ & \pageref{not153} & the linear coefficient of $\Acal_{i,h}$ \\ \hline
$\overline{\alpha}'_i$ & \pageref{not302} & the linear coefficient of $\Acal'_i$ \\ \hline
$\beta_i$ & \pageref{not129} & the constant coefficient of $A_i$ \\ \hline
$\overline{\beta}_i$ & \pageref{not222} & the constant coefficient of $\Acal_i$ \\ \hline
$\overline{\beta}_{i,h}$ & \pageref{not154} & the constant coefficient of $\Acal_{i,h}$ \\ \hline
$\overline{\beta}'_i$ & \pageref{not303} & the constant coefficient of $\Acal'_i$ \\ \hline
$\overline{\beta}''_i$ & \pageref{not304} & the constant coefficient of $(\Acal'_i)^{\ord_{s'_i}(\overline{\alpha}'_i)}$ \\ \hline
$\gamma$ & \pageref{not54} & Euler-Mascheroni constant \\ \hline
$\Gamma$ & \pageref{not11P5} & variable denoting a (usually finite) digraph \\ \hline
$[\Gamma]_{\cong}$ & \pageref{not76} & the digraph isomorphism type of $\Gamma$ \\ \hline
$\Gamma^{\ast}$ & \pageref{not70} & the dual digraph of $\Gamma$ \\ \hline
$\Gamma_g$ & \pageref{not11} & the functional graph of $g$ \\ \hline
$\Gamma_f^{(i)}$ & \pageref{not190} & the functional graph of $f_{\mid U_i}$ \\ \hline
$\Gamma_{\per}$ & \pageref{not56} & the induced subgraph of $\Gamma_f$ on $\bigcup_{i\in\per(\overline{f})}{C_i}$ \\ \hline
$\Delta$ & \pageref{not73} & variable denoting a finite directed rooted tree \\ \hline
$\zeta$ & \pageref{not102} & variable denoting a cycle of a function \\ \hline
$\eta_{i,l}(x)$ & \pageref{not296} & \thead{a certain $s$-congruence, which in case $l$ is a multiple of $\ell$ \\ (the $\overline{f}$-cycle length of $i$) characterizes the $f$-periodic points \\ $x\in C_i$ with $f^l(x)=x$} \\ \hline
$\theta_{i,h}(x)$ & \pageref{not155} & \thead{an $s$-congruence that characterizes when $x\in\IZ/s\IZ$ \\ lies in the image of $\Acal_{i,h}$} \\ \hline
$\Theta(x)$ & \pageref{not157} & a system of $m$-CCs in the single variable $x$ \\ \hline
$\Theta_{i,h}(x)$ & \pageref{not156} & \thead{a system of (at most two) $s$-CCs that characterizes when \\ $x\in C_i\cong\IZ/s\IZ$ has $\hfrak$-value $h$} \\ \hline
$\iota$ & \pageref{not112} & \thead{variable denoting an isomorphism or bijection; see in \\ particular the notation $\iota_{\vec{r}}$ introduced after Problem \ref{crlListTensorProb}} \\ \hline
$\iota_i$ & \pageref{not31} & the bijection $\IZ/s\IZ\rightarrow C_i$, $x\mapsto\omega^{i+dx}$ \\ \hline
$\kappa_{i,p}$ & \pageref{not230} & shorthand for $\nu_p^{(\kappa_p)}(\overline{\beta}_i)$ \\ \hline
$\kappa_p$ & \pageref{not229} & shorthand for $\nu_p(s)$ \\ \hline
$\kappa_{\Pcal,A}(\vec{\nu},\vec{\nu'},J)$ & \pageref{not95} & technical parameter, used in the definition of $\sigma_{\Pcal,A}(\vec{\nu},\vec{\nu'})$ \\ \hline
$\lambda(\Pcal,A)$ & \pageref{not161} & like $\Pfrak'(\Pcal,A)$, with the last spanning congruence deleted \\ \hline
$\lambda_i^t(\Pcal)$ & \pageref{not162} & \thead{defined when $i<d$ is $\overline{f}$-periodic, $t\in\IN_0$, and $\Pcal$ is an \\ arithmetic partition of $C_i$; it is the arithmetic partition \\ of $C_{i_t}$ defined recursively via $\lambda_i^0(\Pcal):=\Pcal$ and \\ $\lambda_i^{t+1}(\Pcal):=\lambda(\lambda_i^t(\Pcal),A_{i_t})$} \\ \hline
$\Lambda$ & \pageref{not101} & \thead{variable denoting a lift function $\IZ/m_1\IZ\rightarrow\IZ/m_2\IZ$ where \\ $m_1,m_2\in\IN^+$ with $m_1\mid m_2$ (i.e., $\Lambda(x)\equiv x\Mod{m_1}$ for all \\ $x\in\IZ/m_1\IZ$); the details of its definition vary by context} \\ \hline
$\mu_a$ & \pageref{not65} & \thead{the endomorphism $x\mapsto ax$ of the group $\IZ/m\IZ$; \\ $m$ must be clear from context} \\ \hline
$\nu$ & \pageref{not44} & variable denoting a logical sign ($\emptyset$ or $\neg$) \\ \hline
$\nu_{l,l'}$ & \pageref{not299} & \thead{(used in Subsubsection \ref{subsubsec5P3P3}) notation for a logical sign; \\ it is defined as $\emptyset$ (the positive logical sign) if $l\mid l'$, and \\ as $\neg$ otherwise} \\ \hline
$\vec{\nu}_{i,L,l}$ & \pageref{not300} & \thead{the logical sign tuple $(\nu_{l,l'})_{l'\in\Cfrak_{i,L}}$; if all $f$-periodic points \\ in $C_i$ have $f$-cycle length at most $L$ and $l\in\Cfrak_{i,L}$, then \\ $\Bcal(\Vcal_{i,L},\vec{\nu}_{i,L,l})$ consists precisely of those $f$-periodic $x\in C_i$ \\ that have $f$-cycle length exactly $l$} \\ \hline
$\vec{\nu}^+$ & \pageref{not275} & a special notation used in Subsubsection \ref{subsubsec5P2P3} \\ \hline
$\vec{\nu'}^-$ & \pageref{not276} & \thead{a special notation used in Subsubsection \ref{subsubsec5P2P3}; \\ inverse to the notation $\vec{\nu}^+$} \\ \hline
$\nu_p(n)$ & \pageref{not67} & the $p$-adic valuation of $n$ \\ \hline
$\nu_p^{(v)}(n)$ & \pageref{not66} & defined as $\min\{v,\nu_p(n)\}$ \\ \hline
$\vec{\xi}_{i,k}$ & \pageref{not160} & \thead{defined when $i\in\{0,1,\ldots,d-1\}$ is $\overline{f}$-periodic and \\ $k\in\{0,1,\ldots,H_i\}$; it is the logical sign tuple of length $H_i$ \\ in which precisely the first $k$ entries are equal to the positive \\ logical sign; the block $\Bcal(\Ucal_i,\vec{\xi}_{i,k})$ consists just of those $x\in C_i$ \\ that are of $\hfrak$-value $k$} \\ \hline
$\pi(m)$ & \pageref{not114} & the set of prime divisors of $m$ \\ \hline
$\rho_{\l}$ & \pageref{not85} & \thead{the left-regular representation of a group $G$ (clear from \\ context) on itself, i.e., the function $G\rightarrow\Sym(G)$, \\ $x\mapsto (y\mapsto xy)$} \\ \hline
$\rho_{\rrm}$ & \pageref{not84} & \thead{the right-regular representation of a group $G$ (clear from \\ context) on itself, i.e., the function $G\rightarrow\Sym(G)$, \\ $x\mapsto (y\mapsto yx)$} \\ \hline
$\sigma$ & \pageref{not53} & divisor sum function \\ \hline
$\sigma_{\Pcal,A}(\vec{\nu},\vec{\nu'})$ & \pageref{not94} & \thead{a technical parameter, for the significance of which see \\ the comments on $\Pfrak'(\Pcal,A)$} \\ \hline
$\tau(m)$ & \pageref{not278P5} & the number of (positive) divisors of $m\in\IN^+$ \\ \hline
$\phi$ & \pageref{not104} & Euler's totient function \\ \hline
$\varphi$ & \pageref{not83} & a variable denoting a group endomorphism \\ \hline
$\varphi_i$ & \pageref{not130} & \thead{defined for $i\in\{0,1,\ldots,d-1\}$ such that $a_i\not=0$; it is $\mu_{\alpha_i}$, \\ the group endomorphism of $\IZ/s\IZ$ associated with $A_i$} \\ \hline
$\chi$ & \pageref{not195} & a variable denoting a characteristic function \\ \hline
$\psi$ & \pageref{not12} & a permutation of $X$ \\ \hline
$\omega$ & \pageref{not18} & a primitive element of $\IF_q$ \\ \hline
\caption{Tabular overview of notation and terminology used in this paper.}
\label{termNotTable}
\end{longtable}

\end{document}